\documentclass[11pt]{amsart}
\usepackage[leqno]{amsmath}
\usepackage{amssymb}
\usepackage{enumerate}

\usepackage{subfigure}
\usepackage{graphicx}
\usepackage[all]{xy}

\usepackage{latexsym,ifthen,xspace}
\usepackage{amsthm}
\usepackage{calc}
\usepackage[colorlinks=true, pdfstartview=FitV, linkcolor=blue,
  citecolor=blue, urlcolor=blue, pagebackref=false]{hyperref}

\usepackage[margin=1in,centering]{geometry}

\usepackage{xcolor}
\usepackage{morefloats}
\usepackage[mathscr]{euscript}
\let\mathescr=\mathscr
\usepackage{mathrsfs}
\usepackage{comment}
\usepackage{bm}
\usepackage{todonotes}
\DeclareMathAlphabet{\mymathbb}{U}{BOONDOX-ds}{m}{n}

\newtheorem{theorem}{Theorem}[section]
\newtheorem{lemma}[theorem]{Lemma}
\newtheorem{proposition}[theorem]{Proposition}
\newtheorem{coro}[theorem]{Corollary}

\theoremstyle{definition}
\newtheorem{definition}[theorem]{Definition}
\newtheorem{assumption}[theorem]{Assumption}
\newtheorem{example}[theorem]{Example}

\theoremstyle{remark}
\newtheorem{remark}[theorem]{Remark}

\numberwithin{equation}{section}

\DeclareMathAlphabet{\mathsl}{OT1}{cmss}{m}{sl}
\SetMathAlphabet{\mathsl}{bold}{OT1}{cmss}{bx}{sl}

\newcommand{\cK}{\ensuremath{\mathcal K}}

\newcommand{\cN}{\ensuremath{\mathcal N}}

\newcommand{\cS}{\ensuremath{\mathcal S}}

\newcommand{\cX}{\ensuremath{\mathcal X}}
\newcommand{\cY}{\ensuremath{\mathcal Y}}

\newcommand{\bbR}{\ensuremath{\mathbb R}}

\newcommand{\rmv}{\ensuremath{\mathrm{v}}}

\newcommand{\rme}{\ensuremath{\mathrm{e}}}

\DeclareMathOperator{\prob}{\mathbb{P}}

\newcommand{\ldef}{\ensuremath{\mathrel{\mathop:}=}}

\def\indicator{{\mathchoice {1\mskip-4mu\mathrm l}
{1\mskip-4mu\mathrm l}{1\mskip-4.5mu\mathrm l}
{1\mskip-5mu\mathrm l}}}

\newcommand{\C}{\mathbb{C}}
\newcommand{\D}{\mathbb{D}}
\newcommand{\E}{\mathbb{E}}

\newcommand{\h}{\mathbb{H}}
\newcommand{\N}{\mathbb{N}}
\newcommand{\Z}{\mathbb{Z}}
\newcommand{\p}{\mathbb{P}}
\newcommand{\Q}{\mathbb{Q}}
\newcommand{\R}{\mathbb{R}}
\newcommand{\s}{\mathbb{S}}

\newcommand{\Fh}{\mathfrak {h}}

\newcommand{\CB}{\mathcal {B}}

\newcommand{\CD}{\mathcal {D}}

\newcommand{\CF}{\mathcal {F}}

\newcommand{\CS}{\mathcal {S}}
\newcommand{\CT}{\mathcal {T}}

\newcommand{\CW}{\mathcal {W}}

\newcommand{\CN}{\mathcal {N}}
\newcommand{\CA}{\mathcal {A}}

\newcommand{\CH}{\mathcal {H}}

\newcommand{\SLE}{{\rm SLE}}
\newcommand{\QLE}{{\rm QLE}}

\newcommand{\dist}{\text{dist}}

\newcommand{\diam}{\mathop{\text{diam}}}

\newcommand{\re}{\text{Re}}
\newcommand{\cov}{\text{cov}}

\newcommand{\one}{{\bf 1}}

\newcommand{\wt}{\widetilde}
\newcommand{\ol}{\overline}

\newcommand{\giv}{\,|\,}

\newcommand{\sol}[1]{{}}

\newcommand{\strip}{{\mathscr{S}}}
\newcommand{\cyl}{\mathscr{C}}

\newcommand{\wh}{\widehat}
\newcommand{\fb}[2]{B^\bullet(#1,#2)}

\newcommand{\bandlaw}[2]{\mu_{{\mathrm{Band}}}^{L=#1,W=#2}}
\newcommand{\diskweighted}[2]{\mu_{\mathrm{BD,W}}^{A=#1,L=#2}}

\newcommand{\bmlaw}[1]{\mu_{\mathrm{BM}}^{A=#1}}
\newcommand{\bminflaw}{\mu_{\mathrm{BM}}}
\newcommand{\bdisklaw}[1]{\mu_{\mathrm{BD}}^{L=#1}}
\newcommand{\bdisklawweighted}[1]{\mu_{\mathrm{BD,W}}^{L=#1}}

\newcommand{\qsphereinflaw}{\mu_{\mathrm{QSPH}}}
\newcommand{\qdiskinflaw}{\mu_{\mathrm{QD}}}
\newcommand{\qdisk}[1]{\mu_{\mathrm{QD}}^{L=#1}}
\newcommand{\qdiskweighted}[1]{\mu_{\mathrm{QD,W}}^{L=#1}}
\newcommand{\qwedge}[1]{\mu_{\mathrm{QW}}^{W=#1}}

\newcommand{\qdiskweightedfixedarea}[2]{\mu_{\mathrm{QD,W}}^{A=#1,L=#2}}

\newcommand{\qdistnoarg}[1]{d_{#1}}
\newcommand{\qdist}[3]{d_{#1}(#2,#3)}
\newcommand{\qball}[3]{B_{#1}(#2,#3)}
\newcommand{\qmeasure}[1]{\mu_{#1}}
\newcommand{\qbmeasure}[1]{\nu_{#1}}
\newcommand{\qfb}[4]{B_{#1,#2}^\bullet(#3,#4)}

\newcommand{\eball}[2]{B(#1,#2)}

\newcommand{\innerboundary}{\partial_{\mathrm {In}}}
\newcommand{\outerboundary}{\partial_{\mathrm {Out}}}
\newcommand{\CRT}{\mathrm{CRT}}
\newcommand{\BM}{\mathrm{BM}}

\newcommand{\MCPU}[1]{\mathbb{M}^{\mathrm{CPU}}_{#1}}

\newcommand{\ccw}{{\boldsymbol {\circlearrowleft}}}
\newcommand{\cw}{{\boldsymbol {\circlearrowright}}}

\definecolor{mygreen}{cmyk}{1,0,1,0}

\begin{document}

\title[Two-sided heat kernel bounds for $\sqrt{8/3}$-Liouville Brownian motion]{Two-sided heat kernel bounds for $\sqrt{8/3}$-Liouville Brownian motion}

\author{Sebastian Andres}
\address{
Institut f\"ur Mathematische Stochastik, Technische Universit\"at Braunschweig}
\curraddr{Institut f\"ur Mathematische Stochastik,
Technische Universit\"at Braunschweig,
Universit\"atsplatz 2, 38106 Braunschweig, Germany}
\email{sebastian.andres@tu-braunschweig.de}
\thanks{}

\author{Naotaka Kajino}
\address{Department of Mathematics, Graduate School of Science, Kobe University
\newline\indent Research Institute for Mathematical Sciences, Kyoto University}
\curraddr{Research Institute for Mathematical Sciences,
Kyoto University,
Kitashirakawa-Oiwake-cho, Sakyo-ku, Kyoto 606-8502, Japan}
\email{nkajino@kurims.kyoto-u.ac.jp}
\thanks{}

\author{Konstantinos Kavvadias}
\address{Department of Mathematics, Massachusetts Institute of Technology}
\curraddr{Department of Mathematics,
Massachusetts Institute of Technology,
77 Massachusetts Avenue, Cambridge MA 02139-4307, USA 
}
\email{kavva941@mit.edu}
\thanks{}

\author{Jason Miller}
\address{Centre for Mathematical Sciences, University of Cambridge}
\curraddr{Centre for Mathematical Sciences,
University of Cambridge,
Wilberforce Road, Cambridge
CB3 0WB, UK}
\email{jpmiller@statslab.cam.ac.uk}



\date{March 12, 2026}

\dedicatory{}

\begin{abstract}
Liouville Brownian motion (LBM) is the canonical diffusion process on a Liouville quantum gravity (LQG) surface.  In this work, we establish upper and lower bounds for the heat kernel for LBM when $\gamma=\sqrt{8/3}$ in terms of the $\sqrt{8/3}$-LQG metric which are sharp up to a polylogarithmic factor in the exponential.
\end{abstract}

\maketitle
\setcounter{tocdepth}{1}
\tableofcontents

\section{Introduction}
\label{sec:intro}

\subsection{Background}
During the past 15 years, substantial research activity has been focused on the study of a random geometry induced by the two-dimensional Gaussian free field, commonly known under the banner of \emph{Liouville quantum gravity}. For an overview, see, for instance, the recent surveys \cite{BP24, SH23, Gw20}. 
Suppose that $h$ is an instance of (some form of) the Gaussian free field (GFF) on a domain $D \subseteq \C$.  The Liouville quantum gravity (LQG) surface described by $h$ formally refers to the random two-dimensional Riemannian manifold with metric tensor
\begin{equation}
\label{eqn:metric_tensor_intro}
e^{\gamma h(z)} (dx^2 + dy^2)
\end{equation}
where $\gamma \in (0,2]$ is a parameter and $dx^2 + dy^2$ denotes the Euclidean metric on $D$.
This expression does not make literal sense because $h$ is a distribution and not a function.
There has been a considerable amount of work in recent years aimed at making rigorous sense of~\eqref{eqn:metric_tensor_intro}.
The construction of the volume form of~\eqref{eqn:metric_tensor_intro} is related to Kahane's theory of \emph{Gaussian multiplicative chaos} \cite{k1985gmc} and appears in a number of places.  The approach taken in \cite{ds2011kpz} is to let $h_\epsilon(z)$ be the average of $h$ on $\partial \eball{z}{\epsilon}$ and then take
\begin{equation}
\label{eqn:area_intro}
\qmeasure{h} = \lim_{\epsilon \to 0} \epsilon^{\gamma^2/2} e^{\gamma h_\epsilon(z)} dz
\end{equation}
where $dz$ denotes Lebesgue measure on $D$, see \cite{Be17} for an elementary approach establishing a universal limit measure for a general class of mollifications of the field.

The metric (i.e., two-point distance function) associated with~\eqref{eqn:metric_tensor_intro} was first constructed in the case $\gamma=\sqrt{8/3}$ in \cite{ms2015qle1,ms2016qle2,ms2016qle3,ms2015mapmaking}, building on \cite{ms2016qledef} and using the tools from \cite{dms2014mating,ms2015spheres}.
The construction is indirect and is based on defining a growth process called \emph{quantum Loewner evolution} which turns out to describe the growth of metric balls in the resulting metric space.
The metric associated with~\eqref{eqn:metric_tensor_intro} was subsequently constructed for all $\gamma \in (0,2)$ as a limit of the type~\eqref{eqn:area_intro} in \cite{dddf2019tight,dfgs2019weak,gm2019local,gm2019conf,gm2019metric,gm2019confcov} but with $\gamma$ replaced by $\xi = \gamma/d_\gamma$ where $d_\gamma$ is the exponent constructed in \cite{dzz2019heat}.  In the present work, we will focus on the case that $\gamma=\sqrt{8/3}$.  This value is special because it turns out to be equivalent to the Brownian map \cite{ms2016qle2,ms2015mapmaking}, which is the Gromov-Hausdorff-Prokhorov scaling limit of random quadrangulations \cite{lg2013brownianmap,m2013brownianmap}.

Our main focus will be on the relationship between \emph{Liouville Brownian motion} (LBM) and the $\sqrt{8/3}$-LQG metric.
Recall that LBM is the Brownian motion associated with~\eqref{eqn:metric_tensor_intro} and was first constructed in \cite{grv2016lbm,Be15}.
It is defined as a time change of a standard planar Brownian motion where the change of time depends on the underlying LQG surface.
By general theory, the LBM turns out to be symmetric with respect to the Liouville measure $\qmeasure{h}$.
In \cite{grv2014dirichlet} Garban, Rhodes and Vargas also identified the Dirichlet form associated with the LBM and showed that its transition semigroup is absolutely continuous with respect to $\qmeasure{h} $,
meaning that the Liouville heat kernel $p_t(x,y)$ exists. Moreover, they observed that the intrinsic metric
generated by that Dirichlet form is identically zero, which indicates some non-Gaussian heat kernel behavior.
This degeneracy of the intrinsic metric is known to occur typically for diffusions on fractals, whose heat kernels indeed satisfy the so-called sub-Gaussian estimates; see e.g.\ the survey articles \cite{Ba13,Ku14} and references therein.
The works \cite{mrvz2016heat,AK16} establish the continuity of the Liouville heat kernel $p_{t}(x,y)$ in $(t,x,y)$ and some upper and lower bounds on it.
The bounds in \cite{AK16} have successfully identified the order of the on-diagonal part $p_{t}(x,x)$ for small $t$ as $t^{-1}$, up to a factor of a power of $\log t^{-1}$ reflecting the randomness of the environment and except that the lower bound is proved only for $\qmeasure{h}$-a.e.\ $x$.
On the other hand, for the off-diagonal behavior of $p_{t}(x,y)$, the sub-Gaussian upper bounds obtained in \cite{mrvz2016heat,AK16} are stated in terms of the \emph{Euclidean metric} and thereby are expected to be far from being sharp, and the known sub-Gaussian lower bound due to \cite{mrvz2016heat} gives a decay estimate only in $t$, with an exponent which is also not expected to be sharp, for each fixed $x,y$.
The work \cite{dzz2019heat} is focused on the transition kernel for a random walk on a certain graph approximation to LQG between fixed points and uses it to construct the dimension exponent for LQG.

The definition of LBM is further motivated by recent works which have shown that it arises as the scaling limit of simple random walk on certain graph approximations to LQG.  The convergence was first proved modulo time parameterization in \cite{gms2017matedcrt} (for the so-called mated-CRT map) and \cite{gms2018tuttepoissonvoronoi} (for the Poisson-Voronoi tessellation of the Brownian map, equivalently $\sqrt{8/3}$-LQG) using the invariance principle established in \cite{gms2018rwre}.  The convergence was upgraded to obtain the time parameterization in \cite{bg2020lbm}.

Most results in LQG, including the present work, remain specific to two dimensions and have not yet been extended to
higher dimensions, which is mainly due to the critical role of conformal invariance in two-dimensional results that is not available in higher dimension. Nonetheless, recent
works \cite{Ce22,DGZ23, DHKS24, DHKS23, BG25}  initiated the study of higher-dimensional analogs of LQG. In this context, higher dimensional version of LBM have been constructed in \cite{DHKS24, BG25}, where in \cite{BG25} the short-time asymptotics of the heat kernel along the diagonal and the spectral dimension are identified.

\subsection{Main results}

The main contributions of the present work are sharp off-diagonal upper and lower bounds for the (continuous) Liouville heat kernel $p_{t}(u,v)$ in the case $\gamma=\sqrt{8/3}$ in terms of the $\sqrt{8/3}$-LQG metric, denoted by $\qdistnoarg{h}$ in the following,  which hold for all points simultaneously for a.e.\ instance of the $\sqrt{8/3}$-LQG surface and match up to polylog errors in the exponent.
We will state and prove our results in the case that the underlying LQG surface is the $\sqrt{8/3}$-LQG sphere \cite{dms2014mating,dkrv2016sphere}.  By absolute continuity, one can extract similar heat kernel bounds for other LQG surfaces.

\begin{theorem}
\label{thm:ubd}
There exists a finite constant $\kappa >0$ such that the following is true.  For a.e.\ instance $\CS = (\s^2,h)$ of the $\sqrt{8/3}$-LQG sphere, there exist random positive constants $C_i=C_i(h)$, $i=1,2$, such that for all $u,v \in \CS$ and $t \in (0,1/2]$, 
\begin{equation} \label{eq:ubd}
p_t(u,v) \leq \frac{C_1 (\log t^{-1})^\kappa}{t} \exp\Biggl( - C_2 \biggl(\frac{\qdist{h}{u}{v}^{4}}{t} \biggr)^{\frac{1}{3}} \biggl( \log\biggl( e + \frac{\qdist{h}{u}{v}}{t} \biggr) \biggr)^{-\kappa} \Biggr).
\end{equation}
\end{theorem}

\begin{theorem}
\label{thm:lbd}
There exists a finite constant $\kappa>0$ such that the following is true.  For a.e.\ instance $\CS = (\s^2,h)$ of the $\sqrt{8/3}$-LQG sphere, there exists a random constant $C=C(h) \in (0,1)$ such that for all $u,v \in \CS$ and $0 < t < C \qdist{h}{u}{v}$,

\begin{equation} \label{eq:lbd}
p_t(u,v) \geq \exp\Biggl( - \biggl(\frac{\qdist{h}{u}{v}^{4}}{t} \biggr)^{\frac{1}{3}} \biggl( \log\biggl( \frac{\qdist{h}{u}{v}}{t} \biggr) \biggr)^{\kappa} \Biggr).
\end{equation}
Moreover, for all $u \in \CS$ and $t \in (0,1/2]$, we have that $p_t(u,u) \geq C t^{-1} (\log t^{-1})^{-\kappa}$.
\end{theorem}

  \begin{remark}
  (i) The lower bound in Theorem~\ref{thm:lbd} becomes effective in the regime $0 < t < \qdist{h}{u}{v}^4$, which reflects that the LBM needs to travel a sufficiently long distance so that the possibly bad local geometries around the starting point become irrelevant.
  A similar phenomenon can be observed, for instance, for simple random walks on supercritical percolation clusters, where one has to allow the random walk some random time to exit bad parts of the cluster before Gaussian heat kernel decay emerges; see \cite{Ba04}.
  
  (ii) While we do not discuss here the necessity of the polylogarithmic corrections in the established heat kernel bounds, we remark that in many instances of stochastic processes on random media, in particular for processes in low dimensions and models at criticality, heat kernel fluctuations are known to occur, caused by local irregularities in the random medium; see \cite{ACK23} for a recent review on this topic.
  \end{remark}

Off-diagonal upper and lower bounds of the heat kernel similar to~\eqref{eq:ubd} and~\eqref{eq:lbd}
have been proved in \cite{Cr08PTRF,Cr12PRIMS,BarCrKum2017AOP,Ar2021AIHP,AngCrHTShi2021AOP},
for the canonical diffusions on various random fractals which are trees or sufficiently
close to being trees so that their heat kernel behavior can be described very well in terms of
the effective resistance metric as established in the general results in \cite{Cr07}.
Our main results, Theorems~\ref{thm:ubd} and~\ref{thm:lbd} above, are in sharp contrast to
those preceding results in that the effective resistance metric is no longer well-defined
since the LBM a.s.\ does not hit a given point. To the best of our knowledge,
Theorems~\ref{thm:ubd} and~\ref{thm:lbd} are the first result in the literature establishing
sharp sub-Gaussian heat kernel bounds for diffusions on random fractals
which do not admit well-defined effective resistance metrics.
We also mention the work \cite{berestycki2023weyl} establishing a version of Weyl's law for LQG surfaces.
Precise heat kernel estimates are also a key point in that paper, but this concerns the on-diagonal
behaviour rather than the off-diagonal one and especially getting the right leading constant in some sense.

As for general values of $\gamma \in (0,2)$, we expect that it should be possible to obtain some
Liouville heat kernel estimates involving both the time and the LQG metric by using the methods
introduced in the present paper. However, we do not expect them to be as sharp as our current estimates.
The main difference between the general $\gamma \in (0,2)$ case and the $\gamma = \sqrt{8/3}$ case
(i.e., that of the Brownian map) is that in the latter case, we have a way to explore the metric
structure of the space in a Markovian way while such an exploration is not available in the former case.
More precisely, suppose that we choose two points $x,y$ on the Brownian map $(\CS,d)$ independently
according to the quantum area measure $\nu$ and consider the filled metric ball $B^{\bullet}(x,r)$
centered at $x$ with radius $r > 0$ (i.e., the set of points which are disconnected from $y$ by the
metric ball $B(x,r)$ in $(\CS,d)$). Then conditional on the boundary length of $\partial B^{\bullet}(x,r)$,
we have that $\mathcal{S} \setminus B^{\bullet}(x,r)$, considered as a random metric measure space
with its canonical metric and measure, has an explicit law that depends only on its boundary length
and it is independent of $B^{\bullet}(x,r)$ considered similarly as a random metric measure space.
Moreover, the process describing the boundary length evolution of $\partial B^{\bullet}(x,r)$
with respect to $r$ has an explicit form and it can be described by using tools from the theory of
L\'evy processes; for details see Section~\ref{sec:bm-LQG}, in particular Subsection \ref{ssec:bm-review}.
That type of explicit description and Markovian structure allow us to obtain such sharp estimates
on the Liouville heat kernel as in Theorems \ref{thm:ubd} and \ref{thm:lbd}, and we do not have
them for other values of $\gamma$.

\subsection{Strategy and outline}

The proof of Theorem~\ref{thm:ubd} is based on a criterion for off-diagonal sub-Gaussian heat kernel upper bounds, implied by results in \cite{GT12,GK17}, in terms of an exit time estimate and an on-diagonal heat kernel upper bound.
However, since such estimates hold only with polylogarithmic corrections in the present context of LQG, in a first step we extend the relevant results in \cite{GT12,GK17}.
More precisely, we derive a perturbed off-diagonal sub-Gaussian heat kernel upper bound which allows fluctuations in the exit time estimate and in the on-diagonal heat kernel upper bound to be given by a general class of perturbation functions, including polynomial and polylogarithmic corrections.
Assuming a similar upper bound on the volumes of metric balls instead of an on-diagonal heat kernel upper bound, we also deduce a similarly perturbed on-diagonal heat kernel lower bound by a standard method.
These are done in Section~\ref{sec:slD}, which is written in the framework of a general diffusion (without killing inside) as it might be of independent interest.
We refer to \cite{Cr07} for similar results on heat kernel bounds with fluctuations for local resistance forms (symmetric diffusions with well-defined effective resistance metrics).

After a review of the Brownian map, LQG and LBM in Section~\ref{sec:bm-LQG},
we next establish volume estimates for $\sqrt{8/3}$-LQG which hold with polylogarithmic corrections in Section~\ref{sec:bm-volume}.  More precisely, it was shown in earlier work by Le Gall \cite{lg2010geodesics} in the context of the Brownian map that for each $\delta > 0$ there a.s.\ exists $r_0 > 0$ so that for all $r \in (0,r_0)$ the volume of every ball of radius $r$ is between $r^{4+\delta}$ and $r^{4-\delta}$.  We improve this to show that there exists a constant $\kappa > 0$ and a.s.\ exists $r_0 > 0$ so that the volume of every ball of radius $r \in (0,r_0)$ is between $r^4 (\log r^{-1})^{-\kappa}$ and $r^4 (\log r^{-1})^\kappa$.  The proofs in this part of the work are based purely on Brownian map techniques.  By the equivalence of $\sqrt{8/3}$-LQG and the Brownian map, the same estimates hold also for $\sqrt{8/3}$-LQG.  We remark that some of our volume estimates in Section~\ref{sec:bm-volume} can also be extracted from more recent work of Le Gall \cite{LG21} but we have included our proofs since the reader might find them of independent interest.

As a preparation for the proofs of the quenched exit time lower bound for LBM and
the off-diagonal heat kernel lower bound~\eqref{eq:lbd} in Theorem~\ref{thm:lbd},
in Section~\ref{sec:percolation-exploration} we prove a percolation result for graphs
formed by tilings of $\sqrt{8/3}$-LQG surfaces by chunks of $\SLE_6$
(Propositions~\ref{prop:good_chunks_percolate} and~\ref{prop:good_chunks_percolate_half_plane}).
Roughly speaking, they state that with high probability we can construct a strongly
supercritical configuration of such tiles each of which have a certain prescribed set of properties,
under the assumption that with sufficiently high probability each tile has that set of properties.
We first prove this result in the simpler setting of the half-plane
as Proposition~\ref{prop:good_chunks_percolate_half_plane}, and then
translate it into an analogous result, Proposition~\ref{prop:good_chunks_percolate},
for the setting of the disk.

We then proceed in Section~\ref{sec:exit-time-bounds} to establish the quenched exit time upper and lower bounds for LBM,
which states that there exists a constant $\kappa > 0$ and a.s.\ exists $r_0 > 0$ so that for all $r \in (0,r_0)$ the conditional expectation (given the underlying LQG surface)
of the amount of time it takes an LBM to exit any ball of radius $r$ when starting from its center is between $r^4 (\log r^{-1})^{-\kappa}$ and $r^4 (\log r^{-1})^{\kappa}$.
The upper bound follows easily by combining the volume upper bound from Section~\ref{sec:bm-volume} with the known two-sided H\"{o}lder continuity estimate for $\qdistnoarg{h}$ with respect to the usual spherical metric due to \cite[Theorem~1.2]{ms2016qle2}
and the fact that the LBM has the same Green functions as the Brownian motion on the sphere.
The proof of the lower bound is based on an application of Proposition~\ref{prop:good_chunks_percolate}
with the prescribed property for each tile being that with uniformly positive probability it takes LBM a certain amount of time to cross.
Altogether, this is enough to conclude Theorem~\ref{thm:ubd} and the on-diagonal heat kernel lower bound in Theorem~\ref{thm:lbd}
thanks to the on-diagonal heat kernel upper bound in \cite{AK16}, the general result from Section~\ref{sec:slD} and the volume upper bound from Section~\ref{sec:bm-volume}.

We finish by establishing~\eqref{eq:lbd} in Theorem~\ref{thm:lbd} using a chaining argument in Section~\ref{sec:LBM-HK-bounds}.
We will construct our chains out of annuli consisting of chunks of $\SLE_6$ by applying Proposition~\ref{prop:good_chunks_percolate} in a manner similar to the proof of the exit time lower bound above.
The special property that the chunks which make up these annuli will have is an a priori lower bound on the amount of time it takes the LBM to cross.

\subsection*{Acknowledgements}
N.K.\ was supported in part by JSPS KAKENHI Grant Numbers JP18H01123, JP22H01128, JP23K22399.
J.M.\ was supported by ERC starting grant 804116 (SPRS) and ERC consolidator grant ARPF (Horizon Europe UKRI G120614).
This work was supported by the Research Institute for Mathematical Sciences (RIMS), an International Joint Usage/Research Center located in Kyoto University.
In particular, part of this work was achieved while S.A.\ was staying at RIMS as a visiting associate professor, and he gratefully acknowledges the hospitality and the received support.

\section{Heat kernel bounds with fluctuations for general diffusions} \label{sec:slD}

The purpose of this section is to give some sufficient conditions for the heat kernel
(transition density) of a diffusion (without killing inside) on a general state space to satisfy
a sub-Gaussian type off-diagonal upper bound and an on-diagonal lower bound which possibly involve some
lower order correction terms that are typically polylogarithmic. The main results of
this section (Theorem~\ref{thm:uhk} and Proposition~\ref{prop:lower-ondiag} below),
stated in their simplest possible forms that are still applicable to the case of the
$\sqrt{8/3}$-Liouville Brownian motion, yield the following theorem;
see \cite[Appendix A.2 and Section 4.5]{FOT11} and \cite[Appendix A.1]{CF12} for the basics of Markov processes.

\begin{theorem} \label{thm:heat-kernel-simple}
Let $(\cX,d)$ be a compact metric space with at least two points,
let $\mu$ be a Radon measure on $\cX$ with full support,
and let $X=(\{X_t\}_{t\in[0,\infty)},\{ P_x\}_{x \in \cX})$ be a conservative diffusion on $\cX$ which
admits a (unique) continuous function $p=p_{t}(x,y):(0,\infty)\times \cX \times \cX \to [0,\infty)$
such that for any $(t,x) \in (0,\infty)\times \cX$,
\begin{equation} \label{eq:transition-density}
P_{x}[X_{t} \in dy] = p_{t}(x,y)\,\mu(dy)\quad\text{(as Borel measures on $\cX$).}
\end{equation}
Further set $\diam\cX := \sup_{x,y\in\cX}d(x,y)(\in(0,\infty))$,
$B(x,r):=\{y\in\cX \mid d(x,y)<r\}$ for $(x,r)\in\cX \times (0,\infty)$,
$\tau_{A}:=\inf\{t\in[0,\infty)\mid X_{t}\not\in A\}$ \textup{($\inf\emptyset:=\infty$)} for $A\subset \cX$,
let $\alpha \in (0,\infty)$, $\beta \in (1,\infty)$, and consider the following conditions:
\begin{enumerate}
\item[$\mathrm{(V)}_{\leq}$]\textup{(Volume upper bound)}
	There exist $\kappa_{\mathrm{vu}}\in[0,\infty)$ and $C_{\mathrm{V}}\in(0,\infty)$
	such that for any $(x,r)\in \cX \times (0,\diam\cX]$,
	\begin{align} \label{eq:volume-simple}
	\mu(B(x,r)) \leq C_{\mathrm{V}} r^{\alpha} \bigl(\log(e+r^{-1})\bigr)^{\kappa_{\mathrm{vu}}}.
	\end{align} 
\item[$\mathrm{(E)}$]\textup{(Mean exit time estimate)}
	There exist $\kappa_{\mathrm{el}},\kappa_{\mathrm{eu}}\in[0,\infty)$, $a_{\rme}\in[1,\infty)$
	and $C_{\rme}\in(0,\infty)$ such that for any $(x,r)\in \cX \times (0,a_{\rme}^{-1}\diam\cX]$,
	\begin{align} \label{eq:exit_est-simple}
	C_{\rme}^{-1} r^{\beta} \bigl(\log(e+r^{-1})\bigr)^{-\kappa_{\mathrm{el}}}
		\leq E_x[ \tau_{B(x,r)} ]
		\leq C_{\rme} r^{\beta} \bigl(\log(e+r^{-1})\bigr)^{\kappa_{\mathrm{eu}}}.
	\end{align}
\item[$\mathrm{(DU)}$]\textup{(On-diagonal upper bound)}
	There exist $\kappa_{\mathrm{du}}\in[0,\infty)$ and $C_{\mathrm{du}}\in(0,\infty)$
	such that for any $(t,x,y) \in (0,(\diam\cX)^{\beta}]\times \cX \times \cX$,
	\begin{align} \label{eq:duhk-simple}
	p_{t}(x,y) \leq C_{\mathrm{du}} t^{-\alpha/\beta} \bigl(\log(e+t^{-1})\bigr)^{\kappa_{\mathrm{du}}}.
	\end{align}
\end{enumerate}
Then the following hold:
\begin{enumerate}[\upshape(i)]
\item \textup{(Off-diagonal upper bound)} Assume $\mathrm{(E)}$ and $\mathrm{(DU)}$,
	set $\kappa_{\mathrm{u}}:=(2+\beta)(\kappa_{\mathrm{el}}+\kappa_{\mathrm{eu}})$,
	let $\varepsilon_{h}\in(0,\infty)$ satisfy $\varepsilon_{h}\kappa_{\mathrm{u}}<1$ and
	set $\kappa_{\mathrm{du}}':=(1-\varepsilon_{h}\kappa_{\mathrm{u}})^{-1}(\kappa_{\mathrm{du}}+\kappa_{\mathrm{u}}\alpha/\beta)$.
	Then there exist $c_{1},c_{2}\in(0,\infty)$ such that for any $(t,x,y) \in (0,(\diam\cX)^{\beta}]\times \cX \times \cX$,
	\begin{align} \label{eq:uhk-simple}
	p_{t}(x,y) \leq
		c_{1} \frac{\bigl(\log(e+t^{-1})\bigr)^{\kappa_{\mathrm{du}}'}}{t^{\alpha/\beta}}
		\exp\Biggl(-c_{2}\biggl(\frac{d(x,y)^{\beta}}{t}\biggr)^{\frac{1}{\beta-1}} \biggl(\log\biggl(e+\frac{d(x,y)}{t}\biggr)\biggr)^{-\frac{\kappa_{\mathrm{u}}}{\beta-1}}\Biggr).
\end{align}
\item \textup{(On-diagonal lower bound)} Assume $\mathrm{(V)}_{\leq}$, $\mathrm{(E)}$ and that $X$ is $\mu$-symmetric,
	i.e., $p_{t}(x,y)=p_{t}(y,x)$ for any $(t,x,y)\in(0,\infty) \times \cX \times \cX$.
	Set $\kappa_{\mathrm{u}}:=(2+\beta)(\kappa_{\mathrm{el}}+\kappa_{\mathrm{eu}})$ and
	$\kappa_{\mathrm{dl}}:=\kappa_{\mathrm{vu}}+\kappa_{\mathrm{u}}\alpha/\beta$.
	Then there exists $c_{3}\in(0,\infty)$ such that for any $(t,x) \in (0,(\diam\cX)^{\beta}] \times \cX$,
	\begin{align} \label{eq:lower-ondiag-simple}
	p_{t}(x,x) \geq c_{3} t^{-\alpha/\beta} \bigl(\log(e+t^{-1})\bigr)^{-\kappa_{\mathrm{dl}}}.
	\end{align}
\end{enumerate}
\end{theorem}

Theorem~\ref{thm:heat-kernel-simple} is obtained at the end of this section as a corollary of the
main results of this section (Theorem~\ref{thm:uhk} and Proposition~\ref{prop:lower-ondiag} below),
which, for potential future applications, we state and prove in the general setting
of a diffusion without killing inside on a locally compact separable metric space.

The rest of this section is organized as follows.
First in Subsection~\ref{ssec:setting_general}, we introduce the general setting and
suitable generalizations of conditions $\mathrm{(V)}_{\leq}$, $\mathrm{(E)}$ and $\mathrm{(DU)}$
of Theorem~\ref{thm:heat-kernel-simple}. In Subsection~\ref{ssec:uhk_general}, we state
the generalization of Theorem~\ref{thm:heat-kernel-simple}-(i) (Theorem~\ref{thm:uhk})
and prove it as an application of \cite[Theorems 6.2 and 6.4]{GK17}.
Subsection~\ref{ssec:dlhk_general} states and proves the generalization of
Theorem~\ref{thm:heat-kernel-simple}-(ii) (Proposition~\ref{prop:lower-ondiag}),
which is a mere adaptation of a well-known argument to our setting. Then
Theorem~\ref{thm:heat-kernel-simple} is deduced from Theorem~\ref{thm:uhk} and
Proposition~\ref{prop:lower-ondiag} in Subsection~\ref{ssec:proof-hksimple}.

\subsection{Setting and conditions} \label{ssec:setting_general}

Throughout the rest of Section~\ref{sec:slD}, we assume that $(\cX,d)$ is a locally compact
separable metric space, and that $\mu$ is a Radon measure on $\cX$ with full support,
i.e., a Borel measure on $\cX$ which is finite on any compact subset of $\cX$ and
strictly positive on any non-empty open subset of $\cX$.
We will refer to such a triple $(\cX, d, \mu)$ as a \emph{metric measure space}.
For $(x,r) \in \cX \times (0,\infty)$, we set $B(x,r) := \{ y \in \cX \mid d(x,y)<r \}$,
and the closure of $B(x,r)$ in $\cX$ is denoted by $\overline{B}(x,r)$.
Let $\cX_{\partial}:=\cX\cup\{\partial\}$ be the one-point compactification of $\cX$, so that
the Borel $\sigma$-field $\mathescr{B}(\cX_{\partial})$ of $\cX_{\partial}$ can be expressed,
in terms of that $\mathescr{B}(\cX)$ of $\cX$, as
$\mathescr{B}(\cX_{\partial})=\mathescr{B}(\cX)\cup\{A\cup\{\partial\}\mid A\in\mathescr{B}(\cX)\}$.
In what follows, $[-\infty,\infty]$-valued functions on $\cX$ are always set
to be $0$ at $\partial$ unless their values at $\partial$ are already defined:
$f(\partial):=0$ for $f:\cX\to[-\infty,\infty]$.

Let
$X=\bigl(\Omega,\mathescr{M},\{X_{t}\}_{t\in[0,\infty]},\{P_{x}\}_{x\in \cX_{\partial}}\bigr)$
be a diffusion without killing inside on $(\cX,\mathescr{B}(\cX))$
with life time $\zeta$ and shift operators $\{\theta_{t}\}_{t\in[0,\infty]}$.
By definition, $(\Omega,\mathescr{M})$ is a measurable space, $\{X_{t}\}_{t\in[0,\infty]}$
is a family of $\mathescr{M}/\mathescr{B}(\cX_{\partial})$-measurable maps $X_{t}:\Omega\to\cX_{\partial}$
such that $[0,\infty)\ni t\mapsto X_{t}(\omega)\in \cX_{\partial}$ is continuous and
$X_{t}(\omega)=\partial$ for any $t\in[\zeta(\omega),\infty]$ for each $\omega\in\Omega$,
where $\zeta(\omega):=\inf\{t\in[0,\infty)\mid X_{t}(\omega)=\partial\}$, and
$\{\theta_{t}\}_{t\in[0,\infty]}$ is a family of maps $\theta_{t}:\Omega\to\Omega$
satisfying $X_{s}\circ\theta_{t}=X_{s+t}$ for any $s,t\in[0,\infty]$. The pair $X$
of such a stochastic process $\bigl(\Omega,\mathescr{M},\{X_{t}\}_{t\in[0,\infty]}\bigr)$
and a family $\{P_{x}\}_{x\in \cX_{\partial}}$ of probability measures
on $(\Omega,\mathescr{M})$ is then called a
\emph{diffusion without killing inside} on $(\cX,\mathescr{B}(\cX))$
if and only if $X$ is a normal Markov process on
$(\cX,\mathescr{B}(\cX))$ whose minimum completed admissible filtration
$\mathescr{F}_{*}=\{\mathescr{F}_{t}\}_{t\in[0,\infty]}$ is \emph{right-continuous}
and which is \emph{strong Markov} with respect to $\mathescr{F}_{*}$; see
\cite[Section A.2, (M.2)--(M.5), the paragraph before Lemma A.2.2, and (A.2.3)]{FOT11}
for the precise definitions of these notions.
We set $E_{x}[(\cdot)]:=\int_{\Omega}(\cdot)\,dP_{x}$ for $x\in \cX_{\partial}$.
For each $\sigma$-finite Borel measure $\nu$ on $\cX_{\partial}$, the function
$\cX_{\partial}\ni x\mapsto P_{x}[B]$ is measurable with respect to the $\nu$-completion
of $\mathescr{B}(\cX_{\partial})$ for any $B\in\mathescr{F}_{\infty}$
by \cite[Exercise A.1.20-(i)]{CF12}, and associated with $\nu$ is a measure $P_{\nu}$ on
$(\Omega,\mathescr{F}_{\infty})$ given by $P_{\nu}[B]:=\int_{\cX_{\partial}}P_{x}[B]\,d\nu(x)$.
We also set $\dot{\sigma}_{B}(\omega):=\inf\{t\in[0,\infty)\mid X_{t}(\omega)\in B\}$ ($\inf\emptyset:=\infty$)
and $\tau_{B}(\omega):=\dot{\sigma}_{\cX_{\partial}\setminus B}(\omega)$
for $B\subset \cX_{\partial}$ and $\omega\in\Omega$, so that $\dot{\sigma}_{B},\tau_{B}$ are
$\mathescr{F}_{*}$-stopping times if $B\in\mathescr{B}(\cX_{\partial})$ by \cite[Theorem A.2.3]{FOT11}.

The most general form of heat kernel bounds we aim to establish in this section involves
three functions $v(r)$, $\Psi(r)$ and $h(r)$ as introduced below, which correspond to the functions
$r^{\alpha}$, $r^{\beta}$ and $\log(e+r^{-1})$ in Theorem~\ref{thm:heat-kernel-simple},
respectively.

\begin{assumption} \label{ass:scale-functions}
Throughout the rest of Section~\ref{sec:slD}, we fix homeomorphisms
$v,\Psi:[0,\infty) \rightarrow [0,\infty)$ and a non-increasing continuous function $h:(0,\infty] \rightarrow [1,\infty)$ with the properties that there exist positive constants $C_{v},\alpha_{1},\alpha_{2},C_{\Psi},\beta_{1},\beta_{2},C_{h},\varepsilon_{h}$
with $\alpha_{1}\leq \alpha_{2}$ and $1<\beta_{1}\leq\beta_{2}$ such that the following hold:
for any $s,t\in(0,\infty)$ with $s\leq t$,
\begin{align} \label{eq:scale_v}
C_{v}^{-1} \Bigl( \frac{t}{s} \Bigr)^{\alpha_{1}}
	&\leq \frac{v(t)}{v(s)} \leq C_{v} \Bigl( \frac{t}{s} \Bigr)^{\alpha_{2}},\\
\label{eq:def:scalefct}
C_{\Psi}^{-1} \Bigl( \frac{t}{s} \Bigr)^{\beta_{1}}
	&\leq \frac{\Psi(t)}{\Psi(s)} \leq C_{\Psi} \Bigl( \frac{t}{s} \Bigr)^{\beta_{2}},\\
\label{eq:ass_h}
\frac{h(s)}{h(t)} &\leq C_{h} \Bigl( \frac{t}{s} \Bigr)^{\varepsilon_{h}}.
\end{align}
Note that~\eqref{eq:def:scalefct} is equivalent to the property that for any $s,t\in(0,\infty)$ with $s\leq t$,
\begin{align} \label{eq:inv_scalefct}
C_{\Psi}^{-1/\beta_{2}} \Bigl( \frac{t}{s} \Bigr)^{1/\beta_{2}}
	\leq \frac{\Psi^{-1}(t)}{\Psi^{-1}(s)}
	\leq C_{\Psi}^{1/\beta_{1}} \Bigl( \frac{t}{s} \Bigr)^{1/\beta_{1}}.
\end{align}
\end{assumption} 

\begin{example} \label{ex:scale-functions-simple}
Let $\alpha\in(0,\infty)$ and $\beta\in(1,\infty)$. Then the triple of functions
$v,\Psi:[0,\infty)\to[0,\infty)$ and $h:(0,\infty]\to[1,\infty)$ given by
\begin{align} \label{eq:scale-functions-simple}
v(r) := r^{\alpha}, \quad \Psi(r) := r^{\beta} \quad \text{and} \quad h(r) := \log(e+r^{-1})
\end{align}
satisfies Assumption~\ref{ass:scale-functions} with $\alpha_{1}=\alpha_{2}=\alpha$, $\beta_{1}=\beta_{2}=\beta$,
arbitrary $\varepsilon_{h}\in(0,\infty)$ and $C_{h}=2+e^{-1}\varepsilon_{h}^{-1}$.
Indeed, to see~\eqref{eq:ass_h}, let $\varepsilon_{h}\in(0,\infty)$ and let $s,t\in(0,\infty)$ satisfy $s\leq t$.
Then clearly $h(s)/h(t)\leq\log(2e)\leq 2(t/s)^{\varepsilon_{h}}$ for $s\geq e^{-1}$, whereas for $s\leq e^{-1}$ we have
\begin{align*}
\frac{\log(e+s^{-1})}{\log(e+t^{-1})}
& \; = \;
1+ \frac{\log \bigl( \frac{e+s^{-1}} {e+t^{-1}}\bigr)} {\log(e+t^{-1})}
\; \leq \;
1+ \log \Bigl( \frac{e+s^{-1}} {e+t^{-1}}\Bigr)
\; \leq \;
1+ \log\Bigl(et + \frac t s\Bigr) \\
& \; \leq \;
1+ \log \frac{2t}{s}
\; \leq \;
2 + \log \frac{t}{s} \; \leq \; (2+e^{-1}\varepsilon_{h}^{-1}) \Bigl( \frac{t}{s} \Bigr)^{\varepsilon_{h}}.
\end{align*}
\end{example}

For the sake of the applicability of the main results of this section
(Theorem~\ref{thm:uhk} and Proposition~\ref{prop:lower-ondiag} below)
to diffusions on random non-compact spaces and to strongly local regular
symmetric Dirichlet forms, we further introduce the following assumption.

\begin{assumption} \label{ass:RYN}
Throughout the rest of Section~\ref{sec:slD}, we fix $R\in (0,\infty]$, a non-empty
open subset $\cY$ of $\cX$, and a Borel subset $\cN$ of $\cX$ with the property that
\begin{align} \label{eq:exceptional}
P_{x}[\dot{\sigma}_{\cN} = \infty] = 1 \qquad \text{for any $x\in \cX \setminus \cN$.}
\end{align}
\end{assumption}

Our conditions $\mathrm{(V)}_{\leq}$, $\mathrm{(E)}$ and $\mathrm{(DU)}$ for heat kernel bounds,
which we state next, concern the behavior of the measure $\mu$ and the diffusion $X$
\emph{only within metric balls contained in $\cY$ with radii at most $R$}, and hence
can often be verified even for diffusions on random non-compact spaces by choosing $R$
to be finite and $\cY$ to be bounded. The set $\cN$ can be considered as being removed
from the set of starting points of the diffusion $X$ by virtue of~\eqref{eq:exceptional},
and is thereby going to play the role of a set of \emph{``capacity zero with respect to $X$''};
we remark that the presence of such $\cN$ is inevitable in analyzing symmetric diffusions on
the basis of the general theory of regular symmetric Dirichlet forms presented in \cite{FOT11,CF12},
as illustrated in Remark~\ref{rmk:duhk-regularDF} below.

\begin{definition} [Volume upper bound] \label{def:volume}
We say that condition $\mathrm{(V)}_{\leq}$ holds if there exist constants
$\kappa_{\mathrm{vu}}\in[0,\infty)$, $a_{\rmv}\in[1,\infty)$ and $C_{\mathrm{V}}\in(0,\infty)$ such that
for any $(x,r)\in (\cY \setminus \cN) \times \bigl(0,\frac{R}{a_{\rmv}}\bigr)$ with $B(x,a_{\rmv} r) \subset \cY$,
\begin{align} \label{eq:volume}
\mu(B(x,r)) \leq C_{\mathrm{V}} v(r) h(r)^{\kappa_{\mathrm{vu}}}.
\end{align}
\end{definition}

\begin{definition} [Mean exit time estimate] \label{def:E}
We say that the \emph{mean exit time estimate} $\mathrm{(E)}$ holds if there exist constants
$\kappa_{\mathrm{el}},\kappa_{\mathrm{eu}}\in[0,\infty)$, $a_{\rme}\in[1,\infty)$ and $C_{\rme}\in(0,\infty)$
such that for any $(x,r)\in (\cY \setminus \cN) \times \bigl(0,\frac{R}{a_{\rme}}\bigr)$ with $B(x,a_{\rme} r) \subset \cY$,
\begin{align} \label{eq:exit_est}
C_{\rme}^{-1} \Psi(r) h(r)^{-\kappa_{\mathrm{el}}}
	\leq E_{x}[ \tau_{B(x,r)} ] \leq C_{\rme} \Psi(r) h(r)^{\kappa_{\mathrm{eu}}}.
\end{align}

\end{definition}

\begin{definition} [On-diagonal upper bound] \label{def:duhk}
We say that condition $\mathrm{(DU)}$ holds if there exist constants $\kappa_{\mathrm{du}}\in[0,\infty)$,
$a_{\mathrm{du}} \in [1,\infty)$ and $C_{\mathrm{du}}\in(0,\infty)$ such that for any
$(x_{0},r)\in \cY \times \bigl(0,\frac{R}{a_{\mathrm{du}}}\bigr)$ with $B(x_{0},a_{\mathrm{du}} r) \subset \cY$,
any $t\in (0,\Psi(r))$ and any Borel subset $A$ of $B(x_{0},r)$,
\begin{align} \label{eq:duhk}
P_{x}[X_{t} \in A,\, t < \tau_{B(x_{0},r)}]
	\leq \frac{C_{\mathrm{du}}}{v(\Psi^{-1}(t))} h(t)^{\kappa_{\mathrm{du}}} \mu(A)
	\quad \text{for any $x\in B(x_{0},r) \setminus \cN$.}
\end{align}
\end{definition}

\begin{remark} \label{rmk:duhk-regularDF}
Assume that $X$ is $\mu$-symmetric and associated with a regular symmetric Dirichlet form on $L^{2}(\cX,\mu)$
(see \cite[Sections 1.1, 1.4, 4.1 and 4.2]{FOT11} for the precise definitions of these notions).
In this case, \emph{the validity of condition $\mathrm{(DU)}$ for some Borel subset $\cN$
of $\cX$ satisfying $\mu(\cN)=0$ and~\eqref{eq:exceptional} follows from $\mathrm{(DU)}$ with
``any $x\in B(x_{0},r) \setminus \cN$'' in~\eqref{eq:duhk} replaced by ``$\mu$-a.e.\ $x\in B(x_{0},r)$''}.

Indeed, for each $(x_{0},r)\in \cY \times (0,\frac{R}{a_{\mathrm{du}}})$ with $B(x_{0},a_{\mathrm{du}} r) \subset \cY$,
\cite[Theorem 5.4]{GK17} implies the existence of $\cN_{x_{0},r}\in\mathescr{B}(\cX)$
satisfying $\mu(\cN_{x_{0},r})=0$ and~\eqref{eq:exceptional} such that~\eqref{eq:duhk} with
$\cN_{x_{0},r}$ in place of $\cN$ holds for all $t\in (0,\Psi(r))$ and all Borel subset $A$ of $B(x_{0},r)$.
Then, choosing a countable dense subset $\cY_{0}$ of $\cY$, we see from \cite[Theorem 4.1.1]{FOT11}
that there exists $\cN \in \mathescr{B}(\cX)$ with the properties $\mu(\cN)=0$, \eqref{eq:exceptional} and
$\bigcup_{x_{0}\in\cY_{0},\,r\in(0,R/a_{\mathrm{du}})\cap\mathbb{Q},\,B(x_{0},a_{\mathrm{du}} r) \subset \cY}\cN_{x_{0},r} \subset \cN$.
Now it is elementary to see that $\mathrm{(DU)}$ holds with this $\cN$ and the same
constants $a_{\mathrm{du}}, C_{\mathrm{du}}, \kappa_{\mathrm{du}}$.
\end{remark}

\subsection{Off-diagonal upper bounds of the heat kernel} \label{ssec:uhk_general}

The statement of our main result on off-diagonal upper bounds of the heat kernel
(Theorem~\ref{thm:uhk} below) requires the following definition.

\begin{definition} \label{def:Phikappa}
For any $\kappa \in [0,\infty)$, we define a lower semi-continuous function
$\Phi_{\kappa}:[0,\infty)\times(0,\infty)\to[0,\infty]$ by
\begin{align} \label{eq:defPhi}
\Phi_{\kappa}(r,t) := \sup_{s\in(0,\infty)} \biggl( \frac{r}{\Psi^{-1}(sh(s)^{\kappa})} -\frac{t}{s} \biggr),
\end{align}
so that for any $r,t\in(0,\infty)$, $\Phi_{\kappa}(\cdot,t)$ is non-decreasing,
$\Phi_{\kappa}(r,\cdot)$ is non-increasing and $\Phi_{\kappa}(0,t)=0<\Phi_{\kappa}(r,t)<\infty$
by the upper inequality in~\eqref{eq:inv_scalefct}, $\beta_{1}>1$ and the assumption
that $h$ is $[1,\infty)$-valued and non-increasing.
\end{definition}

\begin{example} \label{ex:power_scaling}
Let $\beta\in(1,\infty)$ and assume that $\Psi(r)=r^{\beta}$ for any $r\in[0,\infty)$.
Then an elementary differential calculus easily shows that
for any $(r,t)\in[0,\infty)\times(0,\infty)$,
\begin{align}\label{eq:power_scaling_Phi0}
\Phi_{0}(r,t)=c_{\beta}\biggl(\frac{r^{\beta}}{t}\biggr)^{\frac{1}{\beta-1}},
\end{align}
where $c_{\beta}:=\beta^{-\beta/(\beta-1)}(\beta-1)=\beta^{-1/(\beta-1)}-\beta^{-\beta/(\beta-1)}$. On the other hand,
for each $\kappa\in[0,\infty)$, the effect of the correction term $h(s)^{\kappa}$ in~\eqref{eq:defPhi} can be estimated as
\begin{align}\label{eq:power_scaling_Phikappa}
\Phi_{\kappa}(r,t) \geq c_{\beta} \biggl( \frac{r^{\beta}}{t} \biggr)^{\frac{1}{\beta-1}} h\Bigl( ( t/r )^{\frac{\beta}{\beta-1}} \Bigr)^{-\frac{\kappa}{\beta -1}}
	= \Phi_{0}(r,t) h\Bigl( ( t/r )^{\frac{\beta}{\beta-1}} \Bigr)^{-\frac{\kappa}{\beta -1}}
\end{align}
($t/0:=\infty$) for any $(r,t)\in[0,\infty)\times(0,\infty)$.
Indeed, noting that~\eqref{eq:power_scaling_Phikappa} is obvious for $r=0$ and
that $h$ is $[1,\infty)$-valued and non-increasing, let $r,t\in(0,\infty)$ and set
\begin{align*}
s := \beta^{\frac{\beta}{\beta-1}} ( t/r )^{\frac{\beta}{\beta-1}} h\Bigl( ( t/r )^{\frac{\beta}{\beta-1}} \Bigr)^{\frac{\kappa}{\beta -1}}
	\in \bigl[ ( t/r )^{\frac{\beta}{\beta-1}},\infty\bigr),
\end{align*}
so that $h(s)\leq h\bigl( (t/r)^{\beta/(\beta-1)} \bigr)$. Then by this last inequality and~\eqref{eq:defPhi},
\begin{align*}
\Phi_{\kappa}(r,t) & \geq \frac{r}{\Psi^{-1}(sh(s)^{\kappa})} -\frac{t}{s}
	= \frac{r}{s^{1/\beta} h(s)^{\kappa/\beta}} -\frac{t}{s}  \\
& = \biggl( \frac{r^{\beta}}{t} \biggr)^{\frac{1}{\beta-1}} h\Bigl( ( t/r )^{\frac{\beta}{\beta-1}} \Bigr)^{-\frac{\kappa}{\beta -1}}
	\biggl( \beta^{-\frac{1}{\beta-1}} h(s)^{-\kappa/\beta} h\Bigl( ( t/r )^{\frac{\beta}{\beta-1}} \Bigr)^{\kappa/\beta} -\beta^{-\frac{\beta}{\beta-1}} \biggr) \\
&\geq \biggl( \frac{r^{\beta}}{t} \biggr)^{\frac{1}{\beta-1}} h\Bigl( ( t/r )^{\frac{\beta}{\beta-1}} \Bigr)^{-\frac{\kappa}{\beta -1}}
	\Bigl( \beta^{-\frac{1}{\beta-1}} -\beta^{-\frac{\beta}{\beta-1}} \Bigr),
\end{align*}
proving~\eqref{eq:power_scaling_Phikappa}.
\end{example}

In fact, the lower bound on the ratio $\Phi_{\kappa}(r,t)/\Phi_{0}(r,t)$ exhibited in~\eqref{eq:power_scaling_Phikappa}
extends to the case of general $\Psi$, as follows.

\begin{lemma}\label{lem:Phikappa_Phi0}
Let $\kappa\in[0,\infty)$. Then there exists $C_{\Phi_{\kappa}}\in(0,\infty)$
such that for any $(r,t)\in[0,\infty)\times(0,\infty)$,
\begin{equation} \label{eq:Phikappa_Phi0}
\Phi_{\kappa}(r,t) \geq C_{\Phi_{\kappa}} \Phi_{0}(r,t) h\Bigl( (t/r)^{\frac{\beta_{1}}{\beta_{1}-1}} \Bigr)^{-\frac{\kappa}{\beta_{1}-1}}.
\end{equation}
\end{lemma}

\begin{proof}
Set $\kappa':=(\beta_{1}-1)^{-1}\kappa$. \eqref{eq:Phikappa_Phi0} is obvious for $r=0$, so let $r,t\in(0,\infty)$,
and set $r':=C_{\Psi}^{-1/\beta_{1}}r$. Since $\{sh(s)^{-\kappa'}\mid s\in(0,\infty)\}=(0,\infty)$
and $\Phi_{0}(r',t)>0$, in view of~\eqref{eq:defPhi} we can choose $s\in(0,\infty)$ so that
\begin{align}\label{eq:Phikappa_Phi0-proof1}
\frac{r'}{\Psi^{-1}(sh(s)^{-\kappa'})}-\frac{t}{sh(s)^{-\kappa'}}>\frac{1}{2}\Phi_{0}(r',t)>0.
\end{align}
Then by~\eqref{eq:inv_scalefct}, $r'=C_{\Psi}^{-1/\beta_{1}}r$ and $\kappa'=(\beta_{1}-1)^{-1}\kappa$ we have
\begin{align} \label{eq:Phikappa_Phi0-proof2} 
\frac{r'h(s)^{-\kappa'}}{\Psi^{-1}(sh(s)^{-\kappa'})}
	\leq \frac{C_{\Psi}^{1/\beta_{1}}r'h(s)^{-\kappa'}}{\Psi^{-1}(sh(s)^{\kappa})}\biggl(\frac{sh(s)^{\kappa}}{sh(s)^{-\kappa'}}\biggr)^{1/\beta_{1}}
	= \frac{r}{\Psi^{-1}(sh(s)^{\kappa})},
\end{align}
and therefore from \cite[Lemma 2.10]{Mu20} (see also~\eqref{eq:Phi_kappa_ratio} and~\eqref{eq:Phi_kappa_comparable} below),
\eqref{eq:Phikappa_Phi0-proof1}, \eqref{eq:Phikappa_Phi0-proof2} and~\eqref{eq:defPhi} we obtain
\begin{align}
\Phi_{0}(r,t)h(s)^{-\kappa'}
	&\leq C_{\Psi}^{\frac{1}{\beta_{1}-1}}\Bigl(\frac{r}{r'}\Bigr)^{\frac{\beta_{1}}{\beta_{1}-1}}\Phi_{0}(r',t)h(s)^{-\kappa'}
	\leq 2C_{\Psi}^{\frac{2}{\beta_{1}-1}}\biggl(\frac{r'h(s)^{-\kappa'}}{\Psi^{-1}(sh(s)^{-\kappa'})}-\frac{t}{s}\biggr) \notag \\
&\leq 2C_{\Psi}^{\frac{2}{\beta_{1}-1}}\biggl(\frac{r}{\Psi^{-1}(sh(s)^{\kappa})}-\frac{t}{s}\biggr)
	\leq 2C_{\Psi}^{\frac{2}{\beta_{1}-1}}\Phi_{\kappa}(r,t).
\label{eq:Phikappa_Phi0-proof3}
\end{align}
On the other hand, setting $C_{\Psi}':=\bigl(C_{\Psi}^{-1/\beta_{1}}\Psi^{-1}(1)\bigr)^{\beta_{1}/(\beta_{1}-1)}\wedge 1$
and noting that
\begin{align*}
\frac{r'}{\Psi^{-1}(u)}-\frac{t}{u}
	=\frac{t}{\Psi^{-1}(u)}\biggl(\frac{r'}{t}-\frac{\Psi^{-1}(u)}{u}\biggr)
	\leq \frac{t}{\Psi^{-1}(u)}\biggl(\frac{r'}{t}-\Bigl(\frac{C_{\Psi}'}{u}\Bigr)^{\frac{\beta_{1}-1}{\beta_{1}}}\biggr)
\end{align*}
for any $u\in(0,1]$ by~\eqref{eq:inv_scalefct}, we have
\begin{align}\label{eq:Phikappa_Phi0-proof4}
\frac{r'}{\Psi^{-1}(u)}-\frac{t}{u}\leq 0
	\qquad\text{for any $u\in\biggl(0,C_{\Psi}'\Bigl(\frac{t}{r'}\wedge 1\Bigr)^{\frac{\beta_{1}}{\beta_{1}-1}}\biggr]$.}
\end{align}
It then follows from~\eqref{eq:Phikappa_Phi0-proof4} and~\eqref{eq:Phikappa_Phi0-proof1} that
$s\geq sh(s)^{-\kappa'}>C_{\Psi}'\bigl(\frac{t}{r'}\wedge 1\bigr)^{\beta_{1}/(\beta_{1}-1)}$,
and hence further from~\eqref{eq:ass_h} that, with
$C_{\Psi,h}:=C_{h}\bigl(\Psi^{-1}(1)\wedge 1\bigr)^{-\varepsilon_{h}\beta_{1}/(\beta_{1}-1)}$,
\begin{align*}
h(s)\leq
	\begin{cases}
	h\Bigl(C_{\Psi}'(t/r')^{\frac{\beta_{1}}{\beta_{1}-1}}\Bigr) \leq C_{\Psi,h}h\Bigl( (t/r)^{\frac{\beta_{1}}{\beta_{1}-1}} \Bigr) & \text{if $r'\geq t$,} \\
	h(C_{\Psi}') \leq h(C_{\Psi}')h(\infty)^{-1}h\Bigl( (t/r)^{\frac{\beta_{1}}{\beta_{1}-1}} \Bigr) & \text{if $r'\leq t$,}
	\end{cases}
\end{align*}
which together with~\eqref{eq:Phikappa_Phi0-proof3} shows~\eqref{eq:Phikappa_Phi0}.
\end{proof}

Now we can state the main result of this subsection; note that by Lemma~\ref{lem:Phikappa_Phi0} we can replace
$\exp\bigl(-c_{2}\Phi_{\kappa_{\mathrm{u}}}\bigl( d(x,y)\wedge R_{y},t\bigr)\bigr)$ in the right-hand side of~\eqref{eq:uhk} below by
\begin{align*}
\exp\Biggl(-c_{2}C_{\Phi_{\kappa_{\mathrm{u}}}}\Phi_{0}\bigl(d(x,y)\wedge R_{y},t\bigr)h\biggl(\Bigl(\frac{t}{d(x,y)\wedge R_{y}}\Bigr)^{\frac{\beta_{1}}{\beta_{1}-1}}\biggr)^{-\frac{\kappa_{\mathrm{u}}}{\beta_{1}-1}}\Biggr).
\end{align*}

\begin{theorem}[Off-diagonal upper bound] \label{thm:uhk}
Assume $\mathrm{(E)}$, $\mathrm{(DU)}$ and that $\overline{B}(x,r)$ is compact
for any $(x,r)\in(\cY\setminus\cN)\times(0,\frac{R}{2 a_\rme})$ with $B(x,2(a_\rme+1) r) \subset \cY$.
Set $\kappa_{\mathrm{u}}:=(2+\beta_2)(\kappa_{\mathrm{el}}+\kappa_{\mathrm{eu}})$ and
assume that $\varepsilon_{h}\kappa_{\mathrm{u}}<1$. Then there exists a Borel measurable function
$p=p_{t}(x,y): (0,\infty) \times (\cX\setminus\cN) \times \cY \rightarrow [0,\infty)$
such that the following hold:
\begin{enumerate}[\upshape(i)]
\item For any $(t,x) \in (0,\infty) \times (\cX\setminus\cN)$,
	\begin{align*}
	P_{x}[X_{t} \in dy] = p_{t}(x,y) \, \mu(dy).
	\end{align*}
\item There exist $c_{1}, c_{2} \in (0,\infty)$ 
	such that, with
	$\kappa_{\mathrm{du}}':=(1-\varepsilon_{h}\kappa_{\mathrm{u}})^{-1}(\kappa_{\mathrm{du}}+\kappa_{\mathrm{u}}\alpha_{2}/\beta_{1})$,
	for any $(t,x,y) \in (0,\infty) \times (\cX\setminus\cN) \times \cY$,
	\begin{align} \label{eq:uhk}
	p_{t}(x,y) \leq
	\frac{c_{1} h\bigl(t\wedge\Psi(R_{y})\bigr)^{\kappa_{\mathrm{du}}'}}{v\bigl(\Psi^{-1}(t)\wedge R_{y}\bigr)}
		\exp\Bigl(-c_{2}\Phi_{\kappa_{\mathrm{u}}}\bigl( d(x,y)\wedge R_{y},t\bigr)\Bigr),
	\end{align}
	where $R_{y} := R\wedge \inf_{z \in \cX \setminus \cY}d(y,z)$ \textup{($\inf\emptyset:=\infty$)}
	and $\Psi(\infty):=\infty$.
\end{enumerate}
\end{theorem}

The proof of Theorem~\ref{thm:uhk} is concluded at the end of this subsection. For this purpose,
we need some preliminary results on basic properties of $\Phi_{\kappa}$ and on upper bounds on
$P_{x}[\tau_{B(x,r)}\leq t]$. First, for $\Phi_{\kappa}$ we have the following lemma,
which asserts that, provided $\varepsilon_{h}\kappa<1$, the function
$\Psi^{-1}(sh(s)^{\kappa})$ in~\eqref{eq:defPhi} is comparable to $\Psi_{\kappa}^{-1}(s)$
for some homeomorphism $\Psi_{\kappa}:[0,\infty)\to[0,\infty)$ that still satisfies
\eqref{eq:def:scalefct} (for some different constants) even though $\Psi^{-1}(sh(s)^{\kappa})$
might not be itself strictly increasing.

\begin{lemma}\label{lem:Phi_kappa}
Let $\kappa\in[0,\varepsilon_{h}^{-1})$ and define a homeomorphism
$\Psi_{\kappa}:[0,\infty)\to[0,\infty)$ by
\begin{equation}\label{eq:Phi_kappa}
\Psi_{\kappa}^{-1}(t):=\sup_{s\in(0,t]}\Psi^{-1}(sh(s)^{\kappa})+\Psi^{-1}(t),
\end{equation}
which can be defined since $\lim_{s\downarrow 0}\Psi^{-1}(sh(s)^{\kappa})=0$ by
$\varepsilon_{h}\kappa<1$ and~\eqref{eq:ass_h}. Then there exists $C>0$ such that
\begin{equation}\label{eq:Psi_kappa_comparable}
C\Psi_{\kappa}^{-1}(t)\leq\Psi^{-1}(th(t)^{\kappa})\leq\Psi_{\kappa}^{-1}(t)
	\qquad\text{for any $t\in(0,\infty)$,}
\end{equation}
and $\Psi_{\kappa}$ satisfies~\eqref{eq:def:scalefct} with
$\beta_{2}$ replaced by $\beta_{\kappa,2}:=(1-\varepsilon_{h}\kappa)^{-1}\beta_{2}$ and
$C_{\Psi}$ by some $C_{\Psi,\kappa}\in(0,\infty)$. Moreover, with $\beta_{\kappa,1}:=\beta_{1}$ and
$\widetilde{\Phi}_{\kappa}(r,t):=\sup_{s\in(0,\infty)} \bigl( r/\Psi_{\kappa}^{-1}(s) -t/s \bigr)$ for
$(r,t)\in[0,\infty)\times(0,\infty)$, there exists $C'\in(0,\infty)$ such that for any $r,t,s\in(0,\infty)$ with $s\leq r$,
\begin{align}\label{eq:Phi_kappa_bounds}
C'^{-1}\min_{j\in\{1,2\}}\biggl(\frac{\Psi_{\kappa}(r)}{t}\biggr)^{\frac{1}{\beta_{\kappa,j}-1}}
	&\leq\widetilde{\Phi}_{\kappa}(r,t)
	\leq C'\max_{j\in\{1,2\}}\biggl(\frac{\Psi_{\kappa}(r)}{t}\biggr)^{\frac{1}{\beta_{\kappa,j}-1}},\\
\label{eq:Phi_kappa_ratio}
C'^{-1}\Bigl(\frac{r}{s}\Bigr)^{\frac{\beta_{\kappa,2}}{\beta_{\kappa,2}-1}}
	&\leq\frac{\widetilde{\Phi}_{\kappa}(r,t)}{\widetilde{\Phi}_{\kappa}(s,t)}
	\leq C'\Bigl(\frac{r}{s}\Bigr)^{\frac{\beta_{\kappa,1}}{\beta_{\kappa,1}-1}},\\
\label{eq:Phi_kappa_comparable}
\widetilde{\Phi}_{\kappa}(r,t) &\leq \Phi_{\kappa}(r,t) \leq C'\widetilde{\Phi}_{\kappa}(r,t).
\end{align}
\end{lemma}

\begin{proof}
Let $t\in(0,\infty)$. Clearly $\Psi^{-1}(th(t)^{\kappa})\leq\Psi_{\kappa}^{-1}(t)$.
By~\eqref{eq:inv_scalefct} and~\eqref{eq:ass_h},
for any $s\in(0,t]$, if $sh(s)^{\kappa}\geq th(t)^{\kappa}$ then
\begin{align}\label{eq:prf:Phi_kappa-1}
\Bigl(\frac{s}{t}\Bigr)^{1/\beta_{1}}\leq 1
	\leq\frac{\Psi^{-1}(sh(s)^{\kappa})}{\Psi^{-1}(th(t)^{\kappa})}
	\leq c\biggl(\frac{sh(s)^{\kappa}}{th(t)^{\kappa}}\biggr)^{1/\beta_{1}}
	\leq cC_{h}^{\kappa}\Bigl(\frac{s}{t}\Bigr)^{\frac{1-\varepsilon_{h}\kappa}{\beta_{1}}}
	\leq cC_{h}^{\kappa}\Bigl(\frac{s}{t}\Bigr)^{\frac{1-\varepsilon_{h}\kappa}{\beta_{2}}},
\end{align}
and if $sh(s)^{\kappa}\leq th(t)^{\kappa}$ then
\begin{align}\label{eq:prf:Phi_kappa-2}
c^{-1}\Bigl(\frac{s}{t}\Bigr)^{1/\beta_{1}}
	\leq c^{-1}\biggl(\frac{sh(s)^{\kappa}}{th(t)^{\kappa}}\biggr)^{1/\beta_{1}}
	\leq\frac{\Psi^{-1}(sh(s)^{\kappa})}{\Psi^{-1}(th(t)^{\kappa})}
	\leq c\biggl(\frac{sh(s)^{\kappa}}{th(t)^{\kappa}}\biggr)^{1/\beta_{2}}
	\leq cC_{h}^{\kappa}\Bigl(\frac{s}{t}\Bigr)^{\frac{1-\varepsilon_{h}\kappa}{\beta_{2}}}.
\end{align}
In particular, we see from~\eqref{eq:prf:Phi_kappa-1} and~\eqref{eq:prf:Phi_kappa-2} that
$\Psi^{-1}(sh(s)^{\kappa})\leq cC_{h}^{\kappa}\Psi^{-1}(th(t)^{\kappa})$ for any $s\in(0,t]$ and hence that
$\Psi_{\kappa}^{-1}(t)\leq 2\sup_{s\in(0,t]}\Psi^{-1}(sh(s)^{\kappa})\leq cC_{h}^{\kappa}\Psi^{-1}(th(t)^{\kappa})$,
proving~\eqref{eq:Psi_kappa_comparable}. It follows from~\eqref{eq:prf:Phi_kappa-1}, \eqref{eq:prf:Phi_kappa-2} and
\eqref{eq:Psi_kappa_comparable} that $\Psi_{\kappa}$ satisfies~\eqref{eq:def:scalefct} with $\beta_{2}$ replaced by
$\beta_{\kappa,2}=(1-\varepsilon_{h}\kappa)^{-1}\beta_{2}$ and $C_{\Psi}$ by some constant $C_{\Psi,\kappa}\in(0,\infty)$.
Finally, we have~\eqref{eq:Phi_kappa_bounds} by \cite[Lemma 3.19]{GT12} and
\cite[Lemma 5.7]{GK17}, \eqref{eq:Phi_kappa_ratio} by \cite[Lemma 2.10]{Mu20}, and
then~\eqref{eq:Phi_kappa_comparable} by~\eqref{eq:Psi_kappa_comparable} and
\eqref{eq:Phi_kappa_ratio}.
\end{proof}

Next, we prove some upper bounds on $P_{x}[\tau_{B(x,r)}\leq t]$, which is a key condition
for applying \cite[Theorems 6.2 and 6.4]{GK17} to conclude Theorem~\ref{thm:uhk}.
For any open set $U\subset \cX$ we set
\begin{align*}
\overline{E} (U) \; \ldef \; \sup_{x\in U\setminus \mathcal{N}} E_x [\tau_U].
\end{align*}

\begin{lemma} \label{lem:exit1}
For any open subset $U$ of $\cX$ with $\overline{E}(U)<\infty$
and any $(t,x)\in(0,\infty)\times (\cX \setminus \cN)$,
\begin{align*}
P_{x}[\tau_{U}<t] \leq 1 - \frac{E_x[\tau_U]}{\overline{E}(U)} + \frac{t}{\overline{E}(U)}.
\end{align*} 
\end{lemma}

\begin{proof}
See \cite[Lemma~3.12]{GT12}.
\end{proof}

\begin{lemma} \label{lem:exit2}
Assume that condition $\mathrm{(E)}$ holds, and set $\kappa_{1}:=\kappa_{\mathrm{el}}+\kappa_{\mathrm{eu}}$.
Then there exist constants $c_{1}, c_{2} \in (0,\infty)$ such that for any
$(x,r) \in (\cY\setminus \cN) \times \bigl(0,\frac{R}{2 a_{\rme}}\bigr)$ with
$B(x,(2a_{\rme}+1) r) \subset \cY$ and any $\lambda \in [c_{1} h(r)^{2\kappa_{1}}/\Psi(r),\infty)$,
\begin{align*}
E_{x}\bigl[ e^{-\lambda \tau_{B(x,r)}} \bigr] \leq 1 - c_{2} h(r)^{-\kappa_{1}}.
\end{align*}
\end{lemma}

\begin{proof}
For $r$ and $x$ as in the statement  set $B\ldef B(x,r)$. Then, by Lemma~\ref{lem:exit1} we have for all $t,\lambda \in (0,\infty)$,
\begin{align*}
E_x\big[ e^{-\lambda \tau_B} \big]  & \; \leq \; E_x\big[ e^{-\lambda \tau_B} \indicator_{\{\tau_B < t\}} \big] +
E_x\big[ e^{-\lambda \tau_B} \indicator_{\{\tau_B \geq t\}} \big]
\; \leq \; 
P_x \big[ \tau_B < t \big] + e^{-\lambda t} \\
& \; \leq \ 
1 - \frac{E_x[\tau_B]}{\overline{E}(B)} + \frac{t}{\overline{E}(B)}
+   e^{-\lambda t}.
\end{align*}
Next, note that for any $y\in B\setminus \mathcal{N}$ we have $B(y, 2a_{\rme} r)   \subset B(x,(2a_{\rme}+1) r) \subset \cY$ and therefore by condition  $\mathrm{(E)}$ we obtain
$\overline{E}(B) \geq E_x[\tau_B] \geq C_\rme^{-1} \Psi(r) h(r)^{-\kappa_{\mathrm{el}}}$ and
\begin{align} \label{eq:est:barE_B}
\overline{E}(B) &\; = \sup_{y\in B\setminus \mathcal{N}} E_y \big[\tau_B \big] \; \leq  
 \sup_{y\in B\setminus \mathcal{N}} E_y \big[\tau_{B(y,2r)} \big]
 \; \leq \;
 C_\rme \,  \Psi(2r)\, h(2r)^{\kappa_{\mathrm{eu}}} \nonumber \\
 & \; \leq \;
 c \, \Psi(r) \, h(r)^{\kappa_{\mathrm{eu}}} 
 \; \leq \;  c \,  h(r)^{\kappa_{\mathrm{el}}+\kappa_{\mathrm{eu}}} \,
 E_x \big[\tau_B \big].
\end{align}
Hence, setting $\kappa_1\ldef \kappa_{\mathrm{el}}+\kappa_{\mathrm{eu}}$ we get
\begin{align*}
E_x\big[ e^{-\lambda \tau_{B(x,r)}} \big] \; \leq \; 1- c_3 \,  h(r)^{-\kappa_1} +
c_4 \,  t \, \Psi(r)^{-1} \, h(r)^{\kappa_{\mathrm{el}}} +  e^{-\lambda t}.
\end{align*}
Now choose $t$ such that $c_3  h(r)^{-\kappa_1}=
2 c_4 \,  t \, \Psi(r)^{-1} \, h(r)^{\kappa_{\mathrm{el}}}$, so that
\begin{align*}
E_x\big[ e^{-\lambda \tau_{B(x,r)}} \big] \; \leq \; 1- \frac  12  c_3  h(r)^{-\kappa_1} +  e^{-\lambda t}.
\end{align*}
Finally, note that for all  $\lambda \geq c_1 h(r)^{2\kappa_1}/  \Psi(r)$ with $c_1>0$ sufficiently large and $t$ as chosen above  we have $ e^{-\lambda t}\leq \frac 1 4 c_3  h(r)^{-\kappa_1}$, which completes the proof.
\end{proof}

\begin{proposition} \label{prop:laplace_tau}
Assume that condition $\mathrm{(E)}$ holds, and set
$\kappa_{\mathrm{u}} \ldef (2+\beta_2)(\kappa_{\mathrm{el}}+\kappa_{\mathrm{eu}})$.
Then there exist $C_5, \gamma \in (0,\infty)$ such that, for any
$(x,r)\in (\cY\setminus \cN) \times \bigl(0,\frac{R}{2 a_{\rme}}\bigr)$ with $\overline{B}(x,r)$
compact and $B(x,2(a_\rme+1) r) \subset \cY$, and for any $\lambda \in (0,\infty)$,
\begin{align*}
E_{x}\bigl[ e^{-\lambda \tau_{B(x,r)}} \bigr] \leq C_5 \exp\biggl( - \frac{\gamma r}{\Psi^{-1}\bigl(\lambda^{-1} h(\lambda^{-1})^{\kappa_{\mathrm{u}}}\bigr) } \biggr).
\end{align*}
\end{proposition}
\begin{proof}
We follow the arguments in \cite[Lemma~3.14]{GT12}. 

\emph{Step~1.} Let $\lambda>0$, fix some $\rho <r$ to be specified later and set $n=\lfloor \frac{r}{\rho} \rfloor$. Further, set $\tau\ldef \tau_{B(x,r)}$ and 
\begin{align*}
u(y)\ldef E_y\big[ e^{-\lambda \tau}\big], \qquad m_k \ldef \sup_{ \overline{B}(x,k\rho) \setminus \cN} u, \qquad k=1,2, \ldots, n.
\end{align*}
For abbreviation define $\varepsilon \ldef C_2 h(\rho)^{-\kappa_1}$ (cf.\ Lemma~\ref{lem:exit2}) and $\varepsilon' \ldef \varepsilon/2$. Let $y_k$ be a point in $\overline{B}(x,k\rho) \setminus \cN$ such that 
\begin{align*}
(1-\varepsilon') \, m_k \; \leq \; u(y_k) \; \leq \; m_k \; \leq \; 1. 
\end{align*}
For $k\leq n-1$ notice that $B(y_k,\rho) \subset B(x, (k+1)\rho) \subset B(x,r)$.
Consider the function $v_k(y)\ldef E_y[e^{-\lambda \tau_k}]$ defined for $y\in B(y_k,\rho)$
where $\tau_k \ldef \tau_{B(y_k,\rho)}$. Then, since $\overline{B}(x,r)$ is compact,
by the continuity of $[0,\infty)\ni t\mapsto X_{t}(\omega)\in \cX_{\partial}$ for each
$\omega\in\Omega$, for all $y\in B(y_k,\rho) \setminus \cN$ we have
$X_{\tau_k} \in \overline{B}(y_k,\rho)\subset \overline{B}(x,(k+1) \rho)$ $P_y$-a.s.\ on $\{ \tau_k<\infty \}$.
By the strong Markov property \cite[Theorem A.1.21]{CF12} of $X$,
 \begin{align*}
 u(y) & \; = \; E_y \big[ e^{-\lambda \tau_k} e^{-\lambda (\tau-\tau_k)}  \big] \; = \; 
  E_y \Big[ e^{-\lambda \tau_k} E_{X_{\tau_k}} \big[ e^{-\lambda \tau} \big]  \Big]  \\
&  \; = \;
   E_y \big[ \indicator_{\{  \tau_k<\infty, X_{\tau_k} \in \overline{B}(x,(k+1)\rho)\setminus \cN \}} e^{-\lambda \tau_k} u(X_{\tau_k}) \big] \\
    &\; \leq \;
    E_y \big[ e^{-\lambda \tau_k} \big] \, \sup_{ \overline{B}(x,(k+1)\rho) \setminus \cN} u 
     \; =\; v_k(y) \, m_{k+1}.
 \end{align*}
 In particular, by choosing $y=y_k$, we get $ u(y_k)  \leq  v_k(y_k) m_{k+1}$ and therefore 
 \begin{align} \label{eq:pre_iter1}
  (1-\varepsilon') \, m_k \; \leq \;  v_k(y_k) \, m_{k+1}.
\end{align}  
If additionally 
\begin{align} \label{eq:lb_lam}
\lambda \; \geq \; C_1  \frac{h(\rho)^{2\kappa_1}}{\Psi(\rho)},
\end{align} 
 since $B(y_k,(2a_{\rme}+1) \rho) \subset  B(x,2(a_{\rme}+1) r) \subset \cY$ we may apply Lemma~\ref{lem:exit2} to $B(y_k, \rho)$ and obtain that $v_k(y_k) \leq 1 -\varepsilon$, which combined with~\eqref{eq:pre_iter1} shows that
\begin{align*}
(1-\varepsilon') \, m_k \leq  (1-\varepsilon) \, m_{k+1}, \qquad \forall k\in \{1,\ldots, n-1\}.
\end{align*}
By iteration we get that
\begin{align*}
u(x) & \; \leq \; m_1 \; \leq \; \Big( \frac{1-\varepsilon}{1-\varepsilon'} \Big)^{n-1} \, m_n  \; \leq \; \Big( \frac{1-2\varepsilon'}{1-\varepsilon'} \Big)^{n-1} \\
&
\; \leq \;  \exp\Big(-(n-1) \log\Big(1+ \frac{\varepsilon'}{1-2\varepsilon'}\Big) \Big) 
\; \leq \;  \exp\big(-(n-1)\log(1+2 \varepsilon') \big),
\end{align*}
where we used in the last step that $m_n \leq 1$. Since we may assume that $C_2\leq \frac  12$ and therefore  $\varepsilon'\leq \frac  14$ and since  $n\geq \frac{r}{\rho} -1$, 
 using  the definition of $\varepsilon'$ we obtain that
 \begin{align} \label{eq:post_iter}
 u(x) & \; \leq \;    \exp\Big(-(n-1) \log\big(1+2\varepsilon'\big) \Big)
 \; \leq \;  C \exp\Big(- C  \frac{r}{\rho h(\rho)^{\kappa_1}} \Big)
\end{align}  
 provided~\eqref{eq:lb_lam} is satisfied.
 
 \emph{Step 2: Choice of $\rho$.} In order to choose an appropriate $\rho \in (0,r)$, set $\kappa_0 \ldef 2\kappa_1$ and $C_0\ldef C_{\Psi}^{-1} \vee (C_1 h(\Psi^{-1}(1))^{\kappa_0})$, where    $C_\Psi>0$ is such that $\Psi^{-1}(t) \geq C_\Psi (t^{1/\beta_1} \wedge t^{1/\beta_2})$ for all $t>0$. Let $\rho \ldef \Psi^{-1} (C_0 \lambda^{-1} h(\lambda^{-1})^{\kappa_0})$. We claim that for this choice of $\rho$,
 \begin{align*}
  \lambda \; \geq \; C_1 \, \frac{h(\rho)^{\kappa_0}}{ \Psi(\rho)}, 
 \end{align*}
in particular, \eqref{eq:lb_lam} holds. To see this, let us first consider the case $\lambda \geq C_0$. Then, since $\beta_2\geq \beta_1> 1$ and
 \begin{align} \label{eq:lam_rho}
 \rho \geq \Psi^{-1} (C_0 \lambda^{-1}) \geq  C_\Psi \left((C_0 \lambda^{-1})^{1/\beta_1} \wedge (C_0 \lambda^{-1})^{1/\beta_2}\right) \geq C_{\Psi} \, C_0 \, \lambda^{-1} \geq \lambda^{-1},
 \end{align}
 we have 
 \begin{align*}
 \frac{h(\rho)^{\kappa_0}}{ \Psi(\rho)} = \frac{\lambda}{C_0} \, h(\lambda^{-1})^{-\kappa_0} h(\rho)^{\kappa_0} 
 \leq \frac{\lambda}{C_0} \leq \frac{\lambda}{C_3}. 
 \end{align*}
 On the other hand, if $\lambda \leq C_0$, then
 \begin{align} \label{eq:rho_1}
  \rho \geq \Psi^{-1} (C_0 \lambda^{-1}) \geq \Psi^{-1}(1)
 \end{align}
 and hence
 \begin{align*}
  \frac{h(\rho)^{\kappa_0}}{ \Psi(\rho)}  \leq 
  \frac{\lambda}{C_0} \, h(\Psi^{-1}(1))^{-\kappa_0} h(\rho)^{\kappa_0} \leq \frac{\lambda}{C_3}. 
 \end{align*}
 Thus, \eqref{eq:lb_lam} is satisfied for this choice of $\rho$ and by~\eqref{eq:post_iter} we have that
 \begin{align} \label{eq:claimrho}
E_x\big[ e^{-\lambda \tau_{B(x,r)}} \big]   \; \leq \;  C \exp\Big(- C  \frac{r}{\rho h(\rho)^{\kappa_1}} \Big)
 \end{align}
 provided $\rho < r$.

 \emph{Step 3: Conclusion.} In order to deduce the desired inequality let us assume first that $\rho \ldef \Psi^{-1} (C_0 \lambda^{-1} h(\lambda^{-1})^{\kappa_0})<r$. We need an upper estimate on $\rho h(\rho)^{\kappa_1}$.
 
  Let us again consider the case  $\lambda \geq C_0$ first. Then, by~\eqref{eq:lam_rho} we get
 \begin{align} \label{eq:boundrhoh1}
 \rho \, h(\rho)^{\kappa_1} & \; \leq \;  \rho \, h(\lambda^{-1})^{\kappa_1} \;
 = \;
 \Psi^{-1} \big(C_0 \lambda^{-1} h(\lambda^{-1}\big)^{\kappa_0}) \,
 h(\lambda^{-1})^{\kappa_1} \nonumber  \\
 & \; \leq \;
 C \, \Psi^{-1} \big(C_0 \lambda^{-1} h(\lambda^{-1}\big)^{\kappa_0+\beta_2\kappa_1}) \, \bigg( \frac{  C_0 \lambda^{-1} h(\lambda^{-1}\big)^{\kappa_0}}{C_0 \lambda^{-1} h(\lambda^{-1}\big)^{\kappa_0+\beta_2\kappa_1}} \bigg)^{\! \frac 1 {\beta_2}} 
  h(\lambda^{-1})^{\kappa_1}  \nonumber \\
  & \; \leq \;
  C \,  \Psi^{-1} \big( \lambda^{-1} h(\lambda^{-1})^{\kappa_{\mathrm{u}}}\big)
 \end{align}
 with $\kappa_{\mathrm{u}} \ldef \kappa_0+\beta_2\kappa_1 = (2+\beta_2)(\kappa_{\mathrm{el}}+\kappa_{\mathrm{eu}})$,
where we used  that $\Psi^{-1}(r_2)/\Psi^{-1}(r_1) \leq C (r_2/r_1)^{1/\beta_2}$ for any $0<r_1\leq r_2$.

On the other hand, if $\lambda \leq C_0$, then $\rho \geq \Psi^{-1}(1)$  by~\eqref{eq:rho_1} and therefore
\begin{align}  \label{eq:boundrhoh2}
\rho \, h(\rho)^{\kappa_1} & \; \leq \;  \rho \, h(\Psi^{-1}(1))^{\kappa_1} \; =\; 
\Psi^{-1} (C_0 \lambda^{-1} h(\lambda^{-1})^{\kappa_0}) 
 \, h(\Psi^{-1}(1))^{\kappa_1} \nonumber \\
  & \; \leq \; 
    C \,  \Psi^{-1} \big( \lambda^{-1} h(\lambda^{-1}\big)^{\kappa_{\mathrm{u}}}).
\end{align}
The claim now follows combining~\eqref{eq:claimrho} with~\eqref{eq:boundrhoh1} and~\eqref{eq:boundrhoh2}.

Now assume  that $\rho \ldef \Psi^{-1} (C_0 \lambda^{-1} h(\lambda^{-1})^{\kappa_0})\geq r$. Then,
\begin{align*}
r \; \leq \; \Psi^{-1} ( C_0 \lambda^{-1} h(\lambda^{-1})^{\kappa_0}) \; \leq \;
C \,  \Psi^{-1} \big( \lambda^{-1} h(\lambda^{-1})^{\kappa_{\mathrm{u}}}\big),
\end{align*}
and the desired estimate follows from $ e^{-\lambda \tau_{B(x,r)}} \leq 1$ by adjusting the constants $\gamma$ and $C_5$.
\end{proof}

\begin{coro} \label{cor:P}
Suppose that condition $\mathrm{(E)}$ holds, and set
$\kappa_{\mathrm{u}} \ldef (2+\beta_2)(\kappa_{\mathrm{el}}+\kappa_{\mathrm{eu}})$.
Then there exist $c_{1},c_{2} \in (0,\infty)$ such that, for any
$(x,r)\in (\cY\setminus \cN) \times \bigl(0,\frac R {2 a_{\rme}}\bigr)$ with $\overline{B}(x,r)$
compact and $B(x,2(a_\rme+1) r) \subset \cY$, and for any $t\in(0,\infty)$,
\begin{align*}
P_{x}[ \tau_{B(x,r)} \leq t ] \; \leq \; c_{1} \exp\bigl( -\Phi_{\kappa_{\mathrm{u}}}(c_{2}r,t) \bigr).
\end{align*}
\end{coro}
\begin{proof}
For $x$ and $r$ as in the statement and any $s,t\in (0,\infty)$, by Proposition~\ref{prop:laplace_tau}
\begin{align*} 
P_x[ \tau_{B(x,r)}\leq t ]
	\; = \; P_x \big[  e^{-\tau_{B(x,r)}/s} \geq e^{-t/s}\big]
	\; &\leq \; e^{t/s} \, E_x\big[ e^{-\tau_{B(x,r)}/s}  \big] \\
& \leq \; C_5 \exp\biggl( \frac{ t}{s} -  \frac{c r}{\Psi^{-1}(sh(s)^{\kappa_{\mathrm{u}}})}\biggr).
\end{align*}
Now the assertion follows by taking the infimum in $s\in(0,\infty)$ of the right-hand side of
this inequality and recalling the definition~\eqref{eq:defPhi} of $\Phi_{\kappa_{\mathrm{u}}}$.
\end{proof}

\begin{proof}[Proof of Theorem~\ref{thm:uhk}]
If $\mathcal{Y}=\mathcal{X}$ and $R=\infty$, then by Lemma~\ref{lem:Phi_kappa} along with
$\varepsilon_{h}\kappa_{\mathrm{u}}<1$, Corollary~\ref{cor:P} along with $(\mathrm{E})$, and
$\mathrm{(DU)}$ along with $\Psi_{\kappa_{\mathrm{u}}}\leq\Psi$, all the assumptions of \cite[Theorem 6.4]{GK17}
with $\Psi_{\kappa_{\mathrm{u}}}$ in place of $\Psi$ are satisfied, and hence \cite[Theorem 6.4]{GK17}
together with~\eqref{eq:Phi_kappa_ratio} and~\eqref{eq:Phi_kappa_comparable} yields the assertions.

Thus we may assume that either $\mathcal{Y}\not=\mathcal{X}$ or $R<\infty$ holds,
so that $R_{y}\in(0,\infty)$ for any $y\in\mathcal{Y}$.
Let $\{y_{n}\}_{n\in\mathbb{N}}$ be a countable dense subset of $\mathcal{Y}$.
For each $y\in\mathcal{Y}$,
set $R_{y}':=R_{y}/\bigl((2a_{\mathrm{du}})\vee(4a_{\mathrm{e}}+4)\bigr)$, so that
$R_{y}'\in(0,\frac R {a_{\mathrm{du}}})$, $B(y,a_{\mathrm{du}}R_{y}')\subset\mathcal{Y}$,
$R_{y}'\in(0,\frac R {2a_{\mathrm{e}}})$ and $B(x,2(a_{\mathrm{e}}+1)r)\subset\mathcal{Y}$
for any $(x,r)\in B(y,R_{y}'/2)\times(0,R_{y}')$. Therefore for each $n\in\mathbb{N}$,
Lemma~\ref{lem:Phi_kappa} along with $\varepsilon_{h}\kappa_{\mathrm{u}}<1$,
Corollary~\ref{cor:P} along with $(\mathrm{E})$, and
$\mathrm{(DU)}$ along with $\Psi_{\kappa_{\mathrm{u}}}\leq\Psi$
together imply that the open subset $B(y_{n},R_{y_{n}}'/2)$ of $\cX$ and the function $\Psi_{\kappa_{\mathrm{u}}}$
satisfy all the assumptions of \cite[Theorem 6.2]{GK17}, and hence we see from \cite[(6.4) and (6.5)]{GK17},
\eqref{eq:Phi_kappa_bounds}, \eqref{eq:Phi_kappa_ratio} and~\eqref{eq:Phi_kappa_comparable} that for any
$(t,x)\in(0,\infty)\times(\mathcal{X}\setminus\mathcal{N})$ and any Borel subset $A$ of $B(y_{n},R_{y_{n}}'/4)$,
\begin{align}\notag
P_{x}[X_{t}\in A]
	&\leq\int_{A} \frac {c h\bigl(t\wedge\Psi_{\kappa_{\mathrm{u}}}(R_{y_{n}}')\bigr)^{\kappa_{\mathrm{du}}}}{v\bigl(\Psi^{-1}(t\wedge\Psi_{\kappa_{\mathrm{u}}}(R_{y_{n}}'))\bigr)}
		\exp\Bigl(-\widetilde{\Phi}_{\kappa_{\mathrm{u}}}\bigl(c \bigl(d(x,y)\wedge R_{y_{n}}'\bigr),t\bigr)\Bigr) \, \mu(dy)\\
&\leq\int_{A} \frac {c h\bigl(t\wedge\Psi_{\kappa_{\mathrm{u}}}(R_{y_{n}}')\bigr)^{\kappa_{\mathrm{du}}}}{v\bigl(\Psi^{-1}(t\wedge\Psi_{\kappa_{\mathrm{u}}}(R_{y_{n}}'))\bigr)}
	\exp\Bigl(-c_{2}\Phi_{\kappa_{\mathrm{u}}}\bigl( d(x,y)\wedge R_{y}, t \bigr)\Bigr) \, \mu(dy).
\label{eq:prf:uhk-1}
\end{align}
On the other hand, for any $t\in(0,\infty)$,
$C\Psi(t)\leq\Psi_{\kappa_{\mathrm{u}}}(t)h(\Psi_{\kappa_{\mathrm{u}}}(t))^{\kappa_{\mathrm{u}}}\leq\Psi(t)$
by~\eqref{eq:Psi_kappa_comparable} and~\eqref{eq:def:scalefct}, hence
$h(\Psi_{\kappa_{\mathrm{u}}}(t))\leq h\bigl(C\Psi(t)h(\Psi_{\kappa_{\mathrm{u}}}(t))^{-\kappa_{\mathrm{u}}}\bigr)
	\leq Ch(\Psi_{\kappa_{\mathrm{u}}}(t))^{\varepsilon_{h}\kappa_{\mathrm{u}}}h(\Psi(t))$
by~\eqref{eq:ass_h}, therefore
$h(\Psi_{\kappa_{\mathrm{u}}}(t))\leq Ch(\Psi(t))^{\frac{1}{1-\varepsilon_{h}\kappa_{\mathrm{u}}}}$ and thus
\begin{equation}\label{eq:prf:uhk-2}
h\bigl(t\wedge\Psi_{\kappa_{\mathrm{u}}}(R_{y_{n}}')\bigr)
	\leq Ch\bigl(\Psi(\Psi_{\kappa_{\mathrm{u}}}^{-1}(t))\wedge\Psi(R_{y_{n}}')\bigr)^{\frac{1}{1-\varepsilon_{h}\kappa_{\mathrm{u}}}}
	\leq C'h\bigl(t\wedge\Psi(R_{y})\bigr)^{\frac{1}{1-\varepsilon_{h}\kappa_{\mathrm{u}}}}
\end{equation}
for any $y\in B(y_{n},R_{y_{n}}'/4)$ by $\Psi_{\kappa_{\mathrm{u}}}^{-1}\geq\Psi^{-1}$,
\eqref{eq:ass_h} and~\eqref{eq:def:scalefct}. Also for any
$t\in(0,\infty)$, by $\Psi_{\kappa_{\mathrm{u}}}^{-1}\geq\Psi^{-1}$, \eqref{eq:scale_v},
\eqref{eq:Psi_kappa_comparable} and~\eqref{eq:inv_scalefct} we have
\begin{equation*}
\frac{v(\Psi_{\kappa_{\mathrm{u}}}^{-1}(t))}{v(\Psi^{-1}(t))}
	\leq C_{v}\biggl(\frac{\Psi_{\kappa_{\mathrm{u}}}^{-1}(t)}{\Psi^{-1}(t)}\biggr)^{\alpha_{2}}
	\leq C'\biggl(\frac{\Psi^{-1}(th(t)^{\kappa_{\mathrm{u}}})}{\Psi^{-1}(t)}\biggr)^{\alpha_{2}}
	\leq C''h(t)^{\kappa_{\mathrm{u}}\alpha_{2}/\beta_{1}},
\end{equation*}
and therefore for any $y\in B(y_{n},R_{y_{n}}'/4)$ we further obtain
\begin{equation}\label{eq:prf:uhk-3}
\frac{1}{v\bigl(\Psi^{-1}(t\wedge\Psi_{\kappa_{\mathrm{u}}}(R_{y_{n}}'))\bigr)}
	\leq\frac{C''h\bigl(t\wedge\Psi_{\kappa_{\mathrm{u}}}(R_{y_{n}}')\bigr)^{\kappa_{\mathrm{u}}\alpha_{2}/\beta_{1}}}{v\bigl(\Psi_{\kappa_{\mathrm{u}}}^{-1}(t)\wedge R_{y_{n}}'\bigr)}
	\leq\frac{C'''h\bigl(t\wedge\Psi_{\kappa_{\mathrm{u}}}(R_{y_{n}}')\bigr)^{\kappa_{\mathrm{u}}\alpha_{2}/\beta_{1}}}{v\bigl(\Psi^{-1}(t)\wedge R_{y}\bigr)}
\end{equation}
by $\Psi_{\kappa_{\mathrm{u}}}^{-1}\geq\Psi^{-1}$ and~\eqref{eq:scale_v}. Combining
\eqref{eq:prf:uhk-1}, \eqref{eq:prf:uhk-2} and~\eqref{eq:prf:uhk-3},
for any $(t,x)\in(0,\infty)\times(\mathcal{X}\setminus\mathcal{N})$ and
any Borel subset $A$ of $B(y_{n},R_{y_{n}}'/4)$ we get
\begin{equation}\label{eq:prf:uhk-4}
P_{x}[X_{t}\in A]
	\leq\int_{A} \frac{c_{1} h\bigl(t\wedge\Psi(R_{y})\bigr)^{\kappa_{\mathrm{du}}'}}{v\bigl(\Psi^{-1}(t)\wedge R_{y}\bigr)}
	\exp\Bigl(-c_{2}\Phi_{\kappa_{\mathrm{u}}}\bigl( d(x,y)\wedge R_{y}, t \bigr)\Bigr) \, \mu(dy),
\end{equation}
which then holds for any Borel subset $A$ of $\mathcal{Y}$ since
$\mathcal{Y}=\bigcup_{n\in\mathbb{N}}B(y_{n},R_{y_{n}}'/4)$. Now the assertions
follow from~\eqref{eq:prf:uhk-4} and \cite[Proposition 5.6]{GK17}.
\end{proof}

\subsection{On-diagonal lower bounds of the heat kernel} \label{ssec:dlhk_general}

In the present general setting, our result on on-diagonal lower bounds of the heat kernel
is formulated as in the following proposition.

\begin{proposition} [On-diagonal lower bound] \label{prop:lower-ondiag}
Assume $\mathrm{(V)}_{\leq}$ and $\mathrm{(E)}$, set
$\kappa_{\mathrm{u}}:=(2+\beta_2)(\kappa_{\mathrm{el}}+\kappa_{\mathrm{eu}})$
and assume $\varepsilon_{h}\kappa_{\mathrm{u}}<1$. Also set
$\kappa_{\mathrm{dl}}:=\kappa_{\mathrm{vu}}+\kappa_{\mathrm{u}}\alpha_{2}/\beta_{1}$
and $\kappa_{\mathrm{t}}:=(1-\varepsilon_{h}\kappa_{\mathrm{u}})^{-1}\kappa_{\mathrm{u}}$.
Then there exist $\varepsilon_{\mathrm{t}}\in(0,1)$ and $C_{\mathrm{dl}}\in(0,\infty)$ such that,
for any $(x_{0},r)\in \cY \times \bigl(0,\frac{R}{a_{\rmv}\vee a_{\rme}}\bigr)$ with
$\overline{B}(x_{0},r)$ compact and $B\bigl(x_{0},((a_{\rmv}\vee a_{\rme})+2) r\bigr) \subset \cY$,
and for any $(t,x) \in (0,\varepsilon_{\mathrm{t}}\Psi(r)h(\Psi(r))^{-\kappa_{\mathrm{t}}}] \times (B(x_{0},r/2)\setminus\cN)$
with $P_{x}[X_{t}\in dy,\, t<\tau_{B(x_{0},r)}]=p^{B(x_{0},r)}_{t,x}(y)\,\mu(dy)$
for some Borel measurable function $p^{B(x_{0},r)}_{t,x}:B(x_{0},r)\to[0,\infty)$,
\begin{align}\label{eq:lower-ondiag-L2}
\int_{B(x_{0},r)}p^{B(x_{0},r)}_{t,x}(y)^{2}\,\mu(dy)
	\geq \frac{C_{\mathrm{dl}}}{v(\Psi^{-1}(t))} h(t)^{-\kappa_{\mathrm{dl}}}.
\end{align}
\end{proposition}

We need the following lemma for the proof of Proposition~\ref{prop:lower-ondiag}.

\begin{lemma} \label{lem:lower-ondiag-time}
Let $\kappa\in[0,\varepsilon_{h}^{-1})$ and $\delta\in(0,1)$. Then there exists $\varepsilon\in(0,1)$
such that $t h(t)^{\kappa} \leq \Psi(\frac{\delta}{2}r)$ for any $r,t\in(0,\infty)$
with $t \leq \varepsilon \Psi(r) h(\Psi(r))^{-\kappa/(1-\varepsilon_{h}\kappa)}$.
\end{lemma}

\begin{proof}
Set $\kappa_{\mathrm{t}}:=(1-\varepsilon_{h}\kappa)^{-1}\kappa$.
Let $\varepsilon\in(0,1)$, which we will choose to be sufficiently small later in this proof,
and let $r,t\in(0,\infty)$ satisfy $t \leq \varepsilon \Psi(r) h(\Psi(r))^{-\kappa_{\mathrm{t}}}$.
Then, setting $s:=\Psi(r) h(\Psi(r))^{-\kappa_{\mathrm{t}}} \in (t,\Psi(r)]$, by~\eqref{eq:ass_h} we have
\begin{align} \label{eq:lower-ondiag-time-proof}
sh(s)^{\kappa}
	\leq s \bigl( C_{h} h(\Psi(r)) (\Psi(r)/s)^{\varepsilon_{h}} \bigr)^{\kappa}
	=C_{h}^{\kappa}\Psi(r).
\end{align}
Therefore we see from~\eqref{eq:prf:Phi_kappa-1}, \eqref{eq:prf:Phi_kappa-2},
$t\leq \varepsilon s$, \eqref{eq:lower-ondiag-time-proof} and~\eqref{eq:inv_scalefct} that
\begin{align*}
\Psi^{-1}(th(t)^{\kappa}) \leq C(t/s)^{1/\beta_{2,\kappa}}\Psi^{-1}(sh(s)^{\kappa})
	\leq C\varepsilon^{1/\beta_{2,\kappa}}\Psi^{-1}(C_{h}^{\kappa}\Psi(r))
	\leq C'\varepsilon^{1/\beta_{2,\kappa}}r
\end{align*}
for a constant $C'\in[1,\infty)$ independent of $\varepsilon\in(0,1)$, and hence
the assertion follows by assuming that $\varepsilon$ has been chosen to be
$\varepsilon:=(\delta/(2C'))^{\beta_{2,\kappa}}\in(0,1)$.
\end{proof}

\begin{proof}[Proof of Proposition~\ref{prop:lower-ondiag}]
We follow the standard argument for proving on-diagonal lower bounds as presented, e.g.,
in \cite[Proof of Lemma 5.13]{GT12}. Set $\kappa:=\kappa_{\mathrm{u}}\in[0,\varepsilon_{h}^{-1})$,
let $\delta\in(0,1)$, which we will choose to be sufficiently small later in this proof,
let $\varepsilon\in(0,1)$ be as in Lemma~\ref{lem:lower-ondiag-time} and set
$\varepsilon_{\mathrm{t}}:=\varepsilon$. Let $(x_{0},r),(x,t)$ be as in the statement
and set $\rho := \Psi^{-1}(t h(t)^{\kappa_{\mathrm{u}}})/\delta$, so that $\rho \in (0,\frac{r}{2}]$
by the property of $\varepsilon_{\mathrm{t}}=\varepsilon$ from Lemma~\ref{lem:lower-ondiag-time}.
Then $C\Psi_{\kappa_{\mathrm{u}}}^{-1}(t) \leq \Psi^{-1}(t h(t)^{\kappa_{\mathrm{u}}})=\delta\rho$
by~\eqref{eq:Psi_kappa_comparable}, hence
$t\leq\Psi_{\kappa_{\mathrm{u}}}(C^{-1}\delta\rho)
	\leq C'\delta^{\beta_{\kappa_{\mathrm{u}},2}}\Psi_{\kappa_{\mathrm{u}}}(c_{2}\rho)$
by~\eqref{eq:def:scalefct} for $\Psi_{\kappa_{\mathrm{u}}}$ proved in Lemma~\ref{lem:Phi_kappa},
with $c_{2}$ as in Corollary~\ref{cor:P}, and therefore
\begin{align} \label{eq:lower-ondiag-proof1}
\Phi_{\kappa_{\mathrm{u}}}(c_{2}\rho,t)
	\geq \widetilde{\Phi}_{\kappa_{\mathrm{u}}}(c_{2}\rho,t)
	\geq C'' \min_{j\in\{1,2\}}\biggl(\frac{\Psi_{\kappa_{\mathrm{u}}}(c_{2}\rho)}{t}\biggr)^{\frac{1}{\beta_{\kappa_{\mathrm{u}},j}-1}}
	\geq C'''\delta^{-\frac{\beta_{\kappa_{\mathrm{u}},2}}{\beta_{\kappa_{\mathrm{u}},2}-1}}
\end{align}
by~\eqref{eq:Phi_kappa_comparable} and~\eqref{eq:Phi_kappa_bounds}.
On the other hand, $\rho\leq\frac{r}{2}<\frac{R}{2 a_\rme}$,
the compactness of $\overline{B}(x_{0},r)$ implies that of $\overline{B}(x,\rho)$
by $\overline{B}(x,\rho)\subset\overline{B}(x_{0},r)$, and we also have
$B(x,2(a_\rme + 1)\rho)\subset B(x,(a_\rme + 1)r)\subset B(x_{0},(a_\rme + 2)r) \subset \cY$,
so that by Corollary~\ref{cor:P} and~\eqref{eq:lower-ondiag-proof1} we get
\begin{align} \label{eq:lower-ondiag-proof2}
P_{x}[\tau_{B(x,\rho)}\leq t]
	\leq c_{1}\exp\bigl(-\Phi_{\kappa_{\mathrm{u}}}(c_{2}\rho,t)\bigr)
	\leq c_{1}\exp\Bigl(-C'''\delta^{-\frac{\beta_{\kappa_{\mathrm{u}},2}}{\beta_{\kappa_{\mathrm{u}},2}-1}}\Bigr)
\end{align}
with $c_{1}$ as in Corollary~\ref{cor:P}. Since $c_{1},C'''$ in~\eqref{eq:lower-ondiag-proof2}
are independent of $\delta\in(0,1)$, we may assume that $\delta$ has been chosen to satisfy
$c_{1}\exp\bigl(-C'''\delta^{-\beta_{\kappa_{\mathrm{u}},2}/(\beta_{\kappa_{\mathrm{u}},2}-1)}\bigr) \leq \frac{1}{2}$,
and then by~\eqref{eq:lower-ondiag-proof2} we have
\begin{align} \label{eq:lower-ondiag-proof3}
P_{x}[t<\tau_{B(x,\rho)}]=1-P_{x}[\tau_{B(x,\rho)}\leq t] \geq \frac{1}{2}.
\end{align}
It therefore follows from $B(x,\rho)\subset B(x_{0},r)$ and~\eqref{eq:lower-ondiag-proof3} that
\begin{align}
\int_{B(x_{0},r)}p^{B(x_{0},r)}_{t,x}(y)^{2}\,\mu(dy)
	&\geq \int_{B(x,\rho)}p^{B(x_{0},r)}_{t,x}(y)^{2}\,\mu(dy) \notag \\
&\geq \frac{1}{\mu(B(x,\rho))} \biggl(\int_{B(x,\rho)}p^{B(x_{0},r)}_{t,x}(y)\,\mu(dy)\biggr)^{2} \notag \\
&= \frac{1}{\mu(B(x,\rho))} P_{x}[X_{t}\in B(x,\rho), \, t<\tau_{B(x_{0},r)}]^{2} \notag \\
&\geq \frac{1}{\mu(B(x,\rho))} P_{x}[t<\tau_{B(x,\rho)}]^{2} \notag \\
&\geq \frac{1}{4\mu(B(x,\rho))}.
\label{eq:lower-ondiag-proof4}
\end{align}
Finally, noting that $\Psi^{-1}(t)\leq\Psi^{-1}(t h(t)^{\kappa_{\mathrm{u}}})/\delta = \rho\leq\frac{r}{2}<\frac{R}{a_{\rmv}}$ and that
$B(x,a_{\rmv}\rho)\subset B(x_{0},(a_{\rmv}+1)\frac{r}{2})\subset \cY$, we see from
$\mathrm{(V)}_{\leq}$, \eqref{eq:scale_v} and~\eqref{eq:inv_scalefct} that
\begin{align}
\mu(B(x,\rho)) \leq C_{\mathrm{V}} v(\rho) h(\rho)^{\kappa_{\mathrm{vu}}}
	&\leq C_{\mathrm{V}}C_{v} v(\Psi^{-1}(t)) h(\Psi^{-1}(t))^{\kappa_{\mathrm{vu}}} \Bigl(\frac{\rho}{\Psi^{-1}(t)}\Bigr)^{\alpha_{2}} \notag \\
&\leq C'''' v(\Psi^{-1}(t)) h(\Psi^{-1}(t))^{\kappa_{\mathrm{vu}}} h(t)^{\kappa_{\mathrm{u}}\alpha_{2}/\beta_{1}}.
\label{eq:lower-ondiag-proof5}
\end{align}
Now~\eqref{eq:lower-ondiag-L2} follows from~\eqref{eq:lower-ondiag-proof4}, \eqref{eq:lower-ondiag-proof5}
and the fact that $h( \Psi^{-1}(t) ) \leq h\bigl(C_{\Psi}^{-1/\beta_{2}}\bigr) \leq h\bigl(C_{\Psi}^{-1/\beta_{2}}\bigr) h(t)$
by~\eqref{eq:inv_scalefct} if $t\geq 1$ and
$h(\Psi^{-1}(t)) \leq h\bigl(C_{\Psi}^{-1/\beta_{1}}\Psi^{-1}(1)t^{1/\beta_{1}}\bigr) \leq C'''''h(t)$
by~\eqref{eq:inv_scalefct} and~\eqref{eq:ass_h} if $t\leq 1$.
\end{proof}

\subsection{Proof of Theorem~\ref{thm:heat-kernel-simple}} \label{ssec:proof-hksimple}

We conclude this section with deducing Theorem~\ref{thm:heat-kernel-simple}
from Theorem~\ref{thm:uhk} and Proposition~\ref{prop:lower-ondiag}.

\begin{proof}[Proof of Theorem~\ref{thm:heat-kernel-simple}]
Define $v,\Psi:[0,\infty)\to[0,\infty)$ and $h:(0,\infty]\to[1,\infty)$ by~\eqref{eq:scale-functions-simple},
so that Assumption~\ref{ass:scale-functions} holds with arbitrary $\varepsilon_{h}\in(0,\infty)$
by Example~\ref{ex:scale-functions-simple}. Set $R:=2\diam\cX\in(0,\infty)$, $\cY:=\cX$ and
$\cN:=\emptyset$, which satisfy Assumption~\ref{ass:RYN}.

For (i), since $R_{y} = R\wedge \inf_{z \in \cX \setminus \cY}d(y,z) = 2\diam\cX$ for any $y\in\cX$, by Theorem~\ref{thm:uhk}
there exist a Borel measurable function $\tilde{p}=\tilde{p}_{t}(x,y):(0,\infty)\times\cX\times\cX\to[0,\infty)$
and $c_{1},c_{2}\in(0,\infty)$ such that for each $(t,x)\in(0,(\diam\cX)^{\beta}]\times\cX$, for $\mu$-a.e.\ $y\in\cX$,
\begin{align}\label{eq:heat-kernel-simple-proof1}
p_{t}(x,y)=\tilde{p}_{t}(x,y)
	\leq c_{1}t^{-\alpha/\beta}\bigl(\log(e+t^{-1})\bigr)^{\kappa_{\mathrm{du}}'}
		\exp\Bigl(-c_{2}\Phi_{\kappa_{\mathrm{u}}}\bigl(d(x,y),t\bigr)\Bigr).
\end{align}
Then by the continuity of $p_{t}(x,\cdot)$ and the lower semi-continuity of $\Phi_{\kappa_{\mathrm{u}}}$,
the upper bound on $p_{t}(x,y)$ in~\eqref{eq:heat-kernel-simple-proof1} extends to any
$(t,x,y)\in(0,(\diam\cX)^{\beta}]\times\cX\times\cX$, from which we obtain~\eqref{eq:uhk-simple}
by noting that~\eqref{eq:power_scaling_Phikappa} in Example~\ref{ex:power_scaling} yields
\begin{align*}
\Phi_{\kappa_{\mathrm{u}}}\bigl(d(x,y),t\bigr)
	&\geq c_{\beta}\biggl(\frac{d(x,y)^{\beta}}{t}\biggr)^{\frac{1}{\beta-1}} \biggl(\log\biggl(e+\Bigl(\frac{d(x,y)}{t}\Bigr)^{\frac{\beta}{\beta-1}}\biggr)\biggr)^{-\frac{\kappa_{\mathrm{u}}}{\beta-1}} \\
&\geq c_{\beta}(1-\beta^{-1})^{\frac{\kappa_{\mathrm{u}}}{\beta-1}} \biggl(\frac{d(x,y)^{\beta}}{t}\biggr)^{\frac{1}{\beta-1}} \biggl(\log\biggl(e+\frac{d(x,y)}{t}\biggr)\biggr)^{-\frac{\kappa_{\mathrm{u}}}{\beta-1}}.
\end{align*}

For (ii), note first that for any $(x_{0},r),(x,t)\in\cX\times(0,\infty)$ a function
$p^{B(x_{0},r)}_{t,x}$ as in Proposition~\ref{prop:lower-ondiag} exists and satisfies
$p^{B(x_{0},r)}_{t,x}\leq p_{t}(x,\cdot)$ $\mu$-a.e.\ on $B(x_{0},r)$
by~\eqref{eq:transition-density} and the Radon--Nikodym theorem. Therefore,
choosing any $\varepsilon_{h}\in(0,\infty)$ with $\varepsilon_{h}\kappa_{\mathrm{u}}<1$,
we immediately see that all the assumptions of Proposition~\ref{prop:lower-ondiag} are satisfied and thus that,
with $\kappa_{\mathrm{t}},\varepsilon_{\mathrm{t}},C_{\mathrm{dl}}$ as in Proposition~\ref{prop:lower-ondiag},
\eqref{eq:lower-ondiag-L2} holds for any $(x_{0},r)\in \cX \times (0,a_{\rme}^{-1}\diam\cX)$ and any
$(t,x) \in \bigl(0,\varepsilon_{\mathrm{t}}r^{\beta}\bigl(\log(e+r^{-\beta})\bigr)^{-\kappa_{\mathrm{t}}}\bigr] \times B(x_{0},r/2)$.
Now set $r_{0}:=a_{\rme}^{-1}\diam\cX$,
$t_{0}:=\varepsilon_{\mathrm{t}}r_{0}^{\beta}\bigl(\log(e+r_{0}^{-\beta})\bigr)^{-\kappa_{\mathrm{t}}}$
and let $(t,x)\in(0,t_{0})\times\cX$. Then
$t \in \bigl(0,\varepsilon_{\mathrm{t}}r^{\beta}\bigl(\log(e+r^{-\beta})\bigr)^{-\kappa_{\mathrm{t}}}\bigr]$
for some $r\in (0,r_{0})$, and hence from $p_{t}(x,\cdot)=p_{t}(\cdot,x)$ on $\cX$,
$p_{t}(x,\cdot)\geq p^{B(x,r)}_{t,x}\geq 0$ $\mu$-a.e.\ on $B(x,r)$ and~\eqref{eq:lower-ondiag-L2} we obtain
\begin{align*}
p_{2t}(x,x) &= \int_{\cX}p_{t}(x,y)p_{t}(y,x)\,\mu(dy)
	= \int_{\cX}p_{t}(x,y)^{2}\,\mu(dy)
	\geq \int_{B(x,r)}p_{t}(x,y)^{2}\,\mu(dy) \notag \\
&\geq \int_{B(x,r)}p^{B(x,r)}_{t,x}(y)^{2}\,\mu(dy)
	\geq C_{\mathrm{dl}} t^{-\alpha/\beta} \bigl(\log(e+t^{-1})\bigr)^{-\kappa_{\mathrm{dl}}},
\end{align*}
proving~\eqref{eq:lower-ondiag-simple} for $(t,x)\in(0,2t_{0})\times\cX$. Finally,
\eqref{eq:lower-ondiag-simple} for $(t,x)\in[2t_{0},(\diam\cX)^{\beta}]\times\cX$
follows since, for each $s\in(0,\infty)$, we can define a bounded self-adjoint operator
$T_{s}:L^{2}(\cX,\mu)\to L^{2}(\cX,\mu)$ satisfying $T_{s}\one=\one$ by
$T_{s}f:=\int_{\cX}p_{s}(\cdot,y)f(y)\,\mu(dy)$, then
$\int_{\cX}|T_{s}f|^{2}\,d\mu
	=\int_{\cX}\bigl|T_{s}f-\frac{1}{\mu(\cX)}\int_{\cX}f\,d\mu\bigr|^{2}\,d\mu
		+\frac{1}{\mu(\cX)}\bigl(\int_{\cX}f\,d\mu\bigr)^{2}$
for any $f\in L^{2}(\cX,\mu)$ and hence for any $(t,x)\in(s,\infty)\times\cX$,
\begin{align*}
p_{2t}(x,x)=\int_{\cX}p_{t}(\cdot,x)^{2}\,d\mu
	=\int_{\cX}\bigl|T_{s}(p_{t-s}(\cdot,x))\bigr|^{2}\,d\mu
	\geq\frac{\bigl(\int_{\cX}p_{t-s}(\cdot,x)\,d\mu\bigr)^{2}}{\mu(\cX)}
	=\frac{1}{\mu(\cX)}.
\end{align*}
We have thus completed the proof of Theorem~\ref{thm:heat-kernel-simple}.
\end{proof}

\section{Background on the Brownian map and Liouville quantum gravity} \label{sec:bm-LQG}
\subsection{A review on the Brownian map} \label{ssec:bm-review}

\subsubsection{Definition} \label{sssec:bm-dfn}

The Brownian map is a random metric measure space which was shown by Le Gall \cite{lg2013brownianmap} and Miermont \cite{m2013brownianmap} to arise as the Gromov-Hausdorff scaling limit of certain types of uniformly random planar maps.  The name Brownian map was introduced by Marckert and Mokkadem in \cite{mm2006bm} who proved a weaker form of convergence of rescaled uniformly random quadrangulations.  The standard (unit area) Brownian map is defined from the Brownian snake using the following procedure.  Let $X$ be a normalized Brownian excursion on $[0,1]$.  Let $\CT$ be the instance of the continuum random tree (CRT) \cite{ald1991crt} which is encoded by $X$.  That is, for $0 \leq s < t \leq 1$, we let
\[ m_X(s,t) = X_s + X_t - 2 \inf_{r \in [s,t]} X_r.\]
Then $m_X$ defines a pseudometric on $[0,1]$.  We say that $s \sim t$ if $m_X(s,t) = 0$.  Then $\CT$ is given by the metric quotient of $[0,1]$ with respect to the equivalence relation~$\sim$.  Let $\rho_\CRT \colon [0,1] \to \CT$ be the natural projection map.  Given $X$, we let $Y$ be the mean-zero Gaussian process with
\[ \cov(Y_s,Y_t) = \inf_{r \in [s,t]} X_r \]
so that $\mathbb{E}[(Y_s-Y_t)^2]=m_X(s,t)$. In particular, if $s \sim t$ then $Y_s = Y_t$, so that $Y$ induces a Brownian process on the branches of the CRT instance $\CT$, and is called the \emph{Brownian snake}. We note that $Y$ is a.s.\ $\alpha$-H\"older continuous for any $\alpha < 1/4$.  This follows from the usual Kolmogorov-Centsov argument.

For $s,t \in [0,1]$ with $s < t$ and $[t,s] = [0,1] \setminus (s,t)$, let
\[ d^\circ(s,t) = Y_s + Y_t - 2\max\left(\inf_{r \in [s,t]} Y_r, \inf_{r \in [t,s]} Y_r \right).\]
For $a,b \in \CT$, we set
\[ d_\CT^\circ(a,b) = \min\{ d^\circ(s,t) : \rho_{\CRT}(s) = a, \rho_\CRT(t) = b\}.\]
Finally, for $a,b \in \CT$, we set
\[ d(a,b) = \inf\left\{ \sum_{j=1}^k d_\CT^\circ(a_{j-1},a_j) \right\}\]
where the infimum is over all $k \in \N$ and $a_0 = a,a_1,\ldots,a_k = b$ in $\CT$.  We say that $a \cong b$ if and only if $d(a,b) = 0$.  The Brownian map $(\CS,d)$ is then defined to be the metric quotient $\CT / \cong$.  Let $\rho \colon \CT \to \CS$ be the natural projection map associated with this metric quotient and let $\rho_\BM = \rho \circ \rho_\CRT$.  Let $\nu$ denote the pushforward of Lebesgue measure from $[0,1]$ to $\CS$ by $\rho_\BM$ so that $(\CS,d,\nu)$ is a metric measure space.

The Brownian map is also naturally marked by two points.  The first marked point $x$ is called the \emph{root} and is equal to $\rho_\BM(s^*)$ where $s^*$ is the a.s.\ unique point in $[0,1]$ at which $Y$ attains its infimum \cite[Lemma~16]{mm2006bm} (see also \cite[Proposition~2.5]{lgw2006condbrowniantrees}).  The second marked point $y$ is called the \emph{dual root} and is equal to $\rho_\BM(0) = \rho_\BM(1)$.  The reason for this terminology is that $x$ is the root of the tree of geodesics from every point $z \in \CS$ to $x$ and $y$ is the root of the dual tree, where the tree encoded by $X$ is called the \emph{dual tree} and the tree encoded by $Y$ is called the \emph{geodesic tree} (tree of geodesics from every point back to $x$).
It turns out that the conditional law of $x,y$ given $(\CS,d,\nu)$ is that of independent picks from $\nu$ \cite[Section~8]{lg2010geodesics}.  That is, the law of $(\CS,d,\nu,x,y)$ is invariant under the operation of resampling $x,y$ independently using $\nu$.

We let $\bmlaw{1}$ denote the law of $(\CS,d,\nu,x,y)$.  The superscript $A=1$ serves to emphasize that $\nu(\CS) = 1$ as $X$ is defined on $[0,1]$.  If $a > 0$ is fixed then we define $\bmlaw{a}$ by replacing $X$ with a Brownian excursion of length $a$ and in this case we have that $\nu(\CS) = a$.

As we will see later, in many situations it is more convenient to consider the Brownian map with \emph{random} rather than fixed (unit) area.  The starting point for the definition of the (doubly marked) Brownian map with random area is the infinite measure on Brownian excursions \cite{ry1999mg}.  We remind the reader that one can ``sample'' from this distribution using the following two steps:
\begin{itemize}
\item Pick a lifetime $t$ from the measure $c t^{-3/2} dt$ where $c > 0$ is a constant and $dt$ denotes Lebesgue measure on $\R_+$.
\item Given $t$, pick a Brownian excursion $X$ of length $t$.	
\end{itemize}
Recall that a Brownian excursion $X$ of length $t$ can be constructed by first sampling a Brownian excursion $\wt{X}$ of unit length and then by setting $X_s = \sqrt{t} \wt{X}_{s/t}$.

We define $\bminflaw$ to be the law of $(\CS,d,\nu,x,y)$ where $X$ in the definition of the standard (unit area) Brownian map is replaced by a sample from the infinite measure on Brownian excursions.  Since the law on Brownian excursions is an infinite measure, so is $\bminflaw$.  However, if we condition $\bminflaw$ on $\nu(\CS) = a$ for $a > 0$ then we obtain the probability measure $\bmlaw{a}$.  It is also possible to condition $\bminflaw$ on other events to obtain a probability measure.  Another important example which we will discuss momentarily in more detail is the event $d(x,y) > 1$. More precisely,  it follows from \cite[Theorem~1.1]{ms2015mapmaking} that $\bminflaw[d(x,y) >1] \in (0,\infty)$.  In particular,  $\bminflaw[d(x,y) > 1]^{-1}$ times the restriction of $\bminflaw$ to $\{d(x,y)>1\}$ is a probability measure. In fact,  we have that $\bminflaw[\diam(\CS) >1] \in (0,\infty)$.  Indeed,  as we noted earlier,  the Brownian snake $Y$ used to construct a sample $(\CS,d)$ from $\bmlaw{1}$ is $\alpha$-H\"older continuous for any $\alpha < 1/4$.  Also,  the Kolmogorov-Centsov argument implies that $\wt{Y} = \max_{0 \leq t \leq 1} |Y_t| \in L^p$ for all $p \in [1,\infty)$.  Furthermore we have that $\bminflaw[\diam(\CS)>1] = c \int_{t=0}^{\infty}\bmlaw{t}[\diam(\CS)>1] t^{-3/2}dt$ for some constant $c \in (0,\infty)$ which combined with scaling and change of coordinates implies that $\bminflaw[\diam(\CS)>1] = 4c \int_{s=0}^{\infty}\bmlaw{1}[\diam(\CS)>s] s ds$.  Fix $p>2$.  Since $\diam(\CS)\leq 2\wt{Y}$ by construction and $\bmlaw{1}[\wt{Y} > s] \lesssim s^{1-p}$ by Markov's inequality,  we obtain that $\bminflaw[\diam(\CS) >1] \lesssim 1 + \int_{s=1}^{\infty} s^{1-p} ds \lesssim 1$.

We now record an estimate for the volume of metric balls in the Brownian map.  This result will involve a polynomial correction, which we will improve later to a polylogarithmic correction.

\begin{lemma}
\label{lem:brownian_map_volume_estimates}
Suppose that $(\CS,d,\nu,x,y)$ is an instance of the Brownian map.  For each $u > 0$ there a.s.\ exists $r_0 > 0$ so that
\[ r^{4+u} \leq \nu( B(z,r)) \leq r^{4-u} \quad\text{for all}\quad r \in (0,r_0) \quad\text{and all} \quad z \in \CS.\]
\end{lemma}
\begin{proof}
The upper bound is proved in \cite[Corollary~6.2]{lg2010geodesics}.  The lower bound follows from the H\"older continuity of the Brownian snake.	
\end{proof}

\subsubsection{Breadth-first construction}
\label{subsec:bread_first}

We will now review the basic properties of the breadth-first construction of the Brownian map developed in \cite{ms2015mapmaking}.

\subsubsection*{Continuous state branching processes}
\label{subsubsec:csbp} 
Recall that a \emph{continuous state branching process} (CSBP) with \emph{branching mechanism} $\psi$ is the cadlag Markov process $Y$ on $\R_+$ whose law is characterized by its Laplace transforms
\begin{equation}
\label{eqn:csbp_laplace_transform}
 \E[ \exp(-\lambda Y_t) \giv Y_s] = \exp(-Y_s u_{t-s}(\lambda) ) \quad\text{for}\quad 0 \leq s \leq t
\end{equation}
where
\[ \frac{\partial u_t}{\partial t}(\lambda) = -\psi(u_t(\lambda)) \quad\text{for}\quad u_0(\lambda) = \lambda.\]
See \cite[Chapter~10]{kyp2006levy} or \cite{lg1999spatial} for an introduction to CSBPs.  There is a correspondence between CSBPs and L\'evy processes via the so-called \emph{Lamperti transform} \cite{lam1967csbp}.  In particular, if $Y$ is a L\'evy process stopped upon hitting $(-\infty,0)$ with only upward jumps and Laplace exponent $\psi$ and we define the time-change
\begin{equation}
\label{eqn:lamperti_levy_to_csbp}
s(t) = \inf\{r > 0 : \int_0^r \frac{1}{Y_u} \, du \geq t\}	
\end{equation}
then the process $Y_{s(t)}$ is a CSBP with branching mechanism $\psi$.  Conversely, if $Y$ is a CSBP with branching mechanism $\psi$ and we set
\begin{equation}
\label{eqn:lamperti_csbp_to_levy}
s(t) = \inf\{r > 0 : \int_0^r Y_u \, du \geq t \}
\end{equation}
then the time-changed process $Y_{s(t)}$ is a L\'evy process stopped upon hitting $(-\infty,0)$ with only upward jumps and Laplace exponent $\psi$.

In this work, we will be primarily interested in the case that $\psi(\lambda) = c\lambda^\alpha$ where $\alpha \in (1,2)$ and $c > 0$ is a constant.  We will call the corresponding CSBP an \emph{$\alpha$-stable CSBP} since the associated L\'evy process is $\alpha$-stable.  In this case, we have the explicit formula
\begin{equation}
\label{eqn:u_t_form}
u_t(\lambda) = (\lambda^{1-\alpha} + c t)^{1/(1-\alpha)}.	
\end{equation}
Note that the explicit form of $u_t$ combined with~\eqref{eqn:csbp_laplace_transform} implies that the following is true.  If $Y$ is an $\alpha$-stable CSBP starting from $Y_0 > 0$ then for each $c > 0$ we have that $t \mapsto Y_{c^{\alpha-1} t}$ has the same law as $c Y$, up to a change of starting point.

Let $Y$ be an $\alpha$-stable CSBP and let $\zeta = \inf\{t \geq 0 : Y_t = 0\}$ be its extinction time.  The formula~\eqref{eqn:csbp_laplace_transform} yields an explicit formula for the distribution of $\zeta$, which is as follows
\begin{equation}
\label{eqn:csbp_extinction_time}
\p[ \zeta > t ] = \p[ Y_t > 0 ] = 1 - \lim_{\lambda \to \infty} \E[ e^{-\lambda Y_t} ] = 1-\exp(-c t^{1/(1-\alpha)} Y_0).
\end{equation}

\subsubsection*{Boundary length and conditional independence of inside and outside of filled metric balls}
\label{subsubsec:boundary_length}

Suppose that $(\CS,d,\nu,x,y)$ is distributed according to $\bminflaw$.  For each $r \geq 0$, we let $\fb{x}{r}$ be the complement of the $y$-containing component of $\CS \setminus \ol{B(x,r)}$, where we set $B(x,0)=\emptyset$. We call $\fb{x}{r}$ the \emph{filled metric ball} of radius $r$ centered at~$x$.  In other words, $\fb{x}{r}$ is defined as the closure of the union of $B(x,r)$ together with all of the components of $\CS \setminus \ol{B(x,r)}$ which do not contain $y$.

On the event $\{d(x,y) > r\}$, it is shown in \cite{ms2015mapmaking} how to associate with $\partial \fb{x}{r}$ a \emph{boundary length} $L_r$ in a manner which is measurable with respect to $(\CS,d,\nu,x,y)$.  It turns out that the marginal law on $L_r$ can be ``sampled'' from by first ``picking'' a $3/2$-stable CSBP excursion from the infinite measure on such excursions and then taking the cadlag modification of the time-reversal.
The infinite measure on $3/2$-stable CSBP excursions can be described as follows.  Recall from \cite[Chapter~VIII, Section~4]{bertoin1996levy} that the infinite measure on $3/2$-stable L\'evy excursions with only upward jumps can be sampled from as follows.
\begin{itemize}
\item Pick a lifetime $t$ from the measure $ct^{-5/3} dt$ where $c > 0$ is a constant and $dt$ denotes Lebesgue measure on $\R_+$.
\item Given $t$, pick a $3/2$-stable L\'evy excursion with only upward jumps of time-length $t$.
\end{itemize}
The infinite measure on $3/2$-stable CSBPs can be sampled from by first sampling from the infinite measure on $3/2$-stable L\'evy excursions and then applying the Lamperti transform~\eqref{eqn:lamperti_levy_to_csbp}.

For each $r > 0$, we can view $\fb{x}{r}$ as a metric measure space which is marked by $x$ and equipped with the restriction of $\nu$ to $\fb{x}{r}$.  The metric that we put on $\fb{x}{r}$ is the \emph{interior-internal metric}, which is defined by setting the distance between points $u,v \in \fb{x}{r}$ to be the infimum of the $d$-length of paths which connect $u,v$ and stay in the interior of $\fb{x}{r}$ except possibly at their endpoints.  We can similarly view $\CS \setminus \fb{x}{r}$ as a metric measure space which is marked by $y$ and equipped with the interior-internal metric.

It is shown in \cite{ms2015mapmaking} that $L_r$ is a.s.\ determined by $\fb{x}{r}$, that $L_r$ is a.s.\ determined also by $\CS \setminus \fb{x}{r}$ and, moreover, that $\fb{x}{r}$ and $\CS \setminus \fb{x}{r}$ are conditionally independent given $L_r$.  On the event that $\{d(x,y) > r\}$, the same also holds if we replace $r$ by $s = d(x,y) - r$.  In this case, $\CS \setminus \fb{x}{s}$ is the region which has been explored after performing $r$ units of \emph{reverse metric exploration}.  It is also shown in \cite{ms2015mapmaking} that for each $r > 0$ the metric measure spaces corresponding to $\CS \setminus \ol{B(x,r)}$ are conditionally independent given their boundary lengths.  The conditional law of the components of $\CS \setminus \ol{B(x,r)}$ which do not contain $y$ are given by $\bdisklaw{\ell}$ (where $\ell$ is the hole boundary length) and the conditional law of the component which contains $y$ is given by $\bdisklawweighted{\ell}$ (where $\ell$ is again the hole boundary length);
here $\bdisklaw{\ell}$ and $\bdisklawweighted{\ell}$ denote the law of a Brownian disk with boundary length equal to $\ell$ and that of a Brownian disk weighted by its area, respectively, introduced below at the end of Subsection~\ref{subsec:bread_first}. 
This in fact follows from the conditional independence of $\fb{x}{r}$ and $\CS \setminus \fb{x}{r}$ and a re-rooting argument. These statements hold more generally if $r$ is replaced by a stopping time $\tau$ for the boundary length process of the whole collection of components of $\CS \setminus \ol{B(x,s)}$, $s\geq 0$.

\subsubsection*{Metric bands}

The inside-outside independence of filled metric balls allows us to decompose an instance of the Brownian map into conditionally independent \emph{metric bands}.  More precisely, suppose that we have fixed $r_0 = 0 < r_1 < r_2 < \cdots < r_k$ and we let $s_j = d(x,y) - r_j$ for each $0 \leq j \leq k$.  Then we can view each $\CB_j = \fb{x}{s_{j-1}} \setminus \fb{x}{s_j}$, $1 \leq j \leq k$, as a metric measure space with its interior-internal metric $d_{\CB_j}$ and measure $\nu_{\CB_j} = \nu|_{\CB_j}$, where we set $B(x,r)=\fb{x}{r}=\emptyset$ for $r\leq 0$. Note that each $\CB_j$ is either a topological disk or annulus.  Its \emph{inner} (resp.\ \emph{outer}) boundary is the component of $\partial \CB_j$ whose distance to $x$ is equal to $s_{j-1}$ (resp.\ $s_j$).  We will denote the inner (resp.\ outer) boundary of $\partial \CB_j$ by $\innerboundary \CB_j$ (resp.\ $\outerboundary \CB_j$).  We also note that $\innerboundary \CB_j$ is naturally marked by the point visited by the a.s.\ unique geodesic connecting $x$ and $y$.  The \emph{width} of $\CB_j$ is $s_{j-1} - s_j = r_j - r_{j-1}$.  We note that the independence property for the reverse metric exploration implies that $\CB_j$ is conditionally independent of $\CB_1,\ldots,\CB_{j-1}$ given the boundary length of $\innerboundary \CB_j$.  Moreover, the conditional law of $\CB_j$ given that its inner boundary length is equal to $\ell$ depends only on $\ell$ and the width $r_j - r_{j-1}$.

We let $\bandlaw{\ell}{w}$ be the law on metric bands $(\CB,d_\CB,\nu_\CB,z)$ with inner boundary length $\ell$, width $w$, and marked by a point $z$ on the inner boundary of $\CB$.  We note that if $(\CB,d_\CB,\nu_\CB,z)$ has law $\bandlaw{\ell}{w}$, $a > 0$ and we rescale distances by $a$, boundary lengths by $a^2$, and areas by $a^4$, then we obtain a sample from $\bandlaw{a^2 \ell}{aw}$.

\subsubsection*{Brownian disks and the metric net}

\newcommand{\mnet}{{\mathrm{MetNet}}}

The \emph{metric net} of $(\CS,d,\nu,x,y)$ is defined as $\mnet(\CS):=\bigcup_{r \geq 0} \partial \fb{x}{r}$.  The components of $\CS \setminus \mnet(\CS)$ are each topological disks.  They correspond to the downward jumps of the boundary length process $L_r$ where the magnitude of a given downward jump gives the boundary length of the component. The components are conditionally independent given their boundary lengths and are instances of the \emph{Brownian disk} with the given boundary length (equipped with their interior-internal metric).  For $\ell > 0$, we let $\bdisklaw{\ell}$ denote the law of a Brownian disk with boundary length equal to $\ell$.  We also let $\bdisklawweighted{\ell}$ denote the law of a Brownian disk weighted by its area.  In other words, the Radon-Nikodym derivative of $\bdisklawweighted{\ell}$ with respect to $\bdisklaw{\ell}$ is equal to a normalizing constant times the area of the surface.  If $r > 0$, we have that $\CS \setminus \fb{x}{r}$ has the law $\bdisklawweighted{\ell}$ where $\ell = L_r$.

It turns out that the law of the area of a sample from $\bdisklawweighted{\ell}$ is equal to the law of the amount of time it takes a standard Brownian motion on $\R$ starting from $0$ to hit $-\ell$ \cite{bm2017disk}.  Recall that the density for this law with respect to Lebesgue measure on $\R_+$ at $a$ is given by
\begin{equation}
\label{eqn:bdiskweighted_area_density}
\frac{\ell}{\sqrt{2 \pi a^3}} \exp\left(-\frac{\ell^2}{2a} \right).	
\end{equation}
The law of the area of a sample from $\bdisklaw{\ell}$ thus has density with respect to Lebesgue measure on $\R_+$ at $a$ given by
\begin{equation}
\label{eqn:bdisk_area_density}
\frac{\ell^3}{\sqrt{2\pi a^5}} \exp\left( - \frac{\ell^2}{2a} \right).	
\end{equation}

We note that the Brownian disk can be constructed using a variant of the Brownian snake and this perspective is developed in \cite{bm2017disk}.  We will not describe this construction further here because in what follows we will only need to know that Brownian disks arise as the complementary components when performing a metric exploration of the Brownian map.  The equivalence between these two perspectives was proved in \cite{lg2019disksnake}.

\subsubsection{Explorations of Brownian disks}

Recall from the above that if $(\CS,d,\nu,x,y)$ has distribution $\bminflaw$ and $r > 0$ then the conditional law of $\CS \setminus \fb{x}{r}$ given $L_r$ is equal to $\bdisklawweighted{L_r}$.  This implies that there is a natural exploration of an instance $(\CD,d,\nu,y)$ sampled from $\bdisklawweighted{\ell}$, $\ell > 0$, which is given by considering for each $r \geq 0$ the $y$-containing complementary component $\CD_r$ of the (closed) $r$-neighborhood of $\partial \CD$.  The boundary length $L_r$ of $\partial \CD_r$ is then well-defined and evolves in the same manner as in the case of a metric exploration of the Brownian map.  Moreover, the successive components disconnected from $y$ correspond to the downward jumps of $L_r$ and are conditionally independent Brownian disks with boundary length equal to the size of the corresponding downward jump.  Finally, for each $r \geq 0$, the conditional law given $L_r$ of the metric measure space given by $\CD_r$ and marked by $y$, the restriction of $\nu$ to $\CD_r$, and the interior-internal metric in $\CD_r$ is equal to $\bdisklawweighted{L_r}$.

In the case of $(\CD,d,\nu)$ sampled from $\bdisklaw{\ell}$, there is another exploration which is natural to consider and is called the \emph{center exploration} in \cite{ms2015mapmaking}.  It is analogous to the targeted exploration described just above except that it always continues into the complementary component with the largest boundary length. It can more precisely be constructed as follows.  Suppose that $y \in \CD$ is picked according to $\nu/\nu(\CD)$ independently of everything else.  Then one can consider the metric exploration starting from $\partial \CD$ and targeted at $y$ up until the first time a component $\CD_1$ is disconnected from $y$ with boundary length larger than the $y$-containing component.  At this time, we pick a point $y_1$ in $\CD_1$ from $\nu/\nu(\CD_1)$ independently of everything else and then continue the exploration inside of $\CD_1$ towards $y_1$ until the first time a component $\CD_2$ is disconnected from $y_1$ with boundary length larger than the $y_1$-containing component.
We then continue iterating the above procedure to have a sequence $\{\CD_{j}\}_{j}$ of components inside of which we continue the exploration at each step, until the boundary length of the target component first hits $0$.  We note that this time $R$ is a.s.\ finite since it is at most the diameter of $\CD$, and that a.s.\ the sequence $\{\CD_{j}\}_{j}$ is infinite but only finitely many of them appear by exploration time $r$ for any $r<R$.

Let us now record some important properties of the center exploration.
By the \emph{target component} of the center exploration at exploration time $r$ we mean the component in which the exploration continues after $r$ units of exploration, and each component which remains after the termination of the exploration is called a \emph{component cut off by the center exploration}.
Let $M_r$ denote the boundary length of the target component at exploration time $r$.
Then the downward jumps of $M_r$ correspond to components which are disconnected from the target component.  Moreover, the components are conditionally independent Brownian disks (given the realization of $M$) where the boundary length of the disk is given by the length of the corresponding jump.  Furthermore, the conditional law of the target component given $M_r$ is given by $\bdisklaw{M_r}$.

Let $(a_i)$ denote the sequence of downward jumps made by the center exploration run until $M_r$ hits $0$ and let $\alpha > 0$.  Then it is shown in \cite[Section~4.6]{ms2015mapmaking} that 
\begin{equation}
\label{eqn:center_exploration}
\ell^2 = \E\!\left[ \sum_i a_i^2 \right] \quad\text{and}\quad \ell^\alpha > \E\!\left[ \sum_i a_i^\alpha \right] \quad\text{for}\quad \alpha > 2
\end{equation}
where both expectations are under the law $\bdisklaw{\ell}$.  Moreover, by the scale invariance of the Brownian disk we note that the value of $\E[\sum_i (a_i/\ell)^\alpha]$ does not depend on $\ell$.  The first equality in~\eqref{eqn:center_exploration} can be seen because the conditional expectation of the amount of area inside of $\CD$ given the center exploration up to some time evolves as a martingale and, as explained above, the complementary components are conditionally independent Brownian disks.  The inequality in~\eqref{eqn:center_exploration} can be seen by a direct calculation with the jump law for $M_r$ (which is explicitly identified in \cite{ms2015mapmaking}).

\subsection{Liouville quantum gravity review}

\subsubsection{Basic definitions}

Suppose that $h$ is an instance of (some form of) the Gaussian free field (GFF) on a domain $D \subseteq \C$ (e.g., with Dirichlet or free boundary conditions, defined on the whole-plane, any of the above plus a harmonic function).  The Liouville quantum gravity surface described by $h$ refers to the random two-dimensional Riemannian manifold with metric tensor
\begin{equation}
\label{eqn:metric_tensor}
e^{\gamma h(z)} (dx^2 + dy^2)
\end{equation}
where $\gamma \in (0,2)$ is a parameter and $dx^2 + dy^2$ denotes the Euclidean metric on $D$.  This expression does not make literal sense because $h$ is a distribution and not a function.  The volume form associated with~\eqref{eqn:metric_tensor} was constructed in \cite{ds2011kpz}.  The construction proceeds by letting $h_\epsilon(z)$ be the average of $h$ on $\partial \eball{z}{\epsilon}$ and then taking
\begin{equation}
\label{eqn:limiting_procedure_area}
\qmeasure{h} = \lim_{\epsilon \to 0} \epsilon^{\gamma^2/2} e^{\gamma h_\epsilon(z)} dz
\end{equation}
where $dz$ denotes Lebesgue measure on $D$.  In the case that $D$ has free boundary conditions on a linear segment $L$, one can similarly define a boundary length measure on $L$ by setting
\begin{equation}
\label{eqn:limiting_procedure_boundary}
\qbmeasure{h} = \lim_{\epsilon \to 0} \epsilon^{\gamma^2/4} e^{\gamma h_\epsilon(z)/2} dz
\end{equation}
where $dz$ denotes Lebesgue measure on $L$.  The limiting procedure~\eqref{eqn:limiting_procedure_area} implies that the measures $\qmeasure{h}$ satisfy the following change of coordinates formula.  If $\varphi \colon \wt{D} \to D$ is a conformal transformation and
\begin{equation}
\label{eqn:q_surface_equiv}
\wt{h} = h \circ \varphi + Q \log|\varphi'| \quad\text{where}\quad Q = \frac{2}{\gamma} + \frac{\gamma}{2}
\end{equation}
then $\qmeasure{\wt{h}}(A) = \qmeasure{h}(\varphi(A))$ for all Borel sets $A$.  The same is also true with $\qbmeasure{h}$ in place of $\qmeasure{h}$.  In particular, this gives a way to define $\qbmeasure{h}$ on boundary segments which are not necessarily linear because we can conformally map to such a domain and then compute the boundary length there.

We say that two domain/field pairs $(D,h)$ and $(\wt{D},\wt{h})$ are equivalent as quantum surfaces if $h$, $\wt{h}$ are related as in~\eqref{eqn:q_surface_equiv}.
A \emph{quantum surface} is an equivalence class of domain/field pairs under this equivalence relation.
An \emph{embedding} of a quantum surface is a particular choice of representative, and a quantum surface with an embedding $(D,h)$ for a domain $D\subset\mathbb{C}$ is said to be \emph{parameterized} by $D$.
We will also discuss \emph{marked quantum surfaces} which refer to quantum surfaces with extra marked points and the equivalence relation~\eqref{eqn:q_surface_equiv} is generalized so that
the map $\varphi$ must also take the marked points associated with the surface described by $\wt{h}$ to the marked points associated with the surface described by $h$.
It is not always immediate that two domain/field pairs describe equivalent quantum surfaces, so there is some subtlety to this definition.
For example, two very different looking constructions of the quantum sphere have been given in \cite{dms2014mating} and \cite{dkrv2016sphere} and proved to be equivalent in \cite{ahs2017spheres}.

\subsubsection{Quantum disks, spheres, and wedges} \label{sssec:QD-QS-QW}

We will now describe the construction of a quantum \emph{disk} and a quantum \emph{sphere}, as introduced in \cite{dms2014mating}.
The precise definitions will not be important for what follows, but we include them here for completeness.
Throughout Subsection~\ref{sssec:QD-QS-QW}, $\gamma \in (0,2)$ is arbitrary, and for functions $f,g$ with $L^2$ gradients we define the \emph{Dirichlet inner product} by
\begin{equation}
\label{eqn:dirichlet_inner_product}
(f,g)_\nabla = \frac{1}{2\pi} \int \nabla f(x) \cdot \nabla g(x) \, dx.	
\end{equation}
Let $\| \cdot \|_\nabla$ be the associated norm.

The starting point for the definition of the quantum surfaces which we will discuss is the excursion measure for a Bessel process, which we recall is defined as follows.  Suppose that $\delta \in (-\infty,2)$.  Then one can sample from the measure on Bessel excursions using the following procedure:
\begin{enumerate}[(i)]
\item Pick a lifetime $T$ from the measure $c_\delta T^{\delta/2-2} \, dT$ where $dT$ denotes Lebesgue measure on $\R_+$ and $c_\delta > 0$ is a constant.
\item Sample an independent normalized excursion $\wt{Z} \colon [0,1] \to \R_+$ of a Bessel process of dimension $\delta$.
\item Take $Z$ to be $t \mapsto T^{1/2} \wt{Z}(t/T)$.
\end{enumerate}
We note that the Bessel excursion measure is an infinite measure (as the measure $c_\delta T^{\delta/2-2} \, dT$ is an infinite measure).  It is also defined even for $\delta \leq 0$, although in this case it is not possible to concatenate a Poisson point process of such Bessel excursions to obtain a continuous process.

While reading what follows, the reader may find it helpful to look at \cite[Figure~1.2]{dms2014mating}.

\subsubsection*{Quantum disks}

We will now describe the construction of a \emph{quantum disk}.  This is a finite volume quantum surface which is homeomorphic to the unit disk $\D$ and is naturally marked by two points.  It is easiest to give the definition of the quantum disk when the surface is parameterized by the infinite horizontal strip $\strip = \R \times [0,\pi]$.  We let $\CH(\strip)$ be the closure of the $C^\infty$ functions on $\strip$ with $L^2$ gradient with respect to $\| \cdot \|_\nabla$, viewed modulo additive constant.  Then $\CH(\strip)$ admits the orthogonal decomposition $\CH_1(\strip) \oplus \CH_2(\strip)$ where $\CH_1(\strip)$ (resp.\ $\CH_2(\strip)$) contains those functions in $\CH(\strip)$ which are constant (resp.\ have mean-zero) on vertical lines.

A \emph{quantum disk} is a quantum surface $(\strip,h,-\infty,+\infty)$ whose law can be sampled from using the following steps:
\begin{itemize}
\item Take the projection of $h$ onto $\CH_1(\strip)$ to be given by $\tfrac{2}{\gamma} \log Z$ where $Z$ is a Bessel excursion of dimension $3-\tfrac{4}{\gamma^2}$, reparameterized to have quadratic variation $2\, dt$.
\item Take the projection of $h$ onto $\CH_2(\strip)$ to be independently given by the corresponding projection of a GFF on $\strip$ with free boundary conditions.
\end{itemize}

The above specifies the embedding of the surface modulo one free parameter (since as we will explain below, the points at $\pm \infty$ are marked), which corresponds to the horizontal translation.  There are various ways of fixing the horizontal translation.  One possibility is to let $X$ denote the projection of $h$ onto $\CH_1(\strip)$ and then choose the horizontal translation so that $\sup_{u \in \R} X_u$ is attained at $u=0$.

Since the measure on Bessel excursions of dimension $3-\tfrac{4}{\gamma^2}$ is an infinite measure, this defines an infinite measure on quantum surfaces which we denote by $\qdiskinflaw$.  One can obtain a probability measure by conditioning, for example, on the boundary length $\nu_{h}(\partial\strip)$ or area $\mu_{h}(\strip)$ being equal to a given value.  We let $\qdisk{\ell}$ denote the probability measure obtained when we condition on the boundary length $\nu_{h}(\partial\strip)$ being $\ell > 0$.

The two marked points at $+\infty, -\infty$ are uniformly random given the quantum surface structure \cite[Proposition~A.8]{dms2014mating}.  This means that the law of $(\strip,h,-\infty,+\infty)$ is invariant under the operation of sampling $x,y$ independently from $\qbmeasure{h}$, letting $\varphi \colon \strip \to \strip$ be a conformal transformation which takes $+\infty$ to $x$ and $-\infty$ to $y$, and then replacing $h$ with $h \circ \varphi + Q\log|\varphi'|$.

We will often refer to the law on quantum disks which is obtained by weighting the law $\qdisk{\ell}$ by its total volume $\mu_{h}(\strip)$ and then adding an extra marked point in the interior chosen at random from the quantum measure $\mu_{h}$.  We let $\qdiskweighted{\ell}$ denote this probability measure.

\subsubsection*{Quantum spheres}

We will now describe the construction of the \emph{quantum sphere}.  This is a finite volume surface which is homeomorphic to the two-dimensional sphere $\s^2$ and is naturally marked by two points.  It is easiest to give the definition of the quantum sphere when parameterized by the infinite cylinder $\cyl = \R \times [0,2\pi]$ with the top and bottom identified.  Let $\CH(\cyl)$ be the closure of the $C^\infty(\cyl)$ functions with respect to $\| \cdot \|_\nabla$, defined modulo additive constant.  Then $\CH(\cyl)$ admits the orthogonal decomposition $\CH_1(\cyl) \oplus \CH_2(\cyl)$ where $\CH_1(\cyl)$ (resp.\ $\CH_2(\cyl)$) contains those functions in $\CH(\cyl)$ which are constant (resp.\ have mean-zero) on vertical lines.

Let $h$ be the field on $\cyl$ whose law can be sampled from as follows:
\begin{itemize}
\item Take its projection onto $\CH_1(\cyl)$ to be given by $\tfrac{2}{\gamma} \log Z$ where $Z$ is a Bessel excursion of dimension $4-\tfrac{8}{\gamma^2}$, reparameterized to have quadratic variation $dt$.
\item Take its projection onto $\CH_2(\cyl)$ to be independently given by the corresponding projection of a whole-plane GFF on $\cyl$.
\end{itemize}

As in the case of the quantum disk, this specifies the embedding of the surface modulo one free parameter.  One possible way of fixing it is to let $X$ denote the projection of $h$ onto $\CH_1(\cyl)$ and then choose the horizontal translation so that $\sup_{u \in \R} X_u$ is attained at $u=0$.

We emphasize that since the Bessel excursion measure is an infinite measure, the measure on quantum spheres that we have just defined is also an infinite measure.  We let $\qsphereinflaw$ denote this infinite measure.  One can obtain a probability measure by conditioning on its total volume.

The two marked points at $+\infty,-\infty$ turn out to be uniformly random given the quantum surface structure \cite[Proposition~A.13]{dms2014mating}.  This means that the law of $(\cyl,h,-\infty,+\infty)$ is invariant under the operation of sampling $x,y$ independently from $\qmeasure{h}$, letting $\varphi \colon \cyl \to \cyl$ be a conformal transformation which takes $+\infty$ to $x$ and $-\infty$ to $y$, and then replacing $h$ with $h \circ \varphi + Q\log|\varphi'|$.

\subsubsection*{Quantum wedges}

We will now describe the construction of a \emph{quantum wedge}.  This is a doubly-marked infinite volume surface which is homeomorphic to the upper half-plane $\h$ and is naturally marked by two points (an ``origin'' point and an ``infinity'' point).  Bounded neighborhoods of the origin point a.s.\ have a finite amount of mass and neighborhoods of the infinity point a.s.\ have an infinite amount of mass.  It can be convenient to describe a quantum wedge parameterized either by $\h$ (as this is the setting in which $\SLE$ is easiest to describe) or by $\strip$ (as in this case the field is easiest to describe).

Let us first describe the sampling procedure in the case of $\h$.  We let $\CH(\h)$ be the Hilbert space closure of the $C^\infty$ functions on $\h$ with $L^2$ gradient with respect to $\| \cdot \|_\nabla$, viewed modulo additive constant.  Then we can write $\CH(\h) = \CH_1(\h) \oplus \CH_2(\h)$ where $\CH_1(\h)$ (resp.\ $\CH_2(\h)$) denotes the subspace of functions which are radially symmetric (resp.\ have zero mean on origin-centered semicircles).  The law of the quantum wedge $(\h,h,0,\infty)$ parameterized by $\h$ where $0$ is the origin point and $\infty$ is the infinity point can be sampled from using the following procedure.

First, we define a process $A_s$ as follows.  For $s \geq 0$, we let $A_s = B_{2s} + \gamma s$ where $B$ is a standard Brownian motion with $B_0 = 0$.  For $s \leq 0$, we let $A_s = \wh{B}_{-2s} + \gamma s$ where $\wh{B}$ is a standard Brownian motion independent of $B$ with $\wh{B}_0 = 0$ conditioned so that $\wh{B}_{2u} + (Q - \gamma) u > 0$ for all $u > 0$.
\begin{itemize}
\item We take the projection of $h$ onto $\CH_1(\h)$ to be the function whose common value on $\partial \eball{0}{e^{-s}} \cap \h$ is given by $A_s$ as described above.
\item We take the projection of $h$ onto $\CH_2(\h)$ to be independently given by the corresponding projection of a GFF on $\h$ with free boundary conditions.
\end{itemize}

We note that if we have a quantum wedge parameterized by $\h$, then taking the origin point to be $0$ and the infinity point to be $\infty$ specifies the embedding up to one free parameter which corresponds to the scaling.  In the above construction, we have taken the free parameter so that if we let $r = \sup\{s > 0 : h_s(0) + Q \log s = 0\}$, where $h_s(0)$ denotes the average of $h$ on $\partial \eball{0}{s} \cap \h$, then $r = 1$.  This is sometimes called the \emph{circle average embedding}.  There are other ways of fixing this scaling, but the circle average embedding is often convenient because then the restriction of $h$ to $\D \cap \h$ has the same law as the corresponding restriction to $\D \cap \h$ of a free boundary GFF on $\h$ plus $-\gamma \log|\cdot|$ with the additive constant fixed so that its average on $\partial \D \cap \h$ vanishes.

We can sample from the law of a quantum wedge parameterized by $\strip$ using the following procedure:
\begin{itemize}
\item Take the projection of $h$ onto $\CH_1(\strip)$ to be given by $\tfrac{2}{\gamma} \log Z$, $Z$ a Bessel process of dimension $1 + \tfrac{4}{\gamma^2}$, parameterized to have quadratic variation $2dt$.
\item Take the projection of $h$ onto $\CH_2(\strip)$ to be independently given by the corresponding projection of a free boundary GFF on $\strip$.	
\end{itemize}
We will use the notation $(\strip,h,-\infty,+\infty)$ where $-\infty$ (resp.\ $+\infty$) denotes the origin (resp.\ infinity) point.  This actually specifies the embedding of the surface into $\strip$ modulo horizontal translation.  It is often convenient to take the horizontal translation so that if $X$ denotes the projection of $h$ onto $\CH_1(\strip)$ then $\inf\{ u \in \R : X_u = 0\} = 0$.

We let $\qwedge{2}$ denote the law of a quantum wedge.  The reason for the superscript $W=2$ is that one can in fact consider variants of the quantum wedge which are obtained using a Bessel process of dimension $1 + \tfrac{2}{\gamma^2} W$ and the parameter $W > 0$ is called the \emph{weight}.  In this article, we will only need the $W=2$ quantum wedge so we will not discuss the other variants in further detail.

The weight-$2$ quantum wedge has the special feature that it is invariant under the operation of translating its boundary point by a fixed amount of quantum length.  More precisely, suppose that $(\h,h,0,\infty)$ is a quantum wedge, $L > 0$, and $x >0$ is such that $\qbmeasure{h}([0,x]) = L$.  Then $(\h,h,x,\infty)$ is again a quantum wedge.

\subsubsection{Review of the metric}

A metric for $\sqrt{8/3}$-Liouville quantum gravity which is isometric to the Brownian surfaces is constructed in \cite{ms2015qle1,ms2016qle2,ms2016qle3} using the process $\QLE(8/3,0)$ first defined in \cite{ms2016qledef}.

If we have a quantum surface described by the field $h$ on the domain $D$, we will let $\qdistnoarg{h}$ denote the corresponding metric and $\qball{h}{z}{r}$ the metric ball with respect to $\qdistnoarg{h}$ centered at $z$ of radius $r$.  If we have a marked point $y$, then we will let $\qfb{h}{y}{z}{r}$ denote the filled metric ball relative to $y$ with respect to $\qdistnoarg{h}$.  In other words, $\qfb{h}{y}{z}{r}$ is the complement of the $y$-containing component of $D \setminus \ol{\qball{h}{z}{r}}$.

If $(\cyl,h,-\infty,+\infty)$ is a quantum sphere, then the associated doubly-marked metric measure space $(\cyl,\qdistnoarg{h},\qmeasure{h},-\infty,+\infty)$ has distribution $\bminflaw$.  Moreover, the field $h$ which describes the quantum sphere can be measurably recovered from the metric measure space structure.  Similarly, if $(\strip,h,-\infty,+\infty)$ is a quantum disk then the associated metric measure space $(\strip,\qdistnoarg{h},\qmeasure{h})$ has the law of a Brownian disk.

There are also infinite volume versions of this, whereby the so-called quantum cone is equivalent to the Brownian plane and a quantum wedge is equivalent to the Brownian half-plane.

\subsubsection{$\SLE_6$ explorations of quantum wedges, disks and spheres}
\label{subsec:sle_explorations}

We will now review the basic properties of $\SLE_6$ explorations of quantum disks and spheres when $\gamma=\sqrt{8/3}$.  These results come from \cite{dms2014mating} and \cite{ms2015spheres}.

\begin{figure}[ht!]
\begin{center}
\includegraphics[scale=0.85]{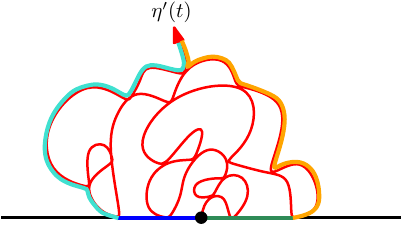}	
\end{center}
\caption{\label{fig:sle6topbottom} Shown is a chordal $\SLE_6$ process $\eta'$ (red) on $\h$ from $0$ to $\infty$ drawn up to time $t$.  The \emph{hull} $K_t$ of $\eta'([0,t])$ is the set of points disconnected from $\infty$ by $\eta'([0,t])$.  The \emph{top} of $K_t$ is the part of $\partial K_t$ in $\h$.  The left (resp.\ right) side of the top of $K_t$ is the part of the top which is to the left (resp.\ right) of $\eta'(t)$ and shown in teal (resp.\ orange).  The \emph{bottom} of $K_t$ is the part of $\partial K_t$ in $\partial \h$.  The left (resp.\ right) side of the bottom of $K_t$ is the part which is to the left (resp.\ right) of $0$ and is shown in blue (resp.\ dark green).  We set $L_t$ to be the quantum length of the top left minus the bottom left and $R_t$ to be the quantum length of the top right minus the bottom right.  The quantum length of the top left is $L_t - \inf_{0 \leq s \leq t} L_s$ and the quantum length of the bottom left is $-\inf_{0 \leq s \leq t} L_s$.}
\end{figure}

Suppose that $(\h,h,0,\infty)$ is a weight-$2$ quantum wedge.  Let $\eta'$ be an independent $\SLE_6$ on $\h$ from $0$ to $\infty$.  We assume that $\eta'$ is parameterized according to quantum natural time.  Recall that this is not the standard capacity parameterization but rather a time parameterization which comes from the quantum surface structure defined by $h$.  It is the continuum analog of the time parameterization of an interface from a statistical mechanics model on a random planar map where the interface is explored one edge at a time.  Suppose that $t \geq 0$ and that $\h_t$ is the unbounded component of $\h \setminus \eta'([0,t])$ and let $K_t = \h \setminus \h_t$ be the associated hull.  The \emph{top} of $K_t$ is $K_t \cap \partial \h_t =\partial K_t \cap \h$ and the \emph{bottom} is $\partial K_t \cap \partial \h$.
We let $L_t=L_t(\eta')$ denote the difference in the quantum length of the part of the top which is to the left of $\eta'(t)$ and the part of the bottom which is to the left of $0$. We similarly let $R_t=R_t(\eta')$ denote the difference in the quantum length of the part of the top which is to the right of $\eta'(t)$ and the part of the bottom which is to the right of $0$.  See Figure~\ref{fig:sle6topbottom} for an illustration of these definitions.  Then $L,R$ evolve as a pair of independent $3/2$-stable L\'evy processes with only downward jumps occurring  whenever $\eta'$ disconnects a bubble from $\infty$ on its left (resp.\ right) side. The magnitude of the downward jump corresponds to the boundary length of the bubble. 

Moreover, for each a.s.\ finite $\{ \mathescr{F}^{\mathrm{W}}_{t+} \}_{ t \in [0,\infty) }$-stopping time $\tau$ (with $\mathescr{F}^{\mathrm{W}}_{t}$ denoting the $\sigma$-field generated by the quantum surfaces disconnected by $\eta'|_{[0,t]}$ from $\infty$
and $\mathescr{F}^{\mathrm{W}}_{t+}:=\bigcap_{s\in(t,\infty)}\mathescr{F}^{\mathrm{W}}_{s}$)
we have that the quantum surface parameterized by the unbounded component of $\h \setminus \eta'([0,\tau])$ and marked by $\eta'(\tau)$ and $\infty$ is a weight-$2$ quantum wedge independent of $\mathescr{F}^{\mathrm{W}}_{\tau+}$.
The quantum surfaces parameterized by the bounded components of $\h \setminus \eta'([0,\tau])$ correspond to the downward jumps of $L|_{[0,\tau]}$ and $R|_{[0,\tau]}$ (depending on whether they are to the left or right of $\eta'$) and are conditionally independent quantum disks given $L|_{[0,\tau]}$, $R|_{[0,\tau]}$ with boundary length given by the size of the corresponding jump.
Here and in what follows, for a quantum surface with an embedding $(D,h)$, we consider each open subset $D_{0}$ of $D$ as parameterizing a quantum surface by equipping $D_{0}$ with the field $h|_{D_{0}}$.

Suppose that $(\CS,x,y)$ is a doubly marked quantum sphere. Let $\eta'$ be a whole-plane $\SLE_6$
from $x$ to $y$ which is sampled independently of $\CS$ and then reparameterized by quantum natural time, and set $t_{\eta'}:=\inf(\eta')^{-1}(y)$, so that $t_{\eta'} < \infty$ a.s.
For $t \in [0,\infty)$,  we let $L_t$ be the quantum boundary length of the connected component of $\CS \setminus \eta'([0,t])$ containing $y$ and let $\mathescr{F}^{\mathrm{S}}_{t}$ denote the $\sigma$-field generated by the quantum surfaces disconnected by
$\eta'|_{[0,t]}$ from $y$, and set $\mathescr{F}^{\mathrm{S}}_{t+} := \bigcap_{ s \in (t,\infty) } \mathescr{F}^{\mathrm{S}}_{s}$.
Then,  \cite[Proposition~6.4]{ms2015spheres} implies that for each $\{ \mathescr{F}^{\mathrm{S}}_{t+} \}_{ t \in [0,\infty) }$-stopping time $\tau$ with $\tau < t_{\eta'}$ a.s.,
conditionally given $\mathescr{F}^{\mathrm{S}}_{\tau+}$, the quantum surface parameterized by the $y$-containing component of $\CS \setminus \eta'([0,\tau])$ has law $\qdiskweighted{L_{\tau}}$ and the location of $\eta'(\tau)$ on its boundary is uniformly random from the quantum boundary measure.
Moreover, the boundary length process $\{ L_{t} \}_{ t \in [0,t_{\eta'}) }$ evolves as the time-reversal of a $3/2$-stable L\'evy excursion.
The same facts apply if $\ell > 0$ is fixed and we explore an instance of $\qdiskweighted{\ell}$ using a radial $\SLE_6$ starting from a uniformly random boundary point according to quantum boundary length measure and targeted at the marked interior point.

In the case of the quantum sphere, the distribution of the amount of time that it takes $\eta'$ to go from $x$ to $y$ is the same as the lifetime distribution for the infinite excursion measure on $3/2$-stable L\'evy excursions with only upward jumps.  Recall from just after~\eqref{eqn:csbp_extinction_time} that this distribution is given by a constant times $t^{-5/3} dt$ where $dt$ denotes Lebesgue measure on $\R_+$.

\subsection{Review of Liouville Brownian Motion}
The Liouville Brownian motion has been constructed in \cite{grv2016lbm, Be15} as the canonical diffusion process under the geometry induced
by the measure $\qmeasure{h}$, where $h$ is a zero-boundary GFF on a planar domain $D \subseteq \mathbb{C}$. It is defined for any choice of the parameter $\gamma\in(0,2)$
as the time change  of the planar Brownian motion on $D$ in terms of the right-continuous inverse of the  positive continuous additive functional (PCAF) associated with $\qmeasure{h}$.

More precisely, let  $B=(B_t)_{t\geq 0}$  be the planar Brownian motion on $D$ defined as the coordinate process on the Wiener space $(C([0,\infty),D), (\mathcal{G}_t)_{t\geq 0}, (P_x)_{x\in D})$ with transition kernel denoted by $q_t(x,y)$, $t>0$, $x,y\in D$.
 In \cite{grv2016lbm}, Garban, Rhodes and Vargas constructed
a PCAF $F=\{ F_t\}_{t\geq 0}$ of  $B$ 
 with the Liouville measure $\qmeasure{h}$ as the associated Revuz measure, that is
\begin{align*}
 \int_{D} f(y) \, d \qmeasure{h}(y)
	\; = \;
	\lim_{t\downarrow 0} \frac 1 t \int_{D} E_x \Bigl[ \int_0^t f(B_s) \, dF_s \Bigr] \, dx
\end{align*}
for any non-negative Borel function $f$ on $D$. 
 Similarly as $\qmeasure{h}$, the PCAF $F$ can be obtained via a regularization procedure from the circle average $h_\epsilon$, i.e.\ for all $x\in D$,
\begin{align*}
F_t \; = \; \lim_{\epsilon\to 0} \int_0^t \epsilon^{\gamma^2/2} \, e^{\gamma h_{\epsilon}(B_s)} \, ds \qquad \text{in $P_x$-probability}
\end{align*}
in  the space $C([0,\infty),\bbR)$ equipped with the topology of uniform
convergence on compact sets (cf.\ \cite[Theorem~2.7]{grv2016lbm}).
 Moreover, for all $x\in D$, $P_x$-a.s., $F$ is strictly increasing and
satisfies $\lim_{t \to \infty} F_t=\infty$.

Then the \emph{($\gamma$-)Liouville Brownian motion} $(X_t)_{t\geq 0}$, abbreviated as \emph{($\gamma$-)LBM}, is defined as $X_t=B_{F^{-1}_t}$.
By the general theory of time
changes of Markov processes we have the following properties of the LBM:
First, it is a recurrent diffusion on $D$ by \cite[Theorems~A.2.12 and 6.2.3]{FOT11}.
Furthermore by \cite[Theorem~6.2.1 (i)]{FOT11} (see also \cite[Theorem~2.18]{grv2016lbm}),
the LBM is $\qmeasure{h}$-symmetric, i.e.\ its transition semigroup $(P_t)_{t>0}$ given by
\begin{align*}
P_t (x,A):=E_x[X_t \in A]
\end{align*}
for $t\in(0,\infty)$, $x\in D$ and a Borel set $A\subset D$,
satisfies
\begin{align*}
 \int_{D} P_t f \cdot g \, d\qmeasure{h} = \int_{D} f \cdot P_t g \, d\qmeasure{h} 
\end{align*}
for all Borel measurable functions $f,g:\  D \rightarrow [0,\infty]$.
By \cite[Theorem~0.4]{grv2014dirichlet} the Liouville semigroup $(P_t)_{t>0}$ is absolutely continuous with respect to the Liouville measure, so there exists the \emph{Liouville heat kernel}
$p=p_{t}(x,y):(0,\infty)\times D \times D \to[0,\infty)$ so that
\begin{align*}
P_t f(x):=E_x[f(X_t)]=\int_{D} p_t(x,y) f(y) \, d\qmeasure{h}(y),
	\qquad x\in D.
\end{align*}
Furthermore, by \cite[Theorem~1.1]{AK16} the Liouville heat kernel admits a jointly continuous version is $(0,\infty)$-valued,
 in particular the LBM is irreducible (cf.\ \cite{mrvz2016heat} for similar results).
Moreover, the transition semigroup $(P_t)_{t>0}$
is strong Feller, i.e.\ $P_t f$ is continuous for any bounded Borel measurable $f: D \rightarrow \bbR$.  We note that the aforementioned properties of the LBM still hold if we replace $h$ by a random field on $D$ whose law is locally absolutely continuous with respect to the law of $h$.  We also note that the fields that we are considering in this paper are locally absolutely continuous with respect to the zero-boundary GFF.

\subsubsection*{The killed Liouville Brownian motion}
Let $U$ be a non-empty open subset  of $D$ and let $U\cup \{ \partial_U \}$
be its one-point compactification. We denote by
$T_U:=\inf \{ s\geq 0: \, B_s \not \in U\}$ the exit time of the
Brownian motion $B$ from $U$ and by
$\tau_U:=\inf \{ s\geq 0: \, X_s \not \in U\}$
that of the LBM $X$, where $\inf\emptyset:=\infty$.
Since by definition $X_t=B_{F^{-1}_t}$, $t\geq 0$, and
$F$ is a homeomorphism on $[0,\infty)$, we have $\tau_U=F_{T_U}$. Let now
$B^U=(B^U_t)_{t\geq 0}$ and $X^U=(X^U_t)_{t\geq 0}$
denote the Brownian motion and the LBM, respectively,
killed upon exiting $U$. That is, they are diffusions on $U$ defined by
\begin{align*}
 B^U_t:= \begin{cases}
B_t & \text{if $t< T_U$,} \\
\partial_U & \text{if $t\geq T_U$,}
 \end{cases}
\qquad
 X^U_t:= \begin{cases}
X_t & \text{if $t< \tau_U$,} \\
\partial_U & \text{if $t\geq \tau_U$.}
 \end{cases}
\end{align*}
Then for $t\in (0,\infty)$, the semigroup operator $P^U_t$ associated with the killed LBM $X^U$ is expressed as  $
 P^U_t f(x):= E_x\bigl[f(X^U_t)\bigr]$, $x\in D$, for each  Borel function $f:U\rightarrow [-\infty, \infty]$ with the convention $f(\partial_U):=0$ for which the expectation exist. By \cite[Proposition~5.1]{AK16} there exists a (unique) jointly continuous function
	$p^U=p^U_t(x,y):(0,\infty)\times U \times U \rightarrow [0,\infty)$
	such that for all $(t,x)\in (0,\infty) \times U$,
	$P_x[X^U_t \in dy]=p^U_t(x,y) \, d\qmeasure{h}(y)$,
	which we refer to as the \emph{Dirichlet Liouville heat kernel} on $U$.
Furthermore, the semigroup operator $P_t^U$ is strong Feller, i.e.\ it maps Borel measurable bounded functions on $U$ to continuous bounded functions on $U$.

If $U$ is bounded, as a time change of $B^U$ the
killed LBM $X^U$ has the same integral kernel for its Green operator
$G^U$ as $B^U$, namely for any non-negative Borel function
$f: U\rightarrow [0,\infty]$ and $x\in D$,
\begin{align} \label{eq:defG_U}
 G^Uf(x):= E_x\Bigl[\int_0^{\tau_U} f(X_t) \, dt \Bigr]
	= E_x\Bigl[\int_0^{T_U} f(B_t) \, dF_t \Bigr]
	= \int_U g_U(x,y) f(y) \, d\qmeasure{h}(y)
\end{align}
Here $g_U$ denotes the Euclidean Green kernel given by
\begin{align} \label{eq:killed_green_U}
 g_U(x,y)=\int_0^{\infty} q^U_t(x,y) \, dt, \qquad x,y\in \bbR^2,
\end{align}
for the heat kernel $q^U_t(x,y)$ of $B^U$:
$q^U=q^U_t(x,y):(0,\infty)\times U\times U\to[0,\infty)$ is the jointly
continuous function such that $P_x[B^U_t\in dy]=q^U_t(x,y)\,dy$ for $t>0$ and
$x\in U$, and we set $q^U_t(x,y):=0$ for $t>0$ and $(x,y)\in(U\times U)^{c}$.  Again we note that the aforementioned properties still hold if we replace $h$ by a a random field whose law is locally absolutely continuous with respect to the law of $h$.
Finally, we recall (see e.g.\ \cite[Example~1.5.1]{FOT11}) that
the Green function $g_{B(x_0,R)}$ over a ball $B(x_0,R)$ is of the form
\begin{align} \label{eq:killed_green}
 g_{B(x_0,R)}(x,y)=\frac 1 \pi \log \frac 1 {|x-y|}+\Psi_{x_0,R}(x,y)
	, \qquad x,y\in B(x_0,R),
\end{align}
for some continuous function $\Psi_{x_0,R}:B(x_0,R)\times B(x_0,R)\to\bbR$.

\section{Quenched estimates on the volume growth} \label{sec:bm-volume}

For an instance $(\CS,d,\nu)$ of the Brownian map, it was shown by Le Gall \cite[Corollary~6.2]{lg2010geodesics} that  $\nu(B(x,r))$ is of order $r^4$ up to a polynomial correction for all $r > 0 $ small enough and every $x$ (recall also Lemma~\ref{lem:brownian_map_volume_estimates}).  We will now improve this estimate to polylogarithmic corrections.

\begin{theorem}
\label{thm:ball_concentration}
Suppose that $(\CS,d,\nu)$ is an instance of the unit area Brownian map.  For each $u > 0$ there a.s.\ exists $r_0 \in (0,1)$ so that
\[ r^4 (\log \tfrac{1}{r})^{-6-u} \leq \nu(B(x,r)) \leq r^4 (\log \tfrac{1}{r})^{8+u}, \qquad \forall r \in (0,r_0), \,  x \in \CS.\]
\end{theorem}

The proofs of the lower and upper bounds will be presented separately in Sections~\ref{subsubsec:vol_lbd} and~\ref{subsubsec:vol_ubd}, respectively.  In what follows, we will use several times the following standard concentration result for Poisson random variables.  Namely, if $Z$ is Poisson with mean $\lambda > 0$ then we have (cf.\ \cite[Lemma~2.9]{ms2016qle2})
\begin{align}
 \p[ Z \leq \alpha \lambda] &\leq e^{  \lambda(\alpha-\alpha\log \alpha-1)}, \qquad \forall \alpha \in (0,1), \label{eqn:poisson_below_mean} \\
 \p[ Z \geq \alpha \lambda] &\leq e^{  \lambda(\alpha-\alpha \log \alpha-1)}, \qquad \forall  \alpha \in(1,\infty). \label{eqn:poisson_above_mean}
\end{align}

\subsection{Lower bound}
\label{subsubsec:vol_lbd}

We will begin by working towards proving the lower bound in Theorem~\ref{thm:ball_concentration}.  The starting point is Lemma~\ref{lem:vol_lower_bound_bl_cond}, which is a pointwise lower bound for the volume in a metric ball when we condition on the event that the filled metric ball boundary length process is not too small.  We will then extend this in Lemma~\ref{lem:vol_lower_bound_dist_cond} when we condition instead on the event that the distance between $x$ and $y$ is at least $1$, from which the lower bound in Theorem~\ref{thm:ball_concentration} easily follows. First,  we will state the following lemma which allows us to compare the laws of a weighted quantum disk and a quantum wedge when both restricted in a small metric neighborhood around the point chosen uniformly according to the quantum boundary length measure.  Its proof is essentially the same as the proof of \cite[Proposition~4.2]{gwynne2017scalinglimituniforminfinite}.

\begin{lemma}\label{lem:weighted_browniab_disk_and_brownian_half_plane}
Fix $\ell,\alpha_1,\alpha_2>0$ with $\alpha_1<\alpha_2$ and let $(\mathbb{H},d)$ be a Brownian disk with area $\alpha$ and boundary length $\ell$ (in the sense of \cite[Section~3.5]{gwynne2017scalinglimituniforminfinite}),  where $\alpha \in [\alpha_1,\alpha_2]$ is fixed.  Let $x \in \partial \mathbb{H}$ be sampled uniformly according to the boundary length measure induced by $(\mathbb{H},d)$.  Let $(\mathbb{H},\wt{d})$ be a Brownian half-plane (equivalently weight-$2$ quantum wedge).  Then,  for each $\epsilon \in (0,1)$ there exists $\wt{\alpha} >0$ depending only on $\epsilon,\ell,\alpha_1$,  and $\alpha_2$,  and a coupling of $(\mathbb{H},d)$ and $(\mathbb{H},\wt{d})$ so that the following is true.  With probability at least $1-\epsilon$,  we have that the metric spaces $B_d(x,\wt{\alpha})$ and $B_{\wt{d}}(0,\wt{\alpha})$ agree in the sense of \cite{gwynne2017scalinglimituniforminfinite}.
\end{lemma}

\begin{proof}
It follows from the argument used to prove \cite[Proposition~4.2]{gwynne2017scalinglimituniforminfinite}.
\end{proof}

\begin{lemma}
\label{lem:vol_lower_bound_bl_cond}
Let $r>0$, and suppose that $(\CS,d,\nu,x,y)$ is distributed according to $\bminflaw$ conditioned so that
if $Y_s$ denotes the boundary length of $\partial \fb{x}{d(x,y)-s}$ then $\sup_{s \geq 0} Y_s\geq r^2$.
There exist constants $c_0, M_0 > 0$ which are independent of $r$ so that
\[ \p[ \nu(B(x,r)) \leq r^4/c^4] \leq c_0 \exp(-M_0 c^{2/3}) \quad\text{for all}\quad c \geq 1.
\]
\end{lemma}
\begin{proof}
Let $c>1$.

\emph{Step~1.} We will first argue that with high probability under $\p$ there exists $s \in [0,r]$ so that the boundary length of $\partial \fb{x}{s}$ is at least $r^2/c^2$, i.e., $\sup_{s \in [0,r]} Y_{(d(x,y)-s)\vee 0} \geq r^2/c^2$ with high probability.  We note that it follows from \cite[Chapter~IV,  Section~4]{bertoin1996levy} that if $\tau$ is the first time that $Y$ enters $[r^2,\infty)$,  then the conditional law of $Y_{t+\tau}$ given $\tau<\infty$ is that of a $3/2$-stable CSBP starting from $Y_{\tau}$.  Also,  it is not difficult to see from the scaling properties of the $3/2$-stable L\'{e}vy process that
the distribution of the supremum under the $3/2$-stable L\'{e}vy excursion measure is given by a constant times $t^{-2}\,dt$;
see, e.g., \cite[Corollary 1]{Rivero2005} for a careful proof.  In particular,  we have that $\bminflaw[\tau<\infty] \in (0,\infty)$.  Next, we define stopping times as follows.
We let $\tau_0 = \inf\{s \geq 0 :\sup_{u\in[0,s]} Y_{u} \geq r^{2}, \  Y_{s} \leq r^2/c^2\}$.
Given that $\tau_0,\ldots,\tau_k$ have been defined, we let $\tau_{k+1} = \inf\{s \geq \tau_k+r/c : Y_{s} \leq r^2/c^2\}$.
Let $K = \max\{ k \in \N : \tau_k < \zeta\}$, where $\zeta=\inf\{s > 0: Y_{s}=0\}$.
Note that $\tau<\tau_0$ $\p$-a.s.\  and that $Y_{\tau_k} \leq r^2 / c^2$ for each $k \geq 1$ $\p$-a.s.\  on the event that $\tau_k < \infty$ while $Y_{\tau_0} = r^2/c^2$ since $Y$ has only upward jumps.  It follows that conditionally on $\tau_k < \infty$ and $Y|_{[0,\tau_k]}$, the random variable $\tau_{k+1}-\tau_k$ is stochastically dominated from above by $\tau_1 - \tau_0$.  Set $\wt{Y}_t = C^{-1} Y_{C^{1/2}t}$ with $C = r^2 / c^2$.  Then the scaling property of $Y$ implies that $\wt{Y}$ has the law of a $3/2$-stable CSBP starting from $1$.  Moreover,  by~\eqref{eqn:csbp_extinction_time},  we have that
\begin{align*}
\p[\tau_{k+1} < \infty \giv \tau_k < \infty] \leq \p[\tau_1 < \infty] \leq \p[\zeta > r/c] = p_0
\end{align*}
where $p_0 \in (0,1)$ is a constant which does not depend on $r$ or $c$.  Therefore we obtain that $K$ is stochastically dominated by a geometric random variable with parameter $p_0$.  It follows that there exists a universal constant $M_0>0$ so that
\begin{equation}
\label{eqn:k_geo_tail}
\p[ K \geq c'] \leq e^{-2 M_0 c'} \quad\text{for all}\quad c' \geq 1.
\end{equation}
As the event that $\sup_{s \in [0,r]} Y_{(d(x,y)-s)\vee 0} \leq r^2/c^2$  and $d(x,y) \geq r$ implies $K \geq \lfloor c- \rfloor$, where $\lfloor c- \rfloor:= \max(\mathbb{Z}\cap(-\infty,c))$, we thus see from~\eqref{eqn:k_geo_tail} that
\begin{equation}
\label{eqn:y_not_too_small1}
\p[\sup_{s \in [0,r]} Y_{(d(x,y)-s)\vee 0} \leq r^2/c^2,\ d(x,y) \geq r] \leq e^{-M_0 c} \quad\text{for all}\quad c > 1.
\end{equation}
We note that $d(x,y) \leq r$ implies that $\sup_{s \in [0,r]} Y_{(d(x,y)-s)\vee 0} = \sup_{s \geq 0} Y_s \geq r^2 > r^2/c^2$.  Combining this with~\eqref{eqn:y_not_too_small1}, we thus have
\begin{equation}
\label{eqn:y_not_too_small}
\p[\sup_{s \in [0,r]} Y_{(d(x,y)-s)\vee 0} \leq r^2/c^2] \leq e^{-M_0 c} \quad\text{for all}\quad c > 1.
\end{equation}

\emph{Step~2.} We emphasize that $Y$ has upward but no downward jumps.  Therefore it hits points at its running infimum continuously.  We now let $\zeta_0= \inf\{s \geq 0 : \sup_{u\in[0,s]}Y_{u}\geq r^{2}, \  Y_s = r^2/c^2\}$ and $\xi_0 = \inf\{s \geq \zeta_0 : Y_s = r^2/(2c^2)\}$.
Given that $\zeta_0,\xi_0,\ldots,\zeta_k,\xi_k$ have been defined, we let $\zeta_{k+1} = \inf\{s \geq \xi_k : Y_s > r^2/c^2\}$ and $\xi_{k+1} = \inf\{s \geq \zeta_{k+1} : Y_s = r^2/(2c^2)\}$.  Let $N = \max\{ k \geq 0 : \zeta_k < \infty\}$.  
Note that $\p[\zeta_{k+1} < \infty \giv \zeta_k<\infty] = 1-\p[X_{\wh{\tau}} = -\frac{1}{2}]$,  where $X$ has the law of a $3/2$-stable L\'evy process starting from $0$ with only upward jumps and $\wh{\tau} = \inf\{t \geq 0 : X_t \notin [-\frac{1}{2},\frac{1}{2}]\}$.  This follows by scaling and the strong Markov property of a $3/2$-stable CSBP,  together with the Lamperti transform.  Also,  \cite[Chapter~VII,  Theorem~8]{bertoin1996levy} implies that $\p[X_{\wh{\tau}} = -\frac{1}{2}] >0$.  In particular,  there exists a universal constant $p_1 \in (0,1)$ so that $N$ is stochastically dominated by a geometric random variable with parameter $p_1$.  Thus,  there exists a universal constant $M_1>0$ such that
\begin{equation}
\label{eqn:n_geo_tail}
\p[ N \geq n ] \leq e^{-M_1 n} \quad\text{for all}\quad n \in \N.
\end{equation}

By~\eqref{eqn:y_not_too_small}, off an event with probability at most $e^{-M_0 c}$, there exists $k\geq 0$ so that
$d(x,y) - r < \xi_k < \infty$ and so that $Y|_{[\xi_k,\xi_{k}']}$, where $\xi_{k}' = \inf\{s \geq \xi_k : Y_s = r^2/(4c^2)\}$,
describes the boundary length evolution of $\partial \fb{x}{s}$ for some interval of $s \in [0,r]$.

\emph{Step~3.} Let $Z$ be a $3/2$-stable L\'evy process with only upward jumps starting from $r^2/(2c^2)$ (run even after hitting $(-\infty,0)$) and let $\xi = \inf\{t \geq 0 : Z_t = r^2/(4 c^2)\}$.
For each $k\geq 0$, by the Lamperti transform, conditionally on the event $\{\zeta_k < \infty\} = \{\xi_k < \infty\}$ we have for each $u > 0$ that the number of upward jumps made by $Y|_{[\xi_k,\xi_{k}']}$ with size at least $u$ is equal in distribution to the number of jumps that $Z|_{[0,\xi]}$ makes of size at least $u$.

Let $t>0$. Recalling the form of the L\'evy measure for $Z$, we know that the number of upward jumps of size at least $u$ made by $Z|_{[0,t r^3/c^3]}$ is Poisson with parameter given by $c_0 (tr^3/c^3) u^{-3/2}$ for a constant $c_0 > 0$.
In particular, if we fix $\lambda \geq 1$ and let $u = (\tfrac{c_0 t r^3}{c^3 \lambda})^{2/3}$ then by~\eqref{eqn:poisson_below_mean} the probability that $Z|_{[0,t r^3/c^3]}$ makes fewer than $\lambda/2$ such jumps is at most
\begin{equation}
\label{eqn:levy_jump_bound}
\exp(-a_0 \lambda) \quad\text{for}\quad a_0 = 1 + \tfrac{1}{2} \log \tfrac{1}{2} - \tfrac{1}{2} > 0.
\end{equation}

Note that $\p[\xi \leq tr^3 / c^3] = \p[T(r^2/(4c^2)) \leq tr^3 / c^3]$,  where $T(r^2 / (4c^2))$ is the first time that a $3/2$-stable L\'evy process with only upward jumps and starting from $0$ hits $-r^2 / (4c^2)$.  Set $q  = c_1 (tr / c)^{-3}$ for $c_1>0$ sufficiently small (to be chosen).  Then,  the proof of \cite[Chapter VII,  Corollary~2]{bertoin1996levy} implies that
\begin{equation}\label{eqn:levy_hit_too_fast}
\p[T(r^2 / (4c^2)) \leq tr^3 / c^3] = \p[\exp(-q T(r^2 / (4c^2))) \geq \exp(-qtr^3 / c^3)] \leq \exp(-a_1 t^{-2}),
\end{equation}
where we choose $c_1$ sufficiently small so that $a_1 = c_1^{2/3}(\frac{1}{4}-c_1^{1/3}) > 0$.

Take $t = \lambda^{-1/2}$ so that $t^{-2} = \lambda$.  By decreasing the value of $a_1 > 0$ if necessary and combining~\eqref{eqn:levy_jump_bound} and~\eqref{eqn:levy_hit_too_fast},
we see for each $k\geq 0$ that conditionally on $\{\zeta_k < \infty\} = \{\xi_k < \infty\}=\{N\geq k\}$ the probability that $Y|_{[\xi_k,\xi_{k}']}$ makes fewer than $\lambda/2$ jumps of size at least $u$ is at most
\begin{equation}
\label{eqn:levy_jump_bound2}
2e^{-a_1 \lambda}.
\end{equation}
Note that with the choice $t = \lambda^{-1/2}$, we have that $u = c_1 \lambda^{-1} r^2/c^2$ where $c_1 = c_0^{2/3}$.

By a union bound and combining~\eqref{eqn:levy_jump_bound2} with~\eqref{eqn:n_geo_tail} we have that, the probability that there exists $0 \leq k \leq N$ such that $Y|_{[\xi_k,\xi_{k}']}$ makes fewer than $\lambda/2$ upward jumps of size at least $u$, is for each $n \in \N$ at most
\begin{equation}
\label{eqn:csbp_jump_bound}
2e^{-a_{1}\lambda}\sum_{k=0}^{n-1} \p[\xi_{k}<\infty] + e^{-M_{1}n}
\leq c_{2}e^{-a_{1}\lambda} + e^{-M_{1}n}
\end{equation}
with $M_1 > 0$ as in~\eqref{eqn:n_geo_tail} and $c_{2}=2(1-e^{-M_{1}})^{-1}$ and is thus at most $c_{2}e^{-a_{1}\lambda}$ by letting $n\to\infty$ in~\eqref{eqn:csbp_jump_bound}.

Combining this with~\eqref{eqn:y_not_too_small} implies that the probability that the metric exploration started from $x$ and targeted at $y$ run for time $r$ disconnects fewer than $\lambda/2$ components with boundary length at least $c_{1}\lambda^{-1}r^2/c^2$ is at most $c_{2} e^{-a_1 \lambda} + e^{-M_0 c}$.  We take $\lambda =(\frac{c_1}{c_3}) c$ for some constants $c_3,c_4>0$ (to be chosen) so that the holes have boundary length at least $c_{3}r^2/c^3$.

\emph{Step~4.}
Next we consider the components with boundary length at least $c_{3}r^2/c^3$,
disconnected by the metric exploration started from $x$ and targeted at $y$ run for time $r$. Then
there exists $p_2 \in (0,1)$ which does not depend on $r,c$ so that conditionally on their boundary lengths
each such component independently has probability at least $p_2$ of having area at least $r^4/c^6$ within distance $r$ from its boundary.  Indeed, each component is a Brownian disk conditionally given its boundary length which is at least $c_3r^2/c^3$.  So the claim becomes that there exists $p_2 \in (0,1)$ which does not depend on $r,c$ such that conditionally on the boundary lengths of the above holes,  each such hole independently has probability at least $p_2$ of having area at least $r^4 / c^6$ within metric distance $r$ from the boundary.  By the scaling properties of Brownian disks this is equivalent to the statement that a Brownian disk with boundary length $\ell \in [c_3,\infty)$ has area at least $1$ in the $c^{3/2}$-neighborhood of its boundary,  which is further equivalent to the statement that a Brownian disk with boundary length $1$ has area at least $\ell^{-2}$ in the $c^{3/2} \ell^{-1/2}$-metric neighborhood of its boundary,  where $c \geq c_4$ and $\ell \geq c_3$.  We will first prove the claim for a sample from $\bdisklawweighted{\ell}$.  Note that we have the disintegration $\bdisklawweighted{1}[\cdot] = \int_{\alpha=0}^{\infty}\diskweighted{\alpha}{1}[\cdot ] s \frac{1}{\sqrt{2\pi \alpha^3}}\exp(-1/(2\alpha))d\alpha$.  Thus it suffices to prove the claim for a sample from $\diskweighted{\alpha}{1}$ where the constant $p_2$ can be made to uniform in $\alpha \in [1,2]$.  Fix $\alpha \in [1,2],  \ell \geq c_3,c \geq c_4$ and let $\wt{\alpha}$ be as in Lemma~\ref{lem:weighted_browniab_disk_and_brownian_half_plane} with $\alpha_1 = 1, \alpha_2 = 2$ and $\epsilon=\frac{1}{2}$.  Suppose that we have a coupling between a sample $(\mathbb{H},h)$ from $\diskweighted{\alpha}{1}$ and a quantum wedge of weight 2 $(\mathbb{H},\wt{h},0,\infty)$ such that with probability at least $\frac{1}{2}$,  the metric spaces $B_{d_h}(w,\wt{\alpha})$ and $B_{d_{\wt{h}}}(0,\wt{\alpha})$ agree in the sense of Lemma~\ref{lem:weighted_browniab_disk_and_brownian_half_plane},  where $w$ is chosen uniformly with respect to the boundary length measure $\nu_h$ on $\mathbb{H}$ and suppose that this event holds.  Hence,  it suffices to prove the corresponding claim for $B_{d_{\wt{h}}}(0,\wt{\alpha})$ with high probability if we take $c_3,c_4$ sufficiently large.  To show this,  first we assume that $c^{3/2} \ell^{-1/2} \leq \wt{\alpha}$ and note that $(\mathbb{H},\wt{h},0,\infty)$ is scale invariant in the sense that $(\mathbb{H},\wt{h},0,\infty)$ and $(\mathbb{H},\wt{h}+\wt{C},0,\infty)$ have the same law for each $\wt{C} \in \mathbb{R}$ (see \cite[Proposition~4.7]{dms2014mating}).  Note that adding a constant $\wt{C}$ to the field scales areas by $e^{\gamma \wt{C}}$ and distances by $e^{\gamma \wt{C} /4}$.  We pick $\wt{C} \in \mathbb{R}$ such that $c^{3/2} \ell^{-1/2} e^{\gamma \wt{C}/4} =1$,  and so the statement for $(\mathbb{H},\wt{h},0,\infty)$ becomes equivalent to the statement that the $1$-metric neighborhood of $0$ has area at least $c^{-6}$.  Note that the LQG metric induces the Euclidean topology on $\overline{\mathbb{H}}$ (see \cite[Theorem~1.3]{hughes2024equivalencemetricgluingconformal}) and that $\mu_{\wt{h}}$ gives positive mass to every open subset of $\mathbb{H}$ a.s.  Hence,  we can pick $c_4>0$ sufficiently large such that with probability at least $2/3$,  we have that the $1$-metric neighborhood of $0$ with respect to $d_{\wt{h}}$ has area at least $c^{-6}$,  which implies that the claim holds for $(\mathbb{H},h)$ with probability at least $1/6$.  Suppose now that $c^{3/2}\ell^{-1/2}\geq \wt{\alpha}>0$.  In that case,  we need to bound from below the probability of the even that the $\wt{\alpha}$-metric neighborhood of $0$ with respect to $d_{\wt{h}}$ has area at least $\ell^{-2}$.  Then it suffices to bound from below the probability that the $\wt{\alpha}$-metric neighborhood of $0$ with respect to $d_{\wt{h}}$ has area at least $c_3^{-2}$.  Therefore,  using the scale invariance of $(\mathbb{H},\wt{h},0,\infty)$ and arguing as before,  we obtain that we can choose $c_3>0$ sufficiently large (depending only on $\wt{\alpha}$) such that the probability of the latter event is at least $2/3$.  It follows that the claim holds for $\diskweighted{\alpha}{1}$ with probability at least $1/6$,  for each $\alpha \in [1,2]$,  and hence the same is true for $\bdisklawweighted{1}$ by possibly taking the lower bound on the probability to be smaller.  Finally,  to deduce the claim for $\bdisklaw{1}$,  we note that 
\begin{align*}
\bdisklawweighted{1} = \Big(\int \nu(\CD)d\bdisklaw{1}\Big)^{-1} \nu(\CD)d\bdisklaw{1}
\end{align*}
and that $\nu(\CD)$ has finite moments of all orders under $\bdisklaw{1}$.  Therefore,  the claim holds by combining with H\"older's inequality.

\emph{Step~5.} Note that the event $\nu(B(x,2r)) \leq r^4/c^6$ implies that either there are fewer than $\lambda/2 = c_1c/(2c_3)$ disks of boundary length at least $c_{3}r^2/c^3$ cut off  by
the metric exploration started from $x$ and targeted at $y$ run for time $r$ or there are at least $\lambda/2$ disks cut off and each of them has area less than $r^4/c^6$ at distance $r$ from its boundary.
As we showed above in Step~3, the former event has probability at most $c_{2} e^{-a_1 \lambda} + e^{-M_0 c}$. Since the disks are conditionally independent given their boundary length, the conditional probability that all of the disks cut off 
have area less than $r^4/c^6$ each at distance $r$ from their boundary given there are at least $\lambda/2$ of them is at most $e^{-a_2 \lambda}$ for a constant $a_2 > 0$.  Altogether, this gives that
\[ \p[ \nu(B(x,2r)) \leq r^4/c^6 ] \leq c_{2}\exp(-a_1 c) + \exp(-a_2 c)  + \exp(-M_0 c).\]
Replacing $c$ with $c^{2/3}$ and $r$ with $r/2$ implies the result. 
\end{proof}

\begin{lemma}
\label{lem:csbp_trace}
Fix $c > 0$ and suppose that $Y$ is a $3/2$-stable CSBP starting from $Y_0 \in [0,c]$.  Let $\zeta = \inf\{r \geq 0 : Y_r = 0\}$.  There exists a constant $c_0 > 0$ so that
\[ \p\big[ \sup_r Y_r \leq c, \zeta \geq T\big] \leq \exp(-c_0 \tfrac{T}{c^{1/2}}) \quad\text{for all}\quad T \geq c^{1/2}.\]
\end{lemma}
\begin{proof}
Let $c = 1$.  Let $\CF_r = \sigma(Y_s : s \leq r)$.  Then there exists $p_0 \in (0,1)$ so that 
\[ \p[ \zeta \leq r+1 \giv \CF_r] \, \one_{\{\zeta \geq r, Y_r \leq 1\}} \geq p_0 \, \one_{\{\zeta \geq r, Y_r \leq 1\}}.\]
Iterating this implies the desired bound for $c=1$.  The general case $c > 0$ follows from the case $c=1$ and the scaling property of $3/2$-stable CSBPs (see just after~\eqref{eqn:u_t_form}).
\end{proof}

\begin{lemma}
\label{lem:vol_lower_bound_dist_cond}
Suppose that $(\CS,d,\nu,x,y)$ is sampled from $\bminflaw$  conditioned on $D = \{ d(x,y) > 1 \}$.
Suppose that $u \in \CS$ is picked independently from $\nu/\nu(\CS)$ and let $M_{r,c} = \{ \nu(B(u,r)) \leq r^4/c^4\}$
for $r\in(0,1)$ and $c\geq 1$.  For each $\rho \in (0,1)$ there exist constants $c_0,M_0,M_1 > 0$ such that  for all  $r \in (0,1)$ and $c \geq 1$,
\[ \p[M_{r,c},\ \nu(B(u,1/8))/\nu(\CS) \leq \rho] \leq c_0 r^{-2} (\log \tfrac{1}{r})^2 \exp(-M_0 c^{2/3})  + \exp(-M_1 (\log \tfrac{1}{r})^2).\]
\end{lemma}

\begin{proof}
In the proof, we will work under the measure $\bminflaw$ and denote it conditioned on an event $E$ with $\bminflaw[E] \in (0,\infty)$ by $\bminflaw[\cdot \giv E]$.  Let $(y_j)$ be an i.i.d.\ sequence in $\CS$ picked from $\nu/\nu(\CS)$.  Let $D_j = \{ d(u,y_j) \geq 1/8\}$.  Let $Y^j$ be the boundary length process associated with the metric exploration from $u$ targeted at $y_j$ and let $E_j$ be the event that $\sup_{s \geq 0} Y_s^j \geq r^2$.  For each $n \in \N$ we have that
\begin{align}
&  \bminflaw\big[ M_{r,c},\ \nu(B(u,1/8))/\nu(\CS) \leq \rho \giv D\big] \nonumber \\
& \mspace{36mu} \leq \bminflaw\bigg[M_{r,c} \cap \bigcup_{j=1}^n D_j \, \Big| \,  D \bigg]
 	  + \bminflaw\bigg[\bigcap_{j=1}^n D_j^c,\ \nu(B(u,1/8))/\nu(\CS) \leq \rho \, \Big| \,  D\bigg] \notag\\
& \mspace{36mu} \leq \sum_{j=1}^n \bminflaw\big[M_{r,c} \cap D_j \cap E_j^c \giv D \big] + \sum_{j=1}^n \bminflaw\big[ M_{r,c} \cap E_j \giv D\big] \notag\\
&\mspace{65mu} + \bminflaw\bigg[\bigcap_{j=1}^n D_j^c,\ \nu(B(u,1/8))/\nu(\CS) \leq \rho \, \Big| \,  D\bigg] . \label{eqn:m_giv_d_bound}
\end{align}
To bound the first term in the right hand side of~\eqref{eqn:m_giv_d_bound}, we let $E_{j,k} = \{e^{-k-1} \leq \sup_{s \geq 0} Y_{s}^j < e^{-k}\}$ and let $\zeta_j$ be the lifetime of $Y^j$. With $c_0 = 1/\bminflaw[D]$, we note that
\begin{equation}
\label{eqn:m_dj_ejc}
\bminflaw\big[ M_{r,c} \cap D_j \cap E_j^c \giv D\big] \leq \bminflaw\big[D_j \cap E_j^c \giv D\big] \leq c_0 \sum_{k=N}^\infty \bminflaw\big[D_j \giv E_{j,k}\big] \, \bminflaw[ E_{j,k}],
\end{equation}
where $N = \log r^{-2}$.  For $\epsilon \in (0,1)$ small,  we set $\tau_{\epsilon,j} = \inf\{s \geq 0 : Y_s^j \geq \epsilon\}$.  Then,  for each $j,k \in \N$,  we have that 
\begin{align*}
&\bminflaw\big[ D_j \cap E_{j,k} \giv \zeta_j \geq 1/100 \big] \\
&= \lim_{\epsilon \to 0} \bminflaw\big[D_j \cap E_{j,k} \cap \{\zeta_j - \tau_{\epsilon,j} \geq 1/8\} \giv \zeta_j \geq 1/100,  \tau_{\epsilon,j} < 1/100,  Y_{\tau_{\epsilon,j}}^j < e^{-k-1} \big].
\end{align*}
Fix $\epsilon \in (0,e^{-k-1})$ small.  Then it holds that
\begin{align*}
&\bminflaw\big[ D_j \cap E_{j,k} \cap \{\zeta_j - \tau_{\epsilon,j} \geq 1/8\} \giv \zeta_j \geq 1/100,  \tau_{\epsilon,j} < 1/100,  Y_{\tau_{\epsilon,j}}^j < e^{-k-1} \big]\\
&\leq \bminflaw\left[ \zeta_j - \tau_{\epsilon,j} \geq 1/8,\  \sup_{s \geq 0} Y_{s+\tau_{\epsilon,j}}^j \leq e^{-k} \giv \zeta_j \geq 1/100,  \tau_{\epsilon,j} < 1/100, Y_{\tau_{\epsilon,j}}^j < e^{-k-1} \right].
\end{align*}
Also, the resampling property of the Brownian disk combined with Lemma~\ref{lem:csbp_trace} imply that
\begin{align*}
&\bminflaw\left[ \zeta_j - \tau_{\epsilon,j} \geq 1/8,\  \sup_{s \geq 0} Y_{s+\tau_{\epsilon,j}}^j \leq e^{-k} \giv \zeta_j \geq 1/100,  \tau_{\epsilon,j} < 1/100,  Y_{\tau_{\epsilon,j}}^j < e^{-k-1}\right]\\
&\leq \p\left[ \sup_{s \geq 0}Y_s^j \leq e^{-k},\  \zeta_j \geq 1/8 \giv \zeta_j \geq 1/100 \right]\\
&\leq \p\big[\zeta_j \geq 1/100 \big]^{-1} \p\big[ \sup_{s \geq 0} Y_s^j \leq e^{-k},\zeta_j \geq 1/8 \big]\\
&\leq c_1 \exp(-c_0 (1/8) e^{k/2})
\end{align*}
with $c_1 = (\p[\zeta_j \geq 1/100])^{-1}$ and $c_0$ the constant in Lemma~\ref{lem:csbp_trace}.  Therefore,  combining with~\eqref{eqn:m_dj_ejc} we obtain that there exist universal constants $a_1,a_2>0$ so that 
\begin{align*}
\bminflaw\big[ M_{r,c} \cap D_j \cap E_j^c \big] \leq a_1 \exp(-a_2 r^{-1}).
\end{align*}

We turn to bound the second term in~\eqref{eqn:m_giv_d_bound}.  We have that
\begin{align}
\label{eqn:m_ej_d}
\bminflaw\big[ M_{r,c} \cap E_j \giv D \big] = c_0 \, \bminflaw\big[ M_{r,c} \cap E_j \cap D\big] \leq c_0 \, \bminflaw[M_{r,c} \giv E_j\big] \,  \bminflaw[E_j].
\end{align}
After possibly decreasing the value of $a_2 > 0$,  Lemma~\ref{lem:vol_lower_bound_bl_cond} gives that  $\bminflaw[M_{r,c} \giv E_j] \leq c_1 \exp(-a_2 c^{2/3})$ for a constant $c_1>0$.
Recall that under $\bminflaw$ we have that $Y^j$ is distributed as a $3/2$-stable CSBP excursion. By the Lamperti transform, the amount of mass that the $3/2$-stable CSBP excursion measure puts on excursions with maximum in a given interval is the same as the amount of mass that the $3/2$-stable L\'evy excursion measure (with only upward jumps) puts on such excursions.
Recall that the distribution of the supremum under the $3/2$-stable L\'{e}vy excursion measure is given by a constant times $t^{-2}\,dt$.
Therefore the same is also true for the $3/2$-stable CSBP excursion measure. Hence we have that $\bminflaw[E_j] = b_0 r^{-2}$ for a constant $b_0 > 0$. Combining, we have that~\eqref{eqn:m_ej_d} is at most a constant times $r^{-2} \exp(-a_2 c^{2/3})$.

The last term in~\eqref{eqn:m_giv_d_bound} is at most $\rho^n$, since on the event
$\{\nu(B(u,1/8))/\nu(\CS) \leq \rho\}$ the $\nu$-probability of $D_{j}^{c}$ is at most $\rho$
and $(D_{j}^{c})_{j}$ are independent. Taking $n = (\log \tfrac{1}{r})^2$ completes the proof.
\end{proof}

\begin{proof}[Proof of Theorem~\ref{thm:ball_concentration}, lower bound]
It suffices to prove the assertion for a sample $(\CS,d,\nu,x,y)$ from $\bminflaw$
conditioned on $\{d(x,y)>1\}$, since $(\CS,\delta d,\delta^{4}\nu,x,y)$ is then a sample
from $\bminflaw$ conditioned on $\{d(x,y)>\delta\}$ for each $\delta\in(0,\infty)$ and $\bminflaw\big[ \nu(\CS) > 0 \giv d(x,y) > 1 \big] = 1$.
Let $(x_j)$ be an i.i.d.\ sequence chosen from $\nu$.  Fix $\rho \in (0,1)$ and let $E_\rho$ be the event that $\sup_{z \in \CS} \nu(B(z,1/8))/\nu(\CS) \leq \rho$.
Note that $\p[E_\rho] \to 1$ as $\rho$ increases to $1$. Fix $u'>u > 0$.  Lemma~\ref{lem:vol_lower_bound_dist_cond} implies that $\p[ \nu(B(x_j,r)) \leq r^4 / (\log \tfrac{1}{r})^{6+u}, \ E_\rho]$ decays to $0$ faster than any power of $r$ for every $j \in \N$.  Moreover,  since $\bminflaw\big[ \nu(\CS) < \infty \giv d(x,y) > 1 \big] = 1$,  it is a.s.\ the case under $\bminflaw\big[ \cdot \giv d(x,y) >1 \big]$ that there exists $r_0 \in (0,1)$ such that $\nu(B(z,r))/\nu(\CS) \geq r^{4+u}$ for each $r \in (0,r_0),  z \in \CS$,  by the corresponding property of the unit area Brownian map.  Suppose that the above holds and pick $z \in \CS$ uniformly sampled from $\nu/\nu(\CS)$ and independent of $(x_j)$.  Set $N = r^{-4-u'}$ and $A = \CS \setminus \cup_{j=1}^N B(x_j,r/2),  \wt{A} = \CS \setminus \cup_{j=1}^N B(x_j,r)$.  Then,  the probability under $\nu/\nu(\CS)$ that $z \notin \cup_{j=1}^N B(x_j,r/2)$ is at most $(1-(r/2)^{4+u})^N \leq \exp(-r^{u-u'}/2^{4+u})$.  In particular,  we have that $\nu(A)/\nu(\CS) \leq \exp(-r^{u-u'}/2^{4+u})$.  Suppose that $\wt{A} \neq \emptyset$ and fix $w \in \wt{A}$.  Then $\nu(B(w,r/2)/\nu(\CS) \geq (r/2)^{4+u}$.  If $B(w,r/2) \subseteq A$,  then $(r/2)^{4+u} \leq \exp(-r^{u-u'}/2^{4+u})$ and that is a contradiction for $r>0$ sufficiently small.  Hence there exists $1 \leq j \leq N$ and $\wt{w} \in B(x_j,r/2) \cap B(w,r/2)$ which implies that $d(w,x_j) <r$,  and that is a contradiction.  Therefore,  we have that $\CS = \cup_{j=1}^N B(x_j,r)$ with high probability if $r>0$ is chosen sufficiently small.  Suppose that both of the events $\{\CS = \cup_{j=1}^N B(x_j,r)\}$ and $E_{\rho} \cap \bigl( \cap_{j=1}^N \{\nu(B(x_j,r)) \geq r^4 \log(r^{-1})^{-6-u}\}\bigr)$ hold.  Then for each $w \in \CS$,  there exists $1 \leq j \leq N$ such that $w \in B(x_j,r)$ and so $B(x_j,r) \subseteq B(w,2r)$.  In particular,  we have that $\nu(B(x,2r)) \geq \nu(B(x_j,r)) \geq r^4 \log(r^{-1})^{-6-u}$.  The proof is then complete since the probability of $E_{\rho} \cap \bigl( \cap_{j=1}^N \{\nu(B(x_j,r)) \geq r^4 \log(r^{-1})^{-6-u}\} \bigr)$ can be made to be arbitrarily close to that of $E_{\rho}$ if $r>0$ is sufficiently small.
\end{proof}

\subsection{Upper bound}
\label{subsubsec:vol_ubd}

We will prove the upper bound in Theorem~\ref{thm:ball_concentration} by dominating the amount of area inside of the metric ball from above using a subcritical Galton-Watson tree with geometric offspring distribution.  The branching structure will come from a metric exploration on the Brownian map or Brownian disk.  We will begin in Lemma~\ref{lem:boundary_length} by bounding the tail of the maximum of the boundary length of a filled metric ball explored up to radius $r$ and then deduce from this in Lemma~\ref{lem:component_boundary_length} an upper bound on the maximum boundary length of any of the complementary components a metric ball explored up to radius $r$.  We will then proceed in Lemma~\ref{lem:num_disks_cut_out} to prove a bound for the amount of area near the boundary of a Brownian disk, which will eventually be used to show that the dominating Galton-Watson tree is subcritical.

\begin{lemma}
\label{lem:boundary_length}
Suppose that $(\CS,d,\nu,x,y)$ has distribution $\bminflaw$. Let $r>0$, and let
$M_r = \sup_{s \in [0,r]} L_s$ where $L_s$ is the boundary length of $\partial \fb{x}{s}$.
There exist constants $c_0,m_0 > 0$ which are independent of $r$ so that
\[ \bminflaw[ M_r \geq c r^2] \leq m_0 e^{-c_0 c} r^{-2} \quad\text{for all}\quad c \geq 1.\]
\end{lemma}
\begin{proof}
Let $Y_t$ be the boundary length of $\partial \fb{x}{s}$ where $s = d(x,y) - t$ so that $Y_t$ evolves as a $3/2$-stable CSBP excursion.  Let $\tau_1 = \inf\{t \geq 0 : Y_t \geq  c r^2\}$ and $\sigma_1 = \inf\{t \geq \tau_1 : Y_t = c r^2/2 \}$.  Given that $\tau_1,\sigma_1,\ldots,\tau_j,\sigma_j$ have been defined, let $\tau_{j+1} = \inf\{t \geq \sigma_j : Y_t \geq c r^2\}$ and $\sigma_{j+1} = \inf\{t \geq \tau_{j+1} : Y_t = c r^2/2 \}$.  Let $J = \max\{ j : \sigma_j < \infty\}$ ($J=0$ if $\sigma_{1}=\infty$) and $\zeta = \inf\{t > 0 : Y_t = 0\}$.  Note that the event that $M_r \geq c r^2$ implies that $J\geq 1$ and that $\zeta - \sigma_J \leq r$.  Therefore it suffices to give an upper bound for $\bminflaw[J\geq 1,\,\zeta - \sigma_J \leq r]$.

By the scaling properties of a $3/2$-stable CSBP (recall just after~\eqref{eqn:u_t_form}), we note that there exists $p_0 \in (0,1)$ which does not depend on $c$ or $r$ so that $\bminflaw[ \tau_{j+1} = \infty \giv \sigma_j < \infty] \geq p_0$.  Therefore $J$ under $\bminflaw[ \cdot \mid \sigma_{1} < \infty ]$ is stochastically dominated by a geometric random variable with parameter $p_0$.  Moreover, by~\eqref{eqn:csbp_extinction_time}, we have for a constant $c_0 > 0$ that
\[ \bminflaw\big[ \zeta - \sigma_j \leq r \giv \sigma_j < \infty\big] = e^{-c_0 c}.\]
Therefore
\begin{align*}
\bminflaw[ J\geq 1,\, \zeta - \sigma_J \leq r] 
&\leq \sum_{j=1}^\infty \bminflaw[ J \geq j,\, \zeta - \sigma_j \leq r ]
 = \sum_{j=1}^\infty e^{- c_0 c} \, \bminflaw[J \geq j]  \\
& = e^{-c_0 c} \, \int J \, d\bminflaw
\leq m_0 e^{-c_0 c}\bminflaw[\sigma_{1}<\infty],
\end{align*}
where $m_0$ is the mean of a geometric random variable with parameter $p_0$.
We note that $\bminflaw[\sigma_1 < \infty] = \bminflaw[\tau_1 < \infty]$ is the same as the measure under the infinite measure on $3/2$-stable L\'evy excursions that the maximum is at least $cr^2$.  This, in turn, is equal to a constant times $c^{-1} r^{-2}$ as shown in the paragraph of~\eqref{eqn:m_ej_d}.
\end{proof}

\begin{lemma}
\label{lem:component_boundary_length}
Suppose that $(\CS,d,\nu,x,y)$ is distributed according to $\bminflaw$.  Fix $u > 0$.  The $\bminflaw$-measure of the event that there exists $s \in (0,r]$ such that some component of $\CS \setminus \overline{B(x,s)}$ has boundary length larger than $r^2 (\log \tfrac{1}{r})^{1+u}$ decays to $0$ as $r \to 0$ faster than any polynomial of $r$.
\end{lemma}
\begin{proof}
Let $\tau$ be the first $s \geq 0$ such that $\CS \setminus \overline{B(x,s)}$ has a component with boundary length at least $r^2 (\log \tfrac{1}{r})^{1+u}$.  On $\tau \leq r$, we let $\CD$ be such a component, breaking ties in a measurable manner, and we let $\ell$ be the boundary length of $\partial \CD$.  Then we know that the conditional law of $\CD$ given $\ell$ is given by $\bdisklaw{\ell}$ if $y \notin \CD$ and is given by $\bdisklawweighted{\ell}$ if $y \in \CD$.
By~\eqref{eqn:bdisk_area_density}, \eqref{eqn:bdiskweighted_area_density}, and the strong Markov property of the metric exploration, conditionally on $\tau \leq r$, the probability that $\nu(\CD) \leq r^4$ decays to $0$ as $r \to 0$ faster than any polynomial of $r$.  
Next,  we note that $0 < \bminflaw\big[ s \leq \tau \leq r \big] \leq \bminflaw\big[ \diam(\CS) \geq s\big] < \infty$ for each $r>0$,  $s \in (0,r)$ and the above imply that $\bminflaw\big[\nu(\CD) \leq r^4 \giv s \leq \tau \leq r \big] \leq \frac{1}{2}$ for each $r>0$ sufficiently small,  uniformly in $s \in (0,r)$.  Thus,  we have that
\begin{align*}
\bminflaw\big[ s \leq \tau \leq r \big] & = \bminflaw\big[ s \leq \tau \leq r,  \nu(\CD) \leq r^4 \big] + \bminflaw\big[ s \leq \tau \leq r,  \nu(\CD) > r^4 \big]\\
&\leq \frac{1}{2} \bminflaw\big[ s \leq \tau \leq r\big] + \bminflaw\big[ s \leq \tau \leq r,  \nu(\CD) > r^4 \big],
\end{align*}
for all $r \in (0,1)$ sufficiently small,  and all $s \in (0,r)$.  Sending $s \to 0$ gives that
\begin{align*}
\bminflaw\big[ \tau \leq r \big] \leq \frac{1}{2} \bminflaw\big[ \tau \leq r \big] + \bminflaw\big[ \tau \leq r,  \nu(\CD) > r^4 \big]
\end{align*}
which implies that $\bminflaw\big[ \tau \leq r \big] \leq 2 \bminflaw\big[ \tau \leq r,  \nu(\CD) > r^4 \big]$ for all $r \in (0,1)$ sufficiently small.
Let $p > 0$ and recall that the $\bminflaw$ mass of the event that $\nu(\CS) \geq r^{-p}$ is $O(r^{p/2})$ (as it is the same as the mass put by the infinite measure on Brownian excursions on those excursions with length at least $r^{-p}$).
On $\tau \leq r$, $\nu(\CS) \leq r^{-p}$ and $\nu(\CD)\geq r^4$, the conditional probability that $y \in \CD$ is at least $r^{4+p}$ as the conditional law of $y$ given $(\CS,d,\nu)$ is $\nu$ normalized to be a probability measure (i.e., $\nu/\nu(\CS)$).  Let $M_r$ be as in Lemma~\ref{lem:boundary_length}.  Altogether, we have shown that
\begin{align*}
   \bminflaw[ \tau \leq r ]
&\leq  2\bminflaw[ \tau \leq r, \nu(\CS) \leq r^{-p}, \nu(\CD)\geq r^4] + 2\bminflaw[ \tau \leq r, \nu(\CS) > r^{-p}] \\
&\leq  2\bminflaw[\tau \leq r, \nu(\CS) \leq r^{-p}, \nu(\CD)\geq r^4] + O(r^{p/2}) \\
&\leq  2r^{-4-p} \bminflaw[\tau \leq r, \nu(\CS) \leq r^{-p}, \nu(\CD)\geq r^4,  y \in \CD] + O(r^{p/2}) \\
&\leq  2 r^{-4-p} \bminflaw[ M_r \geq r^2 (\log \tfrac{1}{r})^{1+u}] + O(r^{p/2}).
\end{align*}
This completes the proof as Lemma~\ref{lem:boundary_length} implies that $\bminflaw[ M_r \geq r^2 (\log \tfrac{1}{r})^{1+u}]$ decays to $0$ as $r \to 0$ faster than any polynomial of $r$ and $p > 0$ was arbitrary.
\end{proof}

\begin{lemma}
\label{lem:num_disks_cut_out}
There exist $c_0 \geq 1$ and $\delta \in (0,1)$ such that the following is true for all $\ell \in (0,1]$ and $c \geq c_0$.  Suppose that $(\CD,d,\nu)$ has law given by $\bdisklaw{\ell}$ conditioned on having area at least~$c$.  Consider the center exploration from $\partial \CD$ run for one unit of time.  Let $p_c$ be the probability that this exploration cuts off a component with area at least $c$.  Then $p_c \leq 1-\delta$.
The same also holds if we instead suppose that $(\CS,d,\nu,x,y)$ has law $\bminflaw$ conditioned on having area at least $c$ and consider the metric exploration starting at $x$ and targeted at $y$.
\end{lemma}

\begin{proof}
First, let $\ell\in(0,1]$ and let $(\CD,d,\nu)$ have law given by $\bdisklaw{\ell}$ conditioned on having area at least $c$.
Let $E_{c}$ denote the event that the center exploration from $\partial \CD$ which is targeted at $y$ and run for one unit of time cuts off a component from $y$ with area at least $c$.
Note that $p_c = \bdisklaw{\ell}[E_c] / \bdisklaw{\ell}[A_c]$ where $A_c = \{\nu(\CD) \geq c\}$.  By~\eqref{eqn:bdisk_area_density}, for each $\epsilon > 0$ there exists $c_0 \geq 1$ so that $c \geq c_0$ implies that
\begin{equation}
\label{eqn:a_c_lbd}
\bdisklaw{\ell}[A_c] \geq (1-\epsilon) \frac{\sqrt{2} \ell^3}{3\sqrt{\pi} c^{3/2}}.
\end{equation}
Let $(a_i)$ be the set of downward jumps made by the center exploration.  It follows from a union bound, \eqref{eqn:bdisk_area_density}, and~\eqref{eqn:center_exploration} (with $\alpha=3$) that there exists $0 < \delta <1$ so that
\begin{equation}
\label{eqn:e_c_ubd}
\bdisklaw{\ell}[E_c] \leq \E\left[ \sum_{i} \frac{\sqrt{2} a_i^3}{3\sqrt{\pi} c^{3/2}} \right] \leq \frac{(1-2\delta) \sqrt{2} \ell^3}{3\sqrt{\pi} c^{3/2}}.
\end{equation}
Combining~\eqref{eqn:a_c_lbd} and~\eqref{eqn:e_c_ubd}, we have that
\[ p_c  = \frac{ \bdisklaw{\ell}[E_c] }{ \bdisklaw{\ell}[A_c] } \leq \frac{1-2\delta}{1-\epsilon}.\]
Thus,  by taking $\epsilon \in (0,1)$ sufficiently small such that $\frac{1-2\delta}{1-\epsilon} \leq 1-\delta$,  we obtain that there exists $c_0 \geq 1$ sufficiently large such that $p_c \leq 1-\delta$,  for each $c \geq c_0$.

Next, let $(\CS,d,\nu,x,y)$ have law $\bminflaw$ conditioned on having area at least $c$.
Let $E_{c}$ denote the event that the metric exploration starting at $x$, targeted at $y$
and run for one unit of time cuts off a component from $y$ with area at least $c$. Then
recalling that $\bminflaw[\cdot]=\int_{0}^{\infty}\bmlaw{a}[\cdot]c_{0}a^{-3/2}\,da$
as explained in Subsection~\ref{sssec:bm-dfn}, we obtain
\begin{align} \label{eq:area_fbx1_decay}
p_{c} = \bminflaw[ E_{c} \mid \nu(\cS)\geq c]
&\leq \bminflaw[ \nu(\fb{x}{1}) \geq c \mid \nu(\cS)\geq c] \\
&= \frac{ \int_{c}^{\infty} \bmlaw{a}[ \nu(\fb{x}{1}) \geq c ] a^{-3/2}\,da }{ \int_{c}^{\infty} a^{-3/2}\,da} \notag\\
&=\frac{1}{2}\int_{1}^{\infty} \bmlaw{b}\bigl[ \nu(\fb{x}{c^{-1/4}}) \geq 1 \bigr] b^{-3/2}\,db
	\xrightarrow{c \to \infty} 0, \notag
\end{align}
proving the assertion for $(\CS,d,\nu,x,y)$ with law $\bminflaw$.
\end{proof}

In the following two lemmas and the completion of the proof of the upper bound of Theorem~\ref{thm:ball_concentration} given below, we will consider the following exploration.  Suppose that $r\in(0,1/e]$, $c > 1$ and $u > 0$.  Fix $\ell \in [0, r^2 (\log r^{-1})^{3+4u}] $.  If $\ell > 0$, we suppose that we have a Brownian disk $(\CD,d,\nu)$ with boundary length $\ell$ conditioned on having area at least $c r^4 (\log r^{-1})^{6+8u}$.   Consider the center exploration from $\partial \CD$.  Let $\tau_1$ be the first time that the center exploration has cut off a component with area at least $c r^4 (\log r^{-1})^{6+8u}$.  Given that $\tau_1,\ldots,\tau_k$ have been defined, let $\tau_{k+1}$ be the first time after $\tau_k$ that the center exploration cuts off another component with area at least $c r^4 (\log r^{-1})^{6+8u}$.  If $\ell = 0$, we suppose that we have a Brownian map instance $(\CS,d,\nu,x,y)$ conditioned on having area at least $c r^4 (\log r^{-1})^{6+8u}$ and define the exploration analogously except starting at $x$ and targeted at $y$.

\begin{lemma}
\label{lem:num_components}
Suppose that we have the setup described just above.  Let
\[
E_r =\Bigl\{\sup_{s\in[0,r]}L_{s}\leq r^2 (\log r^{-1})^{3+4u}\Bigr\}
\]
denote the event that the boundary length $L_{s}$ of the metric exploration is at most $r^2 (\log r^{-1})^{3+4u}$ for all $s \in [0,r]$.  There exist $c_0 > 1$ and $\delta\in (0,1)$ which are independent of $r,u,\ell$ so that for all $c \geq c_0$ and $r\in(0,1/e]$ the following is true. With $N_{r}=\sup\{k : \tau_k \leq r\}$ ($\sup\emptyset:=0$) we have that the probability of $E_{r}\cap\{N_{r} \geq n\}$ is at most $(1-\delta)\wt{p}_c^{n-1}$ where $\wt{p}_c \in (0,\infty)$ depends only on $c$ and satisfies $\wt{p}_c \to 0$ as $c \to \infty$.
\end{lemma}
\begin{proof}
Recall that by rescaling
boundary lengths by $( r^2 (\log r^{-1})^{3+4u})^{-1}$, areas by $(r^4 (\log r^{-1})^{6+8u})^{-1}$
and distances by $(r (\log r^{-1})^{3/2+2u})^{-1/2}$, we obtain a sample from the law $\bdisklaw{\ell_{r}}$
with $\ell_{r}$ given by $\ell_{r}=\ell( r^2 (\log r^{-1})^{3+4u})^{-1} \in (0,1]$ conditioned on having area at least $c$ if $\ell>0$,
and a sample from the law $\bminflaw$ conditioned on having area at least $c$ if $\ell=0$.
Under this rescaling, the event $E_{r}$ and the random variable $N_{r}$ become
\begin{align} \label{eq:E_r-N_r-rescaled}
E_{r}=\Bigl\{\sup_{s\in[0,r_{0}]}L_{s}\leq 1\Bigr\}
\qquad\text{and}\qquad
N_{r}=\sup\{k: \tau_{k} \leq r_{0}\},
\end{align}
respectively, where $r_{0}=(\log r^{-1})^{-3/2-2u}$, and we are to prove the assertion for these $E_{r}$ and $N_{r}$ instead.

Set $E_{k}=\{\sup_{s\in[0,\tau_{k}]}L_{s}\leq 1\}$ for each $k\geq 1$. For each $t\in[0,\infty)$,
let $\mathcal{F}_{t}$ be the $\sigma$-algebra generated by $\{L_{r}\}_{r\in[0,t]}$ and the components cut off
by the exploration up to time $t$. Then for each $k\geq 1$, a.s.\ on the event $\{\tau_{k}<\infty\}$,
the conditional law given $\mathcal{F}_{\tau_{k}}$ of the target component at time $\tau_{k}$
is $\bdisklaw{L_{\tau_{k}}}$ if $\ell>0$ and $\bdisklawweighted{L_{\tau_{k}}}$ if $\ell=0$,
and therefore by~\eqref{eqn:bdiskweighted_area_density} and~\eqref{eqn:bdisk_area_density} we have, a.s.,
\begin{align}
\prob[ \tau_{k+1} \leq r_{0},\, E_r \mid \mathcal{F}_{\tau_k} ]
& \leq
\prob[ \tau_{k+1} \leq r_{0},\, E_{k+1} \mid \mathcal{F}_{\tau_k} ]
\leq
\prob[ \tau_{k+1} \leq r_{0},\, E_{k} \mid \mathcal{F}_{\tau_k} ] \notag \\
&\leq
\prob[ \nu(\text{the target component at $\tau_k$}) \geq c \mid \mathcal{F}_{\tau_k}] \, \indicator_{\{\tau_{k} \leq r_{0}\}\cap E_k} \notag \\
&\leq
\begin{cases}
	\bdisklaw{L_{\tau_{k}}}[ \nu( \CD ) \geq c ] \, \indicator_{\{\tau_{k} \leq r_{0}\}\cap E_k} & \text{if $\ell>0$,} \\
	\bdisklawweighted{L_{\tau_{k}}}[ \nu( \CD ) \geq c ] \, \indicator_{\{\tau_{k} \leq r_{0}\}\cap E_k} & \text{if $\ell=0$,}
\end{cases} \notag \\
& \leq
\wt{p}_c \, \indicator_{\{\tau_{k} \leq r_{0}\}\cap E_k},
\label{eq:center-exp-use-strong-Markov}
\end{align}
where $\wt{p}_c = c^{-1/2}\vee c^{-3/2}$. Therefore, taking $c_{0}\in[1,\infty)$ and
$\delta\in(0,1)$ as in Lemma~\ref{lem:num_disks_cut_out} and applying
\eqref{eq:center-exp-use-strong-Markov} and Lemma~\ref{lem:num_disks_cut_out},
for any $n\geq 1$ and any $c\geq c_{0}$ we get
\begin{align*}
\prob[N_{r}\geq n, \, E_{r}]
=
\prob[\tau_{n}\leq r_{0}, \, E_{r}]
&\leq
\prob[\tau_{n}\leq r_{0}, \, E_{n}]
\leq
\wt{p}_c^{n-1} \prob[\tau_{1}\leq r_{0}, \, E_{1}] \notag \\
&\leq
\wt{p}_c^{n-1} \prob[\tau_{1}\leq r_{0}]
\leq
\wt{p}_c^{n-1}(1-\delta),
\end{align*}
completing the proof.
\end{proof}

\begin{lemma}
\label{lem:area_in_rest}
Suppose we have the same setup as in Lemma~\ref{lem:num_components}. Let $A_{r}$ denote the sum of the areas of the components which are cut off by the exploration within exploration time $r$ and have area at most $cr^{4}(\log r^{-1})^{6+8u}$ each.
The probability of the event $E_r \cap \{ A_{r} >r^4 (\log r^{-1})^{7+12u} \}$ decays to $0$ as $r \to 0$ with a decay rate independent of $\ell$ and faster than any polynomial of $r$.

\end{lemma}
\begin{proof}
Set
\[ R = r^3 (\log r^{-1})^{3+4u} \quad \text{and} \quad R' = r^{24} (\log r^{-1})^{36+54u}.\]

\emph{Step~1.} Let $Z$ be a $3/2$-stable L\'evy process with only upward jumps and let $\Lambda = \{(t,u)\}$ be the set of pairs consisting of the jump times and sizes for $Z$.  That is, $(t,u) \in \Lambda$ if and only if $u = Z_t - Z_{t-} > 0$.  Then $\Lambda$ is a Poisson point process with intensity measure given by a constant times $u^{-5/2} \,du \, dt$ where $du,dt$ both denote Lebesgue measure on $\R_+$.  Let $f$ denote the density function for the area of a sample from $\bdisklaw{1}$ as given in~\eqref{eqn:bdisk_area_density}.  We associate with each upward jump of $Z$ an independent random variable $a$ with density given by $f$.  Then $\Lambda_0 = \{(t,u,a)\}$ is a Poisson point process with intensity measure given by a constant times $dt \otimes u^{-5/2} du \otimes f(a) da$ where $dt$, $du$, and $da$ all denote Lebesgue measure on $\R_+$.  Thus \cite[Lemma~4.19]{dms2014mating} implies that $\wt{\Lambda} = \{(t,u^2 a)\} = \{ (t,v)\}$ is a Poisson point process with intensity measure given by a constant times $dt \otimes v^{-7/4} dv$ where again $dt$ and $dv$ denote Lebesgue measure on $\R_+$.  Note that $u^2 a$ is equal in distribution to the area of a Brownian disk with boundary length $u$.

Fix $S \in (0,R')$.  For $k \in \Z$, the number of elements $(t,v) \in \wt{\Lambda}$ with $t \in [1-S,1+S]$ and $v \in (2^{-k-1},2^{-k}]$ is distributed as a Poisson random variable with mean $m_k$ given by a constant times $S 2^{3k/4}$.  Let $k_0 \in \Z$ be the smallest $k \in \Z$ so that $S^{-1/2} m_k \geq 1/2$.
For $k \geq k_0$, \eqref{eqn:poisson_above_mean} with $\alpha = (\log \tfrac{1}{r})^{1+u} S^{-1}$ implies that the probability that there are more than $(\log \tfrac{1}{r})^{1+u} S^{-1} m_k$ such elements is at most $\exp(-(\log \tfrac{1}{r})^{1+u} S^{-1} m_k)$, which decays to $0$ as $S \to 0$ faster than any power of $S$ when $r > 0$ is small enough and fixed.
It likewise decays to $0$ faster than any power of $r$ as $r \to 0$ provided $S \in (0,R')$.  Let $k_1$ be the smallest $k$ so that $2^{-k} \leq r^4 (\log \frac{1}{r})^{6+9u}$.
By~\eqref{eqn:poisson_above_mean} with $\alpha = (\log \frac{1}{r})^{1+u}/m_k$, for each $p > 0$ there exists $r_0 > 0$ so that for all $r \in (0,r_0)$ and for each $k \in [k_1,k_0]\cap\mathbb{Z}$, the probability that we have more than $(\log \tfrac{1}{r})^{1+u}$ elements $(t,v) \in \wt{\Lambda}$ with $t \in [1-S,1+S]$ and $v \in (2^{-k-1},2^{-k}]$ is $O(S^p)$ as $S \to 0$.  Likewise, it tends to $0$ as $r \to 0$ faster than any power of $r$ provided $S \in (0,R')$. 
Altogether, we see that for each $p > 0$ there exists $r_0 > 0$ so that for all $r \in (0,r_0)$ the probability that $\sum_{(t,v) \in \wt{\Lambda} : t \in [1-S,1+S], \, v \leq cr^{4}(\log r^{-1})^{6+8u} } v$ exceeds $r^4 (\log \tfrac{1}{r})^{7+11u}$ is $O(S^p)$ as $S \to 0$. Likewise, it tends to $0$ as $r \to 0$ faster than any power of $r$ provided $S \in (0,R')$.
We also see that for each $q \in [0,\infty)$, the probability that $\sum_{(t,v) \in \wt{\Lambda} : t \in [(q-4R)\vee 0,q+4R], \, v \leq cr^{4}(\log r^{-1})^{6+8u} } v$ exceeds $r^4 (\log \tfrac{1}{r})^{7+11u}$ tends to $0$ as $r \to 0$ faster than any power of $r$,
by the argument in this paragraph with $m_k$ a constant times $R 2^{3k/4}$, $k_0 = \min\{k\in \mathbb{Z}: m_k \geq 1/2\}$ and~\eqref{eqn:poisson_above_mean} used with $\alpha = (\log \frac{1}{r})^{1+u}$ for $k\geq k_0$.

\emph{Step~2.}
We are now going to transfer the result of Step~1 about $3/2$-stable L\'evy processes to the setting of $3/2$-stable L\'evy excursions.  Let $Z$ be a $3/2$-stable L\'evy process and set $S_t = \sup\{0 \vee Z_s : 0 \leq s \leq t\}$ for each $t \geq 0$.  Let also $L = (L_t)_{t \geq 0}$ be a local time of $S-Z$ at $0$ in the sense of \cite[Chapter IV,  Sections 2-4]{bertoin1996levy} and let $L^{-1}$ be the right-continuous inverse of $L$.  Note that we can choose $L$ by setting $L_t = -\inf\{0 \wedge Z_s : 0 \leq s \leq t\}$ for each $t \geq 0$ (see \cite[Chapter VI]{bertoin1996levy}).
By \cite[Chapter VIII, Lemma 1]{bertoin1996levy}, we have that $L^{-1}$ is a stable subordinator of index $1/3$.  Let $A = \sup\{t \leq 1 : Z_t - \inf_{0 \leq s \leq t} Z_s = 0\}$ and $B = \inf\{t \geq 1 : Z_t - \inf_{0 \leq s \leq t} Z_s = 0\}$.  Then $B-A$ is equal to the length of the interval in the complement of the range of $L^{-1}$ which contains $1$.
In particular, by \cite[Chapter~III, Proposition~2\,(i)]{bertoin1996levy}, the probability that $B - A \in [s,2s]$ is of order $s^{1/3}$ as $s \to 0$ and is of order $s^{-1/3}$ as $s\to\infty$.    By \cite[Chapter~VIII, Proposition 15]{bertoin1996levy}, the conditional law of the process 
$(B-A)^{-2/3} (Z_{A+(B-A)t} - \inf_{0 \leq s \leq 1} Z_s)$ is that of a unit length $3/2$-stable L\'evy excursion. By Step~1, the probability that $\sum_{(t,v) \in \wt{\Lambda} : t \in [1-S,1+S], \, v \leq cr^{4}(\log r^{-1})^{6+8u} } v$ exceeds $r^4 (\log \tfrac{1}{r})^{7+11u}$ decays to $0$ as $S \to 0$ faster than any fixed power of $S$ provided $r > 0$ is sufficiently small and fixed.
Therefore,  combining with the above description of the unit length $3/2$-stable L\'evy excursion with the scaling property of a $3/2$-stable L\'evy process,  we obtain that the same is true if $\wt{Z}$ is a $3/2$-stable L\'evy excursion of time length $S \in (0,R')$ and we sum $v$ over all of the associated jumps.
Similarly, for a $3/2$-stable L\'evy excursion $\wt{Z}$ of time length $S \in [R',\infty)$ and for each $p>0$, the probability that the sum of $v$ over all the jumps within some time interval
$I \subset [0,S]$ of length at most $R$ exceeds $r^4 (\log \tfrac{1}{r})^{7+11u}$ decays to $0$ as $r\to 0$ with a decay rate determined solely by $u$ and $p$ and faster than any power of $r$ provided $R' \leq S \leq r^{-p}$.

\emph{Step~3.} We are now going to deduce the result in the setting of $\bminflaw$ from the above estimates.  We will subsequently explain how to transfer the result to the setting of the Brownian disk in Steps~4 and 5 below.
Suppose that $(\CS,d,\nu,x,y)$ has distribution $\bminflaw$ and let $Y_t$ be the time-reversal of the boundary length process so that $Y_t$ is a $3/2$-stable CSBP excursion.
By the Lamperti transform~\eqref{eqn:lamperti_csbp_to_levy}, if we let $s(t) = \inf\{r' \geq 0 : \int_0^{r'} Y_{s'} ds' \geq t\}$ then we have that $\widetilde{Z}_t = Y_{s(t)}$ is a $3/2$-stable L\'evy excursion.
Let $\widetilde{E}_{r}$ denote the event that the sum of $v$ over all the jumps of $\widetilde{Z}$ within some time interval $I$ of length at most $R$ exceeds $r^4 (\log \tfrac{1}{r})^{7+11u}$.
For each $q,q'\in[0,\infty)$ with $0 \leq q'-q \leq r$, let $A_{[q,q']}$ denote the sum of the areas of the components which are cut off by the metric exploration within time interval $[q,q']$ and have area at most $cr^{4}(\log r^{-1})^{6+8u}$ each.
If such $q,q'$ satisfy $\sup_{s\in[q,q']}L_s \leq r^{2}(\log r^{-1})^{3+4u}$, then
the amount of $\widetilde{Z}_t$-time which corresponds to the metric exploration over time interval $[q,q']$ is
\[ \int_{q}^{q'} Y_{(d(x,y)-s)\vee 0} \, ds \leq (q'-q) r^2 (\log r^{-1})^{3+4u} \leq R,\]
and therefore
\begin{align}
\label{eq:Event-CutAreaLarge}
\bigcup_{\substack{q,q'\in[0,\infty)\\0\leq q'-q\leq r}}\biggl\{\sup_{s\in[q,q']}L_s \leq r^{2}(\log r^{-1})^{3+4u}, \, A_{[q,q']} > r^{4}(\log r^{-1})^{7+11u}\biggr\}
\subset \widetilde{E}_{r}.
\end{align}
On the other hand, since the distribution for the time length $\widetilde{\zeta}$ of $\widetilde{Z}$ is given by a constant times $s^{-5/3} ds$,
we see that $\bminflaw[\widetilde{E}_{r}]$ decays to $0$ as $r\to 0$ faster than any polynomial of $r$, by choosing arbitrarily large $p>0$ and integrating
our upper bound on $\bminflaw[\widetilde{E}_{r} \mid \widetilde{\zeta}=s]$ from Step~2
with respect to $s^{-5/3} \, ds$ on $(0,R')$, $[R',r^{-p})$ and $[r^{-p},\infty)$ separately.
In particular, the assertion of the lemma in the case of $\bminflaw$ follows from this estimate on $\bminflaw[\widetilde{E}_{r}]$ and~\eqref{eq:Event-CutAreaLarge}.

\emph{Step~4.} The result of Step~3 can be extended to the setting of an instance sampled from $\bdisklawweighted{\ell}$ (with the same exploration). Indeed,
let $Z'_t$ denote the Lamperti transform of the time-reversal of the boundary length process under $\bdisklawweighted{\ell}$, and let $E_r'$ denote the event that the sum of $v$ over all the jumps of $Z'$ within some time interval $I$ of length at most $R$ exceeds $r^4 (\log \tfrac{1}{r})^{7+11u}$,
so that~\eqref{eq:Event-CutAreaLarge} with $E'_{r}$ in place of $\widetilde{E}_{r}$ holds as events under $\bdisklawweighted{\ell}$.
If we start with a Brownian map instance and explore the filled metric ball until the first time $\tau$ it has boundary length $\ell$,
then by the strong Markov property of the metric exploration the conditional law of the complement $\CD_{\tau}$ given $\tau<\infty$ is $\bdisklawweighted{\ell}$,
and $\{\tau<\infty, \, \CD_{\tau}\text{ satisfies }E'_{r}\}\subset \widetilde{E}_{r}$ as events under $\bminflaw$.
Recall also that $\bdisklawweighted{\ell}[\nu(\CD)\geq c r^4 (\log r^{-1})^{6+8u}] \geq \ell/(3 c^{1/2} r^{2}(\log r^{-1})^{3+4u})$ by~\eqref{eqn:bdiskweighted_area_density} and $\ell \in (0, r^2 (\log r^{-1})^{3+4u}]$ and that $\bminflaw[\tau < \infty]=\bminflaw[\sup_{s\geq 0}L_s \geq \ell] = b_{0}/\ell$
for a constant $b_{0}>0$ as shown in the paragraph of~\eqref{eqn:m_ej_d}. Combining these facts, we get the following upper bound independent of $\ell$ on the probability of $E'_{r}$:
\begin{align}
&\bdisklawweighted{\ell}[E'_{r} \mid \nu(\CD)\geq c r^4 (\log r^{-1})^{6+8u}] \notag \\
&\leq \frac{\bminflaw\bigl[\indicator_{\{\tau<\infty\}}\bminflaw[\CD_{\tau}\text{ satisfies }E'_{r} \mid \tau<\infty] \bigr] }{ \bminflaw[\tau<\infty]\bdisklawweighted{\ell}[\nu(\CD)\geq c r^4 (\log r^{-1})^{6+8u}] } \notag \\
&\leq 3b_{0}^{-1}c^{1/2}r^{2}(\log r^{-1})^{3+4u} \bminflaw[\tau<\infty,\,\CD_{\tau}\text{ satisfies }E'_{r}] \notag \\
&\leq 3b_{0}^{-1}c^{1/2}r^{2}(\log r^{-1})^{3+4u} \bminflaw[\widetilde{E}_{r}]
\label{eq:Event-CutAreaLarge-BM-BDW}
\end{align}
for each $\ell \in (0,r^2 \log(r^{-1})^{3+4u})$.

\emph{Step~5.} On the basis of Steps~3 and 4, we can now obtain the result in the case of the center exploration of an instance $(\CD,d,\nu)$ with law $\bdisklaw{\ell}$ as follows.
Let $\CD_{t}$ denote the target component of the center exploration at time $t$ for each $t \geq 0$ and set $\tau = \inf\{t \geq 0: \nu(\CD_{t})<r^{4}(\log r^{-1})^{6+8u}\}$,
so that on the event $\{\tau>0\}$ we have $\nu(\CD_{\tau}) \leq r^{4}(\log \tfrac{1}{r})^{6+8u} \leq \nu(\CD_{\tau-})$, where $\CD_{\tau-}=\bigcap_{t\in[0,\tau)}\CD_{t}$.
Pick a random marked point $y \in \CD$ according to $\nu/\nu(\CD)$ independently of the center exploration. Then on the event $\{\tau>0\}$,
the conditional probability of $\{y\in\CD_{\tau-}\}$ given a realization of $\CD$ and the center exploration is at least $\nu(\CD_{\tau-})/\nu(\CD) \geq r^{4}(\log \tfrac{1}{r})^{6+8u}/\nu(\CD)$,
and on $\{\tau>0, \, y\in\CD_{\tau-}\}$ the center exploration agrees with the exploration targeted at $y$ over the time interval $[0,\tau)$.
Set $V_{r} = \{\nu(\CD) \geq cr^{4}(\log \tfrac{1}{r})^{6+8u}\} \subset \{\tau > 0\}$, and
recall that $d\bdisklawweighted{\ell} = \bigl( \int \nu(\CD) \, d\bdisklaw{\ell} \bigr)^{-1} \nu(\CD) \, d\bdisklaw{\ell}$,
that $\int \nu(\CD) \, d\bdisklaw{\ell} = \ell^{2}$ by~\eqref{eqn:bdisk_area_density}
and that $\bdisklawweighted{\ell}[V_{r}]/\bdisklaw{\ell}[V_{r}] \leq 9 \ell^{-2} cr^{4}(\log \tfrac{1}{r})^{6+8u}$
by~\eqref{eqn:bdiskweighted_area_density}, \eqref{eqn:bdisk_area_density} and $\ell \in (0, r^2 (\log \tfrac{1}{r})^{3+4u}]$.
Now, keeping writing $E_{r},A_{r}$ for the event and the sum of the areas as in the statement for the \emph{metric exploration targeted at $y$}
and letting $E^{\mathrm{cen}}_{r},A^{\mathrm{cen}}_{r}$ denote those as in the statement for the \emph{center exploration}, we see from~\eqref{eq:Event-CutAreaLarge}
with $E'_{r}$ in place of $\widetilde{E}_{r}$ for instances of $\bdisklawweighted{\ell}$ and from~\eqref{eq:Event-CutAreaLarge-BM-BDW} that the following holds:
\begin{align*}
& r^{4}(\log \tfrac{1}{r})^{6+8u} \bdisklaw{\ell}[V_{r} \cap E^{\mathrm{cen}}_{r} \cap \{A^{\mathrm{cen}}_{r} > r^{4}(\log\tfrac{1}{r})^{7+12u}\}] \\
&\leq \int_{V_{r} \cap E^{\mathrm{cen}}_{r} \cap \{A^{\mathrm{cen}}_{r} > r^{4}(\log r^{-1})^{7+12u}, \, y \in \CD_{\tau-}\}} \nu(\CD) \, d\bdisklaw{\ell} \\
&= \int \nu(\CD) \, d\bdisklaw{\ell} \cdot \bdisklawweighted{\ell}[V_{r} \cap E^{\mathrm{cen}}_{r} \cap \{A^{\mathrm{cen}}_{r} > r^{4}(\log \tfrac{1}{r})^{7+12u}, \, y \in \CD_{\tau-}\}] \\
&\leq \ell^{2} \bdisklawweighted{\ell}[V_{r} \cap E_{r} \cap \{A_{r} > r^{4}(\log \tfrac{1}{r})^{7+12u}, \, y \in \CD_{\tau-}, \, \tau > r\}] \\
&  \qquad + \ell^{2} \bdisklawweighted{\ell}[V_{r} \cap E^{\mathrm{cen}}_{r} \cap \{A^{\mathrm{cen}}_{r} > r^{4}(\log \tfrac{1}{r})^{7+12u}, \, y \in \CD_{\tau-}, \, \tau \leq r\}] \\
&\leq \ell^{2} \bdisklawweighted{\ell}[V_{r} \cap E_{r} \cap \{A_{r} > r^{4}(\log \tfrac{1}{r})^{7+11u}\}] \\
&  \qquad + \ell^{2} \bdisklawweighted{\ell}\biggl[V_{r} \cap \biggl\{\sup_{s\in[0,\tau]}L_s \leq r^{2}(\log \tfrac{1}{r})^{3+4u}, \, A_{[0,\tau]} > r^{4}(\log \tfrac{1}{r})^{7+11u}, \, \tau \leq r \biggr\} \biggr] \\
&\leq 2 \ell^{2} \bdisklawweighted{\ell}[V_{r} \cap E'_{r}] 
  = 2 \ell^{2} \bdisklawweighted{\ell}[V_{r}] \bdisklawweighted{\ell}[E'_{r} \mid V_{r}] \\
&\leq 18cr^{4}(\log \tfrac{1}{r})^{6+8u} \bdisklaw{\ell}[V_{r}]  \bdisklawweighted{\ell}[E'_{r} \mid V_{r}] \\
&\leq 54b_{0}^{-1}c^{3/2}r^{6}(\log \tfrac{1}{r})^{9+12u} \bdisklaw{\ell}[V_{r}] \bminflaw[\widetilde{E}_{r}];
\end{align*}
here, to bound the second term of the fourth line from above by that of the fifth line we used the following two facts.
First, $\sup_{s\in[0,\tau]}L_s = \sup_{s\in[0,\tau)}L_s$ on the event $\{\tau>0\}$ since $L_s$ has only downward jumps.
Second, on the event $\{A^{\mathrm{cen}}_{r} > r^{4}(\log \tfrac{1}{r})^{7+12u}, \, y \in \CD_{\tau-}, \, 0< \tau \leq r\}$,
the components cut off by the center exploration within the time interval $[\tau,r]$
contribute to $A^{\mathrm{cen}}_{r}$ by at most the sum of $\nu(\CD_{\tau})$
and the area of the component cut off at time $\tau$, where the latter is counted only if it is at most $cr^{4}(\log \tfrac{1}{r})^{6+8u}$,
and hence by $y \in \CD_{\tau-}$ and $0< \tau \leq r$,
\begin{align*}
A_{[0,\tau]}
&\geq A^{\mathrm{cen}}_{r} - \nu(\CD_{\tau}) - cr^{4}(\log \tfrac{1}{r})^{6+8u} \\
&> r^{4}(\log \tfrac{1}{r})^{7+12u} - (c+1) r^{4}(\log \tfrac{1}{r})^{6+8u}
  \geq r^{4}(\log \tfrac{1}{r})^{7+11u}.
\end{align*}
Consequently, we obtain
\begin{align*}
\bdisklaw{\ell}[E^{\mathrm{cen}}_{r} \cap \{A^{\mathrm{cen}}_{r} > r^{4}(\log\tfrac{1}{r})^{7+12u}\} \mid V_{r}]
\leq 54b_{0}^{-1}c^{3/2}r^{2}(\log \tfrac{1}{r})^{3+4u} \bminflaw[\widetilde{E}_{r}],
\end{align*}
which has been already shown in Step~3 to decay to $0$ as $r \to 0$ with a decay rate (independent of $\ell \in (0,r^2 \log(r^{-1})^{3+4u})$ and) faster than any polynomial of $r$.
\end{proof}

\begin{proof}[Proof of Theorem~\ref{thm:ball_concentration}, upper bound]
Suppose that $(\CS,d,\nu,x,y)$ has the law $\bminflaw$ conditioned on  having area at least $c r^{4}(\log \frac{1}{r})^{6+8u}$.
As we mentioned earlier, we will prove the upper bound by dominating the amount of area in $B(x,r)$ from above using a subcritical Galton-Watson tree.

Fix $c \geq 1$.  We will adjust its value later in the proof.
Let $E_{0,r}$ denote the event that any component of $\CS \setminus \overline{B(x,s)}$
has boundary length at most $r^2 (\log \frac{1}{r})^{3+4u}$ for any $s \in (0,r(\log \frac{1}{r})^{1+u}]$,
so that $\bminflaw[(E_{0,r})^{c}]$ decays to $0$ as $r \to 0$ faster than any polynomial of $r$ by Lemma~\ref{lem:component_boundary_length}.
Let $E_r$ be as in the statement of Lemma~\ref{lem:num_components}.
Moreover,  arguing as in the paragraph after~\eqref{eqn:m_ej_d} implies that
\begin{align*}
\bminflaw\big[ \nu(\CS) \geq c r^4 \log(r^{-1})^{6+8u}\big] = c_0 \int_{t \geq cr^4 \log(r^{-1})^{6+8u}} t^{-2}dt = c_0 c^{-1} r^{-4} \log(r^{-1})^{-6-8u}
\end{align*}
for some universal constant $c_0 \in (0,\infty)$.  Combining the above, we obtain that $\bminflaw\big[ (E_{0,r})^c \giv \nu(\CS) \geq cr^4 \log(r^{-1})^{6+8u} \big]$ decays to $0$ as $r \to 0$ faster than any polynomial of $r$.  Furthermore,  applying Lemma~\ref{lem:boundary_length} with $c = \log(r^{-1})^{3+4u}$ implies that $\bminflaw\big[ (E_r)^c \big] \leq m_0 \exp(-c_0 \log(r^{-1})^{3+4u}) r^{-2}$ for all $r \in (0,1)$ sufficiently small,  for some universal constants $m_0,  c_0 \in (0,\infty)$.  It follows that $\bminflaw\big[ (E_r)^c \giv \nu(\CS) \geq cr^4 \log(r^{-1})^{6+8u}\big] $ decays to $0$ as $r \to 0$ faster than any polynomial of $r$.  Next,  we consider the exploration described just above Lemma~\ref{lem:num_components} and let $\p$ be the probability measure given by conditioning $\bminflaw$ on the positive probability event $\{\nu(\CS) \geq cr^4 \log(r^{-1})^{6+8u}\}$.  Lemma~\ref{lem:num_components} implies that the number $N$ of components cut off by $B(x,r)$ with area at least $cr^4 \log(r^{-1})^{6+8u}$ satisfies $\p\big[ E_r \cap \{N \geq n\}\big] \leq (1-\delta) \wt{p}_c^{n-1}$ for each $n \in \mathbb{N}$,  $c \geq c_0$,  where $c_0 \geq 1,  \delta \in (0,1)$ are as in the statement of Lemma~\ref{lem:num_components}.  We assume that we have chosen $c$ sufficiently large so that $\sum_{n \geq 1} n (1-\delta) \wt{p}_c^{n-1} < 1$.  Note that each of these $N$ holes are conditionally independent given their boundary lengths.  If the boundary length of such a hole is $\ell$,  then we recall that its law is given by $\bdisklaw{\ell}$ conditioned on having area at least $cr^4 \log(r^{-1})^{6+8u}$.  We then branch the exploration into each of these holes by performing the center exploration and then proceed as in Lemma~\ref{lem:num_components}. Then Lemma~\ref{lem:num_components} implies that inside of each hole,  the number of additional holes which are cut off and have area at least $cr^4 \log(r^{-1})^{6+8u}$ is stochastically dominated by a random variable with mean strictly smaller than $1$ and with an exponential tail.  We have thus shown that the number of holes discovered in the entire branching exploration is dominated from above by a Galton-Watson process with offspring distribution given by a distribution with mean strictly smaller than $1$ and with an exponential tail.  By our choice of $c$,  this Galton-Watson process is subcritical,  so that the law of the total progeny exhibits exponential tails.  Therefore,  the probability that the Galton-Watson tree has more than $\log(r^{-1})^{1+u}$ nodes decays to $0$ as $r \to 0$ faster than any power of $r$.

Next,  we note that \cite[Lemma~4.12]{ms2015mapmaking} implies that the metric net $\cup_{r \geq 0} \partial \fb{x}{r}$ has $\nu$-measure zero for $\bminflaw$-a.e.\  instance of $(\CS,d,\nu,x,y)$.  We claim that the same is true if $(\CS,d,\nu,x,y)$ is sampled from $\bdisklawweighted{\ell}$ instead,  for each $\ell>0$.  Indeed,  we fix $r>0$ and perform the metric exploration from a sample $(\CS,d,\nu,x,y)$ from $\bminflaw$ starting from $x$ and targeted at $y$.  Then we know that the conditional law of $\CS \setminus \fb{x}{r}$ given $L_r$ is equal to $\bdisklawweighted{L_r}$.  Also,  the metric net of $\CS \setminus \fb{x}{r}$ is contained in the metric net of $(\CS,d,\nu,x,y)$ and so the claim follows for a sample from $\bdisklawweighted{L_r}$.  Combining with the rescaling property of the Brownian map,  we obtain that the claim is true for $\bdisklawweighted{\ell}$ for each $\ell>0$.  Furthermore,  observing that the probability measures $\bdisklawweighted{\ell}$ and $\bdisklaw{\ell}$ are mutually absolutely continuous for each $\ell>0$,  and combining with the scaling property of the Brownian disk,  we obtain that the metric net of a sample from $\bdisklaw{\ell}$ has area zero $\bdisklaw{\ell}$-a.e.\ for each $\ell>0$.  Therefore,  combining everything we obtain that on the event that $E_{0,r} \cap E_r$ holds and the Galton-Watson tree has at most $\log(r^{-1})^{1+u}$ nodes,  we see from Lemma~\ref{lem:num_components} that the total mount of area in $B(x,r)$ is at most $r^4 \log(r^{-1})^{8+13u}$ off an event whose probability decays to zero as $r \to 0$ faster than any power of $r$ under $\p$.  It follows that the probability that $\nu(B(x,r)) > cr^4\log(r^{-1})^{8+13u}$ under $\p$ decays to zero as $r \to 0$ faster than any power of $r$.  Hence,  by picking i.i.d.\  points from $\frac{\nu}{\nu(\CS)}$ and arguing as in the proof of the lower bound of the theorem,  we obtain that $\nu(B(z,r)) \leq cr^4 \log(r^{-1})^{8+14u}$ for each $z \in \CS$,  off an event whose probability under $\p$ decays to zero as $r \to 0$ faster than any power of $r$.  Therefore the proof is complete by observing that 
\begin{align*}
\p = \frac{\int_{t \geq cr^4 \log(r^{-1})^{6+8u}} t^{-3/2} \bmlaw{t} dt}{c^{-1} r^{-4} \log(r^{-1})^{6+8u}}
\end{align*}
and combining with the scaling property of the Brownian map.
\end{proof}

\begin{remark}
Alternatively, the upper bound in Theorem~\ref{thm:ball_concentration} can be deduced directly from a result in \cite{LG21}. In fact, for a ``typical point'' $x$,
meaning a point $x$ chosen uniformly according to the volume measure, it has been shown in
\cite[Proposition~11]{LG21} that the $k$-th moment of $\nu(B(x,r))$ is bounded above by $C_0^k\, k! \, r^{4k}$,
where $C_0$ is a constant. Hence, for any $\lambda \in (0, C_0)$ the expectation of $\exp(\lambda \nu(B(x,r)) / r^4 ))$
is finite, so the probability that $\nu(B(x,r))$ is at least a constant times $r^4 \log(r^{-1})$
decays to zero faster than any polynomial, and the upper bound in Theorem~\ref{thm:ball_concentration}
follows by a union bound. However, we decided to include the above proof based on a branching argument since it is different from the proof in \cite{LG21} and it does not rely on such precise moment estimates, so may be of independent interest.
\end{remark}

\section{Percolation exploration}
\label{sec:percolation-exploration}

Later in this work, we are going to use $\SLE_6$ chunks in percolation style arguments
as illustrated in Proposition~\ref{prop:good_chunks_percolate} below, the proof of which
is the purpose of this long section. As necessary preparations for its statement,
in Subsection~\ref{ssec:SLE6hull-MCPU} we introduce a suitable state space $\MCPU{2}$
for random quantum surfaces and certain $\MCPU{2}$-valued random variables defined through $\SLE_{6}$
explorations of quantum disks and wedges considered in Subsection~\ref{subsec:sle_explorations}.
In Subsection~\ref{ssec:statement-percolation-exploration} we state the main result of this
section (Proposition~\ref{prop:good_chunks_percolate}), which formulates a percolation
argument in the setting of a quantum disk weighted by its area, and its analog in the
simpler setting of a quantum wedge (Proposition~\ref{prop:good_chunks_percolate_half_plane}).
We first prove the latter in Subsection~\ref{ssec:good_chunks_percolate_half_plane}, and
then the former in Subsection~\ref{ssec:good_chunks_percolate_disk} on the basis of the latter.

\subsection{Preliminaries: $\SLE_{6}$ hulls as random curve-decorated quantum surfaces} \label{ssec:SLE6hull-MCPU}

Throughout, we consider a radial $\SLE_6$ process $\eta'$ on $\D$ targeted at $0$.
Let $t \in [0, \inf(\eta')^{-1}(0))$. In the same way as in Subsection~\ref{subsec:sle_explorations}
(see also Appendix~\ref{sec:SLE6-chunk-Jordan}), we define the \emph{hull} $K_{t}$ of $\eta'([0,t])$
as the complement in $\D$ of the $0$-containing component of $\D \setminus \eta'([0,t])$, and
we divide the boundary $\partial K_{t}$ of $K_{t}$ into the \emph{top} $\partial K_{t}\cap \D$ of $K_{t}$ and the
\emph{bottom} $\partial K_{t} \cap \partial \D$ of $K_{t}$. On the event $\{\partial K_{t} \cap \partial \D \not= \partial \D \}$,
i.e., that the bottom of $K_{t}$ is not the whole of the unit circle $\partial \D$, we can further divide the top (resp.\ bottom)
into its left and right sides: the \emph{left} (resp.\ \emph{right}) \emph{side}
of the top is the part which is to the left (resp.\ right) of $\eta'(t)$, and the \emph{left} (resp.\ \emph{right}) \emph{side}
of the bottom is the part which is to the left (resp.\ right) of $\eta'(0)$.

Moreover, we also consider below the space $\MCPU{2}$ of
\emph{curve-decorated quantum surfaces with two marked boundary points}, following \cite[Subsection 2.2.5]{gm2017charSLE}.
It is defined as the set of equivalence classes modulo conformal maps of quintuples $(D,\mu,\eta,x,y)$ of a simply connected domain
$D\subsetneq\mathbb{C}$, a Radon measure $\mu$ on $D$, a continuous map $\eta:[0,\infty]\to D\cup\widetilde{\partial}D$
with $\eta(0)\in\widetilde{\partial}D$, and $x,y\in\widetilde{\partial}D\setminus\{\eta(0)\}$ with $x\not=y$
and $\eta(0)\in[x,y]_{\widetilde{\partial}D}^{\ccw}$, where $\widetilde{\partial}D$ denotes the set of prime ends of $D$
and $[x,y]_{\widetilde{\partial}D}^{\ccw}$ the counterclockwise arc of $\widetilde{\partial}D$ from $x$ to $y$.
Noting that each equivalence class $\mathcal{K} \in \MCPU{2}$ has a unique representative of the form
$(\D,\mu_{\mathcal{K}},\eta_{\mathcal{K}},-\sqrt{-1},\sqrt{-1})$ with $\eta_{\mathcal{K}}(0)=1$ by the Riemann mapping theorem,
we equip $\MCPU{2}$ with the \emph{conformal Prokhorov-uniform metric} $\mymathbb{d}^{\mathrm{CPU}}_{2}$ given by
\begin{equation}\label{eq:dCPU2}
\mymathbb{d}^{\mathrm{CPU}}_{2}(\mathcal{K}_{1},\mathcal{K}_{2})
  :=\mymathbb{d}^{\mathrm{P}}_{\D}(\mu_{\mathcal{K}_{1}},\mu_{\mathcal{K}_{2}})+\mymathbb{d}^{\mathrm{U}}_{\D}(\eta_{\mathcal{K}_{1}},\eta_{\mathcal{K}_{2}}),
  \qquad \mathcal{K}_{1},\mathcal{K}_{2} \in \MCPU{2},
\end{equation}
where $\mymathbb{d}^{\mathrm{P}}_{\D}$ is a complete metric
on the space of Radon measures on $\D$ compatible with the vague topology and
$\mymathbb{d}^{\mathrm{U}}_{\D}(\eta_{1},\eta_{2})
  :=\sum_{n=1}^{\infty}2^{-n}\sup_{t\in[0,n]}|\eta_{1}(t)-\eta_{2}(t)|$,
so that $\MCPU{2}$ becomes a complete separable metric space.
Note that a quantum surface $(D,h,\eta,x,y)$ parameterized by a simply connected domain
$D\subsetneq\mathbb{C}$ and equipped with a continuous map $\eta:[0,\infty]\to D\cup\widetilde{\partial}D$
with $\eta(0)\in\widetilde{\partial}D$ and two marked boundary points
$x,y\in\widetilde{\partial}D\setminus\{\eta(0)\}$ with $x\not=y$ and $\eta(0)\in[x,y]_{\widetilde{\partial}D}^{\ccw}$
can be identified with an almost-surely defined random element $(D,\qmeasure{h},\eta,x,y)$ of $\MCPU{2}$,
since $h$ is a measurable function of $\qmeasure{h}$ by \cite[Theorem 1.1 and Remark 1.2]{BSS14equivLQGGFF}.

$\SLE_6$ explorations of quantum surfaces as introduced in Subsection~\ref{subsec:sle_explorations}
naturally define $\MCPU{2}$-valued random variables as follows. Let $\sigma$ be an
\emph{$\MCPU{2}$-stopping time}, i.e., a function $\sigma:\MCPU{2}\to[0,\infty]$
which is a stopping time with respect to the filtration generated by the
$\MCPU{2}$-valued stochastic process $\{\mathcal{K}_{t}\}_{t\geq 0}$ on $\MCPU{2}$
given by $\mathcal{K}_{t}(D,\mu,\eta,x,y):=(D,\mu,\eta(\cdot\wedge t),x,y)$.
We define a Borel subset $E_{\sigma}$ of $\MCPU{2}$, considered as an event for $\omega\in\MCPU{2}$, by
\begin{equation} \label{eq:exploration_Esigma}
E_{\sigma} := \Biggl\{ \omega=(D_{\omega},\mu_{\omega},\eta_{\omega},x_{\omega},y_{\omega}) \in \MCPU{2} \Biggm|
	\begin{minipage}{180pt}
	$\sigma(\omega) \in (0,t_{\eta_{\omega}})$,
	$\widetilde{\partial}K^{\omega}_{\sigma(\omega)}\cap\widetilde{\partial}D_{\omega} \not= \widetilde{\partial}D_{\omega}$,
	$\eta_{\omega}(\sigma(\omega)) \in \widetilde{\partial}D_{\omega}$
	\end{minipage}
	\Biggr\},
\end{equation}
where $t_{\eta_{\omega}}:=\inf\eta_{\omega}^{-1}(\eta_{\omega}(\infty))$ and, for $t \in [0,t_{\eta_{\omega}})$,
$K^{\omega}_{t}$ denotes the complement in $D_{\omega}$ of the component of
$D_{\omega}\setminus\eta_{\omega}([0,t])$ whose closure in $D_{\omega}\cup\widetilde{\partial}D_{\omega}$ contains $\eta_{\omega}(\infty)$ and
$\widetilde{\partial}K^{\omega}_{t}$ denotes its boundary in $D_{\omega}\cup\widetilde{\partial}D_{\omega}$.
Now let $\ell > 0$, suppose that $\CD = (\D,h,0)$ has law $\qdiskweighted{\ell}$, let
$\eta'$ be an independent radial $\SLE_6$ on $\CD$ starting from a uniformly random point
on $\partial \CD$, targeted at $0$ and parameterized by quantum natural time,
let $x_{\CD,\eta'(0)},y_{\CD,\eta'(0)} \in \partial \D\setminus\{\eta'(0)\}$ be such that $\eta'(0)\in[x_{\CD,\eta'(0)},y_{\CD,\eta'(0)}]_{\partial \D}^{\ccw}$
and $\qbmeasure{h}([x_{\CD,\eta'(0)},\eta'(0)]_{\partial \D}^{\ccw}) = \qbmeasure{h}([\eta'(0),y_{\CD,\eta'(0)}]_{\partial \D}^{\ccw}) = \ell/4$,
and set $\sigma_{\CD}:=\sigma(\D,\qmeasure{h},\eta',x_{\CD,\eta'(0)},y_{\CD,\eta'(0)})$.
Then we have $\eta'([t_{\eta'},\infty])=\{0\}$ and $t_{\eta'}=\inf(\eta')^{-1}(0)<\infty$ a.s., and
it follows from Proposition~\ref{prop:SLE6-chunk-Jordan-D} that, a.s.\ on the event
\begin{equation} \label{eq:exploration_Esigma_disk}
\begin{split}
E^{\CD}_{\sigma} &:= \{(\D,\qmeasure{h},\eta',x_{\CD,\eta'(0)},y_{\CD,\eta'(0)}) \in E_{\sigma}\} \\
	&= \bigl\{ \sigma_{\CD} \in (0,t_{\eta'}), \, \partial K_{\sigma_{\CD}}\cap \partial \CD \not= \partial \CD, \, \eta'(\sigma_{\CD}) \in \partial \CD \bigr\},
\end{split}
\end{equation}
$K_{\sigma_{\CD}}\setminus\partial K_{\sigma_{\CD}}$ is a Jordan domain in $\mathbb{C}$ with boundary $\partial K_{\sigma_{\CD}}$,
$\eta'([0,\sigma_{\CD}])\subset K_{\sigma_{\CD}}\cup\partial K_{\sigma_{\CD}}$, and the bottom
$\partial K_{\sigma_{\CD}}\cap \partial \CD$ of $K_{\sigma_{\CD}}$ is a compact interval in $\partial\CD$
containing $\eta'(0)$ in its interior. In particular, letting
$x_{\sigma_{\CD}}$ and $y_{\sigma_{\CD}}$ denote the endpoints of the left and right sides,
respectively, of the bottom of $K_{\sigma_{\CD}}$ other than $\eta'(0)$, we see that
\begin{equation} \label{eq:exploration_disk_N}
\CN^{\CD}_{\sigma} := \bigl( K_{\sigma_{\CD}}\setminus\partial K_{\sigma_{\CD}}, \qmeasure{h}|_{K_{\sigma_{\CD}}\setminus\partial K_{\sigma_{\CD}}}, \eta'(\cdot\wedge\sigma_{\CD}), x_{\sigma_{\CD}}, y_{\sigma_{\CD}} \bigr)
\end{equation}
is an $\MCPU{2}$-valued random variable defined a.s.\ on $E^{\CD}_{\sigma}$.

We remark that the construction in the last paragraph can be applied also to a weight-$2$ quantum wedge
$\CW = (\h,h,0,\infty)$, an independent chordal $\SLE_6$ $\eta'$ on $\h$ from $0$ to $\infty$ parameterized
by quantum natural time, and $x_{\CW},y_{\CW} \in \partial \h$ with $x_{\CW}<0<y_{\CW}$ and $\qmeasure{h}([x_{\CW},0])=\qmeasure{h}([0,y_{\CW}])=1$.
In this case we have $\eta'(\infty)=\infty$ and $t_{\eta'}=\inf(\eta')^{-1}(\infty)=\infty$ a.s.\ and,
as introduced in Subsection~\ref{subsec:sle_explorations}, the hull $K_{t}$ of $\eta'([0,t])$,
its top and bottom are defined for any $t \in [0, \infty)$ in the same way as above
with the $0$-containing component replaced by the unbounded component.
Also for each $\MCPU{2}$-stopping time $\sigma$, thanks to Proposition~\ref{prop:SLE6-chunk-Jordan-H}
an $\MCPU{2}$-valued random variable $\CN^{\CW}_{\sigma}$ is defined a.s.\ on the event
\begin{equation} \label{eq:exploration_Esigma_wedge}
E^{\CW}_{\sigma}:=\{(\h,\qmeasure{h},\eta',x_{\CW},y_{\CW})\in E_{\sigma}\}=\{\sigma_{\CW} \in (0,\infty), \, \eta'(\sigma_{\CW}) \in \partial \CW\}
\end{equation}
by~\eqref{eq:exploration_disk_N} with $\sigma_{\CW}$ in place of $\sigma_{\CD}$, where $\sigma_{\CW}:=\sigma(\h,\qmeasure{h},\eta',x_{\CW},y_{\CW})$.

\subsection{Statement of percolation exploration} \label{ssec:statement-percolation-exploration}

To state the main result of this section (Proposition~\ref{prop:good_chunks_percolate} below),
for each $\delta\in(0,\infty)$ we define an $\MCPU{2}$-stopping time $\sigma_{\delta}$ by
\begin{equation} \label{eq:sigma_delta_A}
\sigma_{\delta}(\omega) := \inf\bigl\{ t \in [\delta,\infty) \bigm|
	\text{$t < t_{\eta_{\omega}}$,
	$\widetilde{\partial}K^{\omega}_{t}\cap\widetilde{\partial}D_{\omega} \not= \widetilde{\partial}D_{\omega}$,
	$\eta_{\omega}(t) \in \widetilde{\partial}D_{\omega}$}
	\bigr\}
\end{equation}
for $\omega=(D_{\omega},\mu_{\omega},\eta_{\omega},x_{\omega},y_{\omega})\in\MCPU{2}$
(recall that $t_{\eta_{\omega}}:=\inf\eta_{\omega}^{-1}(\eta_{\omega}(\infty))$).

\begin{proposition}
\label{prop:good_chunks_percolate}
There exist $A_{0} \in [2,\infty)$ and $c_{0,\max} \in (0,\infty)$ such that for any
$A \in [A_{0},\infty)$ and any $c_{0} \in (0,c_{0,\max}]$ the following is true with
$\epsilon_{0} := c_{0} A^{-2/3}$. Let $\delta > 0$, set
$\sigma := \sigma_{\delta/A} \wedge \delta$ and let $E$ be any Borel subset of $\MCPU{2}$
such that $\qwedge{2}\bigl[ E^{\CW}_{\sigma} \cap \{ \CN^{\CW}_{\sigma} \in E\} \bigm| \sigma < \delta \bigr] \geq 1 - \epsilon_{0}$.
Suppose that $\CD = (\D,h,0)$ has law $\qdiskweighted{1}$,
and consider the following exploration of $\CD$ by radial $\SLE_6$ curves.
Set $D_{0} := \D$, $\CD_{0} := \CD$, let $\eta'_{0}$ be a radial $\SLE_6$ on $\CD_{0}$ starting
from a uniformly random point on $\partial \CD_0$ and targeted at $0$, and set $\Xi_{0}:=\emptyset$.
We then inductively define a sequence $\{(\CD_{j}=(D_{j},h|_{D_{j}},0),\eta'_{j},\Xi_{j})\}_{j\geq 0}$
of triples of a quantum disk $\CD_{j}$ weighted by its area, a radial $\SLE_{6}$ $\eta'_{j}$ on $\CD_{j}$
starting from a uniformly random point $\eta'_{j}(0)$ on $\partial \CD_{j}$ and targeted at $0$,
and a $2^{\mathbb{Z}\cap[0,j)}$-valued random variable $\Xi_{j}$, as follows.

\begin{itemize}
\item Set $\sigma_{j} := \sigma\bigl(D_{j},\qmeasure{h|_{D_{j}}},\eta'_{j},x_{\CD_{j},\eta'_{j}(0)},y_{\CD_{j},\eta'_{j}(0)}\bigr)$
and $E_{j} := E^{\CD_{j}}_{\sigma} \cap \{\sigma_{j}<\delta, \, \CN^{\CD_{j}}_{\sigma} \in E\}$.
If $\sigma_{j} \geq t_{\eta'_{j}}$, set $D_{j+1}:=D_{j}$, $\CD_{j+1}:=\CD_{j}$,
$\eta'_{j+1}:=\eta'_{j}$ and $\Xi_{j+1}:=\Xi_{j}$. If $\sigma_{j} < t_{\eta'_{j}}$,
let $D_{j+1}$ be the $0$-containing component of $D_{j}\setminus \eta'_{j}([0,\sigma_{j}])$,
set $K^{j}_{\sigma_{j}} := D_{j} \setminus D_{j+1}$, and let $\CN_{j}$ and $\CD_{j+1}$
be the quantum surfaces parameterized by $K^{j}_{\sigma_{j}} \setminus \partial K^{j}_{\sigma_{j}}$
and $D_{j+1}$, respectively. If $E_{j}$ occurs, we also say that \emph{$E$ occurs for $\CN_{j}$}.
\item If $\sigma_{j} < t_{\eta'_{j}}$, we take an independent radial $\SLE_6$ $\eta'_{j+1}$
on $\CD_{j+1}$ targeted at $0$, choosing its initial point $\eta'_{j+1}(0)$ and $\Xi_{j+1}$ according to the following rule:
\begin{enumerate}[(a)]
\item\label{it:disk_good_chunk}If $E_{j}$ occurs, let $\eta_{j+1}'(0)$ be the leftmost point of $\partial(\CD_{j} \setminus \CD_{j+1}) \cap \partial \CD_{j}$
  and set $\Xi_{j+1} := \{ j \} \cup \{i \in \Xi_{j} \mid \partial \CD_{j+1} \cap \partial \CN_{i} \not= \emptyset \}$.
\item\label{it:disk_bad_chunk}If $E_{j}$ does not occur,
  $\widetilde{\partial} K^{j}_{\sigma_{j}} \cap \widetilde{\partial} D_{j} \not= \widetilde{\partial} D_{j}$ and
  $\bigcup_{i \in \Xi_{j}} ( \partial \CD_{j+1} \cap \partial \CN_{i} ) \not= \emptyset$,
  noting that $[a_{j},b_{j}]_{\partial \CD_{j+1}}^{\ccw}=\bigcup_{i \in \Xi_{j}} ( \partial \CD_{j+1} \cap \partial \CN_{i} )$
  for unique $a_{j},b_{j} \in \partial \CD_{j+1}$,
  let $\eta'_{j+1}(0)$ be the first point on
  $[a_{j},b_{j}]_{\partial \CD_{j+1}}^{\ccw}$ from $a_{j}$ that belongs to
  \begin{equation} \label{eq:bad_chunks_initial_point_disk}
  \{ b_{j} \} \cup \bigcup\nolimits_{\text{$i \in \Xi_{j}$,\,$\partial \CD_{j+1} \cap \partial \CN_{i}$ has quantum length at least $\epsilon_{0}\delta^{2/3}$}}\partial\CN_{i},
  \end{equation}
  and set $\Xi_{j+1}:=\bigl\{ i \in \Xi_{j} \bigm| \emptyset \not= \partial \CD_{j+1} \cap \partial \CN_{i} \subset [\eta'_{j+1}(0),b_{j}]_{\partial \CD_{j+1}}^{\ccw} \bigr\}$,
  where $[b_{j},b_{j}]_{\partial \CD_{j+1}}^{\ccw} := \{b_{j}\}$.
\item\label{it:disk_bad_chunk_too_big}If $E_{j}$ does not occur,
  $\widetilde{\partial} K^{j}_{\sigma_{j}} \cap \widetilde{\partial} D_{j} \not= \widetilde{\partial} D_{j}$ and
  $\bigcup_{i \in \Xi_{j}} ( \partial \CD_{j+1} \cap \partial \CN_{i} ) = \emptyset$,
  let $\eta_{j+1}'(0)$ be the rightmost point of $\partial(\CD_{j} \setminus \CD_{j+1}) \cap \partial \CD_{j}$
  and set $\Xi_{j+1} := \emptyset$.
\item\label{it:disk_bad_chunk_disconnection}If $E_{j}$ does not occur and
  $\widetilde{\partial} K^{j}_{\sigma_{j}} \cap \widetilde{\partial} D_{j} = \widetilde{\partial} D_{j}$,
  let $\eta'_{j+1}(0):=\eta'_{j}(\sigma_{j})$ and set $\Xi_{j+1}:=\emptyset$.
\end{enumerate}
\end{itemize}
Noting for any $j \geq 0$ that $\bigl\{ \text{$\sigma_{j} < t_{\eta'_{j}}$, $\partial \CD_{j+1} \subset \bigcup_{ i \in \Xi_{j+1} } \partial \CN_{i}$} \bigr\} \subset E_{j}$
and that on $E_{j}$ the bottom left of $\CD_{j} \setminus \CD_{j+1}$ can be written as
$[x_{\sigma_{\CD_{j}}},\eta_{j}'(0)]_{\partial \CN_{j}}^{\ccw}$ with
$x_{\sigma_{\CD_{j}}}$ as in~\eqref{eq:exploration_disk_N}, set
\begin{equation} \label{eq:good_chunks_percolate_round_number}
N_{\delta} := \min\biggl\{ j \geq 0 \biggm|
\begin{minipage}{240.5pt}
$\sigma_{j} < t_{\eta'_{j}}$, $\partial \CD_{j+1} \subset \bigcup_{ i \in \Xi_{j+1} } \partial \CN_{i}$,
$[x_{\sigma_{\CD_{j}}},\eta_{j}'(0)]_{\partial \CN_{j}}^{\ccw} \cap \bigcup_{ i \in \Xi_{j} } ( \partial \CD_{j} \cap \partial \CN_{i} )$
has quantum length at least $\epsilon_{0} \delta^{2/3}$
\end{minipage}\biggr\}.
\end{equation}
Let $u>0$ and define an event $E_{u,\delta}$ by
\begin{equation} \label{eq:good_chunks_percolate_event_boundary_length}
E_{u,\delta} := \biggl\{
\begin{minipage}{306pt}
for any $j \in \mathbb{Z} \cap [0,\delta^{-2/3-u}]$ and any $t \in [0,\sigma_{j}]$,
the boundary length of the $0$-containing component of $D_{j} \setminus \eta'_{j}([0,t])$
is at least $\delta^{2/3-u}$
\end{minipage}
\biggr\}.
\end{equation}
Then there exist $c_{1},c_{2},a \in (0,\infty)$ determined solely by $A,c_{0},u$ such that the following hold:
\begin{enumerate}[(i)]
\item\label{it:good_chunks_percolate_round_number}
  $\mathbb{P}\bigl[ E_{u,\delta} \cap \{ N_{\delta} \geq \delta^{-2/3-u} \} \bigr] \leq c_{1}\exp(-c_{2}\delta^{-a})$.
\item\label{it:good_chunks_percolate_dist_boundary}
\begin{equation} \label{eq:good_chunks_percolate_dist_boundary}
\mathbb{P}\!\left[ E_{u,\delta} \setminus \left\{
\begin{minipage}{230pt}
for each $0 \leq i \leq N_{\delta}$ there exist $n \in \mathbb{Z}\cap[0,\delta^{-u}]$ and
$\{i_{j}\}_{j=1}^{n} \subset \{0,\ldots,N_{\delta}\}$ such that $\partial \CN_{i_{n}} \cap \partial \CD_{0} \not= \emptyset$,
$i_{1} = i$ and $\partial \CN_{i_{j}} \cap \partial \CN_{i_{j+1}} \not= \emptyset$ for each $1 \leq j < n$
\end{minipage}
\right\} \right] \leq c_{1}\exp(-c_{2}\delta^{-a}).
\end{equation}
\end{enumerate}

\end{proposition}

\begin{figure}[ht!]
\begin{center}
\includegraphics[scale=0.85]{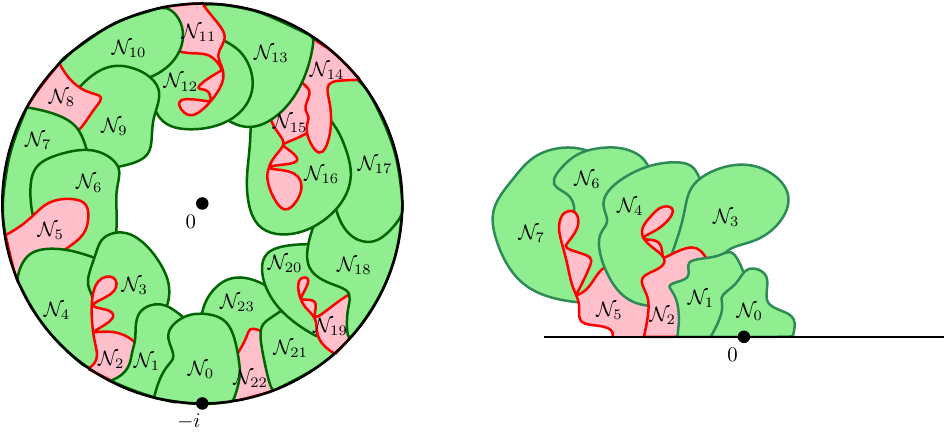}
\end{center}
\caption{\label{fig:exploration_disk_wedge} {\bf Left:} Illustration of the setup of Proposition~\ref{prop:good_chunks_percolate}.  The embedding of $\CD$ into $\D$ is taken so that the interior (resp.\ boundary) marked point is $0$ (resp.\ $-i$).  Chunks for which the event $E$ occurs (resp.\ does not occur) are shown in green (resp.\ red).  We will choose the event $E$ so that the $\CN_i$ for which it occurs are necessarily homeomorphic to $\D$, which is why the green chunks have this property while the red chunks do not.  Shown is the event that the chunks for which $E$ occurs disconnects $\partial \D$ from $0$.  {\bf Right:} Illustration of the setup for Proposition~\ref{prop:good_chunks_percolate_half_plane}, with the same color scheme as on the left.}
\end{figure}

The first step in the proof of Proposition~\ref{prop:good_chunks_percolate} is a related result in the setting of the half-plane.  This setting will be slightly easier to prove because the law of a weight-$2$ quantum wedge is invariant under the operation of exploring a chordal $\SLE_6$ curve for a given amount of quantum natural time.

\begin{proposition} \label{prop:good_chunks_percolate_half_plane}
There exist $A_{0} \in [2,\infty)$ and $c_{0,\max} \in (0,\infty)$ such that for any
$A\in[A_{0},\infty)$ and any $c_{0}\in(0,c_{0,\max}]$ the following is true with $\epsilon_{0}:=c_{0}A^{-2/3}$.
Set $\sigma:=\sigma_{1/A}\wedge 1$ and let $E$ be any Borel subset of $\MCPU{2}$ such that
$\qwedge{2}\bigl[E^{\CW}_{\sigma} \cap \{ \CN^{\CW}_{\sigma} \in E\} \bigm| \sigma < 1 \bigr] \geq 1-\epsilon_{0}$.
Suppose that $\CW = (\h,h,0,\infty)$ has law $\qwedge{2}$, and consider the following
exploration of $\CW$ by chordal $\SLE_6$ curves. Set $H_{0} := \h$, $\CW_{0} := \CW$ and
let $\eta'_0$ be a chordal $\SLE_6$ on $\CW_{0}$ from $0$ to $\infty$.
We then inductively define a sequence $\{(\CW_{j}=(H_{j},h|_{H_{j}},\eta'_{j}(0),\infty),\eta'_{j})\}_{j\geq 0}$
of pairs of a quantum wedge $\CW_{j}$ and a chordal $\SLE_{6}$ $\eta'_{j}$ on $\CW_{j}$ as follows.
\begin{itemize}
\item Set $\sigma_{j} := \sigma(H_{j},\qmeasure{h|_{H_{j}}},\eta'_{j},x_{\CW_{j}},y_{\CW_{j}})$
and $E_{j} := E^{\CW_{j}}_{\sigma} \cap \{ \sigma_{j} < 1, \, \CN^{\CW_{j}}_{\sigma} \in E\}$.
Let $H_{j+1}$ be the unbounded component of $H_{j}\setminus \eta'_{j}([0,\sigma_{j}])$,
set $K^{j}_{\sigma_{j}} := H_{j} \setminus H_{j+1}$, and let $\CN_{j}$ and $\CW_{j+1}$
be the quantum surfaces parameterized by $K^{j}_{\sigma_{j}} \setminus \partial K^{j}_{\sigma_{j}}$
and $H_{j+1}$, respectively. If $E_{j}$ occurs, we also say that \emph{$E$ occurs for $\CN_{j}$}.
\item Let $\eta_{j+1}'$ be an independent chordal $\SLE_6$ on $\CW_{j}$ to $\infty$
whose initial point $\eta'_{j+1}(0)$ is chosen according to the following rule.
If $E_{j}$ occurs, let $\eta_{j+1}'(0)$ be the leftmost point of
$\partial(\CW_{j} \setminus \CW_{j+1}) \cap \partial \CW_{j}$. If $E_{j}$ does not occur,
let $\eta'_{j+1}(0)$ be the first point on $\partial \CW_{j+1}$ that is to the right
of the rightmost point of $\partial(\CW_{j} \setminus \CW_{j+1}) \cap \partial \CW_{j}$
and belongs to
\begin{equation} \label{eq:bad_chunks_initial_point_half_plane}
(\partial \CW_{j+1} \cap \partial \CW_{0}) \cup
\bigcup\nolimits_{\text{$0\leq i<j$,\,$\partial \CW_{j+1} \cap \partial \CN_{i}$ has quantum length at least $\epsilon_{0}$}}\partial\CN_{i}.
\end{equation}
\end{itemize}
Let $L_{j}$ denote the quantum length of the top left of $\CW_{0}\setminus\CW_{j}$
(i.e., the part of $\partial(\CW_{0}\setminus\CW_{j}) \cap \CW_{0}$ which is to the left of $\eta'_{j}(0)$)
minus the quantum length of the bottom left of $\CW_{0}\setminus\CW_{j}$
(i.e., the part of $\partial(\CW_{0}\setminus\CW_{j}) \cap \partial \CW_{0}$ which is to the left of $0$).
Equivalently, set $L_{0}:=0$ and let $L_{j+1} - L_{j}$ be the quantum length of the part of $\partial \CW_{j+1}$
from $\eta'_{j+1}(0)$ to the leftmost point of $\partial(\CW_{j} \setminus \CW_{j+1}) \cap \partial \CW_{j}$
minus the quantum length of the bottom left of $\CW_{j} \setminus \CW_{j+1}$
(i.e., the part of $\partial(\CW_{j} \setminus \CW_{j+1}) \cap \partial \CW_{j}$ which is to the left of $\eta'_{j}(0)$).
Also let $R^{\mathrm{bot}}_{j}$ denote the quantum length of the bottom right of $\CW_{0}\setminus\CW_{j}$
(i.e., the part of $\partial(\CW_{0}\setminus\CW_{j}) \cap \partial \CW_{0}$ which is to the right of $0$).
Then the following hold:
\begin{enumerate}[(i)]
\item\label{it:move_left} There exist constants $c_{1},c_{2} \in (0,\infty)$ determined solely by $A,c_{0}$
  such that $\p[L_{N} \geq -c_{1} N] \leq \exp(-c_{2} N)$ for any $N \in \N$.
  More precisely, $c_{1}$ can be chosen as $c_{1}=\frac{1}{16}cA^{-2/3}\log A$
  with $c\in(0,\infty)$ the constant in Proposition~\ref{prop:running-inf-mean-diverge}.
\item\label{it:growth_right} For each $u\in(0,\frac{1}{2})$, there exists a constant $c_{3} \in (0,\infty)$ determined solely by $u$
  such that $\p[R^{\mathrm{bot}}_{N} \geq N/10] \leq c_{3} N^{u-1/2}$ for any $N \in \N$.
\item\label{it:bad_cluster_size} There exist constants $c_{4},c_{5},c_{6}\in(0,\infty)$
  determined solely by $A,c_{0}$ such that for any $u\in[2,\infty)$ and any $N \in \N$,
  \begin{equation}\label{eq:bad_cluster_size}
  \mathbb{P}\!\left[
    \begin{minipage}{235pt}
    for each $i \in \mathbb{Z}\cap[0,N]$ there exist $n \in \N \cap [1,c_{4}u^{2}]$ and
    $\{i_{j}\}_{j=1}^{n} \subset \mathbb{Z}\cap[0,N]$ such that $i_{1} = i$, $\partial \CN_{i_n} \cap \partial \CW_{0} \neq \emptyset$
    and $\partial \CN_{i_{j}} \cap \partial \CN_{i_{j+1}} \neq \emptyset$ for each $1 \leq j < n$
    \end{minipage}
  \right]
  \geq 1-c_{5}N^{2}\exp(-c_{6}u).
  \end{equation}
\item\label{it:good_chunks_initial_boundary}There exist constants $c_{7},c_{8},c_{9}\in(0,\infty)$
  determined solely by $A,c_{0}$ such that, with $\partial_{\mathrm{L}} \CW := (-\infty,0) \subset \partial \CW$,
  for any $u \in [2,\infty)$ and any $N\in\mathbb{N}\cap[c_{7}u^{3},\infty)$,
  \begin{equation}\label{eq:good_chunks_initial_boundary}
  \mathbb{P}\biggl[
    \begin{minipage}{245pt}
    for each $i \in \mathbb{Z}\cap[0,N]$ there exists $j \in \mathbb{Z}\cap[0,N)$
    such that $\lvert j-i \rvert \leq c_{7}u^{3}$, $\indicator_{E_{j}\cap E_{j+1}}=1$
    and $\partial \CN_{j} \cap \partial_{\mathrm{L}} \CW \not=\emptyset$
    \end{minipage}
  \biggr]
  \geq 1-c_{8}N^{2}\exp(-c_{9}u).
  \end{equation}
\end{enumerate}
\end{proposition}

\subsection{Proof of Proposition~\ref{prop:good_chunks_percolate_half_plane}} \label{ssec:good_chunks_percolate_half_plane}
The proof of Proposition~\ref{prop:good_chunks_percolate_half_plane} is long and divided
into several steps. Until the end of Subsection~\ref{ssec:good_chunks_percolate_half_plane},
we fix the situation of the statement of Proposition~\ref{prop:good_chunks_percolate_half_plane},
with $A \in [2,\infty)$ and $c_{0} \in (0,\infty)$ arbitrary and the way of choosing
$A_{0},c_{0,\max}$ specified in the course of the proof, and we also fix the following setting.
Recall that for each $j \geq 0$, $\CW_{j}$ is a weight-$2$ quantum wedge independent of
$\{(\CW_{i} \setminus \CW_{i+1},h|_{\CW_{i} \setminus \CW_{i+1}},\eta'_{i}|_{[0,\sigma_{i}]},\widetilde{L}_{i})\}_{0 \leq i < j}$
by the properties of quantum wedges described in Subsections~\ref{sssec:QD-QS-QW} and~\ref{subsec:sle_explorations},
where $\widetilde{L}_{i}$ denotes the quantum length of the part of $\partial \CW_{i+1}$
from $\eta'_{i+1}(0)$ to the leftmost point of $\partial(\CW_{i} \setminus \CW_{i+1}) \cap \partial \CW_{i}$.
Let $\{L^{j}_{t}\}_{t \geq 0}$ (resp.\ $\{R^{j}_{t}\}_{t \geq 0}$) be the left
(resp.\ right) boundary length process associated with $(\CW_{j},\eta'_{j})$
as introduced in Subsection~\ref{subsec:sle_explorations},
let $T^{j}$ denote the quantum length of the top $\partial(\CW_{j} \setminus \CW_{j+1}) \cap \CW_{j}$
of $\CW_{j} \setminus \CW_{j+1}$ and $B^{j}_{L}$ (resp.\ $B^{j}_{R}$) the quantum
length of the bottom left (resp.\ bottom right) of $\CW_{j} \setminus \CW_{j+1}$.
Then $(\{L^{j}_{t}\}_{t \geq 0},\{R^{j}_{t}\}_{t \geq 0})$ is a pair of independent
$3/2$-stable L\'{e}vy processes with only downward jumps and is independent of
$\{(\{L^{i}_{t}\}_{t \in [0,\sigma_{i}]},\{R^{i}_{t}\}_{t \in [0,\sigma_{i}]})\}_{0 \leq i < j}$ for each $j \geq 0$,
the sequence $\{(\{L^{j}_{t}\}_{t \in [0,\sigma_{j}]},\{R^{j}_{t}\}_{t \in [0,\sigma_{j}]})\}_{j=0}^{^\infty}$
is i.i.d., and we have
\begin{align} \label{eq:expression_top_length}
0 \leq T^{j} &= (L^{j}_{\sigma_{j}}-\inf\nolimits_{0\leq s\leq \sigma_{j}}L^{j}_{s}) + (R^{j}_{\sigma_{j}}-\inf\nolimits_{0\leq s\leq \sigma_{j}}R^{j}_{s}), \\
0 < B^{j}_{L} &= -\inf\nolimits_{0 \leq s \leq \sigma_{j}}L^{j}_{s},
  \qquad 0 < B^{j}_{R} = -\inf\nolimits_{0 \leq s \leq \sigma_{j}}R^{j}_{s},
  \label{eq:expression_bottom_lengths} \\
\sigma_{j} & = ( \inf\{ t \in [1/A,\infty) \mid \eta'_{j}(t) \in \partial \CW_{j} \} ) \wedge 1
  = ( \tau^{L}_{j} \wedge \tau^{R}_{j} ) \wedge 1
\label{eq:expression_sigma_running_inf}
\end{align}
(for~\eqref{eq:expression_sigma_running_inf} recall~\eqref{eq:sigma_delta_A}),
where $\tau^{L}_{j}:=\inf\{ t \in [1/A,\infty) \mid L^{j}_{t} = \inf_{0 \leq s \leq t}L^{j}_{s} \}$
and $\tau^{R}_{j}:=\inf\{ t \in [1/A,\infty) \mid R^{j}_{t} = \inf\nolimits_{0 \leq s \leq t}R^{j}_{s} \}$.
Also, following Appendix~\ref{sec:Levy_process_estimates}, let $X^{1}, X^{2}$ be i.i.d.\ $3/2$-stable L\'{e}vy
processes with only downward jumps and starting from $0$, and set $I^{j}_{t} := \inf_{0 \leq s \leq t} X^{j}_{s}$,
$\tau^{j} := \inf\{t \in [1,\infty) \mid X^{j}_{t} = I^{j}_{t}\}$ for $j=1,2$ and $\tau := \tau^{1} \wedge \tau^{2}$,
so that~\eqref{eq:expression_sigma_running_inf} and the scaling property of $X^{1},X^{2}$ imply that
$\bigl(\{L^{j}_{t/A}\}_{t \geq 0},\{R^{j}_{t/A}\}_{t \geq 0},A(\tau^{L}_{j} \wedge \tau^{R}_{j})\bigr)$ and
$\bigl(\{A^{-2/3}X^{1}_{t}\}_{t \geq 0},\{A^{-2/3}X^{2}_{t}\}_{t \geq 0},\tau\bigr)$
have the same law for any $j \geq 0$.

We first prove Proposition~\ref{prop:good_chunks_percolate_half_plane}-\eqref{it:growth_right},
which is an easy consequence of the properties mentioned in the previous paragraph,
Proposition~\ref{prop:running-inf-mean-upper-bound} and~\eqref{eqn:infbound}.

\begin{proof}[Proof of Proposition~\ref{prop:good_chunks_percolate_half_plane}-\eqref{it:growth_right}]
Setting $R^{\mathrm{bot}}_{0}:=0=:R^{\mathrm{top}}_{0}$ and letting $R^{\mathrm{top}}_{j}$
denote the quantum length of the part of $\partial \CW_{j}$ from $\eta'_{j}(0)$ to the rightmost
point of $\partial(\CW_{0} \setminus \CW_{j}) \cap \partial \CW_{0}$ for each $j \geq 1$,
we easily see from the definition of the exploration that for any $j \geq 1$,
\begin{equation} \label{eq:growth_right_comparison}
R^{\mathrm{bot}}_{j}
  =R^{\mathrm{bot}}_{j-1}+(B^{j-1}_{R}-R^{\mathrm{top}}_{j-1})^{+}
  \leq R^{\mathrm{bot}}_{j-1}+B^{j-1}_{R}
  \leq \sum_{k=0}^{j-1}B^{k}_{R}.
\end{equation}
As noted in the first paragraph of Subsection~\ref{ssec:good_chunks_percolate_half_plane},
$\bigl\{\bigl(\{R^{j}_{t/A}\}_{t \in [0,A\sigma_{j}]},A\sigma_{j}\bigr)\bigr\}_{j=0}^{\infty}$
is i.i.d.\ with the same law as $\bigl(\{A^{-2/3}X^{2}_{t}\}_{t \in [0,\tau\wedge A]},\tau\wedge A\bigr)$,
and therefore $\{B^{j}_{R}\}_{j=0}^{^\infty}$ is i.i.d.\ with the same law as
$-A^{-2/3}I^{2}_{\tau \wedge A}$ by~\eqref{eq:expression_bottom_lengths}. Moreover,
letting $p\in(1,\frac{3}{2})$, by the scaling property of $X^{2}$ and
\cite[Chapter VIII, Proposition 4]{bertoin1996levy} we have
\begin{equation} \label{eq:bottom_right_p-th-moment}
\mathbb{E}\bigl[(B^{0}_{R})^{p}\bigr]
  =\mathbb{E}\bigl[(A^{-2/3}|I^{2}_{\tau \wedge A}|)^{p}\bigr]
  \leq \mathbb{E}\bigl[(A^{-2/3}|I^{2}_{A}|)^{p}\bigr]
  = \mathbb{E}\bigl[|I^{2}_{1}|^{p}\bigr]
  <\infty,
\end{equation}
which further implies that for any $s \in (0,\infty)$,
\begin{align} \label{eq:bottom_right_WLLN_tail}
s\mathbb{P}[ B^{0}_{R} \geq s ]
  &\leq s^{1-p} \mathbb{E}\bigl[(B^{0}_{R})^{p}\bigr]
  \leq s^{1-p} \mathbb{E}\bigl[|I^{2}_{1}|^{p}\bigr], \\
\mathbb{E}[ B^{0}_{R} ] - \mathbb{E}\bigl[ B^{0}_{R} \indicator_{\{ B^{0}_{R} \leq s \}} \bigr]
  &= \mathbb{E}\bigl[ B^{0}_{R} \indicator_{\{ B^{0}_{R} > s \}} \bigr]
  \leq s^{1-p} \mathbb{E}\bigl[ (B^{0}_{R})^{p} \bigr]
  \leq s^{1-p} \mathbb{E}\bigl[|I^{2}_{1}|^{p}\bigr].
\label{eq:bottom_right_WLLN_mean_tail}
\end{align}
It follows from~\eqref{eq:bottom_right_WLLN_mean_tail} with $p=\frac{5}{4}$,
a version \cite[Exercise 1.2.11]{Str2011} of the weak law of large numbers with an explicit
remainder estimate and~\eqref{eq:bottom_right_WLLN_tail} that for any $N \in \mathbb{N}$
and any $s \in \mathbb{R}$ with $s > N^{-1/4} \mathbb{E}\bigl[|I^{2}_{1}|^{5/4}\bigr]$,
\begin{align}
&\mathbb{P}\Biggl[ \Biggl|\sum_{j=0}^{N-1} B^{j}_{R} - N \mathbb{E}[ B^{0}_{R} ]\Biggr| \geq Ns \Biggr] \notag \\
&\leq \mathbb{P}\Biggl[ \Biggl|\frac{1}{N}\sum_{j=0}^{N-1} B^{j}_{R} - \mathbb{E}\bigl[ B^{0}_{R} \indicator_{\{B^{0}_{R} \leq N\}} \bigr]\Biggr| \geq s - N^{-1/4} \mathbb{E}\bigl[|I^{2}_{1}|^{5/4}\bigr] \Biggr] \notag \\
&\leq 2N^{-1}\bigl(s - N^{-1/4} \mathbb{E}\bigl[|I^{2}_{1}|^{5/4}\bigr]\bigr)^{-2}\int_{0}^{N}t^{1-p} \mathbb{E}\bigl[|I^{2}_{1}|^{p}\bigr]\,dt + N^{1-p} \mathbb{E}\bigl[|I^{2}_{1}|^{p}\bigr] \notag \\
&= \Bigl( 2(2-p)^{-1}\bigl(s - N^{-1/4} \mathbb{E}\bigl[|I^{2}_{1}|^{5/4}\bigr]\bigr)^{-2} + 1\Bigr) \mathbb{E}\bigl[|I^{2}_{1}|^{p}\bigr] N^{1-p}.
\label{eq:bottom_right_WLLN_apply}
\end{align}
Finally, by Proposition~\ref{prop:running-inf-mean-upper-bound} and
\eqref{eqn:infbound}, as long as $A$ is large enough,
$\mathbb{E}[B^{0}_{R}]=\mathbb{E}[ -A^{-2/3}I^{2}_{\tau \wedge A} ] \leq 1/20$,
and for any such $A$ and any $N \in \mathbb{N}$ with $N \geq 40^{4}\mathbb{E}\bigl[ |I^{2}_{1}|^{5/4} \bigr]^{4}$
we see from~\eqref{eq:growth_right_comparison}, $\mathbb{E}[B^{0}_{R}]\leq 1/20$
and~\eqref{eq:bottom_right_WLLN_apply} with $s=1/20$ that
$\mathbb{P}[ R^{\mathrm{bot}}_{N} \geq N/10 ] \leq c_{3} N^{1-p}$
with $c_{3}:=(3200(2-p)^{-1}+1) \mathbb{E}\bigl[|I^{2}_{1}|^{p}\bigr]$,
completing the proof.
\end{proof}

\begin{figure}[ht!]
\begin{center}
\includegraphics[scale=0.85]{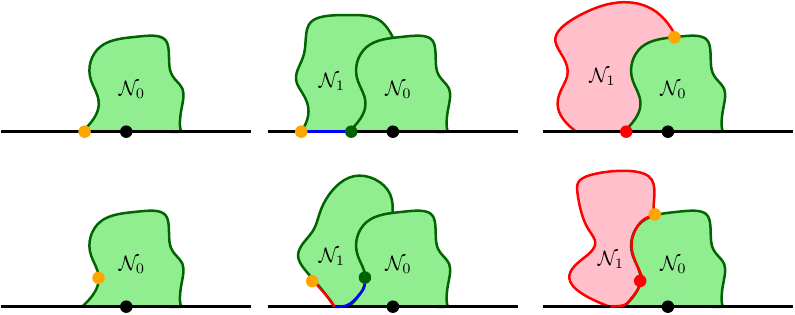}	
\end{center}
\caption{\label{fig:exploration_def} {\bf Top:} Shown on the left is the exploration in the statement of Proposition~\ref{prop:good_chunks_percolate_half_plane}
when the first chunk satisfies the event $E$. The exploration continues starting from the orange dot.
Shown on the top middle is the exploration when the second chunk also satisfies the event $E$ and the exploration continues from the orange dot.
On the top right, the second chunk does not satisfy $E$ and the exploration starts from the orange dot.
{\bf Bottom:} Shown is the exploration in the proof of Proposition~\ref{prop:good_chunks_percolate_half_plane},
which is more to the right than the exploration described in the statement of Proposition~\ref{prop:good_chunks_percolate_half_plane}.
Whenever a chunk for which $E$ occurs is discovered, the exploration continues from the point which is $\epsilon_0$ units
of boundary length from the leftmost intersection of the chunk with the surface boundary (orange dot, bottom middle).
Whenever a chunk for which $E$ does not occur is discovered, the exploration continues from the rightmost intersection
of the chunk with the surface boundary (orange dot, bottom right).}
\end{figure}

In order to prove Proposition~\ref{prop:good_chunks_percolate_half_plane}-\eqref{it:move_left},
we will consider a slightly modified exploration, illustrated in  Figure~\ref{fig:exploration_def},
for which the analog of  $\{L_{j}\}_{j=0}^{\infty}$, denoted by $\{L'_{j}\}_{j=0}^{\infty}$, dominates
$\{L_{j}\}_{j=0}^{\infty}$ in the sense that $L_N\leq L_N'$ for any $N\geq 1$. Moreover,
$L'_{N} = \sum_{j=0}^{N-1} (L'_{j+1} - L'_{j})$  and the increments $\{L'_{j+1} - L'_{j}\}_{j=0}^{\infty}$
form a sequence of i.i.d.\ random variables with negative mean and such that $(L'_{j+1} - L'_{j})^+$
has a finite exponential moment. Then Proposition~\ref{prop:good_chunks_percolate_half_plane}-\eqref{it:move_left}
can be derived from the following large deviation bound. This rather explicit large deviation estimate is required
in order to prove that the constants $c_{1},c_{2}$ in Proposition~\ref{prop:good_chunks_percolate_half_plane}-\eqref{it:move_left}
can be chosen to be dependent only on $A,c_{0}$.

\begin{lemma} \label{lem:Cramer_explicit_rate}
Let $\{ Y_{n} \}_{n=1}^{\infty}$ be i.i.d.\ real random variables, let
$\beta,\delta,K,M \in (0,\infty)$, $p \in (1,\infty)$ and assume that
$\mathbb{E}[e^{\beta Y_{1}^{+}}] \leq K$, $\mathbb{E}[ (Y_{1}^{-})^{p} ] \leq M$ and
$\mathbb{E}[ Y_{1} ] \leq -\delta$. Then there exists $\lambda \in (0,\infty)$
which is an explicit function of $\beta,\delta,K,M,p$ such that
\begin{equation} \label{eq:Cramer_explicit_rate}
\mathbb{P}\Biggl[ \sum_{j=1}^{n} Y_{j} \geq -\frac{1}{4} \delta n \Biggr] \leq e^{-\delta \lambda n/8}
  \qquad \text{for any $n \in \mathbb{N}$.}
\end{equation}
\end{lemma}

\begin{proof}
Set $a := (2M/\delta)^{\frac{1}{p-1}}$ and $b := \mathbb{E}[ (Y_{1}+a)^{+} ]$, so that
$b \leq \mathbb{E}[ Y_{1}^{+} ] + a \leq \beta^{-1} K + a$. Then by H\"{o}lder's and Markov's inequalities,
\begin{align*}
b - a - \mathbb{E}[ Y_{1} ]
  = \mathbb{E}[ (Y_{1}+a)^{-} ]
  \leq \mathbb{E}\bigl[ Y_{1}^{-} \indicator_{ \{ Y_{1}^{-} \geq a \} } \bigr]
  &\leq \mathbb{E}\big[ (Y_{1}^{-})^{p} \big]^{1/p} \, \mathbb{P}\big[ Y_{1}^{-} \geq a \big]^{1-1/p} \\
&\leq a^{1-p} \,  \mathbb{E}\big[ (Y_{1}^{-})^{p} \big]
  \leq a^{1-p} M
  = \tfrac{1}{2} \delta,
\end{align*}
and hence $a-b \geq  - \mathbb{E}[ Y_{1} ] - \frac{1}{2} \delta \geq \frac{1}{2} \delta$.
Therefore, setting $\widetilde{Y}_{j} := (Y_{j}+a)^{+} - b$ for $j \in \mathbb{N}$ and noting that
$\mathbb{E}[ \widetilde{Y}_{1} ] = 0$, we see that for any $n \in \mathbb{N}$ and any $\lambda \in (0,\beta]$,
\begin{align}
\mathbb{P}\Biggl[ \sum_{j=1}^{n} Y_{j} \geq - \frac{1}{4} \delta n \Biggr]
  &\leq \mathbb{P}\Biggl[ \sum_{j=1}^{n} \widetilde{Y}_{j} \geq \Bigl( a - b - \frac{1}{4} \delta \Bigr) n \Biggr]
  \leq \mathbb{P}\Biggl[ \sum_{j=1}^{n} \widetilde{Y}_{j} \geq \frac{1}{4} \delta n \Biggr] \notag \\
&\leq e^{- \delta \lambda n/4} \,  \mathbb{E}\Biggl[ \exp\biggl( \lambda \sum_{j=1}^{n} \widetilde{Y}_{j} \biggr) \Biggr]
  = e^{- \delta \lambda n/4} \,  \bigl( \mathbb{E}[ e^{ \lambda \widetilde{Y}_{1} } ] \bigr)^{n} \notag \\
&\leq e^{- \delta \lambda n/4} \biggl( 1 + \sum_{k=2}^{\infty} \frac{ (\lambda/\beta)^{k} }{ k! } \mathbb{E}\bigl[ ( \beta | \widetilde{Y}_{1} | )^{k} \bigr] \biggr)^{n} \notag \\
&\leq e^{- \delta \lambda n/4} \Bigl( 1 + ( \lambda / \beta )^{2} \mathbb{E}\bigl[ e^{ \beta | \widetilde{Y}_{1} | } \bigr] \Bigr)^{n} \notag \\
&\leq \exp\Bigl( - \delta \lambda n/4 + n ( \lambda / \beta )^{2} \, \mathbb{E}\bigl[ e^{ \beta | \widetilde{Y}_{1} | } \bigr] \Bigr).
\label{eq:Cramer_explicit_rate_proof}
\end{align}
Now since $|\widetilde{Y}_{1}| \leq Y_{1}^{+} + a + b \leq Y_{1}^{+} + 2a + \beta^{-1} K$ and thus
$\mathbb{E}\bigl[ e^{ \beta | \widetilde{Y}_{1} | } \bigr] \leq K e^{2 \beta a+K}$,
\eqref{eq:Cramer_explicit_rate} follows by choosing
$\lambda := \bigl( \frac{1}{8} \delta \beta^{2} K^{-1} e^{-2 \beta a-K} \bigr) \wedge \beta$
in~\eqref{eq:Cramer_explicit_rate_proof}. 
\end{proof}

\begin{proof}[Proof of Proposition~\ref{prop:good_chunks_percolate_half_plane}-\eqref{it:move_left}]
We define $\{L'_{j}\}_{j=0}^{\infty}$ by $L'_{0}:=0$ and
\begin{equation} \label{eqn:bad_length_change}
L'_{j+1} - L'_{j} := (\epsilon_{0}-B^{j}_{L}) \indicator_{E_{j}} + (T^{j} - B^{j}_{L}) \indicator_{E_{j}^{c}}.
\end{equation}
Then since $\{(\CW_{j} \setminus \CW_{j+1},h|_{\CW_{j} \setminus \CW_{j+1}},\eta'_{j}|_{[0,\sigma_{j}]})\}_{j=0}^{\infty}$
is i.i.d.\ and $\indicator_{E_{j}},T^{j},B^{j}_{L},B^{j}_{R}$ are measurable with respect to
$(\CW_{j} \setminus \CW_{j+1},h|_{\CW_{j} \setminus \CW_{j+1}},\eta'_{j}|_{[0,\sigma_{j}]})$
by~\eqref{eq:exploration_Esigma_wedge}, \eqref{eq:exploration_disk_N}, \eqref{eq:expression_top_length}
and~\eqref{eq:expression_bottom_lengths} for any $j \geq 0$,
we have that $\{(\indicator_{E_{j}},T^{j},B^{j}_{L},B^{j}_{R})\}_{j=0}^{\infty}$ is i.i.d.\ and
thus $\{ L'_{j+1} - L'_{j} \}_{j=0}^{\infty}$ is also i.i.d.
Note that $\{L'_{j}\}_{j=0}^{\infty}$ defined by~\eqref{eqn:bad_length_change} has the law
of the analog of $\{L_{j}\}_{j=0}^{\infty}$ for another exploration in which $\eta'_{j+1}(0)$
is instead chosen on the event $E_{j}$ to be the point on $\partial \CW_{j+1}$ which is
$\epsilon_{0}$ boundary length units to the right of the leftmost point of
$\partial(\CW_{j} \setminus \CW_{j+1}) \cap \partial \CW_{j}$ and on the event $E_{j}^{c}$
to be the rightmost point of $\partial(\CW_{j} \setminus \CW_{j+1}) \cap \partial \CW_{j}$;
see Figure~\ref{fig:exploration_def} for an illustration.
This is because $\CW_{j}$ is a weight-$2$ quantum wedge independent of
$\{(\CW_{i} \setminus \CW_{i+1},h|_{\CW_{i} \setminus \CW_{i+1}},\eta'_{i}|_{[0,\sigma_{i}]},\widetilde{L}_{i})\}_{0 \leq i < j}$
for each $j \geq 0$ for either of the ways of choosing $\{\eta'_{j+1}(0)\}_{j=0}^{\infty}$, by the
properties of quantum wedges described in Subsections~\ref{sssec:QD-QS-QW} and~\ref{subsec:sle_explorations}.

We claim that $L_{N} \leq L'_{N}$ for any $N \geq 1$. Indeed, let $j \geq 0$, set
\begin{equation} \label{eq:SLE6_percolate_H-right-chunks}
I_{j} := \Biggl\{i \Biggm|
	\begin{minipage}{295pt}
	$0\leq i < j$, $\partial \CW_{j+1} \cap \partial \CN_{i}$ has at least two
	elements and is included in the part of $\partial \CW_{j+1}$ from the rightmost point of
	$\partial(\CW_{j} \setminus \CW_{j+1}) \cap \partial \CW_{j}$ to the rightmost
	point of $\partial (\CW_{0} \setminus \CW_{j+1}) \cap \partial \CW_{0}$
	\end{minipage}
	\Biggr\},
\end{equation}
let $T^{i,j}$ denote the quantum length of $\partial \CW_{j+1} \cap \partial \CN_{i}$
for $i \in I_{j}$, set $I^{0}_{j}:=\emptyset$ on $E_{j}$ and
\begin{equation} \label{eq:SLE6_percolate_H-right-chunks-skipped}
I^{0}_{j} := \biggl\{i \in I_{j} \biggm|
	\begin{minipage}{245pt}
	$\partial \CW_{j+1} \cap \partial \CN_{i}$ is included in the part of
	$\partial \CW_{j+1}$ from the rightmost point of
	$\partial(\CW_{j} \setminus \CW_{j+1}) \cap \partial \CW_{j}$ to $\eta'_{j+1}(0)$
	\end{minipage}
	\biggr\}
\end{equation}
on $E_{j}^{c}$.
Then we easily see from the definition of $\{(\CW_{i},\eta'_{i})\}_{i=0}^{\infty}$
that $\indicator_{E_{i}}=1$ for any $i \in I_{j}$,
that $I^{0}_{j} \subset \{i\in I_{j} \mid T^{i,j} < \epsilon_{0}\}$,
that $I^{0}_{j} \cap I^{0}_{k} = \emptyset$ for any $k >j$,
that the part of $\partial \CW_{j+1}$ from the rightmost point of
$\partial(\CW_{j} \setminus \CW_{j+1}) \cap \partial \CW_{j}$ to the rightmost point of
$\partial (\CW_{0} \setminus \CW_{j+1}) \cap \partial \CW_{0}$
(is a singleton or) consists of closed intervals $\{\partial \CW_{j+1} \cap \partial \CN_{i}\}_{i \in I_{j}}$
with disjoint interiors in $\partial \CW_{j+1}$, and thus that
\begin{equation} \label{eq:SLE6_percolate_H_length_change}
\begin{split}
L_{j+1} - L_{j}
	&= -B^{j}_{L} \indicator_{E_{j}} + \Bigl( T^{j} - B^{j}_{L} + \sum\nolimits_{i\in I^{0}_{j}}T^{i,j} \Bigr) \indicator_{E_{j}^{c}} \\
&\leq -B^{j}_{L} \indicator_{E_{j}} + ( T^{j} - B^{j}_{L} + \epsilon_{0} \# I^{0}_{j}) \indicator_{E_{j}^{c}};
\end{split}
\end{equation}
here $\# A$ denotes the number of elements of a set $A$. It follows from~\eqref{eqn:bad_length_change},
\eqref{eq:SLE6_percolate_H_length_change}, the disjointness of $\{I^{0}_{j}\}_{j=0}^{\infty}$ and
$I^{0}_{j} \subset \{ i \in \mathbb{Z}\cap[0,j) \mid \indicator_{E_{i}}=1 \}$ that for any $N \geq 1$,
\begin{align} \label{eq:SLE6_percolate_H_length_comparison}
L_{N}&-L'_{N} = \sum_{j=0}^{N-1} \bigl( (L_{j+1}-L_{j}) - (L'_{j+1}-L'_{j}) \bigr)
	\leq \sum_{j=0}^{N-1}( -\epsilon_{0} \indicator_{E_{j}} + \epsilon_{0} \# I^{0}_{j} ) \\
&\leq\epsilon_{0}\bigl(-\#\{j \in \mathbb{Z}\cap[0,N) \mid \indicator_{E_{j}}=1 \} + \#\{j \in \mathbb{Z}\cap[0,N-1) \mid \indicator_{E_{j}}=1 \}\bigr)
	\leq 0, \notag
\end{align}
proving that $L_{N} \leq L'_{N}$. It thus suffices to prove~\eqref{it:move_left}
for $\{L'_{j}\}_{j=0}^{\infty}$ instead of $\{L_{j}\}_{j=0}^{\infty}$.

Recall that the sequence $\bigl\{\bigl(\{L^{j}_{t/A}\}_{t \in [0,A\sigma_{j}]},\{R^{j}_{t/A}\}_{t \in [0,A\sigma_{j}]},A\sigma_{j}\bigr)\bigr\}_{j=0}^{\infty}$
is i.i.d.\ with the same law as $\bigl(\{A^{-2/3}X^{1}_{t}\}_{t \in [0,\tau\wedge A]},\{A^{-2/3}X^{2}_{t}\}_{t \in [0,\tau\wedge A]},\tau\wedge A\bigr)$.
It thus follows from~\eqref{eqn:bad_length_change}, \eqref{eq:expression_top_length},
\eqref{eq:expression_bottom_lengths}, Propositions~\ref{prop:running-inf-mean-diverge} and~\ref{prop:reflection_mean_before_tau} that
\begin{align}
&\mathbb{E}[L'_{1} - L'_{0}] \notag \\
&= \epsilon_{0}\mathbb{P}[E_{0}]+\mathbb{E}[T^{0}\indicator_{E_{0}^{c}}]-\mathbb{E}[B^{0}_{L}] \notag \\
&\leq \epsilon_{0} + \mathbb{E}\bigl[T^{0}\indicator_{\{\sigma_{0}<1\}\setminus E_{0}}\bigr] + \mathbb{E}\bigl[T^{0}\indicator_{\{\sigma_{0}=1\}}-B^{0}_{L}\bigr] \notag \\
&= \epsilon_{0} + \mathbb{E}\bigl[T^{0}\indicator_{\{\sigma_{0}<1\}\setminus E_{0}}\bigr] + A^{-2/3} \mathbb{E}\bigl[(X^{1}_{A}-I^{1}_{A}+X^{2}_{A}-I^{2}_{A})\indicator_{\{\tau \geq A\}}+I^{1}_{\tau \wedge A}\bigr] \notag \\
&\leq \epsilon_{0} + \mathbb{E}\bigl[T^{0}\indicator_{\{\sigma_{0}<1\}\setminus E_{0}}\bigr] + A^{-2/3} ( 2c' - c \log A),
\label{eq:length_change_mean_upper1}
\end{align}
where $c':=\sup_{A' \in [1,\infty)}\E[ (X^{1}_{A'} - I^{1}_{A'}) \indicator_{\{\tau \geq A'\}}] < \infty$
and $c$ is the constant in Proposition~\ref{prop:running-inf-mean-diverge}.
Moreover, setting $Z^{j}:=\sup_{0\leq t\leq 1}(X^{j}_{t}-I^{j}_{t})$ for $j=1,2$, we have
$\mathbb{E}\bigl[e^{\alpha T^{0}}\bigr] \leq \mathbb{E}\bigl[e^{\alpha(Z^{1}+Z^{2})}\bigr] = \mathbb{E}\bigl[e^{\alpha Z^{1}}\bigr]^{2} < \infty$
for some $\alpha\in(0,\infty)$ by~\eqref{eq:expression_top_length} and Proposition~\ref{prop:stable-reflected-exp-moment},
and therefore by Jensen's inequality and $\mathbb{P}[ E_{0}^{c} \mid \sigma_{0} < 1 ] \leq \epsilon_{0}$ we get
\begin{align}
\mathbb{E}\bigl[T^{0}\indicator_{\{\sigma_{0}<1\}\setminus E_{0}}\bigr]
&= \alpha^{-1}\mathbb{P}\bigl[\{\sigma_{0}<1\}\setminus E_{0}\bigr] \mathbb{E}\bigl[\alpha T^{0} \bigm| \{\sigma_{0}<1\}\setminus E_{0} \bigr] \notag \\
&\leq \alpha^{-1}\mathbb{P}\bigl[\{\sigma_{0}<1\}\setminus E_{0}\bigr] \log \mathbb{E}\bigl[e^{\alpha T^{0}} \bigm| \{\sigma_{0}<1\}\setminus E_{0} \bigr] \notag \\
&\leq \alpha^{-1}\mathbb{P}\bigl[\{\sigma_{0}<1\}\setminus E_{0}\bigr] \log \frac{\alpha'}{\mathbb{P}\bigl[\{\sigma_{0}<1\}\setminus E_{0}\bigr]} \notag \\
&\leq \alpha^{-1}\epsilon_{0}\log(\alpha'/\epsilon_{0})
\label{eq:length_change_mean_upper2}
\end{align}
provided $\epsilon_{0}\leq \alpha'/e$, where $\alpha':=\mathbb{E}\bigl[e^{\alpha Z^{1}}\bigr]^{2}$.
Combining~\eqref{eq:length_change_mean_upper1}, \eqref{eq:length_change_mean_upper2} and
$\epsilon_{0}=c_{0}A^{-2/3}$, we conclude that
\begin{equation} \label{eq:length_change_mean_negative}
\mathbb{E}[L'_{1} - L'_{0}]
\leq A^{-2/3} \biggl(c_{0} + \alpha^{-1}c_{0}\log\frac{\alpha'}{c_{0}A^{-2/3}} + 2c' - c \log A \biggr)
\end{equation}
provided $c_{0}A^{-2/3}\leq \alpha'/e$, and hence we obtain $\mathbb{E}[L'_{1} - L'_{0}] \leq - \frac{1}{4} c A^{-2/3} \log A < 0$
from~\eqref{eq:length_change_mean_negative} by choosing $c_{0}$ arbitrarily from $(0, c \alpha]$
and taking $A$ large enough so that $c \alpha A^{-2/3}\leq \alpha'/e$ and
$\frac{1}{12}\log A\geq \alpha+2c'/c+\log(\alpha'/(c \alpha))$. Finally, by
\eqref{eqn:bad_length_change} and \cite[Chapter VIII, Proposition 4]{bertoin1996levy} we also have
\begin{gather*}
\mathbb{E}\bigl[e^{\alpha (L'_{1} - L'_{0})^{+}}\bigr]
  \leq \mathbb{E}\bigl[ e^{\alpha(\epsilon_{0}+T^{0})} \bigr]
  \leq e^{\alpha c_{0}}  \,\mathbb{E}\bigl[e^{\alpha Z^{1}}\bigr]^{2}
  = e^{\alpha c_{0}} \alpha' < \infty, \\
\mathbb{E}\bigl[ ( (L'_{1} - L'_{0})^{-} )^{p} \bigr]
  \leq \mathbb{E}\bigl[ ( B^{0}_{L} )^{p} \bigr]
  \leq \mathbb{E}\bigl[ | I^{1}_{1} |^{p} \bigr] < \infty,
\end{gather*}
where $p \in (1,\frac{3}{2})$ is arbitrary.
Altogether, we thus have shown that the sequence $\{L'_{j+1} - L'_{j}\}_{j=0}^{\infty}$ of real random variables
is i.i.d.\ and satisfies the assumptions of Lemma~\ref{lem:Cramer_explicit_rate}
with $\beta := \alpha$, $\delta := \frac{1}{4} c A^{-2/3} \log A$, $K := e^{\alpha c_{0}} \alpha'$,
$M := \mathbb{E}\bigl[ | I^{1}_{1} |^{5/4} \bigr]$ and $p := \frac{5}{4}$,
which are all determined solely by (the law of $(X^{1},X^{2})$ and) $A,c_{0}$.
Therefore by Lemma~\ref{lem:Cramer_explicit_rate} there exists $\lambda \in (0,\infty)$
which is an explicit function of $\beta,\delta,K,M,p$ and hence determined solely by $A,c_{0}$
such that $L'_{N} = \sum_{j=0}^{N-1} (L'_{j+1} - L'_{j})$ satisfies
\begin{equation} \label{eq:move_left_Lprime}
\mathbb{P}\biggl[ L'_{N} \geq -\frac{1}{4} \delta N \biggr]
  \leq \exp\Bigl( - \frac{1}{8} \delta \lambda N \Bigr)
  \qquad\text{for any $N \in \mathbb{N}$,}
\end{equation}
which together with $L_{N} \leq L'_{N}$ proves Proposition~\ref{prop:good_chunks_percolate_half_plane}-\eqref{it:move_left}.
\end{proof}

We now turn to the proof of Proposition~\ref{prop:good_chunks_percolate_half_plane}-\eqref{it:bad_cluster_size}.
We briefly describe its main steps before proceeding to the actual proof.
First, in Lemma~\ref{lem:right-length-change-SLLN-run_good_chunks_qn0}
we show that the quantum length of the top right of $\CW_{0}\setminus\CW_{j}$ minus
the quantum length of the bottom right of $\CW_{0}\setminus\CW_{j}$, with
the influences of the slides of $\{ \eta'_{i}(0) \}_{0 \leq i < j}$ to the right caused by
\eqref{eq:bad_chunks_initial_point_half_plane} neglected, grows linearly in $j$
with high probability, on the basis of some moment analysis similar to the proof of
Proposition~\ref{prop:good_chunks_percolate_half_plane}-\eqref{it:move_left} above.
Next, in Lemma~\ref{lem:right-length-positive-prob-positive} we prove by using
Lemma~\ref{lem:right-length-change-SLLN-run_good_chunks_qn0} and a comparison argument
similar to~\eqref{eq:SLE6_percolate_H_length_comparison} that for any $j \geq 0$,
with some probability uniformly positive in $j$ the event $E_{j}$ occurs,
$\partial \CN_{j} \cap \partial \CW_{j+1}$ has quantum length greater than $3\epsilon_{0}$
and the chunks $\{ \CN_{k} \}_{k\geq j+1}$ never enter the rightmost interval of
quantum length $\epsilon_{0}$ in $\partial \CN_{j} \cap \partial \CW_{j+1}$
or the part of $\partial \CW_{j+1}$ to the right of this interval.
Then in Lemma~\ref{lem:run_good_chunks_decay_fast} we deduce from
Lemma~\ref{lem:right-length-positive-prob-positive} that,
with probability exponentially high in $u$, at most $u$ of the subintervals
$\{ \partial \CN_{j} \cap \partial \CW_{n+1} \mid \text{$0\leq j\leq n$, $E_{j}$ occurs} \}$
of $\partial \CW_{n+1}$ with quantum length in $(0,\epsilon_{0})$ can consecutively align
for any $0 \leq n \leq N$, so that the slides of $\{ \eta'_{i}(0) \}_{0 \leq j \leq N}$
to the right caused by~\eqref{eq:bad_chunks_initial_point_half_plane} have quantum lengths
at most $\epsilon_{0}u$. Finally, we combine
Lemmas~\ref{lem:right-length-positive-prob-positive} and~\ref{lem:run_good_chunks_decay_fast}
to show that, with probability exponentially high in $u$, for any $0 \leq j \leq N$ with
$\partial \CN_{j} \cap (-\infty,0) \not=\emptyset$ we have
$\bigl(\bigcup_{0<k\leq u^{2}}\partial \CN_{j+k}\bigr) \cap (-\infty,0) \not=\emptyset$,
which is easily seen to imply the property stated in
Proposition~\ref{prop:good_chunks_percolate_half_plane}-\eqref{it:bad_cluster_size}.

\begin{lemma} \label{lem:right-length-change-SLLN-run_good_chunks_qn0}
Let $j \geq 0$ and set $S_{j,j}:=0$ and
\begin{equation} \label{eq:right-length-change}
S_{j,k} := \sum_{i=j+1}^{k}\bigl( (T^{i}-B^{i}_{R}) \indicator_{E_{i}} - B^{i}_{R} \indicator_{E_{i}^{c}} \bigr)
  = \sum_{i=j+1}^{k}( T^{i} \indicator_{E_{i}} - B^{i}_{R} )
  \quad \text{for $k>j$.}
\end{equation}
Let $c',c,\alpha,\alpha'\in(0,\infty)$ be the constants as in the proof of
Proposition~\ref{prop:good_chunks_percolate_half_plane}-\eqref{it:move_left} above,
fix an arbitrary $A \in [2,\infty)$ satisfying $c \alpha A^{-2/3}\leq \alpha'/e$ and
$\frac{1}{12}\log A\geq 2c'/c+\log(\alpha'/(c \alpha))$, and for each $n_{0} \in \mathbb{N}$ set
\begin{equation} \label{eq:right-length-change-SLLN-run_good_chunks_qn0}
q_{n_{0}} := \mathbb{P}\biggl[ \inf_{ k \geq n_{0} } \frac{S_{j,j+k}}{k} \geq \frac{1}{12} c A^{-2/3} \log A \biggr],
\end{equation}
which is independent of $j \geq 0$. Then, as long as $c_{0} \in (0, c \alpha]$,
\begin{equation} \label{eq:right-length-change-SLLN-run_good_chunks_qn0_converge}
\lim_{n_{0}\to\infty}q_{n_{0}}=1 \qquad
  \begin{minipage}{200pt}
  with a rate of convergence determined solely by $A$
  (and thus independent of $c_{0}$ and $E$).
  \end{minipage}
\end{equation}
\end{lemma}

\begin{proof}
$\{(\indicator_{E_{k}},T^{k},B^{k}_{R})\}_{k=0}^{\infty}$ is i.i.d.\ since
$\{(\CW_{k} \setminus \CW_{k+1},h|_{\CW_{k} \setminus \CW_{k+1}},\eta'_{k}|_{[0,\sigma_{k}]})\}_{k=0}^{\infty}$
is i.i.d.\ and $\indicator_{E_{k}},T^{k},B^{k}_{R}$ are measurable with respect to
$(\CW_{k} \setminus \CW_{k+1},h|_{\CW_{k} \setminus \CW_{k+1}},\eta'_{k}|_{[0,\sigma_{k}]})$
by~\eqref{eq:exploration_Esigma_wedge}, \eqref{eq:exploration_disk_N}, \eqref{eq:expression_top_length}
and~\eqref{eq:expression_bottom_lengths} for any $k \geq 0$, and therefore
$q_{n_{0}}$ is independent of $j$.

We first show that, as long as $c_{0} \in (0, c \alpha]$,
\begin{equation} \label{eq:right-length-change-SLLN-mean}
\mathbb{E}\bigl[ T^{0}-B^{0}_{R} \bigr] \in [cA^{-2/3}\log A,\infty)
  \quad \text{and} \quad
\mathbb{E}\bigl[ T^{0} \indicator_{E_{0}^{c}} \bigr] \leq \frac{3}{4} c A^{-2/3} \log A.
\end{equation}
Indeed, noting that $\epsilon_{0} = c_{0} A^{-2/3} \leq c \alpha A^{-2/3} \leq \alpha'/e$ and that
$\frac{1}{12}\log A\geq 2c'/c+\log(\alpha'/(c \alpha))$, we see from~\eqref{eq:length_change_mean_upper2}, which requires $\epsilon_{0}\leq \alpha'/e$,
and the calculations made in~\eqref{eq:length_change_mean_upper1} for $\mathbb{E}\bigl[T^{0}\indicator_{E_{0}^{c}}\bigr]$ that
\begin{align*}
\mathbb{E}\bigl[T^{0}\indicator_{E_{0}^{c}}\bigr]
  &= \mathbb{E}\bigl[T^{0}\indicator_{\{\sigma_{0}<1\}\setminus E_{0}}\bigr] + \mathbb{E}\bigl[T^{0}\indicator_{\{\sigma_{0}=1\}}\bigr]
  \leq \alpha^{-1}\epsilon_{0}\log(\alpha'/\epsilon_{0}) + 2c'A^{-2/3} \\
&\leq A^{-2/3} \bigl( \tfrac{2}{3}c\log A + c \log(\alpha'/(c \alpha)) + 2c' \bigr)
  \leq \tfrac{3}{4} c A^{-2/3} \log A.
\end{align*}

Next, for $\mathbb{E}\bigl[ T^{0}-B^{0}_{R} \bigr]$, note that
$T^{0}-B^{0}_{R}=L^{0}_{\sigma_{0}}+B^{0}_{L}+R^{0}_{\sigma_{0}}$
by~\eqref{eq:expression_top_length} and~\eqref{eq:expression_bottom_lengths} and that,
since $\bigl(\{L^{0}_{t/A}\}_{t \geq 0},\{R^{0}_{t/A}\}_{t \geq 0},A\sigma_{0}\bigr)$ and
$\bigl(\{A^{-2/3}X^{1}_{t}\}_{t \geq 0},\{A^{-2/3}X^{2}_{t}\}_{t \geq 0},\tau\wedge A\bigr)$
have the same law, so do $(L^{0}_{\sigma_{0}},B^{0}_{L},R^{0}_{\sigma_{0}},B^{0}_{R})$ and
$A^{-2/3} ( X^{1}_{\tau \wedge A}, -I^{1}_{\tau \wedge A}, X^{2}_{\tau \wedge A}, -I^{2}_{\tau \wedge A} )$.
Further, for $j=1,2$, for any $t\in[0,\infty)$ we have $\mathbb{E}[|X^{j}_{t}|]<\infty$
by \cite[Chapter~VII, Corollary~2-(i) and Chapter~VIII, Proposition~4]{bertoin1996levy}
and $\mathbb{E}[X^{j}_{t}]=0$ by the stationarity and the scaling property of $X^{j}$, and
hence the Markov property of $X^{j}$ (see, e.g., \cite[Chapter~I, Proposition~6]{bertoin1996levy})
implies that $X^{j}$ is a martingale with respect to the completed filtration generated by $X^{j}$.
It thus follows from the optional sampling theorem and the independence of $\{X^{j},\tau^{3-j}\}$
that $\mathbb{E}\bigl[ |X^{j}_{\tau \wedge A}| \bigr] < \infty$ and $\mathbb{E}\bigl[ X^{j}_{\tau \wedge A} \bigr] = 0$.
Also, recalling that $c$ is the constant from Proposition~\ref{prop:running-inf-mean-diverge},
by \cite[Chapter~VIII, Proposition~4]{bertoin1996levy}, the scaling property of $X^{1}$
and Proposition~\ref{prop:running-inf-mean-diverge} we have
\begin{equation*} 
\infty > A^{2/3} \mathbb{E}\bigl[ |I^{1}_{1}| \bigr]
  = \mathbb{E}\bigl[ |I^{1}_{A}| \bigr]
  \geq \mathbb{E}\bigl[ - I^{1}_{\tau \wedge A} \bigr]
  \geq \mathbb{E}\bigl[ - I^{1}_{\tau} \indicator_{ \{ \tau < A \} } \bigr]
  \geq c \log A.
\end{equation*}
Combining the facts mentioned above in this paragraph, we see that
$\mathbb{E}\bigl[ | T^{0} - B^{0}_{R} | \bigr] < \infty$ and that
$\mathbb{E}\bigl[ T^{0} - B^{0}_{R} \bigr]
  = A^{-2/3} \mathbb{E}\bigl[ - I^{1}_{\tau \wedge A} \bigr]
  \geq c A^{-2/3} \log A$,
proving~\eqref{eq:right-length-change-SLLN-mean}.

Turning to the proof of~\eqref{eq:right-length-change-SLLN-run_good_chunks_qn0_converge},
we define partial sums $\{S^{1}_{k}\}_{k\geq 1},\{S^{2}_{k}\}_{k\geq 1}$ of i.i.d.\ real
random variables by $S^{1}_{k}:=\sum_{i=1}^{k}(T^{i}-B^{i}_{R})$ and
$S^{2}_{k}:=\sum_{i=1}^{k}T^{i}\indicator_{E_{i}^{c}}$ for $k \in \mathbb{N}$,
so that $S_{0,k}=S^{1}_{k}-S^{2}_{k}$ by~\eqref{eq:right-length-change}. Then
since the law of $T^{0}-B^{0}_{R}$ is determined solely by $A$, the strong law of large
numbers together with the first half of~\eqref{eq:right-length-change-SLLN-mean} yields
\begin{equation} \label{eq:right-length-change-SLLN-run_good_chunks_qn0_converge_S1}
\lim_{n_{0}\to\infty}\mathbb{P}\biggl[ \inf_{ k \geq n_{0} } \frac{S^{1}_{k}}{k} \geq \frac{11}{12} c A^{-2/3} \log A \biggr]
  = 1 \qquad
  \begin{minipage}{125pt}
  with a rate of convergence determined solely by $A$.
  \end{minipage}
\end{equation}
On the other hand,
$\mathbb{E}\bigl[(T^{0})^{4}\bigr] \leq 24\alpha^{-4}\mathbb{E}\bigl[e^{\alpha T^{0}}\bigr] \leq 24\alpha^{-4}\alpha' < \infty$
by the choice of $\alpha,\alpha'$ specified just before and after~\eqref{eq:length_change_mean_upper2},
and hence for any $a\in(0,\infty)$ and any $k \in \mathbb{N}$, by Markov's and H\"{o}lder's inequalities we have
\begin{equation*}
\mathbb{P}\biggl[ \biggl| \frac{S^{2}_{k}}{k} - \mathbb{E}\bigl[ T^{0} \indicator_{E_{0}^{c}} \bigr] \biggr| \geq a \biggr]
  \leq \frac{\mathbb{E}\bigl[ | S^{2}_{k} - k \mathbb{E}[ T^{0} \indicator_{E_{0}^{c}} ] |^{4} \bigr]}{(ka)^{4}}
  \leq \frac{48}{k^{2}a^{4}} \mathbb{E}\bigl[(T^{0})^{4}\bigr]
  \leq \frac{1152\alpha'}{k^{2}a^{4}\alpha^{4}}.
\end{equation*}
From this inequality with $a = \frac{1}{12} c A^{-2/3} \log A$ and the second half of
\eqref{eq:right-length-change-SLLN-mean} we get
\begin{equation} \label{eq:right-length-change-SLLN-run_good_chunks_qn0_converge_S2}
\mathbb{P}\Biggl[ \bigcup_{k=n_{0}}^{\infty} \biggl\{ \frac{S^{2}_{k}}{k} \geq \frac{5}{6} c A^{-2/3} \log A \biggr\} \Biggr]
  \leq \frac{12^{7}\alpha'}{c^{4}\alpha^{4}} \frac{A^{8/3}}{(\log A)^{4}} \sum_{k=n_{0}}^{\infty} \frac{1}{k^{2}}.
\end{equation}
Now~\eqref{eq:right-length-change-SLLN-run_good_chunks_qn0_converge} follows by
$S_{0,k} = S^{1}_{k} - S^{2}_{k}$, \eqref{eq:right-length-change-SLLN-run_good_chunks_qn0_converge_S1}
and~\eqref{eq:right-length-change-SLLN-run_good_chunks_qn0_converge_S2}.
\end{proof}

\begin{lemma} \label{lem:right-length-positive-prob-positive}
Let $j \geq 0$, $l > j$, let $\{ S_{j,k} \}_{k=j}^{\infty}$ be as in
Lemma~\ref{lem:right-length-change-SLLN-run_good_chunks_qn0},
and define events $\widetilde{E}_{j,l},\widetilde{E}_{j}$ by
\begin{align} \label{eq:right-length-positive-finite}
\widetilde{E}_{j,l} &:= E_{j} \cap \{ T^{j} > 3 \epsilon_{0} \} \cap
  \bigcap\nolimits_{k=j+1}^{l} \{ S_{j,k-1} - B^{k}_{R} > (2(k-j)-3) \epsilon_{0} \}, \\
\widetilde{E}_{j} &:= E_{j} \cap \{ T^{j} > 3 \epsilon_{0} \} \cap
  \bigcap\nolimits_{k=j+1}^{\infty} \{ S_{j,k-1} - B^{k}_{R} > (2(k-j)-3) \epsilon_{0} \}.
\label{eq:right-length-positive}
\end{align}
Also let $z_{j} \in \partial(\CW_{j} \setminus \CW_{j+1}) \cap \partial \CW_{j+1}$ be such that the
quantum length of the part of $\partial(\CW_{j} \setminus \CW_{j+1}) \cap \partial \CW_{j+1}$
from its leftmost point to $z_{j}$ is $(2\epsilon_{0})\wedge T^{j}$. Then
\begin{align} \label{eq:right-boundary-avoided-finite}
\widetilde{E}_{j,l} &\subset \biggl\{
  \begin{minipage}{228pt}
  $\partial(\CW_{k} \setminus \CW_{k+1})$ does not intersect the part of
  $\partial \CW_{j+1}$ which is to the right of $z_{j}$ for any $k \in \mathbb{N}\cap(j,l]$
  \end{minipage}
  \biggr\}, \\
\widetilde{E}_{j} &\subset \biggl\{
  \begin{minipage}{215.2pt}
  $\partial(\CW_{k} \setminus \CW_{k+1})$ does not intersect the part of
  $\partial \CW_{j+1}$ which is to the right of $z_{j}$ for any $k>j$
  \end{minipage}
  \biggr\}
\label{eq:right-boundary-avoided}
\end{align}
and, provided $A^{-1} \vee c_{0}$ is small enough, there exists $\tilde{q} \in (0,1)$ determined solely by $A,c_{0}$ such that
\begin{equation} \label{eq:right-length-positive-prob-positive}
\mathbb{P}[\widetilde{E}_{j}] = \mathbb{P}[\widetilde{E}_{0}] \geq 1-\tilde{q}.
\end{equation}
\end{lemma}

\begin{proof}
It is clear that~\eqref{eq:right-boundary-avoided} follows from~\eqref{eq:right-boundary-avoided-finite}
by taking the intersection over $l \in \mathbb{Z}\cap(j,\infty)$. We show
\eqref{eq:right-boundary-avoided-finite} by a comparison argument similar to
\eqref{eq:SLE6_percolate_H_length_comparison}. Let $y_{k}$ denote the rightmost point of
$\partial(\CW_{k} \setminus \CW_{k+1}) \cap \partial \CW_{k}$ for each $k\geq 0$,
and fix any instance of the exploration for which $\widetilde{E}_{j,l}$ holds, so that the
quantum length of the part of $\partial(\CW_{j} \setminus \CW_{j+1}) \cap \partial \CW_{j+1}$
from $z_{j}$ to $y_{j}$ is $T^{j}-2\epsilon_{0} \in (\epsilon_{0},\infty)$.
Set $k_{0}:=j$, $V_{0}:=0$ and $k_{n}:=\inf\{ k > k_{n-1} \mid \indicator_{E_{k}} = 0 \}$
for $n \geq 1$. Also for each $n \geq 1$, set
\begin{equation} \label{eq:right-boundary-avoided-passed-chunks}
J_{n} := \biggl\{k \biggm|
	\begin{minipage}{303pt}
	$k \in \mathbb{Z} \cap (j,k_{n}) \setminus \{k_{i} \mid i \geq 1\}$,
	$\partial \CW_{k_{n}+1} \cap \partial \CN_{k}$ has at least two elements and is
	included in the part of $\partial \CW_{k_{n}+1}$ from $y_{k_{n}}$ to $\eta'_{k_{n}+1}(0)$
	\end{minipage}
	\biggr\}
\end{equation}
if $k_{n}<\infty$ and $J_{n}:=\emptyset$ if $k_{n}=\infty$,
let $T^{k,k_{n}}$ denote the quantum length of $\partial \CW_{k_{n}+1} \cap \partial \CN_{k}$
for $k \in J_{n}$, set $U_{n} := \sum_{k \in J_{n}} T^{k,k_{n}}$ and $V_{n} := \sum_{i=1}^{n}U_{i}$,
so that the definition of the exploration easily implies that
$T^{k,k_{n}} < \epsilon_{0}$ for any $k \in J_{n}$ by $\indicator_{E_{k_{n}}}=0$
(recall~\eqref{eq:bad_chunks_initial_point_half_plane}) and
that $J_{n} \cap J_{i} = \emptyset$ for any $i > n$. In particular,
$V_{n} \leq \epsilon_{0} \sum_{i=1}^{n} \# J_{i} \leq (k_{n}-j-n) \epsilon_{0}$
for any $n \geq 0$.

To show~\eqref{eq:right-boundary-avoided-finite} by induction, let $n \geq 0$ satisfy $k_{n}<l$
and suppose that $\partial(\CW_{k} \setminus \CW_{k+1})$ does not intersect the part of
$\partial \CW_{j+1}$ which is to the right of $z_{j}$ for any $k \in \mathbb{Z}\cap(j,k_{n}]$,
that $\eta'_{k_{n}+1}(0)$ is located to the left of $z_{j}$ in $\partial \CW_{k_{n}+1}$
and that the quantum length of the part of $\partial \CW_{k_{n}+1}$ from $\eta'_{k_{n}+1}(0)$ to $z_{j}$
is given by $2\epsilon_{0} + S_{j,k_{n}} - V_{n}$; note that this supposition holds
for $n=0$ since $\eta'_{j+1}(0)$ is the leftmost point of
$\partial(\CW_{j} \setminus \CW_{j+1}) \cap \partial \CW_{j}$ by $\indicator_{E_{j}}=1$.
Then an inductive argument on $k$ based on the definition of the exploration and the third part of
\eqref{eq:right-length-positive-finite} easily shows for any $k \in \mathbb{Z} \cap (k_{n},k_{n+1}\wedge l]$ that
$\partial(\CW_{k} \setminus \CW_{k+1})$ does not intersect the part of
$\partial \CW_{j+1}$ which is to the right of $z_{j}$ and that the quantum length
of the part of $\partial \CW_{k+1}$ from $y_{k}$ to $z_{j}$ is given by
$2\epsilon_{0} + S_{j,k-1} - B^{k}_{R} - V_{n}$ and hence greater than
$(2k-k_{n}-j+n-1) \epsilon_{0} \geq (n+1) \epsilon_{0}$ by~\eqref{eq:right-length-positive-finite}
and $V_{n} \leq (k_{n}-j-n)\epsilon_{0}$. Moreover,
suppose further that $k_{n+1} < l$. Then
since the definition of the exploration implies also that
the part of $\partial \CW_{k_{n+1}+1}$ from $y_{k_{n+1}}$ to $z_{j}$
involves only $\partial \CW_{k_{n+1}+1} \cap \partial \CN_{k}$ for
$k \in \mathbb{Z} \cap [j,k_{n+1}) \setminus \{k_{i} \mid i \geq 1\}$ and has quantum
length greater than $(2k_{n+1}-k_{n}-j+n-1) \epsilon_{0} \geq (k_{n+1}-j) \epsilon_{0}$,
there exists $j_{n+1} \in \mathbb{Z} \cap [j,k_{n+1}) \setminus \{k_{i} \mid i \geq 1\}$
such that $\partial \CW_{k_{n+1}+1} \cap \partial \CN_{j_{n+1}}$ has quantum length greater than
$\epsilon_{0}$. It thus follows from $\indicator_{E_{k_{n+1}}}=0$ and the way of
choosing $\eta'_{k_{n+1}+1}(0)$ (recall~\eqref{eq:bad_chunks_initial_point_half_plane})
that the quantum length of the part of $\partial \CW_{k_{n+1}+1}$ from $y_{k_{n+1}}$ to
$\eta'_{k_{n+1}+1}(0)$ is given by $U_{n+1}$, thereby that $\eta'_{k_{n+1}+1}(0)$
is located to the left of $z_{j}$ in $\partial \CW_{k_{n+1}+1}$ and that
the quantum length of the part of $\partial \CW_{k_{n+1}+1}$ from $\eta'_{k_{n+1}+1}(0)$ to $z_{j}$ is given by
$2\epsilon_{0} + S_{j,k_{n+1}-1} - B^{k_{n+1}}_{R} - V_{n} - U_{n+1} = 2\epsilon_{0} + S_{j,k_{n+1}} - V_{n+1}$.
This completes the induction and proves~\eqref{eq:right-boundary-avoided-finite} and thereby~\eqref{eq:right-boundary-avoided}.

\medskip

Next, for~\eqref{eq:right-length-positive-prob-positive} we begin by noting that,
since $\CW_{j},\CW_{0},\CW_{1}$ are weight-$2$ quantum wedges, $\CW_{1}$ is independent of 
$(\CW_{0} \setminus \CW_{1},h|_{\CW_{0} \setminus \CW_{1}},\eta'_{0}|_{[0,\sigma_{0}]})$
and $\indicator_{E_{0}},T^{0}$ are measurable with respect to
$(\CW_{0} \setminus \CW_{1},h|_{\CW_{0} \setminus \CW_{1}},\eta'_{0}|_{[0,\sigma_{0}]})$
in view of~\eqref{eq:exploration_Esigma_wedge}, \eqref{eq:exploration_disk_N} and
\eqref{eq:expression_top_length},
\begin{align}
\mathbb{P}[\widetilde{E}_{j}] &= \mathbb{P}[\widetilde{E}_{0}] \notag \\
&= \mathbb{P}\bigl[ E_{0} \cap \{ T^{0} > 3 \epsilon_{0} \} \bigr] \cdot
  \mathbb{P}\bigl[ \text{$S_{0,k-1} - B^{k}_{R} > (2k-3) \epsilon_{0}$ for any $k \geq 1$} \bigr].
\label{eq:right-length-positive-independent}
\end{align}
Therefore it suffices to prove that the last two probabilities in~\eqref{eq:right-length-positive-independent}
are bounded from below by positive constants determined solely by $A,c_{0}$
as long as $A^{-1} \vee c_{0}$ is small enough. To this end,
recall that $X^{1}, X^{2}$ are i.i.d.\ $3/2$-stable L\'{e}vy processes with only downward jumps
and starting from $0$, and that we have set $I^{j}_{t} := \inf_{0 \leq s \leq t} X^{j}_{s}$,
$\tau^{j} := \inf\{t \in [1,\infty) \mid X^{j}_{t} = I^{j}_{t}\}$ for $j=1,2$ and $\tau := \tau^{1} \wedge \tau^{2}$.
By $E_{0} \subset \{ \sigma_{0} < 1 \}$ and~\eqref{eq:expression_top_length},
\begin{equation} \label{eq:SLE6-chunk-top-long}
E_{0} \cap \{ T^{0} > 3 \epsilon_{0} \}
  = E_{0} \cap \bigl\{ L^{0}_{\sigma_{1/A,0}}-\inf\nolimits_{0\leq s\leq \sigma_{1/A,0}}L^{0}_{s} + R^{0}_{\sigma_{1/A,0}}-\inf\nolimits_{0\leq s\leq \sigma_{1/A,0}}R^{0}_{s} > 3 \epsilon_{0} \bigr\},
\end{equation}
where $\sigma_{1/A,0} := \sigma_{1/A}(\h,\qmeasure{h},\eta'_{0},0,\infty)$. Since
$\bigl(\{L^{0}_{t/A}\}_{t \geq 0},\{R^{0}_{t/A}\}_{t \geq 0},A\sigma_{1/A,0}\bigr)$ has the
same law as $\bigl(\{A^{-2/3}X^{1}_{t}\}_{t \geq 0},\{A^{-2/3}X^{2}_{t}\}_{t \geq 0},\tau\bigr)$,
we see in view of $\epsilon_{0}=c_{0}A^{-2/3}$ that
\begin{equation} \label{eq:SLE6-chunk-top-long-prob}
\begin{split}
&\mathbb{P}\bigl[ L^{0}_{\sigma_{1/A,0}}-\inf\nolimits_{0\leq s\leq \sigma_{1/A,0}}L^{0}_{s} + R^{0}_{\sigma_{1/A,0}}-\inf\nolimits_{0\leq s\leq \sigma_{1/A,0}}R^{0}_{s} > 3 \epsilon_{0} \bigr] \\
&=\mathbb{P}\bigl[ X^{1}_{\tau} - I^{1}_{\tau} + X^{2}_{\tau} - I^{2}_{\tau} > 3 c_{0} \bigr]
  \xrightarrow{c_{0} \downarrow 0} \mathbb{P}\bigl[ X^{1}_{\tau} - I^{1}_{\tau} + X^{2}_{\tau} - I^{2}_{\tau} > 0 \bigr]
  = 1,
\end{split}
\end{equation}
where the last equality follows by the independence of $X^{1},X^{2}$, \eqref{eq:Xj-Ij-fixed-time} and~\eqref{eqn:tautail}.
Moreover, the assumption on the Borel subset $E$ of $\MCPU{2}$ and~\eqref{eqn:tautail} together imply that
there exists $c'' \in (0,\infty)$ determined solely by the law of $(X^{1},X^{2})$ such that
\begin{equation} \label{eq:SLE6-chunk-E-prob-large}
\mathbb{P}[ E_{0} ]
  \geq ( 1 - \epsilon_{0} ) \mathbb{P}[ \sigma_{0}<1 ]
  = ( 1 - \epsilon_{0} ) \mathbb{P}[ \tau < A ]
  \geq ( 1 - c_{0} A^{-2/3}  ) ( 1 - c'' A^{-2/3} ).
\end{equation}
It thus follows that $\mathbb{P}\bigl[ E_{0} \cap \{ T^{0} > 3 \epsilon_{0} \} \bigr] \geq 1/2 > 0$,
by choosing first $c_{0} \in (0,1]$ small enough on the basis of~\eqref{eq:SLE6-chunk-top-long-prob}
so that $\mathbb{P}\bigl[ X^{1}_{\tau} - I^{1}_{\tau} + X^{2}_{\tau} - I^{2}_{\tau} > 3 c_{0} \bigr] \geq 3/4$,
then taking $A$ large enough on the basis of~\eqref{eq:SLE6-chunk-E-prob-large} so that
$\mathbb{P}[ E_{0} ] \geq 3/4$, and combining these with~\eqref{eq:SLE6-chunk-top-long}
and~\eqref{eq:SLE6-chunk-top-long-prob}.

To see that the last probability in~\eqref{eq:right-length-positive-independent} is bounded from below by a positive constant determined solely by $A,c_{0}$,
assume that $c_{0} \in (0, c \alpha]$ and that $A$ is large enough so that
$c \alpha A^{-2/3}\leq \alpha'/e$ and $\frac{1}{12}\log A\geq \bigl(2c'/c+\log(\alpha'/(c \alpha))\bigr)\vee(4\alpha)$.
Then $\frac{1}{12}cA^{-2/3}\log A \geq 4c\alpha A^{-2/3} \geq 4c_{0}A^{-2/3} = 4\epsilon_{0}$
and hence $q_{n_{0}}$ as in~\eqref{eq:right-length-change-SLLN-run_good_chunks_qn0} satisfies
$\mathbb{P}\bigl[ \inf_{ k \geq n_{0} } k^{-1} S_{0,k} > 3\epsilon_{0} \bigr] \geq q_{n_{0}}$
for any $n_{0} \in \mathbb{N}$. Also, since
$\mathbb{P}[ B^{k+1}_{R} \geq k^{9/10} ]
  \leq k^{-9/8} \mathbb{E}\bigl[ (B^{k+1}_{R})^{5/4} \bigr]
  \leq k^{-9/8} \mathbb{E}\bigl[ |I^{1}_{1}|^{5/4} \bigr]
  <\infty$
for any $k \in \mathbb{N}$ by~\eqref{eq:expression_bottom_lengths} and \cite[Chapter~VIII, Proposition~4]{bertoin1996levy},
for any $n_{0} \in \mathbb{N} \cap (\epsilon_{0}^{-10},\infty)$ we have
\begin{equation} \label{eq:right-length-change-bottom-right-small}
\mathbb{P}\biggl[ \sup_{k \geq n_{0}} \frac{B^{k+1}_{R}}{k} < \epsilon_{0} \biggr]
  \geq \mathbb{P}\Biggl[ \bigcap_{k \geq n_{0}} \{ B^{k+1}_{R} < k^{9/10} \} \Biggr]
  \geq 1 - \mathbb{E}\bigl[ |I^{1}_{1}|^{5/4} \bigr] \sum_{k=n_{0}}^{\infty}k^{-9/8}.
\end{equation}
Thus by~\eqref{eq:right-length-change-SLLN-run_good_chunks_qn0_converge} and
\eqref{eq:right-length-change-bottom-right-small} there exists $n_{0} \in \mathbb{N}$ determined solely by $A$ such that
$\mathbb{P}\bigl[ \inf_{ k \geq n_{0} } k^{-1} S_{0,k} > 3\epsilon_{0} \bigr] \geq q_{n_{0}} \geq \frac{7}{8}$
and $\mathbb{P}\bigl[ \sup_{k \geq n_{0}} k^{-1} B^{k+1}_{R} < \epsilon_{0} \bigr] \geq \frac{7}{8}$, 
and then since $\{ B^{k}_{R} \}_{k=0}^{\infty}$ is i.i.d.\ with the law determined solely by $A$,
we can further take $M \in (0,\infty)$ depending only on $A$ so that
$\mathbb{P}\bigl[ \max_{ 1 \leq k \leq n_{0} } k^{-1} \sum_{i=1}^{k} B^{i}_{R} \leq M \bigr] \geq \frac{3}{4}$.
It follows from these inequalities and $S_{0,k-1} - B^{k}_{R} \geq - \sum_{i=1}^{k} B^{i}_{R}$
for $k \in \mathbb{N}$ that
\begin{equation} \label{eq:right-length-change-SLLN1}
\mathbb{P}\biggl[ \inf_{ k \geq n_{0} } \frac{S_{0,k}}{k} > 3\epsilon_{0}, \,
  \sup_{ k \geq n_{0} } \frac{B^{k+1}_{R}}{k} < \epsilon_{0}, \,
  \min_{ 1 \leq k \leq n_{0} } \frac{ S_{0,k-1} - B^{k}_{R} }{ k } \geq -M \biggr]
    \geq \frac{1}{2}.
\end{equation}
On the other hand, let $c_{1} \in (0,\infty)$ be the constant from
Proposition~\ref{prop:stable-X1-I1large-I2small} and assume that
$A \geq (2/c_{1})^{3/2}$ and $c_{0} \leq 1$.
Then by $\epsilon_{0} = c_{0} A^{-2/3} \leq \frac{1}{2}c_{1}c_{0}$ and~\eqref{eq:stable-X1-I1large-I2small} we have
\begin{equation} \label{eq:stable-X1-I1large-I2small-SLLN-apply}
\begin{split}
&\mathbb{P}\bigl[ T^{0} \indicator_{E_{0}} > 4\epsilon_{0}, \, B^{0}_{R} < \epsilon_{0} \bigr] \\
&= \mathbb{P}\bigl[ E_{0} \cap \{ T^{0} > 4\epsilon_{0}, \, B^{0}_{R} < \epsilon_{0} \} \bigr] \\
&=\mathbb{P}\bigl[ \sigma_{0} < 1, \, T^{0} > 4\epsilon_{0}, \, B^{0}_{R} < \epsilon_{0} \bigr] 
  - \mathbb{P}\bigl[ \{ \sigma_{0} < 1, \, T^{0} > 4\epsilon_{0}, \, B^{0}_{R} < \epsilon_{0} \} \setminus E_{0} \bigr] \\
&\geq \mathbb{P}\bigl[ \tau < A, \, X^{1}_{\tau} - I^{1}_{\tau} > 4c_{0}, \, I^{2}_{\tau} > -c_{0} \bigr]
  - \mathbb{P}[ E_{0}^{c} \mid \sigma_{0} < 1] \\
&\geq c_{1}c_{0} - \epsilon_{0}
  \geq \tfrac{1}{2}c_{1}c_{0}.
\end{split}
\end{equation}
Setting $n_{1} := \min\bigl(\mathbb{N} \cap [(2\epsilon_{0}^{-1}M+4)n_{0},\infty)\bigr)$,
which is determined solely by $A,c_{0}$ as a function of $n_{0},M,\epsilon_{0}=c_{0}A^{-2/3}$,
and recalling~\eqref{eq:right-length-change} and that $\{(\indicator_{E_{k}},T^{k},B^{k}_{R})\}_{k=0}^{\infty}$ is i.i.d.,
from~\eqref{eq:right-length-change-SLLN1} and~\eqref{eq:stable-X1-I1large-I2small-SLLN-apply} now we obtain
\begin{align}
&\mathbb{P}\bigl[ \text{$S_{0,k-1} - B^{k}_{R} > (2k-3) \epsilon_{0}$ for any $k \geq 1$} \bigr] \notag \\
&\geq \mathbb{P}\!\!\left[
  \begin{minipage}{362pt}
    $\inf\limits_{ k \geq n_{0} } \dfrac{S_{n_{1},n_{1}+k}}{k} > 3\epsilon_{0}$,
    $\sup\limits_{ k \geq n_{0} } \dfrac{B^{n_{1}+k+1}_{R}}{k} < \epsilon_{0}$,
    $\min\limits_{ 1 \leq k \leq n_{0} } \dfrac{ S_{n_{1},n_{1}+k-1} -  B^{n_{1}+k}_{R} }{k} \geq -M$,
    $T^{k} \indicator_{E_{k}} > 4\epsilon_{0}$ and
    $B^{k}_{R} < \epsilon_{0}$ for any $k \in \{1,\ldots,n_{1}\}$
  \end{minipage}
  \right] \notag \\
&\geq \frac{(c_{1}c_{0}/2)^{n_{1}}}{2} > 0
\label{eq:right-length-positive-prob-positive-end-of-proof}
\end{align}
provided $A^{-1} \vee c_{0}$ is small enough. Here we have the first inequality in
\eqref{eq:right-length-positive-prob-positive-end-of-proof} since the event in the
first line of~\eqref{eq:right-length-positive-prob-positive-end-of-proof} is seen to
hold on the event in the second line as follows: for any $k \in \{1,\ldots,n_{1}\}$ we have
$S_{0,k-1} - B^{k}_{R} > 3\epsilon_{0}(k-1) - \epsilon_{0} \geq (2k-3)\epsilon_{0}$,
for any $k \in \mathbb{N}\cap(n_{1},n_{1}+n_{0}]$ we see from $n_{1} \geq (2\epsilon_{0}^{-1}M+4)n_{0}$ that
\begin{align*}
S_{0,k-1}-B^{k}_{R}
  &= S_{0,n_{1}} + S_{n_{1},k-1} - B^{k}_{R}
  > 3\epsilon_{0}n_{1} - M(k-n_{1}) \\
&\geq 2\epsilon_{0}n_{1}+\epsilon_{0}(2\epsilon_{0}^{-1}M+4)n_{0}-Mn_{0}
  >2(n_{1}+n_{0})\epsilon_{0}
  >(2k-3)\epsilon_{0},
\end{align*}
and for any $k \in \mathbb{N}\cap(n_{1}+n_{0},\infty)$ we obtain
\begin{align*}
S_{0,k-1}-B^{k}_{R}
  &= S_{0,n_{1}} + S_{n_{1},k-1} - B^{k}_{R}
  > 3\epsilon_{0}n_{1} + (3\epsilon_{0}-\epsilon_{0})(k-n_{1}-1)
  >(2k-3)\epsilon_{0}.
\end{align*}
Thus by~\eqref{eq:right-length-positive-independent}, \eqref{eq:right-length-positive-prob-positive-end-of-proof}
and the discussion following~\eqref{eq:SLE6-chunk-E-prob-large} we conclude that
$\mathbb{P}[\widetilde{E}_{j}] = \mathbb{P}[\widetilde{E}_{0}] \geq \frac{1}{4}(c_{1}c_{0}/2)^{n_{1}}$
provided $A^{-1} \vee c_{0}$ is small enough, proving~\eqref{eq:right-length-positive-prob-positive}.
\end{proof}

\begin{lemma} \label{lem:run_good_chunks_decay_fast}
Let $u \in [2,\infty)$. For each $N \in \mathbb{N}$, let $\widehat{E}_{N}=\widehat{E}_{u,N}$ be the event that there exists
$0 \leq n \leq N$ so that for some $m \in \mathbb{N} \cap (u,\infty)$
and $0 \leq j_{1} < \cdots < j_{m} \leq n$ the following hold:
\begin{enumerate}[(i)]
\item\label{it:chunks_good} $E_{j_{i}}$ occurs for any $1 \leq i \leq m$.
\item\label{it:tops_exposed} The boundary length of $\mathcal{I}_{j_{i},n}:=\partial \CN_{j_{i}} \cap \partial \CW_{n+1}$ is in $(0,\epsilon_0)$ for any $1 \leq i \leq m$.
\item\label{it:adjacent} $\mathcal{I}_{j_{i},n} \cap \mathcal{I}_{j_{i+1},n} \not= \emptyset$
  and $\mathcal{I}_{j_{i+1},n}$ is located to the left of $\mathcal{I}_{j_{i},n}$ in $\partial \CW_{n+1}$
  for any $1 \leq i \leq m-1$.
\end{enumerate}
(Note that $\{\mathcal{I}_{j_{i},n}\}_{i=1}^{m}$ are closed subintervals of $\partial \CW_{n+1}$
with disjoint interiors under~\eqref{it:chunks_good} and~\eqref{it:tops_exposed} by Proposition~\ref{prop:SLE6-chunk-Jordan-H}.)
Assume further that $A^{-1} \vee c_{0}$ is small enough so that
\eqref{eq:right-length-positive-prob-positive} holds, let
$\tilde{q}$ be as in~\eqref{eq:right-length-positive-prob-positive} and
set $\tilde{c} := \frac{1}{2}\log(\tilde{q}^{-1})$. Then for any $N \in \mathbb{N}$,
\begin{equation}\label{eq:run_good_chunks_decay_fast}
\mathbb{P}[\widehat{E}_{N}] \leq \tilde{q}^{-1} N \exp(-\tilde{c}u).
\end{equation}

\end{lemma}

\begin{proof}
Let $j \geq 0$. On the event $E_{j}$, we consider a random tree rooted at $j$ with the
set of vertices $\mathcal{T}_{j} \subset \mathbb{Z} \cap [j,\infty)$, defined as follows.
Given $k \in \mathcal{T}_{j}$, we say that $l \in \mathbb{N} \cap (k,\infty)$ is a child
of $k$ in the tree if the following three conditions are satisfied:
\begin{enumerate}[(I)]
\item\label{it:chunks_good_proof} $E_{l}$ occurs.
\item\label{it:tops_exposed_proof} The boundary length of $\mathcal{I}_{k,l}:=\partial \CN_{k} \cap \partial \CW_{l+1}$ is in $(0,\epsilon_0)$.
\item\label{it:adjacent_proof} $\mathcal{I}_{k,l} \cap \partial \CN_{l} \not= \emptyset$
  and $\partial \CN_{l}$ is located to the left of $\mathcal{I}_{k,l}$ in $\partial \CW_{l+1}$.
\end{enumerate}
Recall that $T^{j}$ denotes the quantum length of the top $\partial \CN_{j} \cap \CW_{j}$
of $\CW_{j} \setminus \CW_{j+1}$. We claim that, provided $A^{-1} \vee c_{0}$ is small enough
so that~\eqref{eq:right-length-positive-prob-positive} holds, the conditional law
of the cardinality $\# \mathcal{T}_{j}$ of $\mathcal{T}_{j}$ given $E_{j}$ and $T^{j}$ is
stochastically dominated by a geometric distribution whose success probability is explicit
in $\tilde{q}$, with $\tilde{q}$ as in~\eqref{eq:right-length-positive-prob-positive}.
Indeed, for each $k \geq 0$ let $\mathescr{G}_{k}$ denote the $\sigma$-algebra generated by
$\{(\CW_{i} \setminus \CW_{i+1},h|_{\CW_{i} \setminus \CW_{i+1}},\eta'_{i}|_{[0,\sigma_{i}]},\widetilde{L}_{i})\}_{0 \leq i \leq k}$,
and for $i \geq 0$ let $m_{i}$ denote the $(i+1)$-th smallest element of
$\mathcal{T}_{j} \subset \mathbb{Z} \cap [j,\infty)$ on the event $\{ \# \mathcal{T}_{j} > i \}$
and set $m_{i}:=\infty$ on $\{ \# \mathcal{T}_{j} \leq i \} \cup E_{j}^{c}$, so that $m_{i}$
is a $\{ \mathescr{G}_{k} \}$-stopping time. Let $i \geq 0$. Then setting
$\widetilde{E}_{m_{i}+1}:=\bigcup_{k \in \mathbb{N}\cup\{0\}}\bigl(\widetilde{E}_{k+1}\cap\{m_{i}=k\}\bigr)$,
on $\widetilde{E}_{m_{i}+1}$ we have $T^{m_{i}+1} > 3 \epsilon_{0}$ by~\eqref{eq:right-length-positive},
hence see from~\eqref{eq:right-boundary-avoided} and the definition of the exploration
that~\eqref{it:tops_exposed_proof} with $k=m_{i}+1$ and~\eqref{it:adjacent_proof} with any
$k\in\{m_{0},\ldots,m_{i}\}$ fail to hold for any $l\in\mathbb{N}\cap[m_{i}+2,\infty)$,
and thus obtain $\mathcal{T}_{j}\subset\{m_{0},\ldots,m_{i},m_{i}+1\}$ and in particular
$\# \mathcal{T}_{j} \leq i+2$. Therefore
\begin{equation} \label{eq:run_good_chunks_tree_size_bound_proof}
\begin{split}
\mathbb{P}[ m_{i+2} < \infty \mid \mathescr{G}_{m_{i}} ]
  &\leq \indicator_{ \{ m_{i} < \infty \} } - \mathbb{P}\bigl[ \widetilde{E}_{m_{i}+1} \bigm| \mathescr{G}_{m_{i}} \bigr] \\
&= \bigl( 1 - \mathbb{P}[\widetilde{E}_{0}] \bigr) \indicator_{ \{ m_{i} < \infty \} }
  \leq \tilde{q} \indicator_{ \{ m_{i} < \infty \} },
\end{split}
\end{equation}
where the equality holds by the facts that
$\widetilde{E}_{m_{i}+1} \cap \{m_{i}=k\} = \widetilde{E}_{k+1} \cap \{m_{i}=k\}$,
$\widetilde{E}_{k+1}$ is independent of $\mathescr{G}_{k}$ and
$\mathbb{P}[\widetilde{E}_{k+1}]=\mathbb{P}[\widetilde{E}_{0}]$ for any $k \geq 0$. It thus follows
by an induction on $i$ based on~\eqref{eq:run_good_chunks_tree_size_bound_proof} that for any $i \geq 0$,
\begin{equation} \label{eq:run_good_chunks_tree_size_bound}
\mathbb{P}[ \# \mathcal{T}_{j} > 2i \mid \mathescr{G}_{j} ] \indicator_{ E_{j} }
  = \mathbb{P}[ m_{2i} < \infty \mid \mathescr{G}_{j} ] \indicator_{ E_{j} }
  \leq \tilde{q}^{i} \indicator_{E_{j}},
\end{equation}
which implies the desired claim.

Now we can conclude the assertion as follows. Let $N \in \mathbb{N}$, and for each $j \geq 0$
let $\widehat{E}_{N,j}$ be the event that~\eqref{it:chunks_good}, \eqref{it:tops_exposed}
and~\eqref{it:adjacent} as in the statement hold for some $m \in \mathbb{N} \cap (u,\infty)$
and $j = j_{1} < \cdots < j_{m}\leq n$, so that $\widehat{E}_{N} \subset \bigcup_{j=0}^{N-1}\widehat{E}_{N,j}$.
Then since $\{ j_{1}, \ldots, j_{m} \} \subset \mathcal{T}_{j}$ for such $m$ and $j_{1},\ldots,j_{m},n$
by the definition of the tree $\mathcal{T}_{j}$ given at the beginning of the previous paragraph,
we have $\widehat{E}_{N,j} \subset E_{j} \cap \{ \# \mathcal{T}_{j} > u \}$, hence
$\mathbb{P}[ \widehat{E}_{N,j} ]
  \leq \mathbb{P}\bigl[ E_{j} \cap \{ \# \mathcal{T}_{j} > u \} \bigr]
  \leq \tilde{q}^{-1} \exp(-\tilde{c}u)$
by~\eqref{eq:run_good_chunks_tree_size_bound}, where $\tilde{c} := \frac{1}{2}\log(\tilde{q}^{-1})$, and thus
$\mathbb{P}[ \widehat{E}_{N} ]
  \leq \sum_{j=0}^{N-1} \mathbb{P}[ \widehat{E}_{N,j} ]
  \leq \tilde{q}^{-1} N \exp(-\tilde{c}u)$.
\end{proof}

\begin{proof}[Proof of Proposition~\ref{prop:good_chunks_percolate_half_plane}-\eqref{it:bad_cluster_size}]
Assume that $A^{-1} \vee c_{0}$ is small enough so that Proposition~\ref{prop:good_chunks_percolate_half_plane}-\eqref{it:move_left}
and~\eqref{eq:right-length-positive-prob-positive} hold with $c_{1},c_{2},\tilde{q}$ as stated in these places.
Recalling that $\CW = (\h,h,0,\infty)$ is the original quantum wedge on which the exploration is defined,
set $\partial_{\mathrm{L}} \CW := (-\infty,0) \subset \partial \CW$.
Let $\mathescr{G}_{k}$ denote the $\sigma$-algebra generated by
$\{(\CW_{i} \setminus \CW_{i+1},h|_{\CW_{i} \setminus \CW_{i+1}},\eta'_{i}|_{[0,\sigma_{i}]},\widetilde{L}_{i})\}_{0 \leq i \leq k}$
for each $k \geq 0$, set $m_{0} := 0$ and inductively
$m_{j} := \inf\{ i > m_{j-1} \mid \partial \CN_{i} \cap \partial_{\mathrm{L}} \CW \not= \emptyset \}$
for each $j \geq 1$, so that for any $j \geq 0$, $m_{j}$ is a $\{\mathescr{G}_{k}\}$-stopping time and
Proposition~\ref{prop:good_chunks_percolate_half_plane}-\eqref{it:move_left} easily implies that $m_{j}<\infty$ a.s.
Also let $N \in \mathbb{N}$ and set $J_{N} := \min\{ j \geq 0 \mid m_{j} \geq N \}$,
so that $J_{N} \leq N$. Then we easily see from the definition of the exploration that
for any $j \geq 0$ and any $i \in \mathbb{Z} \cap [ m_{j}, m_{j+1} ]$
there exist $i_{1},\ldots,i_{n} \in \mathbb{Z} \cap [ m_{j}, m_{j+1} ]$
with $n \leq m_{j+1} - m_{j}$ and $i_{1} = i$ such that
$\partial \CN_{i_{n}} \cap \partial_{\mathrm{L}} \CW \neq \emptyset$ and
$\partial \CN_{i_{k}} \cap \partial \CN_{i_{k+1}} \neq \emptyset$ for each $1 \leq k \leq n-1$,
i.e., the number of adjacent chunks necessary to get from $\CN_{i}$ to $\partial_{\mathrm{L}} \CW$
is at most $m_{j+1} - m_{j}$. It therefore suffices to bound
$\max_{ 0 \leq j < J_{N} }(m_{j+1}-m_{j})$.

Let $\widetilde{E}_{j}$ be as in~\eqref{eq:right-length-positive} for $j \geq 0$, let $u \in [2,\infty)$
and let $\widehat{E}_{N} = \widehat{E}_{u,N}$ be as in Lemma~\ref{lem:run_good_chunks_decay_fast}.
Let $j \geq 0$, define $\mathescr{G}_{m_{j}}$-measurable random variables $R_{j},t_{j}$
by $R_{j}:=0$ on $\{m_{j}=\infty\}$,
\begin{equation} \label{eq:run_good_chunks_max_length}
R_{j} := 1 + \max\biggl\{ m \geq 1 \biggm|
  \begin{minipage}{240.5pt}
 \eqref{it:chunks_good}, \eqref{it:tops_exposed} and~\eqref{it:adjacent}
  in the statement of Lemma~\ref{lem:run_good_chunks_decay_fast} hold with
  $n=m_{j}$ for some $0 \leq j_{1} < \cdots < j_{m} \leq m_{j}$
  \end{minipage}
  \biggr\}
\end{equation}
on $\{m_{j}<\infty\}$ ($\max \emptyset := 0$) and
$t_{j} := T^{m_{j}} \indicator_{E_{m_{j}}^{c}} := \sum_{k=0}^{\infty} T^{k} \indicator_{ E_{k}^{c} \cap \{ m_{j} = k \} }$,
and set $n_{u} := \max\bigl( \mathbb{N} \cap [ 1, u ] \bigr)$ and
$\widetilde{n}_{u} := (2\epsilon_{0}/c_{1})n_{u}$.

We claim that
\begin{equation} \label{eq:bad_cluster_size_claim1}
\mathbb{P}\bigl[ m_{j+1} - m_{j} > \widetilde{n}_{u} R_{j} \bigm| \mathescr{G}_{m_{j}} \bigr]
  \indicator_{ \{ m_{j} < N, \, t_{j} \leq \epsilon_{0} n_{u} \} }
  \leq (\tilde{q}^{n_{u}-1}+e^{-c_{2}\widetilde{n}_{u}}) \indicator_{ \{ m_{j} < N, \, t_{j} \leq \epsilon_{0} n_{u} \} }.
\end{equation}
To see this, set $\tau_{j,0} := m_{j} -1 $ and inductively
$\tau_{j,n} := \inf\bigl\{ k > \tau_{j,n-1} \bigm| \eta'_{k+1}(0) \not\in \bigcup_{i=m_{j}}^{k} \partial \CN_{i} \bigr\}$
for each $n \geq 1$, so that $\tau_{j,n}$ is a $\{ \mathescr{G}_{k} \}$-stopping time for any $n \geq 1$;
note that $\tau_{j,n}$ is the $n$-th smallest $k \geq m_{j}$ such that
$E_{k}$ does not occur and the process of deciding the location of $\eta'_{k+1}(0)$
according to~\eqref{eq:bad_chunks_initial_point_half_plane} involves skipping
some of the intervals $\{ \partial \CW_{k+1} \cap \partial \CN_{i} \}_{0 \leq i < m_{j}}$
in $\partial \CW_{k+1}$ whose quantum lengths are in $(0,\epsilon_{0})$.
Then for any $n \geq 1$ we have
$\{ \tau_{j,n+1} < \infty \} \subset \bigcup_{k=0}^{\infty}( \{ \tau_{j,n} = k \} \setminus \widetilde{E}_{k+1})$
by~\eqref{eq:right-length-positive} and~\eqref{eq:right-boundary-avoided}, hence
$\mathbb{P}[ \tau_{j,n+1} < \infty \mid \mathescr{G}_{\tau_{j,n}} ]
  \leq \tilde{q} \indicator_{ \{ \tau_{j,n} < \infty \} }$
by the independence of $\widetilde{E}_{k+1}$ and $\mathescr{G}_{k}$ and
\eqref{eq:right-length-positive-prob-positive} with $j=k+1$ for $k \geq 0$, thus
$\mathbb{P}[ \tau_{j,n} < \infty \mid \mathescr{G}_{m_{j}} ]
  \leq \tilde{q}^{n-1} \indicator_{ \{ m_{j} < \infty \} }$
and in particular
$\mathbb{P}[ \tau_{j,n_{u}} < \infty \mid \mathescr{G}_{m_{j}} ]
  \leq \tilde{q}^{n_{u}-1} \indicator_{ \{ m_{j} < \infty \} }$.
Therefore, to show~\eqref{eq:bad_cluster_size_claim1} it suffices to prove instead that
\begin{equation} \label{eq:bad_cluster_size_claim2}
\mathbb{P}\bigl[ m_{j+1} - m_{j} > \widetilde{n}_{u} R_{j}, \, \tau_{j,n_{u}} = \infty \bigm| \mathescr{G}_{m_{j}} \bigr] \indicator_{ \{ m_{j} < N, \, t_{j} \leq \epsilon_{0} n_{u} \} }
  \leq e^{-c_{2}\widetilde{n}_{u}} \indicator_{ \{ m_{j} < N, \, t_{j} \leq \epsilon_{0} n_{u} \} }.
\end{equation}
For this purpose, for each $n \geq 1$ let $s_{j,n}$ denote the quantum length of the part of
$\partial \CW_{\tau_{j,n}+1}$ between $\eta'_{\tau_{j,n}+1}(0)$ and the rightmost point of
$\bigcup_{i=m_{j}}^{\tau_{j,n}}\partial \CN_{i} \cap \partial \CW_{\tau_{j,n}+1}$,
so that $s_{j,n} \leq R_{j} \epsilon_{0}$ by~\eqref{eq:bad_chunks_initial_point_half_plane} and
\eqref{eq:run_good_chunks_max_length}. Then since, conditionally on $\{ m_{j} < \infty \}$,
\begin{equation*}
\bigl\{ \bigl( \CW_{i+1+m_{j}} \setminus \CW_{i+2+m_{j}}, h|_{\CW_{i+1+m_{j}} \setminus \CW_{i+2+m_{j}}}, \eta'_{i+1+m_{j}}|_{[0,\sigma_{i+1+m_{j}}]}, \widetilde{L}_{m_{j},i+1+m_{j}} \bigr) \bigr\}_{i=0}^{\infty},
\end{equation*}
where $\widetilde{L}_{m_{j},i+1+m_{j}}:=\widetilde{L}_{i+1+m_{j}}-\sum_{n=1}^{\infty}s_{j,n}\indicator_{\{\tau_{j,n}=i+1+m_{j}\}}$,
is independent of $\mathescr{G}_{m_{j}}$ and has the same law as
$\{(\CW_{i} \setminus \CW_{i+1},h|_{\CW_{i} \setminus \CW_{i+1}},\eta'_{i}|_{[0,\sigma_{i}]},\widetilde{L}_{i})\}_{i=0}^{\infty}$,
it follows that, conditionally on $\{ m_{j} < \infty \}$,
$\bigl\{ L_{k+1+m_{j}} - L_{1+m_{j}} - \sum_{n=1}^{\infty} s_{j,n} \indicator_{ \{ m_{j} < \tau_{j,n} \leq k+m_{j} \} } \bigr\}_{k=0}^{\infty}$
is independent of $\mathescr{G}_{m_{j}}$ and has the same law as $\{ L_{k} \}_{k=0}^{\infty}$.
On the other hand, noting that $t_{j} + s_{j,1} \indicator_{ \{ \tau_{j,1} = m_{j} \} }$ is the quantum length
$\widetilde{L}_{m_{j}}$ of the part of $\partial \CW_{1+m_{j}}$ between $\eta_{1+m_{j}}'(0)$ and
the leftmost point of $\partial(\CW_{m_{j}} \setminus \CW_{1+m_{j})} \cap \partial \CW_{m_{j}}$,
we see from the definition of $m_{j},m_{j+1}$ that the quantum length of
$\partial(\CW_{0} \setminus \CW_{1+m_{j}}) \cap \partial_{\mathrm{L}} \CW$
is given by $-(L_{1+m_{j}}-t_{j}-s_{j,1}\indicator_{\{\tau_{j,1}=m_{j}\}})$ and that
on the event $\{ m_{j+1} - m_{j} > \widetilde{n}_{u} R_{j} \}$ we have
$\partial(\CW_{0} \setminus \CW_{\widetilde{n}_{u}R_{j}+1+m_{j}}) \cap \partial_{\mathrm{L}} \CW
  = \partial(\CW_{0} \setminus \CW_{1+m_{j}}) \cap \partial_{\mathrm{L}} \CW$
and hence $L_{\widetilde{n}_{u}R_{j}+1+m_{j}} \geq L_{1+m_{j}} - t_{j} - s_{j,1} \indicator_{ \{ \tau_{j,1} = m_{j} \} }$.
Recalling the $\mathescr{G}_{m_{j}}$-measurability of $R_{j}$ and $t_{j} = T^{m_{j}} \indicator_{E_{m_{j}}^{c}}$,
from the facts in the last two sentences, $s_{j,n} \leq R_{j} \epsilon_{0}$ and
Proposition~\ref{prop:good_chunks_percolate_half_plane}-\eqref{it:move_left} we get
\begin{align*}
&\mathbb{P}\bigl[ m_{j+1} - m_{j} > \widetilde{n}_{u} R_{j}, \, \tau_{j,n_{u}} = \infty \bigm| \mathescr{G}_{m_{j}} \bigr] \indicator_{ \{ m_{j} < N, \, t_{j} \leq \epsilon_{0} n_{u} \} } \\
&\leq \mathbb{P}\bigl[ L_{\widetilde{n}_{u}R_{j}+1+m_{j}} \geq L_{1+m_{j}} - t_{j} - s_{j,1} \indicator_{ \{ \tau_{j,1} = m_{j} \} }, \, \tau_{j,n_{u}} = \infty \bigm| \mathescr{G}_{m_{j}} \bigr] \indicator_{ \{ m_{j} < N, \, t_{j} \leq \epsilon_{0} n_{u} \} } \\
&\leq \mathbb{P}\biggl[ \begin{minipage}{250pt} $L_{\widetilde{n}_{u}R_{j}+1+m_{j}} - L_{1+m_{j}} - \sum_{n=1}^{\infty} s_{j,n} \indicator_{ \{ m_{j} < \tau_{j,n} \leq \widetilde{n}_{u}R_{j}+m_{j} \} }$ \\ $\geq - R_{j}\epsilon_{0}n_{u} - t_{j}$ \end{minipage} \biggm| \mathescr{G}_{m_{j}} \biggr] \indicator_{ \{ m_{j} < N, \, t_{j} \leq \epsilon_{0} n_{u} \} }  \\
&= \bigl( \mathbb{P}[ L_{k} \geq - (\epsilon_{0}n_{u}/\widetilde{n}_{u}) k - t ]|_{ (k,t) = (\widetilde{n}_{u}R_{j},t_{j}) } \bigr) \indicator_{ \{ m_{j} < N, \, t_{j} \leq \epsilon_{0} n_{u} \} } \\
&\leq \bigl( \mathbb{P}[ L_{k} \geq - (2\epsilon_{0}n_{u}/\widetilde{n}_{u}) k ]|_{ k = \widetilde{n}_{u}R_{j} } \bigr) \indicator_{ \{ m_{j} < N, \, t_{j} \leq \epsilon_{0} n_{u} \} } \\
&= \bigl( \mathbb{P}[ L_{k} \geq - c_{1}k ]|_{ k = \widetilde{n}_{u}R_{j} } \bigr) \indicator_{ \{ m_{j} < N, \, t_{j} \leq \epsilon_{0} n_{u} \} }
  \leq e^{-c_{2}\widetilde{n}_{u}} \indicator_{ \{ m_{j} < N, \, t_{j} \leq \epsilon_{0} n_{u} \} },
\end{align*}
proving~\eqref{eq:bad_cluster_size_claim2} and thereby~\eqref{eq:bad_cluster_size_claim1}.

Finally, we give similar upper bounds on $\mathbb{P}\bigl[ m_{j} < N, \, R_{j} > u + 1 \bigr]$
and $\mathbb{P}\bigl[ m_{j} < N, \, t_{j} > \epsilon_{0}n_{u} \bigr]$ to conclude the proof.
Indeed, $\{ m_{j} < N, \, R_{j} > u + 1 \} \subset \widehat{E}_{N}$
by~\eqref{eq:run_good_chunks_max_length} and hence
\begin{equation} \label{eq:bad_cluster_size_claim3}
\mathbb{P}\bigl[ m_{j} < N, \, R_{j} > u + 1 \bigr]
  \leq \mathbb{P}[ \widehat{E}_{N} ]
  \leq \tilde{q}^{-1} N \exp(-\tilde{c}u)
\end{equation}
by Lemma~\ref{lem:run_good_chunks_decay_fast}, where $\tilde{c} := \frac{1}{2}\log(\tilde{q}^{-1})$.
For the latter, noting that $t_{j} = T^{m_{j}} \indicator_{E_{m_{j}}^{c}} \leq \max_{0 \leq k < N} T^{k}$
on the event $\{ m_{j} < N \}$, that $\{ T^{k} \}_{0 \leq k < N}$ is i.i.d.\ and that $T^{0}$ is
stochastically dominated by $Z^{1}+Z^{2}$ with $Z^{i}:=\sup_{0 \leq t \leq 1}(X^{i}_{t}-I^{i}_{t})$
for $i=1,2$ by~\eqref{eq:expression_top_length} and $\sigma_{0}\leq 1$,
we see from Proposition~\ref{prop:stable-reflected-exp-moment} that
\begin{equation} \label{eq:bad_cluster_size_claim4}
\begin{split}
&\mathbb{P}\bigl[ m_{j} < N, \, t_{j} > \epsilon_{0}n_{u} \bigr]
  \leq \mathbb{P}\bigl[ \max\nolimits_{0 \leq k < N} T^{k} > \epsilon_{0}n_{u} \bigr]
  \leq N\, \mathbb{P}\bigl[ T^{0} > \epsilon_{0}n_{u} \bigr] \\
&\leq N \, \mathbb{P}\bigl[ Z^{1} + Z^{2} > \epsilon_{0}n_{u} \bigr]
  \leq 2N \,  \mathbb{P}\bigl[ Z^{1} > \tfrac{1}{2}\epsilon_{0}n_{u} \bigr]
  \leq 2c_{3} N \exp( - \tfrac{1}{2} c_{4} \epsilon_{0} n_{u} )
\end{split}
\end{equation}
for some $c_{3},c_{4} > 0$ determined solely by the law of $X^{1}$.
Combining~\eqref{eq:bad_cluster_size_claim3} and~\eqref{eq:bad_cluster_size_claim4}
with~\eqref{eq:bad_cluster_size_claim1}, we now obtain
\begin{align}
&\mathbb{P}\bigl[ \max\nolimits_{ 0 \leq j < J_{N} }(m_{j+1}-m_{j}) > (3\epsilon_{0}/c_{1})u^{2} \bigr] \notag \\
&\leq \mathbb{P}\bigl[ \max\nolimits_{ 0 \leq j < J_{N} }(m_{j+1}-m_{j}) > \widetilde{n}_{u} ( u + 1 ) \bigr] \notag \\
&\leq \sum_{j=0}^{N-1} \mathbb{P}\bigl[ m_{j} < N, \, m_{j+1}-m_{j} > \widetilde{n}_{u} ( u + 1 ) \bigr] \notag \\
&\leq \sum_{j=0}^{N-1} \bigl( \mathbb{P}\bigl[ m_{j} < N, \, m_{j+1}-m_{j} > \widetilde{n}_{u} R_{j} \bigr] + \mathbb{P}\bigl[ m_{j} < N, \, R_{j} > u + 1 \bigr] \bigr) \notag \\
&\leq \sum_{j=0}^{N-1} \bigl( \mathbb{E}\bigl[ \indicator_{ \{ m_{j} < N, \, t_{j} \leq \epsilon_{0} n_{u}, \, m_{j+1}-m_{j} > \widetilde{n}_{u} R_{j} \} } \bigr] + \mathbb{P}\bigl[ m_{j} < N, \, t_{j} > \epsilon_{0} n_{u} \bigr] + \mathbb{P}[ \widehat{E}_{N} ] \bigr) \notag \\
&\leq N \bigl( \tilde{q}^{n_{u}-1}+e^{-c_{2}\widetilde{n}_{u}} + 2c_{3} N \exp( - \tfrac{1}{2} c_{4} \epsilon_{0} n_{u} ) + \tilde{q}^{-1} N \exp(-\tilde{c}u) \bigr) \notag \\
&\leq c_{5} N^{2} \exp(-c_{6}u),
\label{eq:bad_cluster_size_conclusion}
\end{align}
where $c_{5}:=\tilde{q}^{-2}+e^{2\epsilon_{0}c_{2}/c_{1}}+2c_{3}e^{c_{4}\epsilon_{0}/2}+\tilde{q}^{-1}$
and $c_{6}:=\min\bigl\{2\epsilon_{0}c_{2}/c_{1},\frac{1}{2}c_{4}\epsilon_{0},\tilde{c}\bigr\}$.
Then Proposition~\ref{prop:good_chunks_percolate_half_plane}-\eqref{it:bad_cluster_size}
follows from the first paragraph of this proof and~\eqref{eq:bad_cluster_size_conclusion}.
\end{proof}

Finally, in the rest of this subsection we prove
Proposition~\ref{prop:good_chunks_percolate_half_plane}-\eqref{it:good_chunks_initial_boundary},
on the basis of the following strategy. We first prove in Lemma~\ref{lem:WLLN_bottom_left}
that for each $j \geq 0$ and $n \in \mathbb{N}$, with probability tending to $1$ as $n\to\infty$,
the amount of quantum length occupied by
$\{ \partial \CN_{i} \mid \text{$j \leq i < j+n$, $\indicator_{E_{i}\cap E_{i+1}}=0$}\}$
in the part of $\partial \CW_{j}$ to the left of $\eta'_{j}(0)$ is at most a constant multiple of $A^{-2/3}n$.
This is expected to be much smaller than the quantum length of the part of
$\partial(\CW_{j} \setminus \CW_{j+n}) \cap \partial \CW_{j}$ to the left of $\eta'_{j}(0)$
provided $A$ is large enough, since Proposition~\ref{prop:good_chunks_percolate_half_plane}-\eqref{it:move_left}
indicates that the latter length should typically be at least $(\frac{1}{16}cA^{-2/3}\log A)n$.
While this expectation is not true with arbitrarily high probability because of the possible
slides of the exploration to the right caused by $\{ \CN_{i} \mid 0 \leq i < j \}$
(recall~\eqref{eq:bad_chunks_initial_point_half_plane}),
we show in Lemma~\ref{lem:good_chunks_initial_boundary} that it is true with \emph{some}
probability, thanks to the fact that with some probability $\{ \CN_{i} \mid 0 \leq i < j \}$
causes no slide to the right by Lemma~\ref{lem:right-length-positive-prob-positive}.
As a consequence, with some probability the part of
$\partial(\CW_{j} \setminus \CW_{j+n}) \cap \partial \CW_{j}$ to the left of $\eta'_{j}(0)$
cannot be covered by $\{ \partial \CN_{i} \mid \text{$j \leq i < j+n$, $\indicator_{E_{i}\cap E_{i+1}}=0$}\}$
and hence intersects $\partial \CN_{i}$ for some $i \in \mathbb{Z}\cap [j,j+n)$ with
$\indicator_{E_{i}\cap E_{i+1}}=1$. By taking $n$ large enough, we can further assume
that this $\partial \CN_{i}$ is sufficiently far away to the left from $\eta'_{j}(0)$,
and then $\partial \CN_{i}$ is going to intersect $\partial_{\mathrm{L}} \CW = (-\infty,0)$
provided the quantum length of the part of $\partial(\CW_{0} \setminus \CW_{j}) \cap \partial \CW_{j}$
to the left of $\eta'_{j}(0)$ is reasonably bounded. Now since, with very high probability,
this boundedness holds for any $j \in \mathbb{Z} \cap [0,N]$ with $\partial \CN_{j} \cap \partial_{\mathrm{L}} \CW \not=\emptyset$
by~\eqref{eq:bad_cluster_size_claim3} and~\eqref{eq:bad_cluster_size_claim4} and
such $j$ appears sufficiently often by~\eqref{eq:bad_cluster_size_conclusion},
we can make trials for sufficiently many $j \in \mathbb{Z} \cap [0,N]$,
each with some probability of success, to find $i \in \mathbb{Z} \cap [j,j+n)$ with
$\indicator_{E_{i}\cap E_{i+1}}=1$ and $\partial \CN_{i} \cap \partial_{\mathrm{L}} \CW \not=\emptyset$
for fixed $n$ large enough, which turns out to yield~\eqref{eq:good_chunks_initial_boundary}.

\begin{lemma} \label{lem:WLLN_bottom_left}
Let $j \geq 0$ and set
\begin{equation} \label{eq:bottom_left_sum}
S'_{j,k} := \sum_{i=j}^{k-1} B^{i}_{L} \indicator_{(E_{i}\cap E_{i+1})^{c}}
  \qquad \text{for $k>j$.}
\end{equation}
Then for any $p \in (1,\frac{3}{2})$, there exist $c_{p,1},c_{p,2} \in (0,\infty)$
determined solely by $p$ ($c_{p,1}$ can be chosen so as to be independent of $p$)
such that, as long as $c_{0} \in (0,1]$, for any $n\in\mathbb{N}$,
\begin{equation} \label{eq:WLLN_bottom_left}
\mathbb{P}\bigl[ S'_{j,j+n} \geq c_{p,1} A^{-2/3} n \bigr]
  = \mathbb{P}\bigl[ S'_{0,n} \geq c_{p,1} A^{-2/3} n \bigr]
  \leq c_{p,2} A^{2/3} n^{1-p}.
\end{equation}
\end{lemma}

\begin{proof}
Since $\{(\indicator_{E_{i}},B^{i}_{L})\}_{i=0}^{\infty}$ is i.i.d.\ as noted after~\eqref{eqn:bad_length_change},
the equality in~\eqref{eq:WLLN_bottom_left} holds, and we may therefore assume that $j=0$.
We set $S^{3}_{n}:=\sum_{i=0}^{n-1} B^{i}_{L} \indicator_{E_{i}^{c}}$ and
$S^{4}_{n}:=\sum_{i=0}^{n-1} B^{i}_{L} \indicator_{E_{i}\cap E_{i+1}^{c}}$ for $n\in\mathbb{N}$,
so that $S'_{0,n}=S^{3}_{n}+S^{4}_{n}$, and we will apply to $\{ B^{i}_{L} \indicator_{E_{i}^{c}} \}_{i=0}^{n-1}$
and $\{ B^{i}_{L} \indicator_{E_{i}\cap E_{i+1}^{c}} \}_{i=0}^{n-1}$ separately a version
\cite[Exercise 1.2.11]{Str2011} of the weak law of large numbers with an explicit remainder estimate.

Let $p\in[1,\frac{3}{2})$. We first prove that, as long as $c_{0}\in(0,1]$,
\begin{equation} \label{eq:WLLN_bottom_left-p-moment}
\mathbb{E}[(B^{0}_{L})^{p}\indicator_{E_{0}^{c}}] \leq c_{p,3}A^{-2/3} \qquad\text{and}\qquad
\mathbb{E}[(B^{0}_{L})^{p}\indicator_{E_{0}\cap E_{1}^{c}}] \leq c_{p,3}A^{-2/3}
\end{equation}
for some $c_{p,3}\in(0,\infty)$ determined solely by $p$. Recall that
$\mathbb{E}[(B^{0}_{L})^{p}]\leq\mathbb{E}[\vert I^{1}_{1}\vert^{p}]<\infty$
by~\eqref{eq:expression_bottom_lengths}, \eqref{eq:expression_sigma_running_inf}
and \cite[Chapter VIII, Proposition 4]{bertoin1996levy}, that
$\mathbb{P}[E_{0}^{c}\mid\sigma_{0}<1]=1-\mathbb{P}[E_{0}\mid\sigma_{0}<1]\leq\epsilon_{0}=c_{0}A^{-2/3}$
by the assumption on the Borel subset $E$ of $\MCPU{2}$ and hence that
$\mathbb{P}[E_{1}^{c}]=\mathbb{P}[E_{0}^{c}]
  =\mathbb{P}[\sigma_{0}=1]+\mathbb{P}[\{\sigma_{0}<1\}\setminus E_{0}]
  \leq\mathbb{P}[\tau\geq A]+c_{0}A^{-2/3}\leq(c''+c_{0})A^{-2/3}$
for some $c''\in(0,\infty)$ determined solely by the law of $(X^{1},X^{2})$ by~\eqref{eqn:tautail}.
By the independence of $(B^{0}_{L})^{p}\indicator_{E_{0}},\indicator_{E_{1}^{c}}$ we have
\begin{equation} \label{eq:WLLN_bottom_left-p-moment-proof1}
\mathbb{E}[(B^{0}_{L})^{p}\indicator_{E_{0}\cap E_{1}^{c}}]
=
\mathbb{E}[(B^{0}_{L})^{p}\indicator_{E_{0}}]\mathbb{P}[E_{1}^{c}]
\leq
\mathbb{E}[\vert I^{1}_{1}\vert^{p}](c''+c_{0})A^{-2/3}.
\end{equation}
For $\mathbb{E}[(B^{0}_{L})^{p}\indicator_{E_{0}^{c}}]$, we decompose it as
\begin{equation} \label{eq:WLLN_bottom_left-p-moment-proof2}
\mathbb{E}[(B^{0}_{L})^{p}\indicator_{E_{0}^{c}}]
  =\mathbb{E}[(B^{0}_{L})^{p}\indicator_{\{\sigma_{0}=1\}}]
    +\mathbb{E}[(B^{0}_{L})^{p}\indicator_{\{\sigma_{0}<1\}\setminus E_{0}}],
\end{equation}
and then since $I^{1}_{A}=I^{1}_{1}$ a.s.\ on $\{\tau\geq A\}$, for the first term we have
\begin{equation} \label{eq:WLLN_bottom_left-p-moment-proof3}
\mathbb{E}[(B^{0}_{L})^{p}\indicator_{\{\sigma_{0}=1\}}]
=
\mathbb{E}[\vert A^{-2/3}I^{1}_{A}\vert^{p}\indicator_{\{\tau\geq A\}}]
\leq
A^{-2p/3}\mathbb{E}[\vert I^{1}_{1}\vert^{p}]
\leq
A^{-2/3}\mathbb{E}[\vert I^{1}_{1}\vert^{p}].
\end{equation}
For the last term in~\eqref{eq:WLLN_bottom_left-p-moment-proof2},
setting $q:=\frac{1}{2}+\frac{3}{4}p^{-1}\in(1,\frac{3}{2}p^{-1})$,
by H\"{o}lder's inequality, $\mathbb{P}[E_{0}^{c}\mid\sigma_{0}<1]\leq c_{0}A^{-2/3}$
and Proposition~\ref{prop:running-inf-mean-upper-bound} we get
\begin{align}
\mathbb{E}[(B^{0}_{L})^{p}\indicator_{\{\sigma_{0}<1\}\setminus E_{0}}]
&\leq
\mathbb{P}[\{\sigma_{0}<1\}\setminus E_{0}]^{1-1/q} \mathbb{E}[(B^{0}_{L})^{pq}\indicator_{\{\sigma_{0}<1\}}]^{1/q} \notag \\
&=
\mathbb{P}[\{\sigma_{0}<1\}\setminus E_{0}]^{1-1/q} \mathbb{E}[\vert A^{-2/3}I^{1}_{\tau}\vert^{pq}\indicator_{\{\tau<A\}}]^{1/q} \notag \\
&\leq
(c_{0}A^{-2/3})^{1-1/q} (A^{-2pq/3} c_{pq}A^{2(pq-1)/3})^{1/q} \notag \\
&=
c_{0}^{1-1/q}c_{pq}^{1/q}A^{-2/3},
\label{eq:WLLN_bottom_left-p-moment-proof4}
\end{align}
where $c_{pq}\in(0,\infty)$ is as in Proposition~\ref{prop:running-inf-mean-upper-bound}
with $pq$ in place of $p$. Thus~\eqref{eq:WLLN_bottom_left-p-moment} follows by combining
\eqref{eq:WLLN_bottom_left-p-moment-proof1}, \eqref{eq:WLLN_bottom_left-p-moment-proof2},
\eqref{eq:WLLN_bottom_left-p-moment-proof3} and~\eqref{eq:WLLN_bottom_left-p-moment-proof4}.

Now setting $E^{3}_{i}:=E_{i}^{c}$ and $E^{4}_{i}:=E_{i}\cap E_{i+1}^{c}$ for $i\geq 0$
and letting $p\in(1,\frac{3}{2})$, $n\in\mathbb{N}$, $l\in\{3,4\}$ and $s\in(0,\infty)$, we see from
\cite[Exercise 1.2.11]{Str2011} applied to $\{ B^{i}_{L}\indicator_{E^{l}_{i}} \}_{i=0}^{n-1}$
and~\eqref{eq:WLLN_bottom_left-p-moment} that
\begin{align}
\mathbb{P}\Bigl[\bigl\vert S^{l}_{n} - n \mathbb{E}[ B^{0}_{L}\indicator_{E^{l}_{0}\cap\{B^{0}_{L}\leq n\}}]\bigr\vert \geq ns \Bigr]
  \leq \frac{2}{ns^{2}}\int_{0}^{n}&t\mathbb{P}[ B^{0}_{L} \indicator_{E^{l}_{0}} \geq t ]\,dt + n\mathbb{P}[ B^{0}_{L} \indicator_{E^{l}_{0}} \geq n ] \notag \\
\leq \biggl(\frac{2}{ns^{2}}\int_{0}^{n}t^{1-p}\,dt + n^{1-p} \biggr) \mathbb{E}[ (B^{0}_{L})^{p} \indicator_{E^{l}_{0}} ]
  &\leq n^{1-p}(4s^{-2}+1)c_{p,3} A^{-2/3};
\label{eq:WLLN_bottom_left-proof1}
\end{align}
to be precise, $\{ B^{i}_{L}\indicator_{E^{4}_{i}} \}_{i=0}^{n-1}$ is not independent
as assumed in \cite[Exercise 1.2.11]{Str2011}, but the inequality in the first line of
\eqref{eq:WLLN_bottom_left-proof1} still holds since the covariance of
$B^{i}_{L}\indicator_{E^{4}_{i}\cap\{B^{i}_{L}\leq n\}},B^{k}_{L}\indicator_{E^{4}_{k}\cap\{B^{k}_{L}\leq n\}}$
is $0$ for $i,k\geq 0$ with $\vert i-k \vert\geq 2$ by the independence and
at most $0$ for $i,k\geq 0$ with $\vert i-k \vert=1$ by
$\indicator_{E^{4}_{i}\cap\{B^{i}_{L}\leq n\}}\indicator_{E^{4}_{k}\cap\{B^{k}_{L}\leq n\}}=0$
and $B^{i}_{L}\wedge B^{k}_{L}\geq 0$. Finally, noting that $S'_{0,n}=S^{3}_{n}+S^{4}_{n}$ and that
$\mathbb{E}[ B^{0}_{L}\indicator_{E^{l}_{0}\cap\{B^{0}_{L}\leq n\}}] \leq \mathbb{E}[ B^{0}_{L}\indicator_{E^{l}_{0}}] \leq c_{1,3}A^{-2/3}$
by~\eqref{eq:WLLN_bottom_left-p-moment}, we conclude from~\eqref{eq:WLLN_bottom_left-proof1} with $s=A^{-2/3}$ that
\begin{align*}
\mathbb{P}[S'_{0,n}\geq 2(c_{1,3}+1)A^{-2/3}n]
  &\leq \mathbb{P}[S^{3}_{n}\geq(c_{1,3}+1)A^{-2/3}n]+\mathbb{P}[S^{4}_{n}\geq(c_{1,3}+1)A^{-2/3}n] \\
&\leq 2n^{1-p}(4A^{4/3}+1)c_{p,3} A^{-2/3}
  \leq 9c_{p,3}A^{2/3}n^{1-p},
\end{align*}
proving~\eqref{eq:WLLN_bottom_left}.
\end{proof}

The essence of Proposition~\ref{prop:good_chunks_percolate_half_plane}-\eqref{it:good_chunks_initial_boundary}
consists in the following lemma.

\begin{lemma} \label{lem:good_chunks_initial_boundary}
Let $c\in(0,\infty)$ be the constant from Proposition~\ref{prop:running-inf-mean-diverge},
let $c_{5/4,1}\in(0,\infty)$ be as in Lemma~\ref{lem:WLLN_bottom_left} with $p=\frac{5}{4}$,
define $\{L'_{j}\}_{j=0}^{\infty}$ by $L'_{0}:=0$ and~\eqref{eqn:bad_length_change},
and let $\{S'_{j,k}\}_{j\geq 0,\,k>j}$ be as in~\eqref{eq:bottom_left_sum}.
Let $j \geq 0$, and for each $n\geq 2$, let $y_{j,n}\in\partial\CW_{j}$ be the point to
the left of $\eta'_{j}(0)$ in $\partial\CW_{j}$ such that the quantum length of
the part of $\partial\CW_{j}$ from $\eta'_{j}(0)$ to $y_{j,n}$ is $(\frac{1}{64}cA^{-2/3}\log A)n$,
and define an event $\widetilde{E}'_{j,n}$ by
\begin{equation} \label{eq:good_chunks_initial_boundary-event}
\widetilde{E}'_{j,n}:=\widetilde{E}_{j,j+n-1}\cap\biggl\{ \frac{L'_{j}-L'_{j+n}}{n} > \frac{1}{16}cA^{-2/3}\log A,\,\frac{S'_{j,j+n}}{n}<c_{5/4,1}A^{-2/3}\biggr\},
\end{equation}
where $\widetilde{E}_{j,j+n-1}$ is as defined in~\eqref{eq:right-length-positive-finite} with $l=j+n-1$.
Then, provided $A^{-1} \vee c_{0}$ is small enough, there exist $n_{1}\geq 2$ and
$\tilde{q}'\in(0,1)$ determined solely by $A,c_{0}$ such that for any $n\geq n_{1}$,
\begin{gather} \label{eq:good_chunks_initial_boundary_fixed_steps}
\widetilde{E}'_{j,n} \subset \biggl\{
  \begin{minipage}{260pt}
  there exists $l\in\mathbb{Z}\cap [j,j+n)$ such that $\indicator_{E_{l}\cap E_{l+1}}=1$ and
  $\partial \CN_{l}$ intersects the part of $\partial \CW_{j}$ which is to the left of $y_{j,n}$
  \end{minipage}
  \biggr\}, \\
\mathbb{P}[\widetilde{E}'_{j,n}] = \mathbb{P}[\widetilde{E}'_{0,n}] \geq 1-\tilde{q}'.
\label{eq:good_chunks_initial_boundary-prob-positive}
\end{gather}
\end{lemma}

\begin{proof}
Assume that $A\geq\exp(64c_{5/4,1}/c)$, that $c_{0} \in (0,1]$ and that $A^{-1} \vee c_{0}$
is small enough so that~\eqref{eq:move_left_Lprime} and~\eqref{eq:right-length-positive-prob-positive}
hold with $\delta=\frac{1}{4}cA^{-2/3}\log A,\lambda,\tilde{q}$ as stated in these places.
Let $c_{5/4,2}\in(0,\infty)$ be as in Lemma~\ref{lem:WLLN_bottom_left} with $p=\frac{5}{4}$, and set
\begin{equation} \label{eq:good_chunks_initial_boundary_minimum_steps}
n_{1}:=\min\biggl(\mathbb{N}\cap\biggl[\max\biggl\{2,\frac{8}{\delta\lambda}\log\frac{4}{1-\tilde{q}},\Bigl(c_{5/4,2}A^{2/3}\frac{4}{1-\tilde{q}}\Bigr)^{4}\biggr\},\infty\biggr)\biggr),
\end{equation}
so that $n_{1}$ is determined solely by $A,c_{0}$ and satisfies
$\exp(-\frac{1}{8}\delta\lambda n_{1})\leq\frac{1}{4}(1-\tilde{q})$ and
$c_{5/4,2}A^{2/3}n_{1}^{-1/4}\leq\frac{1}{4}(1-\tilde{q})$. Now let $n\geq n_{1}$.
Then since $\{(\indicator_{E_{k}},T^{k},B^{k}_{L},B^{k}_{R})\}_{k=0}^{\infty}$ is i.i.d.\ as noted just after
\eqref{eqn:bad_length_change}, the equality in~\eqref{eq:good_chunks_initial_boundary-prob-positive} holds,
and we easily see from~\eqref{eq:right-length-positive-prob-positive}, \eqref{eq:move_left_Lprime}
and~\eqref{eq:WLLN_bottom_left} with $p=\frac{5}{4}$ that 
$\mathbb{P}[\widetilde{E}'_{0,n}]\geq 1-\tilde{q}'$ with $\tilde{q}':=\frac{1}{2}(1+\tilde{q})$.

To show~\eqref{eq:good_chunks_initial_boundary_fixed_steps}, fix any instance of the
exploration for which $\widetilde{E}'_{j,n}$ holds, and for each $k \geq 0$,
as in the proof of Proposition~\ref{prop:good_chunks_percolate_half_plane}-\eqref{it:move_left}
define $I^{0}_{k}\subset\mathbb{Z}\cap[0,k)$ by $I^{0}_{k}:=\emptyset$ on $E_{k}$ and
by~\eqref{eq:SLE6_percolate_H-right-chunks-skipped} with $k$ in place of $j$ on $E_{k}^{c}$
(recall~\eqref{eq:SLE6_percolate_H-right-chunks} for $I_{k}$), so that
$I^{0}_{k}\subset\{i \in \mathbb{Z}\cap[0,k)\mid \indicator_{E_{i}}=1 \}$ and
$I^{0}_{k}\cap I^{0}_{l}=\emptyset$ for any $l>k$. Then since
$\bigcup_{k=j}^{j+n-1}I^{0}_{k}\subset\mathbb{Z}\cap(j,j+n-1)$
by~\eqref{eq:right-boundary-avoided-finite}, it follows from~\eqref{eq:SLE6_percolate_H_length_change}
in the same way as~\eqref{eq:SLE6_percolate_H_length_comparison} that
\begin{align}
&(L_{j+n}-L_{j})-(L'_{j+n}-L'_{j}) \notag \\
&= \sum_{k=j}^{j+n-1} \bigl( (L_{k+1}-L_{k}) - (L'_{k+1}-L'_{k}) \bigr)
  \leq \sum_{k=j}^{j+n-1}( -\epsilon_{0} \indicator_{E_{k}} + \epsilon_{0} \# I^{0}_{k} )
  \label{eq:good_chunks_initial_boundary_fixed_steps_comparison} \\
&\leq\epsilon_{0}\bigl(-\#\{k \in \mathbb{Z} \cap [j,j+n) \mid \indicator_{E_{k}}=1 \} + \#\{k \in \mathbb{Z} \cap (j,j+n-1) \mid \indicator_{E_{k}}=1 \}\bigr)
	\leq 0. \notag
\end{align}
Recalling that $\frac{1}{64}cA^{-2/3}\log A\geq c_{5/4,1}A^{-2/3}$ by $A\geq\exp(64c_{5/4,1}/c)$,
from~\eqref{eq:good_chunks_initial_boundary_fixed_steps_comparison} and the latter part of
\eqref{eq:good_chunks_initial_boundary-event} we obtain
\begin{align}
&L_{j}-L_{j+n} - \sum_{k=j}^{j+n-1} B^{k}_{L} \indicator_{(E_{k}\cap E_{k+1})^{c}} - \Bigl(\frac{1}{64}cA^{-2/3}\log A\Bigr)n \notag \\
&\geq L'_{j}-L'_{j+n} - S'_{j,j+n} - \Bigl(\frac{1}{64}cA^{-2/3}\log A\Bigr)n
  >\Bigl(\frac{1}{32}cA^{-2/3}\log A\Bigr)n
  >0. \label{eq:good_chunks_initial_boundary_fixed_steps_remaining_length}
\end{align}
Note that the part of $\partial(\CW_{j} \setminus \CW_{j+n}) \cap \partial \CW_{j}$
which is to the left of $\eta'_{j}(0)$ has quantum length at least $L_{j}-L_{j+n}$ and equals
$\bigcup_{k=j}^{j+n-1}(\text{the bottom left of $\CW_{k} \setminus \CW_{k+1}$}) \cap \partial \CW_{j}$.
Now these facts and~\eqref{eq:good_chunks_initial_boundary_fixed_steps_remaining_length}
together imply that $y_{j,n}$ belongs to the interior of this part of
$\partial(\CW_{j} \setminus \CW_{j+n}) \cap \partial \CW_{j}$ and that the part of
$\partial(\CW_{j} \setminus \CW_{j+n}) \cap \partial \CW_{j}$ which is to the left of $y_{j,n}$
has quantum length greater than $\sum_{k=j}^{j+n-1} B^{k}_{L} \indicator_{(E_{k}\cap E_{k+1})^{c}}$
and hence intersects $\bigcup_{j\leq k<j+n,\,\indicator_{E_{k}\cap E_{k+1}}=1}(\text{the bottom left of $\CW_{k} \setminus \CW_{k+1}$}) \cap \partial \CW_{j}$.
This last property means that the event in the right-hand side of
\eqref{eq:good_chunks_initial_boundary_fixed_steps} holds, proving
\eqref{eq:good_chunks_initial_boundary_fixed_steps}.
\end{proof}

\begin{proof}[Proof of Proposition~\ref{prop:good_chunks_percolate_half_plane}-\eqref{it:good_chunks_initial_boundary}]
Assume that $A^{-1} \vee c_{0}$ is small enough so that $c_{0} \in (0,1]$ and
Proposition~\ref{prop:good_chunks_percolate_half_plane}-\eqref{it:move_left},
\eqref{eq:right-length-positive-prob-positive}, \eqref{eq:good_chunks_initial_boundary_fixed_steps}
and~\eqref{eq:good_chunks_initial_boundary-prob-positive} hold with
$c_{1},c_{2},\tilde{q},n_{1},\tilde{q}'$ as stated in these places.
Let $u \in [2,\infty)$ and set $n_{u} := \max\bigl( \mathbb{N} \cap [ 1, u ] \bigr)$
and $\widehat{n}_{u}:=n_{1}\vee\min\bigl(\mathbb{N}\cap[256c^{-1}n_{u},\infty)\bigr)$, where
$c\in(0,\infty)$ is the constant from Proposition~\ref{prop:running-inf-mean-diverge}.
Following the proof of Proposition~\ref{prop:good_chunks_percolate_half_plane}-\eqref{it:bad_cluster_size},
set $\partial_{\mathrm{L}} \CW := (-\infty,0) \subset \partial \CW$,
let $\mathescr{G}_{k}$ denote the $\sigma$-algebra generated by
$\{(\CW_{i} \setminus \CW_{i+1},h|_{\CW_{i} \setminus \CW_{i+1}},\eta'_{i}|_{[0,\sigma_{i}]},\widetilde{L}_{i})\}_{0 \leq i \leq k}$
for each $k \geq 0$, set $\widetilde{m}_{0} := 0$ and inductively
$\widetilde{m}_{j} := \inf\{ i > \widetilde{m}_{j-1}+\widehat{n}_{u} \mid \partial \CN_{i} \cap \partial_{\mathrm{L}} \CW \not= \emptyset \}$
for each $j \geq 1$, so that $\widetilde{m}_{j}$ is a $\{\mathescr{G}_{k}\}$-stopping time for any $j \geq 0$.
Let $\widetilde{E}'_{j,\widehat{n}_{u}}$ be as in~\eqref{eq:good_chunks_initial_boundary-event}
with $n=\widehat{n}_{u}$ for $j\geq 0$, and set
$\widetilde{E}'_{\widetilde{m}_{j}+1,\widehat{n}_{u}}:=\bigcup_{k\in\mathbb{N}\cup\{0\}}\bigl(\widetilde{E}'_{k+1,\widehat{n}_{u}}\cap\{\widetilde{m}_{j}=k\}\bigr)$
for each $j\geq 0$. Then since $\widetilde{m}_{j}+\widehat{n}_{u}<\widetilde{m}_{j+1}<\infty$ for any $j\geq 0$ a.s.\ by
Proposition~\ref{prop:good_chunks_percolate_half_plane}-\eqref{it:move_left}
and $\widetilde{E}'_{k,\widehat{n}_{u}}$ is measurable with respect to
$\{(\indicator_{E_{i}},T^{i},B^{i}_{L},B^{i}_{R})\}_{i=k}^{k+\widehat{n}_{u}}$
and hence independent of $\mathescr{G}_{k-1}$ for any $k\geq 1$, we have
$\widetilde{E}'_{\widetilde{m}_{j}+1,\widehat{n}_{u}} \in \mathescr{G}_{\widetilde{m}_{j+1}}$
and therefore see from~\eqref{eq:good_chunks_initial_boundary-prob-positive} that
$\{\widetilde{E}'_{\widetilde{m}_{j}+1,\widehat{n}_{u}}\}_{j=0}^{\infty}$ is independent and that
$\mathbb{P}\bigl[\widetilde{E}'_{\widetilde{m}_{j}+1,\widehat{n}_{u}}\bigr]=\mathbb{P}\bigl[\widetilde{E}'_{0,\widehat{n}_{u}}\bigr]\geq 1-\tilde{q}'$
for any $j\geq 0$. It thus follows that, with $\tilde{c}':=\log(1/\tilde{q}')$, for any $N\in\mathbb{N}$,
\begin{equation} \label{eq:good_chunks_initial_boundary_failure_prob}
\mathbb{P}\Biggl[ \bigcup_{k=0}^{N-1}\bigcap_{j=kn_{u}}^{(k+1)n_{u}-1}\bigl(\widetilde{E}'_{\widetilde{m}_{j}+1,\widehat{n}_{u}}\bigr)^{c} \Biggr]
  \leq \sum_{k=0}^{N-1}\prod_{j=kn_{u}}^{(k+1)n_{u}-1}\mathbb{P}\bigl[\bigl(\widetilde{E}'_{\widetilde{m}_{j}+1,\widehat{n}_{u}}\bigr)^{c}\bigr]
  \leq Ne^{-\tilde{c}'n_{u}}.
\end{equation}
Furthermore let $m_{j},R_{j},t_{j}$ be as in the proof of
Proposition~\ref{prop:good_chunks_percolate_half_plane}-\eqref{it:bad_cluster_size}
for each $j \geq 0$, let $N \in \mathbb{N}$ and set $J_{N} := \min\{ j \geq 0 \mid m_{j} \geq N \}$,
so that $J_{N} \leq N$. Then by~\eqref{eq:bad_cluster_size_claim3} and~\eqref{eq:bad_cluster_size_claim4} we have
\begin{align}
&\mathbb{P}\bigl[\{\max\nolimits_{0\leq j <J_{N}}R_{j}>n_{u}+1\}\cup\{\max\nolimits_{0\leq j <J_{N}}t_{j}>\epsilon_{0}n_{u}\}\bigr] \notag \\
&\leq \mathbb{P}\bigl[\max\nolimits_{0\leq j <J_{N}}R_{j}>n_{u}+1\bigr] + \mathbb{P}\bigl[\max\nolimits_{0\leq j <J_{N}}t_{j}>\epsilon_{0}n_{u}\bigr] \notag \\
&\leq \sum_{j=0}^{N-1} \bigl( \mathbb{P}\bigl[m_{j}<N, \, R_{j}>n_{u}+1\bigr] + \mathbb{P}\bigl[m_{j}<N, \, t_{j}>\epsilon_{0}n_{u}\bigr] \bigr) \notag \\
&\leq N^{2}\bigl( \tilde{q}^{-1} \exp(-\tilde{c}u) + 2c_{3} \exp( - \tfrac{1}{2} c_{4} \epsilon_{0} n_{u} ) \bigr),
\label{eq:bad_cluster_size_claim3-claim4}
\end{align}
where $\tilde{c} := \frac{1}{2}\log(\tilde{q}^{-1})$ and $c_{3},c_{4}\in(0,\infty)$ are as in
\eqref{eq:bad_cluster_size_claim4}. Thus from~\eqref{eq:bad_cluster_size_conclusion},
\eqref{eq:good_chunks_initial_boundary_failure_prob} and~\eqref{eq:bad_cluster_size_claim3-claim4} we obtain
\begin{align}
&\mathbb{P}\Biggl[ \biggl\{ \begin{minipage}{170pt}$m_{j+1}-m_{j} \leq (3\epsilon_{0}/c_{1})u^{2}$, $R_{j}\leq n_{u}+1$ and $t_{j}\leq\epsilon_{0}n_{u}$ for any $j\in\mathbb{Z}\cap[0,J_{N})$\end{minipage} \biggr\}
  \cap \bigcap_{k=0}^{N-1}\bigcup_{j=kn_{u}}^{(k+1)n_{u}-1}\widetilde{E}'_{\widetilde{m}_{j}+1,\widehat{n}_{u}} \Biggr] \notag \\
&\geq 1-c_{5} N^{2} \exp(-c_{6}u) - Ne^{-\tilde{c}'n_{u}} - N^{2}\bigl( \tilde{q}^{-1} \exp(-\tilde{c}u) + 2c_{3} \exp( - \tfrac{1}{2} c_{4} \epsilon_{0} n_{u} ) \bigr) \notag \\
&\geq 1 - c_{8}N^{2}\exp(-c_{9}u),
\label{eq:good_chunks_initial_boundary_success_prob}
\end{align}
where $c_{5},c_{6}\in(0,\infty)$ are as in~\eqref{eq:bad_cluster_size_conclusion},
$c_{8}:=c_{5}+e^{\tilde{c}'}+\tilde{q}^{-1}+2c_{3}e^{c_{4}\epsilon_{0}/2}$ and
$c_{9}:=\min\bigl\{c_{6},\tilde{c}',\tilde{c},\frac{1}{2}c_{4}\epsilon_{0}\bigr\}$.

It therefore remains to verify that the event in~\eqref{eq:good_chunks_initial_boundary_success_prob}
is included in that in~\eqref{eq:good_chunks_initial_boundary} as long as $N\geq c_{7}u^{3}$,
for some $c_{7}\in(0,\infty)$ determined solely by $A,c_{0}$. To this end, fix any instance of
the exploration for which the event in~\eqref{eq:good_chunks_initial_boundary_success_prob} holds,
set $\widetilde{J}_{N}:=\min\{j \geq 0 \mid \widetilde{m}_{j} + \widehat{n}_{u} \geq N\}$ and
$\widetilde{J}'_{N}:=\max\bigl(\mathbb{Z}\cap[0,\widetilde{J}_{N}/n_{u}]\bigr)$, so that
$\widetilde{J}'_{N}\leq\widetilde{J}_{N}/n_{u}\leq\widetilde{J}_{N}\leq N$. Then since
$m_{j+1}-m_{j} \leq (3\epsilon_{0}/c_{1})u^{2}$ for any $j\in\mathbb{Z}\cap[0,J_{N})$, we easily obtain
\begin{equation} \label{eq:chunks_initial_boundary_frequency}
\widetilde{m}_{j+1}-\widetilde{m}_{j}\leq (3\epsilon_{0}/c_{1})u^{2}+\widehat{n}_{u}
  \qquad \text{for any $j\in\mathbb{Z}\cap[0,\widetilde{J}_{N})$.}
\end{equation}
Let $i \in \mathbb{Z}\cap[0,\widetilde{J}_{N})$, set $j:=\widetilde{m}_{i}+1\leq N$ and let
$y_{j,\widehat{n}_{u}} \in \partial \CW_{j}$ be as in Lemma~\ref{lem:good_chunks_initial_boundary}
with $n=\widehat{n}_{u}$. Then $\widetilde{m}_{i}=m_{k}$ for some $k \in \mathbb{Z}\cap[0,J_{N})$, hence
$R_{k} \leq n_{u}+1$, $t_{k} \leq \epsilon_{0}n_{u}$, and thus by
$\partial \CN_{j-1} \cap \partial_{\mathrm{L}} \CW \not=\emptyset$, \eqref{eq:bad_chunks_initial_point_half_plane},
\eqref{eq:run_good_chunks_max_length} and $t_{k} = T^{j-1} \indicator_{E_{j-1}^{c}}$,
the part of $\partial( \CW_{0} \setminus \CW_{j}) \cap \partial \CW_{j}$ to the left of
$\eta'_{j}(0)$ has quantum length at most $\epsilon_{0}(R_{k}-1)+t_{k} \leq 2\epsilon_{0}n_{u}$.
Moreover, $2\epsilon_{0}n_{u}$ in turn is less than $(\frac{1}{64}cA^{-2/3}\log A)\widehat{n}_{u}$,
the quantum length of the part of $\partial\CW_{j}$ from $\eta'_{j}(0)$ to $y_{j,\widehat{n}_{u}}$,
by $\epsilon_{0}=c_{0}A^{-2/3}$, $A\geq 2$, $c_{0}\leq 1$ and $\widehat{n}_{u}\geq 256c^{-1}n_{u}$.
It follows therefore that $y_{j,\widehat{n}_{u}} \in \partial_{\mathrm{L}} \CW$ and hence
by $\widehat{n}_{u}\geq n_{1}$ and~\eqref{eq:good_chunks_initial_boundary_fixed_steps} that,
\begin{equation} \label{eq:good_chunks_initial_boundary_verify_preliminary}
\begin{minipage}{225pt}
if $\indicator_{\widetilde{E}'_{\widetilde{m}_{i}+1,\widehat{n}_{u}}}=1$, then we have
$\indicator_{E_{l}\cap E_{l+1}}=1$ and $\partial \CN_{l} \cap \partial_{\mathrm{L}} \CW \not=\emptyset$
for some $l\in\mathbb{Z}\cap(\widetilde{m}_{i},\widetilde{m}_{i}+\widehat{n}_{u}]$.
\end{minipage}
\end{equation}
On the other hand, for any $k \in \mathbb{Z}\cap[0,\widetilde{J}'_{N})$, by
the latter part of the event in~\eqref{eq:good_chunks_initial_boundary_success_prob}
there exists $i \in \mathbb{Z}\cap[kn_{u},(k+1)n_{u})$ such that
$\indicator_{\widetilde{E}'_{\widetilde{m}_{i}+1,\widehat{n}_{u}}}=1$,
which together with $\mathbb{Z}\cap[kn_{u},(k+1)n_{u})\subset\mathbb{Z}\cap[0,\widetilde{J}_{N})$
and~\eqref{eq:good_chunks_initial_boundary_verify_preliminary} implies that
$\indicator_{E_{l}\cap E_{l+1}}=1$ and $\partial \CN_{l} \cap \partial_{\mathrm{L}} \CW \not=\emptyset$
for some
$l\in\mathbb{Z}\cap(\widetilde{m}_{i},\widetilde{m}_{i}+\widehat{n}_{u}]
  \subset \mathbb{Z}\cap(\widetilde{m}_{i},\widetilde{m}_{i+1})
  \subset \mathbb{Z}\cap(\widetilde{m}_{kn_{u}},\widetilde{m}_{(k+1)n_{u}})$.
Combining these observations with~\eqref{eq:chunks_initial_boundary_frequency},
we have thus proved the following:
\begin{equation} \label{eq:good_chunks_initial_boundary_verify_final}
\begin{minipage}{345pt}
for any $k \in \mathbb{Z}\cap[0,\widetilde{J}'_{N})$, 
$\widetilde{m}_{(k+1)n_{u}}-\widetilde{m}_{kn_{u}}\leq n_{u}\bigl((3\epsilon_{0}/c_{1})u^{2}+\widehat{n}_{u}\bigr)$,
and there exists $l\in\mathbb{Z}\cap(\widetilde{m}_{kn_{u}},\widetilde{m}_{(k+1)n_{u}})$ such that
$\indicator_{E_{l}\cap E_{l+1}}=1$ and $\partial \CN_{l} \cap \partial_{\mathrm{L}} \CW \not=\emptyset$.
\end{minipage}
\end{equation}
Suppose for the moment that $\widetilde{J}'_{N}\geq 2$, or equivalently, $\widetilde{J}_{N}\geq 2n_{u}$.
Then the left and right ends in $\mathbb{Z}$ of the sequence
$\bigl\{ \mathbb{Z}\cap (\widetilde{m}_{kn_{u}},\widetilde{m}_{(k+1)n_{u}}) \bigr\}_{k=0}^{\widetilde{J}'_{N}-1}$
of intervals appearing in~\eqref{eq:good_chunks_initial_boundary_verify_final} are given by
$\widetilde{m}_{0}=0$ and $\widetilde{m}_{\widetilde{J}'_{N}n_{u}}$, respectively, and we see from
$\widetilde{J}'_{N}+1>\widetilde{J}_{N}/n_{u}$, \eqref{eq:chunks_initial_boundary_frequency}
and $\widetilde{m}_{\widetilde{J}_{N}}+\widehat{n}_{u}\geq N$ that
\begin{equation} \label{eq:chunks_initial_boundary_frequency_right_end}
\begin{split}
N-\widetilde{m}_{\widetilde{J}'_{N}n_{u}}
  &\leq N-\widetilde{m}_{\widetilde{J}_{N}-n_{u}+1}
  \leq N-\widetilde{m}_{\widetilde{J}_{N}}+(n_{u}-1)\bigl((3\epsilon_{0}/c_{1})u^{2}+\widehat{n}_{u}\bigr) \\
&\leq \widehat{n}_{u} + (n_{u}-1)\bigl((3\epsilon_{0}/c_{1})u^{2}+\widehat{n}_{u}\bigr)
  < n_{u}\bigl((3\epsilon_{0}/c_{1})u^{2}+\widehat{n}_{u}\bigr).
\end{split}
\end{equation}
Moreover, for any $k \in \mathbb{Z}\cap[0,\widetilde{J}'_{N}-2]$ we also have
\begin{equation} \label{eq:chunks_initial_boundary_non_right_end_before_N}
\widetilde{m}_{(k+1)n_{u}}
  \leq \widetilde{m}_{(\widetilde{J}'_{N}-1)n_{u}}
  < \widetilde{m}_{\widetilde{J}'_{N}n_{u}-1}
  \leq \widetilde{m}_{\widetilde{J}_{N}-1}
  < N-\widehat{n}_{u}<N.
\end{equation}
Since $2n_{u}\bigl((3\epsilon_{0}/c_{1})u^{2}+\widehat{n}_{u}\bigr)\leq c_{7}u^{3}$
with $c_{7}:=6\epsilon_{0}/c_{1}+2(n_{1}\vee(256c^{-1}+1))$,
it follows from~\eqref{eq:good_chunks_initial_boundary_verify_final},
\eqref{eq:chunks_initial_boundary_frequency_right_end} and
\eqref{eq:chunks_initial_boundary_non_right_end_before_N} that with this choice of
$c_{7}$ the event in~\eqref{eq:good_chunks_initial_boundary} occurs for the present
instance of the exploration. Finally, we conclude this proof by deducing our supposition
$\widetilde{J}_{N}\geq 2n_{u}$ from~\eqref{eq:chunks_initial_boundary_frequency} and
the requirement $N\geq c_{7}u^{3}$; indeed, an induction on $j$ based on
$N\geq c_{7}u^{3}\geq 2n_{u}\bigl((3\epsilon_{0}/c_{1})u^{2}+\widehat{n}_{u}\bigr)$
and~\eqref{eq:chunks_initial_boundary_frequency} easily shows that
$\widetilde{m}_{j}\leq j\bigl((3\epsilon_{0}/c_{1})u^{2}+\widehat{n}_{u}\bigr)<N-\widehat{n}_{u}$
for any $j \in \mathbb{Z}\cap[0,2n_{u})$, whence
$\widetilde{m}_{2n_{u}-1}+\widehat{n}_{u}<N$, namely $\widetilde{J}_{N}\geq 2n_{u}$.
\end{proof}

\subsection{Proof of Proposition~\ref{prop:good_chunks_percolate}} \label{ssec:good_chunks_percolate_disk}

The proof of Proposition~\ref{prop:good_chunks_percolate} requires
Lemma~\ref{lem:boundary_length_rn_disk_weighted} below
in addition to Proposition~\ref{prop:good_chunks_percolate_half_plane}. For each $n \in \mathbb{N}$, we set
$D_{\mathbb{R}^{n}}:=\{\omega\mid\text{$\omega=(\omega_{1},\ldots,\omega_{n})\colon[0,\infty)\to\mathbb{R}^{n}$, $\omega$ is cadlag}\}$,
let $\{\mathescr{F}^{n}_{t}\}_{t\in[0,\infty)}$ denote the filtration in
$D_{\mathbb{R}^{n}}$ generated by the coordinate process, and set
$\mathescr{F}^{n}_{t+}:=\bigcap_{s\in(t,\infty)}\mathescr{F}^{n}_{s}$ for $t\in[0,\infty)$.

\begin{lemma} \label{lem:boundary_length_rn_disk_weighted}
Let $a \in (0,\infty)$, and 
let $Z$ be a $3/2$-stable L\'evy excursion with only upward jumps.  Set $\wt{Z}_t := Z_{(\zeta-t)^-}$ for each $0 \leq t \leq \zeta$,  i.e.,  the cadlag modification of the time-reversal of $Z$,  where $\zeta$ denotes the lifetime of $Z$.  We also set $T := \inf\{t \geq 0 : \wt{Z}_t \geq a\}$ and $Y_t = \wt{Z}_{t+T}$,  where we set $Y_t :=0$ for $t \in [\zeta_Y,\infty)$ with $\zeta_Y$ denoting the lifetime of $Y$.  Then, conditional on the event $\{T<\infty\}$,  we have that $\zeta_{Y}<\infty$ a.s., $Y_{t}>0$ for any $t\in[0,\zeta_{Y})$ a.s.\ and $\lim_{t\uparrow\zeta_{Y}}Y_{t}=0$ a.s.
Moreover, let $X$ be a $3/2$-stable L\'{e}vy process with only downward jumps with $X_{0} = a$,
and set $\tau_{0} := \inf\{ t \in [0,\infty) \mid X_{t} \leq 0 \}$.
Then for any $\{\mathescr{F}^{1}_{t+}\}_{t\in[0,\infty)}$-stopping time $\tau$,
\begin{equation} \label{eq:boundary_length_rn_disk_weighted}
\mathbb{P}[ Y \in d \omega, \ \tau(Y) < \zeta_{Y} ]|_{\mathescr{F}^{1}_{\tau+}}
  = \Bigl(\frac{\omega(\tau)}{a}\Bigr)^{-1/2} \mathbb{P}[ X \in d \omega, \ \tau(X) < \tau_{0} ]|_{\mathescr{F}^{1}_{\tau+}}.
\end{equation}
\end{lemma}

\begin{proof}
Let $Z,\wt{Z}$ be as in the statement of the lemma and let $n$ denote the law of $Z$.  Let also $\wt{n}$ denote the law of $\wt{Z}$ under $n$.  Then \cite[Theorem~4, part 2]{Ch96} implies that $\wt{n}[\wt{Z}_t \in A, t < \zeta \giv \CF_{s^+}^1] = \int_A p_{t-s}(\wt{Z}_s,x)dx$ $\wt{n}$-a.e.\  on $\CF_{s^+}^1$,  for each $0 \leq s \leq t,  A \in \CF_{t^+}^1$,  where $p_t(x,y) = (\frac{x}{y})^{1/2} q_t(x,y)$ for each $t,x,y>0$,  and $q_t$ denotes the semigroup of $X$ killed at the first time that it exits $(0,\infty)$,  where $X$ is as in the statement of the lemma.  Let $\mu$ (resp.  $\nu$) denote the law of $Y$ (resp.  $X$).  It follows that 
\begin{align*}
\mu\big[ Y \in A,  t < \zeta_Y \big]=\int_{A} p_t(a,x) dx
=\int_{A} \Bigl(\frac{\omega(t)}{a}\Bigr)^{-1/2} p_t(a,x) dx
\end{align*}
for each $A \in \CF_{t^+}^1$.  This proves~\eqref{eq:boundary_length_rn_disk_weighted} for deterministic times.  To prove the result for general stopping times,  we first note that it is easy to see that $M_t := \bigl(\frac{\omega(t)}{a}\bigr)^{-1/2} \indicator_{\{t < \tau_0\}}$ is a non-negative supermartingale under $\nu$.  Fix an $(\CF_{t^+}^1)$-stopping time $\tau$ as in the statement of the lemma,  and set $\tau_n := \inf((2^{-n} \mathbb{Z}) \cap (\tau,\infty))$ for each $n \in \mathbb{N}$.  Note that $\tau_n$ decreases to $\tau$ as $n \to \infty$ $\nu$-a.e.  Moreover,  we have that $\mu\big[ Y \in A,  \tau_n(Y) < \zeta_Y \big] = \int_{A} M_{\tau_n} d\nu$ for each $A \in \CF_{\tau_n^+}^1,  n \in \mathbb{N}$.  Note also that the optional stopping theorem implies that $\E\big[ M_{\tau_n} \giv \CF_{\tau_n^+}^1\big] \leq M_{\tau_m}$ $\nu$-a.e.,  for each $1 \leq m \leq n$,  and so $(-M_{\tau_n})$ is a backwards submartingale under $(\CF_{\tau_n^+}^1)_{n \geq 1}$.  Also,  $\E\big[ -M_{\tau_n} \big] \geq -M_0 = -a$,  and so \cite[Problem~3.11]{KS1991} implies that $M_{\tau_n}$ is a uniformly integrable martingale under $\nu$.  It follows that for each $A \in \CF_{\tau^+}^1$,  we have that
\begin{align*}
\mu\big[ Y \in A,  \tau(Y) < \zeta_Y \big] = \lim_{n \to \infty} \mu\big[ Y \in A,  \tau_n(Y) < \zeta_Y \big] = \int_A M_{\tau} d\nu
\end{align*}
and this completes the proof.
\end{proof}

We need the following lemma to prove Lemma~\ref{lem:top_length_tail_disk} below.

\begin{lemma}[{\cite[Theorem~1.2]{gm2018sledisk}}] \label{lem:boundary_length_rn_disk}
Let $\ell \in (0,\infty)$, $a \in (0,1)$, suppose that $\CD = (\D,h)$ has law $\qdisk{\ell}$,
let $x \in \partial \CD$ be chosen uniformly from the boundary measure $\qbmeasure{h}$,
and let $y$ be the point of $\partial \CD$ so that $\qbmeasure{h}([x,y]_{\partial \CD}^{\ccw})=(1-a)\ell$.
Let $\eta'$ be an independent chordal $\SLE_{6}$ on $\CD$ from $x$ to $y$ parameterized according to quantum
natural time, set $t_{\eta'}:=\inf(\eta')^{-1}(y)$, $L_{t}:=0=:R_{t}$ for $t \in [t_{\eta'},\infty)$,
and for each $t \in [0,t_{\eta'})$ let $L_{t}$ (resp.\ $R_{t}$) denote the boundary length of the clockwise (resp.\ counterclockwise)
arc of $\widetilde{\partial} \D_{y,t}$ from $\eta'(t)$ to $y$, where $\D_{y,t}$ denotes the
component of $\D \setminus \eta'([0,t])$ with $y \in \partial \D_{y,t}$.
Then $t_{\eta'}<\infty$ a.s.\ and $\lim_{t\uparrow t_{\eta'}}(L_{t},R_{t})=(0,0)$ a.s.
Moreover, let $X^{1},X^{2}$ be independent $3/2$-stable L\'{e}vy processes with only
downward jumps and with $X^{1}_{0}=a\ell=\ell-X^{2}_{0}$, and set
$\tau_{0}:=\inf\{ t \in [0,\infty) \mid X^{1}_{t} \wedge X^{2}_{t} \leq 0 \}$.
Then for any $\{\mathescr{F}^{2}_{t+}\}_{t\in[0,\infty)}$-stopping time $\tau$,
\begin{equation} \label{eq:boundary_length_rn_disk}
\begin{split}
&\mathbb{P}\bigl[ (L_{\cdot},R_{\cdot}) \in d (\omega_{1},\omega_{2}), \ \tau(L_{\cdot},R_{\cdot}) < t_{\eta'} \bigr]\big|_{\mathescr{F}^{2}_{\tau+}} \\
&= \Bigl(\frac{\omega_{1}(\tau)+\omega_{2}(\tau)}{\ell}\Bigr)^{-5/2} \mathbb{P}\bigl[ (X^{1},X^{2}) \in d(\omega_{1},\omega_{2}), \ \tau(X^{1},X^{2}) < \tau_{0} \bigr]\big|_{\mathescr{F}^{2}_{\tau+}}.
\end{split}
\end{equation}
\end{lemma}

\begin{proof}
The first assertion and~\eqref{eq:boundary_length_rn_disk} for constant $\tau$ follow from 
\cite[Theorem~1.2]{gm2018sledisk} combined with the resampling property of the quantum disk \cite[Proposition~A.8]{dms2014mating}.  Then~\eqref{eq:boundary_length_rn_disk} for general $\tau$ can be
verified in exactly the same way as in the proof of Lemma~\ref{lem:boundary_length_rn_disk_weighted} above.
\end{proof}

Combining Lemma~\ref{lem:boundary_length_rn_disk} with Proposition~\ref{prop:stable-reflected-exp-moment},
we obtain the following lemma.

\begin{lemma} \label{lem:top_length_tail_disk}
Let $\ell \in (0,\infty)$, $a\in(0,1)$, suppose that $\CD = (\D,h,0)$ has law $\qdiskweighted{\ell}$,
let $x \in \partial \CD$ be chosen uniformly from the boundary measure $\qbmeasure{h}$,
and let $y$ be the point of $\partial \CD$ so that $\qbmeasure{h}([x,y]_{\partial \CD}^{\ccw})=(1-a)\ell$.
Let $\eta'$ be an independent radial $\SLE_{6}$ on $\CD$ from $x$ targeted at $0$
parameterized according to quantum natural time, set $t_{\eta'}:=\inf(\eta')^{-1}(0)$,
let $T_{t}$ denote the quantum length of $\partial K_{t} \cap \D$
for $t \in [0,t_{\eta'})$, set $T_{t}:=0$ for $t \in [t_{\eta'},\infty)$, and set
$\tau:=\inf\bigl\{ t \in [0,t_{\eta'}) \bigm| y \in \overline{K_{t}} \bigr\}$,
where $K_{t}$ denotes the complement in $\D$ of the $0$-containing component of $\D \setminus \eta'([0,t])$.
Then $\tau < t_{\eta'}$ a.s., and there exist constants $c_{1},c_{2},c_{3} \in (0,\infty)$
independent of $\ell,a$ such that for any $\delta,u \in (0,\infty)$,
as long as $\ell \leq \delta^{2/3} \exp( c_{3} \delta^{-u} )$ or $a \in [1/10,9/10]$,
\begin{equation} \label{eq:top_length_tail_disk}
\mathbb{P}\biggl[\sup_{t\in[0,\tau\wedge\delta]}T_{t}\geq\delta^{2/3-u}\biggr]
  \leq c_{1} \exp(-c_{2}\delta^{-u}).
\end{equation}
\end{lemma}

\begin{proof}
Note that by the locality of $\SLE_{6}$ (see, e.g., \cite[Theorem 3]{SchrammWilson2005}),
$\eta'|_{[0,\tau]}$ has the same law as $\widetilde{\eta}'|_{[0,\widetilde{\tau}]}$
for an independent chordal $\SLE_{6}$ $\widetilde{\eta}'$ on $\CD$ from $x$ to $y$
parameterized according to quantum natural time and
$\widetilde{\tau}:=\inf\bigl\{ t \in [0,t_{\widetilde{\eta}'}) \bigm| 0 \in \widetilde{K}_{t} \bigr\}$,
where $t_{\widetilde{\eta}'}:=\inf(\widetilde{\eta}')^{-1}(y)$ and
$\widetilde{K}_{t}$ denotes the complement in $\D$ of the component of
$\D \setminus \widetilde{\eta}'([0,t])$ whose boundary contains $y$.
Then since $\widetilde{\tau} < t_{\widetilde{\eta}'}$ a.s.\ by \cite[Proposition 11.7]{BN16LectSLE}
and $0 \not\in \widetilde{\eta}'([0,\widetilde{\tau}])$ a.s.\ by \cite[Theorem 11.2-(b)]{BN16LectSLE},
we have $0 \not\in \eta'([0,\tau])$ a.s.\ and thus $\tau < t_{\eta'}$ a.s.

To prove~\eqref{eq:top_length_tail_disk}, let $\widetilde{T}_{t}$ denote the quantum
length of $\partial \widetilde{K}_{t} \cap \D$ for $t \in [0,t_{\widetilde{\eta}'})$,
set $\widetilde{T}_{t}:=0$ for $t \in [t_{\widetilde{\eta}'},\infty)$,
and let $L_{t},R_{t},X^{1}_{t},X^{2}_{t},\tau_{0}$ be as in Lemma~\ref{lem:boundary_length_rn_disk}
with $\widetilde{\eta}'$ in place of $\eta'$. Let $\delta,u\in(0,\infty)$.
Then by $T_{\tau\wedge\delta}\leq\lim_{t\uparrow\tau\wedge\delta}T_{t}$ (since both of $L$ and $R$ have downward jumps),
the locality of $\SLE_{6}$ mentioned above, and~\eqref{eqn:bdisk_area_density},
\begin{align} \label{eq:top_length_tail_disk_proof1}
&\qdiskweighted{\ell}\biggl[\sup_{t\in[0,\tau\wedge\delta]}T_{t}>\delta^{2/3-u}\biggr] \\
&= \qdiskweighted{\ell}\biggl[\sup_{t\in[0,\tau\wedge\delta)}T_{t}>\delta^{2/3-u}\biggr]
  = \qdiskweighted{\ell}\biggl[\sup_{t\in[0,\widetilde{\tau}\wedge\delta)}\widetilde{T}_{t}>\delta^{2/3-u}\biggr] \notag \\
&\leq \qdiskweighted{\ell}\biggl[\sup_{t\in[0,\delta]}\widetilde{T}_{t}>\delta^{2/3-u}\biggr]
  = \int_{\{\sup_{t\in[0,\delta]}\widetilde{T}_{t}>\delta^{2/3-u}\}} \frac{\qmeasure{h}(\CD)}{\int \qmeasure{h}(\CD) \, d\qdisk{\ell}} \, d\qdisk{\ell} \notag \\
&\leq \qdisk{\ell}\biggl[\sup_{t\in[0,\delta]}\widetilde{T}_{t}>\delta^{2/3-u}\biggr]^{\frac{1}{5}} \frac{\bigl( \int \qmeasure{h}(\CD)^{5/4} \, d\qdisk{\ell} \bigr)^{4/5}}{\int \qmeasure{h}(\CD) \, d\qdisk{\ell}}
  = c_{4} \qdisk{\ell}\biggl[\sup_{t\in[0,\delta]}\widetilde{T}_{t}>\delta^{2/3-u}\biggr]^{\frac{1}{5}}, \notag
\end{align}
where $c_{4}:=\bigl(\int_{0}^{\infty}(2\pi)^{-1/2}s^{-5/4}e^{-1/(2s)}\,ds\bigr)^{4/5}$.

Since~\eqref{eq:top_length_tail_disk} is obvious for $\delta\in[1,\infty)$, we may assume that $\delta\in(0,1)$.
Noting that $\widetilde{T}_{t}=L_{t}-\inf_{s\in[0,t]}L_{s}+R_{t}-\inf_{s\in[0,t]}R_{s}$ for any $t\in[0,\infty)$,
define an $\{ \mathescr{F}^{2}_{t+} \}_{t\in[0,\infty)}$-stopping time $\tau_{\delta}$ by
\begin{equation*}
\tau_{\delta}(\omega_{1},\omega_{2}):=\inf\biggl\{ t \in [0,\infty) \biggm| \omega_{1}(t) - \inf_{s\in[0,t]}\omega_{1}(s) + \omega_{2}(t) - \inf_{s\in[0,t]}\omega_{2}(s)  > \delta^{2/3-u} \biggr\}.
\end{equation*}
Then by Lemma~\ref{lem:boundary_length_rn_disk} with $\tau=\tau_{\delta}$, the scaling
property of $X^{1},X^{2}$ and Proposition~\ref{prop:stable-reflected-exp-moment}, with the constants
$c_{1},c_{2}\in(0,\infty)$ as in Proposition~\ref{prop:stable-reflected-exp-moment} we have
\begin{align}
&\qdisk{\ell}\biggl[ \sup_{t\in[0,\delta]}\widetilde{T}_{t}>\delta^{2/3-u} \biggr]
  = \qdisk{\ell}[ \tau_{\delta}(L_{\cdot},R_{\cdot}) < \delta, \, \tau_{\delta}(L_{\cdot},R_{\cdot}) < t_{\widetilde{\eta}'} ] \notag \\
&=\mathbb{E}\bigl[ \ell^{5/2} \bigl( X^{1}_{\tau_{\delta}(X^{1},X^{2})}+X^{2}_{\tau_{\delta}(X^{1},X^{2})} \bigr)^{-5/2} \indicator_{ \{ \tau_{\delta}(X^{1},X^{2}) < \delta, \, \tau_{\delta}(X^{1},X^{2}) < \tau_{0} \} } \bigr] \notag \\
&\leq \ell^{5/2} ( \delta^{2/3-u} )^{-5/2} \mathbb{P}[ \tau_{\delta}(X^{1},X^{2}) < \delta ] \notag \\
&\leq \ell^{5/2} \delta^{-5/3} \mathbb{P}\biggl[ \sup_{t\in[0,1]}\biggl(X^{1}_{t} - \inf_{s\in[0,t]}X^{1}_{s} + X^{2}_{t} - \inf_{s\in[0,t]}X^{2}_{s} \biggr) > \delta^{-u} \biggr] \notag \\
&\leq 2c_{1} \ell^{5/2} \delta^{-5/3} \exp(-c_{2}\delta^{-u}/2).
\label{eq:top_length_tail_disk_proof2}
\end{align}
It thus follows from~\eqref{eq:top_length_tail_disk_proof1} and~\eqref{eq:top_length_tail_disk_proof2}
that~\eqref{eq:top_length_tail_disk} with $(2c_{1},c_{2}/3)$ in place of $(c_{1},c_{2})$
holds as long as $\ell \leq \delta^{2/3} \exp(c_{2}\delta^{-u}/15)$.

Next, assume that $\ell \geq \delta^{2/3} \exp(c_{2}\delta^{-u}/15)$ and that
$a \in [1/10,9/10]$, and define an $\{ \mathescr{F}^{2}_{t+} \}_{t\in[0,\infty)}$-stopping time $\tau'$ by
\begin{equation} \label{eq:top_length_tail_disk_proof3}
\tau'(\omega_{1},\omega_{2}):=\inf\{ t \in [0,\infty) \mid \omega_{1}(t) \wedge \omega_{2}(t) < \ell/20 \}.
\end{equation}
Then $\tau'(L_{\cdot},R_{\cdot}) < t_{\widetilde{\eta}'}$ a.s.\ since
$\lim_{t\uparrow t_{\widetilde{\eta}'}}(L_{t},R_{t})=(0,0)$ a.s.\ by Lemma~\ref{lem:boundary_length_rn_disk}, and we also have
$X^{1}_{\tau'(X^{1},X^{2})} + X^{2}_{\tau'(X^{1},X^{2})} \geq X^{1}_{\tau'(X^{1},X^{2})} \vee X^{2}_{\tau'(X^{1},X^{2})} \geq \ell/20$
a.s.\ on $\{ \tau'(X^{1},X^{2}) < \tau_{0} \}$ by noting that
$(X^{1}_{t}-\lim_{s\uparrow t}X^{1}_{s})(X^{2}_{t}-\lim_{s\uparrow t}X^{2}_{s})=0$
for any $t\in(0,\infty)$ a.s.\ by the independence of $X^{1},X^{2}$. Therefore by
Lemma~\ref{lem:boundary_length_rn_disk} with $\tau=\tau',\tau_{\delta}$,
the last inequality in~\eqref{eq:top_length_tail_disk_proof2}, $X^{1}_{0}\wedge X^{2}_{0} \geq \ell/10$,
the scaling property of $X^{1},X^{2}$, \cite[Chapter~VIII, Proposition~4]{bertoin1996levy}
and $\delta^{-2/3} \ell \geq \exp(c_{2}\delta^{-u}/15)$ we obtain
\begin{align}
&\qdisk{\ell}\biggl[ \sup_{t\in[0,\delta]}\widetilde{T}_{t}>\delta^{2/3-u} \biggr] \notag \\
&\leq \qdisk{\ell}[ \tau'(L_{\cdot},R_{\cdot}) \leq \delta ]
  + \qdisk{\ell}\biggl[ \sup_{t\in[0,\delta]}\widetilde{T}_{t}>\delta^{2/3-u}, \, \tau'(L_{\cdot},R_{\cdot}) > \delta \biggr] \notag \\
&\leq \qdisk{\ell}[ \tau'(L_{\cdot},R_{\cdot}) \leq \delta ]
  + \qdisk{\ell}\bigl[ \tau_{\delta}(L_{\cdot},R_{\cdot}) < \delta\wedge t_{\widetilde{\eta}'}, \, L_{\tau_{\delta}(L_{\cdot},R_{\cdot})} + R_{\tau_{\delta}(L_{\cdot},R_{\cdot})} \geq \ell/20 \bigr] \notag\\
&\leq 20^{5/2} \bigl( \mathbb{P}[ \tau'(X^{1},X^{2}) \leq \delta ] + \mathbb{P}[ \tau_{\delta}(X^{1},X^{2}) < \delta ] \bigr) \notag \\
&\leq 20^{5/2} \biggl( \mathbb{P}\biggl[ \min_{k\in\{1,2\}}\inf_{t\in[0,1]}(X^{k}_{t}-X^{k}_{0}) \leq -\delta^{-2/3}\frac{\ell}{20} \biggr] + 2c_{1}\exp(-c_{2}\delta^{-u}/2) \biggr) \notag \\
&\leq 20^{5/2} \bigl( c_{5} (\delta^{-2/3}\ell/20)^{-3/2} + 2c_{1}\exp(-c_{2}\delta^{-u}/2) \bigr)
  \leq c_{6}\exp(-c_{2}\delta^{-u}/10)
\label{eq:top_length_tail_disk_proof4}
\end{align}
for some constants $c_{5},c_{6}\in(0,\infty)$ determined solely by the law of $X^{1}-X^{1}_{0}$.
Thus~\eqref{eq:top_length_tail_disk_proof1} and~\eqref{eq:top_length_tail_disk_proof4}
yield~\eqref{eq:top_length_tail_disk} when $\ell \geq \delta^{2/3} \exp(c_{2}\delta^{-u}/15)$
and $a \in [1/10,9/10]$, completing the proof.
\end{proof}

Suppose now that we have the setup of Proposition~\ref{prop:good_chunks_percolate_half_plane}.  We scale time by $\delta \in (0,1)$ and lengths by $\delta^{2/3}$ so that each chunk $\CN_j$ is drawn by a chordal $\SLE_6$ curve stopped at a stopping time which is at most $\delta$ and note that the setup is scale invariant.  For every $j \in \N_0$,  we let $(L^j,R^j)$ be the pair of $3/2$-stable L\'evy processes describing the boundary length evolution of the quantum surface parameterized by $\CN_j$.  We also define a process $X = (X_t)_{t \geq 0}$ as follows.  First,  we set $X_t = L_t^0 + R_t^0$ for each $t \in [0,\sigma_0]$.  Fix $j \in \N$ and suppose that we have defined $X_t$ for each $t \in [0,\sigma_{j-1}]$.  Then we set $X_t = X_{\sigma_{j-1}} + L_t^j + R_t^j$ for each $t \in [\sigma_{j-1},\sigma_j]$ and we proceed inductively.  Since both of $L^j$ and $L^j$ are $3/2$-stable L\'evy processes,  the sequences $(L^j)_{j \geq 0},  (R^j)_{j \geq 0}$ are i.i.d.\  and the exploration is constructed in a Markovian way,  we obtain that $X$ has the law of a $3/2$-stable L\'evy process.  We note that $X$ has only downward jumps and a jump occurs whenever a curve $\eta_j'$ either finishes tracing an $\SLE_6$ bubble whose boundary is bounded away from $\partial \h$ or it disconnects points of $\partial \h$ from $\infty$.  In either case,  we have that the size of the jump of $X$ is equal to the quantum boundary length of the quantum surface parameterized by the region that $\eta_j'$ disconnects from $\infty$.  Moreover,  we define inductively a collection of quantum surfaces $\wt{\CW} = (\wt{\CW}_t)_{t \geq 0}$ as follows.  First,  for $t \in [0,\sigma_0]$,  we let $\wt{\CW}_t$ be the quantum surface parameterized by the unbounded connected component of $\h \setminus \eta_0'([0,t])$.  Fix $j \in \N$ and suppose that we have defined $\wt{\CW}_t$ for $t \in [0,\sigma_{j-1}]$.  Then for $t \in [\sigma_{j-1},\sigma_j]$,  we let $\wt{\CW}_t$ be the quantum surface parameterized by the unbounded connected component of $\h \setminus \left(\eta_j'([0,t]) \cup (\cup_{i=0}^{j-1}\CN_i)\right)$.  Then for each $t \geq 0$,  we have that $X_t$ is equal to the boundary length of $\partial \wt{\CW}_t \setminus \partial \wt{\CW}_0$ minus the boundary length of $\partial \wt{\CW}_0 \setminus \partial \wt{\CW}_t$.

Let us now briefly describe the main strategy for proving Proposition~\ref{prop:good_chunks_percolate}.  First,  we will show in Lemma~\ref{lem:boundary_length_disconnection_chordal} that with positive probability (uniform in $\ell>0,\delta \in (0,1)$),  we have that the chordal exploration in Proposition~\ref{prop:good_chunks_percolate_half_plane} disconnects from $\infty$ at least $\ell/2$ units of quantum boundary length in $(-\infty,0)$ starting from $0$ during the first $\lfloor \ell \delta^{-2/3 - u/3} \rfloor + 1$ number of steps and using only good chunks,  without disconnecting from $\infty$ too many units of quantum area.  Next,  in Lemma~\ref{lem:boundary_length_disconnection_chordal_unit_disk},  we show that the same is true if we consider the chordal $\SLE_6$ exploration on top of a quantum disk instead.  For this,  we are going to use Lemma~\ref{lem:boundary_length_rn_disk} in order to compare the laws of the two explorations on top of the quantum wedge and the quantum disk respectively.  Moreover,  using Lemma~\ref{lem:boundary_length_rn_disk_weighted},  we show in Lemma~\ref{lem:boundary_length_disconnection_radial_unit_disk} that the statement of Lemma~\ref{lem:boundary_length_disconnection_chordal_unit_disk} still holds if we consider the radial $\SLE_6$ exploration instead.  Finally,  we complete the proof of Proposition~\ref{prop:good_chunks_percolate} by applying Lemma~\ref{lem:boundary_length_disconnection_radial_unit_disk} iteratively.

\begin{lemma}\label{lem:boundary_length_disconnection_chordal}
Suppose that we have the setup of Proposition~\ref{prop:good_chunks_percolate_half_plane} and the setup described in the previous paragraphs.  Fix $0<u_1<u<1/3$,  $\ell \in [\delta^{2/3 - u}, \delta^{-u_1}]$ and set $N = \lfloor \ell \delta^{-2/3 - u/3} \rfloor +1$.  Let $b>0$ (resp.  $\wt{b}>0$) be such that $\nu_h([-b,0]) = \ell / 2$ (resp.  $\nu_h([-\wt{b},0]) = 3\ell / 4)$.  We let $E$ be the event that the following hold.
\begin{enumerate}[(i)]
\item There exists $1 \leq j \leq N/2$ such that the first $j$ number of chunks in the exploration disconnect $[-b,0]$ but not $[-\wt{b},0]$ from $\infty$,  and $\partial \CW_j$ intersects the boundaries of only good chunks.
\item We have that $\sup_{0 \leq t \leq 2\ell \delta^{1/3 - u/3}} |X_t| < \ell \delta^{u/200}$ and $\inf_{0 \leq t \leq \delta} L_t^0 > -\ell \delta^{u/200}$.  Moreover,  the total amount of quantum area disconnected from $\infty$ by the first $N$ number of chunks is at most a constant times $\delta^{u/18}$ where the implicit constant depends only on $c_0,A,u$ and $u_1$.
\item The chunk $\CN_0$ is good and the top of $\partial \CN_0$ is not disconnected from $\infty$ by the first $N$ number of chunks of the exploration.  Moreover,  the boundary length of the part of $\partial \CN_0$ contained in the boundary of the unbounded connected component of $\h \setminus \cup_{i=0}^N \CN_i$ is at least $\epsilon_0 \delta^{2/3}$.
\item For each $i \in \Z \cap [0,N]$,  there exists $n \in [1,\log(\delta^{-1})^3] \cap \N$ and $\{i_j\}_{j=1}^n \subseteq \Z \cap [0,N]$ such that $i_1=i>i_2>\cdots >i_n,  \partial \CN_{i_n} \cap \partial \CW_0 \neq \emptyset$ and $\partial \CN_{i_j} \cap \partial \CN_{i_{j+1}} \neq \emptyset$,  for each $1 \leq j<n$.
\end{enumerate}
Then there exist $\wt{q}$, $\delta_0 \in (0,1)$ depending only on $A,c_0,u$ and $u_1$ such that for each $\delta \in (0,\delta_0)$,  we have that $\p[E] \geq 1-\wt{q}$.
\end{lemma}

\begin{proof}
\emph{Step~1.  (i) holds with high probability}.  First we note that part~(\ref{it:move_left}) of Proposition~\ref{prop:good_chunks_percolate_half_plane} implies that there exist constants $c_1,c_2>0$ depending only on $A,c_0$ such that $\p\big[L_{\lfloor N/2 \rfloor} \geq -c_1 N \delta^{2/3}  \big] \leq e^{-c_2 N}$,  where $L_j$ is defined as in the statement of Proposition~\ref{prop:good_chunks_percolate_half_plane}.  Note that if $L_{\lfloor N/2 \rfloor} < -c_1 N \delta^{2/3}$,  then we have that the first $\lfloor N/2 \rfloor$ number of chunks of the exploration have disconnected from $\infty$ at least $c_1 N \delta^{2/3}$ units of quantum boundary length in $(-\infty , 0)$ starting from $0$.  Note also that $c_1 N \delta^{2/3} \geq c_1 \ell \delta^{-u/3} > \ell / 2$ for each $\delta \in (0,1)$ sufficiently small.  In particular,  the first $\lfloor N/2 \rfloor$ chunks of the exploration disconnect $[-b,0]$ from $\infty$ off an event with probability at most $e^{-c_2 N}$.  Moreover,  \cite[Chapter~VIII,  Proposition~4]{bertoin1996levy} implies that $\p\big[ \inf_{0 \leq t \leq \delta} L_t^j < -\ell \delta^{u/100}/16 \big] \lesssim \ell^{-3/2} \delta^{1-3u/200}$ for each $j$,  with the implicit constant being universal.  It follows by taking a union bound over all $0 \leq j \leq N$ and since $\ell \geq \delta^{2/3 - u}$ that
\begin{align*}
\p\big[ \inf_{0 \leq t \leq \delta} L_t^j < -\ell \delta^{u/100} /16\,\,\text{for some}\,\,0 \leq j \leq N \big]\lesssim \delta^{a(u)}
\end{align*}
where $a(u) > 0$ depends only on $u$ and the implicit constant is universal.  Also,  part~(\ref{it:good_chunks_initial_boundary}) of Proposition~\ref{prop:good_chunks_percolate_half_plane} implies that there exist constants $c_3,c_4>0$ depending only on $A,c_0$ such that off an event with probability at most $c_3 N^2 \exp(-c_4 \delta^{-u/200})$ the following holds.  For each $i \in \Z \cap [0,N]$ there exists $j \in \Z \cap [0,N)$ such that $|j-i| \leq \delta^{-u/200},  \indicator_{E_j \cap E_{j+1}} = 1$ and $\partial \CN_j \cap \partial \CW_{\text{L}} \neq \emptyset$.  

Suppose that all of the events described in the previous paragraph hold and let $i$ be the smallest $i \in \Z \cap [0,N/2]$ such that the first $i$ number of chunks of the exploration disconnect $[-b,0]$ from $\infty$.  Then there exists $j>i$ such that $|j-i| \leq 2\delta^{-u/200},  \indicator_{E_j \cap E_{j+1}} = 1$ and $\partial \CN_j \cap \partial \CW_{\text{L}} \neq \emptyset$.  Since $\inf_{0 \leq t \leq \delta} L_t^i > -\ell \delta^{u/100} / 16 > -\ell / 16$,  we obtain that $\partial \CN_i \cap \partial \CW_{\text{L}} \subseteq [-b_1,0]$,  where $b_1>0$ is such that $\nu_h([-b_1,0]) = \ell / 2 + \ell / 16$.  Furthermore,  the boundary length of $\partial(\CW_i \setminus \CW_j) \cap \partial_{\text{L}} \CW$ is at most $-\sum_{m=i+1}^j \inf_{0 \leq t \leq \delta} L_t^m \leq  \ell \delta^{u/200} / 8 < \ell / 8$,  and so we obtain that $\partial \CN_j \cap \partial_{\text{L}} \CW \subseteq [-b_2,0]$ where $b_2>0$ is such that $\nu_h([-b_2,0]) = 11\ell / 16 < 3\ell / 4$.  It follows that there exist $a>0,\delta_0 \in (0,1)$ depending only on $A,c_0$ and $u$ such that for each $\delta \in (0,\delta_0)$,  we have off an event with probability at most $\delta^a$ that there exists $1 \leq j \leq N/2$ such that $E_j$ occurs and the chunks discovered during the first $j$ steps of the exploration disconnect $[-b,0]$ but not $[-\wt{b},0]$ from $\infty$.  Moreover,  it follows from the way that we have defined the exploration that $\CN_m$ is good for each $0 \leq m\leq j$ such that $\partial \CN_m \cap \partial \CW_j \neq \emptyset$.  Therefore (i) holds with high probability.

\emph{Step~2.  (ii) holds with high probability}.  Fix $p \in (1,3/2)$.  Then \cite[Chapter~VII,Corollary~2]{bertoin1996levy} combined with \cite[Chapter~VIII,  Proposition~4]{bertoin1996levy} imply that $\E\big[ \sup_{0 \leq t \leq 1} |X_t|^p \big] < \infty$.  Thus,  the maximal inequality for martingales combined with scaling imply that
\begin{align*}
\E\left[ \sup_{0 \leq t \leq 2 \ell \delta^{1/3 - u/3}} |X_t| \right]& \leq \E\left[ \sup_{0 \leq t \leq 2\ell \delta^{1/3 - u/3}} |X_t|^p \right] ^{1/p}\\
&=(2\ell \delta^{1/3 - u/3})^{2/3} \E\left[ \sup_{0 \leq t \leq 1} |X_t|^p \right]^{1/p} \lesssim \ell^{2/3} \delta^{2/9 - 2u/9}.
\end{align*}
Hence,  Markov's inequality implies that 
\begin{align*}
\p\left[ \sup_{0 \leq t \leq 2\ell^{1/3 - u/3}} |X_t| \geq \ell \delta^{u/200} \right] \lesssim \ell^{-1} \delta^{-u/200} \ell^{2/3} \delta^{2/9 - 2u/9} \lesssim \delta^{u/9 - u/200}
\end{align*}
since we also have that $\ell \geq \delta^{2/3 - u}$. We note that during the first $N$ steps of the exploration,  we have that at most $2\ell \delta^{1/3 - u/3}$ units of time elapsed for $X$.  Also,  the expectation of the sums of squares of the jumps of $X$ of size at most $\ell$ made in the time interval $[0,2\ell \delta^{1/3 - u/3}]$,  is given by $c2\ell \delta^{1/3 - u/3} \int_{x=0}^{\ell} x^2 x^{-5/2} dx = 4c \ell^{3/2} \delta^{1/3 - u/3}$,  where $c>0$ is a universal constant.  Furthermore,  the number of jumps of size larger than $\ell$ made by $X$ in $[0,2\ell \delta^{1/3 - u/3}]$ has the law of a Poisson random variable with mean $2c\ell \delta^{1/3 - u/3} \int_{x=\ell}^{\infty} x^{-5/2}dx = 4c\ell^{-1/2}\delta^{1/3 - u/3} / 3$.  It follows that the probability of having a jump of size at least $\ell$ in $[0,2\ell \delta^{1/3 - u/3}]$ is given by $1-\exp(-4c\ell^{-1/2} \delta^{1/3-u/3} / 3)$ which is at most $4c \delta^{u/6} / 3$ since recall that $\ell \geq \delta^{2/3 - u}$.  Suppose that we are working on the event that we don't have a jump of size at least $\ell$ in $[0,2\ell \delta^{1/3 - u/3}]$.  Then \cite[Theorem~1.16]{dms2014mating} implies that the conditional expectation of the total quantum area disconnected from $\infty$ by the first $N$ number of chunks given the sizes of the jumps of $X$ in $[0,2\ell \delta^{1/3 - u/3}]$ is at most $4c \ell^{3/2} \delta^{1/3 - u/3}$.  It follows from the Markov property that the conditional probability that more than $\ell^{3/2} \delta^{1/3 - u/2 + u/9}$ units of quantum area is disconnected from $\infty$ by the first $N$ number of chunks of the exploration is at most $4c\ell^{-3/2} \delta^{-1/3 + u/2 -u/9} \ell^{3/2} \delta^{1/3 - u/3}/3 = 4c\delta^{u/18} / 3$.  Note that $\ell^{3/2} \delta^{1/3 - u/2 + u/9} \leq \delta^{-3u_1 / 2 + 1/3 - 7u / 8}$ since $\ell \leq \delta^{-u_1}$. Therefore,  possibly by taking the constant $a$ in Step 1 to be smaller,  we can assume that (ii) holds off an event with probability at most $\delta^a$ for all $\delta \in (0,1)$ sufficiently small (depending only on $A,c_0,u_1$ and $u$).

\emph{Step 3.  Conclusion of the proof}.  We note that Lemma~\ref{lem:right-length-positive-prob-positive} implies that provided $A^{-1} \vee c_0$ is small enough,  we have that there exists $\wt{q} \in (0,1)$ depending only on $A,c_0$ such that with probability at least $1-\wt{q}$,  the following holds.  The chunk $\CN_0$ is good and the top of $\partial \CN_0$ is never disconnected from $\infty$ by the exploration.  Note also that if the latter occurs,  we have that the boundary length of the part of $\partial \CN_0$ contained in the boundary of the unbounded connected component of $\h \setminus \cup_{i=0}^N \CN_i$ is at least $\epsilon_0 \delta^{2/3}$.  This proves (iii).  Finally,  for (iv),  we note that part~(\ref{it:bad_cluster_size}) of Proposition~\ref{prop:good_chunks_percolate_half_plane} combined with the fact that $\ell \leq \delta^{-u_1}$ imply that there exist constants $c_5,c_6>0$ depending only on $A,c_0,u$ and $u_1$ such that for all $\delta \in (0,1)$ sufficiently small (depending only on $A,c_0,u$ and $u_1$),  we have that off an event with probability at most $c_5 \exp(-c_6 \log(\delta^{-1})^{3/2})$ the claim in part (iv) of the statement of the lemma holds.  Therefore,  possibly by taking $\wt{q} \in (0,1)$ to be smaller and for all $\delta \in (0,1)$ sufficiently small (depending only on $A,c_0,u$ and $u_1$),  we can assume that parts (i)-(iv) hold simultaneously with probability at least $\wt{q}$.  This completes the proof of the lemma.
\end{proof}

Next,  we will prove that the statement of Lemma~\ref{lem:boundary_length_disconnection_chordal} still holds if we perform the analogous exploration using chordal $\SLE_6$ chunks on top of a quantum disk instead.  This is the content of the following lemma.
\begin{lemma}\label{lem:boundary_length_disconnection_chordal_unit_disk}
Fix $0<u_1<u<1/3$ and $\delta^{2/3 - u} \leq \ell \leq \delta^{-u_1}$ for $\delta \in (0,1)$.  Suppose that $\CD = (\D,h)$ is a sample from $\qdisk{\ell}$ and let $x \in \partial \D$ be chosen uniformly according to $\nu_h$.  Let also $y$ be the point which is antipodal to $x$ (with respect to quantum boundary length),  $z$ be the point on the clockwise arc of $\partial \D$ from $x$ to $y$ such that the boundary length of the clockwise arc of $\partial \D$ from $x$ to $z$ is equal to $\ell / 4$,  and let $w$ (resp.  $\wt{z}$) be the point on the clockwise (resp.  counterclockwise) arc of $\partial \D$ from $x$ to $y$ such that the boundary length of the clockwise (resp.  counterclockwise) arc of $\partial \D$ from $x$ to $w$ (resp.  from $x$ to $\wt{z}$) is equal to $3\ell/8$ (resp.  $\ell / 4$).  Suppose that we perform the exploration on top of $\CD$ in the same way that we did on top of the quantum wedge except that each chunk is formed by a chordal $\SLE_6$ starting from the marked point of the chunk and targeted at $y$.  We stop the exploration at the first time that we discover a chunk which contains $y$,  so that the corresponding chordal $\SLE_6$ curve hits $y$ before the chunk is formed.  Let $E_1$ be the event that the following hold.  
\begin{enumerate}[(i)]
\item There exists $1 \leq j \leq N/2$ such that the first $j$ number of chunks of the exploration disconnect from booth $y$ and $0$ the clockwise arc of $\partial \D$ from $x$ to $z$ without disconnecting from $y$ either the clockwise arc of $\partial \D$ from $x$ to $w$ or the counterclockwise arc of $\partial \D$ from $x$ to $\wt{z}$.  Also,  if $\wt{\CN}_0,\ldots,\wt{\CN}_j$ are the corresponding chunks of the exploration,  we have that the boundary of the connected component of $\D \setminus \cup_{i=0}^j \wt{\CN}_i$ containing $0$ intersects the boundaries of only good chunks.
\item $ \wt{\CN}_0$ is good and the top of $\partial \wt{\CN}_0$ is not disconnected either from $0$ or $y$ and the boundary length of the part of $\partial \wt{\CN}_0$ contained in the boundary of the connected component of $\D \setminus \cup_{i=0}^j \wt{\CN}_i$ containing $0$ is in $[\epsilon_0 \delta^{2/3},\ell \delta^{u/200}]$. 
\item For each $0 \leq i \leq N/2$,  we let $D_i$ be the connected component of $\D \setminus \cup_{m=0}^i \wt{\CN}_m$ whose boundary contains $y$.  Then,  we have that the boundary length of $\D \cap \partial D_i$ is at most $\ell / 100$ plus the boundary length of $\partial \D \setminus \partial D_i$. 
\item For each $0 \leq i \leq j$,  there exist $\{i_j\}_{j=1}^n \subseteq \Z \cap [0,j]$ and $1 \leq n \leq \log(\delta^{-1})^3$,  such that $i = i_1>i_2>\cdots>i_n$, $\partial \wt{\CN}_{i_n} \cap \partial \D$,  and $\partial \wt{\CN}_{i_{m+1}} \cap \partial \wt{\CN}_{i_m} \neq \emptyset$ for each $0 \leq m \leq n-1$.
\end{enumerate}
Then,  there exist $p_1,  \delta_0 \in (0,1)$ depending only on $A,c_0,u_1$ and $u$,  such that for each $\delta \in (0,\delta_0)$,  we have that $\p[E_1] \geq p_1$.
\end{lemma}

\begin{proof}
Suppose that we have the setup of Lemma~\ref{lem:boundary_length_disconnection_chordal}.  Let $\wt{E}$ be the event defined in the same way as $E$ but with $\ell / 2$ in place of $\ell$ and for the exploration with respect to the quantum disk $\CD$ instead and the marked points $0$ and $\infty$ replaced by $x$ and $y$ respectively.  Fix $j \in \N$.  Then Lemma~\ref{lem:boundary_length_rn_disk} implies that conditional on the event that the exploration in $\CD$ has not ended during the first $j-1$ steps and on the boundary length $\ell_{j-1}^{\text{L}}$ (resp.  $\ell_{j-1}^{\text{R}}$) of the clockwise (resp.  counterclockwise) arc of $\partial D_{j-1}$ from $\eta_j'(0)$ to $y$,  we have that the Radon-Nikodym derivative of the law of $\wt{\CN}_j$ with respect to the law of $\CN_j$ (when both viewed as quantum surfaces) is given by $\Bigl(\frac{L_{\sigma_j}^j + R_{\sigma_j}^j}{\ell_{j-1}^{\text{L}}+\ell_{j-1}^{\text{R}}} +1 \Bigr)^{-5/2} \indicator_{\{\sigma_j <S^j\}}$,  where $S^j$ is the first time $t$ such that either $L_t^j \leq -\ell_{j-1}^{\text{L}}$ or $R_t^j \leq -\ell_{j-1}^{\text{R}}$.  Let $j_0 \in \N$ be the smallest integer for which condition (i) of Lemma~\ref{lem:boundary_length_disconnection_chordal} is satisfied. Therefore,  the Radon-Nikodym derivative of the law of the quantum surfaces $(\wt{\CN}_0,\ldots,\wt{\CN}_{j_0})$ with respect to the law of the quantum surfaces $(\CN_0,\ldots,\CN_{j_0})$ is given by $\Bigl(\frac{\ell_{j_0}^{\text{L}}+\ell_{j_0}^{\text{R}}}{\ell}\Bigr)^{-5/2}$.  Note that $\ell_{j_0}^{\text{L}}+\ell_{j_0}^{\text{R}} < 2\ell$ and $\ell_{j_0}^{\text{L}}+ \ell_{j_0}^{\text{R}} > \ell / 2$ if $E$ occurs which implies that $\Bigl(\frac{\ell_{j_0}^{\text{L}}+\ell_{j_0}^{\text{R}}}{\ell}\Bigr)^{-5/2} \asymp 1$ with the implicit constants being universal.  Thus,  combining with Lemma~\ref{lem:boundary_length_disconnection_chordal}, there exist $\delta_0, \wt{p} \in (0,1)$ depending only on $A,c_0,u$ and $u_1$ such that $\p\big[ \wt{E} \big] \geq \wt{p}$ for each $\delta \in (0,\delta_0)$.

Next,  conditionally on $h$,  we sample $w \in \D$ independently according to the probability measure $\frac{\mu_h}{\mu_h(\D)}$ on $\D$,  and set $\wt{h}:=h \circ \phi^{-1} + Q \log |(\phi^{-1})'|$ where $\phi : \D \to \D$ is the conformal transformation such that $\phi(w) = 0$ and $\phi'(w) > 0$.  Then the marginal law of $(\D,\wt{h})$ is given by $\qdiskweighted{\ell}$.  Suppose that $\wt{E}$ occurs.  Then we have that the quantum area of $\cup_{i=0}^{j_0} \wt{\CN}_i$ with respect to $h$ is at most $\ell^{3/2} \delta^{1/3 - u/2 +u/9}$.  Also,  by scaling,  we have that the probability that $\mu_h(\D)$ is at least $\ell^2 \delta^{u/18}$ tends to $1$ as $\delta \to 0$,  at a rate which is uniform in $\delta$.  Hence,  by possibly taking $\wt{p},\delta_0 \in (0,1)$ to be smaller,  we can assume that the probability of $\wt{E} \cap \{\mu_h(\D) \geq \ell^2 \delta^{u/18} \}$ under $\qdisk{\ell}$ is at least $\wt{p}$.  Note that if $\wt{E} \cap \{\mu_h(\D) \geq \ell^2 \delta^{u/18}\}$ occurs,  then we have that conditional on $h$,  the probability that $w$ lies in $\cup_{i=0}^{j_0} \wt{\CN}_i$ is at most $\frac{\ell^{3/2} \delta^{1/3 - u/2+u/9}}{\ell^2 \delta^{u/18}} \leq \delta^{u/18}$. Therefore,  by possibly taking $\wt{p},\delta_0 \in (0,1)$ to be smaller,  we can assume that $\qdiskweighted{\ell}\big[\wt{E}_1\big] \geq \wt{p}$,  where $\wt{E}_1$ is the event defined in the same way as the event $E_1$ of Lemma~\ref{lem:boundary_length_disconnection_chordal} except that we consider the exploration with respect to $\wt{h}$ instead of $h$.  It follows that
\begin{align*}
\wt{p} \leq \qdiskweighted{\ell}\big[ \wt{E}_1\big] \leq \qdisk{\ell}\big[ E_1\big]^{1/5} \frac{(\int \mu_h(\D)^{5/4} d \qdisk{\ell} )^{4/5}}{\int \mu_h(\D) d\qdisk{\ell}}
\lesssim \qdisk{\ell}\big[E_1\big]^{1/5},
\end{align*}
where the implicit constant is universal.  This completes the proof of the lemma.
\end{proof}

Now that we have stated and proved Lemma~\ref{lem:boundary_length_disconnection_chordal_unit_disk},  we can state and prove the analogous version for the exploration using radial $\SLE_6$ chunks instead.
\begin{lemma}\label{lem:boundary_length_disconnection_radial_unit_disk}
Fix $0<u_1<u<1/3$ and $\delta^{2/3 - u} \leq \ell \leq \delta^{-u_1}$ for $\delta \in (0,1)$.  Suppose that $\CD = (\D,h,0)$ has law given by $\qdiskweighted{\ell}$.  Suppose also that we perform the exploration using radial $\SLE_6$ chunks as in the statement of Proposition~\ref{prop:good_chunks_percolate}.  Let $\wt{E}_1$ be the event defined in the same way as the event $E_1$ of Lemma~\ref{lem:boundary_length_disconnection_chordal_unit_disk} except that we consider the radial $\SLE_6$ exploration.  Then,  there exist $p_0,  \delta_0 \in (0,1)$ depending only on $A,c_0,u_1$ and $u$ such that $\p\big[ \wt{E}_1 \big] \geq p_0$ for each $\delta \in (0,\delta_0)$.
\end{lemma}

\begin{proof}
First,  we note that the radial $\SLE_6$ curves used to construct the chunks $(\wh{\CN}_j)$ for the radial exploration can be coupled with chordal $\SLE_6$ curves starting from the same point so that they agree up until the first time that they disconnect $y$ from $0$.  Moreover,  let $X$ be the overall boundary length process for the radial exploration.  Then,  by \cite[Theorem~7.3]{ms2015spheres},  the exploration is determined by $X$,  the conditionally independent family of quantum disks which are cut out,  and the orientations of their boundaries.  We note that the boundaries are oriented i.i.d.\  with equal probability $\frac{1}{2}$ given $X$.  We can use these orientations in order to determine the boundary length process $L$ (resp.  $R$) from the tip of the exploration clockwise (resp.  counterclockwise) to $y$,  up until the first time that the radial exploration disconnects $y$ from $0$.  Indeed,  this is because $L$ (resp.  $R$) is equal to the process which is formed by applying the deterministic function which recovers a $3/2$-stable L\'evy process from its jumps to the downward jumps of $X$ which correspond to quantum disks whose boundaries have a counterclockwise (resp.  clockwise) direction.  This implies that the following is true.  Fix $j \in \N$ and suppose that we are working on the event that the first $j-1$ chunks of the exploration do not disconnect $y$ from $0$.  Let also $\ell_{j-1}$ be the boundary length of the component containing $0$ after $j-1$ steps of radial exploration.  Then,  it follows by combining Lemmas~\ref{lem:boundary_length_rn_disk_weighted} and~\ref{lem:boundary_length_rn_disk} that the Radon-Nikodym derivative of the law of $\wh{\CN}_j$ on the event that $\wh{\CN}_j$ does not disconnect $y$ from $0$ with respect to the law of $\wt{\CN}_j$ where $\wt{\CN}_j$ is as in Lemma~\ref{lem:boundary_length_disconnection_chordal_unit_disk},  is given by $\left(\frac{\ell_j}{\ell_{j-1}}\right)^2$,  where $\ell_j$ is the boundary length of the connected component of $\D \setminus \wt{\CN}_j$ containing $0$.  It follows that on the events $\wt{E}_1,E_1$,  if $j_0$ is as in Lemma~\ref{lem:boundary_length_disconnection_chordal_unit_disk},  then the Radon-Nikodym derivative of the law of $(\wh{\CN}_0,\ldots,\wh{\CN}_{j_0})$ with respect to the law of $(\wt{\CN}_0,\ldots,\wt{\CN}_{j_0})$ is given by $\left(\frac{\wt{\ell}}{\ell}\right)^2$,  where $\wt{\ell}$ is the boundary length of the component of $\D \setminus \cup_{i=0}^{j_0} \wt{\CN}_i$ containing $0$.  Since $\frac{\wt{\ell}}{\ell} \asymp 1$ on $E_1$ with the implicit constants being universal,  the proof of the lemma is complete.
\end{proof}

Now we are ready to prove Proposition~\ref{prop:good_chunks_percolate}.  The main idea of the proof is to apply Lemma~\ref{lem:boundary_length_disconnection_radial_unit_disk} iteratively up until we find the desired chain of good chunks as in the statement of Proposition~\ref{prop:good_chunks_percolate}.   The conditions in the definition of the event $\wt{E}_1$ in Lemma~\ref{lem:boundary_length_disconnection_radial_unit_disk} will guarantee that it is possible to construct the desired chain with high probability since the probability of $\wt{E}_1$ is bounded from below by a constant which is uniform in $\delta$.

Now,  we proceed to the details of the proof.  Suppose that we have the setup of the statement of Proposition~\ref{prop:good_chunks_percolate}. Fix $u_1 \in (0,u/3)$ and let $\wt{E}_{u_1,\delta}$ be the event that for each $j \in \Z \cap [0,\delta^{-2/3 - u}]$ and each $t \in [0,\sigma_j]$,  the boundary length of the $0$-containing connected component of $D_j \setminus \eta_j'([0,t])$ is at most $\delta^{-u_1}$.  We define sequences of marked points $\{\wt{x}_j\}_{j \geq 0},  \{\wt{y}_j\}_{j \geq 0}, \{\wt{z}_j\}_{j \geq 0}$ and domains $\{\wt{D}_j\}_{j \geq 0}$ as follows.  First,  we pick $x \in \partial \D$ uniformly according to the boundary length measure and let $y$ be the point on $\partial \D$ which is antipodal to $x$ with respect to the boundary length measure..  Let also $z$ be the point on the clockwise arc of $\partial \D$ from $x$ to $y$ such that the clockwise arc of $\partial \D$ from $x$ to $z$ has boundary length equal to $1/4$. Then we set $\wt{x}_0 = x,\wt{y}_0 = y$ and $\wt{z}_0=z$.  We also set $\wt{D}_0 = \D = D_0$.  Fix $j \in \N$ and suppose that we have defined marked points $\{(\wt{x}_i,\wt{y}_i,\wt{z}_i)\}_{i=0}^j$ and domains $\{\wt{D}_i\}_{i=0}^j$.  Let $\ell_j$ be the boundary length of $\partial \wt{D}_j$ and set $N_j = \lfloor \ell_j \delta^{-2/3 -u/3}\rfloor +1$.  Let also $F_j$ be the event that the event $\wt{E}_1$ defined in Lemma~\ref{lem:boundary_length_disconnection_radial_unit_disk} occurs for the quantum surface parameterized by $\wt{D}_j$.  We also let $\wt{w}_j$ be the point on the clockwise arc of $\partial \wt{D}_j$ from $\wt{x}_j$ to $\wt{y}_j$ with boundary length distance on $\partial \wt{D}_j$ from $\wt{x}_j$ equal to $3\ell_j / 8$.  Suppose that $F_j$ occurs.  Then,  we let $j_0 \in \N_0$ be the first $0 \leq m \leq N_j$ for which we can find the desired sequence of good chunks in the definition of $F_j$ during the first $m$ steps of the exploration and starting after the last chunk discovered during the formation of $\partial \wt{D}_j$.  Then,  we let $\wt{x}_{j+1}$ be the marked point of the exploration in $\wt{D}_j$ after we have performed it for $j_0$ times.  Suppose that $F_j$ does not occur. Then,  at least one of the following has to occur for the exploration in $\wt{D}_j$ before $\wt{w}_j$ is disconnected from $0$ for the first time.
\begin{enumerate}[(i)]
\item\label{it:b1} There is a chunk which cannot be connected to $\partial \wt{D}_j$ by at most $\log(\delta^{-1})^3$ number of chunks discovered before that chunk.
\item\label{it:b2} $\wt{z}_j$ is not disconnected from $0$ by good chunks.
\item\label{it:b3} The first chunk is not good.
\item\label{it:b4} The first chunk is good but it is disconnected from $0$ before $\wt{z}_j$ is disconnected from $0$.
\item\label{it:b5} $N_j$ number of chunks have been explored.
\item\label{it:b6} $\wt{z}_j$ and $\wt{w}_j$ are disconnected from $0$ simultaneously.
\end{enumerate}
If~\eqref{it:b1} occurs,  we let $\wt{x}_{j+1}$ be the marked point of the exploration after we discover the first chunk which cannot be connected to $\partial \D$ by at most $\log(\delta^{-1})^3$ number of chunks discovered before that chunk.  If either~\eqref{it:b2} or~\eqref{it:b6} occurs,  we let $\wt{x}_{j+1}$ be the marked point of the exploration after we discover the first chunk $\CN$ which disconnects $\wt{z}_j$ from $0$.  If~\eqref{it:b3} occurs,  we let $\wt{x}_{j+1}$ be the marked point of the exploration after we discover the first chunk of the exploration in $\wt{D}_j$.  If~\eqref{it:b4} occurs,  we let $\wt{x}_{j+1}$ be the marked point of the exploration after we discover the first chunk which disconnects from $0$ the first chunk of the exploration in $\wt{D}_j$.  Finally,  if~\eqref{it:b5} occurs,  we let $\wt{x}_{j+1}$ be the marked point of the exploration in $\wt{D}_j$ after $N_j$ number of steps.  In any case,  we let $\wt{D}_{j+1}$ be the connected component containing $0$ in the exploration in $\wt{D}_j$ whose boundary contains $\wt{x}_{j+1}$.  Also,  we let $\wt{y}_{j+1}$ be the point on $\partial \wt{D}_{j+1}$ with boundary length distance from $\wt{x}_{j+1}$ in $\partial \wt{D}_{j+1}$ equal to $\ell_{j+1} / 2$,  where $\ell_{j+1}$ is the boundary length of $\partial \wt{D}_{j+1}$.  Moreover,  we let $\wt{z}_{j+1}$ (resp.  $\wt{w}_{j+1}$) be the point on the clockwise arc of $\partial \wt{D}_{j+1}$ from $\wt{x}_{j+1}$ to $\wt{y}_{j+1}$ with boundary length distance from $\wt{x}_{j+1}$ equal to $\ell_{j+1} / 4$ (resp.  $3\ell_{j+1} / 8$),  and set $N_{j+1} := \lfloor \ell_{j+1} \delta^{-2/3 - u/3} \rfloor + 1$.  We then let $F_{j+1}$ be the event that the event $\wt{E}_1$ defined in Lemma~\ref{lem:boundary_length_disconnection_radial_unit_disk} occurs for the quantum surface parameterized by $\wt{D}_{j+1}$.

\begin{proof}[Proof of Proposition~\ref{prop:good_chunks_percolate}]
Suppose that we have the setup described in the above paragraphs.  For each $j \in \N_0$,  we let $G_j^1$ be the event that the boundary length of $\wt{D}_j$ is at least $\delta^{-2/3 - u}$ and $G_j^2$ the event that the boundary length of $\partial \wt{D}_j$ is at most $\delta^{-u_1}$.  Let also $\CF_j$ be the $\sigma$-algebra generated by the chunks discovered up until $\partial \wt{D}_j$ is formed.  Lemma~\ref{lem:boundary_length_disconnection_radial_unit_disk} implies that there exists $p_0 \in (0,1)$ depending only on $A,c_0,u$ and $u_1$ such that $\p\big[ F_j \giv \CF_j \big] \indicator_{G_j} \geq p_0 \indicator_{G_j}$ for each $j \in \N$ a.s.,  where $G_j = G_j^1 \cap G_j^2$.  Thus,  by iterating and possibly taking $p_0 \in (0,1)$ to be smaller (depending only on $A,c_0,u$ and $u_1$),  we can assume that $\p\big[ \wt{F}_j \giv \CF_j \big] \indicator_{G_j} \geq p_0 \indicator_{G_j}$ for each $j \in \N$ a.s.,  where $\wt{F}_j = \cup_{i=j}^{j+4} F_i$.  Suppose that $\wt{F}_j \cap E_{u,\delta} \cap \wt{E}_{u_1,\delta}$ occurs for some $0 \leq j \leq \delta^{-u/3}$.  Then,  we will show that the required events in the statement of Proposition~\ref{prop:good_chunks_percolate} occur as well.  Indeed,  first we note that it is easy to see that the exploration disconnects $\wt{x}_j$ from $0$ using only good chunks.  Also,  since $F_j$ occurs,  if $\CN$ is the first chunk discovered in the exploration in $\wt{D}_j$,  then the boundary length of the part of $\CN$ contained in $\wt{D}_j$ is in $[\epsilon_0 \delta^{2/3} ,  \ell_j \delta^{u/200}]$. Thus,  possibly by taking $p_0 \in (0,1)$ to be smaller (depending only on $A,c_0,u$ and $u_0$),  we can assume that the aforementioned part of $\partial \CN$ is disconnected from $0$ by the exploration in $\wt{D}_j$.  It follows that both of the conditions in the definition of $N_{\delta}$ hold and that $N_{\delta} \leq 5(1+\delta^{-2/3 - u_0 -u/3}) \leq \delta^{-2/3-u}$ for all $\delta \in (0,1)$ sufficiently small (depending only on $A,c_0,u$ and $u_1$).  Furthermore,  every chunk discovered in the exploration in $\wt{D}_j$ up until $\partial \wt{D}_{j+1}$ is formed can be connected to $\partial \wt{D}_j$ using at most $\log(\delta^{-1})^3$ number of chunks.  Since the definition of the $\wt{D}_i$'s implies that for each $i$,  we have that every chunk discovered in the exploration in $\wt{D}_i$ up until $\partial \wt{D}_{i+1}$ is formed can be connected to $\partial \wt{D}_i$ using at most $2 \log(\delta^{-1})^3$ number of chunks,  we obtain that every chunk discovered in up until $\partial \wt{D}_{i+1}$ is formed can be connected to $\partial \D$ using at most $10\delta^{-u/3} \log(\delta^{-1})^3$ number of chunks and the latter is at most $\delta^{-u}$ for $\delta \in (0,1)$ sufficiently small (depending only on $A,c_0,u_1$ and $u$). Therefore,  the events in~\eqref{it:good_chunks_percolate_round_number} and~\eqref{it:good_chunks_percolate_dist_boundary} occur.

Finally,  to complete the proof,  we $\wt{F}$ be the event that $F_{5j}$ occurs for some $0 \leq j \leq \delta^{-u/3} / 5 -1$.  It follows from the previous paragraph that there exist constants $c_1,c_2>0$ depending only on $p_0$ such that $\p\big[ \wt{F}^c \cap E_{u,\delta} \cap \wt{E}_{u_1,\delta} \big] \leq c_1 \exp(-c_2 \delta^{-u/3})$.  Therefore,  it suffices to give appropriate upper bounds for $\p\big[ \wt{E}_{u_1,\delta}^c\big]$.  But the latter follows by combining Lemma~\ref{lem:boundary_length_rn_disk_weighted} with Proposition~\ref{prop:stable-reflected-exp-moment}.  In particular,  we obtain that there exist universal constants $c_3,c_4>0$ such that $\p\big[ \wt{E}_{u_1,\delta}^c \big] \leq c_3 \exp(-c_4 \delta^{-u_1})$.  This completes the proof.
\end{proof}

\section{Quenched bounds for the expected exit time from a metric ball} \label{sec:exit-time-bounds}

In this section we will  prove one main ingredient which is used to prove the upper bound in the heat kernel estimate, namely  quenched upper and lower bounds for the exit time of a Liouville Brownian motion from a ball.  For a set $A \subseteq \CS$, we let $\tau_A$ be the exit time of the Liouville Brownian motion from $A$.  Let $\qsphereinflaw$ denote the law of the infinite quantum sphere.

\begin{theorem}
\label{thm:bm_exit}
There exists a deterministic constant $\kappa > 0$ so that the following is true.  For $\qsphereinflaw$-a.e.\ instance $(\CS,h,x,y)$ there exists $r_0 > 0$ random such that for every $z \in \CS$ and $r \in (0,r_0)$ we have that
\begin{equation}
\label{eqn:lbm_exit_time}
 r^4 (\log r^{-1})^{-\kappa} \leq E_z[ \tau_{\qball{h}{z}{r}} ] \leq r^4 (\log r^{-1})^{\kappa},
\end{equation}
where the expectation is over just the Brownian motion and the Brownian map instance is fixed.
\end{theorem}

We note that~\eqref{eqn:lbm_exit_time} is equivalent to proving that
\begin{equation}
\label{eqn:lbm_exit_time_green}
r^4 (\log r^{-1})^{-\kappa} \leq \int_{\mathcal{S}} G_{\qball{h}{z}{r}}(z,y) \, d\qmeasure{h}(y) \leq r^4 (\log r^{-1})^{\kappa}
\end{equation}
where $G_{\qball{h}{z}{r}}$ is the Green's function on the ball $\qball{h}{z}{r}$.  To establish~\eqref{eqn:lbm_exit_time} we will in fact establish~\eqref{eqn:lbm_exit_time_green}. 
Also,  we note that \cite[Theorem~1.2]{ms2016qle2} implies that there exists a deterministic constant $\alpha \in (0,1)$ such that $\qsphereinflaw$-a.e.\  there exists a random constant $C\geq 1$ such that for all $u,v \in \CS$ we have that
\begin{equation}
\label{eq:comp_metric}
C^{-1} d(u,v)^{1/\alpha} \leq \qdist{h}{u}{v} \leq C d(u,v)^\alpha,
\end{equation}
where we recall that $d$ denotes the Euclidean metric on $\s^2$ and we assume that $(\mathcal{S},x,y)$ is parameterized by $\s^2$.

The proof of the upper bound in Theorem~\ref{thm:bm_exit} is straightforward and short and given in Subsection~\ref{subsec:exit_time_ubd}.  The proof of the lower bound is much more involved and given in the remainder of this section; see Figure~\ref{fig:exit_time_proof_illustration} for an illustration of the proof.

\subsection{Proof of the exit time upper bound}
\label{subsec:exit_time_ubd}

The upper bound in Theorem~\ref{thm:bm_exit} follows from~\eqref{eq:comp_metric} and the upper volume growth estimate in Theorem~\ref{thm:ball_concentration} by the following argument.

\begin{proof}[Proof of Theorem~\ref{thm:bm_exit}, upper bound]
Let $B=B_h(z,r)$ be fixed. By~\eqref{eq:comp_metric} we have $B \subset B\big(z, c_1 r^\alpha\big)$ for some finite random constant $c_1$.   Therefore,
\begin{align*}
E_{z}\big[ \tau_{B} \big] \; = \; \int_B G_B(z,u) \, d\qmeasure{h}(u)
\; \leq \;
 \int_B G_{B(z,2 c_1 r^\alpha)}(z,u) \,d\qmeasure{h}(u). 
\end{align*}
On the other hand, recall that by~\eqref{eq:killed_green}
\begin{align*}
G_{B(0,  1/2)}(u,v) \; = \; \frac 1 \pi \log \frac{1}{d(u,v)}+ F(u,v)
\end{align*}
for some continuous function $F: B(0,1/2) \times B(0,1/2) \rightarrow \mathbb{R}$. Hence, for all $u\in B$,
\begin{align*}
  G_{B(z,2c_1r^\alpha)}(z,u)
 &\; = \; 
  G_{B(0,2 c_1 r^\alpha)}(0,u-z)
 \;=\;
   G_{B(0,1/2)}\big(0,\tfrac{u-z}{4 c_1r ^\alpha}\big) \\
   &\; = \; 
   \frac 1 \pi \log \frac{4 c_1 r^\alpha}{d(z,u)}+ 
   F\big(0,\tfrac{u-z}{4 c_1 r^\alpha}\big) \\
   &\; \leq \;  
      \frac 1 \pi \log \frac{ c r^\alpha}{ d_h(z,u)^{1/\alpha}}+ \sup_{B(0,1/4 )} F(0,\cdot) 
   \;  = \; 
   \frac  1 \pi \log\!\big( r^\alpha d_h(z,u)^{-1/\alpha} \big) + \wt{c},
\end{align*}
for some random constants $c,\wt{c}$, where we used the second inequality in~\eqref{eq:comp_metric} in the fourth step. For abbreviation we introduce the sets $A_n:= B_h(z,2^{-n}r)\setminus B_h(z,2^{-n-1}r)$, $n\geq 0$. Then the upper estimate on the volume growth in Theorem~\ref{thm:ball_concentration} gives that there exists a deterministic constant $\kappa>0$ such that $\qsphereinflaw$-a.e.\  there exists random $r_0>0$ such that for each $z \in \CS,  r \in (0,r_0)$,  we have that
\begin{align*}
\qmeasure{h}(A_n)  \leq  \big(2^{-n}r \big)^4 \big( \log\big(2^{n}/r \big)\big)^{\kappa}  \leq   2^{-4n}  r^4  \big(n \log 2 + \log(r^{-1}) \big)^{\kappa}.
\end{align*}
By combining the above estimates we obtain that
\begin{align*}
E_{z}\big[ \tau_{B} \big]  & \; \leq \; \frac 1  \pi \int_B \log\!\big( r^\alpha d_h(z,u)^{-1/\alpha} \big) \, d\qmeasure{h}(u) + \wt{c} \, \qmeasure{h}(B) \\
& \; \leq \; 
\frac 1  \pi \sum_{n=0}^\infty \int_{A_n\cap B} \log\!\big( r^\alpha d_h(z,u)^{-1/\alpha} \big) \,d\qmeasure{h}(u) + \wt{c} \, \qmeasure{h}(B)
\\
& \; \leq \; 
\frac 1  \pi \sum_{n=0}^\infty \log\!\big( r^\alpha (2^{-n-1}r)^{-1/\alpha} \big) \, \qmeasure{h}(A_n\cap B) + \wt{c} \,\qmeasure{h}(B) \\
& \; \leq \;
  \frac {1-\alpha^2}{\pi \alpha} \log(r^{-1}) \, \qmeasure{h}(B)  + \wt{c}\, \qmeasure{h}(B)  + \frac {\log 2}{\alpha\pi}\sum_{n=0}^\infty (n+1) \,\qmeasure{h}(A_n) 
  \\
& \; \leq \;  \wh{c} \, \log(r^{-1}) \,  \qmeasure{h}(B) +
\wh{c} \,r^4 \, \log(r^{-1})^\kappa,
\end{align*}
for some random constant $\wh{c}$ and the claim follows from the upper estimate in Theorem~\ref{thm:ball_concentration}.
\end{proof}

\subsection{Definition of the good event}
\label{subsec:event_def}

\begin{figure}
\includegraphics[scale=0.85]{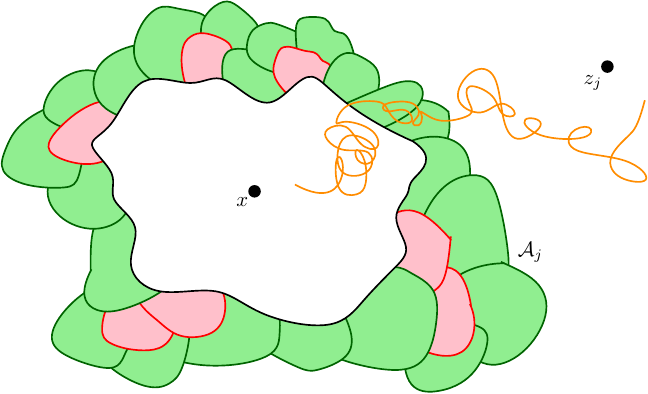}
\caption{\label{fig:exit_time_proof_illustration} Illustration of the main step in the proof of the exit time lower bound from Theorem~\ref{thm:bm_exit}.  We will construct annuli consisting of ``good'' $\SLE_6$ chunks (green) so that whenever a Liouville Brownian motion passes through such an annulus it is likely to take at least a certain amount of time to do so.  The definition of the good chunk (Subsection~\ref{subsec:event_def}) includes a lower bound on the amount of quantum mass that it has disconnected away from its boundary.  We will control the size of the good annuli in Subsection~\ref{subsec:good_annulus_size_bounds}.  The purpose of Subsection~\ref{subsec:sle6_not_too_skinng} is to show that whenever a Liouville Brownian motion passes through such an annulus, it is very likely to enter the interior of such a good chunk and hence take a certain amount of time to pass through.  The proof is completed in Subsection~\ref{subsec:lbd_proof}, where we show that these good annuli are likely to occur all over our quantum sphere.}
\end{figure}

Throughout this section, we will make use of the notation introduced in Subsections~\ref{ssec:SLE6hull-MCPU} and~\ref{ssec:statement-percolation-exploration} some of which we now recall. Suppose that we have a quantum surface $\CD$ which is homeomorphic to $\D$.
For $x,y \in \partial \CD$, we let $[x,y]_{\partial \CD}^\cw$ (resp.\ $[x,y]_{\partial \CD}^\ccw$) denote the clockwise (resp.\ counterclockwise) arc of $\partial \CD$ from $x$ to $y$.

We are now going to make a particular choice of the event $E$ in the context of Proposition~\ref{prop:good_chunks_percolate}, where $E$ is considered as a Borel subset of $\MCPU{2}$.  Fix $\delta \in (0,1)$ and $A\geq A_0$ with $A_0\in[2,\infty)$ as in Proposition~\ref{prop:good_chunks_percolate}. 
Suppose that we consider one of the following two quantum surfaces. Either we let $\CD = (\D,h,0)$ have law $\qdiskweighted{\ell}$ and $\eta'$ be an independent radial $\SLE_6$ starting from a uniformly random point on the boundary and targeted at $0$, or we let
$\CW = (\h,h,0,\infty)$ have law $\qwedge{2}$ and $\eta'$ be an independent chordal $\SLE_6$  on $\h$ from $0$ to $\infty$. Recall the definition of $\sigma_{\delta/A}$ in~\eqref{eq:sigma_delta_A},
and set $\sigma = \delta \wedge \sigma_{\delta/A}$.  Let $\CN=\CN_\sigma^\CD$ or $\CN=\CN_\sigma^\CW$, respectively, be the quantum surface disconnected from $0$ by $\eta'([0,\sigma])$, namely that parameterized by the interior of the hull $K_{\sigma}$ of $\eta'([0,\sigma])$.
Recall that the top (resp.\ bottom) of $\partial \CN$ is given by $\partial \CN \cap \CD$ (resp.\ $\partial \CN \cap \partial \CD$).
The left (resp.\ right) side of the top is the part of the top which is to the left (resp.\ right) of $\eta'(\sigma)$.  Similarly, the left (resp.\ right) side of the bottom is the part of the bottom which is to the left (resp.\ right) of $\eta'(0)$.
Fix $M \geq 1$ and $u, p > 0$. Conditioned on $\sigma = \sigma_{\delta/A}$ (which implies that $\CN$ is simply connected and either the top left or the top right of $\CN$ has zero length),  let $E$ be the event that  the following additional properties hold.
\begin{enumerate}[(I)]
\item\label{it:boundary_length} The length of the top, bottom left, and bottom right of $\CN$ are all at least $\delta^{2/3}/M$.
\item\label{it:diameter_bound} The $d_{h|_\CN}$-diameter of $\CN$ is at most $M \delta^{1/3}$.
\item\label{it:mass_bound} Let $\varphi \colon \CN \to \D$ be the unique conformal transformation which takes the bottom left point to $-1$, the bottom middle point $\eta'(0)$ to $-i$, and the bottom right point to $1$ and consider the embedding of $\CN$ into $\D$ induced by $\varphi$.  Then the quantum mass assigned to $\eball{0}{1/2}$ is at least $\delta^{4/3} / M$.  Also, for each $r \in (0,M^{-1})$, every point with quantum metric distance at least $\delta^{1/3} r$ from~$\partial \D$ (with respect to the field $h|_{\CN} \circ \varphi^{-1} + Q \log |(\varphi^{-1})'|$) has Euclidean distance at least $r^{M}$ from~$\partial \D$.
\item\label{it:reverse_holder_continuity} For every $\epsilon \in (0,\delta^{2/3}/M)$ and points $x,y \in \partial \CN$ such that both $[x,y]_{\partial \CN}^\cw$ and $[x,y]_{\partial \CN}^\ccw$ have boundary length at least $\epsilon$ we have that their $\qdistnoarg{h}$-distance in $\CN$ is at least $\epsilon^{M}$.  If either $[x,y]_{\partial \CN}^\cw$ or $[x,y]_{\partial \CN}^\ccw$ has boundary length at most $\epsilon$ then their $\qdistnoarg{h}$-distance in $\CN$ is at most $\epsilon^{1/M}$.
\item\label{it:boundary_volume_lbd} For every $\epsilon \in (0,\delta^{1/3}/M)$ and $x,y \in \partial \CN$, the $\epsilon$-neighborhood (with respect to $\qdistnoarg{h}$ in $\CN$) of $[x,y]_{\partial \CN}^\ccw$ has quantum mass at least $\epsilon^{2+u}/M$ times the length of $[x,y]_{\partial \CN}^\ccw$.  The same is also true with $[x,y]_{\partial \CN}^\cw$ in place of $[x,y]_{\partial \CN}^\ccw$.
\item\label{it:ngbd_vollume_ubd} For every $\epsilon \in (0,\delta^{1/3} / M)$,  the quantum area of the $\epsilon$-neighborhood of $\partial \CN$ is at most $\epsilon^p \delta^p$.
\end{enumerate}

\begin{proposition}
\label{prop:good_sle_chunks}
Let $A_0\in[2,\infty)$ and $c_{0,\mathrm{max}}\in (0,\infty)$ be as in Proposition~\ref{prop:good_chunks_percolate}. For each $A\geq A_0$ and $c_0\in (0, c_{0,\mathrm{max}}]$ there exist $M_0\in [1,\infty)$ and $p_0 > 0$, depending only $A$, $c_0$ and $u$, such that for all $M\geq M_0$ and $p \in (0,p_0)$, $\qwedge{2}\bigl[ E^{\CW}_{\sigma} \cap \{ \CN \in E\} \bigm| \sigma < \delta \bigr] \geq 1 - c_0 A^{-2/3}$.
\end{proposition}

In order to start to prove Proposition~\ref{prop:good_sle_chunks}, we first need to recall the following lower bound regarding the amount of mass near the boundary for a quantum disk \cite[Lemma~3.4]{gm2019gluing}.

\begin{lemma}
\label{lem:boundary_area_lbd}
Suppose that $\CD = (\D,h,0)$ has law $\qdiskweighted{\ell}$.  For each $u > 0$ there a.s.\ exists $c > 0$ such that for each $\epsilon \in (0,1)$ and $x,y \in \partial \D$ the LQG area of the $\epsilon$-neighborhood of $[x,y]_{\partial \CD}^\ccw$ is at least $c \epsilon^{2+u}$ times the length of $[x,y]_{\partial \CD}^\ccw$.  The same holds with $[x,y]_{\partial \CD}^\cw$ in place of $[x,y]_{\partial \CD}^\ccw$.	
\end{lemma}

We also need the following upper bound for distances between points on the boundary of a quantum disk established in \cite[Lemma~3.2]{gm2019gluing}.
\begin{lemma}
\label{lem:boundary_length_regularity}
Fix $\ell > 0$ and let $\CD = (\D,h)$ have law $\qdisk{\ell}$ or $\qdiskweighted{\ell}$.  For each $\zeta > 0$ there a.s.\ exists $C > 0$ so that for all $x,y \in \partial \D$ we have that
\[ \qdist{h}{x}{y} \leq C \qbmeasure{h}([x,y]_{\partial \CD}^\ccw)^{1/2}\left( | \log \qbmeasure{h}([x,y]_{\partial \CD}^\ccw)| + 1\right)^{7/4+\zeta}\]
and the same is true with $[x,y]_{\partial \CD}^\cw$ in place of $[x,y]_{\partial \CD}^\ccw$.
If we let $C$ be the smallest constant for which this is satisfied, then for $A > 1$ we have that $\p[ C > A]$ decays to $0$ as $A \to \infty$ faster than any negative power of $A$. 
\end{lemma}

\begin{lemma}\label{lem:qbmeasure_holder_cont}
There exists a deterministic constant $\wt{\beta}>0$ such that the following is true.  Suppose that $\CD = (\CD,h,0)$ has law $\qdiskweighted{1}$.  Then,  a.s.\  under $\qdiskweighted{1}$,  the quantum boundary length measure $\nu_h$ is $\wt{\beta}$-H\"older continuous with respect to the Euclidean metric.
\end{lemma}

\begin{proof}
First,  we note that if $\wt{h}$ has the law of a free boundary GFF on $\D$ with some fixed normalization,  then \cite[Proposition~3.7]{robert2008gaussianmultiplicativechaosrevisited} implies that for every $p \in (1,p_{*})$,  there exists $C_p < \infty$ such that $\E\big[ \nu_{\wt{h}}(A)^p \big] \leq C_p \diam(A)^p$ for each $A \subseteq \partial \CD$ Borel,  where $\zeta_p = (2+\frac{\gamma^2}{4})p - \frac{\gamma^2 p^2}{4}$,  $p_{*}$ is the unique $p_{*}>1$ such that $\zeta_{p_{*}}=2$,  and $\diam$ denotes Euclidean diameter.  Hence,  by choosing $p \in (1,p_{*})$ such that $\zeta_p>1$ and combining with Kolmogorov's criterion,  we obtain that there exists $\wt{\beta}>0$ deterministic such that $\nu_{\wt{h}}$ is a.s.\  $\wt{\beta}$-H\"older continuous with respect to the Euclidean metric.

Now,  we recall some results from \cite{ang2022integrabilitysleconformalwelding}.  Suppose that $f$ is sampled from the group $\text{conf}(\h)$ of conformal automorphisms of $\h$ when the latter is endowed with the Haar measure,  and let $h$ be sampled from the infinite measure of a weight-$2$ quantum disk with $\gamma = \sqrt{8/3}$ (see \cite[Section~4.5]{dms2014mating}) weighted by $\nu_{h}(\partial \h)^{-2}$.  Then,  \cite[Theorem~1.2]{ang2022integrabilitysleconformalwelding} implies that there exists a constant $C$ such that the law of $h \circ f^{-1} + Q \log |(f^{-1})'|$ is given by $C$ times the law of $\wh{h}:= \wt{h} - 2Q\log |\cdot|_{+} +c$,  where $\wt{h}$ is a free boundary GFF on $\h$ and $c$ is sampled from the infinite measure on $\R$ given by $\exp(-Q c) dc$.  Consider the conformal transformation $\psi : \h \to \D$ given by $\psi(z) = \frac{z-i}{z+i}$.  Note that $\wt{h} \circ \psi^{-1}$ is a free boundary GFF on $\D$ with some fixed normalization and $-2Q\log \max\{|\psi^{-1}(\cdot)|,1\} + Q \log |(\psi^{-1})'(\cdot)| + c = O(1)$ uniformly in $\partial \D$.  Hence,  combining with Kahane's convexity inequality,  we obtain that the quantum boundary length of $\wh{h} \circ \psi^{-1} + Q \log |(\psi^{-1})'|$ is a.e.\  $\wt{\beta}$-H\"older continuous on $\partial \D$ with respect to the Euclidean metric.  Therefore,  the same is a.e.\  true for $h \circ f^{-1} + Q \log |(f^{-1})'|$ and since the event that the quantum boundary length is $\wt{\beta}$-H\"older continuous with respect to the Euclidean metric is invariant under the coordinate change formula for quantum surfaces,  we obtain that a sample from the infinite measure of a weigh-$2$ quantum disk satisfies the above property a.e. Thus, combining with disintegration with respect to the total boundary length (see \cite[Section~4.5]{dms2014mating}),  we obtain that if $(\D,h)$ is sampled from $\qdisk{1}$,  then its boundary length is a.s.\  $\wt{\beta}$-H\"older continuous with respect to the Euclidean metric.  Therefore,  the same is true for a sample $(\D,h,0)$ from $\qdiskweighted{1}$ by absolute continuity.
\end{proof}

\begin{lemma}
\label{lem:disk_boundary_reverse_holder}
There exists a constant $\beta > 0$ so that the following is true.  Suppose that $\ell > 0$ and $\CD = (\D,h,0)$ has law $\qdiskweighted{\ell}$.  There is a.s.\ $\epsilon_0 > 0$ so that for all $\epsilon \in (0,\epsilon_0)$ and $a,b \in \partial \CD$ with both $\qbmeasure{h}([a,b]_{\partial \CD}^\cw) \geq \epsilon$ and $\qbmeasure{h}([a,b]_{\partial \CD}^\ccw) \geq \epsilon$ we have that $\qdist{h}{a}{b} \geq \epsilon^\beta$.	
\end{lemma}

\begin{proof}
First,  we note that a sample from $\qdiskweighted{\ell}$ can be obtained by starting with a sample from $\qdiskweighted{1}$ and then multiplying lengths with $\ell$,  distances with $\ell^{1/2}$ and areas by $\ell^2$.  Thus,  it suffices to prove the claim in the case that $\ell = 1$.  Suppose that $(\CS,x,y)$ has distribution $\qsphereinflaw$ which we can assume that it is parameterized by the Euclidean sphere $\s^2$.  Recall that the metric $\qdistnoarg{h}$ on $\s^2$ is H\"older continuous with respect to the Euclidean metric $d$ on $\s^2$,  i.e.,  there exists $\alpha \in (0,1)$ deterministic and $C\geq 1$ random such that~\eqref{eq:comp_metric} holds for each $u,v \in \s^2$.  Let $\tau$ be the smallest $r>0$ such that the boundary length of $\qfb{h}{y}{x}{r}$ is equal to $1$.  Conditional on $\tau<\infty$,  the quantum surface $\mathcal{D}$ parameterized by $\CS \setminus \qfb{h}{y}{x}{r}$ and marked by $y$ has law $\qdiskweighted{1}$.  Now,  suppose that we conformally map $\CD$ to $\D$ with $y$ sent to $0$,  and consider the surface parameterized by $\D$.  In particular,  we consider the conformal transformation $\phi$ mapping $\CD$ onto $\D$ such that $\phi(y) = 0$ and $\phi'(y) > 0$,  and set $\wt{h}:= h \circ \phi^{-1} + Q \log |(\phi^{-1})'|$.  Note that there a.s.\  exists $A>0$ such that $\dist_d (y,\partial \CD) \geq A$.  Fix $p>1$ sufficiently large and deterministic (to be chosen) and let $a,b  \in \partial \D$ such that both $\qbmeasure{\wt{h}}([a,b]_{\partial \D}^\cw) \geq \epsilon$ and $\qbmeasure{\wt{h}}([a,b]_{\partial \D}^\ccw) \geq \epsilon$.  Then,  we have that both $\qbmeasure{h}([\wt{a},\wt{b}]_{\partial \CD}^\cw) \geq \epsilon$ and $\qbmeasure{h}([\wt{a},\wt{}b]_{\partial \CD}^\ccw) \geq \epsilon$ with $\wt{a} = \phi^{-1}(a)$ and $\wt{b} = \phi^{-1}(b)$.  Suppose that $\qdist{h}{\wt{a}}{\wt{b}} \leq \qdist{\wt{h}}{a}{b} \leq \epsilon^p$.  Then,  we have that $d(\wt{a},\wt{b}) \leq C^{\alpha} \epsilon^{\alpha p} < A$ for $\epsilon > 0$ sufficiently small.  Moreover,  since $\qdistnoarg{\wt{h}}$ is a.s.\  equivalent to the Euclidean metric on $\D$,  we obtain that it is a.s.\  the case that $B_{\qdistnoarg{\wt{h}}}(z,\epsilon) \subseteq \overline{\D} \cap B_d(z,1/2)$ for all $z \in \partial \D$ and all $\epsilon>0$ sufficiently small.  In particular,  for $\epsilon>0$ sufficiently small,  we have that the geodesic $\gamma$ with respect to $\qdistnoarg{\wt{h}}$ from $a$ to $b$ disconnects from $0$ either $[a,b]_{\partial \D}^\cw$ or $[a,b]_{\partial \D}^\ccw$,  and so $\phi^{-1}(\gamma)$ disconnects from $y$ either $[\wt{a},\wt{b}]_{\partial \CD}^\cw$ or $[\wt{a},\wt{b}]_{\partial \CD}^\ccw$.  Note that $\qdist{h}{\wt{a}}{z} \leq \epsilon^p$ for all $z \in \phi^{-1}(\gamma)$,  since $\phi^{-1}(\gamma)$ is the geodesic in $\CD$ connecting $\wt{a}$ to $\wt{b}$ with respect to the interior-internal metric $\qdistnoarg{h}|_{\CD}$.  It follows that $B_d(\wt{a},C^{\alpha} \epsilon^{\alpha p})$ disconnects from $y$ either $[\wt{a},\wt{b}]_{\partial \CD}^\cw$ or $[\wt{a},\wt{b}]_{\partial \CD}^\ccw$ for all $\epsilon>0$ sufficiently small.  We can assume that the latter holds.  Then,  the Beurling estimate implies that for all $\epsilon>0$ sufficiently small,  the probability that a Brownian motion starting from $y$ exits $\CD$ in $[\wt{a},\wt{b}]_{\partial \CD}^\ccw$ is $\lesssim\epsilon^{\alpha p /2}$,  where the implicit constant depends only on $C^{\alpha}$ and $A$.  This implies that $[\wt{a},\wt{b}]_{\partial \CD}^\ccw$ gets mapped to an arc in $\partial \D$ with Euclidean length at most $O(\epsilon^{\alpha p /2})$.  Finally,  to complete the proof,  we let $\wt{\beta}$ be the constant of Lemma~\ref{lem:qbmeasure_holder_cont} and choose $p>1$ such that $\alpha p \wt{\beta} /2 >1$.  Then,  the $\qdiskweighted{1}$-a.e.\  $\wt{\beta}$-H\"older continuity of $\nu_{\wt{h}}$ with respect to $d$ shown in Lemma~\ref{lem:qbmeasure_holder_cont} implies that the quantum length of $[a,b]_{\partial \D}^\ccw$ is at most $O(\epsilon^{\alpha p \wt{\beta} /2})$ for all $\epsilon>0$ sufficiently small.  But that is a contradiction since $\nu_{\wt{h}}([a,b]_{\partial \D}^\ccw) \geq \epsilon$.  Hence,  $\qdist{\wt{h}}{a}{b} \geq \epsilon^p$ and this completes the proof.
\end{proof}

\begin{lemma}\label{lem:metric_ngbd_quantum_area_ubd}
Fix $\ell>0, u\in (0,2)$ and let $\CD = (\CD,h,0)$ be a sample from $\qdiskweighted{\ell}$.  Then,  $\qdiskweighted{\ell}$-a.e., there exist random constants $C>0,\delta_0 \in (0,1)$ such that the quantum area of the $\delta$-neighborhood of $\partial \D$ with respect to $\qdistnoarg{h}$ is at most $C \delta^{2-u}$ for all $\delta \in (0,\delta_0)$.
\end{lemma}

\begin{proof}
We assume that we have the setup of the proof of Lemma~\ref{lem:boundary_length_regularity} given in \cite{gm2019gluing}.  In particular,  we let $p:[0,1] \to \CD$ be the quotient map introduced in \cite[Section~3.1]{gm2019gluing}.  Then,  we know from \cite[Section~3.1]{gm2019gluing} that $\partial \CD = p(\{T_r : r \in [0,\ell]\})$,  where $T_r = \inf\{t \geq 0 : B_t = -r\}$ and $B$ is a standard Brownian motion coupled with $\CD$.  Fix $\wt{u} \in (0,1/2)$ sufficiently small (to be chosen).  Then,  Lemma~\ref{lem:boundary_length_regularity} implies that there a.s.\  exists a constant $C>0$ such that $\qdist{h}{x}{y} \leq C \qbmeasure{h}([x,y]_{\partial \CD}^\cw)^{\frac{1}{2}-\wt{u}}$ for all $x,y \in \partial \CD$.  Moreover,  by possibly taking $C$ to be larger,  we can assume using \cite[Lemma~3.3]{gm2019gluing} that $\mu_h(B_{\qdistnoarg{h}}(z,\delta)) \leq C \delta^{4-\wt{u}}$ for all $z \in \partial \CD$,  $\delta \in (0,1)$.  Let $\delta_0 \in (0,1/2)$ be sufficiently small such that $\wt{\delta} \in (0,1/2)$ where $\wt{\delta}>0$ is such that $C \wt{\delta}^{1/2 - \wt{u}} = \delta$.  Then,  we have that $p(\{T_r : r \in [(k-1)\wt{\delta},k\wt{\delta}]\}) \subseteq B_{\qdistnoarg{h}}(T((k-1)\wt{\delta}),\delta)$ for all $1 \leq k \leq \wt{\delta}^{-1}$
.  It follows that the $\delta$-neighborhood of $\partial \CD$ with respect to $\qdistnoarg{h}$ is contained in $\cup_{k=1}^{\wt{\delta}^{-1}} B_{\qdistnoarg{h}}(p(T((k-1)\wt{\delta})),2\delta)$.  Therefore,  a union bound implies that the quantum area of the $\delta$-neighborhood of $\partial \CD$ with respect to $\qdistnoarg{h}$ is at most $\lesssim \delta^{4-\wt{u} - \frac{2}{1-2\wt{u}}}$,  and so the proof of the lemma is complete by choosing $\wt{u}$ such that $4-\wt{u}-\frac{2}{1-2\wt{u}} > 2-u$.
\end{proof}

\begin{proposition}\label{prop:good_quantum_disk}
Fix $\epsilon \in (0,1)$ and let $\CD = (\D,h)$ be sampled from $\qdiskweighted{1}$.  Define the event $E(h)$ that conditions~\eqref{it:diameter_bound}-\eqref{it:ngbd_vollume_ubd} hold for the quantum surface $\CD$ instead of $\CN$ and with $\delta = 1$.  Then,  we have that $\qdiskweighted{1}\big[ E(h) \big] \geq 1-\epsilon$ provided $M$ is large enough and $p$ is small enough.
\end{proposition}

\begin{proof}
Recall that a sample from $\qdiskweighted{1}$ can be produced as follows.  Let $(\CS,h,x,y)$ be a sample from $\qsphereinflaw$ and let $\eta'$ be an independent whole-plane $\SLE_6$ in $\CS$ from $x$ to $y$ parameterized by quantum natural time.  Let also $L$ be the process describing the boundary length evolution of the connected component $U_t$ of $\CS \setminus \eta'([0,t])$ containing $y$.  We let $\tau$ be the first time $t$ that $L_t = 1$.  Then,  conditional on $\tau<\infty$,  we let $\CD$ be the surface obtained by conformally mapping $U_{\tau}$ onto $\CD$ using the conformal map $\psi: U_{\tau} \to \CD$ such that $\psi(y) = 0$ and $\psi'(y) > 0$ and then applying the coordinate change formula for quantum surfaces.  Since $\CD$ is a metric space of finite diameter a.s.,  we can arrange so that part~\eqref{it:diameter_bound} holds with probability as close to $1$ as we want by taking $M$ sufficiently large.  Moreover,  since $\mu_h(B(0,1/2)) > 0$ $\qdiskweighted{1}$-a.s.\  and there exists a deterministic constant $\alpha \in (0,1)$ such that the metric in $\CS$ is $\alpha$-H\"older continuous with respect to the Euclidean metric and $\psi$ is H\"older continuous with some fixed and deterministic exponent (see \cite[Theorem~5.2]{rohde2011basic}),  we obtain that part~\eqref{it:mass_bound} holds with probability as close to $1$ as we want by taking $M$ sufficiently large.  Also,  it follows by combining Lemmas~\ref{lem:boundary_length_regularity} and~\ref{lem:disk_boundary_reverse_holder} that part~\eqref{it:reverse_holder_continuity} holds with arbitrarily high probability provided we choose $M$ sufficiently large.  Furthermore,  by arguing in the same way but using Lemma~\ref{lem:boundary_area_lbd} instead,  we obtain that part~\eqref{it:boundary_volume_lbd} holds with arbitrarily high probability if we choose $M$ sufficiently large.  Finally,  Lemma~\ref{lem:metric_ngbd_quantum_area_ubd} implies that part~\eqref{it:ngbd_vollume_ubd} holds with arbitrarily high probability as well if we choose $M$ large enough and $p$ small enough. This completes the proof.
\end{proof}

Next,  we focus on proving Proposition~~\ref{prop:good_sle_chunks}.  First,  we will prove that condition~\eqref{it:reverse_holder_continuity} holds with high probability provided $M$ is sufficiently large.  We begin by proving that the lower bound on quantum distances in condition~\eqref{it:reverse_holder_continuity} holds with high probability.  Since the proof of the lower bound will be technical,  we will give its proof in the next two lemmas.  First,  we will deal with the case that both of the boundary points $x$ and $y$ lie on $\h \cap \partial \CN$.  This is the content of the following lemma.

\begin{figure}
\includegraphics[scale=0.85,page=1]{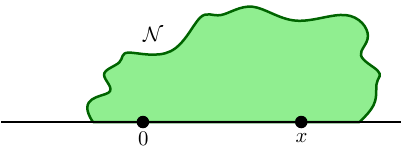} \hspace{0.025\textwidth} \includegraphics[scale=0.85,page=2]{figures/chunk_retargeting.pdf}

\vspace{0.025\textheight}

\includegraphics[scale=0.85,page=3]{figures/chunk_retargeting.pdf}
\caption{Illustration of the proof of Lemma~\ref{lem:reverse_holder_continuity_top_boundary}.  {\bf Top left}: The chunk $\CN$.  {\bf Top right:} An $\SLE_6$ process $\wt{\eta}'$ in $\h$ from $0$ to $x$ coupled to agree with $\eta'$ up until they first separate $x$ from $\infty$, shown on the event that the hull that it separates from $\infty$ parameterizes $\CN$.  {\bf Bottom:} The right (standing from $x$ looking towards $0$) boundary of $\wt{\eta}'$ is $\wt{\eta}$, which we view as a curve from $x$ to $0$.}
\end{figure}

\begin{lemma}\label{lem:reverse_holder_continuity_top_boundary}
Let $\mathcal{W} = (\h,h,0,\infty)$ have law $\qwedge{2}$ and let $\eta'$ be a chordal $\SLE_6$ in $\h$ from $0$ to $\infty$ which is independent of $\mathcal{W}$.  Fix $\ell>0$ and let $x>0$ be such that $\qbmeasure{h}([0,x]) = \ell$.  Let also $\tau$ be the first time that $\eta'$ disconnects $x$ from $\infty$.  Moreover,  let $(K_t)$ denote the hulls of $\eta'$ and let $\CN$ be the quantum surface parameterized by the interior of $K_{\tau}$.  Then,  there a.s.\  exist $\epsilon_0 \in (0,1)$ and $M \in (1,\infty)$ such that for all $\epsilon \in (0,\epsilon_0)$ the following holds.  Let $a,b  \in  \h \cap \partial \CN$ be such that $\qbmeasure{h}([a,b]_{\partial \CN}^\cw) \geq \epsilon$ and $\qbmeasure{h}([a,b]_{\partial \CN}^\ccw) \geq \epsilon$.  Then,  we have that $\qdistnoarg{h|_{\CN}}(a,b) \geq \epsilon^M$,  where $\qdistnoarg{h|_{\CN}}$ denotes the interior-internal metric on $\CN$ with respect to $h$. 
\end{lemma}

\begin{proof}
\emph{Step 1.  Overview and setup.} First we note that the locality property of $\SLE_6$ implies that we can couple $\eta'$ with an $\SLE_6$ $\wt{\eta}'$ in $\h$ from $0$ to $x$ such that $\eta'$ and $\wt{\eta}'$ agree up until the first time that they disconnect $x$ from $\infty$.  From now on,  we assume that we are working with this coupling.  Let $\wt{\eta}$ be the left outer boundary of $\wt{\eta}'$ and for $\rho \geq 0$,  we let $\wt{\tau}_{\rho}$ be the first time $t$ that $\dist(\wt{\eta}(t), \R_-) \leq \rho$.  Note that $\wt{\tau}_{\rho} < \infty$ a.s.\  for all $\rho \geq 0$.  Let also $\wt{D}_{\rho}$ be the unbounded connected component of $\h \setminus \wt{\eta}([0,\wt{\tau}_{\rho}])$.  Moreover,  we fix $0<\rho_2<\rho_1<x$ and let $\wt{I}$ be the counterclockwise segment traced by $\wt{\eta}$ between times $\wt{\sigma}$ and $\wt{\tau}_{\rho_1}$,  where $\wt{\sigma}$ is the last time before time $\wt{\tau}_{\rho_1}$ that $\wt{\eta}$ intersects $[x,\infty)$ (when viewed as a set of prime ends on the left side of $\wt{\eta}$).  We also set $\wt{J} = \wt{\eta}([\wt{\sigma},\wt{\tau}_{\rho_2}])$.  The proof of the lemma consists of three steps.  In Step 2,  we will prove that the lower bound on quantum distances in condition~\eqref{it:reverse_holder_continuity} holds for a quantum wedge $\wh{\mathcal{W}}$ parameterized by $\h$ of weight $\frac{\gamma^2}{2} = \frac{4}{3}$  for points on compact intervals on $\R$ which are bounded away from $0$.  Since the law of the field $h$ restricted to a small neighborhood of $\wt{I}$ in $\CN$ is absolutely continuous with  a quantum wedge of weight $\frac{4}{3}$,  we will deduce in Step 3 the claim of the lemma for $a,b  \in \wt{I} \cap \partial \CN$ and the field $h|_{\h \setminus \wt{J}}$ in place of $h|_{\CN}$.  Finally,  we will complete the proof in Step 4 using the time-reversal invariance of the law of $\wt{\eta}$.

\emph{Step 2.  Proof of the claim for a quantum wedge of weight $\frac{\gamma^2}{2} = \frac{4}{3}$.} Let $\wt{h}$ be a free boundary GFF on $\h$ with the additive constant taken so that its average on $\h \cap \partial \D$ is equal to $0$.  Then Proposition~\ref{prop:good_quantum_disk} combined with the argument in Lemma~\ref{lem:qbmeasure_holder_cont} imply that for every fixed $-\infty < a < b < \infty$,  there a.s.\  exist $\epsilon_0 \in (0,1),  M \in (1,\infty)$ such that the following is true.  For all $\epsilon \in (0,\epsilon_0)$ and all $z,w \in [a,b]$ such that $\qbmeasure{\wt{h}}([z,w]) \geq \epsilon$,  we have that $\qdistnoarg{\wt{h}}(z,w) \geq \epsilon^M$,  where $\qdistnoarg{\wt{h}}$ denotes the interior-internal metric on $\h$ with respect to $\wt{h}$.

Fix $0<a<b$.  Then the above implies that there a.s.\  exists $\epsilon_0 \in (0,1)$ such that for all $\epsilon \in (0,\epsilon_0)$ and all $a \leq z < w \leq b$ such that $\qbmeasure{\wt{h}}([z,w]) \geq \epsilon$,  we have that $\qdistnoarg{\wt{h}|_{A_{r,R}}}(z,w) \geq \epsilon^M$,  for all $0<r<a<b<R$,  where $A_{r,R}:= B(0,R) \setminus \overline{B(0,r)}$.  Note that if we fix such $r,R$,  the above event is determined by $\wt{h}|_{A_{r,R}}$.  Suppose that $\wh{\mathcal{W}} = (\h,\wh{h},0,\infty)$ has the law of a weight $\frac{\gamma^2}{2} = \frac{4}{3}$ quantum wedge with the circle average embedding.  Then,  for fixed $0<r<a<b<R$,  the laws of $\wh{h}|_{A_{r,R}}$ and $\wt{h}|_{A_{r,R}}$ are mutually absolutely continuous and so it is a.s.\  the case that there exist $\epsilon_0 \in (0,1),  M \in (1,\infty)$ such that for all $\epsilon \in (0,\epsilon_0)$,  if $z,w \in [a,b]$ are such that $\qbmeasure{\wh{h}}([z,w]) \geq \epsilon$,  we have that $\qdistnoarg{\wh{h}|_{A_{r,R}}}(z,w) \geq \epsilon^M$.  Note that $\qdistnoarg{\wh{h}}$ a.s.\  induces the Euclidean topology (see \cite[Theorem~1.3]{hughes2024equivalencemetricgluingconformal}).  Moreover,  the function $t \to \qbmeasure{\wh{h}}([0,t])$ is a homeomorphism on $\R_+$ with respect to the Euclidean topology (and hence with respect to the topology induced by $\qdistnoarg{\wh{h}}$ on $\R_+$),  which implies that it is a.s.\  the case that there exists $\delta_0 \in (0,1)$ such that $\qdistnoarg{\wh{h}}(z,w) \geq \delta_0$ for all $z,w \in [a,b]$ such that $\qbmeasure{\wh{h}}([z,w]) \geq \epsilon_0$.  Also,  since $\dist_{\qdistnoarg{\wh{h}}}([a,b],\partial A_{r,R}) > 0$ a.s.,  possibly by taking $\epsilon_0$ to be smaller,  we can assume that $\qdistnoarg{\wh{h}}(z,w) = \qdistnoarg{\wh{h}|_{A_{r,R}}}(z,w)$ for all $z,w \in [a,b]$ such that $\qbmeasure{\wh{h}}([z,w]) \leq \epsilon_0$.  Combining,  we obtain that possibly by taking $M$ to larger,  we have that $\qdistnoarg{\wh{h}}(z,w) \geq \epsilon^M$ for all $a\leq z < w \leq b$ such that $\qbmeasure{\wh{h}}([z,w]) \geq \epsilon$,  and all $\epsilon \in (0,\epsilon_0)$.  The same holds for any $-\infty < a < b < 0$ by symmetry.

\emph{Step 3.  Proof of the claim for the field $h|_{\h \setminus \wt{J}}$.} Next,  we note that It follows from \cite[Section~2]{mw2017intSLE} that $\wt{\eta}$ has the law of an $\SLE_{\kappa}(\frac{\kappa}{2}-2 ;  \kappa-4)$ process in $\h$ from $x$ to $0$,  where $\kappa = \frac{8}{3}$ and the force points are located at $x^-$ and $x^+$ respectively.  Let also $\eta$ be an $\SLE_{\kappa}(\frac{\kappa}{2}-2;\kappa-4)$ process in $\h$ from $x$ to $\infty$ with the force points located at $x^-$ and $x^+$ respectively,  which is independent of $\mathcal{W}$.  Then,  combining \cite[Proposition~1.7]{sheffield2015conformalweldingsrandomsurfaces} with \cite[Theorem~1.2]{dms2014mating},  we obtain the following description of the quantum surface parameterized by the region to the left of $\eta$.  Let $\phi$ be the conformal transformation mapping the connected component of $\h \setminus \eta$ lying to the left of $\eta$ onto $\h$ such that $\phi(x) = 0,  \phi(\infty) = \infty$ and $\phi(0) = -1$.  Then,  if we parameterize the field $\check{h} := h \circ \phi^{-1} + Q \log |(\phi^{-1})'|$ by the circle average embedding,  it has the same law with the field obtained when we parameterize $\wh{\mathcal{W}}$ by the circle average embedding.  For $\rho \geq 0$,  we let $\tau_{\rho}$ be the first time $t$ that $\dist(\eta(t) ,  \R_-) \leq \rho$.  Suppose that we are working on the event that $\tau_{\rho_2} < \infty$,  where $0<\rho_2 < \rho_1 <x$ are small but fixed.  Let $I$ be the counterclockwise segment traced by $\eta$ between the last time before $\tau_{\rho_1}$ that it intersects $[x,\infty)$ and time $\tau_{\rho_1}$ (seen as set of prime ends on the left side of $\eta$).  Then,  if we apply the results of Step 2 for $\wh{\mathcal{W}}$ on the time-interval $\phi(I)$,  we obtain that there a.s.\  exist $\epsilon_0 \in (0,1),  M \in (1,\infty)$ such that for all $\epsilon \in (0,\epsilon_0)$,  the following is true.  Let $a,b \in I$ be such that $\eta$ hits $a$ before it hits $b$ and $\qbmeasure{h}([a,b]) \geq \epsilon$.  Then,  it holds that $\qdistnoarg{\check{h}}(\phi(a),\phi(b)) \geq \epsilon^M$.  Suppose that $\epsilon_0 \in (0,1)$ is chosen such that $\epsilon_0 < \dist_{\qdistnoarg{\check{h}}}(\phi(I) ,  \phi(\eta([\tau_{\rho_2},\infty))))$.  We claim that if we further assume that $\qbmeasure{h}([a,b]) \leq \epsilon_0$,  we have that $\qdistnoarg{h|_{D_{\rho_2}}}(a,b) = \qdistnoarg{h|_{\phi^{-1}(\h)}}(a,b)$,  where $D_{\rho}$ denotes the unbounded connected component of $\h \setminus \eta([0,\tau_{\rho}])$.  Indeed,  clearly we have that $\qdistnoarg{h|_{D_{\rho_2}}}(a,b) \leq \qdistnoarg{h|_{\phi^{-1}(\h)}}(a,b)$.  Let $\gamma$ be a $\qdistnoarg{h|_{D_{\rho_2}}}$-geodesic path in $D_{\rho_2}$ from $a$ to $b$.  Suppose that $\gamma$ intersects $\eta((\tau_{\rho_2},\infty))$,  and let $t$ be the last time that $\gamma$ hits $\eta((\tau_{\rho_2},\infty))$,  when $\gamma$ is parameterized by $[0,1]$ and $\gamma(0) = a,  \gamma(1) = b$.  Then,  we have that the $\qdistnoarg{h}$-length of $\gamma|_{[t,1]}$ is at least $\dist_{\qdistnoarg{\check{h}}}(\phi(I),\phi(\eta([\tau_{\rho_2},\infty)))) \geq \epsilon_0$ and so $\qdistnoarg{h|_{D_{\rho_2}}}(a,b) \geq \dist_{\qdistnoarg{\check{h}}}(\phi(I),\phi(\eta([\tau_{\rho_2},\infty)))) \geq \epsilon_0$,  but that is a contradiction.  Thus,  it follows that $\qdistnoarg{h|_{D_{\rho_2}}}(a,b) = \qdistnoarg{h|_{\phi^{-1}(\h)}}(a,b)$.  Combining everything,  we obtain that the following is true a.s.\  on the event that $\tau_{\rho_2} < \infty$.  There exists $\epsilon_0 \in (0,1),  M \in (1,\infty)$ such that for all $\epsilon \in (0,\epsilon_0)$,  the following holds.  Let $a,b \in I$ be such that $\eta$ hits $a$ before $b$ and $\epsilon \leq \qbmeasure{h}([a,b]_{\partial \CN}^\ccw) \leq \epsilon_0$.  Then $\qdistnoarg{h|_{D_{\rho_2}}}(a,b) \geq \epsilon^M$.  Also,  it follows from \cite[Theorem~6]{SchrammWilson2005} that the law of $\wt{\eta}|_{[0,\wt{\tau}_{\rho_2}]}$ is absolutely continuous with respect to the law of $\eta|_{[0,\tau_{\rho_2}]}$ when the latter is restricted to the event that $\tau_{\rho_2} < \infty$.  It follows that there a.s.\  exist $\epsilon_0 \in (0,1),  M \in (1,\infty)$ such that for all $\epsilon \in (0,\epsilon_0)$,  the following is true.  Let $a,b \in \wt{I}$ be such that $a$ is hit before $b$ by $\wt{\eta}$ and $\epsilon \leq \qbmeasure{h}([a,b]) \leq \epsilon_0$.  Then,  it holds that $\qdistnoarg{h|_{\wt{D}_{\rho_2}}}(a,b) \geq \epsilon^M$. We claim that $\qdistnoarg{h|_{\h \setminus \wt{J}}}(a,b) \geq \epsilon^M$ for such points as well,  possibly by taking $\epsilon_0$ to be smaller.  Indeed,  let $\gamma$ be a path in $\h \setminus \wt{J}$ connecting $a$ to $b$ with $\gamma : (0,1) \to \h \setminus \wt{J}$ and $\gamma(0) = a,  \gamma(1) = b$.  If $\gamma$ is contained in $\wt{D}_{\rho_2}$,  then clearly the $\qdistnoarg{h}$-length of $\gamma$ is at least $\epsilon^M$ since $\qdistnoarg{h|_{\wt{D}_{\rho_2}}}(a,b) \geq \epsilon^M$.  Suppose that $\gamma$ intersects $\partial \wt{D}_{\rho_2} \setminus \partial (\h \setminus \wt{J})$.  Then we have that $\gamma$ intersects $\wt{\eta}([0,\wt{\sigma}])$ and let $t$ be the last time that this occurs.  Then we have that $\gamma|_{(t,1)}$ is a path in $\wt{D}_{\rho_2}$ connecting $\gamma(t)$ to $b$ and the boundary length of the counterclockwise arc of $\partial \wt{D}_{\rho_2}$ from $\gamma(t)$ to $b$ is at least $\epsilon$.  Hence,  the $\qdistnoarg{h}$-length of $\gamma|_{(t,1)}$ is at least $\epsilon^M$ and so the $\qdistnoarg{h}$-length of $\gamma$ is at least $\epsilon^M$ in any case.  It follows that $\qdistnoarg{h|_{\h \setminus \wt{J}}}(a,b) \geq \epsilon^M$.

\emph{Step 4.  Conclusion of the proof.} Now,  let $\wh{\eta}$ be the time-reversal of $\wt{\eta}$ and note that  $\wh{\eta}$ has the law of an $\SLE_{\kappa}(\kappa-4;\frac{\kappa}{2}-2)$ process in $\h$ from $0$ to $x$,  where the force points are located at $0^-$ and $0^+$ respectively \cite{ms2016ig2}.  Similarly, for $\rho \in [0,1)$,  we set $\wh{\tau}_{\rho} = \inf\{t \geq 0 : \dist(\wh{\eta}(t),[x,\infty)) < \rho\}$ and let $\wh{\sigma}$ be the last time before time $\wh{\tau}_{\rho_1}$ that $\wh{\eta}$ hits $\R_-$.  Set also $\wh{J} = \wh{\eta}([\wh{\sigma},\wh{\tau}_{\rho_2}])$.  Similarly with Step 3,  possibly by taking $\epsilon_0 \in (0,1)$ to be smaller and $M \in (1,\infty)$ to be larger,  we have that the following holds for all $\epsilon \in (0,\epsilon_0)$.  Let $a,b$ be points on $\wh{J}$ viewed as prime ends on the right side of $\wh{\eta}$ such that $\wh{\eta}$ hits $a$ before it hits $b$ and $\epsilon \leq \qbmeasure{h}([a,b]) \leq \epsilon_0$.  Then we have that $\qdistnoarg{h|_{\h \setminus \wh{J}}}(a,b) \geq \epsilon^M$.  Suppose that these two events hold.  Note that we can choose $0<\rho_2<\rho_1<x$ sufficiently small such that $\wt{\eta}([\wt{\tau}_{\rho_1},\infty)) \cap \wh{\eta}([\wh{\tau}_{\rho_1},\infty)) = \emptyset$ with as high probability as we want.  Suppose that this occurs as well.  Note also that in that case,  we have that $\wt{J} \cup \wh{J}  = \h \cap \partial \CN$.  Moreover,  possibly by taking $\epsilon_0 \in (0,1)$ to be smaller,  we can assume that the boundary length distance in $\partial \CN$ between $\wt{\eta}([\wt{\tau}_{\rho_1},\infty)) \cap \partial \CN$ and $\wh{\eta}([\wh{\tau}_{\rho_1},\infty)) \cap \partial \CN$ is at least $\epsilon_0$.  Fix $a,b \in \h \cap \partial \CN$ such that $\epsilon \leq \min\{\qbmeasure{h}([a,b]_{\partial \CN}^\cw),  \qbmeasure{h}([a,b]_{\partial \CN}^\ccw)\} \leq \epsilon_0$.  Without loss of generality,  we can assume that $[a,b]_{\partial \CN}^\ccw \subseteq \h$.  Suppose that $a \in \wt{\eta}([0,\wt{\tau}_{\rho_1}])$.  If $b \in \wt{\eta}([0,\wt{\tau}_{\rho_1}])$,  then we have that $\qdistnoarg{h|_{\CN}}(a,b) \geq \qdistnoarg{h|_{\h \setminus \wt{J}}}(a,b) \geq \epsilon^M$.  If $b \notin \wt{\eta}([0,\wt{\tau}_{\rho_1}])$,  then we must have that $b \in \wh{\eta}([0,\wh{\tau}_{\rho_1}])$.  Suppose that $a \notin \wh{\eta}([0,\wh{\tau}_{\rho_1}])$.  If $b \in \wh{\eta}([0,\wh{\tau}_{\rho_1}])$,  then clearly we have that $\qdistnoarg{h|_{\CN}}(a,b) \geq \qdistnoarg{h|_{\h \setminus \wh{J}}}(a,b) \geq \epsilon^M$.  If $b \notin \wh{\eta}([0,\wh{\tau}_{\rho_1}])$,  then we must have that the boundary length distance in $\partial \CN$ between $a$ and $b$ is at least $\epsilon_0$,  and that is a contradiction.  Therefore,  in any case,  we have that $\qdistnoarg{h|_{\CN}}(a,b) \geq \epsilon^M$.  Moreover,  following similar arguments as the ones given in the previous paragraphs,  we obtain that there a.s.\  exists $\delta_0 > 0$ such that $\qdistnoarg{h|_{\CN}}(a,b) \geq \delta_0$ for all $a,b \in \partial \CN$ such that $[a,b]_{\partial \CN}^\ccw \subseteq \partial \CN$ and $\qbmeasure{h}([a,b]_{\partial \CN}^\ccw) \geq \epsilon_0$.  Combining everything,  we obtain that there a.s.\  exist $\epsilon_0 \in (0,1) ,  M \in (1,\infty)$ such that for all $\epsilon \in (0,\epsilon_0)$,  the following holds.  Let $a,b \in \h \cap  \partial \CN$ be such that $\qbmeasure{h}([a,b]_{\partial \CN}^\cw) \geq \epsilon$ and $\qbmeasure{h}([a,b]_{\partial \CN}^\ccw) \geq \epsilon$.  Then we have that $\qdistnoarg{h|_{\CN}}(a,b) \geq \epsilon^M$.  This completes the proof of the lemma.
\end{proof}

Now we are ready to prove that condition~\eqref{it:reverse_holder_continuity} holds with high probability in the context of Proposition~\ref{prop:good_sle_chunks}.

\begin{lemma}\label{lem:reverse_holder_continuity}
Suppose that we have the setup of Lemma~\ref{lem:reverse_holder_continuity_top_boundary}.  Then there a.s.\  exist $\epsilon_0 \in (0,1)$ and $M \in (1,\infty)$ such that for all $\epsilon \in (0,\epsilon_0)$ the following hold.  Let $a,b  \in  \partial \CN$ be such that $\qbmeasure{h}([a,b]_{\partial \CN}^\cw) \geq \epsilon$ and $\qbmeasure{h}([a,b]_{\partial \CN}^\ccw) \geq \epsilon$.  Then,  we have that $\qdistnoarg{h|_{\CN}}(a,b) \geq \epsilon^M$.  Also,  we have that if $z,w \in \partial \CN$ are such that either $\qbmeasure{h}([z,w]_{\partial \CN}^\cw) \leq \epsilon$ or $\qbmeasure{h}([z,w]_{\partial \CN}^\ccw) \leq \epsilon$,  then $\qdistnoarg{h|_{\CN}}(z,w) \leq \epsilon^{1/M}$.
\end{lemma}

\begin{proof}
We will only prove the first claim of the lemma (lower bound) since the second claim (upper bound) follows from similar arguments.  Lemma~\ref{lem:reverse_holder_continuity_top_boundary} implies that it suffices to prove the claim in the case that either both $a$ and $b$ lie on $\R \cap \partial \CN$ or one of the points lies on $\R \cap \partial \CN$ and the other one lies on $\h \cap \partial \CN$.  We will first prove the claim in the case that $\{a,b\} \subseteq \R \cap \partial \CN$.  Recall the decomposition $\qdiskweighted{1}[\,\cdot \,] = \int_0 ^{\infty} \qdiskweightedfixedarea{\alpha}{1}[\,  \cdot \,] \frac{1}{\sqrt{2\pi \alpha^3}} e^{-1/(2\alpha)} d\alpha$.  Also,  Proposition~\ref{prop:good_quantum_disk} implies that the following is true for $\qdiskweighted{1}$-a.e.\  instance $(\h,\check{h})$.  There exist $\epsilon_0 \in (0,1),  M \in (1,\infty)$ such that the following holds for all $\epsilon \in (0,\epsilon_0)$.  Let $a<b \in \R$ be such that $\qbmeasure{\check{h}}([a,b]) \geq \epsilon$ and $\qbmeasure{\check{h}}(\R \setminus [a,b]) \geq \epsilon$.  Then we have that $\qdistnoarg{\check{h}}(a,b) \geq \epsilon^M$.  Thus,  combining with the above disintegration,  we obtain that the same is a.e.\  true if we replace $\qdiskweighted{1}$ by $\qdiskweightedfixedarea{\alpha}{1}$ for all $\alpha > 0$.  In particular,  it holds when $\alpha = 1$.  Let $(\h,\check{h})$ be a sample from $\qdiskweightedfixedarea{1}{1}$ and let $y \in \partial \h$ be sampled uniformly according to $\qbmeasure{\check{h}}$.  Fix $p \in (0,1)$.  Then Lemma~\ref{lem:weighted_browniab_disk_and_brownian_half_plane} implies that there exists $\wt{a} >0$ depending only on $p$ and a coupling between $(\h,\check{h})$ and $(\h,h)$ such that with probability at least $1-\frac{p}{2}$,  we have that the metric spaces $\qball{\check{h}}{0}{\wt{a}}$ and $\qball{h}{0}{\wt{a}}$ agree in the sense of Lemma~\ref{lem:weighted_browniab_disk_and_brownian_half_plane}.  Thus,  combining with the scale invariance of the law of $\mathcal{W}$ (see \cite[Proposition~4.7]{dms2014mating}),  we obtain that for $R>0$ fixed,  there exist $\epsilon_0 \in (0,1),  M \in (1,\infty)$ depending only on $p$ and $R$,  such that the following holds with probability at least $1-\frac{p}{2}$.  For all $\epsilon \in (0,\epsilon_0)$,  we have that if $z<w \in \R$ are such that $[z,w] \subseteq \qball{h}{0}{R}$ and $\qbmeasure{h}([z,w]) \geq \epsilon$,  we have that $\qdistnoarg{h}(z,w) \geq \epsilon^M$.  Also,  we can choose $R>0$ sufficiently large such that $K_{\tau} \subseteq \qball{h}{0}{R}$ with probability at least $1-\frac{p}{2}$.  Combining,  we obtain that there exist $\epsilon_0 \in (0,1),  M \in (1,\infty)$ depending only on $p$,  such that the following holds with probability at least $1-p$.  For all $\epsilon \in (0,\epsilon_0)$ and all $z<w \in \R \cap K_{\tau}$ such that $\qbmeasure{h}([z,w]) \geq \epsilon$,  we have that $\qdistnoarg{h|_{\CN}}(z,w) \geq \qdistnoarg{h}(z,w) \geq \epsilon^M$.  This completes the proof of the lower bound in the case that both points lie on $\R \cap \partial \CN$ since $p \in (0,1)$ was arbitrary.

Now,  we prove the lower bound in the case that one of the points lies on $\R \cap \partial \CN$ and the other one lies on $\h \cap \partial \CN$.  Let $U$ be the connected component of $\h \setminus \eta'([0,\tau])$ whose boundary contains $x$.  Then combining \cite[Theorem~1.18]{dms2014mating} with \cite[Corollary~4.3]{ang2023conformal},  we obtain that the quantum surface $(U,h|_U)$ can be sampled as follows.  Let $\mathcal{D}_1,  \mathcal{D}$ be two independent samples from the infinite measures on quantum disks of weights $W = \gamma^2-2 = \frac{2}{3}$ and $\gamma^2 - W = 2$ respectively,  conditioned on the event that $\ell_1<\ell<\ell_1+\ell_2$,  where $\ell_1,\ell_2$ are the right boundary lengths of $\mathcal{D}_1$ and $\mathcal{D}$ respectively.  Note that both of $\mathcal{D}_1$ and $\mathcal{D}$ are equipped with two marked points.  Then the quantum surface $(U,h|_U)$ equipped with the first and last boundary point hit by $\eta$ has the same law as the marginal of $\mathcal{D}$ under the above conditional law.  It follows by combining with Proposition~\ref{prop:good_quantum_disk} that it is a.s.\  the case that possibly by taking $\epsilon_0 \in (0,1)$ to be smaller and $M \in (1,\infty)$ to be larger,  we have that for all $\epsilon \in (0,\epsilon_0)$,  if $z,w \in \partial U$ are such that $\min\{\qbmeasure{h}([z,w]_{\partial U}^\ccw),\qbmeasure{h}([z,w]_{\partial U}^\cw)\} \geq \epsilon$,  then $\qdistnoarg{h|_U}(z,w) \geq \epsilon^M$.  Moreover,  Lemma~\ref{lem:boundary_length_regularity} combined with absolute continuity imply that there a.s.\  exists $C<\infty$ such that $\qdistnoarg{h|_U}(z,w) \leq C (\min\{\qbmeasure{h}([z,w]_{\partial U}^\ccw),\qbmeasure{h}([z,w]_{\partial U}^\cw)\})^{1//3}$ for all $z,w \in \partial U$.  Furthermore,  we can assume that $\dist_{\qdistnoarg{h|_{\CN}}}([0,c] ,  \h \cap \partial \CN) \geq \epsilon_0$,  where $c = \frac{a+b}{2}$ and $a$ (resp.  $b$) is the first (resp.  last) point of $\partial U$ visited by $\eta'$,  since $\dist_{\qdistnoarg{h|_{\CN}}}([0,c] ,  \h \cap \partial \CN) > 0$ a.s.  Now,  fix $\epsilon \in (0,\epsilon_0)$ and let $z \in [c,b] ,  w \in \h \cap \partial \CN$ be such that $\qbmeasure{h}([z,w]_{\partial \CN}^\ccw) \geq \epsilon$ and $\qbmeasure{h}([z,w]_{\partial \CN}^\cw) \geq \epsilon$.  Suppose that $\qbmeasure{h}([z,b]) \leq \epsilon^{M^2}$.  Then,  we have that $\qdistnoarg{h|_{\CN}}(z,b) \leq \qdistnoarg{h|_U}(z,b) \leq C \epsilon^{M^2 / 3}$ and so combining with triangle inequality and Lemma~\ref{lem:reverse_holder_continuity_top_boundary},  we obtain that $\qdistnoarg{h|_{\CN}}(z,w) \geq \epsilon^M / 2$ possibly by taking $\epsilon_0>0$ to be smaller.  Suppose that $\qbmeasure{h}([z,b]) > \epsilon^{M^2}$.  Then,  we have that $\dist_{\qdistnoarg{h|_U}}(z,\h \cap \partial U) \geq \epsilon^{M^3}$ and note that any path $P$ in $\CN$ connecting $w$ to $z$ has to intersect $\partial U$.  Suppose that $P$ is parameterized by $[0,1]$ and $P(0) = w,  P(1) = z$.  Let $t$ be the last time that $P$ hits $\h \cap \partial U$.  Then $P|_{(t,1)}$ is a path in $U$ connecting some point on $\h \cap \partial U$ to $z$ and so the $\qdistnoarg{h}$-length of $P|_{(t,1)}$ is at least $\epsilon^{M^3}$.  It follows that $\qdistnoarg{h|_{\CN}}(z,w) \geq \epsilon^{M^3}$ since $P$ was arbitrary.  Combining,  we obtain that there a.s.\  exist $\epsilon_0 \in (0,1),  M \in (1,\infty)$ such that the following holds for all $\epsilon \in (0,\epsilon_0)$.  Let $z \in [0,b] ,  w \in \h \cap \partial \CN$ be such that $\min\{\qbmeasure{h}([z,w]_{\partial \CN}^\ccw),\qbmeasure{h}([z,w]_{\partial \CN}^\cw)\} \geq \epsilon$.  Then we have that $\qdistnoarg{h|_{\CN}}(z,w) \geq \epsilon^M$. 

Note that since the law of the outer boundary of $\eta'$ when targeted at $x$ has time-reversal symmetry and the law of $\mathcal{W}$ is invariant under translating horizontally by a fixed number of quantum length units (see \cite[Proposition~1.7]{sheffield2015conformalweldingsrandomsurfaces}),  we obtain that the law of the quantum surface parameterized by $\CN$ is the same with the law of the quantum surface parameterized by $\wt{\CN}$,  where $\wt{\CN}$ is defined as follows.  We let $y \in \R_-$ be such that $\qbmeasure{h}([y,0]) = \ell$ and draw $\eta'$ up until the first time that it disconnects $y$ from $\infty$.  Then $\wt{\CN}$ is the quantum surface induced by the restriction of $h$ to the hull of the above curve stopped at  the above time.  Then,  arguing as in the previous paragraph,  we obtain that there a.s.\  exist $\epsilon_0 \in (0,1),  M \in (1,\infty)$ such that the following is true for all $\epsilon \in (0,\epsilon_0)$.  Let $z,w \in \partial \wt{\CN}$ be such that $z \in \h \cap \partial \wt{\CN},  w \in \R_- \cap \partial \wt{\CN}$ and 
 $\min\{\qbmeasure{h}([z,w]_{\partial \wt{\CN}}^\ccw),\qbmeasure{h}([z,w]_{\partial \wt{\CN}}^\cw)\} \geq \epsilon$.  Then it holds that $\qdistnoarg{h|_{\wt{\CN}}}(z,w) \geq \epsilon^M$.  Therefore,  combining with the result of the previous paragraph,  we complete the proof of the lower bound in the lemma statement in the case that one of the points lies on $\h \cap \partial \CN$ and the other on $\R \cap \partial \CN$.  This completes the proof of the lemma.
\end{proof}

Next, we state that conditions~\eqref{it:boundary_volume_lbd} and~\eqref{it:ngbd_vollume_ubd} hold with high probability in the context of Proposition~\ref{prop:good_sle_chunks} provided $M$ is large enough and $p$ is small enough.  We will not give a detailed proof since it follows from a combination of Proposition~\ref{prop:good_quantum_disk} with the arguments presented in the proofs of Lemmas~\ref{lem:reverse_holder_continuity_top_boundary} and~\ref{lem:reverse_holder_continuity}.

\begin{lemma}\label{lem:boundary_volume_lbd_ngbd_vollume_ubd}
Fix $u > 0$ and let $p > 0$ be a constant (corresponding to condition~\eqref{it:ngbd_vollume_ubd}) which is small enough as specified in Proposition~\ref{prop:good_quantum_disk}. Suppose that we have the setup of Lemmas~\ref{lem:reverse_holder_continuity_top_boundary} and~\ref{lem:reverse_holder_continuity}.  Then there a.s.\  exist $\epsilon_0 \in (0,1)$ and $M \in (1,\infty)$ such that for all $\epsilon \in (0,\epsilon_0)$ the following hold.  For all $x, y \in \partial \CN$,  the $\epsilon$-neighborhood of $[x,y]_{\partial \CN}^\ccw$ with respect to $\qdistnoarg{h|_{\CN}}$ has quantum area at least $\frac{\epsilon^{2+u}}{M}$.  The same is also true with $[x,y]_{\partial \CN}^\cw$ in place of $[x,y]_{\partial \CN}^\ccw$.  Moreover,  the quantum area of the $\epsilon$-neighborhood of $\partial \CN$ with respect to $\qdistnoarg{h|_{\CN}}$ is at most $\epsilon^p$.
\end{lemma}

\begin{proof}
The two claims of the lemma essentially follow from the same arguments presented in Lemmas~\ref{lem:reverse_holder_continuity_top_boundary} and~\ref{lem:reverse_holder_continuity} combined with Proposition~\ref{prop:good_quantum_disk}.
\end{proof}

Now we prove that conditions~\eqref{it:diameter_bound} and~\eqref{it:mass_bound} hold with high probability if we choose $M$ sufficiently large.  We first prove in the following lemma that condition~\eqref{it:diameter_bound} holds with high probability provided we choose $M$ sufficiently large.

\begin{lemma}\label{lem:diameter_bound}
Suppose that we have the setup of Lemmas~\ref{lem:reverse_holder_continuity_top_boundary}-\ref{lem:boundary_volume_lbd_ngbd_vollume_ubd}.  Then we have that the $\qdistnoarg{h|_{\CN}}$-diameter of $\CN$ is finite a.s.
\end{lemma}

\begin{proof}
Suppose that we have the setup of the proof of Lemma~\ref{lem:reverse_holder_continuity_top_boundary}.  Fix $p_1 \in (0,1)$ and let $\rho_1>0$ be chosen such that with probability at least $1-\frac{p_1}{2}$,  we have that $\wh{\eta}([0,\wh{\tau}_{\rho_1}]) \cap \partial U \neq \emptyset$,  where both $\wh{\eta},\wh{\tau}_{\rho}$ are defined in Step 4 in the proof of Lemma~\ref{lem:reverse_holder_continuity_top_boundary} and $U$ is the connected component of $\h \setminus \eta'([0,\tau])$ whose boundary contains $x$.  Let also $\check{h}$ be the random field introduced in Step 3 in the proof of Lemma~\ref{lem:reverse_holder_continuity_top_boundary}.  Since the $\qdistnoarg{\check{h}}$-diameter of any compact set $K \subseteq \overline{\h}$ is finite a.s.\  (see \cite[Theorem~1.3]{hughes2024equivalencemetricgluingconformal}),  arguing as in Step 3 in the proof of Lemma~\ref{lem:reverse_holder_continuity_top_boundary} gives that the $\qdistnoarg{h|_{\h \setminus \wt{\eta}([\wt{\sigma},\wt{\tau}_{\rho_1}])}}$-diameter of $\CN$ is finite a.s.,  where $\wt{\eta},\wt{\tau}_{\rho}$ and $\wt{J}$ are defined in Step 1 in the proof of Lemma~\ref{lem:reverse_holder_continuity_top_boundary}.  Also,  since the quantum surface parameterized by $U$ has the law of a quantum disk conditioned on a positive probability event (see the proof of Lemma~\ref{lem:reverse_holder_continuity}),  we obtain that the $\qdistnoarg{h|_U}$-diameter of $U$ is finite a.s.  Hence,  there exists $M>1$ sufficiently large such that with probability at least $1-\frac{p_1}{2}$,  both of the $\qdistnoarg{h|_{\h \setminus \wh{\eta}([\wh{\sigma},\wh{\tau}_{\rho_1}])}}$-diameter of $\CN$ and the $\qdistnoarg{h|_U}$-diameter of $U$ are at most $M$.  Suppose that we are working on the event that the above holds and that $\wh{\eta}([0,\wh{\tau}_{\rho_1}]) \cap \partial U \neq \emptyset$.  Similarly,  we can assume in addition that the $\qdistnoarg{h|_{\h \setminus \wt{\eta}([\wt{\sigma},\wt{\tau}_{\rho_1}])}}$-diameter of $\CN$ is at most $M$ and that $\wt{\eta}([\wt{\tau}_{\rho_1},\infty)) \cap \wh{\eta}([\wh{\tau}_{\rho_1},\infty)) = \emptyset$.  Let $t$ be the first time that $\wh{\eta}$ intersects $\partial U$.  Fix $z \in \partial \CN \cap \wh{\eta}([0,t))$ and let $P : (0,1) \to \h \setminus \wh{\eta}([\wh{\sigma},\wh{\tau}_{\rho_1}])$ be a path such that $P(0) = 0,  P(1) = z$ and such that the $\qdistnoarg{h}$-length of $P$ is at most $M$.  If $P$ doesn't intersect $\overline{U}$,  then it stays in $\CN$ and so we have that $\qdistnoarg{h|_{\CN}}(0,z) \leq M$.  Suppose that $P \cap \overline{U} \neq \emptyset$,  and let $s_1$ be the last time that $P$ intersects $\partial U$.  Then we have that $P|_{(s_1,1)}$ is a path in $\CN$ from $P(s_1)$ to $z$ with $\qdistnoarg{h}$-length at most $M$.  Let also $s_2$ be the first time that $P$ intersects $\partial U$ and let $Q$ be a path in $U$ from $P(s_2)$ to $P(s_1)$ with $\qdistnoarg{h}$-length at most $M$.  Then the concatenation $\wt{P}$ of $P|_{(0,s_2)},  Q$ and $P|_{(s_1,1)}$ is a path in $\CN$ from $0$ to $z$ with $\qdistnoarg{h}$-length at most $3M$.  Hence $\qdistnoarg{h|_{\CN}}(0,z) \leq 3M$.  Next,  we fix $z \in \partial \CN \cap \wh{\eta}((t,\infty))$ and note that $z \in \partial \CN \cap \wt{\eta}([0,\wt{\tau}_{\rho_1}])$.  Let $A$ be a path in $\h \setminus \wh{\eta}([\wh{\sigma},\wh{\tau}_{\rho_1}])$ from $0$ to $z$ with $\qdistnoarg{h}$-length at most $M$ and let $s_1$ be the first time that $A$ intersects $\partial U$ (it has to intersect $\partial U$ in order to eventually hit $z$).  Let also $\wt{A}$ be a path in $\h \setminus \wt{\eta}([\wt{\sigma},\wt{\tau}_{\rho_1}])$ from $A(s_1)$ to $z$ with $\qdistnoarg{h}$-length at most $M$ and let $s_2$ be the last time that $\wt{A}$ intersects $\partial U$. Then the path $\wt{A}|_{(s_2,1)}$ has to stay in $\CN$ and has $\qdistnoarg{h}$-length at most $M$.  Let also $B$ be a path in $U$ from $A(s_1)$ to $\wt{A}(s_2)$ with $\qdistnoarg{h}$-length at most $M$,  and let $\wh{A}$ be the concatenation of $A|_{(0,s_1)},  B$,  and $\wt{A}|_{(s_2,1)}$.  Then $\wh{A}$ is a path in $\CN$ from $0$ to $z$ with $\qdistnoarg{h}$-length at most $3M$.  It follows that $\qdistnoarg{h|_{\CN}}(0,z) \leq 3M$ for all $z \in \wh{\eta}([\wh{\sigma},t) \cup (t,\infty)) \cap \partial \CN$ and the same bound holds for $z = \wh{\eta}(t)$ by the continuity of the quantum metric with respect to the Euclidean metric.  Therefore,  we obtain that the $\qdistnoarg{h|_{\CN}}$-diameter of $\CN$ is at most $6M$ with probability at least $1-p_1$.  The proof is then complete since $p_1 \in (0,1)$ was arbitrary.
\end{proof}

The final step before proving Proposition~\ref{prop:good_sle_chunks}             is to prove that condition~\eqref{it:mass_bound} holds with high probability if $M$ is large enough. This is the content of the following lemma.

\begin{lemma}\label{lem:mass_bound}
Suppose that we have the setup of Lemmas~\ref{lem:reverse_holder_continuity_top_boundary}-\ref{lem:diameter_bound} and let $\varphi \colon \CN \to \D$ be the conformal transformation as in condition~\eqref{it:mass_bound}.  Then there a.s.\  exists $M>1$ such that for all $r \in (0,M^{-1})$ the following hold.  The quantum area assigned to $B(0,1/2)$ with respect to the field $h|_{\CN} \circ \varphi^{-1} + Q \log |(\varphi^{-1})'|$ is at least $M^{-1}$.  Moreover,  every point with quantum distance at least $r$ from $\partial \D$ (with respect to the field $h|_{\CN} \circ \varphi^{-1} + Q \log |(\varphi^{-1})'|$) has Euclidean distance at least $r^M$ from $\partial \D$.  
\end{lemma}

We will describe the setup of the proof of Lemma~\ref{lem:mass_bound} before proceeding with its proof.  Note that it is a.s.  the case that on the event that $\sigma = \sigma_{\delta / A}$,  we have that $\sigma = \tau_x$ for some $x \in \Q \setminus \{0\}$,  where $\tau_x$ denotes the first time that $\eta'$ disconnects $x$ from $\infty$.  Hence it suffices to prove the claim of the lemma in the case that $\sigma = \tau_x$ for some $x \in \Q \setminus \{0\}$ fixed.  Without loss of generality we can assume that $x>0$.

Let $(K_t)$ denote the family of hulls of $\eta'$.  Note that the locality property of $\SLE_6$ implies that $\eta'$ can be coupled with a chordal $\SLE_6$ $\wt{\eta}'$ in $\h$ from $0$ to $x$ stopped at the first time $\wt{\tau}_x$ that $\wt{\eta}'$ disconnects $x$ from $\infty$ such that $\eta'|_{[0,\tau_x]} = \wt{\eta}'|_{[0,\wt{\tau}_x]}$.  Let also $\wt{\eta}$ denote the left outer boundary of $\wt{\eta}'$ and note that \cite[Theorem~1.4]{ms2016ig1} implies that $\wt{\eta}$ has the law of an $\SLE_{\frac{8}{3}}(\frac{8}{3}-2 ; \frac{8}{3}-4)$ process in $\h$ from $x$ to $0$ with the force points located at $x^-$ and $x^+$ respectively.  Let $U$ be the connected component whose boundary contains $0$ of the complement in $\h$ of the curve $\wt{\eta}$ stopped at the first time that it disconnects $0$ from $\infty$.  Similarly we let $V$ be the connected component whose boundary contains $x$ of the complement in $\h$ of the time-reversal of $\wt{\eta}$ stopped at the first time that it disconnects $x$ from $\infty$.  Moreover we let $G$ be the connected component of $\h \setminus \wt{\eta}$ lying to the left of $\wt{\eta}$.  Let $f : U \to \D,  g : V \to \D$ and $\psi : G \to \D$ be conformal transformations defined in some arbitrary but fixed way.  

We note that all of the maps $f,g$ and $\psi$ are a.s.  well-defined due to Proposition~\ref{prop:SLE6-chunk-Jordan-D} and the time-reversal symmetry of $\wt{\eta}$ (see \cite[Theorem~1.1]{ms2016ig2}).  Recall that \cite[Theorem~5.2]{rohde2011basic} implies that there exists deterministic constant $\alpha \in (0,1)$ such that all of the maps $f^{-1},g^{-1}$ and $\psi^{-1}$ are $\alpha$-H\"older continuous a.s.

\begin{proof}[Proof of Lemma~\ref{lem:mass_bound}]
\emph{Step 1.  Outline.} 
Suppose that we have the setup described in the paragraphs just after the statement of Lemma~\ref{lem:mass_bound}.
The first claim of the lemma (lower bound on the quantum mass assigned to $B(0,1/2)$) follows since the quantum area measure assigns positive mass to every open set a.s.  Hence,  we will focus on proving the second claim of the lemma.  In Step 2,  we will prove that it is a.s.  the case that there exist $r_0 \in (0,1),  M>1$ such that for all $r \in (0,r_0)$ and all $z \in \CN$ such that $\text{dist}_{d_{h|_{\CN}}}(z,\partial \CN) \geq r$,  we have that 
\begin{align*}
\eball{z}{r^M} \subseteq \qball{h|_{\CN}}{z}{r} \subseteq \CN.
\end{align*}
Then in Step 3,  we will show that possibly by taking $M$ to be larger and $r_0$ to be smaller,  we have that there a.s.  exists (random) $c_0>0$ such that if $z,r$ are as above,  we have that the probability that a complex Brownian motion starting from $z$ intersects $\psi^{-1}(\eball{0}{1/2})$ before exiting $\CN$ for the first time is at least $c_0 r^M$,  and then conclude the proof of the lemma.

\emph{Step 2.  $\eball{z}{r^M} \subseteq \qball{h|_{\CN}}{z}{r}$.}
First,  we note that \cite[Proposition~1.8]{hughes2024equivalencemetricgluingconformal} implies that there exists a deterministic constant $\wt{\beta} \in (0,1)$ such that $\qdistnoarg{\wt{h}}|_K$ is a.s.\  $\wt{\beta}$-H\"older continuous with respect to the Euclidean metric for all $K \subseteq \overline{\h}$ compact,  where $\wt{h}$ is a free boundary GFF on $\h$ with the additive constant taken so that the average of $\wt{h}$ on $\h \cap \partial \D$ is equal to zero.  Then,  arguing as in Step 2 in the proof of Lemma~\ref{lem:reverse_holder_continuity_top_boundary},  we obtain that for all $0<r<R$ fixed,  $d_{h}|_{A_{r,R}}$ is a.s.\  $\wt{\beta}$-H\"older continuous with respect to the Euclidean metric,  where $A_{r,R} = \h \cap (B(0,R) \setminus \overline{B(0,r)})$.  Moreover,  combining with the invariance of the law of $\mathcal{W}$ when translating horizontally by fixed number of units of boundary length (see \cite[Proposition~1.7]{sheffield2015conformalweldingsrandomsurfaces}),  we obtain that $d_h|_{\h \cap B(0,R)}$ is a.s.\  $\wt{\beta}$-H\"older continuous with respect to the Euclidean metric for all $R>0$.  Let $R>0$ be such that $\CN \subseteq \h \cap B(0,R)$.  Note that if $z \in \CN$ is such that $\dist_{\qdistnoarg{h|_{\CN}}}(z,\partial \CN) > r$,  we have that $\qball{h|_{\CN}}{z}{r} = \qball{h}{z}{r}$.  Therefore,  combining with the H\"older continuity of $h|_{\h \cap B(0,R)}$ with respect to the Euclidean metric,   we obtain that it is a.s.\  the case that there exists $M>1$ large and $r_0>0$ small such that for all $z \in \CN$ and all $r \in (0,r_0)$ such that $\dist_{\qdistnoarg{h|_{\CN}}}(z,\partial \CN)  > r$,  we have that $\eball{z}{r^M} \subseteq \qball{h|_{\CN}}{z}{r}$. In particular,  the Euclidean distance of $z$ from $\partial \CN$ is at least $r^M$.  

\emph{Step 3.  Conclusion of the proof.}
Let $r_0 \in (0,1),  M>1$ be as in Step 2 and fix $r \in (0,r_0),  z \in \CN$ such  that $\text{dist}_{d_{h|_{\CN}}}(z,\partial \CN) > r$.  Then Step 2 implies that
\begin{align*}
\eball{z}{r^M} \subseteq \qball{h|_{\CN}}{z}{r} \subseteq \CN.
\end{align*}
We will show that possibly by taking $M$ to be larger,  we have that the probability that a complex Brownian motion starting from $z$ intersects $\psi^{-1}(\eball{0}{1/2})$ before exiting $\CN$ for the first time is at least $r^M$. 

Let $\wt{I}$ be the arc traced by $\wt{\eta}$ up until the last time that it intersects $\R_+$ and let $\wt{J}$ be the arc traced by the time-reversal of $\wt{\eta}$ up until the last time that it intersects $\R_-$.  Then we have the following cases.

\emph{Case 1.  $\qball{h|_{\CN}}{z}{r} \cap (\wt{I} \cup \wt{J}) = \emptyset$.}
Then we have that $\eball{z}{r^M} \subseteq G$.  Note that there exists $C>1$ such that
\begin{align*}
|\psi^{-1}(x) - \psi^{-1}(y)| \leq C |x-y|^{\alpha} \quad \text{for all} \quad x,y \in \D.
\end{align*}
In particular we have that $\text{dist}(\psi(z),\partial \D) \gtrsim r^{M / \alpha}$ and so the probability that a complex Brownian motion starting from $\psi(z)$ intersects $\eball{0}{1/2}$ before exiting $\D$ for the first time is $\gtrsim r^{M/\alpha}$.  By conformal invariance we obtain that the probability that a complex Brownian motion starting from $z$ intersects $\psi^{-1}(\eball{0}{1/2})$ before exiting $\D$ for the first time is $\gtrsim r^{M / \alpha}$.

\emph{Case 2.  $\qball{h|_{\CN}}{z}{r} \cap \wt{I} \neq \emptyset$.}
Possibly by taking $r_0 \in (0,1)$ to be smaller,  we can assume that $\qball{h|_{\CN}}{z}{r} \cap \wt{J} = \emptyset$.  Hence we have that $\qball{h|_{\CN}}{z}{r} = \qball{h|_V}{z}{r}$ and so $\eball{z}{r^M} \subseteq V$.  As in Case 1,  we let $C>1$ be such that
\begin{align*}
|g^{-1}(x) - g^{-1}(y)| \leq C |x-y|^{\alpha} \quad \text{for all} \quad x,y \in \D,
\end{align*}
and then we have that
\begin{align*}
\eball{g(z)}{C^{-1/\alpha} r^{M / \alpha}} \subseteq g(\eball{z}{r^M}) \subseteq \D.
\end{align*}
Therefore arguing as in Case 1,  we have that the probability that a complex Brownian motion starting from $z$ intersects $\psi^{-1}(\eball{0}{1/2})$ before exiting $\CN$ for the first time is $\gtrsim r^{M / \alpha}$.

\emph{Case 3.  $\qball{h|_{\CN}}{z}{r} \cap \wt{J} \neq \emptyset$.}
Arguing as in Cases 1 and 2,  we obtain that the probability that a complex Brownian motion starting from $z$ intersects $\psi^{-1}(\eball{0}{1/2})$ before exiting $\CN$ for the first time is $\gtrsim r^{M / \alpha}$.

Combining Cases 1,2 and 3,  we obtain that there a.s.  exists $c_0>0$ such that the probability that a complex Brownian motion starting from $z$ intersects $\psi^{-1}(\eball{0}{1/2})$ before exiting $\CN$ for the first time is at least $c_0 r^{M / \alpha}$.   It follows that the probability that a complex Brownian motion starting from $\varphi(z)$ intersects $\varphi(\psi^{-1}(\eball{0}{1/2}))$ before exiting $\D$ for the first time is at least $c_0 r^{M / \alpha}$.  Set $d:=\text{dist}(\varphi(\psi^{-1}(\eball{0}{1/2})),\partial \D) >0$ and suppose that $\text{dist}(\varphi(z),\partial \D) < r^{3M / \alpha}$.  Then the Beurling estimate implies that the probability that a complex Brownian motion starting from $\varphi(z)$ intersects $\varphi(\psi^{-1}(\eball{0}{1/2}))$ before exiting $\D$ for the first time is at most a universal constant times $(r^{3M / \alpha} / d)^{1/2}$.  In particular we have that $r^{M / \alpha} \lesssim r^{3M / (2\alpha)}$ and so we obtain a contradiction possibly by taking $r_0 \in (0,1)$ to be smaller.  It follows that
\begin{align*}
\text{dist}(\varphi(z),  \partial \D) \gtrsim r^{3M / \alpha}
\end{align*}
and so this completes the proof of the lemma.
\end{proof}

\begin{proof}[Proof of Proposition~\ref{prop:good_sle_chunks}]
First we note that when the quantum natural time is scaled by $\delta$,  we have that the quantum boundary length,  the quantum metric distance and the quantum area are scaled by $\delta^{2/3},  \delta^{1/3}$ and $\delta^{4/3}$ respectively.  Hence,  combining with the scale invariance of the law of $\mathcal{W}$ (\cite[Proposition~4.7]{dms2014mating}),  we obtain that it suffices to prove the claim of the proposition when $\delta = 1$.  Hence,  from now on,  we assume that $\delta = 1$.  Fix $p_1 \in (0,1)$.  For all $k \in \N$ and $N \in \N$, we let $x_{k,N}$ be the point on $\R_+$ such that $\qbmeasure{h}([0,x_{k,N}]) = \frac{k}{N}$.  For all $k \in \Z_-$ and $N \in \N$ we also let $x_{k,N}$ be the point on $\R_-$ such that $\qbmeasure{h}([x_{k,N},0]) = \frac{-k}{N}$.  For $x \in \R \setminus \{0\}$,  we let $\eta_x'$ be a chordal $\SLE_6$ in $\h$ from $0$ to $x$ which is independent of $\mathcal{W}$.  Then by \cite[Section~4.2]{sheffield2009exploration},  we obtain that there exists a coupling of $(\mathcal{W},\eta', ( \eta_x')_{x \in \R \setminus \{0\}})$ such that for all $x \in \R \setminus \{0\}$,  the curves $\eta',\eta_x'$ agree up until the first time that they disconnect $x$ from $\infty$.  Also,  we can choose $M_1,N \in \N$ large enough such that with probability at least $1-\frac{p_1}{2}$,  we have that $\CN = \CN_{x_{k,N}}$ for some $k \in [-N M_1 ,  N M_1] \cap \Z$,  where $\CN_{x_{k,N}}$ denotes the quantum surface parameterized by the hull of $\eta_{x_{k,N}}'$ stopped at the first time that it disconnects $x_{k,N}$ from $\infty$.  Let also $E_{x_{k,N}}$ be the event defined in the same way as $E$ but with $\CN_{x_{k,N}}$ in place of $\CN$.  Note that the quantum boundary lengths of the top,  bottom left,  and bottom right of $\CN_{x_{k,N}}$ are all positive a.s.  Therefore,  combining with Lemmas~\ref{lem:reverse_holder_continuity_top_boundary}-\ref{lem:mass_bound},  we obtain that we can choose $M \in (1,\infty)$ large enough and $p > 0$ small enough such that conditions~\eqref{it:boundary_length}-\eqref{it:ngbd_vollume_ubd} all hold with probability at least $1-\frac{p_1}{4NM_1}$. Thus,  taking a union bound over all $k \in [-NM_1,NM_1] \cap \Z$ gives that $E$ holds with probability at least $1-p_1$ for the above choice of $M,p$.  This completes the proof of the proposition.
\end{proof}

\subsection{Size bounds for the disconnecting good annulus}
\label{subsec:good_annulus_size_bounds}

Fix $\delta > 0$ and suppose that we perform the exploration as in Proposition~\ref{prop:good_chunks_percolate} until the first time that $0$ is disconnected from $\partial \CD$ by chunks for which $E$ occurs.  Let $\CA$ be quantum surface parameterized by the cluster of chunks $\CN$ for which $E$ occurs with the property that there exist chunks $\CN_{i_1},\ldots,\CN_{i_n}$ which are discovered by the exploration for which $E$ occurs for all of them with $\CN_{i_1} = \CN$, $\CN_{i_n}$ on the boundary of the connected component which contains $0$, and with $\partial \CN_{i_j} \cap \partial \CN_{i_{j+1}} \neq \emptyset$ for each $1 \leq j \leq n-1$.  We define the inner boundary of $\CA$ to be the boundary of the connected component of $\C \setminus \CA$ which contains $0$.  We define the outer boundary of $\CA$ to be the boundary of the unbounded connected component of $\C \setminus \CA$. 

\begin{proposition}
\label{prop:good_band_contains}
Suppose that $\CD = (\D,h,0)$ has law $\qdiskweighted{\ell}$.
Fix $\delta \in(0,1)$.  Suppose that we perform the exploration as in Proposition~\ref{prop:good_chunks_percolate} until the first time that $0$ is disconnected from $\partial \CD$ by chunks for which $E$ occurs and let $\CA$ be as above.  Fix $\wt{u} \in (0,4)$,  $s_0 >0$ and suppose that we are working on the event that $\qmeasure{h}(\qball{h}{z}{s}) < s^{4-\wt{u}}$ for all $s \in (0,s_0),  z \in \CD$,  and $\qmeasure{h}(\qball{h}{z}{s}) > s^{4+\wt{u}}$ for all $s \in (0,s_0),  z \in \CD$ such that $\dist_{\qdistnoarg{h}}(z,\partial \CD) > s$.  Then,  there exist a constant $p_1>0$ depending only on $u,\wt{u}$ and $M$,  and a constant $\delta_0 \in (0,1)$ depending only on $u,\wt{u},M$ and $s_0$ such that the distance (with respect to the interior-internal metric on $\CA$) between the inner and outer boundary of $\CA$ is at least $\delta^{p_1}$.  Moreover,  there exists a constant $p_0>0$ depending only on $u$ such that the following holds.  There exist constants $c_1,c_2,\alpha>0$ depending only on $\ell,u$ such that for all $\delta \in (0,1)$ sufficiently small (depending only on $\ell,u$),  on the event $E_{u,\delta}$ defined in the statement of Proposition~\ref{prop:good_chunks_percolate},  off an event with probability at most $c_1 \exp(-c_2 \delta^{-\alpha})$,  we have that $\CA$ is contained in the $\delta^{p_0}$-neighborhood of $\partial \CD$ with respect to $\qdistnoarg{h}$.
\end{proposition}

We will begin by establishing the lower bound in Proposition~\ref{prop:good_band_contains}, the main input of which is the next lemma.  Suppose that we have two $\SLE_6$ chunks $\CN_1$, $\CN_2$ from the exploration.  We say that $\CN_1$ comes before $\CN_2$, or $\CN_2$ comes after $\CN_1$, if $\CN_1$ is discovered by the exploration before $\CN_2$.  We say that $\CN_1$ is adjacent to $\CN_2$ if the following is true.  First, $E$ occurs for both $\CN_1$ and $\CN_2$.  Second, $\CN_2$ is the first chunk discovered after $\CN_1$ for which $E$ occurs whose boundary has non-empty intersection with $\partial \CN_1$.  Then by the way that the exploration is defined, we know that an interval on the right bottom of $\CN_2$ (whose left endpoint is the initial point of the $\SLE_6$ in $\CN_2$) is contained in the top of $\CN_1$ if both of $\partial \CN_1$ and $\partial \CN_2$ intersect the inner boundary of $\CA$.

\begin{lemma}
\label{lem:chunk_intersection}

Suppose that we have the setup described in Proposition~\ref{prop:good_chunks_percolate}.  Fix $\wt{u} \in (0,4)$ and $s_0>0$ and suppose that we are working on the event that $\qmeasure{h}(\qball{h}{z}{s}) < s^{4-\wt{u}}$ for all $s \in (0,s_0),  z \in \CD$,  and $\qmeasure{h}(\qball{h}{z}{s}) > s^{4+\wt{u}}$ for all $s \in (0,s_0),  z \in \CD$ such that $\dist_{\qdistnoarg{h}}(z,\partial \CD) > s$.  Then,  there exist a constant $q>0$ depending only on $u$ and $\wt{u}$ and a constant $\delta_0 \in (0,1)$ depending only on $u,\wt{u},M$,  and $s_0$ such that the following holds for all $\delta \in (0,\delta_0)$.  Suppose that $\CN_1,\CN_2$ are two radial $\SLE_6$ chunks which are adjacent to each other with $\CN_2$ coming after $\CN_1$.  Moreover,  we assume that both of $\partial \CN_1$ and $\partial \CN_2$ intersect the inner boundary of $\CA$.  Then,  we have that the diameter of $\partial \CN_1 \cap \partial \CN_2$ with respect to $\qdistnoarg{h}$ is at least $\delta^q$.
\end{lemma}
\begin{proof}

First,  we note that since both of $\partial \CN_1$ and $\partial \CN_2$ intersect the inner boundary of $\CA$,  we have that $I:=\partial \CN_1 \cap \partial \CN_2$ is an interval contained in in the right bottom of $\CN_2$ and at the top of $\CN_1$,  and such that its left endpoint is the initial point of the radial $\SLE_6$ in $\CN_2$.  In particular,  we have that $I$ is the bottom right of $\CN_2$.  Hence,  it follows from part~\eqref{it:boundary_length} that the quantum boundary length of $I$ is at least $\delta^{2/3} / M$.  Note that if $\delta_0 \in (0,1)$ is such that $\delta_0 < M^{-3/2}$,  we have that $\delta < \delta^{1/3} / M$,  and so part~\eqref{it:boundary_volume_lbd} in the definition of $E$ implies that the amount of volume in $\CN_1$ which has distance in $\CN_1$ at most $\delta$ to $I$,  is at least $\delta^{2+u}$ times the length of $I$,  i.e.,  at least $\delta^{8/3 + u} / M$.  Let $q = \frac{8/3 + 3u}{4-\wt{u}} > 0$.  We pick $\delta_0$ sufficiently small so that we also have $\delta_0^q < s_0$.  Then,  we have that the amount of area in any ball in $\CD$ with respect to $\qdistnoarg{h}$ of radius $\delta^q$ is at most $\delta^{8/3 + 3u}$,  for all $\delta \in (0,\delta_0)$.  Therefore,  if we further choose $\delta_0 \in (0,1)$ such that $\delta_0 < M^{-1/(2u)}$,  then we have that $I$ cannot be contained in such a ball,  which implies that the $\qdistnoarg{h}$-diameter of $I$ is at least $\delta^q$.
\end{proof}

\begin{figure}
\includegraphics[scale=0.85]{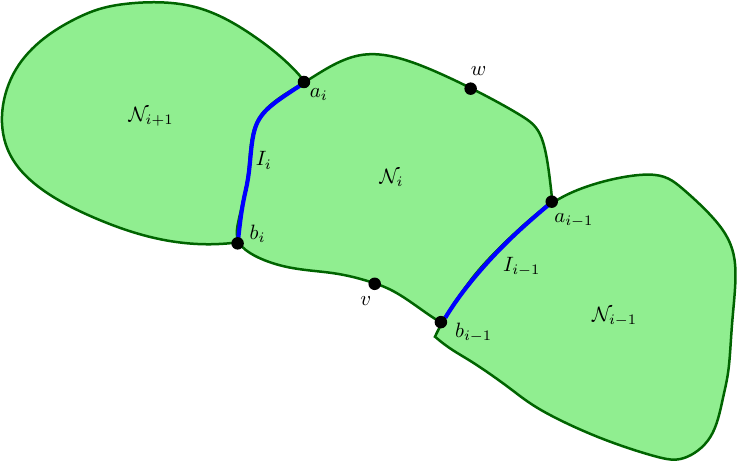}
\caption{\label{fig:intersecting_chunks} Illustration of the setup of the proof of the lower bound of Proposition~\ref{prop:good_band_contains}.}
\end{figure}

\begin{proof}[Proof of Proposition~\ref{prop:good_band_contains}, lower bound]

Suppose that $\delta_0$ and $q$ are as in Lemma~\ref{lem:chunk_intersection}.  Let $\CN_1,\CN_2$ be two adjacent radial $\SLE_6$ chunks with $\CN_2$ coming after $\CN_1$ and such that both of $\partial \CN_1$ and $\partial \CN_2$ intersect the inner boundary of $\CA$.  We will show that the diameter of the interval $I:=\partial \CN_1 \cap \partial \CN_2$ with respect to the interior-internal metric of $\overline{\CN}_1\cup \overline{\CN}_2$ is at least $3\delta^{qM}$ by possibly taking $q$ to be larger (depending only on $u,\wt{u}$ and $M$) and $\delta_0 \in (0,1)$ to be smaller (depending only on $u,\wt{u},M$ and $s_0$).  Indeed,  let $a$ and $b$ be the two endpoints of $I$ with $a$ lying on the inner boundary of $\CA$.  Let also $\gamma$ be the geodesic in $\overline{\CN}_1 \cup \overline{\CN}_2$ from $a$ to $b$ with respect to the interior-internal metric.  Suppose that the length of $\gamma$ (with respect to the interior-internal metric) is at most $\delta^{qM}$.  We can assume that $q>2/3$ and by possibly taking $\delta_0$ to be smaller (depending only on $u,\wt{u},M$ and $s_0$),  we can assume that $\delta^q < \delta^{2/3} / M$.  Then, either $\gamma$ intersects every subinterval of $I$ with boundary length at most $\delta^q$ or it does not.  Suppose that the former holds.  Fix $z \in I$ and let $x,y \in I$ be such that $z \in [x,y]_{\partial \CN_1}^\ccw$ and $\qbmeasure{h}([x,y]_{\partial \CN_1}^\ccw) < \delta^q$.  Fix also $w \in \gamma \cap [x,y]_{\partial \CN_1}^\ccw$.  By condition~\eqref{it:reverse_holder_continuity} of $E$ we obtain that the distance between $z$ and $w$ with respect to the interior-internal metric in $\CN_1$ is at most $\delta^{q/M}$ and so the same holds for the distance between $z$ and $w$ with respect to the interior-internal metric in $\overline{\CN}_1 \cup \overline{\CN}_2$.  Hence,  the distance between $a$ and $z$ with respect to the interior-internal metric in $\overline{\CN}_1 \cup \overline{\CN}_2$ is at most $3\delta^{qM} + \delta^{q/M}$ by the triangle inequality.  Since $z$ was arbitrary,  it follows that the diameter of $I$ with respect to $\qdistnoarg{h}$ is at most $6\delta^{qM} + 2\delta^{q/M}$.  But this is a contradiction due to Lemma~\ref{lem:chunk_intersection} by possibly taking $q>0$ to be larger (depending only on $u,\wt{u}$ and $M$) and $\delta_0 \in (0,1)$ to be smaller (depending only on $u,\wt{u},M$ and $s_0$).     Thus,  there exists a subinterval $J$ of $I$ with boundary length at least $\delta^q$ and such that $\gamma \cap J = \emptyset$.  This implies that there exists a segment of $\gamma$ which connects two points $a',b' \in I$ with boundary length distance at least $\delta^q$ and this segment stays in either $\CN_1$ or $\CN_2$.  Then,  condition~\eqref{it:reverse_holder_continuity} of $E$ implies that the length of the aforementioned segment with respect to the interior-internal metric in $\overline{\CN}_1 \cup \overline{\CN}_2$ is at least $\delta^{qM}$.  Therefore, combining everything and possibly taking $q$ to be larger (depending only on $u,\wt{u}$ and $M$),  we have that the distance between $a$ and $b$ with respect to the interior-internal metric in $\overline{\CN}_1 \cup \overline{\CN}_2$ is at least $3\delta^{qM}$.

See Figure~\ref{fig:intersecting_chunks} for an illustration of the notation for what follows. Now,  let $\CN_1,\ldots,\CN_n$ be the chunks intersecting the inner boundary of $\CA$ such that for all $1 \leq i \leq n-1$,  we have that $\CN_i$ is adjacent to $\CN_{i+1}$ and comes before $\CN_{i+1}$,  and set $I_i := \partial \CN_i \cap \partial \CN_{i+1}$.  Set also $\CN_0 = \CN_n$ and $I_0 = \partial \CN_n \cap \partial \CN_1$.  Let $w$ be a point lying in the inner boundary of $\CA$.  Then,  there exists $1 \leq i \leq n$ such that $w \in \partial \CN_i \setminus (I_{i-1} \cup I_i)$.  Let $v \in \partial \CN_{i-1} \cup \partial \CN_i \setminus I_{i-1}$ be such that $v$ does not lie in the inner boundary of $\CA$.  Let also $a_{i-1},b_{i-1}$ be the two endpoints of $I_{i-1}$ such that $a_{i-1}$ lies in the inner boundary of $\CA$.  Suppose that the boundary length distance of $w$ from $a_{i-1}$ is at most $\delta^{qM^2}$.  Then,  condition~\eqref{it:reverse_holder_continuity} of $E$ implies that the distance between $a_{i-1}$ and $b_{i-1}$ with respect to the interior-internal metric in $\overline{\CN}_{i-1} \cup \overline{\CN}_i$ is at most $\delta^{qM}$.  If the boundary length distance of $v$ from $b_{i-1}$ is at most $\delta^{qM^2}$,  then combining with the triangle inequality,  we obtain that the distance between $w$ and $v$ with respect to the interior-internal metric in $\overline{\CN}_{i-1} \cup \overline{\CN}_i$ is at least $\delta^{qM}$.  Suppose that the boundary length distance of $w$ from $a_{i-1}$ is at least $\delta^{qM^2}$.  Then,  condition~\eqref{it:reverse_holder_continuity} of $E$ implies that the distance between $I_{i-1}$ and $w$ with respect to the interior-internal metric in $\overline{\CN}_1 \cup \overline{\CN}_2$ is at least $\delta^{qM^3}$.  If $v \in \partial \CN_{i-1}$,  then the distance between $w$ and $v$ with respect to the metric in $\overline{\CN}_1 \cup \overline{\CN}_2$ is at least $\delta^{qM^3}$.  Suppose that $v \in \partial \CN_i$.  Note that the boundary lengths of $I_{i-1}$ and $I_i$ are both at least $\delta^{2/3} / M$ and so condition~\eqref{it:reverse_holder_continuity} implies that the distance between the component of $\partial \CN_i \setminus (I_{i-1} \cup I_i)$ intersecting the inner boundary of $\CA$ and the component of $\partial \CN_i \setminus (I_{i-1} \cup I_i)$ not intersecting the inner boundary of $\CA$ is at least $\delta^{qM}$ with respect to the metric in $\overline{\CN}_i$.  Moreover,  the way that the exploration is defined implies that the boundary length of the counterclockwsie arc of $\partial \CN_i$ from $a_{i-1}$ to $a_i$ is at least $\epsilon_0 \delta^{2/3}$ and so combining with~\eqref{it:reverse_holder_continuity},  we obtain that the distance between $a_{i-1}$ and $I_i$ with respect to the interior-internal metric in $\CN_i$ is at least $\epsilon_0^M \delta^{2M/3}$.  In particular,  the triangle inequality implies that the distance between $w$ and $I_i$ with respect to the metric in $\CN_i$ is at least $\epsilon_0^M \delta^{2M/3} / 2$.  It follows that the distance between $w$ and $v$ in $\overline{\CN}_{i-1} \cup \overline{\CN}_i$ is at least $\delta^{qM^3}$.  Note that if $w$ has boundary length distance at least $\delta^{qM}$ from both of $I_{i-1}$ and $I_i$,  then~\eqref{it:reverse_holder_continuity} implies that the $\qdistnoarg{h}$-distance in $\CN_i$ between $w$ and the part of $\partial \CN_i$ not contained in the inner boundary of $\CA$ is at least $\delta^{qM^3}$.  Combining everything,  we obtain that the following is true.  If $w$ lies in the part of $\partial \CN_i$ contained in the inner boundary of $\CA$ with boundary length distance from $I_i$ at least $\delta^{qM}$,  and $v$ lies in the part of $\partial (\overline{\CN}_{i-1} \cup \overline{\CN}_i \cup \overline{\CN}_{i+1})$ not contained in the inner boundary of $\CA$,  then the distance between $w$ and $v$ with respect to the metric in $\overline{\CN}_{i-1} \cup \overline{\CN}_i \cup \overline{\CN}_{i+1}$ is at least $\delta^{qM^3}$.  A similar argument shows that if $w$ and $v$ are as above but $w$ has boundary length distance from $I_{i-1}$ at least $\delta^{qM}$ instead,  then the distance between $w$ and $v$ in $\overline{\CN}_{i-1} \cup \overline{\CN}_i \cup \overline{\CN}_{i+1}$ is at least $\delta^{qM^3}$.  It follows that if $w$ lies in the part of $\partial \CN_i$ contained in the inner boundary of $\CA$ and $v$ lies in the part of $\partial (\overline{\CN}_{i-1} \cup \overline{\CN}_i \cup \overline{\CN}_{i+1})$ not contained in the inner boundary of $\CA$,  then we have that the distance in $\overline{\CN}_{i-1} \cup \overline{\CN}_i \cup \overline{\CN}_{i+1}$ between $w$ and $v$ is at least $\delta^{qM^3}$.

To finish the proof,  we let $\partial^{\text{in}} \CA$ (resp.  $\partial^{\text{out}} \CA$) be the inner (resp.  outer) boundary of $\CA$.  Fix $w\in \partial^{\text{in}}\CA,  v \in \partial^{\text{out}} \CA$ and let $\gamma$ be a geodesic in $\CA$ from $w$ to $v$ with respect to the interior-internal metric in $\CA$.  Then,  there exists $1 \leq i \leq n$ such that $w \in \partial \CN_i$,  and let $\wt{v}$ be the last point of $\partial (\overline{\CN}_{i-1} \cup \overline{\CN}_i \cup \overline{\CN}_{i+1}) \setminus \partial^{\text{in}} \CA$ hit by $\gamma$.  Then,  the results of the previous paragraph imply that the $\qdistnoarg{h}$-length of the part of $\gamma$ from the last time that it hits $\wt{v}$ to the first time that it hits $w$ is at least $\delta^{qM^3}$ and so the distance between $w$ and $v$ with respect to the interior-internal metric in $\CA$ is at least $\delta^{qM^3}$.  This completes the proof of the lower bound.
\end{proof}

We now work towards proving the upper bound in Proposition~\ref{prop:good_band_contains}.  First,  we state the following version of Lemma~\ref{lem:boundary_length_regularity}.
\begin{lemma}\label{lem:boundary_length_weak_regulatiry}
Fix $\ell>0,  u \in (0,1/2)$ and let $\CD = (\D,h,0)$ be a sample from $\qdiskweighted{\ell}$.  Then,  there a.s.\  exists $C>0$ such that for all $x,y \in \partial \D$ we have that $\qdist{h}{x}{y} \leq C \qbmeasure{h}([x,y]_{\partial \CD}^\ccw)^{1/2-u}$
and the same is true with $[x,y]_{\partial \CD}^\cw$ in place of $[x,y]_{\partial \CD}^\ccw$.  Moreover,  there exist constants $c_1,c_2>0$ depending only on $\ell$ and a constant $\alpha>0$ depending only on $u$ such that if $C>0$ is the smallest constant for which the above bound holds,  then we have that $\p\big[ C>A \big] \leq c_1 \exp(-c_2 A^{\alpha})$ for all $A>1$.
\end{lemma}
\begin{proof}
It follows from an argument which is similar to the one given in the proof of Lemma~\ref{lem:boundary_length_regularity}.
\end{proof}

Next,  we state and prove the following lemma which gives an upper bound with high probability on the diameter of the outer boundary of a radial $\SLE_6$ curve drawn for at most $\delta$ units of quantum natural time on top of a weighted quantum disk.
\begin{lemma}\label{lem:chunk_boundary_diameter}
Fix $u \in (0,1/3)$.  Then,  there exists $\wt{u} \in (0,u)$ such that the following holds.  Suppose that $\CD = (\D,h,0)$ has law $\qdiskweighted{\ell}$ with $\ell \leq \delta^{-\wt{u}/2}$.  Let $\eta'$ be a radial $\SLE_6$ curve in $\D$,  independent of $\CD$,  starting from a point chosen uniformly at random on $\partial \D$ with respect to $\nu_h$ and targeted at $0$.  Fix $\delta \in (0,1)$ and suppose that $\sigma$ is a stopping time for $\eta'$ which is a.s.\  at most $\delta$.  Then,  there exist constants $c_1,c_2,\alpha>0$ depending only on $u$ such that on the event that the boundary length of the component of $\D \setminus \eta'([0,t])$ containing $0$ is at least $\delta^{2/3 -u}$ for all $t \in [0,\sigma]$,  we have that off an event with probability at most $c_1 \exp(-c_2 \delta^{-\alpha})$,  the $\qdistnoarg{h}$-diameter of the outer boundary of $\eta'([0,\sigma])$ is at most $\delta^{1/3 - u}$.
\end{lemma}
\begin{proof}
Fix $\wt{u} \in (0,1/3)$ small (to be chosen and depending only on $u$).  Suppose that $x_1 \in \partial \D$ is chosen uniformly at random from the quantum boundary measure and $y_1 \in \partial \D$ is the unique point on $\partial \D$ such that the boundary lengths of $[x_1,y_1]_{\partial \D}^\cw$ and $[x_1,y_1]_{\partial \D}^\ccw$ are both equal to $\ell / 2$.  We let $\sigma_1$ be the first time that $\eta'$ disconnects $y_1$ from $0$.  Suppose that we have defined marked points $x_1,\ldots,x_j,y_1,\ldots,y_j$ and stopping times $\sigma_1,\ldots,\sigma_j$.  Then,  we set $x_{j+1} = \eta'(\sigma_j)$ and let $y_{j+1}$ be the point on the boundary of the component $D_j$ of $\D \setminus \eta'([0,\sigma_j])$ containing $0$ which is antipodal to $x_{j+1}$ with respect to quantum boundary length.  Then,  we let $\sigma_{j+1}$ be the first time that $\eta'$ disconnects $y_{j+1}$ from $0$.  We set $N:=\inf\{j \in \N : \sigma \leq \sigma_j\}$.  Suppose that the boundary length $\ell_j$ of $\partial D_j$ is at least $\delta^{2/3 - u}$.  We note that when scaling the boundary length by $\ell_j^{-1}$,  we have that the quantum natural time is scaled by $\ell_j^{-3/2}$ and $\ell_j^{3/2} \geq \delta^{1-3u/2} > \delta$.  Hence,  it follows that there exists a universal constant $p \in (0,1)$ such that $\p\big[ \sigma_{j+1}-\sigma_j > \delta \giv \eta'|_{[0,\sigma_j]} \big] \geq 1-p$ a.s.  It follows that on the event that the boundary length of $\D \setminus \eta'([0,t])$ is at least $\delta^{2/3 - u}$ for all $t \in [0,\sigma]$,  we have that $N < \delta^{-\wt{u}}$ off an event with probability at most $\exp(\log(p) \delta^{-\wt{u}})$.  From now on,  we assume that we are working on the event that $N < \delta^{-\wt{u}}$.  

Combining Lemma~\ref{lem:boundary_length_rn_disk_weighted} with \cite[Chapter~VII,  Corollary~2]{bertoin1996levy} gives that on the event that the boundary length of the $0$-containing connected component of $\D \setminus \eta'([0,t])$ is at least $\delta^{2/3 - u}$ for all $t \in [0,\sigma]$,  we have that off an event with probability at most $\lesssim \delta^{-1/3 +u/2} \exp(-\delta^{-2/3 - \wt{u}})$ (with the implicit constant being universal),  the boundary length of the $0$-containing connected component of $\D \setminus \eta'([0,t])$ is at most $\delta^{-\wt{u}}$ for all $t \in [0,\sigma]$.  Suppose that this event occurs as well.  Moreover,  Lemma~\ref{lem:top_length_tail_disk} implies that there exist universal constants $c_1,c_2>0$ such that off an event with probability at most $c_1 \exp(-c_2 \delta^{-\wt{u}})$,  we have that for all $1 \leq j \leq N$,  the boundary length of the part of $\eta'|_{[\sigma_{j-1},\sigma_j]}$ contained in $D_{j-1}$ is at most $\delta^{2/3 - \wt{u}}$,  where $\sigma_0 = 0$.  Suppose that this event occurs as well.  Furthermore,  Lemma~\ref{lem:boundary_length_weak_regulatiry} implies that there exists a constant $\alpha > 0$ depending only on $\wt{u}$ such that by possibly taking $c_1$ to be larger and $c_2$ to be smaller (depending only on $\wt{u}$),  we have that for all $j \in \N$,  off an event with probability at most $c_1 \exp(-c_2 \delta^{-\alpha})$,  
$\ell_j^{-1/2} \qdist{h}{x}{y} \leq \delta^{-\wt{u}} (\ell_j^{-1} \qbmeasure{h}([x,y]_{\partial D_j}^\ccw))^{1/2 - \wt{u}}$
for all $x,y \in \partial D_j$ and the same is true with $[x,y]_{\partial D_j}^\cw$ in place of $[x,y]_{\partial D_j}^\ccw$.  Therefore,  combining everything,  we obtain that the following is true.  There exist constants $c_3,c_4,\alpha>0$ depending only on $\wt{u}$ and $u$ such that on the event that the boundary length of the $0$-containing connected component of $\D \setminus \eta'([0,t])$ is at least $\delta^{2/3 - u}$ for all $t \in [0,\sigma]$,  off an event with probability at most $c_3 \exp(-c_4 \delta^{-\alpha})$,  it holds that $N < \delta^{-\wt{u}},  \ell_j \leq \delta^{-\wt{u}}$,  the boundary length of the part of $\eta'|_{[\sigma_j,\sigma_{j+1}]}$ contained in $D_j$ is at most $\delta^{2/3 - \wt{u}}$,  and $\ell_j^{-1/2} \qdist{h}{x}{y} \leq \delta^{-\wt{u}} (\ell_j^{-1} \qbmeasure{h}([x,y]_{\partial D_j}^\ccw))^{1/2 - \wt{u}}$
for all $x,y \in \partial D_j$ and the same is true with $[x,y]_{\partial D_j}^\cw$ in place of $[x,y]_{\partial D_j}^\ccw$,  for all $1 \leq j \leq N$.  In particular,  we have that the $\qdistnoarg{h}$-diameter of the part of $\eta'|_{[\sigma_j,\sigma_{j+1}]}$ contained in $D_j$ is at most $\delta^{-\wt{u}} \ell_j^{\wt{u}} \delta^{(1/2 - \wt{u}) ( 2/3 - \wt{u})} \leq \delta^{-\wt{u}-\wt{u}^2} \delta^{(1/2 - \wt{u}) (2/3 - \wt{u})}$.  It follows by a union bound that the $\qdistnoarg{h}$-diameter of the outer boundary of $\eta'([0,\sigma])$ is at most $\delta^{-2\wt{u}-\wt{u}^2} \delta^{(1/2 - \wt{u})(2/3 - \wt{u})}$.  So the proof is complete if we choose $\wt{u}>0$ sufficiently small (depending only on $u$).
\end{proof}

\begin{proof}[Proof of Proposition~\ref{prop:good_band_contains}, upper bound]
Let $\wt{u} \in (0,u/3)$ be as in Lemma~\ref{lem:chunk_boundary_diameter},  fix $u_0 \in (0,1)$ (to be chosen),  and let $\wt{E}_{\wt{u},\delta}$ be the event that for all $j \in \Z \cap [0,\delta^{-2/3 - u}],  t \in [0,\sigma_j]$,  the boundary length of the $0$-containing connected component of $D_j \setminus \eta_l'([0,t])$ is at most $\delta^{-\wt{u}}$.  Then,  combining Proposition~\ref{prop:good_chunks_percolate} with Lemma~\ref{lem:chunk_boundary_diameter},  we obtain that there exist constants $c_1,c_2,\alpha>0$ depending only on $\ell$ and $u$ such that for all $\delta \in (0,1)$,  the following holds.  On the event $E_{u,\delta} \cap \wt{E}_{\wt{u},\delta}$,  off an event with probability at most $c_1 \exp(-c_2 \delta^{-\alpha})$,  the following conditions hold.  The set $\CA$ consists of at most $\delta^{-2/3 - u}$ number of chunks and disconnects $0$ from $\partial \D$,  and for every chunk $\CN$ in $\CA$,  there exists $1 \leq n \leq \delta^{-u_0}$ and chunks $\CN_1,\ldots,\CN_n$ discovered during the first $\delta^{-2/3 - u}$ number of steps in the exploration such that $\partial \CN_n \cap \partial \D \neq \emptyset,  \CN = \CN_1$,  and $\partial \CN_j \cap \partial \CN_{j+1} \neq \emptyset$,  for all $1 \leq j \leq n-1$.  Moreover,  if $\CN$ is a chunk discovered during the first $\delta^{-2/3-u}$ steps of the exploration and $\eta'$ is the corresponding radial $\SLE_6$ curve which makes $\CN$,  then we have that the $\qdistnoarg{h}$-diameter of the outer boundary of $\eta'$ in $\CN$ is at most $\delta^{1/3 - u}$.  Suppose that the above hold.  Fix $z \in \CA$ and let $\CN$ be the chunk in $\CA$ such that $z \in \CA$.  Let also $\CN_1,\ldots,\CN_n$ be the chunks in $\CA$ such that $\CN = \CN_1,  \partial \CN_n \cap \partial \D \neq \emptyset$,  and $\partial \CN_j \cap \partial \CN_{j+1} \neq \emptyset$,  for all $1 \leq j \leq n-1$.  Note that condition~\eqref{it:diameter_bound} of $E$ implies that the $\qdistnoarg{h}$-distance of $z$ from $\partial \CA$ is at most $M \delta^{1/3}$.  Since $\partial \CN_j \cap \partial \CN_{j+1}$ consists only of points lying in the outer boundary of both $\eta_j'([0,\sigma_j])$ and $\eta_{j+1}'([0,\sigma_{j+1}])$,  and $\partial \CN_n \cap \partial \D \neq \emptyset$,  a union bound gives that the $\qdistnoarg{h}$-distance of $z$ from $\partial \D$ is at most $M \delta^{1/3} + \delta^{1/3 - u -u_0}$,  where we choose $u_0>0$ such that $u_0 < 1/3 - u$.  Furthermore,  the proof of Proposition~\ref{prop:good_chunks_percolate} implies that by possibly taking $c_1$ to be larger and $c_2,\alpha>0$ to be smaller (depending only on $\ell,u$),  we can assume that $\p\big[ \wt{E}_{\wt{u},\delta} \big] \leq c_1 \exp(-c_2 \delta^{-\alpha})$.  Combining everything,  we complete the proof.
\end{proof}

\subsection{$\SLE_6$ hull cannot be too skinny}
\label{subsec:sle6_not_too_skinng}

The main purpose of this subsection is to prove Lemma~\ref{lem:sle_6_cannot_be_skinny} which roughly states that with very high probability,  we have that whenever the whole-plane $\SLE_6$ drawn on top of an independent sample from $\qsphereinflaw$ and parameterized by quantum natural time travels quantum distance at least $\delta^{1/3} \log(\delta^{-1})^{\kappa}$,  then at least $\delta$ units of quantum natural time have elapsed.  The main component of the proof of Lemma~\ref{lem:sle_6_cannot_be_skinny} is Lemma~\ref{lem:sle6_chunk_diameter_upd} which is similar to Lemma~\ref{lem:chunk_boundary_diameter} except that it provides us with a better bound on the diameter of the outer boundary of the radial $\SLE_6$ curve of the form $\delta^{1/3} \log(\delta^{-1})^{\kappa}$.  We will first prove Lemma~\ref{lem:sle_6_cannot_be_skinny} using Lemma~\ref{lem:sle6_chunk_diameter_upd} and then prove Lemma~\ref{lem:sle6_chunk_diameter_upd}.

\begin{lemma}
\label{lem:sle_6_cannot_be_skinny}
There exists $\kappa > 0$ so that the following is true.  Suppose that $(\CS,x,y)$ has distribution $\qsphereinflaw$ and $\eta'$ is an independent whole-plane $\SLE_6$ from $x$ to $y$ which is parameterized by quantum natural time.  Suppose that $\delta > 0$.  The $\qsphereinflaw$ measure of the event that there exists $k \in \N$ so that the $\qdistnoarg{h}$-diameter of the outer boundary of $\eta'([(k-1) \delta,k \delta])$ is at least $\delta^{1/3}(\log \delta^{-1})^\kappa$ decays to $0$ as $\delta \to 0$ faster than any power of $\delta$.  Moreover, the $\qsphereinflaw$ measure of the event that $\eta'$ travels $\qdistnoarg{h}$-distance $\delta^{1/3}$ without disconnecting at least $\delta^{4/3}(\log \delta^{-1})^{-\kappa}$ units of quantum mass from $y$ decays to $0$ as $\delta \to 0$ faster than any power of $\delta$.
\end{lemma}

\begin{lemma}\label{lem:sle6_chunk_diameter_upd}
Fix $u>0$.  Then,  there exists a constant $\kappa>0$ depending only on $u$ such that for all $\delta \in (0,1)$, $0 < \ell \leq \log(\delta^{-1})^u$,  the following holds.  Let $\CD = (\D,h,0)$ be a sample from $\qdiskweighted{\ell}$ and let $\eta'$ be a radial $\SLE_6$ on $\D$,  independent of $\CD$,  starting from a point chosen uniformly at random on $\partial \D$ according to $\qbmeasure{h}$ and targeted at $0$.  Then,  off an event with probability decaying to $0$ as $\delta \to 0$ faster than any power of $\delta$,  we have that the $\qdistnoarg{h}$-diameter of the outer boundary of $\eta'([0,\delta])$ is at most $\delta^{1/3} \log(\delta^{-1})^{\kappa}$.
\end{lemma}

\begin{proof}[Proof of Lemma~\ref{lem:sle_6_cannot_be_skinny}]
\emph{Step 1.  The boundary length process cannot be too large.} Let $\kappa>0$ be the constant of Lemma~\ref{lem:sle6_chunk_diameter_upd} with $u=2$.  First,  we will show that if $(L_t)$ is the boundary length process corresponding to $\eta'$,  then $\sup_{0 \leq t \leq T} L_t \leq \log(\delta^{-1})^2$ off an event whose $\qsphereinflaw$ measure decays to $0$ as $\delta \to 0$ faster than any power of $\delta$ (where $T$ denotes the time duration of the excursion $L$).  Indeed,  Lemma~\ref{lem:boundary_length_rn_disk_weighted} combined with \cite[Chapter~VII,  Corollary~2]{bertoin1996levy} implies that there exist universal constants $c_1,c_2>0$ such that conditional on $\tau:=\inf\{t \geq 0 : L_t = 1\} < \infty$,  the measure under $\qsphereinflaw$ of the event that $(L_{t+\tau})$ exceeds $\log(\delta^{-1})^2$ is at most $c_1 \exp(-c_2 \log(\delta^{-1})^2)$.  Hence,  since $\qsphereinflaw\big[ \tau<\infty \big] \in (0,\infty)$,  we obtain that possibly by taking $c_1$ to be larger and $c_2$ to be smaller,  $\qsphereinflaw\big[ \sup_{0 \leq t \leq T} L_t > \log(\delta^{-1})^2 \big] \leq c_1 \exp(-c_2 \log(\delta^{-1})^2)$ for all $\delta \in (0,1)$.

\emph{Step 2.  Proof of the first claim of the lemma.} We will start by proving the $k=1$ case first.  For each $t \geq 0$,  we let $U_t$ be the connected component of $\CS \setminus \eta'([0,t])$ containing $y$ and set $\wt{U}_t = \CS \setminus \overline{U}_t$.  Let $\wt{\tau}_{\delta}$ be the first time $t$ that the distance between $x$ and $\partial \wt{U}_t$ with respect to the interior-internal metric in $\wt{U}_t$ is at least $\delta$.  Note that the interior-internal metric in $\wt{U}_t$ is determined by the restriction to $\wt{U}_t$ of the field $h$ generating $(\CS,x,y)$.  It follows from the proof of \cite[Theorem~1.2]{ms2015spheres} that on the event $\{T > t\}$,  the path-decorated quantum surface $(\wt{U}_t,h|_{\wt{U}_t},\eta'|_{[0,t]})$ is a.s.\  determined by the ordered sequence of oriented marked components cut out by $\eta'|_{[0,t]}$ from $y$ viewed as quantum surfaces.  It follows that $\wt{\tau}_{\delta}$ is a stopping time with respect to the filtration $(\CF_t)$,  where $\CF_t$ is the $\sigma$-algebra generated by the collection of quantum disks that $\eta'|_{[0,t]}$ has separated from $y$,  each marked by the last point on their boundary visited by $\eta'$ and oriented by the direction in which $\eta'$ has traced their boundary.  Set $\tau_{\delta} = \wt{\tau}_{\delta} \wedge \delta$.  Then,  \cite[Proposition~6.4]{ms2015spheres} implies that conditional on $\CF_{\tau_{\delta}}$,  the conditional law of the surface parameterized by $U_{\tau_{\delta}}$ is that of $\qdiskweighted{L_{\tau_{\delta}}}$.  If $\tau_{\delta} = \delta$,  this implies that the $\qdistnoarg{h}$-diameter of $\eta'([0,\delta])$ is at most $\delta^{1/3}$ and so this handles the case that $k = 1$.  If $\tau_{\delta} < \delta$,  then we can apply Lemma~\ref{lem:sle6_chunk_diameter_upd} to obtain that conditional on $\tau_{\delta} < \delta$,  the $\qdistnoarg{h}$-diameter of the outer boundary of $\eta'([\tau_{\delta},\delta])$ is at most $\delta^{1/3} \log(\delta^{-1})^{\kappa}$ off an event whose $\qsphereinflaw$ measure tends to $0$ as $\delta \to 0$ faster than any power of $\delta$.  Hence,  by possibly taking $\kappa$ to be larger,  we can assume that the $\qdistnoarg{h}$-diameter of the outer boundary of $\eta'([0,\delta])$ is at most $\delta^{1/3} \log(\delta^{-1})^{\kappa}$ off an event whose $\qsphereinflaw$ measure tends to $0$ as $\delta \to 0$ faster than any power of $\delta$.  For general $k \in \N$,  we note that conditional on $\CF_{k\delta}$ and on the event that $T \geq k \delta$,  the law of the surface parameterized by $U_{k \delta}$ is that of $\qdiskweighted{L_{k\delta}}$,  and so by applying again Lemma~\ref{lem:sle6_chunk_diameter_upd},  we obtain that the $\qdistnoarg{h}$-diameter of the outer boundary of $\eta'([(k-1)\delta ,  k \delta])$ is at most $\delta^{1/3} \log(\delta^{-1})^{\kappa}$ off an event whose $\qsphereinflaw$ measure tends to $0$ as $\delta \to 0$ faster than any power of $\delta$.  Since $\qsphereinflaw\big[ T \geq k \delta \big]$ is at most a constant times $\delta^{-2/3} k^{-2/3}$,  (recall the discussion at the end of Subsection~\ref{subsec:sle_explorations}), the first assertion follows by taking a union bound over $k \in \N$.

\emph{Step 3.  Proof of the second claim of the lemma.} We now turn to the second assertion of the lemma.  The time-reversal of $L_t$ is a $3/2$-stable L\'evy excursion with only upward jumps.  This implies that for each $k \in \N, a>0$,  conditional on $\{T \geq k \delta\}$,   the number of jumps in $[k\delta/3 , (k+1)\delta/3]$ with size at least $a$ has the Poisson distribution with mean given by a constant times $(\delta/3) \int_a^\infty s^{-5/2} ds$ which, in turn, is equal to a constant times $\delta a^{-3/2}$.  By~\eqref{eqn:poisson_below_mean}, \eqref{eqn:poisson_above_mean}, the conditional probability given $\{T \geq k\delta\}$ that the time-reversal of $L$ makes fewer than $(\log \delta^{-1})^{\kappa}$ jumps in $[k\delta / 3 , (k+1)\delta / 3]$ of size at least $\delta^{2/3}(\log \delta^{-1})^{-\kappa}$ decays as $\delta \to 0$ faster than any power of $\delta$. By applying a union bound over integer multiples of $\delta/3$, we see that this holds for all such multiples of $\delta/3$ simultaneously off an event whose $\qsphereinflaw$ measure tends to $0$ as $\delta \to 0$ faster than any power of $\delta$.
 
We now consider $L_t$ in the forward time direction again.  Then each interval of length $\delta$ contains at least one interval (for the time-reversal of $L_t$) whose endpoints are an integer multiple of $\delta/3$.  Therefore the $\qsphereinflaw$ measure of the event that there exists an interval of length $\delta$ in which $L$ makes fewer than $(\log \delta^{-1})^\kappa$ downward jumps of size at least $\delta^{2/3}(\log \delta^{-1})^{-\kappa}$ decays to $0$ as $\delta \to 0$ faster than any power of $\delta$.  Each downward jump of $L_t$ corresponds to a quantum disk whose boundary length is given by the size of the downward jump.  Moreover, these quantum disks are conditionally independent given $L$.  The probability that a quantum disk with boundary length at least $\delta^{2/3}(\log \delta^{-1})^{-\kappa}$ has area at least $\delta^{4/3}(\log \delta^{-1})^{-2\kappa}$ is positive uniformly in $\delta$.  Therefore by binomial concentration, the probability that fewer than a fixed fraction of these disks have area at least $\delta^{4/3}(\log \delta^{-1})^{-2\kappa}$ decays to $0$ as $\delta \to 0$ faster than any power of $\delta$. Combining these observations with scaling and the first assertion of the lemma implies the second assertion (up to a redefinition of $\kappa$). 
\end{proof}

Next,  it remains to prove Lemma~\ref{lem:sle6_chunk_diameter_upd}.  First,  we will state and prove the following lemma which gives an upper bound on the diameter of a radial $\SLE_6$ chunk with high probability.

\begin{lemma}\label{lem:quantum_disk_diameter_upd}
Fix $a>0,L_0 < \infty$.  Then,  there exists a constant $\kappa>0$ depending only on $a,L_0$ such that for all $\ell \in (0,L_0)$,  the following is true.  Let $\CD=(\D,h)$ be a sample from $\qdiskweighted{\ell}$.  Then,  on the event that the quantum area of $\CD$ is at most $a$,  off an event with probability decaying to $0$ as $\delta \to 0$ faster than any power of $\delta$,  we have that the $\qdistnoarg{h}$-diameter of $\CD$ is at most $\log(\delta^{-1})^{\kappa}$.
\end{lemma}
\begin{proof}
Fix $\kappa>0$ sufficiently large (to be chosen).  Then,  \cite[Lemma~4.25]{ms2015mapmaking} combined with H\"older's inequality to compare the laws $\qdiskweighted{\ell}$ and $\qdisk{\ell}$ imply that there exists a constant $c \in [1,\infty)$ depending only on $a,L_0$ such that if $d^* = \sup_{z \in \D} \dist_{d_h}(z,\partial \D)$,  then on the event that $\qmeasure{h}(\D) \leq a$,  we have that $d^* \leq \log(\delta^{-1})^{\kappa}$ off an event with probability at most $c \exp(-c^{-1} \log(\delta^{-1})^{4\kappa / 3})$.  Fix $u \in (0,1/2)$ and consider the field $\wt{h}$ obtained by scaling lengths by $\ell$,  distances by $\ell^{1/2}$ and areas by $\ell^2$.  Then,  Lemma~\ref{lem:boundary_length_weak_regulatiry} combined with scaling imply that there exist universal constants $c_1,c_2$ and a constant $\beta>0$ depending only on $u$ such that off an event with probability at most $c_1 \exp(-c_2 \log(\delta^{-1})^{\kappa \beta})$,  we have that $\qdist{h}{x}{y} \leq \log(\delta^{-1})^{\kappa} \ell^{1+u} \leq L_0^{1+u} \log(\delta^{-1})^{\kappa}$ for all $x ,  y \in \partial \D$.  Thus,  if we choose $\kappa$ such that $\kappa > \max\{\beta^{-1}, 1\}$,  then on the event that $\qmeasure{h}(\D) \leq a$,  off an event with probability decaying to $0$ as $\delta \to 0$ faster than any power of $\delta$,  we have that $\qdist{h}{x}{y} \leq L_0^{1+u} \log(\delta^{-1})^{\kappa}$ for all $x,y \in \partial \D$ and $\dist_{\qdistnoarg{h}}(z,\partial \D) \leq \log(\delta^{-1})^{\kappa}$ for all $z \in \D$,  which implies that $\qdist{h}{x}{y} \leq 2 \log(\delta^{-1})^\kappa + L_0^{1+u} \log(\delta^{-1})^\kappa$ for all $z,w \in \D$.  This completes the proof.
\end{proof}

Next,  we mention the following useful result whose proof is essentially the same with the proof of Lemma~\ref{lem:boundary_length_regularity}.

\begin{lemma}\label{lem:quantum_boundary_length_partition}
There exist universal constants $c_1,c_2>0$ such that the following is true.  Fix $\ell>0,\zeta \in (0,1)$ and let $\CD = (\D,h,0)$ be a sample from $\qdiskweighted{\ell}$.  Then,  for all $k,n \in \N$ such that $n\ell^{-1} \geq 1,  k \geq e$,  off an event with probability at most $c_1 n^{1/2} \exp(-c_2 \log(\log(k)) \log(k))$,  the following is true.  There exist points $x_1, \ldots,  x_N$ on $\partial \D$ such that the intervals $[x_j,x_{j+1}]_{\partial \D}^\ccw$ form a partition of $\partial \D$ and the boundary length of $[x_j,x_{j+1}]_{\partial \D}^\ccw$ is given by $\ell n^{-1/2}$ and the $\qdistnoarg{h}$-diameter of $[x_j,x_{j+1}]_{\partial \D}^\ccw$ is at most $8 \ell^{1/2} n^{-1/4} \log(k)^{7/4 + \zeta}$ for all $1 \leq j \leq N-1$.
\end{lemma}
\begin{proof}
It follows from the argument used to prove Lemma~\ref{lem:boundary_length_regularity}.
\end{proof}

\begin{proof}[Proof of Lemma~\ref{lem:sle6_chunk_diameter_upd}]
Fix $u>0$.  We define marked points $(x_j),(y_j)$ and stopping times $(\sigma_j)$ as in the proof of Lemma~\ref{lem:chunk_boundary_diameter}.  We let $N_{\delta}$ be the first $j \in \N$ such that either $\sigma_j - \sigma_{j-1} \geq \delta$ or $\ell_j < \delta^{2/3}$ and $\qmeasure{h}(\D) < M \delta^{4/3}$,  where $\ell_j$ denotes the boundary length of the $0$-containing connected component of $\D \setminus \eta'([0,\sigma_j])$,  and $M>1$ is a universal constant (to be chosen).  We claim that there exists a universal constant $p \in (0,1)$ such that conditional on $\eta'|_{[0,\sigma_{j-1}]}$,  a.s.\  we have that either $\sigma_j - \sigma_{j-1} \geq \delta$ or $\ell_j < \delta^{2/3}$ and $\qmeasure{h}(\D) < M \delta^{4/3}$,  with probability at least $p$.   Indeed,  suppose that we are working on the event that $\ell_j \geq \delta^{2/3}$.  Note that by arguing as in the proof of Lemma~\ref{lem:chunk_boundary_diameter},  we obtain that there exists a universal constant $p \in (0,1)$ such that $\p\big[ \sigma_j - \sigma_{j-1} \geq \delta \giv \eta'|_{[0,\sigma_{j-1}]} \big] \geq p$ a.s.\  on $\{\ell_j \geq \delta^{2/3} / 2\}$.  Suppose that we are working on the event $\ell_j < \delta^{2/3}$.  By possibly decreasing $p$,  taking $M$ sufficiently large and applying scaling,  we have that $\p\big[ \qbmeasure{h}(\partial D_j) < \ell_j,  \qmeasure{h}(D_j) < M \ell_j^2 \giv \eta'|_{[0,\sigma_{j-1}]}\big] \geq p$.  This proves the claim.  It follows that $N_{\delta} < \log(\delta^{-1})^2$ off an event whose probability decays to $0$ as $\delta \to 0$ faster than any power of $\delta$.  The proof of Lemma~\ref{lem:chunk_boundary_diameter} implies that off an event whose probability decays to $0$ as $\delta \to 0$ faster than any power of $\delta$,  we have that the boundary length of the $0$-containing connected component of $\D \setminus \eta'([0,t])$ is at most $2 \log(\delta^{-1})^2$ for all $t \in [0,\delta]$,  and the boundary length of the part of $\eta'|_{[\sigma_{j-1},\sigma_j]}$ contained in $D_{j-1}$ is at most $\delta^{2/3} \log(\delta^{-1})^2$ for all $1 \leq j \leq \log(\delta^{-1})^2$. 

Let $\kappa>0$ be the constant of Lemma~\ref{lem:quantum_disk_diameter_upd} with $L_0 = 1$ and $a = M$.  Then,  Lemma~\ref{lem:quantum_disk_diameter_upd} combined with scaling imply that a.s.\  on the event that $\sigma_{N_{\delta}} - \sigma_{N_{\delta}-1} < \delta$,  we have off an event whose probability decays to $0$ as $\delta \to 0$ faster than any power of $\delta$ that the $\qdistnoarg{h}$-diameter of $D_{N_{\delta}}$ is at most $\delta^{1/3} \log(\delta^{-1})^{\kappa}$.  Furthermore,  Lemma~\ref{lem:quantum_boundary_length_partition} applied with $n = \lfloor \delta^{-4/3} \ell_j \rfloor,  \zeta = 1$ and $k = \lfloor\delta^{-1} \rfloor$ implies that for $j \in \N$ fixed,  off an event whose probability tends to $0$ as $\delta \to 0$ faster than any power of $\delta$,  we have that there exists a partition $[x_i,x_{i+1})_{\partial D_j}^{\ccw},  1 \leq i \leq N_j-1$ of $\partial D_j$ such that the boundary length of $[x_i,x_{i+1})_{\partial D_j}^{\ccw}$ is given by $\ell_j n^{-1/2}$ and its $\qdistnoarg{h}$-diameter is at most $8 \ell_j^{1/2} n^{-1/4} \log(\delta^{-1})^{9/4}$ for all $1 \leq i \leq N_j-1$.  Let $\wt{x}_j$ (resp.  $\wt{y}_j$) be the leftmost (resp.  rightmost) point of $\eta_j'([\sigma_{j-1},\sigma_j]) \cap \partial D_j$ and let $1 \leq i \leq N_j-1$ be such that $\wt{x}_j \in [x_i,x_{i+1})_{\partial D_j}^{\ccw}$.  Note that by the end of the previous paragraph,  we have that the boundary length of the counterclockwise arc of $\partial D_j$ from $\wt{x}_j$ to $\wt{y}_j$ is at most $\delta^{2/3} \log(\delta^{-1})^2$ off an event whose probability tends to zero as $\delta \to 0$ faster than any power of $\delta$. Then,  since $\ell_j n^{-1/2} \geq \delta^{2/3}$,  it follows that at most $\log(\delta^{-1})^2$ intervals of the form $[x_m,x_{m+1})_{\partial \D}^{\ccw}$ are needed to cover the counterclockwise arc of $\partial D_j$ from $\wt{x}_j$ to $\wt{y}_j$,  and so the latter has $\qdistnoarg{h}$-diameter at most $\lesssim \delta^{1/3} \log(\delta^{-1})^{17/4}$,  where the implicit constant is universal.  Combining everything,  we obtain that there exists a universal constant $c>0$ such that off an event whose probability tends to $0$ as $\delta \to 0$ faster than any power of $\delta$,  the following hold.
\begin{enumerate}[(i)]
\item $N_{\delta} < \log(\delta^{-1})^2$.
\item The $\qdistnoarg{h}$-diameter of $D_{N_{\delta}}$ is at most $\delta^{1/3} \log(\delta^{-1})^{\kappa}$.
\item The $\qdistnoarg{h}$-diameter of the part of $\eta'|_{[\sigma_{j-1},\sigma_j]}$ contained in $D_j$ is at most $c \delta^{1/3} \log(\delta^{-1})^2$ for all $1 \leq j \leq \log(\delta^{-1})^2$.
\end{enumerate}
Then,  the proof of the lemma is complete by taking a union bound and possibly taking $\kappa$ to be larger.
\end{proof}

\subsection{Proof of the exit time lower bound}
\label{subsec:lbd_proof}

Now we focus on proving the lower bound of Theorem~\ref{thm:bm_exit}. The main ingredient of the proof of the lower bound is Lemma~\ref{lem:filled_metric_ball_exit_time} which roughly states that if $(\mathcal{S},h,x,y)$ is a sample from $\qsphereinflaw$ and we truncate on an event whose complement has small $\qsphereinflaw$ measure,  then the desired lower bound holds for the metric balls centered at the quantum typical point $x$.

The main idea behind the proof of Lemma~\ref{lem:filled_metric_ball_exit_time} is that Propositions~\ref{prop:good_chunks_percolate} and~\ref{prop:good_band_contains} allow us to construct sufficiently many annuli $\mathcal{A}$ centered at $x$ which have the following property.  If we start a complex Brownian motion $B$ from $x$ which is independent from $(\mathcal{S},h,x,y)$,  then $B$ has to disconnect from $\infty$ a sufficiently large amount of quantum area while making a crossing of such an annulus $\mathcal{A}$ between its inner and outer boundaries.

The proof of the lower bound in Theorem~\ref{thm:bm_exit} will be complete by combining with the fact that off an event whose $\qsphereinflaw$ measure tends to $0$ as $r \to 0$ faster than any power of $r$,  we can cover $\mathcal{S}$ by at most $r^{-A}$ many metric balls of radius $r$ and centered at quantum typical points,  for some finite and deterministic constant $A$.

\begin{lemma}
\label{lem:filled_metric_ball_exit_time}
There exists a constant $\kappa>0$ such that the following is true.  Fix $u \in (0,1/3),r_0>0$ and suppose that we are working on the event that $r^4 \log(r^{-1})^{-6-u} \leq \qmeasure{h}(\qball{h}{z}{r})) \leq r^4 \log(r^{-1})^{8+u}$,  $\frac{\qmeasure{h}(\qball{h}{z}{r}))}{\qmeasure{h}(\CS)} \geq r^{4+u}$, and $\diam(\CS) \geq 6r$ for all $z \in \CS$,  $r \in (0,r_0)$,  where $(\CS,h,x,y)$ has law $\qsphereinflaw$.  Then,  we have that
\begin{align*}
E_x\bigl[ \tau_{\qball{h}{x}{r}} \bigr] \geq r^4 \log(r^{-1})^{-\kappa}
\end{align*}
off an event whose $\qsphereinflaw$ measure tends to $0$ as $r \to 0$ faster than any power of $r$,  where the expectation is over the Brownian motion with $\CS$ fixed.
\end{lemma}

\begin{proof}
\emph{Step 1.  Setup and overview of the proof strategy.} Suppose that we have the setup of Proposition~\ref{prop:good_chunks_percolate} with the above choice of $u$ and the event $E$ defined in Subsection~\ref{subsec:event_def}.  Proposition~\ref{prop:good_sle_chunks} implies that we can choose $M>1$ sufficiently large such that the statement of Proposition~\ref{prop:good_chunks_percolate} applies.  Fix $r>0$ and suppose that the event in the statement of the lemma holds,  and set $N_r = r^{-4-u}$.  Let $(z_j)$ be a sequence of points chosen i.i.d.\  from $\qmeasure{h}$.  Then,  the proof of the lower bound of Theorem~\ref{thm:ball_concentration} implies that $\CS \subseteq \cup_{j=1}^{N_r} \qball{h}{z_j}{r}$ off an event whose $\qsphereinflaw$ measure tends to $0$ as $r \to 0$ faster than any power of $r$.  The main idea of the proof of the lemma is roughly the following.  First,  we will show that off an event whose $\qsphereinflaw$ measure tends to $0$ as $r \to 0$ faster than any power of $r$,  the following holds.  Fix $1 \leq  j \leq N_r$ such that $\qdistnoarg{h}(x,z_j) \geq r$.  Then,  we can find an annulus $\CA_j$ contained in $\qball{h}{x}{r}$ which consists of radial $\SLE_6$ chunks and disconnects $\qball{h}{x}{r/2}$ from $z_j$ as in the statement of Proposition~\ref{prop:good_band_contains}.  Moreover,  the $\qdistnoarg{h}$-distance with respect to the interior-internal metric in $\CA_j$ between the inner and outer boundaries of $\CA_j$ is at least $r \log(r^{-1})^{-p}$ for some constant $p$.  We will show this in Step 2.  Next,  in Step 3,  we will show that the following holds with high probability.  If we start a Brownian motion from $x$,  then it will disconnect at least $r \log(r^{-1})^{-q}$ units of quantum area while crossing $\CA_j$ and before exiting $\qball{h}{x}{r}$ for some constant $q>0$ and some $1 \leq j \leq N_r$.  Then,  we will conclude the proof in Step 4 arguing as in the proof of the upper bound of Theorem~\ref{thm:bm_exit}.

\emph{Step 2.  Constructing good annuli with high probability.} Let $\wt{u} \in (0,1/3)$ be the constant in Lemma~\ref{lem:chunk_boundary_diameter}.  Fix $a_1,a_2>0$ (to be chosen) and for $1 \leq j \leq N_r$,  let $\tau_1^j$ be the first time after $r/2$ that $L^j$ goes above $r^2 \log(r^{-1})^{-a_1}$,  where $L^j$ is the process describing the boundary length evolution of the metric exploration from $x$ to $z_j$.   Given that we have defined $\tau_1^j,\ldots,\tau_{m-1}^j$,  we let $\tau_m^j$ be the first time after $\tau_{m-1}^j + r \log(r^{-1})^{-a_2}$ that $L^j$ goes above $r \log(r^{-1})^{-a_1}$.  Lemma~\ref{lem:boundary_length} implies that $\sup_{0 \leq s \leq r} L_s^j < r^2 \log(r^{-1})^2$ for all $1 \leq j \leq N_r$,  off an event whose $\qsphereinflaw$ measure tends to $0$ as $r \to 0$ faster than any power of $r$.  Thus,  from now on,  we can assume that we are working on that event.  Note that there exists a universal constant $p_0 \in (0,1)$ such that if $Y$ is the time-reversal of a $3/2$-stable CSBP excursion starting from $1$,  then $Y$ hits $0$ before time $1$ with probability at least $p_0$.  It follows that the $\qsphereinflaw$ measure of the event that $\tau_1^j > \frac{r}{2} + r \log(r^{-1})^{-a_2}$ is at most $\exp(\log(1-p_0) \log(r^{-1})^{-a_2 + a_1/2})$.  Thus,  if we choose $a_1,a_2$ such that $a_2 < a_1 / 2 -1$,  we have that $\tau_1^j \leq \frac{r}{2} + r \log(r^{-1})^{-a_2}$ for all $1 \leq j \leq N_r$,  off an event whose $\qsphereinflaw$ measure tends to $0$ as $r \to 0$ faster than any power of $r$.  Similarly,  conditionally on $\tau_{m-1}^j < \infty$,  we have that $\tau_m^j - \tau_{m-1}^j \leq 2 r \log(r^{-1})^{-a_2}$ off an event whose $\qsphereinflaw$ measure tends to $0$ as $r \to 0$ faster than any power of $r$.  Fix $1<a_3<a_2$.  Then,  we have that $\tau_{\log(r^{-1})^{a_3}}^j < \infty$ off an event whose $\qsphereinflaw$ measure tends to $0$ as $r \to 0$ faster than any power of $r$.  Fix also $a_4 >1$ (to be chosen).  Let $\CD_1^j$ be the surface parameterized by $\CS \setminus \qfb{h}{z_j}{x}{\tau_1^j}$.  Suppose that we perform the exploration in $\CD_1^j$ as in Proposition~\ref{prop:good_chunks_percolate} up until either the boundary length of the $z_j$-containing component does not lie in $[ L_{\tau_1^j}^j/2,  L_{\tau_1^j}^j \log(r^{-1})^{-a_4 \wt{u}}]$ or we have discovered a chunk which cannot be connected to $\partial \CD_1^j$ by a sequence of at most $\log(r^{-1})^{a_4 u}$ number of chunks which have already been discovered or we have discovered $\log(r^{-1})^{a_4(2/3 + u)}$ number of chunks.  Let $\CN_1^{1,j},\ldots,\CN_N^{1,j}$ be the chunks that have been completely discovered up until that point.  Lemma~\ref{lem:chunk_boundary_diameter} combined with scaling imply that conditionally on $\tau_1^j < \infty$,  the probability that there exists $1 \leq i \leq N$ such that the $\qdistnoarg{h}$-diameter of the outer boundary of $\CN_i^{1,j}$ is at least $\log(r^{-1})^{-a_4 (1/3 - u)} (L_{\tau_1^j}^j)^{1/2}$ is at most $c_1 \log(r^{-1})^{a_4 (2/3 + u)} \exp(-c_2 \log(r^{-1})^{\alpha a_4})$,  where $c_1,c_2,\alpha>0$ depend only on $u$.  Thus,  if we choose $a_4 > \alpha^{-1}$,  then we have that the above probability tends to $0$ as $r \to 0$ faster than any power of $r$.  Also,  by construction,  we have that if the above event occurs,  then for all $1 \leq i \leq N$,  every point on the outer boundary of $\CN_i^{1,j}$ has $\qdistnoarg{h}$-distance from $\qball{h}{x}{\tau_1^j}$ at most $r \log(r^{-1})^{a_4 (2u - 1/3)+1} $.  In particular,  if we choose $a_4$ so that $1+(2u-1/3)a_4 < -a_2$,  we have that the outer boundary of $\CN_i^{1,j}$ is contained in $\qball{h}{x}{\tau_1^j + r \log(r^{-1})^{-a_2}}$ for all $1 \leq i \leq N$.

Now,  we perform the exploration in $\CD_1^j$ as in Proposition~\ref{prop:good_chunks_percolate} with $\delta = \log(r^{-1})^{-a_4}$ after we scale lengths by $(L_{\tau_1^j}^j)^{-1}$,  distances by $(L_{\tau_1^j}^j)^{-1/2}$ and quantum natural time by $(L_{\tau_1^j}^j)^{-3/2}$.  Let $A_1^j$ be the event that the elements of $\CN_1^{1,j},\ldots,\CN_N^{1,j}$ for which the event $E$ in Proposition~\ref{prop:good_sle_chunks} occurs disconnect $x$ from $z_j$.  We also let $B_1^j$ be the intersection of $A_1^j$ with the event that the annulus $\CA$ formed by the $\CN_i^{1,j}$'s is contained in the $r \log(r^{-1})^{1-p_0 a_4}$ $\qdistnoarg{h}$-neighborhood of $\partial \CD_1^j$,  where $p_0$ is the constant in Proposition~\ref{prop:good_band_contains}.  We pick $a_4$ so that $a_4 > p_0^{-1}$.  Combining Propositions~\ref{prop:good_chunks_percolate} and~\ref{prop:good_band_contains},  we obtain that there exists a constant $q>0$ depending only on $u$ such that $\qsphereinflaw\big[ B_1^j \giv \tau_1^j  <\infty \big] \geq q$.  Note that the choices of $a_1,a_2,a_4$ imply that $\CN_1^{1,j},\ldots,\CN_N^{1,j}$ are all contained in $\qfb{h}{z_j}{x}{\tau_1^j + r \log(r^{-1})^{-a_2}}$,  if we further assume that $1-p_0 a_4 < -a_2$.  Moreover,  it is easy to see that the conditions in the statement of the lemma imply that the conditions in Proposition~\ref{prop:good_band_contains} hold for the field in $\CD_1^j$ obtained by rescaling $h$.  Hence,  it follows from Proposition~\ref{prop:good_band_contains} that the distance with respect to $\qdistnoarg{h}$ between the inner and outer boundaries of $\CA$ is at least $r \log(r^{-1})^{-p_1 a_4 - a_2 / 2}$,  where $p_1$ is the constant in the statement of Proposition~\ref{prop:good_band_contains}.   Furthermore,  for all $i \in \N$ such that $\tau_i^j < \infty$,  we define events $A_i^j,  B_i^j$,  and $\SLE_6$ chunks $\CN_k^{i,j}$ as above.  The same analysis as above implies that given $\tau_i^j < \infty$,  the conditional probability that $B_i^j$ occurs is positive uniformly.  Let $I_j$ be the smallest $i$ such that $B_i^j$ occurs.  Then,  the above analysis implies that the $\qsphereinflaw$ measure of the event that $I_j > \log(r^{-1})^{a_3}$ for some $1 \leq j \leq N_r$ such that $\qdist{h}{x}{z_j} \geq r$ tends to $0$ as $r \to 0$ faster than any power of $r$.  Note that if $\CA_j$ is the annulus corresponding to $I_j$,  then the choice of the constants $a_1,a_2,a_3$ and $a_4$ implies that $\CA_j \subseteq \qfb{h}{z_j}{x}{r} \setminus \qfb{h}{z_j}{x}{r}$.

\emph{Step 3.  A Brownian motion starting from $x$ disconnects $r \log(r^{-1})^{-q}$ units of quantum area while crossing $\CA_j$ for the first time with high probability}.  Suppose that we are working on the event that $I_j \leq \log(r^{-1})^{a_3}$ for all $1 \leq j \leq N_r$.  Let $D_1,\ldots,D_m$ denote the connected components of $\mathcal{S} \setminus \qball{h}{x}{r}$ with the property that they contain a point whose $d_h$-distance from the boundary of the component is at least $r$.  Then $z_j \in \cup_{i=1}^m D_i$ for some $1 \leq j \leq N_r$.  Fix $i,j$ such that $z_j \in D_i$.  Let $K$ be the set of points disconnected from $z_j$ by a Brownian motion starting from $x$ and run until the first time that it exits $\qfb{h}{z_j}{x}{r}$.  Then,  \cite[Proposition~6.32]{law2005conformally} implies that $K$ has the same law as the hull of a whole-plane $\SLE_6$ which is independent of $(\CS,x,y)$,  starting from $x$ and stopped upon hitting $\qfb{h}{z_j}{x}{r}$ since we further have that $\qfb{h}{z_j}{x}{r}$ is a Jordan domain (see \cite[Theorem~1.1]{ms2015mapmaking}).  Then,  we have that the $\SLE_6$ hull has to pass through the annulus $\CA_j$ and upon doing so,  it has to make a crossing in $\CA_j$ from the outer to the inner boundary of $\CA_j$.  In particular,  it has to travel $\qdistnoarg{h}$-distance at least $r \log(r^{-1})^{-a_4p_1 - a_1 / 2}$.  Then,  Lemma~\ref{lem:sle_6_cannot_be_skinny} implies that there exists a constant $q>0$ depending only in $a_1,a_4$ such that off an event whose $\qsphereinflaw$ measure decays to $0$ as $r \to 0$ faster than any power of $r$,  we have that the whole-plane $\SLE_6$ disconnects from $z_j$ at least $r^4 \log(r^{-1})^{-q}$ units of quantum area while crossing $\CA_j$ for the the first time.  

\emph{Step 4.  Conclusion of the proof}. Fix $\beta >1$ large (to be chosen and depending only on $a_1,a_2,a_3,a_4$ and $q$) and suppose that $r$ is sufficiently small such that $\log(r^{-1})^{-\beta} < \log(r^{-1})^{-a_4 / 3} / M$.  Then,  condition~\eqref{it:ngbd_vollume_ubd} implies that the $\left(L_{\tau_{I_j}^j}^j \right)^{1/2} \log(r^{-1})^{-\beta}$ $\qdistnoarg{h}$-neighborhood of the boundary of a good chunk $\CN$ of $\CA_j$ has area at most $r \log(r^{-1})^{1-p\beta -pa_4 /3}$,  where $p$ is the constant in~\eqref{it:ngbd_vollume_ubd}.  Since $L_{\tau_{I_j}^j}^j \geq r \log(r^{-1})^{-a_1}$,  possibly by taking $\beta$ to be larger,  we can assume that the $r \log(r^{-1})^{-\beta}$ $\qdistnoarg{h}$-neighborhood of the boundary of a good chunk $\CN$ of $\CA_j$ has area at most $r \log(r^{-1})^{1-p \beta - p a_4 / 3}$.  It follows that by taking $\beta$ sufficiently large (depending only on $a_1,a_2,a_3,a_4$ and $q$),  we can assume that the union of the $r^{1/2} \log(r^{-1})^{-\beta}$ $\qdistnoarg{h}$-neighborhoods of the boundaries of the chunks in $\CA_j$ has area at most $r \log(r^{-1})^{-q}$.  Therefore,  we obtain that the hull of the whole-plane $\SLE_6$ must exit the $r^{1/2} \log(r^{-1})^{-\beta}$ $\qdistnoarg{h}$-neighborhood of the boundaries of the chunks and so the same is true for the Brownian motion.  Let $\CN$ be the first chunk in $\CA_j$ that the Brownian motion hits the part of $\CN$ which has $\qdistnoarg{h}$-distance at least $r^{1/2} \log(r^{-1})^{-\beta}$ from $\partial \CN$ and let $w$ be the first such point it hits.  Let $\phi:\CN \to \D$ be the conformal transformation as in~\eqref{it:mass_bound}.  Then,  condition~\eqref{it:mass_bound} implies that $\phi(w)$ has Euclidean distance at least $\log(r^{-1})^{-\wt{q}}$ from $\partial \D$,  where $\wt{q} = M(\beta +1 - a_3 / 3)$.  This implies that the probability that the Brownian motion (conformally mapped) hits $B(0,1/4)$ before leaving $\D$ is at least a constant times $\log(r^{-1})^{-\wt{q}}$.  Moreover,  it follows from~\eqref{it:mass_bound} that the quantum mass assigned to $B(0,1/2)$ with respect to the embedding of $\CN$ into $\D$ induced by $\phi$ is at least $\log(r^{-1})^{-4 a_4 / 3} / M$.  Since $G_{B(0,1/2)}(z,w) \gtrsim 1$ for all $z,w \in B(0,1/2)$ with the implicit constant being universal, \eqref{eq:defG_U} implies that the total amount of time in $B(0,1/2)$ that the LBM starting from $w \in \partial B(0,1/4)$ spends up until the first time that it exits $B(0,1/2)$ is at least a universal constant times $\log(r^{-1})^{-4 a_3 / 3} / M$.  Combining everything,  we obtain that $E_x\bigl[ \tau_{\qball{h}{x}{r}} \bigr]$ is at least a universal constant times $\log(r^{-1})^{-\wt{q}-4 a_3 / 3} / M$.  This completes the proof of the lemma.
\end{proof}

\begin{proof}[Proof of Theorem~\ref{thm:bm_exit}, lower bound]
Suppose that $(\CS,x,y)$ has distribution $\qsphereinflaw$.  Since the metric in $\CS$ is bi-H\"older continuous with respect to the Euclidean metric,  we know that there exists $\alpha \in (0,1)$  deterministic such that $\qsphereinflaw$-a.e.,  there exists $r_0 \in (0,1)$ such that $\qball{h}{z}{r^{1 / \alpha^2}} \subseteq \eball{z}{r^{1 / \alpha}} \subseteq \qball{h}{z}{r}$ for all $z \in \CS,  r \in(0,r_0)$.  Fix $r_0 \in (0,1)$ and suppose that we are working on the event that for all $r \in (0,r_0)$,  the above holds and and in addition we have that $\diam(\CS) \geq 6r,  r^4 \log(r^{-1})^{-6-u} \leq \qmeasure{h}(\qball{h}{z}{r}) \leq r^4 \log(r^{-1})^{8+u}$ and $\frac{\qmeasure{h}(\qball{h}{z}{r})}{\qmeasure{h}(\CS)} \geq r^{4+u}$ for all $z \in \CS$,  where $u \in (0,1/3)$ is fixed.  Note that by Theorem~\ref{thm:ball_concentration} and since $\mathcal{S}$ is a finite metric space $\qsphereinflaw$-a.e.,  we obtain that $\qsphereinflaw$-a.e.,  there exists $r_0 \in (0,1)$ satisfying these properties.

Let $(z_j)$ be a sequence of points in $\CS$ chosen i.i.d.\  from $\qmeasure{h}$ normalized to be a probability measure.  	With $N_r = r^{-u-4 / \alpha^2}$,  we know from the proof of Lemma~\ref{lem:filled_metric_ball_exit_time} that $\CS \subseteq \cup_{j=1}^{N_r} \qball{h}{z_j}{(r/4)^{1 / \alpha^2}}$ off an event whose $\qsphereinflaw$ measure tends to $0$ as $r \to 0$ faster than any power of $r$.  Then,  Lemma~\ref{lem:filled_metric_ball_exit_time} implies that there exists a universal constant $\kappa>0$ such that off an event whose $\qsphereinflaw$ measure tends to $0$ as $r \to 0$ faster than any power of $r$,  the following holds for all $1 \leq j \leq N_r$ simultaneously.  The conditional expectation given $\CS$ of the amount of time that the LBM starting from $z_j$ spends in $\qball{h}{z_j}{r} \setminus \qball{h}{z_j}{r/2}$ before leaving $\qball{h}{z_j}{r}$ is at least $r^4 \log(r^{-1})^{-\kappa}$.  Fix $1 \leq j \leq N_r$ and $z \in \eball{z_j}{(r/4)^{1/\alpha}} \subseteq \qball{h}{z_j}{r/2}$.  We note that the Radon-Nikodym derivative between harmonic measure on $\partial \eball{z_j}{(r/4)^{1/\alpha}}$ as seen from $z$ and from $z_j$ is bounded from above and below by universal constants.  By integrating over their first hitting point of $\partial \eball{z_j}{(r/4)^{1/\alpha}}$, we thus have that if $B^z$ (resp.\ $B^{z_j}$) is an LBM starting from $z$ (resp.\ $z_j$) then the expected amount of time that $B^z$ spends in $\qball{h}{z_j}{r} \setminus \qball{h}{z_j}{r/2}$ before exiting $\qball{h}{z_j}{r}$ is comparable to the expected amount of time that $B^{z_j}$ spends in $\qball{h}{z_j}{r} \setminus \qball{h}{z_j}{r/2}$.  Hence,  the proof is complete possibly by taking $\kappa$ to be larger since $\CS \subseteq \cup_{j=1}^{N_r} \eball{z_j}{(r/4)^{1/\alpha}}$.
\end{proof}

\section{Bounds on the Liouville heat kernel}
\label{sec:LBM-HK-bounds}

In this section we prove Theorems~\ref{thm:ubd} and~\ref{thm:lbd}. We start with Theorem~\ref{thm:ubd}, which is now a direct consequence of the results in the previous sections.

\begin{proof}[Proof of Theorem~\ref{thm:ubd}]
First,  we will show the desired upper bound for the LBM killed upon exiting bounded domains in $\C$ whose closure does not contain $0$.  Also,  by scaling and disintegration with respect to area,  it suffices to prove the desired upper bound when $h$ has the law of the unit area LQG sphere.  Fix $U \subseteq \C$ non-empty and open and such that $0 \notin \overline{U}$.  Let also $D \subseteq \C$ be a Jordan domain such that $\overline{U} \subseteq D$ and $0 \notin \overline{D}$.  Note that \cite[Lemma~3.9]{aru2020first} implies that the law of a zero-boundary GFF on $D$ is mutually absolutely continuous with respect to the law of the restriction to $D$ of a massive GFF on $\R^2$ with constant positive mass.  Hence,  combining \cite[Theorem~1.2]{AK16} with absolute continuity,  we obtain that there a.s.\  exists a random constant $C_1<\infty$ such that $\wt{p}_t(x,y) = \wt{p}_t(y,x) \leq C_1 t^{-1} \log(t^{-1})$ for all $x,y \in D$,  $t \in (0,1/2]$,  where $\wt{p}_t(\cdot,\cdot)$ denotes the heat kernel of the LBM killed at the first time that it exits $D$,  with respect to a zero-boundary GFF $h^0$ on $D$.  Note that the definition of the law of $h$ given in \cite[Section~4.5]{dms2014mating} combined with the proof of \cite[Proposition~3.4]{ms2016ig1} imply that the laws of $h|_D$ and $h^0|_D$ are mutually absolutely continuous.  Hence,  it follows that for a.e.\  instance of $h$,  there exists a random constant $C<\infty$ such that
$p_t^D(x,y) = p_t^D(y,x) \leq C t^{-1} \log(t^{-1})$ for all $x,y \in D$,  $t \in (0,1/2]$,
where $p_t^D(\cdot,\cdot)$ denotes the Liouville heat kernel for the LBM with respect to $h$ when killed at the first time that it exits $D$.  Thus,  it follows that for all $(x_0,r) \in D \times (0,2^{-1/4})$ with $\qball{h}{x_0}{2^{1/4}r} \subseteq D$,  $t \in (0,r^4) \subseteq (0,1/2)$, $A \subseteq \qball{h}{x_0}{r}$ Borel and $x \in \qball{h}{x_0}{r}$,  we have that
\begin{align*}
P_x \big[ X_t \in A,  t < \tau_{\qball{h}{x_0}{r}} \big] &= \int_A p_t^{\qball{h}{x_0}{r}} (x,y) \,
 d \qmeasure{h}(y)\\
&\leq \int_A p_t^D(x,y) \, d\qmeasure{h}(y) \leq C t^{-1} \log(e + t^{-1}) \qmeasure{h}(A).
\end{align*}
In particular,  condition $\mathrm{(DU)}$ holds for $h$.

Next,  we proceed to prove that condition $\mathrm{(E)}$ holds.  Indeed,  Theorem~\ref{thm:bm_exit} implies that~\eqref{eq:exit_est-simple} holds with $\beta = 4$ and $\kappa_{\mathrm{el}}=\kappa_{\mathrm{eu}} = \kappa$ for all $(x,r) \in D \times (0,1)$ such that $\qball{h}{x}{r} \subseteq D$,  where $\kappa$ is the constant in the statement of Theorem~\ref{thm:bm_exit}.  Therefore,  by further setting $\alpha_1 = \alpha_2 = \beta_1 = \beta_2 =4$,  Lemma~\ref{lem:Phikappa_Phi0} combined with Theorem~\ref{thm:uhk} and since $\dist_{\qdistnoarg{h}}(\overline{U},\s^2 \setminus D)>0$,  we obtain that the desired upper bound holds for all $(t,x,y) \in (0,1/2] \times \s^2 \times U$.  Now,  suppose that conditional on $(\s^2,h)$,  we choose independently a point $z$ from $\qmeasure{h}$  and let $\phi : \s^2 \to \s^2$ be a conformal map such that $\phi(0) = 0$ and $\phi(z) = \infty$.  Then,  it follows from \cite[Proposition~A.8]{dms2014mating} that $\wt{h}:=h \circ \phi^{-1} + Q \log |(\phi^{-1})'|$ has the same law with $h$ modulo a scaling factor,  where $h$ and $\wt{h}$ are considered to have the embedding introduced in \cite[Section~4.5]{dms2014mating}.  Note that $z \notin \{0,\infty\}$ a.s. Hence,  combining the fact that the desired upper bound on the heat kernel holds in $U$ with \cite[Theorem~1.3]{Be15},  we obtain that there a.s.\  exist open neighborhoods $U_1$ and $U_2$ of $0$ and $\infty$ respectively,  and random constants $C_1,C_2$ such that
\begin{equation}\label{eqn:heat_kernel_upd}
p_t(u,v) \; \leq \; \frac{C_1 (\log t^{-1})^\kappa}{t} \, \exp\Biggl( - C_2 \biggl(\frac{\qdist{h}{u}{v}^4}{t} \biggr)^{\!\frac 1 3} \biggl( \log\biggl( e + \frac{\qdist{h}{u}{v}}{t} \biggr) \biggr)^{\!-\kappa} \Biggr)
\end{equation}
for all $t \in (0,1/2],  (x,y) \in (\s^2 \times U_1) \cup (\s^2 \times U_2)$.  Finally,  we have already shown that~\eqref{eqn:heat_kernel_upd} holds with $t \in (0,1/2]$ and $x,y$ both lying in a bounded domain in $\C$ with positive distance from $0$.  Combining,  we complete the proof.
\end{proof}

As the first step in proving the lower bound in Theorem~\ref{thm:lbd} we record
an on-diagonal heat kernel lower bound as a further consequence of the results in the previous sections.

\begin{lemma} \label{lem:on_diagonal_lbd}
There exists a deterministic constant $\kappa > 0$ such that for $\qsphereinflaw$-a.e.\ instance $(\CS,x,y)$ there exists $C \geq 1$ such that for all $u \in \CS$ and $t\in (0,\frac{1}{2}]$,
\begin{equation}
\label{eqn:on_diagonal_lbd}
p_t(u,u) \geq \frac{1}{C t (\log t^{-1})^\kappa}.
\end{equation}
\end{lemma}

\begin{proof}
First,  we note that assumption~\eqref{eq:transition-density} holds for a.e.\  instance $(\CS,x,y)$ of a sample from the fixed area Brownian map,  since~\eqref{eq:transition-density} is a.s.\  true for the restriction of the GFF on $\R^2$ to every open and bounded subset of $\R^2$,  and the law of the latter is locally absolutely continuous with respect to the law of the fixed area Brownian map as explained in the proof of Theorem~\ref{thm:ubd}.  Hence,  it follows by scaling that assumption~\eqref{eq:transition-density} holds for $\qsphereinflaw$-a.e.\  instance $(\CS,x,y)$.  Moreover,  assumptions $\mathrm{(V)}_{\leq}$ and $\mathrm{(E)}$ have been established for $\qsphereinflaw$ in Theorems~\ref{thm:ball_concentration}
and~\ref{thm:bm_exit}, respectively.  Therefore,  the proof is complete by combining with Theorem~\ref{thm:heat-kernel-simple}.
\end{proof}

In order to prove the off-diagonal lower bound, we want to construct a chain out of order $(d(u,v)^4/t)^{1/3}$ sets which connect $u$ to $v$, each of diameter of order $(t/d(u,v))^{1/3}$ between which the Brownian motion can move with positive probability.  The sets that we will use will be given by annuli of $\SLE_6$ chunks using a certain good event as in the exit time lower bound.  In what follows, we will define the good event in Subsection~\ref{subsec:good_event}.  We will then establish various properties of the corresponding good annuli in Subsection~\ref{subsec:good_annuli}.  We will complete the proof of Theorem~\ref{thm:lbd} in Subsection~\ref{subsec:rest_of_proof}.

\subsection{The good event}
\label{subsec:good_event}

We are now going to give the definition of the good event that we will use for the chunks of $\SLE_6$.

Fix $\delta \in (0,1]$ and $M > 1$, fix a constant $p > 0$ (corresponding to condition~\eqref{it:ngbd_vollume_ubd} in Subsection~\ref{subsec:event_def}) which is small enough as specified in Proposition~\ref{prop:good_quantum_disk}, and let $\CD = (\D,h,0)$ be a sample from $\qdiskweighted{\ell}$.  Let also $\eta'$ be a radial $\SLE_6$ in $\D$ started from the point $x \in \partial \D$ which is sampled uniformly from $\qbmeasure{h}$ (normalized to be a probability measure) and targeted at $0$,  and such that $\eta'$ is independent of $\CD$.  Let $y$ be the point on $\partial \D$ which is antipodal to $x$ with respect to $\qbmeasure{h}$ and let $\wt{\sigma}$ be the first time after $\delta/M$ that the curve $\eta'$ is in the boundary and let $\sigma = \wt{\sigma} \wedge \delta$.  In the case that $\mathcal{W} = (\h,h,0,\infty)$ is a sample from $\qwedge{2}$,  we let $\eta'$ be a chordal $\SLE_6$ in $\h$ from $0$ to $\infty$ which is independent of $\mathcal{W}$ and it is parameterized by quantum natural time with respect to $h$.  We define stopping times $\sigma,\wt{\sigma}$ analogously.  In either case, we let $E$ be the event that $\sigma = \wt{\sigma}$ and the following conditions hold for the quantum surface $\CN$ disconnected by $\eta'([0,\sigma])$.  (Note that $\sigma = \wt{\sigma}$ implies that $\CN$ is homeomorphic to $\D$.)
\begin{enumerate}[(I)]
\item\label{it:hitting_bound} For every $r \in (0, \delta^{1/3} M^{-1})$ and every $x,y \in \CN$ with $d_h$-distance at least $r$ from $\partial \CN$ there exists $t \in [M^{-M} \delta^{4/3}/ (\log \tfrac{1}{r})^M,M^M \delta^{4/3}(\log \tfrac{1}{r})^M]$ so that
\[ \int_{\qball{h}{y}{r}} p_t^\CN(x,w) \, d\qmeasure{h}(w) \geq r^M,\]
where $p^\CN$ denotes the Liouville heat kernel on $\CN$.
\item\label{it:ballnearby} For every $z \in \ol{\CN}$ and $r \in (0,\delta^{1/3}/M)$ there exists $w \in \CN$ with $\qdist{h}{z}{w} \leq r$ and $\qball{h}{w}{2r^M} \subseteq \CN$.
\item\label{it:old_conditions} The conditions in the event $E$ used for the exit time lower bound, described in Subsection~\ref{subsec:event_def}, hold with the parameter $M$. 
\end{enumerate}

The rest of this subsection is devoted to the proof of the following statement.
\begin{proposition}
\label{prop:good_event_happens}
For every $p_0 \in (0,1)$ there exists $M \geq 1$ depending only on $p_0$ so that $\qwedge{2}\bigl[ E^{\CW}_{\sigma} \cap \{E \,\,\text{holds for}\,\, \CN \} \bigm| \sigma < \delta \bigr] \geq 1 - p_0$ for all $\delta \in (0,1]$.
\end{proposition}

First we note that by scaling,  it suffices to prove that conditions~\eqref{it:hitting_bound}--\eqref{it:old_conditions} hold with high probability (if $M$ is sufficiently large) when $\delta = 1$.  From now on,  we will assume that $\delta = 1$.  Also the fact that condition~\eqref{it:old_conditions} holds with as high probability as we want (provided we choose $M$ sufficiently large) follows from Proposition~\ref{prop:good_sle_chunks}.

Next we focus on proving that condition~\eqref{it:ballnearby} holds with high probability if we choose $M$ large enough.  The main idea in order to prove the claim is to prove that condition~\eqref{it:ballnearby} holds with high probability when $\CN$ is the radial $\SLE_6$ chunk drawn on top of an independent sample from $\qdiskweighted{1}$.  Then the claim will follow by arguing in the same way as in Steps 3 and 4 in the proof of Lemma~\ref{lem:reverse_holder_continuity_top_boundary} and the second paragraph  of the proof of Lemma~\ref{lem:reverse_holder_continuity}. We believe that this approach will make it easier for the reader to understand the proof of Proposition~\ref{prop:good_event_happens} (instead of proving directly the claim for the quantum wedge without comparing the law of the latter with that of a quantum disk),  since similar arguments have been presented in Section~\ref{sec:exit-time-bounds}.

Let us now prove that condition~\eqref{it:ballnearby} holds with high probability when the radial $\SLE_6$ chunk $\CN$ is drawn on top of an independent sample from $\qdiskweighted{1}$.  This is the content of the following lemma.

\begin{lemma}\label{lem:metric_non_boundary_tracing}
Suppose that $\CD = (\D,h,0)$ has law $\qdiskweighted{\ell}$ for $\ell>0$ fixed.  Then,  there exists a deterministic constant $M>1$ such that $\qdiskweighted{\ell}$-a.s.\  there exists (random) $ r_0 \in (0,1)$ such that for all $z \in \overline{\D},  r \in (0,r_0)$,  there exists $w \in \D$ with $\qdist{h}{z}{w} \leq r$ and $\qball{h}{w}{2r^M} \subseteq \D$.
\end{lemma}

The main ingredient in the proof of Lemma~\ref{lem:metric_non_boundary_tracing} is the following lemma which states that the LQG distance with respect to a sample from $\qdiskweighted{\ell}$ for $\ell>0$ (when parameterized by $\D$) is bi-H\"older continuous with respect to the Euclidean metric (with deterministic exponents) a.s.

\begin{lemma}
\label{lem:quantum_disk_metric_holder}
Suppose that we have the setup of Lemma~\ref{lem:metric_non_boundary_tracing}. Then,  there exists a deterministic constant $\beta \in (0,1)$ such that $\qdiskweighted{\ell}$-a.s.\  there exists $C>1$ such that 
\begin{align*}
C^{-1} |z-w|^{1/\beta} \leq \qdist{h}{z}{w} \leq C |z-w|^{\beta} \quad \text{for all} \quad z,w \in \D.
\end{align*}
\end{lemma}

\begin{proof}
First we will show the claim of the lemma for a free boundary GFF $\wh{h}$ on $\mathbb{D}$ normalized so that the value of its harmonic part at $0$ is equal to zero and then use the same argument as in the proof of Lemma~\ref{lem:qbmeasure_holder_cont} to deduce the claim for the weighted quantum disk.

\emph{Step 1.  Proof of the claim for a free boundary GFF on $\D$.}
We note that $\wh{h}$ can be sampled as follows.  Let $\wt{h}$ be a free boundary GFF on $\mathbb{H}$ such that its average on $\mathbb{H} \cap \partial \mathbb{D}$ is equal to zero and consider the conformal transformation $F : \mathbb{H} \to \mathbb{D}$ such that $F(z) = -\frac{z-i}{z+i}$.  Let also $\wt{\Fh}$ denote the harmonic part of $\wt{h}$.  Then we have that $\wh{h}$ can be sampled as $\wh{h} = \wt{h} \circ F^{-1} -\wt{\Fh}(i)$.  Moreover \cite[Proposition~1.8]{hughes2024equivalencemetricgluingconformal} implies that there exists deterministic constant $\beta \in (0,1)$ such that the following is true a.s.  For every compact set $K \subseteq \overline{\mathbb{H}}$,  there exists (random) constant $C>1$ such that
\begin{align*}
C^{-1} |z-w|^{1/\beta} \leq \qdist{\wt{h}}{z}{w} \leq C |z-w|^{\beta} \quad \text{for all} \quad z , w \in K.
\end{align*}
Thus the same is true with $\wt{h}-\wt{\Fh}(i)$ in place of $\wt{h}$ since
\begin{align*}
\qdist{\wt{h}-\wt{\Fh}(i)}{z}{w} = \exp\left(-\wt{\Fh}(i) / \sqrt{6} \right) \qdist{\wt{h}}{z}{w} \quad \text{for all} \quad z ,  w \in \h.
\end{align*}
Furthermore it holds that $\qdist{\wt{h} \circ F^{-1} + Q \log |(F^{-1})'|}{z}{w} = \qdist{\wt{h} - \wt{\Fh}(i)}{F^{-1}(z)}{F^{-1}(w)}$ and there exists deterministic constant $M>1$ such that $M^{-1} \leq |(F^{-1})'(w)| \leq M$ for all $w \in \D \setminus \eball{-1}{1/4}$.  It follows that there a.s.  exists a random constant $C>1$ such that
\begin{align}\label{eqn:metric_bounds_away_from_-1}
C^{-1} |z-w|^{1/\beta} \leq \qdist{\wh{h}}{z}{w} \leq C |z-w|^{\beta} \quad \text{for all} \quad z ,w  \in \overline{\D} \setminus B(-1,1/4).
\end{align}
Note that the random fields $h(z)$ and $h(-z)$ have the same law and so combining with~\eqref{eqn:metric_bounds_away_from_-1},  we obtain that it is a.s.  the case that there exists $C>1$ such that
\begin{align}\label{eqn:metric_bounds_near_-1}
C^{-1} |z-w|^{1 / \beta} \leq \qdist{\wh{h}}{z}{w} \leq C |z-w|^{\beta} \quad \text{for all} \quad z,w  \in \D \cap B(-1,1/4).
\end{align}

Fix $z \in \D \setminus \overline{B(-1,1/4)}$, $w \in B(-1,1/4) \cap \D$ and let $y$ be the point of intersection between $\D \cap \partial B(-1,1/4)$ and the segment $[z,w]$.  Then combining~\eqref{eqn:metric_bounds_away_from_-1} and~\eqref{eqn:metric_bounds_near_-1} we obtain that 
\begin{align*}
\qdist{\wh{h}}{z}{w} \leq \qdist{\wh{h}}{z}{y} + \qdist{\wh{h}}{y}{w} \leq 2 C |z-w|^{\beta}.
\end{align*}
To show the lower bound for $\qdist{\wh{h}}{z}{w}$,  we fix $\epsilon>0$.  Since $d_{\wh{h}}$ is a length metric,  we obtain that there exists a path $P : [0,1] \to \D$ such that $P(0) = w,  P(1) = z$ and the $d_{\wh{h}}$-length of $P$ is at most $\qdist{\wh{h}}{z}{w}+\epsilon$.  Let $x$ be the first point of $\D \cap \partial B(-1,1/4)$ that $P$ intersects.  Then either $|z-x| \geq |z-w| / 2$ or $|w-x| \geq |z-w| / 2$.  Suppose that the former case holds.  Then~\eqref{eqn:metric_bounds_away_from_-1} implies that
\begin{align*}
\qdist{\wh{h}}{z}{w} + \epsilon \geq \qdist{\wh{h}}{z}{x} \geq C^{-1} |z-x|^{1/\beta} \geq C^{-1} 2^{-1/\beta} |z-w|^{1 / \beta}.
\end{align*}
Similarly if the latter case holds, \eqref{eqn:metric_bounds_near_-1} implies that
\begin{align*}
\qdist{\wh{h}}{z}{w} + \epsilon \geq \qdist{\wh{h}}{w}{x} \geq C^{-1} |w-x|^{1/\beta} \geq C^{-1} 2^{-1/\beta} |z-w|^{1/\beta}.
\end{align*}
Hence since $\epsilon>0$ was arbitrary,  we obtain that 
\begin{align*}
\qdist{\wh{h}}{z}{w} \geq C^{-1} 2^{-1/\beta} |z-w|^{1/\beta}
\end{align*}
in either case.  Combining we obtain that it is a.s.  the case that there exists $C>1$ such that
\begin{align}\label{eqn:metric_bounds_for_free_boundary_gff}
C^{-1} |z-w|^{1/\beta} \qdist{\wh{h}}{z}{w} \leq C |z-w|^{\beta} \quad \text{for all} \quad z ,w  \in \D.
\end{align}

\emph{Step 2.  Conclusion of the proof.}
Next we will combine Step 1 with the argument in the proof of Lemma~\ref{lem:qbmeasure_holder_cont} to complete the proof of the lemma.  

Recall that \cite[Theorem~1.2]{ang2022integrabilitysleconformalwelding} implies that the following is true a.s.  Suppose that $f$ is sampled from the group $\text{conf}(\h)$ of conformal automorphisms of $\h$ when the latter is endowed with the Haar measure,  and let $h_1$ be sampled from the infinite measure of a weight-2 quantum disk with $\gamma = \sqrt{8/3}$ weighted by $\nu_{h_1}(\partial \h)^{-2}$.  Then there exists a deterministic constant $C>0$ such that the law of the field $h_1 \circ f^{-1} + Q \log |(f^{-1})'|$ is given by $C$ times the law of $\wt{h} -2Q \log |\cdot|_+ + c$ where $c$ is sampled independently from the infinite measure on $\R$ given by $\exp(-Q c) dc$,  and recall that $\wt{h}$ denotes a free boundary GFF on $\h$ normalized so that its average on $\h \cap \partial \D$ is equal to zero.  It follows that that the field $h_1 \circ f^{-1} \circ F^{-1} + Q \log |(F \circ f)^{-1})'|$ has the same law as $C$ times the law of the field
\begin{align*}
\wt{h} \circ F^{-1} -2Q \log | F^{-1}(\cdot)|_+ + Q \log |(F^{-1})'| + c
\end{align*}
where $\log|x|_+ = \log \max(|x|,1)$.

Since $d_{\wt{h} \circ F^{-1}} = \exp\left(\wt{\Fh}(i) / \sqrt{6} \right) d_{\wh{h}}$,  we have by~\eqref{eqn:metric_bounds_for_free_boundary_gff} that it is a.s.\ the case that there exists $M>1$ such that
\begin{align*}
M^{-1} |z-w|^{1/\beta} \leq \qdist{\wt{h} \circ F^{-1}}{z}{w} \leq M |z-w|^{\beta} \quad \text{for all} \quad z , w \in \D. 
\end{align*}
Moreover as explained in the proof of Lemma~\ref{lem:qbmeasure_holder_cont},  we have that
\begin{align*}
-2Q   \log| F^{-1}(\cdot)|_+ + Q \log |(F^{-1})'(\cdot)| = O(1)
\end{align*}
uniformly in $\D$.  Thus it follows that almost everywhere,  there exists $M>1$ such that
\begin{align*}
M^{-1} |z-w|^{1/\beta} \leq \qdist{h_1 \circ (F \circ f)^{-1} + Q \log |((F\circ f)^{-1})'|}{z}{w} \leq M |z-w|^{\beta} \quad \text{for all} \quad z,w \in \D.
\end{align*}

Combining with the fact that the event in the lemma statement is invariant under the coordinate change formula of quantum surfaces with disintegration with respect to the total boundary length of a quantum disk sampled from the infinite measure,  we obtain the lemma statement for a sample from $\qdisk{\ell}$.  Therefore,  it also holds a.s.\  for a sample from $\qdiskweighted{\ell}$ since the measures $\qdiskweighted{\ell}$ and $\qdisk{\ell}$ are mutually absolutely continuous.  This completes the proof of the lemma.
\end{proof}

\begin{proof}[Proof of Lemma~\ref{lem:metric_non_boundary_tracing}.]
By absolute continuity, it suffices to prove the claim for $\qdisk{\ell}$ instead.  
Let $\beta \in (0,1)$ be the constant of Lemma~\ref{lem:quantum_disk_metric_holder}.  Then,  Lemma~\ref{lem:quantum_disk_metric_holder} implies that $\qdisk{\ell}$-a.e.,  there exists a random constant $C>1$ such that $C^{-1} |z-w|^{1 / \beta} \leq \qdist{h}{z}{w} \leq C |z-w|^{\beta}$ for all $z,w \in \D$.  Pick $\wt{M}>1$ deterministic such that $\wt{M} \beta >1$.  Then,  we have that $\D \cap \eball{z}{r^{\wt{M}}} \subseteq \qball{h}{z}{r}$ for all $z \in \overline{\D}$ and all $r>0$ sufficiently small,  which implies that there exists $w \in \D$ such that $\eball{w}{r^{\wt{M}} / 2} \subseteq \D \cap \qball{h}{z}{r}$.  Fix $M > \beta^{-1} \wt{M}$.  Then,  we have that $\qball{h}{w}{r^M} \subseteq \eball{w}{r^{\wt{M}}/2} \subseteq \D \cap \qball{h}{z}{r}$ for all $r>0$ sufficiently small (independent of $z,w$) and this completes the proof.
\end{proof}

Now we proceed on proving that condition~\eqref{it:hitting_bound} holds with high probability for $M$ large enough.  As in the first paragraph of the proof of Lemma~\ref{lem:reverse_holder_continuity},  the main idea is to compare locally the laws of a quantum disk and a quantum wedge of weight 2,  and then deduce the claim by proving that condition~\eqref{it:hitting_bound} holds with high probability when the surface $\CN$ is drawn on top of a sample from $\qdiskweighted{1}$.  The purpose of the following lemma is to show the latter claim.

\begin{lemma}\label{lem:hitting_bound}
Suppose that we have the same setup as in the definition of the good event just before the statement of Proposition~\ref{prop:good_event_happens},  where the surface $\CN$ is drawn on top of a sample $\mathcal{D} = (\D,h,0)$ from $\qdiskweighted{1}$.  Then,  it is a.s.\  the case that there exists $\mathcal{M}_0 \geq 1$ such that condition~\eqref{it:hitting_bound} holds for all $M \geq \mathcal{M}_0$ when $\delta=1$.
\end{lemma}

The main idea of the proof of Lemma~\ref{lem:hitting_bound} is the following.  Suppose that we have the same setup as in the statement of Lemma~\ref{lem:hitting_bound}.  Fix $M>1$ sufficiently large and let $x,y \in \CN$ be such that $\text{dist}_{d_h}(\{x,y\} ,  \partial \CN) \geq r$,  where $r \in (0,M^{-1})$.  Let also $T_{\qball{h}{y}{r}}$ denote the total amount of time that the Liouville Brownian motion starting from $x$ with respect to $h$ and killed upon exiting $\CN$ spends in $\qball{h}{y}{r}$.  Then we will show that if $M$ is sufficiently large,  it is very likely that  there exists $t \in I_{M,r}$ with
\begin{align}\label{eqn:average_heat_kernel_lower_bound}
\frac{1}{|I_{M,r}|} \int_{\qball{h}{y}{r}} p_t^{\CN}(x,w) d\qmeasure{h}(w) \gtrsim E_x[T_{\qball{h}{y}{r}}],
\end{align}
where 
\begin{align*}
I_{M,r}:=[M^{-M/2} \log(r^{-1})^{-M} ,  M^{M/2} \log(r^{-1})^M].
\end{align*}
Moreover we will show that with high probability (if $M$ is large enough),  we have that
\begin{align}\label{eqn:total_time_lower_bound}
E_x[T_{\qball{h}{y}{r}}] \gtrsim r^{M/2}.
\end{align}
Therefore combining~\eqref{eqn:average_heat_kernel_lower_bound} with~\eqref{eqn:total_time_lower_bound},  the proof of Lemma~\ref{lem:hitting_bound} will be complete.

We start with proving a lower bound for $E_x[T_{\qball{h}{y}{r}}]$.  This is the content of the following lemma.

\begin{lemma}\label{lem:lower_bound_on_mean_time_spent_by_LBM}
Suppose that we have the same setup as in the definition of the good event just before the statement of Proposition~\ref{prop:good_event_happens},  where the surface $\CN$ is drawn on top of a sample $\mathcal{D} = (\D,h,0)$ from $\qdiskweighted{1}$.
Let $\varphi \colon \CN \to \D$ be the conformal transformation as introduced in \eqref{it:mass_bound} of Subsection~\ref{subsec:event_def}, and set $\wt{h}:=h|_{\CN} \circ \varphi^{-1} + Q \log |(\varphi^{-1})'|$.
Then there exists a deterministic constant $\alpha > 0$ such that it is a.s.\ the case that there exist $M>1,  c_0>0$ such that the following is true for all $r \in (0,M^{-1})$ and all $x,y \in \D$ such that $\text{dist}_{d_{\wt{h}}}(\{x,y\} ,  \partial \D) \geq r$.  If $T_A$ denotes the amount of time that the Liouville Brownian motion in $\D$ with respect to $\wt{h}$ spends in the set $A \subseteq \D$ before exiting $\D$,  then we have that 
\begin{align*}
E_x[T_{\qball{\wt{h}}{y}{r}}] \geq c_0 r^{\alpha + M/4}.
\end{align*}
\end{lemma}

Let us first briefly describe the setup of the proof of Lemma~\ref{lem:lower_bound_on_mean_time_spent_by_LBM} before proceeding with its proof.  The setup is similar to the setup of the proof of Lemma~\ref{lem:mass_bound}.  Let $I \subseteq \partial \D$ be a fixed countable and dense subset of $\partial \D$.  Then it is a.s.  the case on the event that $\sigma = \wt{\sigma}$,  that there exist $z,w \in I$ such that $\eta'([0,\sigma]) = \eta'([0,\tau])$,  where $\tau$ is the first time that $\eta'$ disconnects $z$ from $w$. Also we have that $z$ lies in the boundary of the connected component of $\D \setminus \eta'([0,\tau])$ which contains $0$.  Note that $\tau$ is also the first time that $\eta'$ disconnects $w$ from $0$ and so $\eta'|_{[0,\tau]}$ has the same law as a chordal $\SLE_6$ process in $\D$ from $x$ to $w$,  stopped at the first time at which it disconnects $w$ from $0$,  and the latter time is the same time at which the chordal process disconnects $w$ from $z$.  Moreover the locality property of $\SLE_6$ implies that the latter chordal process has the same law as a chordal $\SLE_6$ process $\wt{\eta}'$ in $\D$ from $x$ to $z$, stopped at the first time $\wt{\tau}$ at which it disconnects $z$ from $w$.  In particular we have that $\CN$ has the same law with the hull of $\wt{\eta}'|_{[0,\wt{\tau}]}$.  Combining with the rotational invariance of the law of radial $\SLE_6$,  we obtain that it suffices to prove the claim of the lemma when $\CN$ is replaced by the hull of $\wt{\eta}'|_{[0,\wt{\tau}]}$,  where $\wt{\eta}'$ is a chordal $\SLE_6$ in $\D$ from $-i$ to $z$ stopped at the first time $\wt{\tau}$ that it disconnects $z$ from $w$,  where $z,w$ are fixed and distinct points in $I$ and $\wt{\eta}'$ is independent from $\mathcal{D}$.

\begin{proof}[Proof of Lemma~\ref{lem:lower_bound_on_mean_time_spent_by_LBM}.]

\emph{Step 1.  Outline and setup.}
Suppose that we have the same setup as in the paragraph just after the statement of the lemma and let $(\wt{K}_t)$ denote the family of hulls of $\wt{\eta}'$. Again by the locality property,  we have that $\wt{\eta}'$ can be coupled with a chordal $\SLE_6$ $\wh{\eta}'$ in $\D$ from $-i$ to $w$ stopped at the first time $\wh{\tau}$ that $\wh{\eta}'$ disconnects $z$ from $w$  such that $\wt{\eta}'|_{[0,\wt{\tau}]} = \wh{\eta}'|_{[0,\wh{\tau}]}$. 

Without loss of generality,  we can assume that $w$ lies in the counterclockwise arc of $\partial \D$ from $-i$ to $z$.  Let $\wh{\eta}$ denote the left outer boundary of $\wh{\eta}'$ when viewed as a curve from $w$ to $-i$.  It follows from \cite[Theorem~1.4]{ms2016ig1} that $\wh{\eta}$ has the law of an $\SLE_{\frac{8}{3}}(\frac{8}{3}-2 ; \frac{8}{3}-4)$ process in $\D$ from $w$ to $-i$ with the force points located at $w^-$ and $w^+$ respectively.  Let $U$ be the connected component whose boundary contains $-i$ of the complement in $\D$ of the curve $\wh{\eta}$ stopped at the first time that it disconnects $-i$ from $z$.  Similarly we let $V$ be the connected component whose boundary contains $w$ of the complement in $\D$ of the time-reversal of $\wh{\eta}$ stopped at the first time that it disconnects $w$ from $z$.  Let also $G$ be the connected component of $\D \setminus \wh{\eta}$ lying to the left of $\wh{\eta}$.  Note that Proposition~\ref{prop:SLE6-chunk-Jordan-D} implies that $\text{int}(\wt{K}_{\wt{\tau}}),U,V,G$ are all Jordan domains such that $U \cup V \cup G \subseteq \text{int}(\wt{K}_{\wt{\tau}})$ and let $\phi : \text{int}(\wt{K}_{\wt{\tau}}) \to \D ,  f : U \to \D,  g : V \to \D$ and $\psi : G \to \D$ be conformal transformations chosen in some arbitrary but fixed way.  Let also $\wh{I}$ be the arc traced by $\wh{\eta}$ up until the last time that it hits the counterclockwise arc of $\partial \D$ from $-i$ to $z$ and let $\wh{J}$ be the arc traced by the time-reversal of $\wh{\eta}$ stopped at the last time that it hits the clockwise arc of $\partial \D$ from $-i$ to $z$.  Note that $\wh{I} \cap \wh{J} = \emptyset$. Note also that \cite[Theorem~5.2]{rohde2011basic} combined with the time-reversal symmetry of the law of $\wh{\eta}$ (see \cite[Theorem~1.1]{ms2016ig2})) implies that there exists a deterministic constant $\alpha \in (0,1)$ such that all of the maps $f^{-1},g^{-1}$ and $\psi^{-1}$ are $\alpha$-H\"older continuous. 

Fix $M \in (1,\infty)$ sufficiently large (to be chosen) and let $r \in (0,M^{-1})$.  Let $x ,  y \in \wt{\CN}:=\text{int}(\wt{K}_{\wt{\tau}})$ be such that $\text{dist}_{d_{h|_{\wt{\CN}}}}(\{x,y\} ,  \partial \wt{\CN}) \geq r$.  In Step 2,  we will show that the probability that a complex Brownian motion starting from $x$ intersects $\psi^{-1}(\eball{0}{1/2})$ before exiting $\wt{\CN}$ for the first time is $\gtrsim r^{1/(\alpha \beta)}$,  where the implicit constant is independent of $r$ and $\beta \in (0,1)$ is the constant in the statement of Lemma~\ref{lem:quantum_disk_metric_holder}.  Then we will conclude the proof in Step 3 as follows.  Suppose that we are working on the event that the Brownian motion intersects $\psi^{-1}(\eball{0}{1/2})$ before exiting $\wt{\CN}$ for the first time.  Using the $\alpha$-H\"older continuity of the maps $\psi^{-1},f^{-1},g^{-1}$ and $\phi^{-1}$,  we will show that the probability that a Brownian motion starting from a point $u \in \psi^{-1}(\eball{0}{1/2})$ intersects $\qball{h}{y}{r}$ before exiting $\wt{\CN}$ for the first time is $\gtrsim r^{1/(\alpha \beta)}$,  where the implicit constant is uniform in $u$ and $r$.  Therefore the proof will be complete by combining with the Markov property of the Brownian motion and the fact that $\qmeasure{h}(\qball{h}{y}{r}) \geq r^{M/4}$ if $M$ is sufficiently large (see \cite[Lemma~3.3]{gm2019gluing}).

\emph{Step 2.  Lower bound on the probability that the Brownian motion hits $\psi^{-1}(\eball{0}{1/2})$ before leaving $\wt{\CN}$.}
We have the following cases.

\emph{Case 1.  $\qball{h|_{\wt{\CN}}}{x}{r} \cap (\wh{I} \cup \wh{J}) = \emptyset$.} 
In that case,  we have that $\qball{h|_{\wt{\CN}}}{x}{r} \subseteq G$ and so $\qball{h|_{\wt{\CN}}}{x}{r} = \qball{h|_G}{x}{r}$.  Lemma~\ref{lem:quantum_disk_metric_holder} implies that there exists deterministic constant $\beta \in (0,1)$ such that possibly by taking $M$ to be larger,  we have that
\begin{align}\label{eqn:metric_balls_inclusion}
B(x,r^{1/\beta}) \subseteq \qball{h|_G}{x}{r} \subseteq B(x,r^{\beta}).
\end{align}

Let $C>0$ be such that $|\psi^{-1}(u)-\psi^{-1}(v)| \leq C |u-v|^{\alpha}$ for all $u,v \in \D$.  Then combining with~\eqref{eqn:metric_balls_inclusion}  we obtain that 
\begin{align*}
B(\psi(x) ,  C^{-1/\alpha} r^{1/(\alpha \beta)}) \subseteq \qball{h|_G \circ \psi^{-1} + Q \log |(\psi^{-1})'|}{\psi(x)}{r}.
\end{align*}
Note that the probability that starting from $\psi(x)$ the Brownian motion intersects $B(0,1/2)$ before leaving $\D$ is at least a constant times $-\log |\psi(x)|$.  In particular,  combining with conformal invariance,  we obtain that there exists a (random) constant $c_0>0$ such that the probability that a Brownian motion starting from $x$ intersects $\psi^{-1}(B(0,1/2))$ before exiting $\wt{\CN}$ for the first time is at least $c_0 r^{1/(\alpha \beta)}$.

\emph{Case 2.  $\qball{h|_{\wt{\CN}}}{x}{r} \cap \wh{I} \neq \emptyset$.}
Possibly by taking $M$ to be larger,  we can assume that $\qball{h|_{\wt{\CN}}}{x}{r} \cap \wh{J} = \emptyset$ and hence $\qball{h|_{\wt{\CN}}}{x}{r} \subseteq V$,  which implies that $\qball{h|_{\wt{\CN}}}{x}{r} = \qball{h|_V}{x}{r}$.  Similarly to Case 1,  we have that
\begin{align*}
B(g(x) ,  C^{-1/\alpha} r^{1/(\alpha \beta)}) \subseteq \qball{h|_V \circ g^{-1} + Q \log |(g^{-1})'|}{g(x)}{r},
\end{align*}
where $C>0$ is such that
\begin{align*}
|g^{-1}(u)-g^{-1}(v)| \leq C |u-v|^{\alpha} \quad \text{for all} \quad u,v \in \D.
\end{align*}
In particular,  we have that $\text{dist}(g(x) ,  \partial \D ) \geq C^{-1/\alpha} r^{1/(\alpha \beta)}$ and so possibly by taking the constant $c_0 > 0$ in Case 1 to be smaller and combining with conformal invariance,  we can assume that the probability that a Brownian motion starting from $x$ intersects $\psi^{-1}(B(0,1/2))$ before exiting $\wt{\CN}$ is at least $c_0 r^{1/(\alpha \beta)}$.

\emph{Case 3.  $\qball{h|_{\wt{\CN}}}{x}{r} \cap \wh{J} \neq \emptyset$.}
As in Case 2,  possibly by taking $M$ to be larger,  we can assume that $\qball{h|_{\wt{\CN}}}{x}{r} \cap \wh{I} = \emptyset$ and so $\qball{h|_{\wt{\CN}}}{x}{r} \subseteq U$,  which implies that $\qball{h|_U}{x}{r} = \qball{h|_{\wt{\CN}}}{x}{r}$.  Since $f^{-1}$ is $\alpha$-H\"older continuous,  by arguing as in Case 2,  we obtain that possibly by taking the constant $c_0>0$ to be smaller,  we have that the probability that a Brownian motion starting from $x$ intersects $\psi^{-1}(B(0,1/2))$ before exiting $\wt{\CN}$for the first time is at least $c_0 r^{1/(\alpha \beta)}$.

It follows that in every case,  there exists a (random) constant $c_0>0$ such that the probability that a Brownian motion starting from $x$ intersects $\psi^{-1}(B(0,1/2))$ before exiting $\wt{\CN}$ for the first time is at least $c_0 r^{1/(\alpha \beta)}$.

\emph{Step 3.  Conclusion of the proof.}
Next we set $A:=\qball{h|_{\wt{\CN}}}{y}{r}$.  Recall that Lemma~\ref{lem:quantum_disk_metric_holder} implies that possibly by taking $M$ to be larger,  we can assume that
\begin{align*}
B(u,r^{1/\beta}) \subseteq \qball{h}{u}{r} = \qball{h|_{\wt{\CN}}}{u}{r} \subseteq B(u,r^{\beta})
\end{align*}
for all $u \in \D$ such that $\text{dist}_{d_h}(u,\partial \D) \geq r$.

Suppose first that $\qball{h}{y}{r} \cap (\wh{I} \cup \wh{J}) = \emptyset$.  Then arguing as in Case 1,  we obtain that there exists a constant $C>1$ such that
\begin{align*}
B(\psi(y) ,  C^{-1/\alpha} (r/2)^{1/(\alpha \beta)}) \subseteq \psi(B(y,(r/2)^{1/\beta})).
\end{align*}
Then \cite[Exercise~2.7]{law2005conformally} implies that there exists a (random) constant $c>0$ such that for all $u \in B(0,1/2)$,  the probability that a Brownian motion starting from $u$ intersects $B(\psi(y),C^{-1/\alpha} (r/2)^{1/(\alpha \beta)})$ before exiting $\D$ for the first time is at least $c r^{1/(\alpha \beta)}$.  Therefore combining with Step 2 and the Markov property of the Brownian motion,  we obtain that the probability that a Brownian motion starting from $x$ intersects $A$ before exiting $\wt{\CN}$ for the first time is at least $c c_0 r^{2/(\alpha \beta)}$.

Note that \cite[Lemma~3.3]{gm2019gluing} implies that possibly by taking $M$ to be larger,  we have that
\begin{align*}
\qmeasure{h}(\eball{y}{(r/4)^{1\beta}}) \geq \qmeasure{h}(\qball{h}{y}{(r/4)^{1/\beta^2}}) \geq r^{M/4}.
\end{align*}
It follows that
\begin{align*}
E_x[T_A] &\geq c c_0 r^{2/(\alpha \beta)} \inf_{u \in \partial B(y,(r/2)^{1/\beta})} E_u[T_{\eball{y}{(r/4)^{1/\beta}}}]\\
&\geq c c_0 r^{2/(\alpha \beta)} \inf_{u \in \partial \eball{y}{(r/2)^{1/\beta}}} \int_{\eball{y}{(r/4)^{1/\beta}}} G_{\eball{y}{r^{1/\beta}}}(u,v) d \qmeasure{h}(v) \gtrsim r^{\frac{2}{\alpha \beta} + \frac{M}{4}}
\end{align*}
where $G_{\eball{y}{r^{1/\beta}}}$ denotes the Green's function on $\eball{y}{r^{1/\beta}}$ and in the latter inequality we also used that $G_{\eball{y}{r^{1/\beta}}}(u,\cdot)$ is bounded from below on $\eball{y}{(r/4)^{1/\beta}}$ by a universal constant which is uniform on $u \in \partial \eball{y}{(r/2)^{1/\beta}}$.

By arguing in the same way in the cases that either $\qball{h}{y}{r} \cap \wh{I} \neq \emptyset$ or $\qball{h}{y}{r} \cap \wh{J} \neq \emptyset$,  we obtain that 
\begin{align*}
E_x[T_A] \gtrsim r^{\frac{2}{\alpha \beta} + \frac{M}{4}}
\end{align*}
in both cases and so this completes the proof of the lemma.
\end{proof}

\begin{proof}[Proof of Lemma~\ref{lem:hitting_bound}.]
\emph{Step 1.  Outline and setup.}
Let $\alpha>0$ be the deterministic constant in Lemma~\ref{lem:lower_bound_on_mean_time_spent_by_LBM} and let $M>0$ be sufficiently large such that the statement of Lemma~\ref{lem:lower_bound_on_mean_time_spent_by_LBM} holds.  Set $\wt{h}:=h|_{\CN} \circ \varphi^{-1} + Q \log |(\varphi^{-1})'|$ and fix $x ,  y \in \D ,  r \in (0,M^{-1})$ such that $\text{dist}_{d_{\wt{h}}}(\{x,y\} ,  \partial \D) \geq r$.  Set also $A:=\qball{\wt{h}}{x}{r}$ and as in Lemma~\ref{lem:lower_bound_on_mean_time_spent_by_LBM},  we let $T_A$ denote the amount of time that the Liouville Brownian motion in $\D$ with respect to $\wt{h}$ spends in $A$.  Then Lemma~\ref{lem:lower_bound_on_mean_time_spent_by_LBM} implies that there exists a (random) constant $c_0 > 0$ depending only on $M$ such that
\begin{align}\label{eqn:lower_bound_on_time_spent_by_LBM}
E_x[T_A] \geq c_0 r^{\alpha + \frac{M}{4}}.
\end{align}
Moreover \cite[Lemma~3.3]{gm2019gluing} implies that it is a.s.  the case that possibly by taking $M$ to be larger,  we have that
\begin{align}\label{eqn:quantum_area_lower_bound}
\qmeasure{h}(\qball{h}{u}{r}) \geq r^{\frac{M}{4}} \quad \text{for all} \quad u \in \D \quad \text{such that} \quad \text{dist}_{d_h}(u,\partial \D) \geq r.
\end{align}

Set
\begin{align*}
&f_1(M,r):=\int_{\D} \int_{M^{M/2} \log(r^{-1})^M} p_t^{\CN}(x,w) dt d\qmeasure{\wt{h}}(w)\\
&f_2(M,r):=\int_A \int_0^{M^{-M/2} \log(r^{-1})^{-M}} p_t^{\CN}(x,w) dt d\qmeasure{\wt{h}}(w).
\end{align*}
In Step 2,  we will show that $f_1(M,r) \to 0$ as $r \to 0$ faster than any positive power of $r$ while in Step 3 we will complete the proof of the lemma by bounding $f_2(M,r)$ from above and combining with~\eqref{eqn:lower_bound_on_time_spent_by_LBM}.

\emph{Step 2.  $f_1(M,r)$ tends to zero as $r \to 0$ faster than any positive power of $r$.}
First we note that $u \mapsto E_u[\tau_\D] = \int_{\D} G(u,w) \, d\qmeasure{h}(w)$ is an a.s.\ continuous function on $\ol{\D}$ and is therefore a.s.\ bounded,  where $G$ denotes the Green's function on $\D$.  By Markov's inequality, we therefore have that
\[ P_u[ \tau_{\D} \geq t] \leq \frac{E_u[\tau_\D]}{t} \to 0 \quad\text{as}\quad t \to \infty\]
uniformly in $u \in \D$.  In particular, by increasing the value of $M$ if necessary, we have that that $P_u[ \tau_{\D} \geq M] \leq 1/2$ for all $u \in \D$.  By the Markov property, we therefore have that  $P_u[ \tau_{\D} \geq t M] \leq 2^{-t}$ for all $u \in \D$ and $t \geq 1$.  We therefore have by applying the Markov property again that
\begin{align*}
 f_1(M,r) &\leq  \int_{\CD} \int_{M^{M/2} (\log \tfrac{1}{r})^M}^\infty p_t^\CD(\varphi^{-1}(x),w) \, dt \, d\qmeasure{h}(w)\\
 & =E_{\varphi^{-1}(x)}\big[ (\tau_{\D}-M^{M/2} \log(r^{-1})^M) \indicator_{\{\tau_{\D} > M^{M/2} \log(r^{-1})^M\}}\big]\\
& \leq \sup_{z \in \D} E_z\big[ \tau_{\D} \big] P_{\varphi^{-1}(x)}\big[ \tau_{\D} > M^{M/2} \log(r^{-1})^M \big] \leq \sup_{z \in \D} E_z\big[ \tau_{\D} \big] \left(\frac{1}{2}\right)^{M^{M/2-1} \log(r^{-1})^M}
 \end{align*}
which tends to zero as $r\to 0$ faster than any positive power of $r$, provided $M$ is sufficiently large.

\emph{Step 3.  Conclusion of the proof.}
Now we will complete the proof of the lemma by bounding from above the term $f_2(M,r)$.

First we note that a sample from $\mathcal{D}$ can be produced as follows.  Let $(\s^2,\wh{h},0,\infty)$ be a doubly marked quantum sphere and let $\wh{\eta}'$ be a whole-plane $\SLE_6$ in $\s^2$ from $0$ to $\infty$ parameterized by quantum natural time with respect to $\wh{h}$.  Let $\tau$ be the first time $t$ that the quantum boundary length of the $\infty$-containing connected component of $\s^2 \setminus \wh{\eta}'([0,t])$ is equal to $1$ and we condition on the event that $\tau < \infty$.  Let $U$ be that component and let $\psi : U \to \D$ be the conformal transformation such that $\psi(\infty)=0$ and $\psi'(\infty) > 0$.  Then conditional on $\tau<\infty$,  we set
\begin{align*}
h:=\wh{h}|_U \circ \psi^{-1} + Q \log |(\psi^{-1})'| \quad \text{and} \quad \qdist{h}{z}{w} = \qdist{\wh{h}|_U}{\psi^{-1}(z)}{\psi^{-1}(w)} \quad \text{for all} \quad z,w \in \D.
\end{align*}

Set $\wt{x}:=\varphi^{-1}(x),  \wt{y}:=\varphi^{-1}(y)$.  Then we have that either $\qdist{\wh{h}}{\psi^{-1}(\wt{x})}{\psi^{-1}(\wt{y})} \geq 2M^{-1}$ or $\qdist{\wh{h}}{\psi^{-1}(\wt{x})}{\psi^{-1}(\wt{y})} <2M^{-1}$.  We will prove the claim of the lemma in each different case.

\emph{Case 1.  $\qdist{\wh{h}}{\psi^{-1}(\wt{x})}{\psi^{-1}(\wt{y})} \geq 2M^{-1}$.}
Suppose that $\qdist{\wh{h}}{\psi^{-1}(\wt{x})}{\psi^{-1}(\wt{y})} \geq 2M^{-1}$.  Let also $\wh{p}_t(u,v)$ (resp.  $\wh{p}_t^U(u,v)$) denote the heat kernel for the LBM on $\s^2$ (resp.  $U$) with respect to $\wh{h}$ (resp.  $\wh{h}|_{U}$).  Then Theorem~\ref{thm:ubd} implies that there exists a deterministic constant $\kappa>0$ and there a.s.\  exist random constants $c_1,c_2$ such that
\begin{equation}\label{eqn:heat_kenrl_bd_upd}
\wh{p}_t^U(u,v) \leq \wh{p}_t (u,v)  \leq  \frac{c_1 (\log t^{-1})^\kappa}{t} \, \exp\Biggl( - c_2 \biggl(\frac{\qdist{
\wh{h}}{u}{v}^4}{t} \biggr)^{\!\frac 1 3} \biggl( \log\biggl( e + \frac{\qdist{\wh{h}}{u}{v}}{t} \biggr) \biggr)^{\!-\kappa} \Biggr)
\end{equation}
for all $u,v \in \s^2,  t \in (0,1/2]$.  Moreover,  by \cite[Theorem~1.3]{Be15},  $p_t^{\mathcal{D}}(u,v) = \wh{p}_t^{U}(\psi^{-1}(u),\psi^{-1}(v))$ for all $u,v \in \D,t>0$.  It follows that for $M$ sufficiently large and since $\psi^{-1}(\varphi^{-1}(A)) = \qball{\wh{h}}{\psi^{-1}(\wt{y})}{r}$,  we have that
\begin{align*}
 &f_2(M,r) \leq \int_{\varphi^{-1}(A)}  \int_{0}^{M^{-M/2} / (\log \tfrac{1}{r})^M} p_t^\CD(\wt{x},w) \, dt \, d\qmeasure{h}(w)\\
 &\leq c_1 \qmeasure{h}(\D) \int_0^{M^{-M/2} \log(r^{-1})^{-M}} \frac{\log(t^{-1})^{\kappa}}{t} \exp\left( - c_2 ((2M)^{-4} t^{-1})^{1/3} \log(e+(2Mt)^{-1})^{-\kappa}\right)dt 
\end{align*}
and the right-hand side tends to $0$ as $r \to 0$ faster than any positive power of $r$.  Letting $f := f_1+ f_2$ and $I_{M,r}:=[M^{-M/2} / (\log \tfrac{1}{r})^M,M^{M/2} (\log \tfrac{1}{r})^M]$, we thus have that for $M$ sufficiently large,
\[ f(M,r) + \int_A \int_{I_{M,r}} p_t^{\CN}(x,w) \, dt \, d\qmeasure{\wt{h}}(w) \geq E_x[ T_{A} ],\]
where $f(M,r) \to 0$ as $r \to 0$ faster than any positive power of $r$ by combining with Step 2.  Since $|I_{M,r}|$ is of order $M^{M/2} (\log \tfrac{1}{r})^M$,  by dividing both sides by $|I_{M,r}|$ and combining with~\eqref{eqn:lower_bound_on_time_spent_by_LBM} and~\eqref{eqn:quantum_area_lower_bound}, we see that for $r\in (0,M^{-1})$,
\begin{align*}
\frac 1 {| I_{M,r}|} \int_{I_{M,r}} \int_A p_t^{\CN}(x,w) \, d\qmeasure{\wt{h}}(w)   \, dt  \geq  \frac{ r^{\alpha +M/2}}{M^{M/2} (\log \tfrac{1}{r})^M} \geq r^{\alpha+3M/4} (\log \tfrac{1}{r})^{-M} \geq r^M,
\end{align*}
provided $M$ is sufficiently large.  Thus,  there exists $t \in I_{M,r}$ such that the claimed lower bound in the statement of the lemma holds.

\emph{Case 2.   $\qdist{\wh{h}}{\psi^{-1}(\wt{x})}{\psi^{-1}(\wt{y})} < 2 M^{-1}$.}
Suppose now that $\qdist{\wh{h}}{\psi^{-1}(\wt{x})}{\psi^{-1}(\wt{y})} < 2 M^{-1}$.  Note that we have already shown in Step 2 of the proof of Lemma~\ref{lem:lower_bound_on_mean_time_spent_by_LBM} that there exists a deterministic constant $\beta \in (0,\infty)$ and there a.s.  exist (random) $z_0 \in \D,  c_0,s>0$ such that $\eball{z_0}{s} \subseteq \D$ and the probability that a complex Brownian motion starting from $x$ intersects $\eball{z_0}{s/4}$ before exiting $\D$ for the first time is at least $c_0 r^{\beta}$.

Let $(X_t^{\D})$ denote the Liouville Brownian motion on $\D$ with respect to $\wt{h}$,  killed when it first exits $\D$.  Suppose first that $x \in \eball{z_0}{s/3}$ and set $T:=\inf\{t \geq 0 : X_t^{\D} \in \partial \eball{z_0}{s/2}\}$.  Possibly by taking $M$ to be larger,  we can assume that
\begin{align*}
\text{dist}_{d_{\wh{h}}}(\psi^{-1}(\varphi^{-1}(\eball{z_0}{s/3})),\psi^{-1}(\varphi^{-1}(\eball{z_0}{s/2}))) \geq 100 M^{-1}.
\end{align*}
Then we have that $\qdist{\wh{h}}{\psi^{-1}(\varphi^{-1}(X_T^{\D}))}{\psi^{-1}(\wt{y})} \geq 2 M^{-1}$.  Similarly if $x \notin \eball{z_0}{s/3}$,  we let $T$ be the first time that $X^{\D}$ intersects $\eball{z_0}{s/4}$ before exiting $\D$.  Then we have that $P[T < \infty] \geq c_0 r^{\beta}$ and $\qdist{\wh{h}}{\psi^{-1}(\varphi^{-1}(X_T^{\D}))}{\psi^{-1}(\wt{y})} \geq 2M^{-1}$,  if we take $M$ sufficiently large such that
\begin{align*}
\text{dist}_{d_{\wh{h}}}(\psi^{-1}(\varphi^{-1}(\eball{z_0}{s/4})),\psi^{-1}(\varphi^{-1}(\eball{z_0}{s/3}))) \geq 100 M^{-1}.
\end{align*}
Thus we have that 
\begin{align*}
\qdist{\wh{h}}{\psi^{-1}(\varphi^{-1}(X_T^{\D}))}{\psi^{-1}(\wt{y})} \geq 2M^{-1}
\end{align*}
in every case.  Also Lemma~\ref{lem:lower_bound_on_mean_time_spent_by_LBM} implies that
\begin{align*}
E_{X_T^{\D}}[T_A] \gtrsim r^{\alpha + M/4} \quad \text{a.s.  on} \quad \{T<\infty\}
\end{align*}
and so we obtain that
\begin{align*}
E_x \big[ \indicator_{\{T<\infty\}} E_{X^{\D}_T}\big[ T_{A} \big] \big] \gtrsim r^{\alpha + \beta + M/4}.
\end{align*}

Also we have that
\begin{align*}
&E_x \left[ \indicator_{\{M^{M/2}\log(r^{-1})^M / 2 \leq T < \infty\}} \int_0^{\infty} \indicator_{\{X^{\D}_{t+T} \in A\}} \, dt \right] \leq \int_{A} \int_{M^{M/2} \log(r^{-1})^M / 2}^{\infty} p_t^{\CN}(x,w) \, dt \, d\qmeasure{\wt{h}}(w)\\
& \mspace{36mu} \leq \sup_{z \in \D} \left(\int_{A} \int_{M^{M/2} \log(r^{-1})^M / 2}^{\infty} p_t^{\CN}(z,w) \, dt \, d\qmeasure{\wt{h}}(w)\right)
\end{align*}
which implies that 
\begin{align*}
r^{\alpha + \beta  + M/4} \lesssim \sup_{z \in \D} \left( \int_{A} \int_{M^{M/2}\log(r^{-1})^M / 2}^{\infty} p_t^{\CN}(z,w) \, dt \, d\qmeasure{\wt{h}}(w)\right) + I_1+I_2+I_3,
\end{align*}
where
\begin{align*}
& I_1 =E_x\left[ \indicator_{\{T \leq M^{M/2} \log(r^{-1})^M / 2\}} \int_0^{M^{-M/2} \log(r^{-1})^M} \indicator_{\{X^{\D}_{t+T} \in A\}} \, dt \right], \\
& I_2 = E_x\left[ \indicator_{\{T \leq M^{M/2} \log(r^{-1})^M / 2\}} \int_{M^{-M/2} \log(r^{-1})^M}^{M^{M/2} \log(r^{-1})^M} \indicator_{\{X^{\D}_{t+T} \in A\}} \, dt \right], \\
& I_3 = E_x\left[ \indicator_{\{T \leq M^{M/2} \log(r^{-1})^M / 2\}} \int_{M^{M/2} \log(r^{-1})^M}^{\infty} \indicator_{\{X^{\D}_{t+T} \in A\}} \, dt \right].
\end{align*}
Moreover,  we have that
\begin{align*}
I_1 + I_3 &\leq E_x \left[ \indicator_{\{T<\infty\}} \int_{A} \int_0^{M^{-M/2} \log(r^{-1})^M} p_t^{\CN}(X^{\D}_T,w) dtd\qmeasure{\wt{h}}(w)\right]\\
&\qquad+E_x\left[ \int_{M^{M/2} \log(r^{-1})^M}^{\infty} \indicator_{\{X^{\D}_t \in A\}} \, dt \right]\\
&\lesssim \int_0^{M^{-M/2} \log(r^{-1})^{-M}} \frac{\log(r^{-1})^{\kappa}}{t} \exp(-c_2 (2Mt)^{-1/3})dt\\
&\qquad+\int_{A} \int_{M^{M/2} \log(r^{-1})^M}^{\infty} p_t^{\CN}(x,w) \, dt \, \qmeasure{\wt{h}}(w).
\end{align*}
Thus,  by arguing as in Step 2,  we obtain that $I_1 + I_3 \to 0$ as $r \to 0$ faster than any positive power of $r$.  Furthermore,  we have that
\begin{align*}
I_2 \leq \int_{A} \int_{M^{M/2} \log(r^{-1})^M / 2 + M^{-M/2} \log(r^{-1})^{-M}}^{3M^{M/2} \log(r^{-1})^M / 2} p_t^{\CN}(x,w) \,dt \, d\qmeasure{\wt{h}}(w).
\end{align*}
Hence,  by arguing as in Case 1,  we get that there exists $t\in [ M^{-M/2} \log(r^{-1})^{-M}, M^{M/2} \log(r^{-1})^M]$ such that $\int_A p_t^{\CN}(x,w)d\qmeasure{\wt{h}}(w) \geq r^M$.
\end{proof}

We next record the following which will be used in Subsection~\ref{subsec:good_annuli} (see in particular the proof of Lemma~\ref{lem:move_chunks}) in order to bound from below the probabilities that the Liouville Brownian motion intersects certain fixed LQG metric balls.  We chose to state and prove the lemma at this point since its proof follows from the same argument used to prove Lemma~\ref{lem:hitting_bound}.

\begin{lemma}
\label{lem:apriorilbd}
For every $M>1$ there exist constants $K_1,K_2 > 1$, depending on $M$,  so that the following is true.  For $\qsphereinflaw$ a.e.\ instance of $(\CS,h,x,y)$ there exists $\Delta_0 > 0$ so that for all $\delta \in (0,\Delta_0)$ and $u,v \in \s^2$ with $\qdist{h}{u}{v} \leq \delta$ there exists $t \in [\delta^{K_1},\delta^{1/K_1}]$ so that
\[ P_u[ X_t \in \qball{h}{v}{\delta^{M}}] \geq \delta^{K_2},\]
where $X$ denotes Liouville Brownian motion and $P_u$ is the law under which $X$ starts from $u$.
\end{lemma}

\begin{proof}
This follows from the same argument used to prove Lemma~\ref{lem:hitting_bound}.
\end{proof}

\begin{proof}[Proof of Proposition~\ref{prop:good_event_happens}]
First we note that by scaling and as in the proof of Proposition~\ref{prop:good_sle_chunks},  it suffices to prove the claim of the proposition in the case that $\delta = 1$.  Moreover,  combining Lemma~\ref{lem:hitting_bound} with the argument in the first paragraph of the proof of Lemma~\ref{lem:reverse_holder_continuity} in order to compare locally the laws of a quantum disk and a quantum wedge of weight $2$,  we obtain that condition~\eqref{it:hitting_bound} occurs with as high probability as we want provided we choose $\mathcal{M}_0$ sufficiently large.  Furthermore,  combining Lemma~\ref{lem:metric_non_boundary_tracing} with the arguments in Steps 3 and 4 in the proof of Lemma~\ref{lem:reverse_holder_continuity_top_boundary} and the argument in the second paragraph of the proof of Lemma~\ref{lem:reverse_holder_continuity},  we obtain that condition~\eqref{it:ballnearby} holds with high probability as well provided we choose $\mathcal{M}_0$ large enough.  Therefore,  the proof is complete by combining with Proposition~\ref{prop:good_sle_chunks}. 
\end{proof}

\subsection{Definition and properties of the good annuli}
\label{subsec:good_annuli}

The main goal of this section is to prove the following proposition.

\begin{proposition}
\label{prop:cover_by_good_annuli}
There exists a constant $\kappa > 0$ so that for $\qsphereinflaw$-a.e.\ $(\CS,h,x,y)$ there exists $\Delta_0 > 0$ so that for all $\delta \in (0,\Delta_0)$ the following is true.  For every $z \in \CS$ there exists an annulus $\CA$ with the following properties.
\begin{enumerate}[(i)]
\item\label{it:good_annulus_contained} $\CA$ is contained in $\qball{h}{z}{\delta^{1/3}(\log \delta^{-1})^\kappa} \setminus \qball{h}{z}{\delta^{1/3} (\log \delta^{-1})^{-\kappa}}$.
\item\label{it:distance} The $\qdistnoarg{h}$-distance between the inner and outer boundaries of $\CA$ is at least $\delta^{1/3} (\log \delta^{-1})^{-\kappa}$.
\item\label{it:move} If $u,v \in \CA$ have $\qdistnoarg{h}$-distance at least $\delta^{1/3} (\log \delta^{-1})^{-\kappa}/2$ from $\partial \CA$, then there exists an element $s$ of $[\delta^{4/3} (\log \delta^{-1})^{-\kappa}, \delta^{4/3} (\log \delta^{-1})^\kappa]$ so that 
\[ \int_{\qball{h}{v}{\delta^{1/3} (\log \delta^{-1})^{-\kappa}}} p_s^\CA(u,a) \, d\qmeasure{h}(a) \geq \exp(- (\log \delta^{-1})^\kappa)\]
where $p^\CA$ denotes the heat kernel for Liouville Brownian motion on $\CA$.
\item\label{it:move_into_and_out} For each $u \in \CA$ there exists $s \leq \delta^{1/\kappa}$ so that
\begin{align*}
\int_{\qball{h}{u}{\delta^{1/3}}} p_s(z,a) \,  d\qmeasure{h}(a) &\geq \exp(-(\log \delta^{-1})^\kappa) \quad\text{and}\\
\int_{\qball{h}{u}{\delta^{1/3}}} p_s(a,z) \, d\qmeasure{h}(a) &\geq \exp(-(\log \delta^{-1})^\kappa).
\end{align*}
\end{enumerate}
\end{proposition}

As explained just after the proof of Lemma~\ref{lem:on_diagonal_lbd},  we will use Proposition~\ref{prop:cover_by_good_annuli} in order to prove that the following is true for $\qsphereinflaw$-a.e.\ $(\CS,h,x,y)$.  Fix $u,v \in \CS$ distinct points and $0<t \lesssim \qdist{h}{u}{v}$.  Then Proposition~\ref{prop:cover_by_good_annuli} implies that we can find a finite and connected chain of topological annuli $\CA_1,\cdots,\CA_N$ connecting $u$ to $v$ and each of them satisfying properties~\eqref{it:good_annulus_contained}-\eqref{it:move_into_and_out} and consisting of unions of good chunks in the sense of Proposition~\ref{prop:good_event_happens}.  Moreover we have that $N \asymp \bigl(\frac{\qdist{h}{u}{v}^4}{t}\bigr)^{1/3}$ and $\diam(\CA_j) \asymp \bigl(\frac{t}{\qdist{h}{u}{v}}\bigr)^{1/3}$ for all $1 \leq j \leq N$,  and the Brownian motion can move between the annuli with uniformly positive probability.  The reason for introducing  properties~\eqref{it:good_annulus_contained} -~\eqref{it:move_into_and_out} is that they will ensure that the required bounds on the heat kernel of the Liouville Brownian motion hold and hence complete the proof of Theorem~\ref{thm:lbd} in Subsection~\ref{subsec:rest_of_proof}.

We will focus on proving Proposition~\ref{prop:cover_by_good_annuli} for the rest of the section.
From now on we fix $M \geq 1$ so that $p_0 \in (0,1)$ from Proposition~\ref{prop:good_event_happens} is large enough so that the assertion of Proposition~\ref{prop:good_chunks_percolate} holds.  

We start by describing the setup of the proof of Proposition~\ref{prop:cover_by_good_annuli}. Suppose that $(\CS,h,x,y)$ has distribution $\qsphereinflaw$,  let $\kappa >0$ be the constant from the statement of Lemma~\ref{lem:filled_metric_ball_exit_time} and fix $u \in (0,1/3)$.  Fix also $r_0 \in (0,1)$ and from now on,  we assume that we are working on the event that 
\begin{align*}
r^4 \log(r^{-1})^{-6-u} \leq \qmeasure{h}(\qball{h}{z}{r}) \leq r^4 \log(r^{-1})^{8+u}, \quad  \frac{\qmeasure{h}(\qball{h}{z}{r})}{\qmeasure{h}(\CS)} \geq r^{4+u} \quad \text{and} \quad \diam(\CS) \geq 6r
\end{align*}
for all $z \in \CS,  r \in (0,r_0)$,  where $(\CS,h,x,y)$ has law $\qsphereinflaw$.  We note that Theorem~\ref{thm:ball_concentration} implies that we can find such $r_0$ satisfying the above properties for $\qsphereinflaw$-a.e.\  instance $(\CS,h,x,y)$.  Let $k \in \N$ be sufficiently large such that $2^{-k} < r_0$ and set $N_k = 2^{4k + u}$.  Let $(z_j)$ be a sequence chosen i.i.d.\ from $\qmeasure{h}$.  Then,  we know from the proof of Lemma~\ref{lem:filled_metric_ball_exit_time} that $\CS \subseteq \cup_{j=1}^{N_k} \qball{h}{z_j}{2^{-k}}$ off an event whose $\qsphereinflaw$ measure tends to $0$ as $k \to \infty$ faster than any power of $2^{-k}$.  We set $\delta_k = 2^{-3k}$ and for all $1 \leq i,j \leq N_k$ such that $\qdist{h}{z_i}{z_j} \geq 2^{-k/2}$,  suppose that we define stopping times $\tau_m^{i,j}$ and events $A_1^{i,j},\ldots,A_m^{i,j},B_1^{i,j},\ldots,B_m^{i,j}$ as in the proof of Lemma~\ref{lem:filled_metric_ball_exit_time} with $r = \delta_k^{1/3}$.  We also define collections of chunks $\CN_1^{i,m,j},\ldots,\CN_N^{i,m,j}$ in a similar way,  where $z_i$ plays the role of $x$ in the proof of Lemma~\ref{lem:filled_metric_ball_exit_time}.  Then,  arguing in the same way as in the proof of Lemma~\ref{lem:filled_metric_ball_exit_time},  we obtain that off an event whose $\qsphereinflaw$ measure tends to $0$ as $k \to \infty$ faster than any power of $2^{-k}$,  it holds that we can find an annulus $\CA_{i,j,k}$ consisting only of good chunks such that
\begin{align*}
 \CA_{i,j,k} \subseteq \qball{h}{z_i}{2\delta_k^{1/3}}\setminus \qball{h}{z_i}{\delta_k^{1/3}},
\end{align*}
$\CA_{i,j,k}$  disconnects $\qball{h}{z_i}{\delta_k^{1/3}}$ from $z_i$,  and the distance between the inner and outer boundaries of $\CA_{i,j,k}$ with respect to the interior-internal metric in $\CA_{i,j,k}$ is at least $\delta_k^{1/3} \log(\delta_k^{-1})^{-\kappa}$.  Therefore by the Borel-Cantelli lemma,  we have that $\qsphereinflaw$-a.e.,  there exists $\cK_0 \in \N$ such that $k \geq \cK_0$ implies that $\CA_{i,j,k}$ satisfies both~\eqref{it:good_annulus_contained} and~\eqref{it:distance} from Proposition~\ref{prop:cover_by_good_annuli}.  We now show that the $\CA_{i,j,k}$ satisfy the other properties.

We start by proving that condition~\eqref{it:move}  holds with high probability.  This will follow from applying a union bound using the Borel-Cantelli lemma and the following lemma combined with the fact that the number of chunks in $\CA_{i,j,k}$ is at most $\log(\delta_k^{-1})^c$ for some constant $c>0$ by~\eqref{it:good_chunks_percolate_round_number}.

\begin{lemma}
\label{lem:move_chunks}
Fix $b>0$.  Then there exists a deterministic constant $c> 0$, depending on $M$ and $b$ so that the following is true off an event whose $\qsphereinflaw$ measure tends to $0$ as $k \to \infty$. Suppose that we have the setup described just after the statement of Proposition~\ref{prop:cover_by_good_annuli}.  Fix $k \in \N$ and suppose that we are working on the event that $k \geq \cK_0$ and $\wt{\delta}_k:= \delta_k \log(\delta_k^{-1})^{-b}  \leq \Delta_0$ where $\Delta_0$ is as in Lemma~\ref{lem:apriorilbd}.  Fix $1 \leq i,j \leq N_k$ and further suppose we are working on the event that $\qdist{h}{z_i}{z_j} \geq 2^{-k/2}$.  Suppose that $\CN_1,\CN_2$ are two adjacent chunks in $\CA = \CA_{i,j,k}$ for which the event $E$ occurs and let $\CN$ be the quantum surface parameterized by the interior of $\ol{\CN}_1 \cup \ol{\CN}_2$.  For each $u,v \in \CN$ with distance at least $\wt{\delta}_k^{1/3}/M$ from $\partial \CN$ there exists $s \in [M^{-M} \wt{\delta}_k^{4/3} / (\log \wt{\delta}_k^{-1})^M, M^M \wt{\delta}_k^{4/3} (\log \wt{\delta}_k^{-1})^M]$ so that
\begin{align*}
 P_u\big[ X_s \in \qball{h}{v}{\wt{\delta}_k^{1/3} M^{-1}}\big] \geq \wt{\delta}_k^{c},
\end{align*}
where $X$ denotes Liouville Brownian motion and $P_u$ is the law under which $X$ starts from $u$.

\end{lemma}
\begin{proof}
Suppose that $\CA$, $\CN_1$, $\CN_2$, $\CN$ are as in the statement of the lemma.  We also let $K_1,K_2$ be as in the statement of Lemma~\ref{lem:apriorilbd}.  We take $p = 8 K_1/3$ so that $\wt{\delta}_k^{p/K_1} = \wt{\delta}_k^{8/3}$ is of lower order than $\wt{\delta}_k^{4/3}$.  Suppose that $u,v \in \CN$ both have distance at least $\wt{\delta}_k^{1/3}/M$ from $\partial \CN$.  We may assume without loss of generality that $u \in \CN_1$.  By condition~\eqref{it:ballnearby} in the definition of $E$ for the chunk $\CN_1$, there exists $u_1 \in \CN_1$ such that $\qdist{h}{u}{u_1} \leq \wt{\delta}_k^p$ and $\qball{h}{u_1}{2\wt{\delta}_k^{p M}} \subseteq \CN_1$.  Let $u_2 \in \CN_1$ have distance $\wt{\delta}_k^p$ from $\partial \CN_2$ with $\qball{h}{u_2}{\wt{\delta}_k^{pM}} \subseteq \CN_1$ and let $v_2 \in \CN_2$ be such that $\qdist{h}{u_2}{v_2} \leq 2\wt{\delta}_k^p$ with $\qball{h}{v_2}{2\wt{\delta}_k^{p M}} \subseteq \CN_2$ (we apply again  condition~\eqref{it:ballnearby} in the definition of $E$ but for the chunk $\CN_2$ in place of $\CN_1$).  If $v \in \CN_1$,  then the claim of the lemma follows from condition~\eqref{it:hitting_bound} in the definition of $E$ for the chunk $\CN_1$. If $v \in \CN_2$,  then for $s_1,s_2,s_3,s_4 > 0$ and $s = s_1 + s_2 + s_3+s_4$, we have that $P_u[ X_s \in \qball{h}{v}{\wt{\delta}_k^{1/3} M^{-1}}] \geq p_1 p_2 p_3 p_4$ where 
\begin{align*}
p_1 &= P_u[ X_{s_1} \in \qball{h}{u_1}{\wt{\delta}_k^{p M}}],\\
p_2 &= \inf_{w \in \qball{h}{u_1}{\wt{\delta}_k^{p M}}} P_w[ X_{s_2} \in \qball{h}{u_2}{\wt{\delta}_k^{pM}}],\\
p_3 &= \inf_{w \in \qball{h}{u_2}{\wt{\delta}_k^{pM}}} P_w[ X_{s_3} \in \qball{h}{v_2}{\wt{\delta}_k^{pM}}],\\
p_4 &= \inf_{w \in \qball{h}{v_2}{\wt{\delta}_k^{pM}}} P_w[ X_{s_4} \in \qball{h}{v}{\wt{\delta}_k^{1/3}/M}].
\end{align*}
By Lemma~\ref{lem:apriorilbd}, there exists a choice of $s_1,s_3 \in [\wt{\delta}_k^{p K_1}, \wt{\delta}_k^{p/K_1}]$ so that $p_1,p_3 \geq \wt{\delta}_k^{p K_2}$. Furthermore,  by condition~\eqref{it:hitting_bound} in the definition of the event $E$ for both chunks $\CN_1$ and $\CN_2$, applied with $r=\wt{\delta}_k^{pM} < \wt{\delta}_k^{1/3}/M$, we  can find $s_2,s_4 \in [M^{-M} \wt{\delta}_k^{4/3} /(\log \wt{\delta}_k^{-1})^M,M^M \wt{\delta}_k^{4/3} (\log \wt{\delta}_k^{-1})^M]$ so that $p_2, p_4 \geq \wt{\delta}_k^{pM^2}$. Altogether, this implies that there exists $s \in [ \tfrac{1}{2} M^{-M} \wt{\delta}_k^{4/3} / (\log \wt{\delta}_k^{-1})^M, 2M^M \wt{\delta}_k^{4/3} (\log \wt{\delta}_k^{-1})^M]$ and a constant $c > 0$ depending on $M$ and $b$  so that $P_u[ X_s \in \qball{h}{v}{\wt{\delta}_k^{1/3} M^{-1}}] \geq \wt{\delta}_k^c$.
\end{proof}

Next we focus on proving that condition~\eqref{it:move_into_and_out}  holds with high probability.  The main ingredient of the proof of the claim is the following lemma which states that with probability tending to $1$ as $k \to \infty$,  we have that condition~\eqref{it:move_into_and_out}  holds for all $\CA_{i,j,k}$ with $\qdist{h}{z_i}{z_j} \geq 2^{-k/2}$ and $\qfb{h}{z_j}{z_i}{\delta_k^{1/3} \log(\delta_k^{-1})^{\kappa}} \subseteq \qball{h}{z_i}{\delta_k^{1/\kappa_0}}$ for some fixed and deterministic constant $\kappa_0 \in (0,\infty)$.  

\begin{lemma}
\label{lem:into_good_cluster}
There exists a constant $\wt{c}>0$ such that the following is true.  Suppose that we have the setup described just after the statement of Proposition~\ref{prop:cover_by_good_annuli}.  Then,  off an event whose $\qsphereinflaw$ measure tends to $0$ as $k \to \infty$,  we have that the following holds.  Fix $1 \leq i,j  \leq N_k$ and further suppose that we are working on the event that $k \in \cK_0,  \qdist{h}{z_i}{z_j} \geq 2^{-k/2}$ and $\qfb{h}{z_j}{z_i}{\delta_k^{1/3} \log(\delta_k^{-1})^{\kappa}} \subseteq \qball{h}{z_i}{\delta_k^{1/\kappa_0}}$ for some fixed and deterministic constant $\kappa_0 \in (0,\infty)$.  Then,  for all $z \in \qball{h}{z_i}{\delta_k^{1/3} \log(\delta_k^{-1})^{-\kappa}}$,  there exists $s \in \bigl(0, \delta_k^{4/\wt{c}} \log(\delta_k^{-1})^{\wt{c}} \bigr]$ such that
\begin{align*}
\int_{\qball{h}{u}{\delta_k^{1/3}}} p_s(z,a)  \, d\qmeasure{h}(a) &\geq \exp(-(\log \delta_k^{-1})^{\wt{c}}) \quad\text{and}\\
\int_{\qball{h}{u}{\delta_k^{1/3}}} p_s(a,z) \, d\qmeasure{h}(a) &\geq \exp(-(\log \delta_k^{-1})^{\wt{c}})
\end{align*}
for all $u \in \CA:=\CA_{i,j,k}$.
\end{lemma}

\begin{proof}
Since $k \in \cK_0$,  we have that
\begin{align*}
\CA \subseteq \qfb{h}{z_j}{z_i}{\delta_k^{1/3} \log(\delta_k^{-1})^{\kappa}} \subseteq \qball{h}{z_i}{\delta_k^{1/\kappa_0}} \subseteq \qball{h}{z}{\delta_k^{1/\kappa_1}},
\end{align*}
where $0<\kappa_1 < \kappa_0$ is fixed and deterministic.  Let $X$ be the LBM and let $\tau = \inf\{t \geq 0 : X_t \in \CA\}$.  Then,  under $P_z$,  $\tau$ is clearly bounded from above by the first exit time from $\qball{h}{z}{\delta_k^{1/\kappa_1}}$ so that Theorem~\ref{thm:bm_exit} implies that $E_w\big[ \tau \big] \leq \delta_k^{4/\kappa_1} \log(\delta_k^{-1})^{\kappa_2}$ for all $w \in \qball{h}{z}{\delta_k^{1/\kappa_1}}$,  for some fixed and deterministic constant $\kappa_2 \in (0,\infty)$.  In particular,  if $c_1 > \kappa_2 +1$,  we have that $P_z\big[ \tau \geq \delta_k^{4/\kappa_1} \log(\delta_k^{-1})^{c_1} \big] \to 0$ as $k \to \infty$ faster than any power of $2^{-k}$.  Indeed,  Markov's inequality implies that if $\kappa_3 \in (\kappa_2,c_1-1)$ is fixed and deterministic,  then 
\begin{align*}
P_z \big[ \tau \geq \delta_k^{4/\kappa_1} \log(\delta_k^{-1})^{\kappa_3} \big] \leq \delta_k^{-4/\kappa_1} \log(\delta_k^{-1})^{-\kappa_3} E_z\big[ \tau \big] \leq \log(\delta_k^{-1})^{\kappa_2-\kappa_3}.
\end{align*}
Hence,  the Markov property of $X$ implies that 
\begin{align*}
P_z\big[ \tau \geq \delta_k^{4/\kappa_1} \log(\delta_k^{-1})^{c_1} \big] \leq \exp(-(\kappa_3-\kappa_2) \log(\delta_k^{-1})^{c_1-\kappa_3+1}),
\end{align*}
which proves the claim.  Once $X$ hits $\CA$,  we know from Lemma~\ref{lem:move_chunks} by taking $c_1$ to be larger if necessary and depending only on $\kappa$ and $M$,  that $X$ can move from one chunk in $\CA$ to an adjacent chunk in $\CA$ in time between a constant depending only on $M$ times $\delta_k^{4/3} \log(\delta_k^{-1})^{-c_1}$ and a constant depending only on $M$ times $\delta_k^{4/3} \log(\delta_k^{-1})^{c_1}$ with probability at least $\delta_k^{c_1}$.  Note that the number of chunks which make up $\CA$ is at most $\log(\delta_k^{-1})^{\kappa (2/3 - u)}$ which is at most $\log(\delta_k^{-1})^{c_1}$ by taking $c_1$ to be larger if necessary,  where $u \in (0,1/3)$ is as in the statement of Proposition~\ref{prop:good_chunks_percolate}.  Hence,  we can find $c_2>c_1$ fixed and deterministic,  depending only on $c_1$ and $c$ (where $c$ is the constant in the statement of Lemma~\ref{lem:move_chunks}),  such that $X$ hits $\qball{h}{u}{\delta_k^{1/3} / M}$ in time at most $(\delta_k^{4/\kappa_1} + \delta_k^{4/3}) \log(\delta_k^{-1})^{c_2}$ with probability at least $\exp(-\log(\delta_k^{-1})^{c_2})$.  Suppose that $w \in \qball{h}{u}{\delta_k^{1/3} / M}$ and $t \in [\delta_k^{1/c_2},1/2]$.  Then,  by taking $c_2$ to be larger if necessary,  we can assume using Lemma~\ref{lem:on_diagonal_lbd} that there exists a random constant $C\geq 1$ such that $p_t(w,w) \geq \frac{1}{C t \log(\delta_k^{-1})^{c_2}}$. 

Suppose that $t' = \delta_k^{4/3} \log(t^{-1})^{-2c_2}$.  Note that 
\begin{align*}
p_t(w,w) = \int_{\CS} p_{t-t'}(w,a) p_{t'}(a,w) d\qmeasure{h}(a).
\end{align*}
Then,  Theorem~\ref{thm:ubd} implies that by taking $c_2$ to be larger if necessary,  there exist random constants $A_1,A_2>0$ such that for all $a \in \CS \setminus \qball{h}{u}{\delta_k^{1/3}}$,  we have that
\begin{align*}
p_{t'}(a,w)\leq \frac{A_1 \log((t')^{-1})^{c_2}}{t'} \exp\Biggl( -A_2 \biggl( \frac{\qdist{h}{a}{w}^4}{t'}\biggr)^{1/3} \Biggr)
\end{align*}
and so 
\begin{align*}
\int_{\CS \setminus \qball{h}{u}{\delta_k^{1/3}}} p_{t-t'}(w,a) p_{t'}(a,w) d\qmeasure{h}(a) \to 0
\end{align*}
as $k \to \infty$,  faster than any power of $2^{-k}$.  Therefore,  by increasing $C$ if necessary,  we have that 
\begin{align*}
\int_{\qball{h}{u}{\delta_k^{1/3}}} p_{t-t'}(w,a) p_{t'}(w,a) d\qmeasure{h}(a) \geq \frac{1}{Ct\log(t^{-1})^{c_2}}.
\end{align*}
Applying again Theorem~\ref{thm:ubd} to $p_{t'}(w,a)$,  we obtain that
\begin{align*}
\frac{C \log((t')^{-1})^{c_2}}{t'} \int_{\qball{h}{u}{\delta_k^{1/3}}} p_{t-t'}(w,a) d\qmeasure{h}(a) \geq \frac{1}{C t \log(t^{-1})^{c_2}}.
\end{align*}
After possibly increasing $C$,  we obtain that 
\begin{align*}
\int_{\qball{h}{u}{\delta_k^{1/3}}} p_{t-t'}(w,a) d\qmeasure{h}(a) \geq \frac{t'}{Ct\log(t^{-1})^{c_2}}.
\end{align*}
Note that for all $t \in [\delta_k^{4/c_2}, 2 \delta_k^{4/c_2}]$,  we have $t' \gtrsim \delta_k^{4/3} \log(\delta_k^{-1})^{-2c_2}$ and $t \log(t^{-1})^{2c_2} \lesssim \delta_k^{4/c_2} \log(\delta_k^{-1})^{c_2}$ with the implicit constants depending only on $c_2$,  and so 
\begin{align*}
I:=\int_{\delta_k^{4/c_2}}^{2\delta_k^{4/c_2}} \int_{\qball{h}{u}{\delta_k^{1/3}}} p_{t-t'}(w,a) d\qmeasure{h}(a) \gtrsim \exp(-\log(\delta_k^{-1})^{c_2})
\end{align*}
by taking $c_2$ to be larger if necessary.  It follows that the expected amount of time that the LBM $X$ starting from $z$ spends in $\qball{h}{u}{\delta_k^{1/3}}$ in the time interval $[0,\delta_k^{4/c_2} \log(\delta_k^{-1})^{c_2}]$ is at least $\exp(-2\log(\delta_k^{-1})^{c_2})$.  Therefore,  taking $c_3>c_2$ depending only on $c_2$ and dividing by $\delta^{4/c_2} \log(\delta_k^{-1})^{c_2}$,  we obtain that there exists $s \in \bigl(0, \delta_k^{4/c_3} \log(\delta_k^{-1})^{c_3} \bigr]$ such that
\begin{align*}
\int_{\qball{h}{u}{\delta_k^{1/3}}} p_s(z,a) d\qmeasure{h}(a) \geq \exp(-\log(\delta_k^{-1})^{c_3}).
\end{align*}
This proves the first inequality of the lemma.  The second inequality follows from the symmetry of the heat kernel and the first inequality.
\end{proof}

Recall that in order to be able to apply Lemma~\ref{lem:into_good_cluster},  we need to have with high probability as $k \to \infty$ that
\begin{align*}
\qfb{h}{z_j}{z_i}{\delta_k^{1/3} \log(\delta_k^{-1})^{\kappa}} \subseteq \qball{h}{z_i}{\delta_k^{1/\kappa_0}}
\end{align*}
whenever $\qdist{h}{z_i}{z_j} \geq 2^{-k/2}$ for some fixed and deterministic constant $\kappa_0 \in (0,\infty)$.  This will be a consequence of the following lemma together with the H\"older continuity of $d_h$ with respect to the Euclidean metric.

\begin{lemma}
\label{lem:isoperimetry}
There exists a deterministic constant $p > 0$ so that for $\qsphereinflaw$-a.e.\ instance of $(\CS,h,x,y)$ there exists $\epsilon_0 > 0$ so that for all $\epsilon \in (0,\epsilon_0)$ the following is true.  For every set $S \subseteq \CS$ with $\qdistnoarg{h}$-diameter at most $\epsilon^p$ there exists at most one connected component of $\CS \setminus S$ which contains a $\qdistnoarg{h}$-ball of radius $\epsilon$.
\end{lemma}
\begin{proof}
We can assume that $(\CS,h,x,y)$ is parameterized by $\s^2$.  Recall that the metric $\qdistnoarg{h}$ is H\"older continuous with respect to the Euclidean metric on $\s^2$.  This implies that there exists $\alpha \in (0,1)$ deterministic and $C \geq 1$ so that for all $u,v \in \CS$ we have that
\begin{equation}
\label{eqn:holder}
C^{-1} d(u,v)^{1/\alpha} \leq \qdist{h}{u}{v} \leq C d(u,v)^\alpha,
\end{equation}
where $d$ denotes the Euclidean metric on $\s^2$.  Fix $\epsilon, p > 0$ and suppose that $S \subseteq \CS$ has $\qdistnoarg{h}$-diameter at most $\epsilon^p$.  Then~\eqref{eqn:holder} implies that $S$ has Euclidean diameter at most $C \epsilon^{\alpha p}$.  All components of $\s^2 \setminus S$ except one have Euclidean diameter at most $C \epsilon^{\alpha p}$.  Therefore by~\eqref{eqn:holder} all components of $\s^2 \setminus S$ except one have $\qdistnoarg{h}$-diameter at most $C^{1+\alpha} \epsilon^{\alpha^2 p}$.  The result thus follows by assuming that $p$ is sufficiently large so that $\alpha^2 p > 1$.
\end{proof}

\begin{proof}[Proof of Proposition~\ref{prop:cover_by_good_annuli}]
Fix $\delta > 0$ and suppose that $(z_j)$ is an i.i.d.  sequence of points in $\CS$ chosen independently from $\qmeasure{h}$.  We explained after the statement of Proposition~\ref{prop:cover_by_good_annuli} why for $k \geq \cK_0$ and $\qdist{h}{z_i}{z_j} \geq 2^{-k/2}$ we have that $\CA_{i,j,k}$ satisfies properties~\eqref{it:good_annulus_contained} and~\eqref{it:distance}.  Lemma~\ref{lem:move_chunks} and the fact that the number of chunks in $\CA_{i,j,k}$ is at most $\log(\delta_k^{-1})^c$ for some constant $c>0$ by~\eqref{it:good_chunks_percolate_round_number} together with another union bound and the Borel-Cantelli lemma implies that, after possibly increasing the value of $\cK_0$, we have that~\eqref{it:move} also holds for each such $\CA_{i,j,k}$ by possibly increasing $\kappa$.  Lemma~\ref{lem:isoperimetry} together with the H\"older continuity of $\qdistnoarg{h}$ with respect to the Euclidean metric implies that by possibly increasing the values of $\cK_0$ and $\kappa$,  we have that if $\qdist{h}{z_i}{z_j} \geq 2^{-k/2}$  then $\qfb{h}{z_j}{z_i}{2^{-k} k^\kappa} \subseteq \qball{h}{z_i}{2^{-k/\kappa}}$. It thus follows from Lemma~\ref{lem:into_good_cluster} that by possibly increasing $\cK_0$ and $\kappa$ further we have that~\eqref{it:move_into_and_out} holds for each $\CA_{i,j,k}$ with $\qdist{h}{z_i}{z_j} \geq 2^{-k/2}$.

By possibly increasing $\cK_0$ further, for every $1 \leq i \leq N_k$ there exists $1 \leq j \leq N_k$ so that $\qdist{h}{z_i}{z_j} \geq 2^{-k/2}$.  For $z \in \qball{h}{z_i}{2^{-k}}$ we take $\CA = \CA_{i,j,k}$ (breaking ties in an arbitrary way).  This choice of $\CA$ then satisfies properties~\eqref{it:good_annulus_contained}, \eqref{it:distance}, \eqref{it:move}, and~\eqref{it:move_into_and_out}.
\end{proof}

\subsection{Completion of the proof}
\label{subsec:rest_of_proof}

We now turn to complete the proof of Theorem~\ref{thm:lbd}. 

\begin{proof}[Proof of Theorem~\ref{thm:lbd}]
Suppose that $(\CS,h,x,y)$ has distribution $\qsphereinflaw$.  Fix $u,v \in \CS$ and let $k \in \Z$ be such that $2^k \leq \qdist{h}{u}{v} < 2^{k+1}$.  Fix $t \in (0, \frac  1 2\Delta_0\qdist{h}{u}{v})$  where $\Delta_0 > 0$ is as in Proposition~\ref{prop:cover_by_good_annuli}.  For each $j \in \Z$ we let $\delta_j = t 2^{-j}$, in particular $\delta_k \leq \Delta_0$.   Let $\gamma$ be a geodesic from $u$ to $v$.  For each $j\geq k$ we consider the following chain of annuli.  We let $z_{0,j} = u = \gamma(0)$ and we let $\CA_{0,j}$ be an annulus which is centered at $z_{0,j}$ and satisfies the properties from Proposition~\ref{prop:cover_by_good_annuli} with parameter $\delta_j$ and let $\kappa>0$ be the constant from the statement of Proposition~\ref{prop:cover_by_good_annuli}. Given that we have defined $z_{0,j},\ldots,z_{n,j}$ we let $z_{n+1,j}$ be a point on the interval of $\gamma$ from $z_{n,j}$ to $v$ which is in $\CA_{n,j}$ and has $\qdistnoarg{h}$-distance at least $\delta_j^{1/3} (\log \delta_j^{-1})^{-\kappa}$ from $\partial \CA_{n,j}$.  Note that $z_{n+1,j}$ is well-defined since by possibly increasing $\kappa$,  we can assume that the $\qdistnoarg{h}$-distance between the inner and outer boundaries of $\CA_{n,j}$ is at least $2 \delta_j^{1/3} \log(\delta_j^{-1})^{-\kappa}$,  by condition~\eqref{it:distance} in the statement of Proposition~\ref{prop:cover_by_good_annuli}. We then let $\CA_{n+1,j}$ be an annulus which satisfies the properties from Proposition~\ref{prop:cover_by_good_annuli} centered at the point $z_{n+1,j}$ with parameter $\delta_j$.  We continue this until we find the point $z_{N_j,j}$ so that $\qdist{h}{z_{N_j,j}}{v} \leq \delta_j^{1/3} (\log \delta_j^{-1})^{-\kappa}$.  Since we have that $\delta_j^{1/3} (\log \delta_j^{-1})^{-\kappa} \leq \qdist{h}{z_{i,j}}{z_{i+1,j}} \leq \delta_j^{1/3} (\log \delta_j^{-1})^{\kappa}$, it follows that $\qdist{h}{u}{v}/(\delta_j^{1/3} (\log \delta_j^{-1})^{\kappa}) \leq N_j \leq \qdist{h}{u}{v}/(\delta_j^{1/3} (\log \delta_j^{-1})^{-\kappa})$.

We take $j_0 \in \Z$ to be the smallest $j \in \Z$ so that $\delta_j^{1/\kappa} \leq \delta_k^{4/3}$.  We note that $\log \delta_j^{-1}$ is comparable to $\log \delta_k^{-1}$ for $k \leq j \leq j_0$.  Therefore we will phrase estimates from Proposition~\ref{prop:cover_by_good_annuli} in terms of $\log \delta_k^{-1}$. Suppose that we have picked $w_j \in \CA_{0,j}$ for each $k \leq j \leq j_0$.  By Proposition~\ref{prop:cover_by_good_annuli}-\eqref{it:move_into_and_out} we know that there exists $s_{j_0} \leq \delta_{j_0}^{1/\kappa} \leq \delta_k^{4/3}$ so that
\begin{equation}
\label{eqn:get_into_small_annulus}
\int_{\qball{h}{w_{j_0}}{\delta_{j_0}^{1/3}}} p_{s_{j_0}}(u,a) \, d\qmeasure{h}(a) \geq \exp(-(\log \delta_k^{-1})^\kappa).
\end{equation}
For each $k+1 \leq j \leq j_0$, we let $m_j$ be the first index $i$ so that $\CA_{i,j} \cap \CA_{0,j-1} \neq \emptyset$.  Then we have that $m_j \leq (\log \delta_k^{-1})^\kappa$.  By iterating Proposition~\ref{prop:cover_by_good_annuli}-\eqref{it:move} $m_j$ times,  we see that by increasing the value of $\kappa$ if necessary, there exists $\delta_j^{4/3} \log(\delta_j^{-1})^{-\kappa} \leq s \leq \delta_j^{4/3} (\log \delta_j^{-1})^\kappa$ so that
\begin{equation}
\label{eqn:good_annulus_to_next}
\int_{\qball{h}{w_j}{\delta_j^{1/3}}} p_s(w_{j+1},a) \, d\qmeasure{h}(a) \geq \exp(-(\log \delta_k^{-1})^{2\kappa}).
\end{equation}	
By applying the semigroup property and iterating~\eqref{eqn:good_annulus_to_next} over $j_0 \leq j \leq k$ and combining with~\eqref{eqn:get_into_small_annulus}, we thus see that by possibly taking $\kappa$ to be larger,  there exists $\delta_k^{4/3} \log(\delta_k^{-1})^{-\kappa} \leq s \leq \delta_k^{4/3} (\log \delta_k^{-1})^\kappa$ so that
\begin{equation}
\label{eqn:get_into_big_annulus}
\int_{\qball{h}{w_k}{\delta_k^{1/3}}} p_s(u,a) \, d\qmeasure{h}(a) \geq \exp(-(k-j_0)(\log \delta_k^{-1})^{2\kappa}).	
\end{equation}
For each $i$, we let $v_{i,k}$ be a point in $\CA_{i,k} \cap \CA_{i+1,k}$ whose $\qdistnoarg{h}$-distance from both $\partial \CA_{i,k}$ and $\partial \CA_{i+1,k}$ is at least $\delta_k^{1/3} \log(\delta_k^{-1})^{-\kappa} / 2$. We let $v_{0,k} = w_k$.  It follows from Proposition~\ref{prop:cover_by_good_annuli}-\eqref{it:move} that there exists $\delta_k^{4/3} (\log \delta_k^{-1})^{-\kappa} \leq s \leq \delta_k^{4/3} (\log \delta_k^{-1})^{\kappa}$ so that
\begin{equation}
\label{eqn:big_annulus_to_big_annulus}
\int_{\qball{h}{v_{i,k}}{\delta_k^{1/3}}} p_s(v_{i-1,k},a) \, d\qmeasure{h}(a) \geq \exp(-(\log \delta_k^{-1})^\kappa).
\end{equation}
By applying the semigroup property and iterating~\eqref{eqn:big_annulus_to_big_annulus} over $1 \leq i \leq N_k$, we see that there exists $N_k \delta_k^{4/3} (\log \delta_k^{-1})^{-\kappa} \leq s \leq N_k \delta_k^{4/3} (\log \delta_k^{-1})^{\kappa}$ so that
\begin{equation}
\label{eqn:cross_big_annuli}	
\int_{\qball{h}{v_{N_k,k}}{\delta_k^{1/3}}} p_s(v_{0,k},a) \, d\qmeasure{h}(a) \geq \exp(-N_k (\log \delta_k^{-1})^\kappa).
\end{equation}
Arguing in the same way as~\eqref{eqn:get_into_big_annulus}, we have that there exists $\delta_k^{4/3} \log(\delta_k^{-1})^{-\kappa} \leq s \leq \delta_k^{4/3} (\log \delta_k^{-1})^{\kappa}$ so that
\begin{equation}
\label{eqn:get_out_of_big_annulus}
\int_{\qball{h}{v_{N_k,k}}{\delta_k^{1/3}}} p_s(a,v) \, d\qmeasure{h}(a) \geq \exp(-(k-j_0)(\log \delta_k^{-1})^{2\kappa}).	
\end{equation}
Combining~\eqref{eqn:get_into_big_annulus}, \eqref{eqn:cross_big_annuli} and~\eqref{eqn:get_out_of_big_annulus}, increasing the value of $\kappa$ if necessary, and applying the semigroup property we see that there exists 
$t_0 \in [t (\log \delta_k^{-1})^{-\kappa}, t(\log \delta_k^{-1})^{\kappa}]$
so that
\begin{equation}
\label{eqn:initial_lower_bound}
p_{t_0}(u,v) \geq \exp(- N_k(\log \delta_k^{-1})^\kappa ).
\end{equation}
Therefore,  the claim of Theorem~\ref{thm:lbd} holds with $t_0$ in place of $t$. We will now complete the proof by establishing the result for $t$ in place of $t_0$.  If we apply the argument described above but with $t (\log \delta_k^{-1})^{-\kappa}/2$ in place of $t$,  then~\eqref{eqn:initial_lower_bound} implies that there exists $t_1 \in [t (\log \delta_k^{-1})^{-2\kappa}/2, t/2]$ so that $p_{t_1}(u,v)$ satisfies the desired lower bound.  We have that
\[ p_t(u,v) \geq \int_{\qball{h}{v}{\delta_k^{1/3}(\log \delta_k^{-1})^{-\kappa}}} p_{t-t_1}(u,a) \, p_{t_1}(a,v) \, d\qmeasure{h}(a).\]

By arguing as above,  we have that for all $a \in \qball{h}{u}{\delta_k^{1/3} \log(\delta_k^{-1})^{-\kappa}}$,  $p_{t_1}(a,u)$ satisfies the desired lower bound. Therefore we just need to get a lower bound on
\[ \int_{\qball{h}{u}{\delta_k^{1/3}(\log \delta_k^{-1})^{-\kappa}}} p_{t-t_1}(u,a) \, d\qmeasure{h}(a) = P_u[ X_{t-t_1} \in \qball{h}{u}{\delta_k^{1/3}(\log \delta_k^{-1})^{-\kappa}}].\]
Take $t_2 = t (\log \delta_k^{-1})^{-2\kappa}$.  From Lemma~\ref{lem:on_diagonal_lbd}, we have that
\[ \int_\CS p_{t-t_1}(u,a) \, p_{t_2}(a,u) \, d\qmeasure{h}(a) =  p_{t-t_1+t_2}(u,u)  \geq \frac{1}{C t (\log t^{-1})^\kappa}\]
for some random constant $C>0$. The  upper heat kernel estimate in Theorem~\ref{thm:ubd} implies that $p_{t_2}(a,u)$ is negligible if $a \notin \qball{h}{u}{\delta_k^{1/3}(\log \delta_k^{-1})^{-\kappa}}$.  Thus by possibly adjusting the value of $C$, we have that
\[ \int_{\qball{h}{u}{\delta_k^{1/3}(\log \delta_k^{-1})^{-\kappa}}} p_{t-t_1}(u,a) \, p_{t_2}(a,u) \, d\qmeasure{h}(a) \geq \frac{1}{C t (\log t^{-1})^\kappa}.\]
Applying Theorem~\ref{thm:ubd} to $p_{t_2}(a,u)$ yields that
\[ \frac{C (\log t_2^{-1})^\kappa}{t_2} \int_{\qball{h}{u}{\delta_k^{1/3}(\log \delta_k^{-1})^{-\kappa}}} p_{t-t_1}(u,a) \, d\qmeasure{h}(a) \geq \frac{1}{C t (\log t^{-1})^\kappa}.\]
Rearranging and increasing the value of $C$ gives
\[\int_{\qball{h}{u}{\delta_k^{1/3}(\log \delta_k^{-1})^{-\kappa}}} p_{t-t_1}(u,a) \, d\qmeasure{h}(a) \geq \frac{t_2}{C t (\log t^{-1})^{2\kappa}}.\]
Combining this with the above implies that $P_u[ X_{t-t_1} \in \qball{h}{u}{\delta_k^{1/3} (\log \delta_k^{-1})^{-\kappa}}]$ satisfies the desired lower bound.
\end{proof}

\appendix

\section{$\SLE_6$ hulls at the times when the tip is on the boundary are Jordan domains}
\label{sec:SLE6-chunk-Jordan}

In Sections~\ref{sec:percolation-exploration}, \ref{sec:exit-time-bounds} and~\ref{sec:LBM-HK-bounds},
we make use of the hulls of $\SLE_6$ to construct nice cellular decompositions of quantum disks
and wedges. An essential feature of the hulls of $\SLE_6$, on which our arguments heavily rely,
is that \emph{they are Jordan domains at the times when the tip of $\SLE_6$ is located on the boundary}.
For completeness, here we give a detailed proof of this fact on the basis of some
fundamental results from the theory of imaginary geometry developed in \cite{ms2016ig1,ms2016ig3}.

We first give the precise formulation of this property in the case of radial $\SLE_6$,
since our arguments in Sections~\ref{sec:percolation-exploration}, \ref{sec:exit-time-bounds}
and~\ref{sec:LBM-HK-bounds} require it mainly for this case. In the same way as in
Subsection~\ref{subsec:sle_explorations}, for a radial $\SLE_6$ $\eta'$ on $\D$ targeted at $0$
and $t \in [0, \inf(\eta')^{-1}(0))$, we define the \emph{hull} $K_{t}$ of $\eta'([0,t])$
as the complement in $\D$ of the $0$-containing component of $\D \setminus \eta'([0,t])$.

\begin{proposition} \label{prop:SLE6-chunk-Jordan-D}
Let $x \in \partial \D$ and let $\eta'$ be a radial $\SLE_6$ on $\D$ from $x$
targeted at $0$. Then a.s., for any $t\in(0,\inf(\eta')^{-1}(0))$ with 
$\partial K_{t} \cap \partial \D \not= \partial \D$ and $\eta'(t) \in \partial \D$,
\begin{equation}\label{eq:SLE6-chunk-Jordan-D}
\begin{minipage}{310pt}
$K_{t} \setminus \partial K_{t}$ is a Jordan domain in $\mathbb{C}$ with boundary
$\partial K_{t}$, and $\partial K_{t} \cap \partial \D$ is a compact interval in
$\partial \D$ containing $\eta'(0)=x$ in its interior.
\end{minipage}
\end{equation}
\end{proposition}

Proposition~\ref{prop:SLE6-chunk-Jordan-D} can be obtained by combining with a version of the
locality of $\SLE_6$ the analogous statement for chordal $\SLE_6$ on $\h$ stated as follows.

\begin{proposition} \label{prop:SLE6-chunk-Jordan-H}
Let $\eta'$ be a chordal $\SLE_6$ on $\h$ from $0$ to $\infty$.
Then a.s., for any $t\in(0,\infty)$ with $\eta'(t)\in\mathbb{R}$,
\begin{equation} \label{eq:SLE6-chunk-Jordan-H}
\begin{minipage}{275pt}
$K_{t} \setminus \partial K_{t}$ is a Jordan domain in $\mathbb{C}$ with boundary
$\partial K_{t}$, and $\partial K_{t} \cap \mathbb{R}$ is a compact interval in
$\mathbb{R}$ containing $0$ in its interior.
\end{minipage}
\end{equation}
\end{proposition}

Note that each $t\in(0,\infty)$ with $\eta'(t)\in\mathbb{R}$ and the property
\eqref{eq:SLE6-chunk-Jordan-H} also satisfies
\begin{equation} \label{eq:SLE6-chunk-Jordan-H-etat}
\eta'(t)\in\{\max(\partial K_{t} \cap \mathbb{R}),\min(\partial K_{t} \cap \mathbb{R})\},
\end{equation}
since for any $t,s\in[0,\infty)$ with $t<s$ we have $K_{t} \not= K_{s}$ and hence
$\eta'((t,s]) \cap \h \not \subset K_{t}$ (see, e.g., \cite[Sections 7--9]{BN16LectSLE}).

\begin{proof}
\emph{Step~1.} We first recall some basic properties of the chordal $\SLE_6$ $\eta'$.
Set $T_{a}:=\inf\{ t \in [0,\infty) \mid \eta'(t) \in [a,\infty)\}$ and
$T_{-a}:=\inf\{ t \in [0,\infty) \mid \eta'(t) \in (-\infty,-a]\}$ for $a \in (0,\infty)$
and $T_{0}:=0$, so that $T_{a}=\inf\{t\in[0,\infty) \mid a \in \partial K_{t}\}$
for any $a\in\mathbb{R}$ since $\partial K_{t} \cap \mathbb{R}$ is non-empty and connected
for any $t\in(0,\infty)$ (see, e.g., \cite[Sections 7--9]{BN16LectSLE}). Note that
$\lim_{b\downarrow 0}T_{-b}<T_{a}\leq T_{a}\vee T_{-a}<\infty$ a.s.\ for each $a\in(0,\infty)$
(see, e.g., \cite[Propositions 6.33 and 6.34]{law2005conformally} or \cite[Proposition 10.3-(b)]{BN16LectSLE}),
which together with the scale invariance of $\eta'$ under the parameterization by half-plane capacity
(see, e.g., \cite[Proposition 6.5]{law2005conformally} or \cite[Proposition 9.3]{BN16LectSLE})
easily implies that $\lim_{b\downarrow 0}T_{b}\vee T_{-b}=0$ a.s. In particular, a.s.,
\begin{equation} \label{eq:SLE6-both-sides-0-infty}
\begin{minipage}{227pt}
$\eta'([0,\infty)) \cap (-\infty,0)$ and $\eta'([0,\infty)) \cap (0,\infty)$
are unbounded and their boundaries in $\mathbb{C}$ contain $0$.
\end{minipage}
\end{equation}
We also easily see, for $a \in \mathbb{R} \setminus \{0\}$ from
\cite[Proposition 6.34]{law2005conformally} or \cite[Proposition 10.3-(b)]{BN16LectSLE},
and for $a=0$ from $\lim_{t\to\infty}|\eta'(t)|=\infty$ a.s.\ (see, e.g.,
\cite[Proposition 6.10]{law2005conformally} or \cite[Proposition 11.7]{BN16LectSLE})
and the above-mentioned scale invariance of $\eta'$, that
\begin{equation} \label{eq:SLE6-avoid-given-real}
\text{$a \not\in \eta'((0,\infty))$ a.s.\ for each $a \in \mathbb{R}$.}
\end{equation}
It is also known by \cite[Remark 5.3]{mw2017intSLE} that a.s.
\begin{equation} \label{eq:SLE6-no-triple-point}
\text{no $(s,t,u)\in\mathbb{R}^{3}$ with $0\leq s<t<u$ satisfies $\eta'(s)=\eta'(t)=\eta'(u)$.}
\end{equation}

\emph{Step~2.}
Let $\eta'_{L}$ and $\eta'_{R}$ denote the \emph{left} and \emph{right boundaries} of $\eta'([0,\infty))$,
respectively, i.e., $\eta'_{L}:=\partial U_{L} \setminus i(-\infty,0)$ and 
$\eta'_{R}:=\partial U_{R} \setminus i(-\infty,0)$, where
$i(-\infty,0) := \{ ia \mid a \in (-\infty,0) \}$ and $U_{L}$ and $U_{R}$ denote
the components of $\mathbb{C} \setminus ( \eta'([0,\infty)) \cup i(-\infty,0))$
containing $-1-i$ and $1-i$, respectively. We claim that a.s.\ $\eta'_{L}$ and
$\eta'_{R}$ are simple curves starting from $0$ and tending to $\infty$
and satisfy $\eta'_{L}\cap(0,\infty)=\emptyset=\eta'_{R}\cap(-\infty,0)$.
Indeed, let $\psi:\h \to \h$ denote the conformal map given by $\psi(z):=-1/z$,
set $\eta_{R}:=(\psi(\eta'_{L})\cup\{0\})\setminus\{\infty\}$ and
$\eta_{L}:=(\psi(\eta'_{R})\cup\{0\})\setminus\{\infty\}$.
By \cite[Theorem 1.4]{ms2016ig1} (see also \cite[Fig.~5]{mw2017intSLE}),
$\eta_{R}$ (resp.\ $\eta_{L}$) is a (chordal) $\SLE_{8/3}((8/3)/2-2;8/3-4)$
(resp.\ $\SLE_{8/3}(8/3-4;(8/3)/2-2)$) curve (on $\h$) from $0$ to $\infty$ with force points
$0^{+}$ and $0^{-}$, a variant of $\SLE_{8/3}$ introduced in \cite[Section 2]{ms2016ig1},
then $\eta_{R}\cap(-\infty,0)=\emptyset=\eta_{L}\cap(0,\infty)$ a.s.\ by \cite[Remark 5.3]{ms2016ig1}
and hence $\eta'_{L}\cap(0,\infty)=\emptyset=\eta'_{R}\cap(-\infty,0)$ a.s.
Moreover, a.s.\ $\eta'_{L}$ and $\eta'_{R}$ are simple curves starting from $0$ and tending to $\infty$,
because $\eta_{R}$ and $\eta_{L}$ are continuous curves starting from $0$ and tending to $\infty$
a.s.\ by \cite[Proposition 7.3]{ms2016ig1} and are simple curves a.s.\ by the fact that
an $\SLE_{\kappa}(\rho^{L};\rho^{R})$ curve $\eta$ with force points $0^{+}$ and $0^{-}$
is a simple curve a.s.\ for any $\kappa\in(0,4)$ and any $\rho^{L},\rho^{R}\in(-2,\infty)$.

This last fact can be verified as follows. \cite[Lemma~7.1]{ms2016ig1}\footnote{We remark that ``let $\eta_{\psi}$ be the flow line of $h_{\psi}$ starting from
$0$ and targeted at $\infty$'' in \cite[p.\ 668, lines 18--19]{ms2016ig1} should read
``let $\eta_{\psi}$ denote the time reparameterization of $\psi \circ \eta_{C}$ by half-plane capacity''.}
implies that the law of an $\SLE_{\kappa}(\rho^{L};\rho^{R})$ curve $\eta$ can be realized as a certain
conditional law of a conformal image of a segment of an $\SLE_{\kappa}(\widehat{\rho}^{L};\widehat{\rho}^{R})$ curve
$\widehat{\eta}$ from $0$ to $\infty$ with force points $0^{+}$ and $0^{-}$ for a suitable choice of
$\widehat{\rho}^{L},\widehat{\rho}^{R}\in[\kappa/2-2,\infty)$.
Then since a chordal $\SLE_{\kappa}$ on $\h$ from $0$ to $\infty$
is a simple curve a.s.\ (see, e.g., \cite[Propositions 6.9 and 6.12]{law2005conformally} or
\cite[Propositions 11.3 and 11.5]{BN16LectSLE}) and has law mutually absolutely continuous
with respect to that of $\widehat{\eta}$ on any compact time interval in $(0,\infty)$
under the parameterization by half-plane capacity by 
$\widehat{\rho}^{L},\widehat{\rho}^{R}\in[\kappa/2-2,\infty)$ and \cite[Remark 2.3]{ms2016ig1},
it follows that $\widehat{\eta}$ and thereby $\eta$ are simple curves a.s.

\emph{Step~3.} Next, we verify that for each $a\in\mathbb{R}\setminus\{0\}$ we a.s.\ have
\eqref{eq:SLE6-chunk-Jordan-H} with $t=T_{a}$. Indeed, let $g$ be the M\"{o}bius transformation
given by $g(z):=az/(z+a)$, which maps $\h,0,\infty,-a$ onto $\h,0,a,\infty$, respectively. Then
$g\circ\eta'$ is (a time reparameterization of) a chordal $\SLE_6$ on $\h$ from $0$ to $a$
by the conformal invariance of chordal $\SLE_6$ (see, e.g., \cite[Proposition 9.3]{BN16LectSLE}),
and the reparameterization of $g\circ\eta'|_{[0,T_{-a}]}$ by its half-plane capacity
has the same law as $\eta'|_{[0,T_{a}]}$ by the locality of $\SLE_6$
(see, e.g., \cite[Proposition 6.14]{law2005conformally} or \cite[Theorem 13.2]{BN16LectSLE}).
In particular, $g^{-1}(\h \setminus K_{T_{a}})$ has the same law as the component
$U_{-a}$ of $\h \setminus \eta'([0,T_{-a}])$ whose boundary contains $-a$,
but $U_{-a}$ coincides with the component of $\h \setminus \eta'([0,\infty))$
whose boundary contains $-a$ by $\eta'((T_{-a},\infty)) \cap K_{T_{-a}} \subset \partial K_{T_{-a}}$.
Then by~\eqref{eq:SLE6-both-sides-0-infty}, \eqref{eq:SLE6-avoid-given-real} and
Step~2, a.s.\ $U_{-a}$ is a Jordan domain whose boundary $\partial U_{-a}$
is of the form $\gamma((0,1)) \cup [\alpha,\beta]$ for some $\alpha,\beta\in\mathbb{R}$
with $\alpha<-a<\beta$ and $0 \not\in [\alpha,\beta]$ and some simple curve
$\gamma:(0,1) \to \h$ with $\lim_{s\downarrow 0}\gamma(s)=\alpha$ and $\lim_{s\uparrow 1}\gamma(s)=\beta$.
Thus a.s.\ $K_{T_{a}} \setminus \partial K_{T_{a}}$ is a Jordan domain with boundary $\partial K_{T_{a}}$,
which is of the form $g\bigl(\gamma((0,1)) \cup (\mathbb{R} \setminus [\alpha,\beta])\bigr)\cup\{a\}$
for some such $\alpha,\beta,\gamma$, proving~\eqref{eq:SLE6-chunk-Jordan-H} with $t=T_{a}$.

\emph{Step~4.}
By using the reversibility of chordal $\SLE_6$ proved in \cite[Theorem 1.1]{ms2016ig3},
which states that $\eta'_{\leftarrow}:=\iota\circ\eta'((\cdot)^{-1})$
($\eta'_{\leftarrow}(0):=\infty$ and $\iota(z):=-\overline{z}$ for $z\in\mathbb{C}$)
is (a time reparameterization of) a chordal $\SLE_6$ on $\h$ from $\infty$ to $0$, the results
of Step~2 can be (partially) extended to $\eta'([0,s))$ for each $s\in(0,\infty)$ as follows.
Let $K_{s,\infty}$ denote the complement in $\h$ of the component of
$\h\setminus\eta'([s,\infty))=\iota(\h\setminus\eta'_{\leftarrow}([0,s^{-1}]))$
whose boundary in $\mathbb{C}$ contains $0$.
Noting that $-1/\eta'_{\leftarrow}$ is (a time reparameterization of) a chordal $\SLE_6$
on $\h$ from $0$ to $\infty$ by the reversibility and the conformal invariance of chordal
$\SLE_6$ and applying to $-1/\eta'_{\leftarrow}$ the fact that $\lim_{b\downarrow 0}T_{b}\vee T_{-b}=0$
a.s.\ mentioned in Step~1 and the result of Step~3 for each $a\in\mathbb{Q}\setminus\{0\}$,
we easily see that, a.s.,
\begin{equation} \label{eq:Ksinfty-bounded-complement}
\text{$\h \setminus K_{s,\infty}$ is a bounded simply connected domain in $\mathbb{C}$.}
\end{equation}
Let $g_{s}:\h \setminus K_{s,\infty} \to \h$ be the unique conformal map such that
$\lim_{|z|\to\infty}(g_{s}(-1/z)-z)=0$
(see, e.g., \cite[Proposition 3.36]{law2005conformally} or \cite[Theorem 4.3]{BN16LectSLE}).
Then the domain Markov property of chordal $\SLE_6$ (see, e.g., \cite[Proposition 9.4]{BN16LectSLE})
applied to $\eta'_{\leftarrow}$ at time $s^{-1}$ implies that the conditional law of
$(\iota\circ g_{s}\circ\iota)\circ\eta'_{\leftarrow}|_{[s^{-1},\infty)}=\iota\circ g_{s}\circ\eta'((\cdot)^{-1})|_{[s^{-1},\infty)}$
given $\eta'|_{[s,\infty)}$ is (a time reparameterization of)
a chordal $\SLE_6$ on $\h$ from $\iota\circ g_{s}(\eta'(s))$ to $\infty$.
Applying the reversibility \cite[Theorem 1.1]{ms2016ig3} of chordal $\SLE_6$ to
$\iota\circ g_{s}\circ\eta'((\cdot)^{-1})|_{[s^{-1},\infty)}$ and then using the
domain Markov property of chordal $\SLE_6$, we further see that
the conditional law of $\varphi_{s}\circ\eta'|_{[0,s)}$ given $\eta'|_{[s,\infty)}$
is (a time reparameterization of) a chordal $\SLE_6$ on $\h$ from $0$ to $\infty$,
where $\varphi_{s}:\h \setminus K_{s,\infty} \to \h$ is the conformal map defined by
$\varphi_{s}(z):=(g_{s}(\eta'(s))-g_{s}(z))^{-1}$. It thus follows from Step~2 that
for each $s \in (0,\infty)$, a.s.
\begin{equation} \label{eq:SLE6-left-right-simple-s}
\begin{minipage}{352pt}
the left and right boundaries $\eta'^{s}_{L},\eta'^{s}_{R}$ of $\varphi_{s}(\eta'([0,s)))$
are simple curves starting from $0$ and tending to $\infty$ and satisfy
$\eta'^{s}_{L}\cap(0,\infty)=\emptyset=\eta'^{s}_{R}\cap(-\infty,0)$.
\end{minipage}
\end{equation}
Recall (see, e.g., \cite[Theorems II.(4.1) and VI.(2.2)]{Wh1963} and \cite[Theorem 2.1]{Po1992})
that $\varphi_{s}^{-1}$ has a continuous extension to $\overline{\h}\cup\{\infty\}$,
so that $\varphi_{s}^{-1}(\mathbb{R}\cup\{\infty\})=\partial(\h \setminus K_{s,\infty})$.
Note that, while $(\varphi_{s}^{-1})^{-1}(z)$ is a singleton for any
$z\in\partial(\h \setminus K_{s,\infty}) \setminus \partial K_{s,\infty}$ by an application
of Carath\'{e}odory's theorem (see, e.g., \cite[Theorem 2.6 and Proposition 2.14]{Po1992}),
$(\varphi_{s}^{-1})^{-1}(z)$ may have two or more elements for
$z\in\partial(\h \setminus K_{s,\infty}) \cap \partial K_{s,\infty}$.

\emph{Step~5.} From~\eqref{eq:SLE6-avoid-given-real}, Step~3, Step~4 and the domain Markov
property of chordal $\SLE_6$, we choose as follows an event for $\eta'$ of probability $1$
on which the assertion of the proposition will be verified in Steps~6--8 below.
For each $s,t \in [0,\infty)$ with $s \leq t$, let $K_{s,t}$
denote the hull of $\eta'([s,t])$ in $\h \setminus K_{s}$, i.e.,
the complement in $\h \setminus K_{s}$ of the unbounded component of
$(\h \setminus K_{s}) \setminus \eta'([s,t])$, so that $K_{s,t} = K_{t} \setminus K_{s}$,
and set $\partial^{\mathrm{top}} K_{s,t} := \partial K_{s,t} \cap (\h \setminus K_{s})$ and
$\partial^{\mathrm{bot}} K_{s,t} := \partial K_{s,t} \cap \partial(\h \setminus K_{s})$.
Then for each $a,b\in\mathbb{R}$ with either $0 \leq a < b$ or $b < a \leq 0$, it follows from Step~3
and the domain Markov property \cite[Proposition 9.4]{BN16LectSLE} of chordal $\SLE_6$
that, a.s.\ on $\{T_{a} < T_{b}\}$, ($\partial(\h \setminus K_{T_{a}})$ is a simple curve,)
\begin{equation} \label{eq:Jordan-KTaTb}
\begin{minipage}{360pt}
$K_{T_{a},T_{b}} \setminus \partial K_{T_{a},T_{b}}$ is a Jordan domain in $\mathbb{C}$
with boundary $\partial K_{T_{a},T_{b}}$, and\\
$\partial^{\mathrm{bot}} K_{T_{a},T_{b}}$ is a compact interval in
$\partial(\h \setminus K_{T_{a}})$ containing $\eta'(T_{a})$ in its interior
\end{minipage}
\end{equation}
(note that~\eqref{eq:Jordan-KTaTb} with $a=0$ is the same as~\eqref{eq:SLE6-chunk-Jordan-H} with $t=T_{b}$).
Also for each $b\in\mathbb{R}$ and each $s \in (0,\infty)$, we easily see from
the domain Markov property of chordal $\SLE_6$, \eqref{eq:SLE6-avoid-given-real}
and~\eqref{eq:Jordan-KTaTb} with $a=0$ that, a.s.,
\begin{gather} \label{eq:SLE6-avoid-Tb}
\eta'(T_{b}) \not\in \eta'((T_{b},\infty)), \\
\max( \partial K_{s} \cap \mathbb{R}  )^{+}, \min( \partial K_{s} \cap \mathbb{R} )^{-} \not\in \widetilde{\eta}'_{s}((s,\infty));
\label{eq:SLE6-avoid-max-real}
\end{gather}
here $\widetilde{\eta}'_{s}$ denotes the lift of $\eta'|_{[s,\infty)}$ to
$(\h \setminus K_{s})\cup\widetilde{\partial}(\h \setminus K_{s})$, with 
$\widetilde{\partial}(\h \setminus K_{s})$ representing the set of prime ends of $\h \setminus K_{s}$,
and $\max( \partial K_{s} \cap \mathbb{R} )^{+}$ and $\min( \partial K_{s} \cap \mathbb{R} )^{-}$
denote the two elements of the boundary of $(\mathbb{R}\cup\{\infty\}) \setminus \partial K_{s}$
in $\widetilde{\partial}(\h \setminus K_{s})$ corresponding to
$\max( \partial K_{s} \cap \mathbb{R} )$ and $\min( \partial K_{s} \cap \mathbb{R} )$,
respectively.

Recalling Step~4, we can thus choose an event $\Omega_{0}$ for $\eta'$ of probability $1$
such that every instance of $\eta'$ from $\Omega_{0}$ satisfies~\eqref{eq:SLE6-both-sides-0-infty},
$\lim_{t\to\infty}|\eta'(t)|=\infty$, \eqref{eq:SLE6-no-triple-point}, \eqref{eq:Ksinfty-bounded-complement},
\eqref{eq:SLE6-left-right-simple-s} and~\eqref{eq:SLE6-avoid-max-real} for any $s\in \mathbb{Q} \cap (0,\infty)$,
\eqref{eq:SLE6-avoid-Tb} for any $b\in\mathbb{Q}$, and either $T_{a}=T_{b}$ or
\eqref{eq:Jordan-KTaTb} for any $a,b\in\mathbb{Q}$ with either $0 \leq a < b$ or $b < a \leq 0$.

\emph{Step~6.} Now we can proceed to the conclusion of the proof as follows.
Fix any instance of $\eta'$ from $\Omega_{0}$, and let $t\in(0,\infty)$ satisfy $\eta'(t)\in\mathbb{R}$.
If $t=T_{b}$ for some $b\in\mathbb{Q}\setminus\{0\}$, then~\eqref{eq:SLE6-chunk-Jordan-H}
holds by~\eqref{eq:Jordan-KTaTb} with $a=0$. Therefore we may assume that $t\not=T_{a}$
for any $a\in\mathbb{Q}\setminus\{0\}$, and we further set $b:=\eta'(t)$, so that
$b\in\mathbb{R}\setminus\{0\}$ by~\eqref{eq:SLE6-avoid-Tb} with $0$ in place of $b$.
By considering $\iota\circ\eta'$ instead of $\eta'$ when $b<0$, where
$\iota(z):=-\overline{z}$ for $z\in\mathbb{C}$, we may and do assume that $b>0$
in the rest of this proof.

Let $a\in\mathbb{Q}\cap(0,\infty)$, so that $\eta'(T_{a})\geq a$.
If $\eta'(T_{a})>b$, then $b$ belongs to the interior of
the compact interval $\partial^{\mathrm{bot}}K_{0,T_{a}}$ in $\mathbb{R}$, hence
$\eta'(t)=b\not\in\eta'([T_{a},\infty))$ by~\eqref{eq:Jordan-KTaTb} with $0$ in place of $a$ and thus $t<T_{a}$.
Combining this observation with $t\not=T_{a}$ and~\eqref{eq:SLE6-avoid-Tb}, we easily obtain
\begin{align} \label{eq:Ta-a-more-than-b}
t&<T_{a} & & \text{for any $a\in\mathbb{Q}\cap(b,\infty)$,} \\
T_{a}<t \quad \text{and} \quad a&\leq\eta'(T_{a})<b & & \text{for any $a\in\mathbb{Q}\cap(0,b)$.}
\label{eq:Ta-a-less-than-b}
\end{align}
In particular,
\begin{equation} \label{eq:KtR-max}
\eta'(s)\leq b \qquad\text{for any $s \in [0,t]$ with $\eta'(s)\in\mathbb{R}$,}
\end{equation}
since for any $a\in\mathbb{Q}\cap(b,\infty)$ we have $s\leq t<T_{a}$
by~\eqref{eq:Ta-a-more-than-b} and hence $\eta'(s)<a$. Moreover,
by the monotonicity of $T_{a}$ in $a\in(0,\infty)$, $\eta'(t)=b$,
the continuity of $\eta'$ and~\eqref{eq:Ta-a-less-than-b} we have $\lim_{a\uparrow b}T_{a}\leq T_{b}\leq t$,
$\eta'(\lim_{a\uparrow b}T_{a})=\lim_{\mathbb{Q}\ni a\uparrow b}\eta'(T_{a})=b$, therefore
\begin{equation} \label{eq:etaTb-etat}
\lim_{a\uparrow b}T_{a}=T_{b}\leq t, \quad \eta'(T_{b})=b \quad\text{and}\quad
\text{$T_{a}<T_{b}$ for any $a\in(0,b)$.}
\end{equation}

\emph{Step~7.} We claim that $b=\eta'(t)$ satisfies~\eqref{eq:Jordan-KTaTb} with $a=0$
under the setting of Step~6. To see this, we inductively construct a sequence
$\{b_{n}\}_{n=0}^{\infty}\subset\mathbb{Q}\cap(0,b)$ as follows.
First, noting~\eqref{eq:etaTb-etat}, choose $b_{0}\in\mathbb{Q}\cap(0,b)$ so that
$|\eta'(s)-b|<b$ for any $s\in[T_{b_{0}},T_{b}]$, and set $a:=\min(\partial K_{T_{b_{0}}}\cap\mathbb{R})$,
which satisfies $a<0$ by~\eqref{eq:Jordan-KTaTb} with $0,b_{0}$ in place of $a,b$.
Next, letting $n\geq 0$, supposing that $b_{n}\in\mathbb{Q}\cap(0,b)$ is given, and noting
\eqref{eq:Ta-a-less-than-b}, \eqref{eq:etaTb-etat} and that we have
\eqref{eq:SLE6-chunk-Jordan-H-etat} with $t=T_{b_{n}}$ by~\eqref{eq:Jordan-KTaTb} with
$0,b_{n}$ in place of $a,b$, choose $b_{n+1}\in\mathbb{Q}\cap[(b_{n}+b)/2,b)$ so that
$\eta'([T_{b_{n+1}},T_{b}])\cap \overline{K_{T_{b_{n}}}} = \emptyset$.
Then $\{b_{n}\}_{n=0}^{\infty},\{T_{b_{n}}\}_{n=0}^{\infty},\{\eta'(T_{b_{n}})\}_{n=0}^{\infty}$
are strictly increasing, $\lim_{n\to\infty}b_{n}=b$, and for any $n\geq 0$ we have
$\min(\partial K_{T_{b_{n}}}\cap\mathbb{R})=a$ by $\eta'([T_{b_{0}},T_{b_{n}}]) \cap (-\infty,0] = \emptyset$,
$\overline{\partial^{\mathrm{top}} K_{0,T_{b_{n}}}} \subset \eta'([0,T_{b_{n}}])$, and
$\overline{\partial^{\mathrm{top}} K_{T_{b_{n}},T_{b_{n+1}}}} \subset \eta'([T_{b_{n}},T_{b_{n+1}}])$,
which in turn is included in $\overline{\h} \setminus \overline{K_{T_{b_{n-1}}}}$ if $n\geq 1$.
This last property combined with~\eqref{eq:Jordan-KTaTb} further implies that
$\overline{\partial^{\mathrm{top}} K_{T_{b_{n}},T_{b_{n+1}}}}\cap\overline{\partial^{\mathrm{top}} K_{0,T_{b_{n}}}}$
is a singleton $\{z_{n}\}$ for any $n\geq 0$ and that
$z_{n}\in\partial^{\mathrm{top}} K_{T_{b_{n-1}},T_{b_{n}}}$ for any $n\geq 1$.
Using these properties together with~\eqref{eq:Jordan-KTaTb}, we can define
a simple closed curve $\gamma:[-1,\infty]\to\overline{\h}$ by
$\gamma(\infty):=b$, $\gamma(-t):=(1-t)a+tb$ for $t\in[0,1]$,
$\gamma|_{[0,1]}$ being a homeomorphism to the closed interval in
$\overline{\partial^{\mathrm{top}} K_{0,T_{b_{0}}}}$ from $a$ to $z_{0}$, and
$\gamma|_{[n,n+1]}$ being a homeomorphism to the closed interval in
$\overline{\partial^{\mathrm{top}} K_{T_{b_{n-1}},T_{b_{n}}}}$ from $z_{n-1}$
to $z_{n}$ for each $n\geq 1$; note that $\lim_{s\to\infty}\gamma(s)=b$
by $\lim_{n\to\infty}b_{n}=b$, \eqref{eq:etaTb-etat},
the continuity of $\eta'$ at $T_{b}$ and the fact that
$\overline{\partial^{\mathrm{top}} K_{T_{b_{n}},T_{b_{n+1}}}}
	\subset \eta'([T_{b_{n}},T_{b_{n+1}}])$
for any $n\geq 0$. Then $\gamma((0,\infty)) \subset \h \cap \eta'([0,T_{b}])$,
$\gamma([-1,0])=[a,b]$, and it is elementary to see from the construction of $\gamma$
that $\eta'([0,T_{b}])$ is included in the closure in $\mathbb{C}$ of the (bounded)
Jordan domain with boundary $\gamma([-1,\infty])$, which together immediately show
$\gamma([-1,\infty])=\partial K_{T_{b}}$ and that $b=\eta'(t)$ satisfies
\eqref{eq:Jordan-KTaTb} with $0$ in place of $a$.

\emph{Step~8.} It remains to prove that $t=T_{b}$ under the setting of Step~6.
For this purpose, noting~\eqref{eq:etaTb-etat} and the fact that
$\eta'((T_{b},s]) \cap \h \not\subset K_{T_{b}}$ for any $s\in(T_{b},\infty)$
as noted just after~\eqref{eq:SLE6-chunk-Jordan-H-etat}, suppose that $T_{b}<t$, and
choose $s\in\mathbb{Q}\cap(T_{b},t)$ so that $\eta'(s) \in \h \setminus K_{T_{b}}$ and
$|\eta'(r)-b|<b$ for any $r\in[T_{b},s]$. Set $a:=\min(\partial K_{T_{b}}\cap\mathbb{R})$,
so that by combining~\eqref{eq:Jordan-KTaTb} with $0$ in place of $a$ from Step~7,
\eqref{eq:KtR-max}, $\eta'(T_{b})=b$ from~\eqref{eq:etaTb-etat},
$\eta'([T_{b},s]) \cap \{ z \in \mathbb{C} \mid \re(z) \leq 0\} = \emptyset$
and $\eta'([s,\infty)) \subset \overline{\h \setminus K_{s}}$, we get $a<0$,
\begin{equation} \label{eq:Ks-bottom-ab}
\partial K_{s} \cap \mathbb{R} = \partial K_{T_{b}} \cap \mathbb{R} = [a,b],
\end{equation}
$\eta'([s,\infty))\cap(a,b)=\emptyset$, and $U \cap K_{s} = U \cap K_{T_{b}}$ for some
open neighborhood $U$ of $a$ in $\mathbb{C}$. In particular, since $\h \setminus K_{T_{b}}$
is a Jordan domain in $\mathbb{C}\cup\{\infty\}$ by~\eqref{eq:Jordan-KTaTb} with $0$
in place of $a$, the point $a \in \partial(\h \setminus K_{s})$ corresponds to a unique
element of $\widetilde{\partial}(\h \setminus K_{s})$, hence $a \not\in \eta'([s,\infty))$
by~\eqref{eq:Ks-bottom-ab}, \eqref{eq:SLE6-avoid-max-real} and $\re(\eta'(s))>0$, and
it thus follows that
\begin{equation} \label{eq:Ksinfty-bottom-complement}
(a-\varepsilon,b)\cap\eta'([s,\infty))=\emptyset
\qquad \text{for some $\varepsilon\in(0,\infty)$.}
\end{equation}

Now we use the notation and the results from Step~4. Note that by
\eqref{eq:Ksinfty-bottom-complement} we have
$(a-\varepsilon,b) \subset \partial(\h \setminus K_{s,\infty})$ and hence that
by Carath\'{e}odory's theorem \cite[Theorem 2.6]{Po1992} the conformal map
$\varphi_{s}:\h \setminus K_{s,\infty} \to \h$ extends continuously to a bounded
$\mathbb{R}$-valued strictly increasing map on $(a-\varepsilon,b)$, which then
satisfies $\varphi_{s}(0)=0$ and $(\varphi_{s}^{-1})^{-1}(u)=\{\varphi_{s}(u)\}$
for any $u \in (a-\varepsilon,b)$. Set $\varphi_{s}(b-):=\lim_{u \uparrow b}\varphi_{s}(u)$.
Recalling~\eqref{eq:SLE6-left-right-simple-s} and noting that
$\varphi_{s}(a) \in \eta'^{s}_{L}$ by~\eqref{eq:Ks-bottom-ab} and that
$\varphi_{s}(b-) \in \eta'^{s}_{R}$ by~\eqref{eq:Ta-a-less-than-b} and~\eqref{eq:etaTb-etat}, we define simple curves
$\gamma_{L},\gamma_{R}:[0,\infty)\to\overline{\mathbb{H}}$ starting from $0$ and tending to $\infty$
by $\gamma_{L}(u):=\varphi_{s}(ua)$ and $\gamma_{R}(u):=\varphi_{s}(ub)$ for $u\in[0,1)$,
$\gamma_{L}|_{[1,\infty)}$ being a homeomorphism to the interval in $\eta'^{s}_{L}$
from $\varphi_{s}(a)$ to $\infty$ and $\gamma_{R}|_{[1,\infty)}$ being a homeomorphism
to the interval in $\eta'^{s}_{R}$ from $\varphi_{s}(b-)$ to $\infty$. Then
$(\gamma_{L}((1,\infty))\cup\gamma_{R}((1,\infty))) \cap [\varphi_{s}(a),\varphi_{s}(b-))=\emptyset$
and therefore
\begin{equation} \label{eq:etasLetasR-later-avoid-ab}
\bigl(\varphi_{s}^{-1}\circ\gamma_{L}((1,\infty))\cup\varphi_{s}^{-1}\circ\gamma_{R}((1,\infty))\bigr) \cap [a,b) =\emptyset.
\end{equation}
Moreover, noting that $\varphi_{s}(\eta'([0,s)))$ in~\eqref{eq:SLE6-left-right-simple-s},
precisely speaking, denotes the image of the lift $\overline{\eta}'_{s}$ of $\eta'|_{[0,s)}$ to $\h$
under $\varphi_{s}^{-1}$, i.e., the continuous map $\overline{\eta}'_{s}:[0,s) \to \overline{\h}$
satisfying $\eta'|_{[0,s)}=\varphi_{s}^{-1}\circ\overline{\eta}'_{s}$, that
$\overline{\eta}'_{s}(T_{b})=\lim_{\mathbb{Q}\ni u\uparrow b}\varphi_{s}(\eta'(T_{u}))=\varphi_{s}(b-)$
by~\eqref{eq:Ta-a-less-than-b} and~\eqref{eq:etaTb-etat}, and that
$\eta'^{s}_{L}\cup\eta'^{s}_{R}\subset\overline{\eta}'_{s}([0,s))$, we have
\begin{equation} \label{eq:etasLetasR-avoid-b}
\varphi_{s}^{-1}\bigl((\eta'^{s}_{L}\cup\eta'^{s}_{R})\setminus\{\varphi_{s}(b-)\}\bigr)
	\subset\varphi_{s}^{-1}\bigl(\overline{\eta}'_{s}([0,s))\setminus\{\overline{\eta}'_{s}(T_{b})\}\bigr)
	\subset\eta'([0,s)\setminus\{T_{b}\}).
\end{equation}
Since $b \not\in \eta'([0,s)\setminus\{T_{b}\})$ by $\eta'(T_{b})=b=\eta'(t)$, $T_{b}<s<t$
and~\eqref{eq:SLE6-no-triple-point}, it follows from~\eqref{eq:etasLetasR-avoid-b},
\eqref{eq:etasLetasR-later-avoid-ab} and~\eqref{eq:Ks-bottom-ab} that
\begin{equation} \label{eq:gammaLgammaR-avoid-R}
\varphi_{s}^{-1}\circ\gamma_{L}((1,\infty))\cup\varphi_{s}^{-1}\circ\gamma_{R}((1,\infty))
	\subset \eta'([0,s)) \setminus [a,b] = \eta'([0,s)) \cap \h.
\end{equation}
On the other hand, \eqref{eq:SLE6-left-right-simple-s} combined with the definition of
$\eta'^{s}_{L}$ and $\eta'^{s}_{R}$ implies that $\eta'^{s}_{L}$ (resp.\ $\eta'^{s}_{R}$)
is located to the left (resp.\ right) of $\eta'^{s}_{R}$ (resp.\ $\eta'^{s}_{L}$), i.e.,
included in the closure in $\mathbb{C}$ of the component of
$\mathbb{C} \setminus (\eta'^{s}_{R}\cup(-\infty,0])$
(resp.\ $\mathbb{C} \setminus (\eta'^{s}_{L}\cup[0,\infty))$ not containing $-i$,
and $\overline{\eta}'_{s}([0,s))$ is located both to the left of $\eta'^{s}_{R}$
and to the right of $\eta'^{s}_{L}$.
Set $\tau_{R}:=\inf\{u\in [1,\infty] \mid \gamma_{R}(u) \in \gamma_{L}([0,\infty])\}$
and $\{\tau_{L}\}:=\gamma_{L}^{-1}(\gamma_{R}(\tau_{R}))$, where
$\gamma_{R}(\infty):=\infty=:\gamma_{L}(\infty)$, so that $\tau_{R},\tau_{L}\in(1,\infty]$.
Further, recalling~\eqref{eq:gammaLgammaR-avoid-R}, choose a simple curve
$\gamma_{s} \subset \varphi_{s}^{-1}\circ\gamma_{L}([1,\tau_{L}])\cup\varphi_{s}^{-1}\circ\gamma_{R}([1,\tau_{R}]) (\subset \eta'([0,s)) \cap \h)$
from $b$ to $a$, which is possible by \cite[Theorem II.(5.1)]{Wh1963}, and
let $U_{s}$ denote the (bounded) Jordan domain with boundary $\gamma_{s}\cup[a,b]$.
Then we easily see from $\gamma_{s} \subset \eta'([0,s)) \cap \h$ and the above-mentioned
topological configuration of $\overline{\eta}'_{s}([0,s))$ in relation to $\eta'^{s}_{R}$
and $\eta'^{s}_{L}$ that $U_{s} \subset K_{s}$ and that $\eta'([0,s]) \setminus \overline{U_{s}}$
is included in the union $C_{s}$ of $\{\eta'(s)\}=\{\varphi_{s}^{-1}(\infty)\}$ and the image by
$\varphi_{s}^{-1}$ of the part of $\mathbb{C}$ located both to the left of
$\gamma_{R}([\tau_{R},\infty))$ and to the right of $\gamma_{L}([\tau_{L},\infty))$.
Since $C_{s}$ is a compact subset of $\h$ by~\eqref{eq:Ksinfty-bounded-complement} and
\eqref{eq:gammaLgammaR-avoid-R}, it follows that $V \cap K_{s} = V \cap \overline{U_{s}} \cap \h$
for some open neighborhood $V$ of $b$ in $\mathbb{C}$, which together with the definition
of $U_{s}$ shows that $b \in \partial(\h \setminus K_{s})$ corresponds to a unique element
of $\widetilde{\partial}(\h \setminus K_{s})$. Thus $b \not\in \eta'((s,\infty))$
by~\eqref{eq:Ks-bottom-ab} and~\eqref{eq:SLE6-avoid-max-real}, which contradicts the
definition $b = \eta'(t)$ of $b$ in view of $t \in (s,\infty)$ and hence proves that $t = T_{b}$.
\end{proof}

\begin{proof}[Proof of Proposition~\ref{prop:SLE6-chunk-Jordan-D}]
We set $\tau := \inf\{ t \in [0,\inf(\eta')^{-1}(0)) \mid \partial K_{t} \cap \partial \D = \partial \D \}$, so that
$\{ t\in(0,\inf(\eta')^{-1}(0)) \mid \text{$\partial K_{t} \cap \partial \D \not= \partial \D$, $\eta'(t) \in \partial \D$}\} \subset [0,\tau)$.
Also let $\overline{\eta}'$ be a chordal $\SLE_6$ on $\h$ from $0$ to $\infty$,
set $\eta'_{*}(t) := x \exp( i \overline{\eta}'(t) )$ for $t \in [0,\infty)$,
define $K^{*}_{t}$ for $t\in[0,\infty)$ to be the complement in $\D$ of the
$0$-containing component of $\D \setminus \eta'_{*}([0,t])$, and set
$\tau_{*} := \inf\{ t \in [0,\infty) \mid \partial K^{*}_{t} \cap \partial \D = \partial \D \}$.
Then since $\{\eta'\}_{t\in[0,\tau)}$ has the same law as a time reparameterization of
$\{\eta'_{*}\}_{t\in[0,\tau_{*})}$ by a version \cite[Proposition 6.22]{law2005conformally} of
the locality of $\SLE_6$, it suffices to show the assertion for $\{\eta'_{*}\}_{t\in[0,\tau_{*})}$.

Fix any instance of $\overline{\eta}'$ with the property~\eqref{eq:SLE6-chunk-Jordan-H}
for any $t\in(0,\infty)$ with $\overline{\eta}'(t)\in\mathbb{R}$, which occurs
a.s.\ by Proposition~\ref{prop:SLE6-chunk-Jordan-H}. Let $t\in(0,\infty)$
satisfy $\partial K^{*}_{t} \cap \partial \D \not= \partial \D$, which means that
we can take a piecewise linear simple curve $\gamma^{t}:[0,1] \to \overline{\D} \setminus \eta'_{*}([0,t])$
with $\gamma^{t}(0)=0$, $\gamma^{t}([0,1)) \subset \D$ and $\gamma^{t}(1) \in \partial \D$.
Then $D_{t} := \D \setminus \gamma^{t}([0,1))$ is a simply connected domain with
$0 \not \in D_{t}$ and hence there exists a unique continuous branch $f_{t}$ of
$-i\log(\cdot/x)$ on $\widetilde{D}_{t} := D_{t}\cup(\partial \D \setminus \{\gamma^{t}(1)\})$
with $f_{t}(x)=0$, so that
$x\exp(i\cdot)|_{f_{t}(\widetilde{D}_{t})}:f_{t}(\widetilde{D}_{t}) \to \widetilde{D}_{t}$
is a homeomorphism with inverse $f_{t}$. Now since $\eta'_{*}([0,t]) \subset \widetilde{D}_{t}$,
$x\exp(if_{t}(\eta'_{*}(s)))=\eta'_{*}(s)=x\exp(i\overline{\eta}'(s))$
for any $s\in[0,t]$ and $f_{t}(\eta'_{*}(0))=f_{t}(x)=0=\overline{\eta}'(0)$,
it follows that $f_{t}\circ\eta'_{*}|_{[0,t]}=\overline{\eta}'|_{[0,t]}$ and thus
that $f_{t}(K^{*}_{t})$ is the complement in $\h$ of the unbounded component of
$\h \setminus f_{t}(\eta'_{*}([0,t])) = \h \setminus \overline{\eta}'([0,t])$.
In particular, if it also holds that $\eta'_{*}(t) \in \partial \D$,
then $\overline{\eta}'(t) = f_{t}(\eta'_{*}(t)) \in \mathbb{R}$, hence
$f_{t}(K^{*}_{t})$ has the property~\eqref{eq:SLE6-chunk-Jordan-H}, and
therefore its image $K^{*}_{t}$ by the homeomorphism $x\exp(i\cdot)|_{V_{t}}$
on a neighborhood $V_{t}$ of $f_{t}(\widetilde{D}_{t})$ in $\mathbb{C}$
has the property~\eqref{eq:SLE6-chunk-Jordan-D}, completing the proof.
\end{proof}

\section{Some L\'{e}vy process estimates} \label{sec:Levy_process_estimates}

In this appendix, we prove some estimates on the $3/2$-stable L\'evy processes with only downward jumps
(Propositions~\ref{prop:running-inf-mean-diverge}, \ref{prop:reflection_mean_before_tau},
\ref{prop:running-inf-mean-upper-bound}, \ref{prop:stable-X1-I1large-I2small} and
\ref{prop:stable-reflected-exp-moment} below), on which our proof of
Proposition~\ref{prop:good_chunks_percolate_half_plane} heavily relies.

Let $X^1, X^2$ be i.i.d.\ $3/2$-stable L\'evy processes with only downward jumps and starting from $0$.
Let $I_t^j = \inf_{0 \leq s \leq t} X_s^j$ and $S_t^j = \sup_{0 \leq s \leq t} X_s^j$ for $j=1,2$, respectively,
be the running infimum and supremum of $X^j$.  We let $\tau^j = \inf\{t \geq 1 \mid X_t^j = I_t^j\}$ for $j=1,2$ and $\tau = \tau^1 \wedge \tau^2$.
Note that by \cite[Chapter~VII, Theorem~1 and Chapter~VI, Proposition~3]{bertoin1996levy} we have
\begin{equation} \label{eq:Xj-Ij-fixed-time}
\p[ \tau^{j} = x ] \leq \p[ X^{j}_{x} = I^{j}_{x} ] = 0 \qquad\text{for any $x\in [1,\infty)$.}
\end{equation}

\begin{lemma}
\label{lem:tauj_tail}
There exists a constant $c>0$ such that $\p[\tau^j \geq x] = c  x^{-1/3} (1+o(1))$ as $x \to \infty$.
\end{lemma}
\begin{proof}
Let $x \in (1,\infty)$. Then setting $F(y) = y^{-1/2} \p[ I^{2}_{1} > -y ]$ for $y \in (0,\infty)$,
by the Markov property of $X^{1}$ and the scaling property of $X^{2}$ we have that
\begin{equation} \label{eq:tauj_tail_proof}
\begin{split}
  &\p[ \tau^1 > x \giv I_1^1, X_1^1]
 = \p[ I^{2}_{x-1} > I_1^1 - X_1^1 \giv X_1^1,I_1^1]
 = \p[ I^{2}_{1} > (x-1)^{-2/3}(I_1^1 - X_1^1) \giv X_1^1,I_1^1] \\
&= (x-1)^{-1/3} (X_1^1 - I_1^1)^{1/2} F\bigl( (x-1)^{-2/3}(X_1^1 - I_1^1) \bigr) \\
&\leq (x-1)^{-1/3} (X^{1}_{1}-I^{1}_{1})^{1/2} \Bigl( F\bigl( (x-1)^{-2/3}(X_1^1 - I_1^1) \bigr) \indicator_{\{X^{1}_{1}-I^{1}_{1} \leq (x-1)^{2/3}\}} + 1 \Bigr)
\end{split}
\end{equation}
and that $\lim_{y \downarrow 0}F(y)=c$ for some $c>0$
by \cite[Chapter~VIII, Proposition~2]{bertoin1996levy}.
Since $X_1^1 - I_1^1$ has a finite mean
by \cite[Chapter~VI, Proposition~3 and Chapter~VII, Corollary~2-(i)]{bertoin1996levy}
and hence a finite $1/2$-moment, it follows by an application of the dominated convergence
theorem based on~\eqref{eq:tauj_tail_proof} and $\lim_{y \downarrow 0}F(y)=c$ that
$\p[\tau^1 \geq x] = \mathbb{E}\bigl[ \p[ \tau^1 > x \giv I_1^1, X_1^1] \bigr]
  = c \mathbb{E}[(X_1^1 - I_1^1)^{1/2}] x^{-1/3} (1+o(1))$ as $x \to \infty$.
\end{proof}

As a consequence of Lemma~\ref{lem:tauj_tail}, there exists a constant $c>0$ such that
\begin{equation}
\label{eqn:tautail}
\p[ \tau \geq x] = \p[\tau^1 \geq x] \, \p[\tau^2 \geq x] = c  x^{-2/3} (1+o(1))
  \qquad\text{as $x\to\infty$.}
\end{equation}

\begin{lemma} \label{lem:Levy-large-overshoot}
Set $\tau^{1}_{0}=\inf\{ t \geq 0 : X^{1}_{t} < -1\}$. Then there exists a constant
$c>0$ such that for any $x \in [1,\infty)$,
\begin{equation}\label{eq:Levy-large-overshoot}
\p\bigl[ \tau^{1}_{0} <1, \, X^{1}_{\tau^{1}_{0}} < -x \bigr] \geq c x^{-3/2}.
\end{equation}
\end{lemma}

\begin{proof}
For $j=1,2$, by \cite[Theorem~4.3.7, Corollary~4.2.17 and Theorem~4.2.8]{Str2011}
we can decompose $X^{j}$ uniquely as $X^{j}_{t}=X^{j,0}_{t}+X^{j,1}_{t}$, where
$X^{j,0},X^{j,1}$ are independent L\'{e}vy processes with $X^{j,0}$ having jumps only
in $[-2,0)$ and $X^{j,1}$ a compound Poisson process with jumps only in $(-\infty,-2)$.
Then $X^{1,0},X^{1,1},X^{2,0},X^{2,1}$ are independent, $X^{1,0},X^{2,0}$
have the same law, so do $X^{1,1},X^{2,1}$, and hence the process
$X=\{X_{t}\}_{t\geq 0}$ defined by $X_{t}=X^{1,0}_{t}+X^{2,1}_{t}$
has the same law as $X^{1}$. Set $\tau_{0}=\inf\{ t \geq 0 : X_{t} < -1\}$.

Let $E^{1}$ denote the event that $\sup_{0\leq t<1}|X^{1}_{t}|\leq 1$, so that
$\p[ E^{1} ] >0$ by \cite[Chapter VIII, Proposition 3]{bertoin1996levy} and
$X^{1}|_{[0,1)}=X^{1,0}|_{[0,1)}$ a.s.\ on $E^{1}$. Let $x\in[2,\infty)$, and recall that
the numbers $N^{2}_{[-x,-2)},N^{2}_{(-\infty,-x)}$ of jumps in $[-x,-2),(-\infty,-x)$,
respectively, made by $X^{2,1}|_{[0,1)}$ are independent and have the Poisson
distribution with mean $c(2^{-3/2}-x^{-3/2}),cx^{-3/2}$, respectively,
for some $c>0$ independent of $x$. Since
$\{\tau_{0}<1, \, X_{\tau_{0}} < -(x-1) \} \supset E^{1} \cap \bigl\{ N^{2}_{[-x,-2)} = 0 < N^{2}_{(-\infty,-x)} \bigr\}$
a.s., it follows that
\begin{align*}
 \p\bigl[ \tau^{1}_{0} <1, \, X^{1}_{\tau^{1}_{0}} < -(x-1) \bigr]
&=\p\bigl[ \tau_{0} <1, \, X_{\tau_{0}} < -(x-1) \bigr] \\
&\geq \p\bigl[ E^{1} \cap \bigl\{ N^{2}_{[-x,-2)} = 0 < N^{2}_{(-\infty,-x)} \bigr\} \bigr] \\
&= \p[E^{1}] \cdot \p\bigl[ N^{2}_{[-x,-2)} = 0 \bigr] \cdot \p\bigl[ N^{2}_{(-\infty,-x)} \geq 1 \bigr] \\
&= \p[E^{1}] e^{-c(2^{-3/2}-x^{-3/2})} \bigl(1-e^{-cx^{-3/2}}\bigr) \geq c'(x-1)^{-3/2}
\end{align*}
for some $c'>0$ independent of $x \in [2,\infty)$, proving~\eqref{eq:Levy-large-overshoot}.
\end{proof}

\begin{proposition} \label{prop:running-inf-mean-diverge}
There exists a constant $c > 0$ so that $\E[ - I^{1}_{\tau} \indicator_{\{\tau<A\}} ] \geq c \log A$
for any $A \in [2,\infty)$.
\end{proposition}
\begin{proof}
We first claim that
\begin{equation}
\label{eqn:x_giv_tau_tail}
\p[ X_k^1 \geq \alpha k^{2/3} \giv \tau^{1} \geq k] \to 0 \quad\text{as}\quad \alpha \to \infty \quad\text{uniformly in $k$}.
\end{equation}
To see this, fix $\alpha \in [2,\infty)$. For each $k \in \mathbb{N}$, let $\sigma_k^1 = \inf\{t \geq 0 : X_t^1 = k^{2/3}\}$.
Then since $X^1$ has only downward jumps, we have $\{X^{1}_{k} \geq \alpha k^{2/3}\} \subset \{\sigma^{1}_{k}\leq k\}$ and hence that
\begin{align*}
 \p[ X_k^1 \geq \alpha k^{2/3} \giv \tau^1 \geq k]
&= \p[ X_k^1 \geq \alpha k^{2/3},\ \sigma_k^1 \leq k \giv \tau^1 \geq k] \\
&\leq \p[ X_k^1 \geq \alpha k^{2/3} \giv \sigma_k^1 \leq k, \tau^1 \geq k] \\
&= \frac{\p[ X_k^1 \geq \alpha k^{2/3}, \tau^1 \geq k \giv \sigma_k^1 \leq k, \tau^1 \geq \sigma_k^1]}{\p[ \tau^1 \geq k \giv \sigma_k^1 \leq k, \tau^1 \geq \sigma_k^1]} \\
&\leq \frac{\p\bigl[ \sup_{\sigma^{1}_{k} \leq t\leq \sigma^{1}_{k}+k} X^{1}_{t} \geq \alpha k^{2/3} \giv \sigma_k^1 \leq k, \tau^1 \geq \sigma_k^1 \bigr]}{\p[ \tau^1 \geq k \giv \sigma_k^1 \leq k, \tau^1 \geq \sigma_k^1]}.
\end{align*}
Applying the strong Markov property of $X^{1}$ at the time $\sigma_k^1$
and the scaling property of $X^{1}$, we easily see
that the denominator is at least $\p[ I^{1}_{1} > -1 ]$,
which is positive by \cite[Chapter~VIII, Proposition~2]{bertoin1996levy}, and that
the numerator is at most $\p[ S^{1}_{1} \geq \alpha-1 ]$, which tends to $0$
as $\alpha \to \infty$, proving~\eqref{eqn:x_giv_tau_tail}.
On the other hand, since $I^{1}_{k}=I^{1}_{1}$ on $\{\tau^{1} > k\}$,
from Lemma~\ref{lem:tauj_tail} and \cite[Chapter VIII, Proposition 4]{bertoin1996levy} we obtain
\begin{equation} \label{eqn:I_giv_tau_tail}
\p[ I_k^1 \leq -\alpha k^{2/3} \giv \tau^{1} \geq k]
  \leq \frac{\p[ I^{1}_{1} \leq -\alpha k^{2/3} ]}{\p[ \tau^{1} \geq k]}
  \leq c \alpha^{-3/2} k^{-1} k^{1/3} \leq c \alpha^{-3/2}
\end{equation}
for any $k \in \mathbb{N}$ for some $c>0$.

Further, there exists $p_0 > 0$ so that
\begin{align} \label{eq:X1_giv_tau_tail}
\p[ X_k^1 \geq I_k^1+1 \giv \tau^{1} \geq k ]  \geq p_0 \qquad \text{for all $k \in \mathbb{N}$.}
\end{align}
Indeed, for any $t \in [2,\infty)$ and for any $k \in \mathbb{N}$, we have $\p[ X_1^1 \geq I_1^1+t ] > 0$
by \cite[Chapter~VI, Proposition~3 and Chapter~VII, Corollary~2-(i)]{bertoin1996levy},
\begin{align*}
 \p[ X^{1}_{k} \geq I^{1}_{k} + t, \, \tau^{1} \geq k ]
&=\frac{ \p[ \inf_{k \leq t \leq k+1}(X^{1}_{t}-X^{1}_{k}) > -1, \, X^{1}_{k} \geq I^{1}_{k} + t, \, \tau^{1} \geq k ] }{ \p[ I^{1}_{1} > -1 ] } \\
&\leq \p[ I^{1}_{1} > -1 ]^{-1} \p[ X^{1}_{k+1} \geq I^{1}_{k+1} + t-1, \, \tau^{1} \geq k+1 ],
\end{align*}
and hence $\p[ X_k^1 \geq I_k^1+1 \giv \tau^{1} \geq k ] > 0$.
Also, noting that $\p[ I^{1}_{1} \leq -1 ] >0$ by \cite[Chapter~VIII, Proposition~4]{bertoin1996levy}, for any $k \in \mathbb{N}$ we have
\begin{align*}
\p[ X^{1}_{k} < I^{1}_{k} + 1 \giv \tau^{1} \geq k ]
 &=\frac{ \p[ \inf_{k \leq t \leq k+1}(X^{1}_{t}-X^{1}_{k}) \leq -1, \, X^{1}_{k} < I^{1}_{k}+1 \giv \tau^{1} \geq k ] }{ \p[ I^{1}_{1} \leq -1 ] } \\
&\leq \frac{ \p[ \tau^{1} < k+1 \giv \tau^{1} \geq k ] }{ \p[ I^{1}_{1} \leq -1 ] }
 =\frac{ 1 - \p[ \tau^{1} \geq k+1 ]/\p[ \tau^{1} \geq k ] }{ \p[ I^{1}_{1} \leq -1 ] }
 \xrightarrow{k \to \infty} 0,
\end{align*}
where the last limit follows by Lemma~\ref{lem:tauj_tail}. The above results together yield~\eqref{eq:X1_giv_tau_tail}.

Now, combining~\eqref{eqn:x_giv_tau_tail}, \eqref{eqn:I_giv_tau_tail} and~\eqref{eq:X1_giv_tau_tail}, we can choose $\alpha \in [2,\infty)$ sufficiently large so that
\begin{equation}
\label{eqn:x1alphakbound}
\p\bigl[ I_k^1  +1 \leq X_k^1 \leq \alpha k^{2/3}, \, I_k^1 \geq -\alpha k^{2/3} \giv \tau^{1} \geq k \bigr] \geq p_0/2 \qquad\text{for all $k \in \mathbb{N}$.}
\end{equation}
Let $k \in \mathbb{N}$, set $\tau^{1}_{k} = \inf\{ t \geq 0 : X^{1}_{t+k} - X^{1}_{k} < -1 \}$ and let $E^{1}_{k}$
denote the event that $\tau^{1}_{k} < 1$ and $X^{1}_{\tau^{1}_{k}+k} - X^{1}_{k} < -2 \alpha k^{2/3}$.
Then since $\{ X^{1}_{t+k} - X^{1}_{k}\}_{t\geq 0}$ is independent of $X^{1}|_{[0,k]}$
and has the same law as $X^{1}$, we have $\p[E^{1}_{k}] \geq c'k^{-1}$ for a constant $c'>0$
independent of $k$ by Lemma~\ref{lem:Levy-large-overshoot}, and it follows that
\begin{equation}
\label{eqn:ek1givlbd}
\p[E_k^1 \giv  I_k^1 + 1\leq X_k^1 \leq \alpha k^{2/3}, \, I_k^1 \geq -\alpha k^{2/3}, \, \tau^{1} \geq k] = \p[ E_k^1]\geq c' k^{-1}.
\end{equation}
Further, on the event $E_k^1 \cap \{ I_k^1 + 1 \leq X_k^1 \leq \alpha k^{2/3}, \, I_k^1 \geq -\alpha k^{2/3}, \, \tau^{1} \geq k, \tau^{2} \geq k+1 \}$
we have $\tau = \tau^{1} = \tau^{1}_{k}+k \in [k, k+1)$ and $I^1_\tau \leq - \alpha k^{2/3}$.
It therefore follows that there exist constants $c_1,c_2 > 0$ so that for any $A \in \mathbb{N}$,
\begin{align*}
 \E&[ - I^{1}_{\tau} \indicator_{\{ \tau < A \}} ] \\
&\geq \sum_{k=1}^{A-1} \alpha k^{2/3} \p\bigl[ E_k^1 \cap \{  I_k^1 +1\leq X_k^1 \leq \alpha k^{2/3}, \,  I_k^1 \geq -\alpha k^{2/3}, \, \tau^{1} \geq k, \, \tau^{2} \geq k+1 \} \bigr]\\
&\geq \sum_{k=1}^{A-1} k^{2/3} c_1 k^{-1} \p[\tau^{1} \geq k] \, \p[\tau^{2} \geq k+1] \quad \!\text{(by~\eqref{eqn:x1alphakbound}, \eqref{eqn:ek1givlbd}, $X^{1},X^{2}$ independent})\\
&\geq \sum_{k=1}^{A-1} c_2 k^{-1} \quad\text{(by Lemma~\ref{lem:tauj_tail})}\\
 &\geq c_{2} \log A.
\end{align*}
This proves the result.
\end{proof}

\begin{proposition} \label{prop:reflection_mean_before_tau}
$\sup_{A \in [1,\infty)}\E[ (X^{1}_{A} - I^{1}_{A}) \indicator_{\{\tau \geq A\}}] < \infty$.
\end{proposition}

\begin{proof}
Let $A \in [1,\infty)$. Since $I^{1}_{A} = I^{1}_{1}$ on the event $\{\tau > A \}$ and
$\E[|I_1^1|]<\infty$ by \cite[Chapter~VIII, Proposition~4]{bertoin1996levy},
setting $c_{1}:=\E[|I^{1}_{1}|]$ and using~\eqref{eq:Xj-Ij-fixed-time}, we have
\begin{equation} \label{eqn:infbound}
\E[ -I^{1}_{A} \indicator_{\{\tau \geq A \}}]
= \E[ -I^{1}_{1} \indicator_{\{\tau > A \}}]
\leq \E[ |I^{1}_{1}| ] = c_{1} < \infty.
\end{equation}

Set $\sigma^{1}_{A} = \inf\{t \geq 0 : X^{1}_{t} = A^{2/3}\}$.
By the independence of $X^1,X^2$ and Lemma~\ref{lem:tauj_tail} we have that
\begin{equation}
\label{eqn:xm_bound1}
\E[ X^{1}_{A} \indicator_{\{\tau \geq A \}} ]
 = \p[ \tau^{2} \geq A ] \,  \E[ X^{1}_{A} \indicator_{\{\tau^{1} \geq A \}} ]
 \asymp A^{-1/3} \, \E[ X^{1}_{A} \indicator_{\{\tau^{1} \geq A \}} ].
\end{equation}
We moreover have that
\begin{align}
 \E[ X^{1}_{A} \indicator_{\{\tau^{1} \geq A \}} ]
&=\E\bigl[ X^{1}_A \indicator_{\{\tau^{1} \geq A \}} \bigl(\indicator_{\{ \sigma^{1}_{A} \geq A\}} + \indicator_{\{ \sigma^{1}_{A} < A\}}\bigr) \bigr] \notag \\
&\leq A^{2/3}\p[ \tau^{1} \geq A] + \E\bigl[ X^{1}_{A} \indicator_{\{ \sigma^{1}_{A} < A \leq \tau^{1}\}}\bigr] \notag\\
&\lesssim A^{1/3} + \E\bigl[ X^{1}_{A} \indicator_{\{ \sigma^{1}_{A} < A \leq \tau^{1} \}}\bigr] \quad\text{(by Lemma~\ref{lem:tauj_tail})} \notag\\
&\leq A^{1/3} + \E\bigl[ \bigl( X^{1}_{A} - X^{1}_{\sigma^{1}_{A}} \bigr) \indicator_{\{ \sigma^{1}_{A} < A \leq \tau^{1} \}} \bigr] + A^{2/3} \, \p[\tau^{1} \geq A] \notag\\
&\lesssim A^{1/3} + \E\bigl[ \sup\nolimits_{ 0\leq t \leq A } \bigl( X^{1}_{t + \sigma^{1}_{A}} - X^{1}_{\sigma^{1}_{A}} \bigr) \indicator_{\{ \sigma^{1}_{A} < \tau^{1} \}} \bigr] \quad\text{(by Lemma~\ref{lem:tauj_tail})}\notag\\
&=A^{1/3} + \E[S^{1}_{A}]\, \p[ \sigma^{1}_{A} < \tau^{1}] \quad\text{(by strong Markov at time $\sigma^{1}_{A}$)} \notag \\
& \lesssim A^{1/3} + A^{2/3} \, \p[ \sigma^{1}_{A}< \tau^{1}], \label{eqn:xm_bound2}
\end{align}
where the strong Markov property of $X^{1}$ refers to \cite[Chapter~I, Proposition 6]{bertoin1996levy}
and we used in the last step that $\E[ S^{1}_{A}] =A^{2/3} \E[ S^{1}_{1} ] \asymp A^{2/3}$
by the scaling property of $X^{1}$ and \cite[Chapter~VII, Corollary~2-(i)]{bertoin1996levy}.
Further, by considering $\{X^{1}_{t+1}-X^{1}_{1}\}_{t\geq 0}$, which is independent of
$X^{1}|_{[0,1]}$ and has the same law as $X^{2}$, we obtain
\begin{align*}
\p[ \sigma^{1}_{A} < \tau^{1} ] 
 \leq \p[ \sigma^{1}_{A} \leq 1 ] + \p\bigl[ \text{$X^{1}_{1}<A^{2/3}$, $X^{2}$ hits $A^{2/3} - X^{1}_{1}$ before hitting $(-\infty,I^{1}_{1} - X^{1}_{1}]$} \bigr].
\end{align*}
Then $\p[ \sigma^{1}_{A} \leq 1 ] \leq \p[ S^{1}_{1} \geq A^{2/3} ] \leq c_{2} e^{-A^{2/3}}$
for some $c_{2}>0$ independent of $A$ by \cite[Chapter VII, Corollary~2-(i)]{bertoin1996levy}.
Also,  \cite[Chapter~VII, Theorem~8]{bertoin1996levy} implies that
\begin{align*}
\p&\bigl[ \text{$X^{2}$ hits $A^{2/3} - X^{1}_{1}$ before hitting $(-\infty,I^{1}_{1} - X^{1}_{1}]$} \bigm| X^{1}_{1},I^{1}_{1} \bigr] = \frac{W(X_1^1 - I_1^1)}{W(X_1^1 - I_1^1 + A^{2/3} - X_1^1)},
\end{align*}
where $W$ is the scale function of $X^{1}$ introduced in \cite[Section~VII.2]{bertoin1996levy}. 
Moreover, since the Laplace exponent $\psi(\lambda)$ of $X^{1}$ is proportional to $\lambda^{3/2}$ by the scaling property of $X^{1}$ and the Laplace transform of $W$ is $1/\psi$ by definition, we see that $W(x)$ is proportional to $x^{1/2}$. It thus follows that
\begin{align*}
&\p\bigl[ \text{$X^{2}$ hits $A^{2/3} - X^{1}_{1}$ before hitting $(-\infty,I^{1}_{1} - X^{1}_{1}]$} \bigm| X^{1}_{1},I^{1}_{1} \bigr] \\
&= \frac{(X_1^1 - I_1^1)^{1/2}}{(X_1^1 - I_1^1 + A^{2/3} - X_1^1)^{1/2}} \leq (X_1^1 - I_1^1)^{1/2} A^{-1/3}
\end{align*}
a.s.\  on $\{X_1^1 < A^{2/3}\}$, and hence taking expectations yields
$\p[ \sigma^{1}_{A}< \tau^{1}] \lesssim A^{-1/3}$
in view of the fact that $\mathbb{E}[(X_1^1 - I_1^1)^{1/2}] < \infty$
as noted in the proof of Lemma~\ref{lem:tauj_tail}.
Combining this with~\eqref{eqn:xm_bound1} and~\eqref{eqn:xm_bound2} implies that
$\sup_{A \in [1,\infty)}\E[ X^{1}_{A} \indicator_{\{ \tau \geq A\}}] <\infty$,
which together with~\eqref{eqn:infbound} proves the result.
\end{proof}

In fact, we also have an upper bound on $\mathbb{E}[ (- I^{1}_{\tau})^{p} \indicator_{\{\tau<A\}} ]$
for $p \in [1,\frac{3}{2})$, which with $p=1$ matches the lower bound obtained in Proposition~\ref{prop:running-inf-mean-diverge}, as follows.

\begin{proposition} \label{prop:running-inf-mean-upper-bound}
For any $p \in [1,\frac{3}{2})$ there exists $c_{p} \in (0,\infty)$ depending only on $p$ and the law of $X^{1}$ such that for any $A \in [2,\infty)$,
\begin{equation}
\mathbb{E}[ (- I^{1}_{\tau})^{p} \indicator_{\{\tau<A\}} ] \leq
\begin{cases}
  c_{1} \log A & \textrm{if $p=1$,} \\
  c_{p}A^{2(p-1)/3} & \textrm{if $p \in (1,\frac{3}{2})$.}
\end{cases}
\end{equation}
\end{proposition}

\begin{proof}
Let $p \in [1,\frac{3}{2})$, $A \in [2,\infty)$ and set $A':=\min\{n\in\mathbb{N}\mid A\leq 2^{n}\}$. We have
\begin{align}
\mathbb{E}[ (- I^{1}_{\tau})^{p} \indicator_{\{\tau<A\}} ]
\leq
\mathbb{E}[ (- I^{1}_{\tau})^{p} \indicator_{\{\tau<2^{A'}\}} ]
&=
\sum_{k=1}^{A'} \mathbb{E}\bigl[ (- I^{1}_{\tau})^{p} \indicator_{\{2^{k-1} \leq \tau<2^{k}\}} \bigr] \notag \\
&\leq
\sum_{k=1}^{A'} \mathbb{E}\bigl[ (- I^{1}_{2^k})^{p} \indicator_{\{2^{k-1} \leq \tau<2^{k}\}} \bigr].
\label{eq:running-inf-mean-upper-bound-proof1}
\end{align}
Then for each $k \in \mathbb{N}$, since $X^{1}_{t}\geq I^{1}_{1} = I^{1}_{t}$ a.s.\ on $\{\tau^{1} \geq t\}$
for each $t \in [1,\infty)$ and $I^{1}_{2^{k}}=\inf_{2^{k-1}\leq s\leq 2^{k}}X^{1}_{s}$ on $\{2^{k-1} \leq \tau^{1}<2^{k}\}$,
it follows from the independence of $X^{1},X^{2}$, the convexity of the function $\mathbb{R}\ni x\mapsto (x^{+})^{p}$,
the Markov property of $X^{1}$ at time $2^{k-1}$ and
$\mathbb{P}[ \tau^{1} \geq 2^{k-1} ] = \mathbb{P}[ \tau^{2} \geq 2^{k-1} ]$ that
\begin{align}
&\mathbb{E}\bigl[ (- I^{1}_{2^k})^{p} \indicator_{\{2^{k-1} \leq \tau < 2^{k}\}} \bigr] \notag \\
&=
\mathbb{E}\bigl[ (- I^{1}_{2^k})^{p} \indicator_{\{2^{k-1} \leq \tau \leq \tau^{1} < 2^{k}\}} \bigr]
  + \mathbb{E}\bigl[ (- I^{1}_{2^k})^{p} \indicator_{\{2^{k-1} \leq \tau < 2^{k} \leq \tau^{1}\}} \bigr] \notag \\
&\leq
2^{p-1} \mathbb{E}\bigl[ \bigl( X^{1}_{2^{k-1}} - \inf\nolimits_{2^{k-1}\leq s\leq 2^{k}}X^{1}_{s} \bigr)^{p} \indicator_{\{\tau^{1} \geq 2^{k-1}\}} \bigr] \mathbb{P}[ \tau^{2} \geq 2^{k-1} ] \notag \\
&\qquad + 2^{p-1} \mathbb{E}\bigl[ \bigl( (-X^{1}_{2^{k-1}})^{+} \bigr)^{p} \indicator_{ \{ \tau \geq 2^{k-1} \} } \bigr]
  + \mathbb{E}[ (- I^{1}_{1})^{p} \indicator_{\{\tau^{1} \geq 2^{k}\}} ] \mathbb{P}[\tau^{2} \geq 2^{k-1}] \notag \\
&\leq
2^{p-1} \mathbb{E}[ (- I^{1}_{2^{k-1}})^{p} ] \mathbb{P}[ \tau^{2} \geq 2^{k-1} ]^{2}
  + (2^{p-1}+1) \mathbb{E}[ (- I^{1}_{1})^{p} ] \mathbb{P}[\tau^{2} \geq 2^{k-1}].
\label{eq:running-inf-mean-upper-bound-proof2}
\end{align}
Now, recalling that for some $c \in (0,\infty)$ determined solely by the law of $X^{1}$ we have
$\mathbb{E}[ (- I^{1}_{2^{k-1}} )^{p} ] = (2^{k-1})^{2p/3} \mathbb{E}[ (- I^{1}_{1})^{p} ] \leq c (2^{k-1})^{2p/3}$
by the scaling property of $X^{1}$ and \cite[Chapter VIII, Proposition 4]{bertoin1996levy}
and $\mathbb{P}[ \tau^{2} \geq 2^{k-1} ] \leq c 2^{-(k-1)/3}$ by Lemma~\ref{lem:tauj_tail}, we conclude from
\eqref{eq:running-inf-mean-upper-bound-proof1}, \eqref{eq:running-inf-mean-upper-bound-proof2} and $2^{A'}<2A$ that
\begin{align*}
\mathbb{E}[ (- I^{1}_{\tau})^{p} \indicator_{\{\tau<A\}} ] 
 &\leq
\sum_{k=1}^{A'} \bigl( 2^{p-1} c^{3} (2^{k-1})^{2p/3} 2^{-2(k-1)/3} + (2^{p-1}+1) c^{2} 2^{-(k-1)/3} \bigr) \\
&\leq
2^{p-1} c^{3} \sum_{k=1}^{A'} (2^{2(p-1)/3})^{k-1} + \frac{(2^{p-1}+1)c^{2}}{1-2^{-1/3}}
\leq
\begin{cases}
  c_{1} \log A & \textrm{if $p=1$,} \\
  c_{p}A^{2(p-1)/3} & \textrm{if $p \in (1,\frac{3}{2})$}
\end{cases}
\end{align*}
for some $c_{p}\in(0,\infty)$ explicit in $2^{p-1}$ and $c$, completing the proof.
\end{proof}

We also need the following propositions in the proof of Proposition~\ref{prop:good_chunks_percolate_half_plane}.

\begin{proposition} \label{prop:stable-X1-I1large-I2small}
There exists $c_{1} \in (0,\infty)$ such that for any $y \in (0,1]$,
\begin{equation} \label{eq:stable-X1-I1large-I2small}
\mathbb{P}\bigl[ \tau = \tau^{2} < 2, \, X^{1}_{\tau} - I^{1}_{\tau} > 4, \, I^{2}_{\tau} > -y \bigr]
  \geq c_{1}y.
\end{equation}
\end{proposition}

\begin{proof}
Let $x,y \in (0,\infty)$. By~\eqref{eq:Xj-Ij-fixed-time} and the independence of $X^{1},X^{2}$,
\begin{align}
\mathbb{P}&\bigl[ \tau = \tau^{2} < 2, \, X^{1}_{\tau} - I^{1}_{\tau} > x, \, I^{2}_{\tau} > -2y \bigr] \notag \\
&\geq \mathbb{P}\bigl[ \tau^{1} \geq 2, \, \inf\nolimits_{ 1 \leq t \leq 2 }(X^{1}_{t}-I^{1}_{t}) > x, \, \tau^{2} \leq 2, \, I^{2}_{\tau^{2}} > -2y \bigr] \notag \\
&= \mathbb{P}\bigl[ \tau^{1} \geq 2, \, \inf\nolimits_{ 1 \leq t \leq 2 }(X^{1}_{t}-I^{1}_{t}) > x \bigr] \, \mathbb{P}\bigl[ \tau^{2} \leq 2, \, I^{2}_{\tau^{2}} > -2y \bigr].
\label{eq:stable-X1-I1large-I2small-X1X2indep}
\end{align}
For the first term of the product in~\eqref{eq:stable-X1-I1large-I2small-X1X2indep},
we have $\mathbb{P}[ X^{1}_{1} - I^{1}_{1} > 2x ] > 0$
by \cite[Chapter~VI, Proposition~3 and Chapter~VII, Corollary~2-(i)]{bertoin1996levy} and
$\mathbb{P}[ I^{1}_{1} > -x ] > 0$ by \cite[Chapter~VIII, Proposition~2]{bertoin1996levy},
and then, since $\{ X^{1}_{t+1} - X^{1}_{1} \}_{t \geq 0}$ has the same law as $X^{1}$
and is independent of $\{ X^{1}_{t} \}_{0 \leq t\leq 1}$,
\begin{align}
  \mathbb{P}\bigl[ \tau^{1} \geq 2, \, \inf\nolimits_{ 1 \leq t \leq 2 }(X^{1}_{t}-I^{1}_{t}) > x \bigr]
&\geq \mathbb{P}\bigl[ X^{1}_{1} - I^{1}_{1} > 2x, \, \inf\nolimits_{ 0 \leq t \leq 1 }(X^{1}_{t+1}-X^{1}_{1}) > -x \bigr] \notag \\
&= \mathbb{P}[ X^{1}_{1} - I^{1}_{1} > 2x ] \,  \mathbb{P}[ I^{1}_{1} > -x ]
  > 0.
\label{eq:stable-X1-I1large-I2small-X1}
\end{align}
For the second term of the product in~\eqref{eq:stable-X1-I1large-I2small-X1X2indep},
setting $\tau_{z} := \inf\{ t \geq 0 \mid X^{2}_{t} \leq -z \}$ for $z \in (0,\infty)$,
by $\tau^{2} \geq 1$ and the Markov property of $X^{2}$ at time $1$ we have
\begin{align}
&\mathbb{P}\bigl[ \tau^{2} \leq 2, \, I^{2}_{\tau^{2}} > -2y \bigr] \notag \\
&\geq \mathbb{P}\bigl[ I^{2}_{1} > -y, \, X^{2}_{1} \in [1,2], \, \tau_{y} \leq 2, \, X^{2}_{\tau_{y}} > -2y \bigr] \notag \\
&\geq \mathbb{P}\bigl[ I^{2}_{1} > -y, \, X^{2}_{1} \in [1,2] \bigr] \, \inf\nolimits_{z \in [1,2]}\mathbb{P}\bigl[ \tau_{z+y} \leq 1, \, X^{2}_{\tau_{z+y}} > -z-2y \bigr].
\label{eq:stable-X1-I1large-I2small-X2}
\end{align}
Since the left-hand side of~\eqref{eq:stable-X1-I1large-I2small} is non-decreasing in $y$,
in view of~\eqref{eq:stable-X1-I1large-I2small-X1X2indep}, \eqref{eq:stable-X1-I1large-I2small-X1}
and~\eqref{eq:stable-X1-I1large-I2small-X2} it suffices to show that there exist
$c_{2},c_{3},c_{4} \in (0,\infty)$ with $c_{2} \leq 1$ such that for any $y \in (0,c_{2}]$ and any $z \in [1,3]$,
\begin{align} \label{eq:stable-X1-I1large-I2small-X2-term1}
\mathbb{P}\bigl[ I^{2}_{1} > -y, \, X^{2}_{1} \in [1,2] \bigr] &\geq c_{3}y^{1/2}, \\
\mathbb{P}\bigl[ \tau_{z} \leq 1, \, X^{2}_{\tau_{z}} > -z-y \bigr] &\geq c_{4}y^{1/2}.
\label{eq:stable-X1-I1large-I2small-X2-term2}
\end{align}

For~\eqref{eq:stable-X1-I1large-I2small-X2-term1}, set $\sigma_{z} := \inf\{ t \geq 0 \mid X^{2}_{t} > z \}$
for $z \in [0,\infty)$, so that $\mathbb{P}[ \sigma_{z} < \infty ] = 1$
by \cite[Chapter~VII, Proof of Theorem~1]{bertoin1996levy}, and let $y,u \in (0,\infty)$.
Then by the scaling property of $X^{2}$ and \cite[Chapter~VIII, Proposition~2]{bertoin1996levy},
for a constant $c_{5} \in [1,\infty)$ independent of $y,u$ we have
$\mathbb{P}\bigl[ I^{2}_{u/2} > -y \bigr] = \mathbb{P}\bigl[ I^{2}_{1} > -(u/2)^{-2/3}y \bigr] \leq c_{5}u^{-1/3}y^{1/2}$,
which together with the scaling property of $X^{2}$ and \cite[Chapter~VII, Theorem~8]{bertoin1996levy}
implies that, with $a := u^{-2/3}$,
\begin{equation} \label{eq:stable-X1-I1large-I2small-X2-term1-1}
\begin{split}
\mathbb{P}\bigl[ I^{2}_{\sigma_{a}} > -ay, \, \sigma_{a} < 1/2 \bigr]
  &= \mathbb{P}\bigl[ I^{2}_{\sigma_{1}} > -y, \, \sigma_{1} < u/2 \bigr] \\
&= \mathbb{P}\bigl[ I^{2}_{\sigma_{1}} > -y \bigr]
  - \mathbb{P}\bigl[ I^{2}_{\sigma_{1}} > -y, \, \sigma_{1} \geq u/2 \bigr] \\
&\geq y^{1/2}/(1+y)^{1/2} - \mathbb{P}\bigl[ I^{2}_{u/2} > -y \bigr] \\
&\geq \bigl( (1+y)^{-1/2} - c_{5}u^{-1/3} \bigr) y^{1/2}.
\end{split}
\end{equation}
Choosing $u := 125c_{5}^{3}$ and replacing $y$ with $a^{-1}y=u^{2/3}y$ in
\eqref{eq:stable-X1-I1large-I2small-X2-term1-1}, for any $y \in (0,a]$ we obtain
\begin{equation} \label{eq:stable-X1-I1large-I2small-X2-term1-2}
\mathbb{P}\bigl[ I^{2}_{\sigma_{a}} > -y, \, \sigma_{a} < 1/2 \bigr]
  \geq \bigl( 2^{-1/2} - c_{5}u^{-1/3} \bigr) u^{1/3} y^{1/2}
  > 2c_{5}y^{1/2}.
\end{equation}
Moreover, for all $x>0$, we let $\p_x$ denote the law of $X^{2}+x$, and
let $\p_x^{+}$ be the probability measure on the space of $[0,\infty)$-valued cadlag paths
(equipped with the $\sigma$-algebra generated by the coordinate process $X=\{X_{t}\}_{t\geq 0}$) given by 
\begin{align*}
\p_x^{+}\bigl[ X_t \in dy \bigr] = \frac{W(y)}{W(x)} \p_x\bigl[ X_t \in dy, \, t < T_{(-\infty,0)} \bigr],
\end{align*}
where $W$ is as in the proof of Proposition~\ref{prop:reflection_mean_before_tau} and
$T_{(-\infty,0)} := \inf\{ t \geq 0 \mid X_{t} < 0 \}$. Then,
by \cite[Chapter~VII, Proposition~14]{bertoin1996levy}
there exists a probability measure $\p_0^{+}$ such that $\p_x^{+} \to \p_0^{+}$ as $x \downarrow 0$
in the sense of finite-dimensional distributions, and
by \cite[Chapter~VII, Corollary~16]{bertoin1996levy} we have
\begin{align*}
\p_0^{+}\bigl[ X_t \in dy \bigr] = \frac{yW(y)}{t} \p\bigl[ X^{2}_{t} \in dy \bigr].
\end{align*}
It follows that 
\begin{align*}
&\p\bigl[ X_{1/2}^2 > 3/2 - b, \, I_{1/2}^2 \geq - b \bigr] = \p_b\Bigl[ X_{1/2} > 3/2, \, \inf_{0 \leq s \leq 1/2} X_{s} > 0 \Bigr] \\
&=\p_b\bigl[ X_{1/2} > 3/2, \, T_{(-\infty,0)} > 1/2 \bigr] = W(b) \int_{3/2}^{\infty} \frac{1}{W(z)} \p_b^{+}\bigl[ X_{1/2} \in dz \bigr]
\end{align*}
for any $b > 0$, and
\begin{align*}
\int_{3/2}^{\infty} \frac{1}{W(z)} \p_b^{+}\bigl[ X_{1/2} \in dz \bigr]
  &\xrightarrow{b \downarrow 0} \int_{3/2}^{\infty} \frac{1}{W(z)} \p_0^{+}\bigl[ X_{1/2} \in dz \bigr] 
  =2\int_{3/2}^{\infty} z \p\bigl[ X^{2}_{1/2} \in dz \bigr] > 0
\end{align*}
since $\mathbb{P}[ X^{2}_{1/2} > 3/2 ] > 0$ by the scaling property of
$X^{2}$ and \cite[Chapter~VII, Corollary~2-(i)]{bertoin1996levy}.
We see therefore that, if $c_{5} \in [1,\infty)$ is large enough so that
$a = u^{-2/3} = (5c_{5})^{-2}$ is small enough, then 
\begin{equation}\label{eq:stable-X1-I1large-I2small-X2-term1-3}
q_1 := \p\bigl[ \sigma_{3/2 - a} < 1/2, \, I_{\sigma_{3/2-a}}^2 \geq -a \bigr]
  \geq \p\bigl[ X_{1/2}^2 > 3/2 - a, \, I_{1/2}^2 \geq -a \bigr] > 0.
\end{equation}
Noting that $\mathbb{P}\bigl[ X^{2}_{\sigma_{z}}=z \bigr] = 1$ for $z \in (0,\infty)$ and that
$q_{2} := \mathbb{P}\bigl[ \sup_{ 0 \leq t \leq 1 }|X^{2}_{t}| \leq 1/2 \bigr] > 0$
by \cite[Chapter~VIII, Proposition~3]{bertoin1996levy}, from the strong Markov property of
$X^{2}$ at times $\sigma_{a},\sigma_{3/2}$ (see, e.g., \cite[Chapter~I, Proposition~6]{bertoin1996levy}),
\eqref{eq:stable-X1-I1large-I2small-X2-term1-3} and~\eqref{eq:stable-X1-I1large-I2small-X2-term1-2} we get
\begin{align*}
&\mathbb{P}\bigl[ I^{2}_{1} > -y, \, X^{2}_{1} \in [1,2] \bigr] \\
&\geq \mathbb{P}\Bigl[ I^{2}_{\sigma_{a}} > -y, \, \sigma_{a} < \tfrac{1}{2}, \, \inf_{ \sigma_{a} \leq t \leq \sigma_{3/2} } X^{2}_{t} \geq 0, \, \sigma_{3/2} - \sigma_{a} < \tfrac{1}{2}, \, \sup_{ 0 \leq t \leq 1 } \bigl|X^{2}_{t+\sigma_{3/2}}-\tfrac{3}{2}\bigr| \leq \tfrac{1}{2} \Bigr] \\
&= \mathbb{P}\bigl[ I^{2}_{\sigma_{a}} > -y, \, \sigma_{a} < \tfrac{1}{2} \bigr] q_{1} q_{2}
  \geq (2c_{5}q_{1}q_{2}) y^{1/2},
\end{align*}
proving~\eqref{eq:stable-X1-I1large-I2small-X2-term1}.

Next, to see~\eqref{eq:stable-X1-I1large-I2small-X2-term2}, let $y \in (0,1]$ and $z \in [1,3]$.
By \cite[Chapter~VIII, Exercise~3]{bertoin1996levy},
\begin{equation} \label{eq:stable-X1-I1large-I2small-X2-term2-1}
\mathbb{P}\bigl[ X^{2}_{\tau_{z}} > -z-y \bigr]
  = \frac{1}{\pi} \int_{0}^{y/(z+y)}x^{-1/2}(1-x)^{-1/2}\,dx 
  \geq \frac{y^{1/2}}{\pi}.
\end{equation}
Let $b \in (0,\infty)$, and recall that
$\bigl| \mathbb{E}\bigl[ e^{\sqrt{-1}\lambda X^{2}_{1}} \bigr] \bigr| = e^{-c_{6}|\lambda|^{3/2}}$
for any $\lambda \in \mathbb{R}$ for some $c_{6} \in (0,\infty)$ by \cite[Chapter~VIII, equation~(1)]{bertoin1996levy}
and hence that the law of $X^{2}_{1}$ has a bounded continuous density
$f_{1}\colon \mathbb{R} \to [0,\infty)$ by Fourier inversion. Then noting that
$\pi^{-1}\int_{0}^{\alpha/(\alpha+\beta)} x^{-1/2} (1-x)^{-1/2} \, dx \leq (\alpha/\beta)^{1/2}$
for any $\alpha,\beta \in (0,\infty)$ by considering the cases of $\alpha \leq \beta$
and $\alpha > \beta$ separately, we see from the Markov property of $X^{2}$ at time $b$,
\cite[Chapter~VIII, Exercise~3]{bertoin1996levy} and the scaling property of $X^{2}$ that
\begin{align}
&\mathbb{P}\bigl[ \tau_{z} > b, \, X^{2}_{\tau_{z}} > -z-y \bigr] \notag \\
&= \mathbb{E}\bigl[ \indicator_{ \{ \tau_{z} > b \} } \bigl( \mathbb{P}[ X^{2}_{\tau_{z+x}} > -z-y-x ] |_{x=X^{2}_{b}} \bigr) \bigr] \notag \\
&= \mathbb{E}\biggl[ \indicator_{ \{ \tau_{z} > b \} } \frac{1}{\pi} \int_{0}^{y/(z+y+X^{2}_{b})} x^{-1/2} (1-x)^{-1/2} \, dx \biggr] \notag \\
&\leq y^{1/2} \mathbb{E}\bigl[ ( z+X^{2}_{b} )^{-1/2} \indicator_{ \{ \tau_{z} > b \} } \bigr]
  = y^{1/2} b^{-1/3} \mathbb{E}\bigl[ ( w + X^{2}_{1} )^{-1/2} \indicator_{ \{ \tau_{w} > 1 \} } \bigr] \notag \\
&\leq y^{1/2} b^{-1/3} \mathbb{E}\bigl[ ( w + X^{2}_{1} )^{-1/2} \indicator_{ \{ w + X^{2}_{1} > 0 \} } \bigr]
  = y^{1/2} b^{-1/3} \int_{0}^{\infty} x^{-1/2} f_{1}(x-w) \, dx \notag \\
&\leq y^{1/2} b^{-1/3} \biggl( c_{7} \int_{0}^{1} x^{-1/2} \, dx + \int_{1}^{\infty} f_{1}(x-w) \, dx \biggr)
  \leq y^{1/2} b^{-1/3} (2c_{7}+1),
\label{eq:stable-X1-I1large-I2small-X2-term2-2}
\end{align}
where $w := b^{-2/3} z$ and $c_{7} := \sup_{x\in\mathbb{R}}f_{1}(x)$.
Thus by choosing $b:=\pi^{3}(4c_{7}+2)^{3}$, from~\eqref{eq:stable-X1-I1large-I2small-X2-term2-1}
and~\eqref{eq:stable-X1-I1large-I2small-X2-term2-2} we obtain
\begin{equation} \label{eq:stable-X1-I1large-I2small-X2-term2-3}
\begin{split}
\mathbb{P}\bigl[ \tau_{z} \leq b, \, X^{2}_{\tau_{z}} > -z-y \bigr]
  &= \mathbb{P}\bigl[ X^{2}_{\tau_{z}} > -z-y \bigr]
  - \mathbb{P}\bigl[ \tau_{z} > b, \, X^{2}_{\tau_{z}} > -z-y \bigr] \\
&\geq (2\pi)^{-1} y^{1/2}.
\end{split}
\end{equation}
Further, by \cite[Chapter~VIII, Proposition~4]{bertoin1996levy} and the scaling property
of $X^{2}$ there exist $A \in [2,\infty)$ and $c_{8} \in (0,\infty)$ such that for any
$\epsilon \in [ \frac{1}{10}b^{-2/3}, b^{-2/3} ]$, any $x \in [ \frac{1}{10}A, 10A ]$ and any $s \in (0,\infty)$,
\begin{equation*}
\mathbb{P}\bigl[ I^{2}_{s^{-3/2}} \in [ -x/s, -(1-\epsilon)x/s ) \bigr] 
  = \mathbb{P}\bigl[ I^{2}_{1} \in [ -x, -(1-\epsilon)x ) \bigr]
  \geq c_{8},
\end{equation*}
which in turn, with $\epsilon = \epsilon_{0} := \frac{1}{6} b^{-2/3}$,
$x = x_{z,b} := ( z - \frac{1}{2}b^{-2/3} ) A$ and $s=A$, yields
\begin{equation} \label{eq:stable-X1-I1large-I2small-X2-term2-4}
\begin{split}
&\mathbb{P}\bigl[ \tau_{z,b} \leq \tfrac{1}{2}, \, X^{2}_{\tau_{z,b}} \in [ -z+\tfrac{1}{2}b^{-2/3}, -z+b^{-2/3} ] \bigr] \\
&\geq \mathbb{P}\bigl[ I^{2}_{A^{-3/2}} \in [ -z+\tfrac{1}{2}b^{-2/3}, -z+b^{-2/3} ) \bigr] \\
&\geq \mathbb{P}\bigl[ I^{2}_{A^{-3/2}} \in [ -x_{z,b}/A, -(1-\epsilon_{0})x_{z,b}/A ) \bigr]
  \geq c_{8},
\end{split}
\end{equation}
where $\tau_{z,b} := \tau_{z-b^{-2/3}}$. It follows from the strong Markov property
of $X^{2}$ at time $\tau_{z,b}$, \eqref{eq:stable-X1-I1large-I2small-X2-term2-4},
the scaling property of $X^{2}$ and~\eqref{eq:stable-X1-I1large-I2small-X2-term2-3}
that, provided $y \in (0,\frac{1}{2}b^{-2/3}]$,
\begin{align*}
&\mathbb{P}\bigl[ \tau_{z} \leq 1, \, X^{2}_{\tau_{z}} > -z-y \bigr] \\
&\geq \mathbb{P}\bigl[ \tau_{z,b} \leq \tfrac{1}{2}, \, X^{2}_{\tau_{z,b}} \in [ -z+\tfrac{1}{2}b^{-2/3}, -z+b^{-2/3} ], \, \tau_{z} \leq \tau_{z,b} + \tfrac{1}{2}, \, X^{2}_{\tau_{z}} > -z-y \bigr] \\
&= \mathbb{E}\bigl[ \indicator_{ \{ \tau_{z,b} \leq \frac{1}{2}, \, X^{2}_{\tau_{z,b}} \in [ -z+\frac{1}{2}b^{-2/3}, -z+b^{-2/3} ] \} } \bigl( \mathbb{P}[ \tau_{z+x} \leq \tfrac{1}{2}, \, X^{2}_{\tau_{z+x}} > -z-y-x ]|_{ x = X^{2}_{\tau_{z,b}} } \bigr) \bigr] \\
&\geq c_{8} \inf\nolimits_{ x \in [ \frac{1}{2}b^{-2/3}, b^{-2/3} ] } \mathbb{P}\bigl[ \tau_{x} \leq \tfrac{1}{2}, \, X^{2}_{\tau_{x}} > -x-y \bigr] \\
&= c_{8} \inf\nolimits_{ x \in [1,2] } \mathbb{P}\bigl[ \tau_{x} \leq \sqrt{2}b, \, X^{2}_{\tau_{x}} > -x-2b^{2/3}y \bigr]
  \geq c_{8} (2\pi)^{-1} (2b^{2/3})^{1/2} y^{1/2},
\end{align*}
which proves~\eqref{eq:stable-X1-I1large-I2small-X2-term2} and thereby completes
the proof of~\eqref{eq:stable-X1-I1large-I2small}.
\end{proof}

\begin{proposition} \label{prop:stable-reflected-exp-moment}
Let $X = \{X_{t}\}_{t \geq 0}$ be a $3/2$-stable L\'{e}vy process with only downward jumps and
$X_{0} = 0$, and set $Z_{t} := X_{t} - \inf_{0\leq s\leq t}(X_{s} \wedge 0)$ for $t \in [0,\infty)$.
Then there exist $c_{1},c_{2} \in (0,\infty)$ such that
$\mathbb{P}[\sup_{ 0 \leq t \leq 1 } Z_{t} \geq x] \leq c_{1} e^{-c_{2} x}$
for any $x \in [0,\infty)$.
\end{proposition}

\begin{proof}
Set $\tau_{0}:=0$, and define sequences $\{\sigma_{n}\}_{n=1}^{\infty},\{\tau_{n}\}_{n=1}^{\infty}$
of stopping times for $X$ inductively by $\sigma_{n}:=\inf\{t \in [\tau_{n-1},\infty) \mid Z_{t} \geq 1\}$
and $\tau_{n}:=\inf\{t \in [\sigma_{n},\infty) \mid Z_{t} =0\}$ for $n \geq 1$, so that
$Z_{\sigma_{n}}=1$ on $\{\sigma_{n}<\infty\}$ by the absence of upward jumps of $X$.
Since $\{Z_{t}\}_{t\geq 0}$ is strong Markov by \cite[Chapter~VI, Proposition~1]{bertoin1996levy},
it easily follows from \cite[Chapter~VII, Theorem~8]{bertoin1996levy} that
$\sigma_{n} < \tau_{n} < \infty$ for any $n \geq 1$ a.s.\ and hence that
$\{ \{ Z_{t+\sigma_{n}} \}_{t \in [0, \tau_{n}-\sigma_{n})} \}_{n=1}^{\infty}$
is i.i.d.\ with law given by that of $\{ X_{t+\sigma} \}_{t\in[0,\tau-\sigma)}$ with
$\sigma:=\inf\{t\in[0,\infty) \mid X_{t} \geq 1\}$ and $\tau:=\inf\{t\in[\sigma,\infty) \mid X_{t} \leq 0\}$.
Moreover, recalling that $\{ X_{t+\sigma}-1 \}_{t\geq 0}$ has the same law as $X$ by
the strong Markov property \cite[Chapter~I, Proposition~6]{bertoin1996levy} of $X$, we have
$q := \mathbb{P}[\tau-\sigma \leq 1] \in (0,1)$ by \cite[Chapter~VIII, Propositions~2 and 4]{bertoin1996levy}
and $c_{3}:=\mathbb{E}[\exp( 2(-\log q) \sup_{0 \leq t < (\tau-\sigma) \wedge 1} X_{t+\sigma} )] < \infty$
by \cite[Chapter~VII, Corollary~2-(i)]{bertoin1996levy}, and the random variable
$N := \min\{ n \geq 1 \mid \tau_{n} > 1 \}$ satisfies
$\mathbb{P}[N \geq n ] \leq \mathbb{P}[\bigcap_{j=1}^{n-1}\{\tau_{j}-\sigma_{j}\leq 1\}] = q^{n-1}$
for any $n \geq 1$. It thus follows that for any $x \in [1,\infty)$, with
$n:=\min(\mathbb{N}\cap[x,\infty))$, $c_{2}:=-\log q$ and $c_{4}:=c_{3}c_{2}^{-1}e^{c_{2}-1}+1$,
\begin{align*}
\mathbb{P}[\sup\nolimits_{0 \leq t \leq 1} Z_{t} \geq x]
	&\leq \mathbb{P}\Bigl[ \bigcup\nolimits_{1 \leq j \leq n} \{ \sup\nolimits_{0 \leq t < (\tau_{n} - \sigma_{n})\wedge 1} Z_{t+\sigma_{j}} \geq x\} \Bigr]
		+ \mathbb{P}[N > n] \\
&\leq n \mathbb{P}[\sup\nolimits_{0 \leq t < (\tau-\sigma) \wedge 1}X_{t+\sigma} \geq x] + q^{n} \\
&\leq n c_{3} e^{-2c_{2}x}+e^{-c_{2}n} \leq (x+1)c_{3}e^{-2c_{2}x}+e^{-c_{2}x} \leq c_{4} e^{-c_{2}x},
\end{align*}
proving the assertion for $x \in [1,\infty)$. The assertion for $x \in [0,1)$ follows by setting $c_{1} := c_{4} \vee e^{c_{2}}$.
\end{proof}

\bibliographystyle{abbrv}
\bibliography{literature}

\end{document}